\documentclass[11pt,?,reqno,final]{amsbook} 

\usepackage[%
papersize={16.5cm,24cm},
hoffset=-2pt,
voffset=-5pt,
headheight=12pt,%
headsep=19pt,%
marginparsep=8pt,%
marginparwidth=30pt,%
footskip=20pt,
]{geometry}

\usepackage[T1]{fontenc} 

\usepackage{amsmath,amsthm,amssymb,latexsym} 

\usepackage{amsmidx} 

\usepackage{fancyhdr} 

\usepackage{mathrsfs,dsfont,bold-extra} 

\usepackage[linktocpage,bookmarks=true,unicode]{hyperref}

\hypersetup{
pdfauthor={\copyright Marco Thill},
pdftitle={Introduction to Normed *-Algebras and their Representations, 7th ed.},
pdfsubject={representation theory of normed algebras with involution},
pdfkeywords={normed algebras with involution, spectral theory, representation theory},
pdfcreator={AMS-LaTeX}
}

\makeindex{symbols} 
\makeindex{concepts} 

\fancyheadoffset[LE,RO]{0pt}
\newcommand{\st}{\ensuremath{\ast}}

\newcommand{\s}{\mathrm{sp}}

\newcommand{\iu}{\mathrm{i}}

\newcommand{\id}{\mathrm{id}}

\newcommand{\tld}{\widetilde}

\newcommand{\wht}{\widehat}

\newcommand{\sa}{_{\textstyle{sa}}}

\newcommand{\bsa}{_{\textstyle{bsa}}}

\newcommand{\rlambda}{\mathrm{r}_{\textstyle{\lambda}}}

\newcommand{\rsigma}{\mathrm{r}_{\textstyle{\sigma}}}

\newcommand{\subx}{_{\textit{\small{x}}\,}}

\newcommand{\suby}{_{\textit{\small{y}}\,}}

\newcommand{\subxy}{_{\textit{\small{xy}}\,}}

\newcommand{\subxyn}{_{\textit{\small{$xy_{\textit{\tiny{n}}}$}}\,}}

\newcommand{\subg}{_{\textit{\small{g}}\,}}

\newcommand{\subgn}{_{\textit{\small{$g_{\textit{\tiny{n}}}$}}}}

\newcommand{\univ}{_{\textit{\small{u}}}}

\newcommand{\ra}{_{\textit{\small{ra}}}}

\newcommand{\smu}{_{\text{\small{$\mu$}}}}

\newcommand{\0}{\ensuremath{_{\mathrm{o}}}}

\newcommand{\twiddle}{\ensuremath{^{\sim}}}

\newcommand{\tla}{_{\textstyle{a}}}

\newcommand{\cont}{\mathrm{C}}

\newcommand{\blop}{\mathrm{B}}

\newcommand{\rmLeb}{\mathrm{L}}

\newcommand{\domain}{\mathrm{D}}

\newcommand{\range}{\mathrm{R}}

\newcommand{\statespace}{\mathrm{S}}

\newcommand{\quasistates}{\mathrm{QS}}

\newcommand{\purestates}{\mathrm{PS}}

\newcommand{\Cstar}{\mathrm{C}\textnormal{*}}

\newcommand{\probmeas}{\mathrm{M}_1}

\newcommand{\stepfun}{\mathcal{S}}

\newcommand{\bmeas}{\mathcal{M}_{\mathrm{b}}}

\newcommand{\vonNeumAlg}{\mathrm{W}\textnormal{*}}

\newcommand{\diff}{\mathrm{d}}

\newlength{\skipamount}
\newtheoremstyle{mytheorem}%
{\skipamount}
{\skipamount}
{\itshape}
{}
{}
{.}
{0.5em}
{\thmnumber{#2}\thmname{ \scshape{\bfseries#1}}\thmnote{ \upshape(#3)}}

\newtheoremstyle{myplain}%
{\skipamount}
{\skipamount}
{}
{}
{}   
{.}
{0.5em}
{\thmnumber{#2}\thmname{ \scshape{\bfseries#1}}\thmnote{ \upshape(#3)}}

\theoremstyle{mytheorem}
\newtheorem{theorem}{Theorem}[section]

\newtheorem{definition}[theorem]{Definition}
\newtheorem{proposition}[theorem]{Proposition}
\newtheorem{lemma}[theorem]{Lemma}
\newtheorem{corollary}[theorem]{Corollary}
\newtheorem{addendum}[theorem]{Addendum}
\theoremstyle{myplain}
\newtheorem{example}[theorem]{Example}
\newtheorem{counterexample}[theorem]{Counterexample}
\newtheorem{introduction}[theorem]{Introduction}
\newtheorem{reminder}[theorem]{Reminder}
\newtheorem{remark}[theorem]{Remark}
\newtheorem{observation}[theorem]{Observation}

\begin{document}


\frontmatter

\title{Introduction to Normed \st-Algebras \protect\\
and their Representations, 7th ed.}

\author{Marco Thill}

\maketitle










\thispagestyle{empty}

\bigskip\noindent
\copyright \ Copyright : Marco Thill, 2010


\bigskip\noindent
Typeset by \AmS-\LaTeX

\pagebreak


\fancyhead[LE]{\SMALL{CONTENTS}}
\fancyhead[RO]{\SMALL{CONTENTS}}

\tableofcontents

\clearpage


\phantomsection

\chapter*{Introduction}

\fancyhead[LE]{\SMALL{INTRODUCTION}}
\fancyhead[RO]{\SMALL{INTRODUCTION}}

The present text was not intended to be published. The author
does not claim precedence as to some extent he was influenced
by inspiration, collaboration and editing by colleagues and others,
who all remain unnamed at this stage of the work.

A synopsis of the material is given in the next three paragraphs.

In part \ref{part1}, the by now classical spectral theory of Banach
\st-algebras is developed, including the Shirali-Ford Theorem
\ref{ShiraliFord}. This part cannot serve as an introduction to Banach
algebras in general, as the scope is limited to the prerequisites
for the sequel, that is, basic representation theory of \st-algebras.

The eponymous part \ref{part2} is the core of the book, in that it
contains the main innovations. The theory of the GNS construction is
exposed. The context is that of a general normed \st-algebra, stripped
of the usual assumptions of completeness, presence of an approximate
unit and isometry of the involution. We exploit the idea of Berg,
Christensen and Ressel, \cite[Theorem 4.1.14 p.\ 91]{BCR}, which
becomes our items \ref{weaksa} - \ref{sigmaAsa}. These items say
that a representation of a normed \linebreak \st-algebra on a pre-Hilbert
space, which is weakly continuous on the real subspace of Hermitian
elements, actually is contractive on this real subspace. Rather than
merely proving the Gel'fand-Na\u{\i}mark Theorem, the text attaches to
each normed \st-algebra a C*-algebra, which is useful in pulling down
results. The construction of this so-called enveloping C*-algebra in
the present generality is one of the main goals of this exposition.
We also give a partial proof with a representation theoretical flavour
of the Continuity Theorem of Varopoulos.

In part \ref{part3}, we consider the Spectral Theorem using integrals
converging in norm in the bounded case, and pointwise in the
unbounded case. Our Cauchy-Bochner type of approach leads to
easier constructions, simpler proofs, and stronger results than the
more conventional approach via weakly convergent integrals. It is
inspired by \cite[p.\ 57]{AMR}. Other authors used similar approaches
before, like Segal \& Kunze \cite{SeKu} for bounded operators, and
Weidmann \cite{Weid} for unbounded operators. We give the Spectral
Theorem for non-degenerate representations on Hilbert spaces
$\neq \{0\}$ of general commutative \st-algebras, rather than only of
commutative Banach \st-algebras, or even only of commutative
C*-algebras. This is possible at little extra effort, but it does not
seem to have permeated the mathematical folklore that this is possible
at all. We get the main disintegrations (the abstract Bochner
Theorem as well as the Spectral Theorem), in an easy step from the
corresponding results for C*-algebras by taking an image measure.
This approach is inspired by Mosak \cite[pp.\ 95--102.]{MosS}.
One point that we would like to make in this book, is showing
\emph{why} it is enough to study C*-algebras for representation theory.

The material is not limited to unital algebras. Proofs are fairly detailed,
and there are plenty of cross-references for those who don't read the
book from cover to cover. One design goal for this work has been
completeness as far as proofs are concerned. The book is suitable
for accompanying a lecture. It also could serve as an underpinning
for an intermediate course in Functional Analysis leading to the more
advanced treatments.

A prerequisite for reading this work is a course in Functional Analysis,
which nowadays can be expected to cover the unital case of the
Commutative Gel'fand-Na\u{\i}mark Theorem, as well as its ingredients:
the notion of a C*-algebra, the Spectral Radius Formula, multiplicative
linear functionals and ideals, the weak* topology and Alaoglu's Theorem,
as well as the Stone-Weierstrass Theorem. Since some familiarity with
these items is supposed, we take a slightly more advanced point of view
on these matters, yet stay self-contained. For studying part \ref{part2},
acquaintance with the GNS construction for unital C*-algebras (that is,
the Non-Commutative Gel'fand-Na\u{\i}mark Theorem) is necessary in
all probability. A course in Integration Theory covering the Riesz
Representation Theorem is required for reading part \ref{part3} of the
book, cf.\ \hyperlink{appint}{\S\ 52} in the appendix. The material
depending on measure theory has been made to work with most,
if not all, of the prevalent definitions using inner regularity. This is
possible because we only consider bounded measures, cf.\ the
appendix \ref{bdedmeas}.


What is missing from the text:
H*-algebras and group algebras are not treated for lack of free energy.
We refer the reader to Loomis \cite{LoAHA} for an introduction to these
topics. We have avoided material on general Banach algebras, which
is not needed for Banach algebras with involution. In this spirit, the
following omissions have been made. Algebras of analytic functions
are completely absent. The holomorphic functional calculus is not
covered, as we can do with only power series and the Rational Spectral
Mapping Theorem. (See e.g.\ \ref{entire}.) Similarly the radical
is not considered as we can do with only the \st-radical.
Two other omissions are notable, given our more general context.
Namely, the developments of Palmer \cite{Palm} on spectral
algebras, and the theory of H-algebras by M\"{u}ller \cite{MullB},
are not treated. The reader will have little problems with these
topics after having read the present book, which treats different
aspects.

Now for the ``version history'' of the work.

The first four editions were published by Kassel University Press in
Germany in 1998, 1999, 2002 and 2004. The copyright then returned
to the author. Two subsequent editions appeared online at arXiv.org
in 2007 and 2008.

The second edition differs from the first one, in that the first edition
considered normed \st-algebras with isometric involution exclusively.
\linebreak This prompted a complete rearrangement of the text. We
stress that in the present work, the involution may be discontinuous.
I gratefully acknowledge that Theodore W.\ Palmer and Torben Maack
Bisgaard brought to my attention a few crucial mistakes in a draft of the
second edition.

Since the third edition, the text is typeset by \LaTeXe, and contains an
index (with currently more than 450 entries). The content of the third
edition differs from that of the second edition, mainly in that the spectral
theory of representations now deals with general commutative
\st-algebras, as opposed to normed commutative \st-algebras.

Since the fourth edition, the text is typeset by \AmS-\LaTeX. Chapter
\ref{PositiveElements} has been made completely independent of Chapter
\ref{TheGel'fandTransformation}, and thereby freed of the use of the full
version of the Axiom of Choice, cf.\ \ref{Zorn}.

The fifth edition features many minor improvements. The following
paragraphs have been added: \ref{boundary}, \ref{StoneCech},
\ref{polarfactoris}, \ref{leftspectrum}, \ref{separating}, \ref{scopug},
as well as \hyperlink{remtop}{\S\ 51} in the appendix. Many thanks to
Torben Maack Bisgaard, who reported many typos as well as a couple
of errors (one unrecoverable), and who made useful suggestions.

The sixth edition fixes some minor matters, approximately half of which
were reported by Torben Maack Bisgaard. Thank you, Torben!

For the seventh edition, the text has been extensively revised and
expanded. A major gap in the proof of \ref{gap} has been closed.
The following paragraphs have been added: \ref{someoffunctions},
\ref{St-WstThm}, \ref{PtakHerm}, \ref{Borelalg}, \ref{subsum}, as well as
\hyperlink{extBaire}{\S\ 53} in the appendix. The text now contains
a list of symbols.
\pagebreak

\clearpage


\mainmatter

\fancyhead[LE]{\SMALL{\leftmark}}
\fancyhead[RO]{\SMALL{\rightmark}}


\part{Spectral Theory of Banach \texorpdfstring{$*$-}{\80\052\80\055}Algebras}%
\label{part1}%

\chapter{Normed \texorpdfstring{$*$-}{\80\052\80\055}Algebras: %
Basic Definitions and Facts}


\section{\texorpdfstring{$*$-}{\80\052\80\055}Algebras and their Unitisation}

\smallskip
\begin{definition}[algebra]\index{concepts}{algebra}%
An \underline{algebra} is a complex vector space $A$ equipped with
a multiplication $A \times A \to A : (a,b) \mapsto ab$ which is associative
and bilinear. Associativity means of course that
\[ (ab)c = a(bc) \quad \text{for all} \quad a, b, c \in A. \]
Bilinearity means linearity in each of the two variables, that is,
\begin{align*}
( \lambda a + \mu b ) \mspace{2mu}c & = \lambda ( a c ) + \mu ( b c ) \\
c \mspace{2mu}( \lambda a + \mu b ) & = \lambda ( c a ) + \mu ( c b )
\end{align*}
for all $\lambda, \mu \in \mathds{C}$ and all $a, b, c \in A$.
We consider complex algebras only!
\end{definition}

\begin{example}[$\mathrm{End}(V)$]\label{EndV}%
\index{symbols}{E()@$\mathrm{End}(V)$}%
If $V$ is a complex vector space, then the vector space
$\mathrm{End}(V)$ of linear operators on $V$ is an algebra
under \linebreak composition. (Here ``$\mathrm{End}$''
stands for ``endomorphism''.)
\end{example}

\begin{definition}[involution, adjoint]%
\index{concepts}{involution}%
\index{concepts}{adjoint}%
An \underline{involution} on a set $S$ is a mapping
$\st : S \to S$, $s \mapsto s^*$ such that
\[ {(s^*)}^* = s \quad \text{for all} \quad s \in S. \]
For $s \in S$, the element $s^*$ then is called the \underline{adjoint} of $s$.
\end{definition}

\begin{definition}[algebra involution]%
\index{concepts}{algebra!involution}%
An \underline{algebra involution} on an algebra $A$ is an involution
$a \mapsto a^*$ on $A$ such that for all $\lambda, \mu \in \mathds{C}$
and all $a, b \in A$, we have
\begin{align*}
{( \lambda a + \mu b )}^* &
= \overline{\lambda} {a}^* + \overline{\vphantom{\lambda}\mu} {b}^*
\qquad \qquad \text{(conjugate linearity)} \\
{(ab)}^* & = {b}^*{a}^*.
\end{align*}
Note the exact behaviour of an algebra involution as it acts on a product.
\end{definition}

\begin{definition}[\protect\st-algebra]%
\index{concepts}{algebra!s-algebra@\protect\st-algebra}%
\index{concepts}{s01@\protect\st-algebra|seeonly{algebra}}%
A \underline{\st-algebra} is an algebra equipped with an algebra involution.
\pagebreak
\end{definition}

\begin{example}[$\mathds{C}^{\,\text{\Small{$\Omega$}}}$]\label{FOmega}%
\index{symbols}{C05@$\mathds{C}^{\,\text{\Small{$\Omega$}}}$}%
If $\Omega$ is a non-empty set, then $\mathds{C}^{\,\text{\Small{$\Omega$}}}$
denotes the complex vector space of complex-valued functions on $\Omega$
(with the pointwise vector space operations). It becomes an algebra if
multiplication of functions is defined pointwise. It becomes a \st-algebra
if the involution is defined as pointwise complex conjugation. Please note
that the \st-algebra $\mathds{C}^{\,\text{\Small{$\Omega$}}}$ is commutative.
\end{example}

\begin{example}[$\blop(H)$]\index{symbols}{B(H)@$\blop(H)$}%
If $H$ is a complex Hilbert space, then the algebra $\blop(H)$
of bounded linear operators on $H$ is a \st-algebra under the operation
of taking the adjoint of an operator. If $H$ is of dimension at least two,
then the \st-algebra $\blop(H)$ is not commutative.
\end{example}

All pre-Hilbert spaces in this book shall be \underline{complex} pre-Hilbert
spaces.

\begin{definition}[the free vector space \hbox{$\mathds{C}[S]$}]%
\index{symbols}{C07@$\protect\mathds{C}[S]$}%
\index{concepts}{free vector space}%
\index{concepts}{vector!space!free}%
\index{concepts}{carrier}%
Let $S$ be a non-empty set. One denotes by $\mathds{C}[S]$ the
vector space of complex-valued functions $a$ on $S$ of finite carrier
$\{ \,s \in S : a(s) \neq 0 \,\}$. One may consider $S$ as a basis for
$\mathds{C}[S]$. More precisely, for $s$ in $S$, one may consider
the function $\delta _s$ in $\mathds{C}[S]$ defined as the function
taking the value $1$ at $s$ and vanishing everywhere else on $S$.
For $a$ in $\mathds{C}[S]$, we then have
\[ a = \sum _{\text{\small{$s \in S$}}} a(s)\,\delta _{\text{\small{$s$}}} \]
where the sum involves only finitely many non-zero terms. One says
that $\mathds{C}[S]$ is the \underline{free vector space over $S$}.
The idea behind this terminology is that $\mathds{C}[S]$ may be
considered as a vector space with basis $S$.
\end{definition}

\begin{example}[the group ring \hbox{$\mathds{C}[G]$}]%
\index{symbols}{C1@$\protect\mathds{C}[G]$}%
\index{concepts}{group ring}\label{CG}%
Let $G$ be a group, and let $G$ be multiplicatively written.
Consider as above the free vector space $\mathds{C}[G]$.
Since $G$ can be considered as a basis for $\mathds{C}[G]$,
it follows that $\mathds{C}[G]$ carries a unique structure of
algebra, such that its multiplication extends the one of $G$, namely
\[ a\mspace{2mu}b =
\sum _{\text{\small{$g, h \in G$}}} a(g)\,b(h)\,\delta _{\text{\small{$gh$}}}
\qquad \bigl( \,a, b \in \mathds{C}[G] \,\bigr). \]
An involution then is introduced by defining
\[ a^* := \sum _{\text{\small{$g \in G$}}}
\overline{a(g)}\,\delta _{\text{\small{$g$}}^{\text{\footnotesize{\,-1}}}}
\qquad \bigl( \,a \in \mathds{C}[G] \,\bigr). \]
This makes $\mathds{C}[G]$ into a \st-algebra, called the
\underline{group ring} of $G$. \pagebreak
\end{example}

\begin{definition}[\st-semigroup]%
\index{concepts}{semigroup}%
\index{concepts}{s045@\protect\st-semigroup}%
\index{concepts}{semigroup!ssemigroup@\protect\st-semigroup}%
A semigroup is a set equipped with an associative multiplication.
A \underline{\st-semigroup} is a semigroup $S$ equipped with an
involution $\st : S \to S, s \mapsto s^*$ satisfying
\[ {(st)}^* = t^*s^* \quad \text{for all} \quad s, t \in S. \]
For example every group is a \st-semigroup when the involution
is defined to be inversion. Also, every commutative semigroup
is a \st-semigroup when equipped with the identical involution.
\end{definition}

\begin{example}[the \st-semigroup ring \hbox{$\mathds{C}[S]$}]%
\index{symbols}{C07@$\protect\mathds{C}[S]$}\label{CS}%
From a \st-semigroup we can form a \st-algebra in a straightforward
generalisation of example \ref{CG}. Let $S \neq \varnothing$ be
a \st-semigroup. Since $S$ may be considered as a basis for the
free vector space $\mathds{C}[S]$, the latter carries precisely one
structure of \st-algebra, such that its multiplication and involution
extend the ones of $S$, namely
\[ a\mspace{2mu}b =
\sum _{\text{\small{$s, t \in S$}}} a(s)\,b(t)\,\delta _{\text{\small{$st$}}}
\qquad \bigl( \,a, b \in \mathds{C}[S] \,\bigr) \]
and
\[ a^* := \sum _{\text{\small{$s \in S$}}}
\overline{a(s)}\,\delta _{\text{\small{$s$}}^{\mspace{1mu}\st}}
\qquad \bigl( \,a \in \mathds{C}[S] \,\bigr). \]
\end{example}

\begin{definition}[Hermitian and normal elements]%
\index{concepts}{Hermitian!element}%
\index{concepts}{normal!element|(}%
\index{concepts}{self-adjoint!element}%
Let $A$ be a \linebreak \st-algebra. An element $a$ of $A$ is called
\underline{Hermitian}, or \underline{self-adjoint}, if $a = a^*$.
An element $b$ of $A$ is called \underline{normal}, if
$b^*b = b\,b^*$, that is, if $b$ and $b^*$ commute.
Every Hermitian element is normal of course.
\end{definition}

\begin{definition}[$A\protect\sa$]%
\index{symbols}{A7@$A\protect\sa$}%
\index{concepts}{Hermitian!part}%
Let $A$ be a \st-algebra, and let $B$ be a subset of $A$.
One denotes by $B\sa$ the set of Hermitian elements of $B$.
Here ``sa'' stands for ``self-adjoint''. We call $B\sa$ the
\underline{Hermitian part} of $B$. In particular, $A\sa$
will then be a real vector subspace of $A$.
\end{definition}

\begin{observation}
An arbitrary element $c$ of a  \st-algebra $A$ can be written uniquely
as $c = a+\iu b$ with $a,b$ in $A\sa$, namely $a = (c+c^*)/2$ and
$b = (c-c^*)/(2 \iu)$. Furthermore the element $c$ is normal if and only
if $a$ and $b$ commute. Hence:
\end{observation}

\begin{proposition}\index{concepts}{normal!element|)}%
A \st-algebra is commutative if and only if each of its elements is normal.
\pagebreak
\end{proposition}

\begin{observation}\label{Hermprod}
If $a, b$ are Hermitian elements of a \st-algebra, then the product
$ab$ is Hermitian if and only if $a$ and $b$ commute.

We thus see that $A\sa$ need not be closed under multiplication!
\end{observation}

\begin{definition}[unit, unital algebra]%
\index{concepts}{unit}\index{concepts}{algebra!unital}%
\index{concepts}{unital!algebra}%
Let $A$ be an algebra. A \underline{unit} in $A$ is a non-zero element
$e$ of $A$ such that
\[ ea = ae = a \quad \text{for all} \quad a \in A. \]
One says that $A$ is a \underline{unital} algebra if $A$ has a unit. 
\end{definition}

Please note that a unit is required to be different from the zero element.
An algebra can have at most one unit, as is easily seen. Hence any unit
in a \st-algebra is Hermitian. We shall reserve the notation ``$e$'' for units.%
\index{symbols}{e@$e$}

\begin{example}
If $G$ is a group with neutral element $g$, then $\mathds{C}[G]$ is a
unital \st-algebra with unit $\delta_{\text{\small{$g$}}}$.
\end{example}

\begin{definition}[unitisation]\index{symbols}{A8@$\protect\tld{A}$}%
\index{concepts}{unitisation}\index{concepts}{algebra!unitisation}%
Let $A$ be an algebra. If $A$ is unital, one defines $\tld{A} := A$. Assume
now that $A$ has no unit. One then defines \linebreak
$\tld{A} := \mathds{C} \oplus A$ (direct sum of vector spaces). One imbeds
$A$ into $\tld{A}$ via $a \mapsto (0,a)$. One defines $e := (1,0)$, so that
$(\lambda,a) = \lambda e + a$ for $\lambda \in \mathds{C}$, $a \in A$.
In order for $\tld{A}$ to become a unital algebra with unit $e$, and with
multiplication extending the one in $A$, the multiplication in $\tld{A}$ must
be \linebreak defined by
\[ ( \lambda e + a ) ( \mu e + b ) =
( \lambda \mu ) e + ( \lambda b + \mu a + ab )
\qquad ( \,\lambda, \mu \in \mathds{C}, a, b \in A \,), \]
and this definition indeed satisfies the requirements. One says that
$\tld{A}$ is the \underline{unitisation} of $A$. If $A$ is a \st-algebra,
one makes $\tld{A}$ into a unital \st-algebra by putting
\[ ( \lambda e + a )^* := \overline{\lambda} e + a^*
\qquad ( \,\lambda \in \mathds{C}, a \in A \,). \]
\end{definition}

\begin{definition}[algebra homomorphisms]\label{hom}%
\index{concepts}{homomorphism}\index{concepts}{algebra!homomorphism}%
\index{concepts}{algebra!s-algebra@\protect\st-algebra!homomorphism}%
Let $A,B$ be algebras. An \underline{algebra homomorphism}
from $A$ to $B$ is a linear mapping $\pi : A \to B$ such that
$\pi (ab) = \pi (a) \pi(b)$ for all $a,b \in A$.
If $A,B$ are \st-algebras, and if $\pi$ furthermore satisfies
$\pi (a^*) = \pi (a)^*$ for all $a$ in $A$, then $\pi$ is called a
\underline{\st-algebra homomorphism}. \pagebreak
\end{definition}

\begin{definition}[unital algebra homomorphisms]%
\index{concepts}{unital!homomorphism}\label{unitalhom}%
\index{concepts}{homomorphism!unital}%
Let $A,B$ be unital algebras  with units $e_A, e_B$ respectively.
An algebra homomorphism $\pi$ from $A$ to $B$ is called
\underline{unital}, if it satisfies $\pi (e_A) = e_B$.
\end{definition}

Next up are subalgebras.

\begin{definition}[subalgebra, \protect\st-subalgebra]%
\index{concepts}{subalgebra}\index{concepts}{algebra!subalgebra}%
\index{concepts}{algebra!subalgebra!s-subalgebra@\protect\st-subalgebra}%
\index{concepts}{s06@\protect\st-subalgebra}%
\index{concepts}{subalgebra!s-subalgebra@\protect\st-subalgebra}%
\index{concepts}{algebra!s-algebra@\protect\st-algebra!s-subalgebra@\protect\st-subalgebra}%
A \underline{subalgebra} of an \linebreak algebra $A$ is a complex
vector subspace $B$ of $A$ such that $ab \in B$ whenever $a, b \in B$.
A \underline{\st-subalgebra} of a \st-algebra $A$ is a subalgebra
$B$ of $A$ containing with each element $b$ its adjoint $b^*$ in $A$.
\end{definition}

\begin{example}[$\cont (\Omega)$]%
\index{symbols}{C15@$\cont (\Omega)$}\label{CO}%
If $\Omega \neq \varnothing$ is a Hausdorff space, then the vector space
$\cont (\Omega)$ of continuous complex-valued functions on $\Omega$
is a \st-subalgebra of $\mathds{C}^{\,\text{\Small{$\Omega$}}}$,
cf.\ \ref{FOmega}.
\end{example}

\begin{definition}[unital subalgebra]\index{concepts}{unital!subalgebra}%
\index{concepts}{algebra!subalgebra!*-unital@unital}%
\index{concepts}{subalgebra!*-unital@unital}%
\index{concepts}{s06@\protect\st-subalgebra!*-unital@unital}%
If $A$ is a unital algebra, then a subalgebra of $A$ containing the unit of $A$
is called a \underline{unital} subalgebra of $A$.
\end{definition}

The following two definitions are our own. 

\begin{definition}[\twiddle-unital subalgebra]%
\label{twiddleunital}\index{concepts}{t-unital@\protect\twiddle-unital}%
\index{concepts}{subalgebra!t-unital@\protect\twiddle-unital}%
\index{concepts}{algebra!subalgebra!t-unital@\protect\twiddle-unital}%
\index{concepts}{unital!t-unital@\protect\twiddle-unital}%
\index{concepts}{s06@\protect\st-subalgebra!t-unital@\protect\twiddle-unital}%
Let $B$ be a subalgebra of an algebra $A$. We shall say
that $B$ is \underline{\twiddle-unital} in $A$, if either $B$
has no unit, or else the unit in $B$ also is a unit in $A$.
\end{definition}

\begin{definition}[the canonical imbedding]%
\label{canimb}\index{concepts}{canonical!imbedding}%
\index{concepts}{imbedding!canonical}%
Let $B$ be a subalgebra of an algebra $A$. The
\underline{canonical imbedding} of $\tld{B}$ into
$\tld{A}$ is described as follows. If $B$ has a unit, it is the map
$\tld{B} \to \tld{A}$ which is the identity map on $\tld{B} = B$.
If $B$ has no unit, it is the only linear map $\tld{B} \to \tld{A}$
which is the identity map on $B$ and which maps unit to unit.
\end{definition}

We see that the above is the most natural imbedding possible.\linebreak
The next statement gives a rationale for the preceding two definitions.

\begin{proposition}\label{twidunitp}%
Let $B$ be a subalgebra of an algebra $A$. Then $B$ is \twiddle-unital
in $A$ if and only if $\tld{B}$ is a unital subalgebra of $\tld{A}$ under
the canonical imbedding. \pagebreak
\end{proposition}

A large part of this book is about representations of \st-algebras
on Hilbert spaces, or even pre-Hilbert spaces, a concept to be
defined later on, cf.\ \ref{repdef} below. We mean such objects
when we speak about representations. A different concept, to be
introduced next, is that of the so-called ``left regular representation
of an algebra on itself''.

\begin{definition}[the left regular representation]%
\index{concepts}{left!regular representation}%
\index{concepts}{left!translation}\label{leftreg}%
\index{concepts}{representation!left regular}%
\index{symbols}{L17@$L\tla$}%
Let $A$ be an algebra. For $a \in A$, one defines a linear
operator $L\tla : A \to A$ by $L\tla x := ax$ for all $x \in A$.
The operator $L\tla$ is called \underline{left translation with $a$}.
The mapping $L : a \mapsto L\tla$ $(a \in A)$ then is an algebra
homomorphism $A \to \mathrm{End}(A)$, cf.\ \ref{EndV}.
It is called the \underline{left regular representation of}
\underline{$A$ on itself}.
\end{definition}

\begin{observation}\label{leftreginj}%
Let $A$ be an algebra. The left regular representation of $\tld{A}$
on itself is injective. Hence its range is isomorphic to $\tld{A}$ as
an algebra. (Please note the tildes.)
\end{observation}

\begin{proof}
Apply the translation operators to the unit in $\tld{A}$.
\end{proof}

\begin{corollary}
Any algebra is isomorphic to some algebra of \linebreak
linear operators (on the unitisation).
\end{corollary}

\clearpage


\section{Normed \texorpdfstring{$*$-}{\80\052\80\055}Algebras and their Unitisation}

\begin{definition}[algebra norm, normed algebra, Banach algebra]%
\index{concepts}{norm|(}\index{concepts}{algebra!norm}%
\index{concepts}{algebra!normed algebra}%
\index{concepts}{norm!algebra norm}%
\index{concepts}{normed!algebra|seeonly{algebra}}%
\index{concepts}{algebra!Banach algebra}%
\index{concepts}{Banach!algebra|seeonly{algebra}}%
An \underline{algebra norm} is a norm $|\cdot|$ on an algebra $A$ such that
\[ |\,ab\,| \leq |\,a\,|\cdot|\,b\,| \quad \text{for all} \quad a, b \in A. \]
The pair $\bigl( \,A , | \cdot | \,\bigr)$ then is called a \underline{normed algebra}.
A normed algebra is called a \underline{Banach algebra} if the
underlying normed space is complete, i.e.\ if it is a Banach space.
\end{definition}

\begin{example}[$\blop(V)$]\label{algbdedop}%
\index{symbols}{B(V)@$\blop(V)$}%
If $V$ is a complex normed space, we denote by $\blop(V)$ the
set of bounded linear operators on $V$. It is a normed algebra
under the operator norm. If $V$ furthermore is complete, i.e.\ if
it is a Banach space, then $\blop(V)$ is a Banach algebra.
\end{example}

\begin{definition}%
[normed \protect\st-algebra, Banach \protect\st-algebra]%
\index{concepts}{algebra!normed s-algebra@normed \protect\st-algebra}%
\index{concepts}{normed!s-algebra@\protect\st-algebra|seeonly{algebra}}%
\index{concepts}{Banach!s-algebra@\protect\st-algebra|seeonly{algebra}}%
\index{concepts}{algebra!Banach s-algebra@Banach \protect\st-algebra}%
A \st-algebra \linebreak equipped with an algebra norm shall be
called a \underline{normed \st-algebra}. A
\underline{Banach \st-algebra} shall be a normed \st-algebra such
that the underlying normed space is complete.
\end{definition}

\begin{example}[$\blop(H)$]\index{symbols}{B(H)@$\blop(H)$}%
If $H$ is a complex Hilbert space, then the  Banach
algebra $\blop(H)$ of bounded linear operators on $H$ is a Banach 
\st-algebra under the operation of taking the adjoint of an operator.
\end{example}

\begin{definition}[accessory \protect\st-norm]\label{auxnorm}%
\index{symbols}{a05@${"|}{"|}\,a\,{"|}{"|}$}%
\index{concepts}{involution!isometric}%
\index{concepts}{s03@\protect\st-norm}%
\index{concepts}{norm!accessory *-norm@accessory \protect\st-norm}%
\index{concepts}{accessory *-norm@accessory \protect\st-norm}%
\index{concepts}{s03@\protect\st-norm!accessory}%
Let $\bigl( \,A, | \cdot | \,\bigr)$ be a normed \linebreak \st-algebra.
One says that the involution in $A$ is \underline{isometric} if
\[ |\,a^*\,| = |\,a\,|\quad \text{for all} \quad a \in A. \]
One then also says that the norm $| \cdot |$ is a \underline{\st-norm},
and that $A$ is \underline{\st-normed}.
If the involution is not isometric, one can introduce an accessory algebra
norm by putting
\[ \|\,a\,\| := \sup\,\{ \,|\,a\,|,|\,a^*\,| \,\} \qquad ( \,a \in A \,). \]
In this norm the involution is isometric. We shall call this norm the
\underline{accessory \st-norm}.
\end{definition}

We shall reserve the notation $\| \cdot \|$ for \st-norms, in particular
the accessory \st-norm, and the norms on pre-C*-algebras,
see \ref{preC*alg} below. \pagebreak

\begin{observation}\label{involcont}%
\index{concepts}{involution!continuous}%
Let $A$ be a normed \st-algebra. If the involution in $A$ is
continuous, there exists $c > 0$ such that
\[ |\,a^*\,| \leq c\,|\,a\,| \quad \text{for all} \quad a \in A. \]
The accessory \st-norm then is equivalent to the original norm.
\end{observation}

\begin{example}\index{symbols}{C1@$\protect\mathds{C}[G]$}\label{CGn}%
Let $G$ be a group. One introduces an algebra norm $| \cdot |$
on $\mathds{C}[G]$ by putting
\[ |\,a\,| := \sum _{g \in G} |\,a(g)\,| \qquad \bigl( \,a \in \mathds{C}[G] \,\bigr), \]
thus making $\mathds{C}[G]$ into a normed \st-algebra with
isometric involution. (Recall \ref{CG}.)
\end{example}

\begin{definition}[C*-algebra, C*-property]%
\label{preC*alg}\index{concepts}{algebra!C*-algebra}%
\index{concepts}{algebra!pre-C*-algebra}%
\index{concepts}{pre-C*-algebra}%
\index{concepts}{C6@C*-algebra}\index{concepts}{C7@C*-property}%
A \underline{pre-C*-algebra} is a normed \st-algebra
$\bigl( \,A, \| \cdot \| \,\bigr)$, such that for all $a \in A$, one has
\[ {\|\,a\,\| \,}^{2} = \|\,a^*a\,\|\quad \text{as well as} \quad \|\,a\,\| = \|\,a^*\,\|. \]
The first equality is called the \underline{C*-property}.
A complete pre-C*-algebra is called a \underline{C*-algebra}.
(See also \ref{precompl} below.)
\end{definition}

Please note that pre-C*-algebras have isometric involution.

\bigskip
There is some redundancy in the above definition,
as the next proposition shows.

\begin{proposition}\label{condC*}%
Let $\bigl( \,A , \| \cdot \| \,\bigr)$ be a normed \st-algebra such that
\[ {\|\,a\,\| \,}^{2} \leq \|\,a^*a\,\| \quad \text{for all} \quad a \in A. \]
Then $(A,\|\cdot\|)$ is a pre-C*-algebra. The mnemonic is
\begin{center} ``to estimate a square by something positive''. \end{center}
\end{proposition}

\begin{proof} For $a$ in $A$ we have
${\|\,a\,\| \,}^2 \leq \|\,a^*a\,\| \leq \|\,a^*\,\| \cdot \|\,a\,\|$,
whence $\|\,a\,\| \leq \|\,a^*\,\|$. Hence also
$\|\,a^*\,\| \leq \| \,{(a^*)}^* \,\| = \|\,a\,\|$,
which implies $\|\,a^*\,\| = \|\,a\,\|$.
It follows that $\|\,a^*a\,\| \leq {\|\,a\,\| \,}^2$.
The converse inequality holds by assumption.
\end{proof}

\begin{corollary}\index{symbols}{B(H)@$\blop(H)$}%
If $(H, \langle \cdot, \cdot \rangle)$ is a Hilbert space, then $\blop(H)$
is a \linebreak C*-algebra, and with it every closed \st-subalgebra.
\pagebreak
\end{corollary}

\begin{proof}
For $a$ in $\blop(H)$ and $x$ in $H$, we have
\[ {\|\,ax\,\| \,}^2 = \langle ax, ax \rangle = \langle a^*ax, x \rangle
\leq \|\,a^*ax\,\| \cdot \|\,x\,\| \leq \|\,a^*a\,\| \cdot {\|\,x\,\| \,}^2, \]
so that upon taking square roots, we get $\|\,a\,\| \leq {\| \,a^*a \,\| \,}^{1/2}$.
\end{proof}

\medskip
The closed \st-subalgebras of $\blop(H)$, with $H$ a Hilbert space,
are the prototypes of C*-algebras, see the Gel'fand-Na\u{\i}mark
Theorem \ref{Gel'fandNaimark}. 

We can now turn to the unitisation of normed algebras.
We shall have to treat pre-C*-algebras apart from general
normed algebras.

\begin{proposition}\label{Cstarleftreg}%
For a pre-C*-algebra  $\bigl( \,A, \| \cdot \| \,\bigr)$ and for $a \in A$,
we have
\[ \| \,a \,\| = \max\,\{ \,\| \,ax \,\| : x \in A, \,\| \,x \,\| \leq 1\,\}. \]
For if $a \neq 0$, the maximum is achieved at $x = a^*/ \| \,a^* \,\|$.
That is, the left regular representation of a pre-C*-algebra on itself
\ref{leftreg} is isometric.
\end{proposition}

\begin{corollary}\label{unitCstar}%
The unit in a unital pre-C*-algebra has norm $1$.
\end{corollary}

\begin{proposition}\label{preCstarunitis}%
\index{concepts}{unitisation}%
\index{concepts}{pre-C*-algebra!unitisation}%
\index{concepts}{C6@C*-algebra!unitisation}%
\index{symbols}{A8@$\protect\tld{A}$}%
\index{concepts}{algebra!C*-algebra!unitisation}%
Let $\bigl( \,A , \| \cdot \| \,\bigr)$ be a pre-C*-algebra. By defining
\[ \| \,a \,\| = \sup\,\{ \,\| \,ax \,\| : x \in A, \,\| \,x \,\| \leq 1\, \}
\qquad ( \,a \in \tld{A} \ ) \]
one makes $\tld{A}$ into a unital pre-C*-algebra. The above
norm extends of course the norm in $A$, by \ref{Cstarleftreg}.
\end{proposition}

\begin{proof}
Assume that $A$ has no unit. Let $a = \lambda e + b \in \tld{A}$ with
$\lambda \in \mathds{C}$, $b \in A$. We define a linear operator
$L\tla : A \to A$ by $L\tla x := ax $ for $x \in A$. This is well-defined
because for $x \in A$ we have $L\tla x = \lambda x + bx \in A$.
For $x \in A$ with $\| \,x \,\| \leq 1$ we have
$\| \,L\tla x \,\| \leq | \,\lambda \,| + \| \,b \,\|$, so that $L\tla$ is a
bounded linear operator. Furthermore, $\| \,a \,\|$ is the operator
norm of $L\tla$. Thus, in order to prove that $\| \cdot \|$ is an
algebra norm, it suffices to show that if $L\tla = 0$ then $a = 0$.
So assume that $L\tla = 0$. For all $x$ in $A$ we then have
$0 = L\tla x = \lambda x + b x$. Thus, if $\lambda = 0$,
we obtain $b = 0$ by proposition \ref{Cstarleftreg}. It now suffices to
show that $\lambda = 0$. So assume that $\lambda$ is different from
zero. With $g := -b/ \lambda \in A$, we get $gx = x$ for all $x$ in $A$.
Hence also $xg^* = x$ for all $x$ in $A$. In particular we have
$gg^* = g$, whence, after applying the involution, $g^* = g$. Therefore
$g$ would be a unit in $A$, in contradiction with the assumption made
initially. Now, in order to show that $(\tld{A},\| \cdot \|)$ is a pre-C*-algebra,
it is enough to prove that $\| L\tla \,\| \leq {\| L_{\textstyle{a^*a}} \,\| \,}^{1/2}$,
cf.\ \ref{condC*}. For $x$ in $A$, we have
\[ {\| \,L\tla x \,\| \,}^{2} = \| \,(ax)^*(ax) \,\| \leq \| \,x^* \,\| \cdot \| \,a^*ax \,\|
\leq \| \,L_{\textstyle{a^*a}} \,\| \cdot {\| \,x \,\| \,}^{2}. \pagebreak \qedhere \]
\end{proof}

\begin{definition}[unitisation of a normed algebra]%
\label{normunitis}\index{concepts}{unitisation}%
\index{concepts}{norm!unitisation}\index{symbols}{A8@$\protect\tld{A}$}%
Let $\bigl( \,A , | \cdot | \,\bigr)$ be a normed algebra without unit.
If $A$ is not a pre-C*-algebra, one makes $\tld{A}$ into a unital
normed algebra by putting
\[ | \,\lambda e + a \,| := | \,\lambda \,| + | \,a \,|
\qquad \bigl( \,\lambda \in \mathds{C}, \,a \in A \,\bigr). \]
If $A$ is a pre-C*-algebra, one makes $\tld{A}$ into a unital
pre-C*-algebra as in the preceding proposition \ref{preCstarunitis}.
\end{definition}

\begin{theorem}\index{concepts}{unitisation}%
\index{symbols}{A8@$\protect\tld{A}$}\label{Banachunitis}%
If $A$ is a Banach algebra, so is $\tld{A}$.
\end{theorem}

\begin{proof}
This is so because $A$ has co-dimension $1$ in $\tld{A}$
if $A$ has no unit, cf.\ the appendix \ref{quotspaces}.
\end{proof}

\begin{theorem}[equivalence]\label{Cstarequiv}\index{concepts}{norm}%
\index{concepts}{unitisation}\index{symbols}{A8@$\protect\tld{A}$}%
\index{concepts}{norm|)}\index{concepts}{norm!unitisation}%
Let $\bigl( \,A , \| \cdot \| \,\bigr)$ be a C*-algebra without unit.
The unitisation $\tld{A}$ then carries besides its C*-algebra
norm $\| \cdot \| $, the complete norm
$| \,\lambda e+a \,| := | \,\lambda \,| + \|\,a\,\|$
$\bigl( \,\lambda \in \mathds{C}$, $a \in A \,\bigr)$, cf.\ the proof of
\ref{Banachunitis}. Both norms are equivalent on $\tld{A}$.
\end{theorem}

\begin{proof}
This follows from the Open Mapping Theorem because the C*-algebra
norm $\| \cdot \|$ on $\tld{A}$ is dominated by the other complete norm
$| \cdot |$, as is seen from \ref{unitCstar}.
\end{proof}

\medskip
It is actually easy to give constants for this equivalence:
see \ref{C*unitis} below.

\begin{theorem}\label{Banop}%
Let $A$ be a Banach algebra. There then exists a homeomorphic
algebra isomorphism from $A$ onto a Banach algebra of bounded
linear operators on a Banach space.
\end{theorem}

\begin{proof}
Consider the left regular representation $L : a \mapsto L\tla$ $(a \in \tld{A})$
of $\tld{A}$ on itself, cf.\ \ref{leftreg}. For $a \in A$, let $| \,L\tla \,|$ denote the
operator norm of the bounded linear operator $L\tla$ on the Banach space
$\tld{A}$, cf.\ \ref{Banachunitis}. The map $a \mapsto L\tla$ $(a \in A)$ is an
algebra isomorphism from $A$ onto the range $L(A)$, cf.\ \ref{leftreginj}.
This map clearly satisfies $| \,L\tla \,| \leq | \,a \,|$ as well as
$| \,a \,| \leq | \,L\tla \,| \cdot | \,e \,|$ for all $a \in A$, and thus is homeomorphic.
It follows easily that the range $L(A)$ is complete.
\end{proof}

\begin{corollary}\label{eqnorm}%
On a unital Banach algebra there exists an \linebreak algebra norm
equivalent to the original one, assuming the value $1$ at the unit.
\pagebreak
\end{corollary}

\clearpage


\section{The Completion of a Normed Algebra}

\begin{proposition}%
If $\bigl( \,A, | \cdot | \,\bigr)$ is a normed algebra, then
\[ |\,a\0 b\0 - a b\,| \leq |\,a\0\,|\cdot|\,b\0-b\,| +
|\,a\0-a\,|\cdot|\,b\0-b\,| + |\,b\0\,|\cdot|\,a\0-a\,| \]
for all $a\0,b\0,a,b \in A$. It follows that multiplication is jointly
continuous and uniformly so on bounded subsets.
\end{proposition}

\begin{theorem}[unique continuous extension]\label{continuation}%
Let $f$ be a function defined on a dense subset of a metric space $X$
and taking values in a complete metric space $Y$. Assume that $f$
is uniformly continuous on bounded subsets of its domain of definition.
Then $f$ has a unique continuous extension $X \to Y$.
\end{theorem}

\begin{proof}[\hspace{-3.75ex}\mdseries{\scshape{Sketch of a proof}}]
Uniqueness is clear. Let $x \in X$ and let $(x_{n})$ be a sequence in the
domain of definition of $f$ converging to $x$. Then $(x_{n})$ is a Cauchy
sequence, and thus also bounded. Since uniformly continuous maps take
Cauchy sequences to Cauchy sequences, it follows that $\bigl(f(x_{n})\bigr)$
is a Cauchy sequence in $Y$, hence convergent to an element of $Y$
denoted by $g(x)$, say. (One verifies that $g(x)$ is independent of the
sequence $(x_n)$.) The function $x \mapsto g(x)$ satisfies the requirements.
\end{proof}

\begin{corollary}[completion of a normed algebra]%
\index{concepts}{algebra!normed algebra!completion}%
\index{concepts}{completion}%
Let $A$ be a \linebreak normed algebra. Then the \underline{completion}
of $A$ as a normed space carries \linebreak a unique structure of Banach
algebra such that its multiplication \linebreak extends the one of $A$.
\end{corollary}

\begin{proof} This follows from the preceding two items. \end{proof}

\begin{proposition}\label{precompl}
If $A$ is a normed \st-algebra with \underline{continuous} \linebreak
\underline{involution}, then the completion of $A$ carries the structure
of a Banach \linebreak \st-algebra, by continuation of the involution.
In particular, the completion of a pre-C*-algebra is a C*-algebra.
\end{proposition}

The assumption of a continuous involution is essential.
Indeed, we shall later give an example of a commutative
normed \st-algebra which cannot be imbedded in a Banach
\st-algebra at all. See \ref{counterexbis} below. 

\begin{example}[${\ell\,}^1(G)$]\index{symbols}{l1@${\ell\,}^1(G)$}\label{l1G}%
If $G$ is a group, then the completion of $\mathds{C}[G]$ is ${\ell\,}^1(G)$.
It is a unital Banach \st-algebra with isometric involution.
(Recall \ref{CG} and \ref{CGn}.) \pagebreak
\end{example}

Let $A$ be a normed algebra and let $B$ be the completion
of $A$. If $A$ has a unit $e$, then $e$ also is a unit in $B$ as
is easily seen. Hence:

\begin{proposition}\label{twidunitcompl}%
A normed algebra is a \twiddle-unital subalgebra of its completion.
(Recall \ref{twiddleunital}.) 
\end{proposition}

From \ref{Banop} and \ref{eqnorm}, we have the following two consequences.

\begin{corollary}%
For a normed algebra $A$ there exists a Banach space $V$
and a homeomorphic algebra isomorphism from $A$ onto a
subalgebra of the Banach algebra $\blop(V)$.
\end{corollary}

\begin{corollary}%
On a unital normed algebra there exists an algebra norm equivalent
to the original one, assuming the value $1$ at the unit.
\end{corollary}

\begin{observation}\label{sphere}%
Let $A$ be a dense subspace of a normed space $B$.
The unit ball of $A$ then is dense in the unit ball of $B$.
\end{observation}

\begin{proof}
Indeed, it is easily seen that if $B \neq \{ 0 \}$, then the
unit sphere of $A$ is dense in the unit sphere of $B$.
\end{proof}

\begin{proposition}\index{concepts}{norm}%
\index{concepts}{norm!unitisation}\index{concepts}{unitisation}%
Let $\bigl( \,A, \| \cdot \| \,\bigr)$ be a pre-C*-algebra without unit.
Let $\bigl( \,B, \| \cdot \|_B \,\bigr)$ be the completion of $A$.
Assume that $B$ has a unit $e$.

Let $\bigl( \,\tld{A}, \| \cdot \|_{\tld{A}} \,\bigr)$ be the unitisation of $A$
as in \ref{preCstarunitis}. The canonical imbedding identifies $\tld{A}$
as a unital algebra with the unital subalgebra $\mathds{C}e + A$ of $B$.
The two norms $\| \cdot \|_{\tld{A}}$ and $\| \cdot \|_B$ then coincide on $\tld{A}$.
\end{proposition}

\begin{proof}
The statement on the embedding follows form \ref{canimb}.
(See also \ref{twidunitp} together with either \ref{twiddleunital}
or \ref{twidunitcompl}.) For $a \in \tld{A}$, we have
\begin{align*}
\| \,a \,\|_{\tld{A}} & = \sup \,\{ \,\| \,ax \,\| : x \in A, \,\| \,x \,\| \leq1 \,\}
\tag*{by \ref{preCstarunitis}} \\
 & = \sup \,\{ \,\| \,ax \,\|_B : x \in A, \,\| \,x \,\|_B \leq1 \,\} \\
 & = \sup \,\{ \,\| \,ay \,\|_B : y \in B, \,\| \,y \,\|_B \leq1 \,\} = \| \,a \,\|_B
 \tag*{by \ref{Cstarleftreg}}
\end{align*}
because the unit ball of $A$ is dense in the unit ball of $B$, see \ref{sphere}.
\end{proof}

\begin{corollary}\label{complunit}%
If $A$ is a pre-C*-algebra without unit such that the completion of $A$
contains a unit, then $A$ is dense in $\tld{A}$.
\end{corollary}

\begin{proof}
By the above, $\tld{A}$ is contained in the completion of $A$. \pagebreak
\end{proof}

\clearpage

%


\section{Some Pre\texorpdfstring{-}{\80\055}C\texorpdfstring{*-}{\80\052\80\055}Algebras %
and Lattices of Functions}%
\label{someoffunctions}

\begin{reminder}[${\ell}^{\,\infty}(\Omega)$]%
\index{concepts}{norm!supremum norm}\index{concepts}{supremum norm}%
\index{symbols}{l12@${\ell\,}^{\infty}(\Omega)$}%
Let $\Omega$ be a non-empty set. The vector space ${\ell}^{\,\infty}(\Omega)$
of all bounded complex-valued functions on $\Omega$ is a unital \st-subalgebra
of $\mathds{C}^{\,\text{\Small{$\Omega$}}}$, cf.\ \ref{FOmega}.
For $f \in {\ell}^{\,\infty}(\Omega)$, the set
\[ \bigl\{ \,| \,f(x) \,| \in [\,0, \infty \,[ \ : x \in \Omega \,\bigr\} \]
is bounded and thus has a finite least upper bound (or supremum),
denoted by
\[  | \,f \,|_{\,\infty} := \sup_{x \in \Omega} | \,f(x) \,|. \]
The function $f \mapsto | \,f \,|_{\,\infty}$ is a norm on ${\ell}^{\,\infty}(\Omega)$,
called the supremum norm. An alternative name is uniform norm
because it induces uniform convergence as is easily seen.
In this way, $({\ell}^{\,\infty}(\Omega), | \cdot |_{\,\infty})$ becomes a unital
C*-algebra. To see that ${\ell}^{\,\infty}(\Omega)$ is complete indeed,
consider a Cauchy sequence $(f_n)$ in ${\ell}^{\,\infty}(\Omega)$. For
given $\varepsilon > 0$, there exists $n\0(\varepsilon)$ with
\[ | \,f_n - f_m \,|_{\,\infty} \leq \varepsilon
\quad \text{for all} \quad n, m \geq n\0(\varepsilon). \]
Since
\[ | \,f_n(x) - f_m(x) \,| \leq | \,f_n - f_m \,| _{\,\infty}
\quad \text{for all} \quad x \in \Omega, \]
we obtain that the $f_n(x)$ form a Cauchy sequence for each
$x \in \Omega$. Hence the functions $f_n$ converge pointwise
to some complex-valued function $f$ on $\Omega$. Letting
$m \to \infty$, we get that
\[ | \,f_n(x) - f(x) \,|_{\,\infty} \leq \varepsilon
\quad \text{for all} \quad n \geq n\0(\varepsilon)
\quad \text{and all} \quad x \in \Omega.\]
It follows that $f$ is bounded and that
\[ | \,f_n - f \,|_{\,\infty} \leq \varepsilon
\quad \text{for all} \quad n \geq n\0(\varepsilon), \]
so that $f_n \to f$ in ${\ell}^{\,\infty}(\Omega)$.
(This is a standard completeness proof.)
\end{reminder}

\begin{reminder}[$\cont_\mathrm{b}(\Omega)$]%
\index{symbols}{C2@$\cont_\mathrm{b}(\Omega)$}%
Let $\Omega \neq \varnothing$ be a Hausdorff space.
We denote by
$\cont_\mathrm{b}(\Omega) := \cont (\Omega) \cap {\ell}^{\,\infty}(\Omega)$
the unital \st-algebra of bounded continuous complex-valued functions
on $\Omega$, cf.\ \ref{CO}. The subspace $\cont_\mathrm{b}(\Omega)$
is closed in ${\ell}^{\,\infty}(\Omega)$ because a uniform limit of a sequence
of continuous functions is continuous. So $\cont_\mathrm{b}(\Omega)$ is a
unital C*-subalgebra of ${\ell}^{\,\infty}(\Omega)$.
\end{reminder}

\begin{remark}[$\cont (K)$]\index{symbols}{C16@$\cont (K)$}%
Please note that if $K$ is a compact Hausdorff space, then
$\cont (K) = \cont_\mathrm{b}(K)$ is a unital C*-algebra.
\end{remark}

\begin{reminder}[Urysohn's Lemma]\label{Urysohn}%
\index{concepts}{Theorem!Urysohn's Lemma}%
\index{concepts}{Urysohn's Lemma}%
\index{concepts}{Lemma!Urysohn's}%
For two disjoint closed subsets $A, B$ of a normal Hausdorff space
$\Omega$ there exists $f \in \cont (\Omega)$ taking values in $[ \,0, 1 \,]$,
such that $f(x) = 0$ for all $x \in A$ and $f(y) = 1$ for all $y \in B$.
\pagebreak
\end{reminder}

For a proof, see for example \cite[Theorem 1.5.11 p.\ 41]{Eng}.

\bigskip
It is well-known that compact Hausdorff spaces are normal.
(See e.g.\ \cite[Theorem 3.1.9 p.\ 125]{Eng}.) Whence the following
consequences.

\begin{observation}\label{Urysimple}%
For two distinct points $x, y$ of a compact Hausdorff space $K$
there exists $f \in \cont (K)$ with $f(x) = 0$ and $f(y) = 1$.
\end{observation}

\begin{definition}[to separate the points]%
\index{concepts}{separating!the points}\label{sepp}%
Let $\mathcal{G}$ be a set of functions defined on a set
$\Omega \neq \varnothing$. One says that $\mathcal{G}$
\underline{separates the points} of $\Omega$, if for any two
distinct points $x, y \in \Omega$, there exists a function
$f \in \mathcal{G}$ such that $f(x) \neq f(y)$.
\end{definition}

\begin{definition}[to vanish nowhere]%
\index{concepts}{vanishing!nowhere}\label{vanish}%
Let $\mathcal{G}$ be a set of functions defined on a set
$\Omega \neq \varnothing$. One says that $\mathcal{G}$
\underline{vanishes nowhere} on $\Omega$, if there is no
point of $\Omega$, at which all the functions of $\mathcal{G}$ vanish.
\end{definition}

\begin{corollary}\label{KnecStW}%
Let $K \neq \varnothing$ be a compact Hausdorff space.
Then $\cont (K)$ separates the points of $K$ and vanishes
nowhere on $K$.
\end{corollary}

\begin{definition}[divisors of zero]\label{zerodivdef}%
\index{concepts}{divisors of zero}\index{concepts}{zero divisors}%
Two non-zero elements $a, b$ of a ring satisfying $ab = ba = 0$
are called \underline{divisors of zero}.
\end{definition}

\begin{corollary}\label{Kdivzero}%
Let $K$ be a compact Hausdorff space containing at least
two distinct points. Then $\cont (K)$ contains divisors of zero.
\end{corollary}

\begin{proof}
Let $U(x), U(y)$ be two disjoint neighbourhoods of two distinct points
$x, y$ of $K$ respectively. Apply Urysohn's Lemma \ref{Urysohn} with
the closed sets $A := K \setminus U(x)$ and $B := \{ x \}$, then to
$K \setminus U(y)$ and $\{ y \}$.
\end{proof}

\begin{reminder}[the one-point compactification]%
\index{concepts}{one-point compactification}\label{onepcomp}%
\index{concepts}{compactification!one-point}%
\index{concepts}{compactification!Alexandroff}%
Let $\Omega \neq \varnothing$ be a locally compact Hausdorff
space. Let $K := \Omega \cup \{ \infty \}$ with $\infty \notin \Omega$.
\linebreak (Such an $\infty$ exists by an axiom of set theory.)
Then $K$ carries a unique topology in which $K$ is a compact
Hausdorff space containing $\Omega$ as a subspace. We
mention here that the open neighbourhoods of $\infty$ are of the
form $K \setminus C$ with $C$ any compact subset of $\Omega$.
In particular, $\Omega$ is dense in $K$ if and only if $\Omega$ is
not compact. In this case, the set $K$ together with the above
topology is called a \underline{one-point compactification} of
$\Omega$, and the point $\infty$ is called the corresponding
\underline{point at infinity}. The one-point compactification is
also called the Alexandroff compactification. For a detailed proof,
see e.g.\ \cite[section I.7.6 p.\ 67 ff.]{Schu}. \pagebreak
\end{reminder}

We next give a motivation for the definition of $\cont\0(\Omega)$
for a locally compact Hausdorff space $\Omega \neq \varnothing$.

\begin{observation}[$\cont\0(\Omega)$]%
\index{symbols}{C3@$\cont\0(\Omega)$|(}\label{motCnaught}%
\index{concepts}{vanishing!at infinity}
Let $\Omega \neq \varnothing$ be a locally compact Hausdorff space
which is not compact. We consider $\Omega$ to be imbedded in its
one-point compactification $K$, with corresponding point at infinity $\infty$.

A complex-valued function $f$ on $\Omega$ is said to
\underline{vanish at infinity}, if
\[ \lim_{x \to \infty} f(x) = 0, \]
that is, if and only if for given $\varepsilon > 0$, there exists a
neighbourhood $U$ of $\infty$ in $K$ such that $| \,f \,| < \varepsilon$
on $U \setminus \{ \infty \} \subset \Omega$.

Equivalently: for given $\varepsilon > 0$, there should exist a compact
subset $C$ of $\Omega$ such that $| \,f \,| < \varepsilon$ outside $C$.

(Indeed, the open neighbourhoods of $\infty$ in $K$ are of the form
$K \setminus C$ with $C$ any compact subset of $\Omega$, cf.\ the
preceding reminder \ref{onepcomp}.)

It is immediate that if we extend $f$ to $K$ by putting $f(\infty) := 0$,
then $f$ will vanish at infinity if and only if the extended function will
be continuous at $\infty$. (Please note that this continuity condition
is for the point $\infty$ alone.)

That is, we may identify the set of complex-valued functions on $\Omega$
vanishing at infinity with the set of complex-valued functions $g$ on $K$
which are continuous at $\infty$ and satisfy $g(\infty) = 0$.

We define
\[ \cont\0(\Omega) := \{ \,f \in \cont (\Omega) : f \text{ vanishes at infinity} \,\}. \]
Under the above identification, $\cont\0(\Omega)$ will correspond to
\[ \{ \,g \in \cont (K) : g(\infty) = 0 \,\}. \]
(As a function is continuous if and only if it is continuous at each point.)
As a consequence, functions in $\cont\0(\Omega)$ are bounded, so
$\cont\0(\Omega)$ may be considered as a \st-subalgebra of
$\cont_\mathrm{b}(\Omega)$, and will be equipped with the supremum
norm inherited from $\cont_\mathrm{b}(\Omega)$. The supremum norms
in $\cont_\mathrm{b}(\Omega)$ and in $\cont (K)$ coincide on
$\cont\0(\Omega)$:
\[ \sup_{x \in \Omega} | \,g(t) \,| = \sup_{y \in K} | \,g(y) \,|
\quad \text{for } g \in \cont (K) \text{ with } g(\infty) = 0. \]
It follows that $\cont\0(\Omega)$, when identified with
$\{ \,g \in \cont (K) : g(\infty) = 0 \,\}$, then will be a closed,
hence complete, subspace of $\cont (K)$, and thus
a C*-subalgebra of $\cont (K)$.
\end{observation}

To summarise, we note: \pagebreak

\begin{definition}[$\cont\0(\Omega)$]\label{Cnaught}%
\index{symbols}{C3@$\cont\0(\Omega)$|)}%
Let $\Omega \neq \varnothing$ be a locally compact Hausdorff space.
Then $\cont\0(\Omega)$ is defined as
\[ \{ \,f \in \cont (\Omega) :
 \ \forall \ \varepsilon > 0 \ \exists \text{ compact } C \subset \Omega
 \text{ with } | \,f \,| < \varepsilon \text{ outside } C \,\}, \]
or equivalently
\[ \{ \,f \in \cont (\Omega) : \text{for all } \varepsilon > 0 \text{ the set }
  \{ \,x \in \Omega : | \,f(x) \,| \geq \varepsilon \,\} \text{ is compact} \,\}. \]
Please note that if $\Omega$ is compact, then
$\cont\0(\Omega) = \cont (\Omega)$. Also note that if $\Omega$ is
not compact, this definition is the same as in the preceding
observation \ref{motCnaught}.

The \st-algebra $\cont\0(\Omega)$ will be equipped with the
supremum norm.

If $\Omega$ is not compact, let $K$ be the one-point compactification
of $\Omega$, and let $\infty \in K \setminus \Omega$ be the corresponding
point at infinity. We may then identify $\cont\0(\Omega)$ with the
C*-subalgebra $\{ \,g \in \cont (K) : g(\infty) = 0 \,\}$ of $\cont (K)$.

Thus $\cont\0(\Omega)$ is a C*-subalgebra of $\cont_\mathrm{b}(\Omega)$,
namely the \underline{C*-algebra of} \underline{continuous functions
vanishing at infinity}.
\end{definition}

The C*-algebras $\cont\0(\Omega)$, with $\Omega \neq \varnothing$
a locally compact Hausdorff space, are the prototypes of commutative
C*-algebras $\neq \{ 0 \}$, see the Commutative Gel'fand-Na\u{\i}mark
Theorem \ref{commGN} below.)

\begin{proposition}\label{nounit}%
Let $\Omega \neq \varnothing$ be a locally compact Hausdorff space.
If $\Omega$ is not compact, then $\cont\0(\Omega)$ contains no unit.
\end{proposition}

\begin{proof}
Any unit would need to be $1_\Omega$, but $1_\Omega$ does not
vanish at infinity if $\Omega$ is not compact. Indeed, in this case,
the set $\{ \,x \in \Omega : \,1_\Omega (x) \geq \frac12 \,\} = \Omega$
is not compact.
\end{proof}

\begin{proposition}\label{necStW}%
If $\Omega \neq \varnothing$ is a locally compact Hausdorff space,
then $\cont\0(\Omega)$ separates the points of $\Omega$ and vanishes
nowhere on $\Omega$.
\end{proposition}

\begin{proof}
We can assume that $\Omega$ is not compact, cf.\ \ref{KnecStW} above.
Let $K$ be the one-point compactification of $\Omega$, and let
$\infty \in K \setminus \Omega$ be the corresponding point at infinity.
Consider $x \in \Omega$. There exists $f \in \cont (K)$ with $f(\infty) = 0$
and $f(x) = 1$. (By observation \ref{Urysimple} above.) The restriction $g$
of $f$ to $\Omega$ then is in $\cont\0(\Omega)$, and and satisfies $g(x) = 1$.
Thus $\cont\0(\Omega)$ vanishes nowhere on $\Omega$. If $y \in \Omega$
is distinct from $x$, there exists $h \in \cont (K)$ with $h(x) = 1$, $h(y) = 0$.
(By observation \ref{Urysimple} again.) The restriction $k$ of $fh$ to
$\Omega$ then is in $\cont\0(\Omega)$, and $k(x) = 1$, $k(y) = 0$. Thus
$\cont\0(\Omega)$ separates the points of $\Omega$. \pagebreak
\end{proof}

\begin{proposition}\label{exdivzero}%
Let $\Omega$ be a locally compact Hausdorff space containing at least
two distinct points. Then $\cont\0(\Omega)$ contains divisors of zero,
i.e.\ non-zero elements $f, g$ with $f g = 0$.
\end{proposition}

\begin{proof}
We can assume that $\Omega$ is not compact, cf.\ \ref{Kdivzero} above.
One then applies Urysohn's Lemma \ref{Urysohn} to the one-point
compactification of $\Omega$, proceeding as in \ref{Kdivzero}, including
the point at infinity in the closed sets on which functions are to vanish.
\end{proof}

\medskip
We shall next give a definition of vector sublattices of
$\mathds{R}^{\,\text{\Small{$\Omega$}}}$, then characterise them,
cf.\ \ref{lattabs}. One point is, that they are useful.

\begin{definition}[lattice]\index{concepts}{lattice}\label{latticedef}%
An ordered set is called a \underline{lattice}, if with any two elements,
it contains a least upper bound (or supremum), as well as
a greatest lower bound (or infimum), of those two elements.
\end{definition}

\begin{definition}[the lattice operations $\vee$ and $\wedge$]%
\index{concepts}{lattice!operations@operations $\vee$ $\wedge$}%
\index{symbols}{11@$\vee$}\index{symbols}{12@$\wedge$}%
\index{symbols}{f01@$f \vee g$}%
\index{symbols}{f02@$f \wedge g$}%
If $\mathscr{L}$ is a lattice, and if $f, g \in \mathscr{L}$, then
the supremum of $f$ and $g$ is denoted by $f \vee g$, and
the infimum of $f$ and $g$ is denoted by $f \wedge g$.
\end{definition}

\begin{definition}[ordered vector space]%
\index{concepts}{ordered vector space}%
\index{concepts}{vector!space!ordered}%
An \underline{ordered vector space} is a real vector space $\mathscr{L}$
equipped with an order relation $\leq$ such that
\begin{alignat*}{3}
    a \leq b & \Rightarrow a+c \leq b+c
    \quad &\text{for all} \quad &a, b, c \in \mathscr{L}, \\
   a \leq b & \Rightarrow \lambda a \leq \lambda b
   \quad &\text{for all} \quad &a, b \in \mathscr{L},\ \lambda \in [ \,0, \infty \,[.
\end{alignat*}
\end{definition}

\begin{definition}[vector lattice]%
\index{concepts}{lattice!vector}\index{concepts}{vector!lattice}%
A \underline{vector lattice} is an ordered \linebreak vector space
that is a lattice.
\end{definition}

\begin{example}[$\mathds{R}^{\,\text{\Small{$\Omega$}}}$]%
\index{symbols}{R25@$\mathds{R}^{\,\text{\Small{$\Omega$}}}$}%
Let $\Omega$ be a non-empty set. The set
$\mathds{R}^{\,\text{\Small{$\Omega$}}}$ of real-valued functions on
$\Omega$ shall be equipped with the pointwise vector space operations
and the pointwise order given by
\[ f \leq g \Leftrightarrow f(x) \leq g(x) \quad \text{for all} \quad x \in \Omega \]
if $f, g \in \mathds{R}^{\,\text{\Small{$\Omega$}}}$.
Then $\mathds{R}^{\,\text{\Small{$\Omega$}}}$ is a \underline{vector lattice},
with the lattice operations
\[ ( \,f \vee g \,) \,(x) = \max \,\{ \,f(x), g(x) \,\}, \quad
 ( \,f \wedge g \,) \,(x) = \min \,\{ \,f(x), g(x) \,\} \]
for all $x \in \Omega$ and all $f, g \in \mathds{R}^{\,\text{\Small{$\Omega$}}}$.

The \underline{notation} stems from the fact that for $A, B \subset \Omega$,
one then has
\[ 1_A \vee 1_B = 1_{A \cup B}, \qquad 1_A \wedge 1_B = 1_{A \cap B}.
\pagebreak \]
\end{example}

\begin{definition}[sublattice]\index{concepts}{sublattice}%
\index{concepts}{lattice!sublattice}%
A \underline{sublattice} of a lattice $\mathscr{L}$ is a subset containing
with two elements also their supremum and infimum in $\mathscr{L}$.
\end{definition}

\begin{definition}[vector sublattice]%
\index{concepts}{sublattice!vector}%
\index{concepts}{vector!lattice!sublattice}%
A \underline{vector sublattice} of a vector lattice is a vector subspace
that also is a sublattice.
\end{definition}

\begin{example}%
Obviously ${\ell}^{\,\infty}(\Omega)\sa$ is a vector sublattice of
$\mathds{R}^{\,\text{\Small{$\Omega$}}}$ for any non-empty set $\Omega$.
\end{example}

\begin{theorem}\label{lattabs}%
Let $\Omega$ be a non-empty set. A vector subspace
of $\mathds{R}^{\,\text{\Small{$\Omega$}}}$ is a vector sublattice
of $\mathds{R}^{\,\text{\Small{$\Omega$}}}$ if and only if it contains
with each function $f$ also the function
$| \,f \,|\vphantom{\mathds{R}^{\,\text{\Small{$\Omega$}}}}$.
(The pointwise absolute value of $f$.)
$\vphantom{\mathds{R}^{\,\text{\Small{$\Omega$}}}}$
\end{theorem}

\begin{proof}
Indeed, for $f, g, h \in \mathds{R}^{\,\text{\Small{$\Omega$}}}$, one has
\begin{gather*}
| \,h \,| = (\mspace{1mu}-h\mspace{2mu}) \vee h, \\
f \vee g = \frac{f + g}{2} + \frac{| \,f - g \,|}{2}, \\
f \wedge g = \frac{f + g}{2} - \frac{| \,f - g \,|}{2}.
\end{gather*}
In order to see the last two equations, note that $\frac{f(x) + g(x)}{2}$
is the midpoint of the interval spanned by the points $f(x)$ and $g(x)$,
and that $\frac{| \,f(x) - g(x) \,|}{2}$ is the distance from the midpoint
to the endpoints of this interval.$\vphantom{\frac{|1|}{2}}$
\end{proof}

\medskip
The continuity of the function $z \mapsto | \,z \,|$ on $\mathds{R}$ then implies:

\begin{corollary}\label{naughtlatt}%
Let $\Omega \neq \varnothing$ be a Hausdorff space. Then
$\cont (\Omega)\sa$ is a vector sublattice of
$\mathds{R}^{\,\text{\Small{$\Omega$}}}$. Hence
$\cont_\mathrm{b}(\Omega)\sa$ is a vector sublattice of
both $\cont (\Omega)\sa$ and ${\ell}^{\,\infty}(\Omega)\sa$.
If $\Omega \neq \varnothing$ furthermore is a locally compact
Hausdorff space, then $\cont\0(\Omega)\sa$ is a vector sublattice
of $\cont_\mathrm{b}(\Omega)\sa$.
\end{corollary}

We would like to establish the following.

\begin{observation}\label{finitesup}%
\index{symbols}{11@$\vee$}\index{symbols}{12@$\wedge$}%
\index{symbols}{13@$\vee_{i \in I}f_i$}%
In a lattice, any non-empty finite subset has a supremum. This
supremum can be formed iteratively, involving only two elements
at each step. We may thus write $\vee _{i \in I} f_i$ for the supremum
of a finite non-empty family $(f_i)_{i \in I}$ in a lattice. In case of a
sublattice of $\mathds{R}^{\,\text{\Small{$\Omega$}}}$, with
$\Omega$ a non-empty set, we then have
\[ ( \,\vee _{i \in I} f_i \,) \,(x) = \max \,\{ \,f_i(x) \in \mathds{R} : i \in I \,\}
\quad \text{for all} \quad x \in \Omega. \pagebreak \]
\end{observation}

The proof is by induction using the following proposition.

\begin{proposition}%
Let $\mathscr{L}$ be a lattice. Let $F$ be a subset of $\mathscr{L}$
which has a supremum $f$ in $\mathscr{L}$. For any $g \in \mathscr{L}$,
the element $f \vee g$ then is the supremum in $\mathscr{L}$ of
the set $F \cup \{ g \}$.
\end{proposition}

\begin{proof}
Obviously $f \vee g$ is an upper bound of $F \cup \{ g \}$.
Any other upper bound $h$ of $F \cup \{ g \}$ is an upper bound
of $F$, and thus greater than or equal to the least upper bound
$f$ of $F$. Therefore $h \geq f \vee g$, which shows that
$f \vee g$ is the least upper bound of $F \cup \{ g \}$.
\end{proof}

\medskip
Similarly for the infimum of any non-zero finite number of elements of a lattice.

\bigskip
Our next object is a pre-C*-algebra of practical importance.

\begin{reminder}[$\cont_\mathrm{c}(\Omega)$]%
\index{symbols}{s35@$\mathrm{supp}(f)$}%
\index{symbols}{C4@$\cont_\mathrm{c}(\Omega)$}%
\index{concepts}{support!of a function}%
Let $\Omega \neq \varnothing$ be a locally compact Hausdorff space.
The \underline{support} of a complex-valued function $f$ on $\Omega$
is defined to be the closed set
\[ \mathrm{supp}(f) = \overline{\{ \,x \in \Omega: f(x) \neq 0 \,\}}. \]
One denotes by $\cont_\mathrm{c}(\Omega)$ the set of functions in
$\cont (\Omega)$ of compact support.
\end{reminder}

\begin{observation}%
Let $\Omega \neq \varnothing$ be a locally compact Hausdorff space.
Then $\cont_\mathrm{c}(\Omega)$ is a \st-subalgebra of $\cont\0 (\Omega)$.
Also, $\cont_\mathrm{c}(\Omega)\sa$ is a vector sublattice of
$\cont\0 (\Omega)\sa$.
\end{observation}

\begin{proposition}\label{Ccdense}%
Let $\Omega \neq \varnothing$ be a locally compact Hausdorff space.
Then $\cont_\mathrm{c}(\Omega)$ is dense in $\cont\0(\Omega)$.
\end{proposition}

\begin{proof}
We note that $\cont\0(\Omega)$ is spanned by its non-negative
functions, because $\cont\0(\Omega)\sa$ is a vector sublattice of
$\mathds{R}^{\,\text{\Small{$\Omega$}}}$, cf.\ \ref{naughtlatt} above.
For $f \in \cont\0(\Omega)$ with
$f \geq 0$ pointwise, and $\varepsilon > 0$, the function
\[ g := ( \,f - \varepsilon \,1_\Omega \,) \vee 0 \]
is in $\cont_\mathrm{c}(\Omega)$, and we have
\[ | \,f - g \,|_{\,\infty} \leq \varepsilon. \]
Indeed, the support of $g$ is contained in the closed set
\[ \{ \,x \in \Omega : f(x) \geq \varepsilon \,\}
 = \{ \,x \in \Omega : | \,f(x) \,| \geq \varepsilon \,\}, \]
which is compact by \ref{Cnaught}. \pagebreak
\end{proof}

\clearpage

%


\section{The Stone-Weierstrass Theorem}%
\label{St-WstThm}

The objective of this paragraph is theorem \ref{StW} below.

\medskip
We shall give a completely elementary proof of the
following result, which is sometimes useful.

\begin{theorem}\label{sqrtone}%
The square root function $t \mapsto \sqrt{t}$ on $[\,0,1\,]$
can be approximated uniformly by a sequence of rational
polynomials without constant term.

Such a sequence can be defined iteratively as
\begin{align*}
p\0 (x) & := 0, \\
p_{n+1} (x) & := p_n (x) + \frac12 \bigl( \,x - {p_n (x)}^{\,2\,}\bigr)
\qquad (\,n \geq 0\,).
\end{align*}
Indeed, one then has
\[ 0 \leq \sqrt{t} - p_n (t) < \frac2n
\quad \text{for all} \quad t \in [\,0,1\,],\ n \geq 1. \]
Furthermore, the sequence $(p_n)$ is monotone increasing
(and therefore non-negative) on $[\,0,1\,]$.
\end{theorem}

Before going to the proof, we note the following corollary.

\begin{corollary}\label{sqrta}%
The square root function $t \mapsto \sqrt{t}$ on a compact interval
$[\,0, 2^{\,2k}\,]$, with $k$ a non-negative integer, can be approximated
uniformly by a sequence of rational polynomials without constant term.

Such a sequence can be chosen as
\[ q_n (x) := 2^{\,k} \ p_n \,\bigl( \,\frac{x}{2^{\,2k}} \,\bigr), \]
with $(p_n)$ as in the preceding theorem \ref{sqrtone}.
Indeed, one then has
\[ 0 \leq \sqrt{t} - q_n (t) < \frac{2^{\,k+1}}{n}
\quad \text{for all} \quad t \in [\,0, 2^{\,2k}\,],\ n \geq 1. \]
Furthermore, the sequence $(q_n)$ is monotone increasing
(and therefore non-negative) on $[\,0, 2^{\,2k}\,]$.
\end{corollary}

\begin{proof}
For $t \in [\,0, 2^{\,2k}\,]$ and every integer $n \geq 0$, we have
\[ \sqrt{t} - q_n (t) = 2^{\,k} \ \Bigl( \,\sqrt{\frac{t}{2^{\,2k}}}
\,- \,p_n \,\bigl( \,\frac{t}{2^{\,2k}} \,\bigr) \,\Bigr) \pagebreak \qedhere \]
\end{proof}

\begin{proof}[\hspace{-3.8ex}\mdseries{\scshape{Proof of \protect\ref{sqrtone}}}]%
Please note first that for any $t \in [\,0,1\,]$ and any integer $n \geq 0$, we have
\begin{align*}
\sqrt{t} - p_{n+1} (t)
& = \bigl[ \,\sqrt{t} -p_n (t)\,\bigr]
- \frac12 \,\bigl[ \,\sqrt{t} - p_n (t) \,\bigr] \cdot \bigl[ \,\sqrt{t} + p_n (t) \,\bigr] \\
 & = \bigl[ \,\sqrt{t} -p_n (t)\,\bigr] \cdot \biggl[ \,1
- \frac{\sqrt{t} + p_n (t)}{2} \,\biggr]. \tag*{$(*)$}
\end{align*}
We shall next show by induction that for any $t \in [\,0,1\,]$
and any integer $n \geq 0$, one has
\begin{gather*}
p_n (t) \in \bigl[ \,0,\sqrt{t} \,\bigr] \tag*{$(i)$} \\
\sqrt{t} - p_n (t) \leq \frac{2 \,\sqrt{t}}{2+n\,\sqrt{t}}. \tag*{$(ii)$}
\end{gather*}
Clearly these statements hold for all $t \in [\,0,1\,]$ if $n = 0$.
Assume now that these statements hold for some integer $n \geq 0$
and some $t \in [\,0,1\,]$. Then the second factor of
the member on the right hand side of equation $(*)$ satisfies
\[ 0 \leq 1 - \frac{\sqrt{t} + p_n (t)}{2} \leq 1, \tag*{$(**)$} \]
as is seen from $(i)$. From $(*)$ and $(i)$, we then get that
\[ 0 \leq \sqrt{t} - p_{n+1} (t) \leq \sqrt{t} - p_{n} (t), \]
which is equivalent to
\[ p_n (t) \leq p_{n+1} (t) \leq \sqrt{t}. \]
Hence also $p_{n+1} (t) \in \bigl[\,0,\sqrt{t}\,\bigr]$.
It also follows already that the sequence $(p_n)$ is monotone
increasing on $[\,0,1\,]$. By $(*)$, $(**)$, $(i)$, and $(ii)$, we have
\begin{align*}
\sqrt{t} - p_{n+1} (t)
 & = \bigl[ \,\sqrt{t} -p_n (t)\,\bigr] \cdot \biggl[ \,1
- \frac{\sqrt{t} + p_n (t)}{2} \,\biggr] \\
 & \leq \frac{2 \,\sqrt{t}}{2+n\,\sqrt{t}} \cdot \bigl[ \,1 - \frac12 \,\sqrt{t} \,\bigr]
 = \sqrt{t} \cdot \frac{2 - \sqrt{t}}{2 + n \,\sqrt{t}}.
\end{align*}
Now it can be directly verified that
\[ \frac{2 - \sqrt{t}}{2 + n \,\sqrt{t}} \leq \frac{2}{2 + (n+1) \,\sqrt{t}}, \]
from which we get
\[ \sqrt{t} - p_{n+1} (t) \leq \frac{2 \,\sqrt{t}}{2 + (n+1) \,\sqrt{t}}. \]
This finishes our induction proof. The statement follows easily. \pagebreak
\end{proof}

\medskip
The above fact constitutes one of the two main ingredients
of the Stone-Weierstrass Theorem, through the following
direct consequences.

\begin{lemma}\label{sqrtexists}%
Let $\Omega$ be a non-empty set. Let $A$ be a closed subalgebra
of $\ell^{\,\infty}(\Omega)$. Then $A$ contains with each function $f$
such that $f \geq 0$ pointwise, also the function $\sqrt{f}$.
(The pointwise square root of $f$.)
\end{lemma}

\begin{proof}
Let $f \in A$ with $f \geq 0$ pointwise.
The preceding corollary \ref{sqrta} with $2^{\,2k} \geq | \,f \,|_{\,\infty}$
yields a sequence of functions $q_n (f) \in A$ converging
uniformly on $\Omega$ to $\sqrt{f}$. (This uses that the polynomials
$q_n$ have no constant term, as we do not suppose that $A$ is unital.)
Since $A$ is uniformly closed, it follows that $\sqrt{f} \in A$.
\end{proof}

\medskip
We shall later need the following fact.

\begin{proposition}\label{existsqrt}%
Let $\Omega$ be a non-empty set.
Let $A$ be any \linebreak C*-subalgebra of $\ell^{\,\infty}(\Omega)$.
Then $A$ contains with each function $f$
also the functions $| \,f \,|$ as well as $\sqrt{| \,f \,|}$.
(Here $| \,f \,|$ and $\sqrt{| \,f \,|}$ are defined pointwise.)
\end{proposition}

\begin{proof}
This follows from the preceding lemma \ref{sqrtexists} and the fact
that $| \,f \,| = \sqrt{{f}^{\,*} f}$.
\end{proof}

\begin{theorem}\label{Cstarlattice}%
Let $\Omega$ be a non-empty set.
If $A$ is a C*-subalgebra of $\ell^{\,\infty}(\Omega)$, then
$A\sa$ is a vector sublattice of ${\ell}^{\,\infty}(\Omega)\sa$.
\end{theorem}

\begin{proof}
This follows from theorem \ref{lattabs} and proposition \ref{existsqrt}.
\end{proof}

\medskip
The preceding theorem is an important strengthening of
corollary \ref{naughtlatt} above.

It is good to be conscious of the following fact, because it
contains some philosophy of harmonic analysis.

\begin{corollary}[from \st-algebras to lattices]\label{philo}%
Let $\Omega$ be a non-empty set. Let $A$ be a \st-subalgebra
of $\ell^{\,\infty}(\Omega)$. Let $B$ denote the closure of $A$
in $\ell^{\,\infty}(\Omega)$. Then $B\sa$ is a vector sublattice
of ${\ell}^{\,\infty}(\Omega)\sa$.
\end{corollary}

The point is that integration theory can be applied to
vector sublattices of $\mathds{R}^{\,\text{\Small{$\Omega$}}}$,
see e.g.\ \cite{FW}. \pagebreak

The second main ingredient of the Stone-Weierstrass Theorem
is the following surprising fact, which leads to the
Kakutani-Kre\u{\i}n Theorem \ref{Kaku} below. Please note that
if $\Omega \neq \varnothing$ is a locally compact Hausdorff
space, then $\cont\0(\Omega)\sa$ is a vector sublattice of
$\mathds{R}^{\,\text{\Small{$\Omega$}}}$, cf.\ \ref{naughtlatt}.

\begin{lemma}\label{preKaku}%
Let $K \neq \varnothing$ be a compact Hausdorff space, and
consider a sublattice $\mathscr{L}$ of $\cont (K)\sa$. Let
$f \in \cont (K)\sa$, and $\varepsilon > 0$. Assume that for any
$x, y \in K$, there exists a function $g \in \mathscr{L}$ such that
\[ | \,f(z) - g(z) \,| < \varepsilon \quad \text{for} \quad z \in \{ x, y \}. \]
There then exists a function $h \in \mathscr{L}$ satisfying
\[ | \,f - h \,| _{\,\infty} < \varepsilon. \]
\end{lemma}

\begin{proof}
The assumption says that for any $x, y \in K$, the set
\[ G\subxy := \bigl\{ \,g \in \mathscr{L} : | \,f(z) - g(z) \,| < \varepsilon
\text{ for } z \in \{ x, y \} \,\bigr\} \]
is non-empty. Therefore, for any $x, y \in K$ and for any
$g \in G\subxy$, the points $x, y$ both belong to the open sets
\begin{align*}
U\subg  & := \{ \,z \in K : f(z) < g(z) + \varepsilon \,\}, \\
V\subg  & := \{ \,z \in K : g(z) - \varepsilon < f(z) \,\}.
\end{align*}
Thus, for fixed $x \in K$, the open sets $U\subg$ $(g \in G\subxy, y \in K)$
cover $K$. By compactness of $K$, there exists a finite non-empty
family $\{ y_n \}_{n = 1}^m$ in $K$, as well as a finite non-empty family
$\{ g_n \}_{n=1}^m$ of functions $g_n \in G\subxyn$, such that the
corresponding sets $U\subgn$ $(1 \leq n \leq m)$ cover $K$. The function
\[ h\subx := \vee _{\,n=1}^{\,m} \ g_n  \in \mathscr{L}, \]
cf.\ \ref{finitesup}, then satisfies
\[ f (z) < h\subx (z) + \varepsilon \quad \text{for all} \quad z \in K. \]
Consider the finite intersection
\[ W\subx := \cap_{\,n=1}^{\,m} \,V\subgn , \]
which is an open neighbourhood of $x$. We then also have
\[ h\subx (z) - \varepsilon < f (z) \quad \text{for all} \quad z \in W\subx. \]
Letting now $x$ vary over $K$, we may extract a finite non-empty subset
$F$ of $K$ such that the corresponding sets $W\subx$ $(x \in F)$
cover $K$. Putting
\[ h := \wedge _{\textit{\Small{$\,x \in F$}}} \ h\subx \in \mathscr{L}, \]
we get
\[ h(z) - \varepsilon <  f(z) < h(z) + \varepsilon
\quad \text{for all} \quad z \in K. \pagebreak \]
Since the continuous function $| \,f - h \,|$ assumes its
supremum on the compact set $K$, it follows that
$| \,f - h \,| _{\,\infty} < \varepsilon$.
\end{proof}

\begin{observation}
The preceding lemma \ref{preKaku} in particular implies that
if $\Omega \neq \varnothing$ is a locally compact Hausdorff
space, then a sublattice of $\cont\0(\Omega)\sa$, which is
dense in $\cont\0(\Omega)\sa$ with respect to the topology of
pointwise convergence on $\Omega$, actually is uniformly
dense in $\cont\0(\Omega)\sa$. (Use the one-point
compactification if $\Omega$ is not compact,
cf.\ \ref{Cnaught}.)
\end{observation}

The requirements of lemma \ref{preKaku} above are easily met:

\begin{observation}\label{hypo1}%
Let $\Omega$ be a non-empty set.
Let $\mathscr{L}$ be a vector subspace of
$\mathds{R}^{\,\text{\Small{$\Omega$}}}$
separating the points of $\Omega$ and vanishing
nowhere on $\Omega$. Assume that for every two
points $x, y \in \Omega$ with $x \neq y$, there
exists a function $h \in \mathscr{L}$ with $h(x) = h(y) = 1$.
Then for any function $f \in \mathds{R}^{\,\text{\Small{$\Omega$}}}$,
and every two points $x, y \in \Omega$, there exists
a function $g \in \mathscr{L}$ with $g(x) = f(x)$, and $g(y) = f(y)$.
\end{observation}

\begin{proof}
Let $f \in \mathds{R}^{\,\text{\Small{$\Omega$}}}$ and $x, y \in \Omega$.
In order to exhibit $g$ as in the statement, we can assume that $x \neq y$,
as $\mathscr{L}$ vanishes nowhere on $\Omega$. There then exists
$h \in \mathscr{L}$ with $h(x) = h(y) = 1$. It is enough to produce some
$g \subx \in \mathscr{L}$ with $g \subx (x) = 1$ and $g \subx (y) = 0$.
(Indeed, with $g \suby := h - g \subx \in \mathscr{L}$, we then have
$g \suby (x) = 0$ and $g \suby (y) = 1$,
so $g := f(x) g \subx + f(y) g \suby \in \mathscr{L}$ solves the problem.)
There exists $k \in \mathscr{L}$ with $k(x) \neq k(y)$ as $\mathscr{L}$
separates the points of $\Omega$. Then
\[ g \subx := \frac{k - k(y) h}{k(x) - k(y)} \in \mathscr{L} \]
satisfies the requirement that $g \subx (x) = 1$ and $g \subx (y) = 0$.
\end{proof}

\begin{theorem}[Kakutani-Kre\u{\i}n]\label{Kaku}%
\index{concepts}{Theorem!Kakutani@Kakutani-Kre\u{\i}n}%
\index{concepts}{Kakutani}\index{concepts}{Krein@Kre\u{\i}n}%
Let $\Omega \neq \varnothing$ be a locally compact Hausdorff space.
Let $\mathscr{L}$ be a vector sublattice of $\cont\0(\Omega)\sa$
separating the points of $\Omega$ and vanishing nowhere on
$\Omega$. Assume that for every two points $x, y \in \Omega$
with $x \neq y$, there exists a function $h \in \mathscr{L}$ with
$h(x) = h(y) = 1$. Then $\mathscr{L}$ is dense in $\cont\0(\Omega)\sa$.
\end{theorem}

\begin{proof}
This follows from the results \ref{preKaku} \& \ref{hypo1} in case
$\Omega$ is compact. If $\Omega$ is not compact, use the
one-point compactification, cf.\ \ref{Cnaught}.
\end{proof}

\medskip
Also the hypotheses of the Kakutani-Kre\u{\i}n Theorem
\ref{Kaku} are easily satisfied:

\begin{observation}\label{hypo2}%
Let $\Omega$ be a non-empty set, and let $A$ be a \linebreak
\st-subalgebra of $\mathds{C}^{\,\text{\Small{$\Omega$}}}$
vanishing nowhere on $\Omega$. Then for any two points
$x, y \in \Omega$, there exists a function $h \in A\sa$ with
$h(x) = h(y) = 1\vphantom{\mathds{C}^{\,\text{\Small{$\Omega$}}}}$.
\end{observation}

\begin{proof}
For $x, y \in \Omega$, there exist functions $f, g \in A$ with $f(x) = 1$
and $g(y) = 1$ as $A$ vanishes nowhere on $\Omega$.
The function $k := f + g - fg \in A$ satisfies $k(x) = k(y) = 1$,
so we can choose $h := k^*k \in A\sa$.
\end{proof}

\begin{theorem}[the Stone-Weierstrass Theorem]%
\index{concepts}{Theorem!Stone-Weierstrass|(}\label{StW}%
\index{concepts}{Stone-Weierstrass Theorem|(}%
Let $\Omega \neq \varnothing$ be a locally compact Hausdorff
space. Any \st-subalgebra of $\cont\0(\Omega)$, which separates
the points of $\Omega$, and which vanishes nowhere
on $\Omega$, is dense in $\cont\0(\Omega)$.
\end{theorem}

\begin{proof}
Let $A$ be a \st-subalgebra of $\cont\0(\Omega)$, which separates
the points of $\Omega$, and which vanishes nowhere on $\Omega$.
Let $B$ denote the closure of $A$ in $\cont\0(\Omega)$.
The main idea is to apply the Kakutani-Kre\u{\i}n Theorem
\ref{Kaku} with $\mathscr{L} := B\sa$, which is a vector sublattice
of $\cont\0(\Omega)\sa$, by corollary \ref{philo}. Please note that
since $A$ separates the points of $\Omega$, so does $A\sa$, and
thus $\mathscr{L}$, as is seen by an indirect argument. Also, since
$A$ vanishes nowhere on $\Omega$, so does $A\sa$, and thus
$\mathscr{L}$, as is easily seen. It follows by observation \ref{hypo2}
that $\mathscr{L}$ satisfies the hypotheses of the Kakutani-Kre\u{\i}n
Theorem \ref{Kaku}. Thus $\mathscr{L} = \cont\0(\Omega)\sa$ because
$\mathscr{L}$ is closed in $\cont\0(\Omega)\sa$. This implies that
$B = \cont\0(\Omega)$.
\end{proof}

\medskip
Please note that we need a \st-subalgebra, not merely a
subalgebra!

More precisely, the assumption of the Stone-Weierstrass
Theorem that $A$ be a \st-subalgebra cannot be replaced
by the weaker assumption that $A$ merely be a subalgebra.
Indeed, consider for example the disc algebra, which is
defined as the algebra of continuous complex-valued
functions on the closed unit disc $\mathds{D}$, which are
analytic in the interior of $\mathds{D}$. The disc algebra
clearly vanishes nowhere on $\mathds{D}$, and the identical
function separates the points of $\mathds{D}$. Also,
the disc algebra is a closed subalgebra of $\cont (\mathds{D})$.
However it is not all of $\cont (\mathds{D})$, because the
continuous function $z \mapsto \overline{z}$ is not analytic
at any point, as is well-known. Indeed, closed subalgebras
of $\cont (K)$, separating the points of a compact Hausdorff
space $K \neq \varnothing$ and containing the constants,
are studied under the name ``uniform algebras''.

Also, among the assumptions of the Stone-Weierstrass
Theorem, neither the assumption that $A$ should separate
the points, nor the assumption that $A$ should vanish nowhere,
can be dropped. This can easily be inferred from \ref{necStW}.
\pagebreak

\begin{corollary}[extended Stone-Weierstrass Theorem]%
\index{concepts}{Theorem!Stone-Weierstrass|)}\label{StWcomp}%
\index{concepts}{Stone-Weierstrass Theorem|)}%
Let $\Omega$ be a locally compact Hausdorff space $\neq \varnothing$.
Let $A$ be a \st-subalgebra of $\cont\0(\Omega)$ separating the points
of $\Omega$. Let $B$ be the closure of $A$ in $\cont\0(\Omega)$.
Then either $A$ vanishes nowhere on $\Omega$, in which case
$B = \cont\0(\Omega)$, \linebreak or else there exists a point $x$
of $\Omega$ with $B = \{ \,f \in \cont\0(\Omega) : f(x) = 0 \,\}$.
\end{corollary}

\begin{proof}
If $A$ vanishes nowhere on $K$, then $B = \cont\0(\Omega)$
by the preceding Stone-Weierstrass Theorem \ref{StW}. Assume
next that $A$ does vanish somewhere on $K$, which means that
there exists a point $x \in K$ with $f(x) = 0$ for all $f \in A$. Then
$\Omega' := \Omega \setminus \{ x \}$ is a locally compact
Hausdorff space, and $\cont\0(\Omega')$ consists precisely
of the restrictions to $\Omega'$ of the functions in
$\{ \,f \in \cont\0(\Omega) : f(x) = 0 \,\}$, as is easily seen. The
restrictions of the functions in $A$ to $\Omega'$ form a \st-subalgebra
$C$ of $\cont\0(\Omega')$ which separates the points of
$\Omega'$ and vanishes nowhere on $\Omega'$ because $A$
separates the points of $\Omega$ and already vanishes at the
point $x$. It follows from the preceding Stone-Weierstrass Theorem
\ref{StW} that $C$ is dense in $\cont\0(\Omega')$, which implies that
$B = \{ \,f \in \cont\0(\Omega) : f(x) = 0 \,\}$.
\end{proof}

\clearpage

%


\chapter{Basic Properties of the Spectrum}


\setcounter{section}{5}

\section{The Spectrum}

\medskip
In this paragraph, let $A$ be an algebra.

\bigskip
We start with a few remarks concerning invertible elements.%
\index{concepts}{invertible}

\begin{proposition}\label{leftrightinv}%
If $a \in \tld{A}$ has a left inverse $b$ and a right
inverse $c$, then $b = c$ is an inverse of $a$.
\end{proposition}

\begin{proof}
$b = be = b(ac) = (ba)c = ec = c.$
\end{proof}

\medskip
In particular, an element of $\tld{A}$ can have at most one inverse.
If $a,b$ are invertible elements of $\tld{A}$, then $ab$ is
invertible with inverse $b^{-1}a^{-1}$. It follows that the
invertible elements of $\tld{A}$ form a group under multiplication.

\begin{proposition}
Let $a,b \in \tld{A}$. If both of their products $ab$ and $ba$
are invertible, then both $a$ and $b$ are invertible.
\end{proposition}

\begin{proof}
With $c := {(ab)}^{-1}$, $d := {(ba)}^{-1}$, we get
\begin{align*}
 & e = (ab)c = a(bc)   &  & e = c(ab) = (ca)b \\
 & e = d(ba) = (db)a  &  & e = (ba)d = b(ad).  \\
\intertext{The preceding proposition \ref{leftrightinv} implies that now}
 & bc = db = a^{-1}    &  \text{ and \quad \qquad \qquad }  & ca = ad = b^{-1}.
\qedhere
\end{align*}
\end{proof}

\medskip
This result is usually used in the following form:

\begin{corollary}\label{comm1}%
If $a,b$ are \underline{commuting} elements of $\tld{A}$,
and if $ab$ is invertible, then both $a$ and $b$ are invertible.
\pagebreak
\end{corollary}

\begin{proposition}\label{comm2}%
Let $a$ be an invertible element of $\tld{A}$. An element of $\tld{A}$
commutes with $a$ if and only if it commutes with $a^{-1}$.
\end{proposition}

\begin{proof}
For $b \in \tld{A}$ we get:
$ab = ba \ \Leftrightarrow \ b = a^{-1}ba
\ \Leftrightarrow \ ba^{-1} = a^{-1}b$.
\end{proof}

\begin{proposition}\label{Hermlrinv}%
A Hermitian element of a unital \st-algebra is left invertible
if and only if it is right invertible. In this event, the element is
invertible, by \ref{leftrightinv}. See also \ref{normalleftright} below.
\end{proposition}

\begin{proof} Apply the involution. \end{proof}

\begin{definition}[the spectrum, $\s_A(a)$]%
\index{symbols}{s2@$\protect\s(a)$}%
\index{concepts}{spectrum!of an element}%
For $a \in A$ one defines
\[ \s_A(a) := \{ \lambda \in \mathds{C} :
\lambda e - a \in \tld{A} \text{ is not invertible in } \tld{A} \,\}. \]
One says that $\s_A(a)$ is the \underline{spectrum} of the element
$a$ in the algebra $A$. We shall often abbreviate
\[ \s(a) := \s_A(a). \]
\end{definition}

The next result will be used tacitly in the sequel.

\begin{proposition}\label{specunit}%
For $a \in A$ we have
\[ \s_{\tld{A}} (a) = \s_A (a). \]
\end{proposition}

\begin{proof}
The following statements for $\lambda \in \mathds{C}$ are equivalent.
\begin{align*}
 & \ \lambda \in \s_{\tld{A}} (a), \\
 & \ \lambda e - a  \in \tilde{\tilde{A}}
\text{ is not invertible in } \tilde{\tilde{A}}, \\
 & \ \lambda e - a \in \tld{A}
\text{ is not invertible in } \tld{A} \qquad
\text{$\bigl(\mspace{1mu}$because $\tilde{\tilde{A}} = \tld{A}\,\bigr)$}, \\
 & \ \vphantom{\tilde{\tilde{A}}}\lambda \in \s_A (a). \qedhere
\end{align*}
\end{proof}

\begin{proposition}\label{speczero}%
If $A$ has no unit, then $0 \in \s_A(a)$ for all $a \in A$.
\end{proposition}

\begin{proof}
Assume that $0 \notin \s_{A}(a)$ for some $a$ in
$A$. It shall be shown that $A$ is unital. First, $a$
is invertible in $\tld{A}$. So let for example
$a^{-1} = \mu e + b$ with $\mu \in \mathds{C}$, $b \in A$.
We obtain $ e = ( \mu e + b ) a = \mu a + b a \in A$,
so $A$ is unital.
\end{proof}

\medskip
The following Rational Spectral Mapping Theorem
is one of the most widely used results in spectral theory.
\pagebreak

\begin{theorem}[the Rational Spectral Mapping Theorem]%
\index{concepts}{Theorem!Rational Spectral Mapping}%
\index{concepts}{Rational Spec.\ Mapp.\ Thm.}\label{ratspecmapthm}%
Let $a \in \tld{A}$ and let $r(x)$ be a non-constant rational
function without pole on $\s(a)$. We then have
\[ \s\bigl(r(a)\bigr) = r\bigl(\s(a)\bigr). \]
\end{theorem}

\begin{proof}
We may write uniquely
\[ r(x) = \gamma \;\prod  _{i\in I} \;( \,\alpha _i - x \,)
\;\prod _{j\in J} \;{( \,\beta _j - x \,)}^{-1} \]
where $\gamma \in \mathds{C}$ and
$\{ \alpha _i : i \in I \}$, $\{\beta_j : j \in J \}$
are disjoint sets of complex numbers.
The $\beta_j$ $(j \in J) $ are the poles of $r(x)$.
The element $r(a)$ in $\tld{A}$ then is defined by
\[  r(a) := \gamma \;\prod _{i\in I} \;( \,\alpha_i \,e - a \,)
\;\prod _{j\in J} \;{( \,\beta _j \,e  -a \,)}^{-1}. \]
Let now $\lambda$ in $\mathds{C}$ be fixed. The function
$\lambda - r(x)$ has the same poles as $r(x)$, occurring with
the same multiplicities. We may thus write uniquely
\[ \lambda - r(x) = \gamma ( \lambda )
\prod _{k \in K( \lambda )} \;\bigl( \,\delta _k (\lambda ) - x \,\bigr)
\;\prod _{j \in J} \;{( \,\beta _j - x \,)}^{-1}, \]
where $ \gamma ( \lambda ) \neq 0 $ by the assumption that
$r(x)$ is not constant. Hence also
\[ \lambda e - r(a) = \gamma ( \lambda )
\prod _{k \in K( \lambda )} \;\bigl( \,\delta_k ( \lambda ) \,e - a \,\bigr)
\;\prod _{j \in J} \;{( \,\beta_j \,e - a \,)}^{-1}. \]
By \ref{comm1} \& \ref{comm2} it follows that the following
statements are equivalent.
\begin{align*}
 & \lambda \in \s\bigl(r(a)\bigr), \\
 & \lambda e - r(a)
\text{ is not invertible in } \tld{A}, \\
 & \text{there exists } k
\text{ such that } \delta_k( \lambda ) \,e - a
\text{ is not invertible in } \tld{A}, \\
 & \text{there exists } k
\text{ such that } \delta_k ( \lambda ) \in \s(a), \\
 & \lambda - r(x)
\text{ vanishes at some point } x\0 \in \s(a), \\
 & \text{there exists } x\0 \in \s(a)
\text{ such that } \lambda = r(x\0), \\
 & \lambda \in r\bigl(\s(a)\bigr). \qedhere
\end{align*}
\end{proof}

The next result is used extremely often. It will be turned into
a continuity statement later on. See \ref{contraux} below. \pagebreak

\begin{theorem}\label{spechom}%
Let $B$ be another algebra, and let $\pi : A \to B$ be an algebra
homomorphism. We then have
\[ \s_B \bigl(\pi (a)\bigr) \setminus \{0\} \subset \s_A (a) \setminus \{0\}
\quad \text{for all} \quad a \in A. \]
If both $A$ and $B$ are unital algebras, and if $\pi$ is unital as well, then
\[ \s_B \bigl(\pi (a)\bigr) \subset \s_A (a)
\quad \text{for all} \quad a \in A. \]
\end{theorem}

\begin{proof}
Let $a \in A$ and $\lambda \in \mathds{C} \setminus \{0\}$. Assume
that $\lambda \notin \s_{A} (a)$. We will show that
$\lambda \notin \s_{B} \bigl( \pi (a)\bigr)$. Let $e_A$ be the unit in
$\tld{A}$ and $e_B$ the unit in $\tld{B}$. Extend $\pi$ to an algebra
homomorphism $\tld{\pi} : \tld{A} \to \tld{B}$ by requiring
that $\tld{\pi}(e_A) = e_B$ if $\vphantom{\tld{A}}A$ has no unit.
The element $\lambda e_A-a$ has an inverse $b$ in $\tld{A}$.
Then $\tld{\pi} (b) + \lambda ^{-1} \bigl(e_B- \tld{\pi} (e_A)\bigr)$ is the
inverse of $\lambda e_B- \pi (a)$ in $\tld{B}$. Indeed we have for example
\begin{align*}
   & \ \bigl[ \,\tld{\pi} (b) + \lambda^{-1} \bigl( e_B- \tld{\pi} (e_A) \bigr) \,\bigr]
       \,\bigl[ \,\lambda e_B- \pi (a) \,\bigr] \\
 = & \ \bigl[ \,\tld{\pi} (b) + \lambda ^{-1} \bigl( e_B - \tld{\pi} (e_A) \bigr) \,\bigr]
     \,\bigl[ \,\tld{\pi} ( \lambda e_A - a )
     + \lambda \bigl( e_B - \tld{\pi} (e_A) \bigr) \,\bigr] \\
 = & \ \tld{\pi} (e_A) + 0 + 0 + {\bigl( e_B - \tld{\pi} (e_A) \bigr)}^{2} \\
 = & \ \tld{\pi} (e_A) + \bigl(e_B - \tld{\pi} (e_A)\bigr) \\
 = & \ e_B.
\end{align*}
Assume now that $A$, $B$ and $\pi$ are unital, and consider the case
$\lambda = 0$. If $0 \notin \s_{A} (a)$, then $a$ is invertible with inverse
$a^{-1}$, say. But then $\pi (a)$ will be invertible with inverse $\pi (a^{-1})$,
by the unital nature of $\pi$. Hence $0 \notin \s_{B} (\pi (a))$.
\end{proof}

\medskip
For the next result, recall \ref{twiddleunital} -- \ref{twidunitp}.

\begin{theorem}\label{specsubalg}%
If $B$ is a subalgebra of $A$, then
\[ \s_{A}(b) \setminus \{0\} \subset \s_{B}(b) \setminus\{0\}
 \quad \text{for all} \quad b \in B. \]
If $B$ furthermore is \twiddle-unital in $A$, then
\[ \s_{A}(b) \subset \s_{B}(b)
 \quad \text{for all} \quad b \in B. \]
\end{theorem}

\begin{proof}
This follows from the preceding theorem \ref{spechom} by considering the
canonical imbedding of $\tld{A}$ into $\tld{B}$,
cf.\ \ref{twiddleunital} -- \ref{twidunitp}, together with \ref{specunit}.
\end{proof}

\medskip
See also \ref{spbdry} -- \ref{spconvend} and
\ref{specHermincl} -- \ref{specHermend} below.
\pagebreak

We conclude this paragraph with a miscellany of results.

\begin{proposition}\label{speccomm}%
For $a, b \in A$, we have
\[ \s(ab) \setminus \{0\} = \s(ba) \setminus \{0\}. \]
\end{proposition}

\begin{proof}
Let $\lambda \in \mathds{C} \setminus \{0\}$. Assume
that $\lambda \notin \s(ab)$. There then exists $c \in \tld{A}$ with
\[ c(\lambda e - ab) = (\lambda e - ab)c = e.  \]
We need to show that $\lambda \notin \s(ba)$. We claim that
$\lambda ^{-1}(e+bca)$ is the inverse of $\lambda e-ba$, i.e.
\[ (e+bca)(\lambda e - ba) = (\lambda e - ba)(e+bca) = \lambda e. \]
Indeed, one calculates
\begin{align*}
 &\ (e+bca)(\lambda e - ba) &\quad &\ (\lambda e - ba)(e+bca) \\
= &\ \lambda e-ba +\lambda bca-bcaba & = &\ \lambda e-ba+\lambda bca-babca \\
= &\ \lambda e-ba+bc(\lambda e-ab)a & = &\ \lambda e-ba+b(\lambda e-ab)ca \\
= &\ \lambda e-ba+ba = \lambda e & = &\ \lambda e-ba+ba = \lambda e. \qedhere
\end{align*}
\end{proof}

\begin{theorem}\label{spleftreg}%
Consider the left regular representation $L : a \mapsto L\tla$ $(a \in \tld{A})$
of $\tld{A}$ on itself, cf.\ \ref{leftreg}. The spectrum of an element $a$ in
$\tld{A}$ coincides with the spectrum of the translation operator $L\tla$
in $\mathrm{End}(\tld{A})$. That is,
\[ \s_{\tld{A}} (a) = \s_{\mathrm{End}(\tld{A})} (L\tla)
\quad \text{for all} \quad a \in \tld{A}. \]
\end{theorem}

\begin{proof}
Please note that $L_{\textstyle{e}}$ is the unit operator
$\mathds{1} \in \mathrm{End}(\tld{A})$. Let $a \in \tld{A}$.
We have to show that $a$ is invertible in $\tld{A}$ if and
only if $L\tla$ is invertible in $\mathrm{End}(\tld{A})$.
Clearly, if $a$ is invertible in $\tld{A}$ (with inverse $b$, say),
then $L\tla$ is invertible in $\mathrm{End}(\tld{A})$ (with
inverse $L_{\textstyle{b}}$). Conversely, assume that
$L\tla$ is invertible in $\mathrm{End}(\tld{A})$, with inverse
$T \in \mathrm{End}(\tld{A})$, say. Put $b := Te \in \tld{A}$.
Then $b$ is a right inverse of $a$ because
\[ a b = L\tla T e = \mathds{1} e = e. \]
It follows that $L_{\textstyle{b}}$ is a right inverse of $L\tla$.
Since $L\tla$ also has a left inverse in $\mathrm{End}(\tld{A})$,
it follows that $L_{\textstyle{b}}$ is an inverse of $L\tla$,
cf.\ \ref{leftrightinv}. Since $L$ is injective, by \ref{leftreginj},
we have that $b$ is an inverse of $a$ in $\tld{A}$.
\end{proof}

\begin{corollary}\label{spleftregbded}%
If $A$ is a normed algebra, then the left regular representation
$L$ of $\tld{A}$ on itself may be considered as an algebra
homomorphism from $\tld{A}$ to the algebra $\blop(\tld{A})$
of bounded linear operators on $\tld{A}$, cf.\ \ref{algbdedop}.
We then have
\[ \s_{\tld{A}} (a) = \s_{\blop(\tld{A})} (L\tla)
 \quad \text{for all} \quad a \in \tld{A}. \]
\end{corollary}

\begin{proof}
Let $a \in \tld{A}$. Since the left regular representation $L$
of $\tld{A}$ on itself is an algebra isomorphism from $\tld{A}$
onto the range $L(\tld{A})$ of $L$, cf.\ \ref{leftreginj}, we have that
\[ \s_{L(\tld{A})} (L\tla) = \s_{\tld{A}} (a). \]
From the preceding theorem \ref{spleftreg} we see that
\[ \s_{\tld{A}} (a) = \s_{\mathrm{End}(\tld{A})} (L\tla). \]
Applying theorem \ref{specsubalg} to the fact that 
$L(\tld{A}) \subset \blop(\tld{A}) \subset \mathrm{End}(\tld{A})$,
we get
\[ \s_{\mathrm{End}(\tld{A})} (L\tla) \subset \s_{\blop(\tld{A})} (L\tla)
 \subset \s_{L(\tld{A})} (L\tla) = \s_{\tld{A}} (a) = \s_{\mathrm{End}(\tld{A})} (L\tla).
 \qedhere \]
\end{proof}

\medskip
We see that in the above, instead of $\blop(\tld{A})$, we can take any
algebra $B$ such that $L(\tld{A}) \subset B \subset \mathrm{End}(\tld{A})$.

\begin{proposition}\index{concepts}{spectrum!of a bounded operator}%
If $V$ is a complex \underline{Banach} space, and if $a$ is a bounded
linear operator on $V$, then
\[ \s_{\blop(V)} (a) = \s_{\mathrm{End}(V)} (a). \]
\end{proposition}

\begin{proof}
A bounded linear operator on a Banach space $V$ is invertible in
$\blop(V)$ if and only if it is bijective, by the Open Mapping Theorem.
\end{proof}

\begin{proposition}\label{specstar}%
If $A$ is a \st-algebra, then we have
\[ \s(a^*) = \overline{\s(a)} \quad \text{for all} \quad a \in A. \]
\end{proposition}

\begin{proof}
This follows from $(\lambda e - a)^* = \overline{\lambda}e - a^*$.
\end{proof}

\medskip
It follows that the spectrum of a Hermitian element is symmetric
with respect to the real axis. It does not follow that the spectrum
of a Hermitian element is real. In fact, a \st-algebra, in which each
Hermitian element has real spectrum, is called a Hermitian \st-algebra,
cf.\ \ref{secHerm}. \pagebreak

\clearpage


\section[The Spectral Radius Formula: $\protect{\mathrm{r}_\lambda}$]%
{The Spectral Radius Formula: \texorpdfstring{$\protect{\rlambda}$}{r\80\137\83\273}}

\fancyhead[RO]{\SMALL{\S\ 7. \ THE SPECTRAL RADIUS FORMULA: $\mathrm{r}_\lambda$}}

\medskip
We remind of the following for the method.

\begin{reminder}\label{exponential}%
For $0 < \gamma < \infty$, we have
\[ {\gamma}^{\,1/n} = \exp \,\bigl( \,(1/n) \cdot \ln \gamma \,\bigr) \ \to \ 1
\qquad (n \to \infty). \]
\end{reminder}

\begin{lemma}\label{limit}%
Consider a sequence $(\alpha_n)_{n \geq 1}$ of
non-negative numbers satisfying
\[ \alpha_{n+m} \leq \alpha_n \alpha_m
\quad \text{for all integers} \quad n, m \geq 1. \]
The sequence $({\alpha_n}^{\,1/n}\,)_{n \geq 1}$
then converges to its infimum:
\[ \inf _{n \geq 1} {\alpha_n}^{\,1/n}
= \lim _{n \to \infty} {\alpha_n}^{\,1/n}. \]
\end{lemma}

\begin{proof}
Let $\varepsilon > 0$ and choose an integer $k \geq 1$ such that
\[ {\alpha_k}^{\,1/k} \leq \inf _{m \geq 1} {\alpha_m}^{\,1/m} + \varepsilon / 2. \]
Put
\[ \beta := \inf _{m \geq 1} {\alpha_m}^{\,1/m} + \varepsilon / 2 > 0, \qquad
  \gamma := 1 + \max _{1 \leq m \leq k} \,\alpha_m > 0. \]
Every integer $n \geq 1$ can be written uniquely as
$n = p(n)k + q(n)$ with $p(n), q(n)$ non-negative integers
and $1 \leq q(n) \leq k$. Since for $n \to \infty$, we have
\[ q(n)/n \to 0, \]
it follows that
\[ p(n)k/n \to 1, \]
and so
\begin{align*}
{\alpha_n}^{\,1/n} & = {\alpha_{p(n)k+q(n)}}^{\,1/n} \\
 & \leq {( {\alpha_k}^{\,p(n)} \alpha_{q(n)} )}^{\,1/n} \\
 & \leq {(\beta^{\,p(n)k} \gamma)}^{\,1/n}
      = \beta^{\,p(n)k/n} \gamma^{\,1/n} \to \beta
      = \inf _{m \geq 1} {\alpha_m}^{\,1/m} + \varepsilon / 2.
\end{align*}
Therefore
\[ {\alpha_n}^{\,1/n} < \inf _{m \geq 1} {\alpha_m}^{\,1/m} + \varepsilon \]
for all sufficiently large values of $n$. \pagebreak
\end{proof}

\begin{definition}[$\rlambda(a)$]%
\index{symbols}{r1@$\protect\rlambda(a)$}\label{rldef}%
For an element $a$ of some normed algebra, we define
\[ \rlambda(a) := \inf _{n \geq 1} {\bigl|\,{a}^{\,n}\,\bigr|}^{\,1/n} \leq |\,a\,|. \]
The subscript $\lambda$ is intended to remind of the notation for eigenvalues.
The reason for this will become apparent in \ref{specradform} below.
\end{definition}

\begin{theorem}\label{rlamlim}%
For an element $a$ of some normed algebra, we have
\[ \rlambda(a) = \lim _{n \to \infty} {\bigl|\,{a}^{\,n}\,\bigr| \,}^{1/n}. \]
\end{theorem}

\begin{proof}
Apply the above lemma \ref{limit} with $\alpha_n := | \,{a}^{\,n} \,|$ $(n \geq 1)$.
\end{proof}

\medskip
We shall now need some complex analysis. We refer the reader to
R.\ Remmert \cite{Rem}, or to W.\ Rudin \cite{RCAC}, or to P.\ Henrici \cite{Hen}.
We warn the reader that for us, a power series is an expression of the
form $\sum _{n=0}^{\infty } z^{n} a_{n}$, where the $a_n$ are elements
of a Banach space, and where $z$ is a placeholder for a complex variable.

\begin{lemma}[radius of convergence]\label{geoseries}%
\index{concepts}{convergence radius}\index{concepts}{geometric series}%
\index{concepts}{radius of convergence}\index{concepts}{series!geometric}%
Let $A$ be a Banach algebra, and let $a \in A$.
Consider the geometric series in $\tld{A}$ given by
\[ F(z) := \sum _{n=0} ^{\infty} \,z^{\,n} \,a^{\,n}, \]
where we put $z^{\,0} := 1$, and $a^{\,0} := e \in \tld{A}$.

The geometric series $F(z)$ converges absolutely for every
$z \in \mathds{C}$ with $| \,z \,| < 1 \,/ \,\rlambda(a)$, and it does not
converge for any $z \in \mathds{C}$ with $| \,z \,| > 1 \,/ \,\rlambda(a)$.
These two facts are expressed by saying that the
\underline{radius of convergence} of the geometric
series $F(z)$ is $1 \,/ \,\rlambda(a)$.

The sum of $F(z)$ is an inverse of $e-za$ whenever the series
$F(z)$ converges.
\end{lemma}

\begin{proof}
Put $\rho := 1 \,/ \,\rlambda(a)$. Please note that $\rho \in \ ] \,0, \infty \,]$
and that
\[ \rho^{\,-1} = \lim _{n \to \infty} {\bigl| \,a^{\,n} \,\bigr|}^{\,1/n} \]
by theorem \ref{rlamlim} above. Let $z \in \mathds{C}$ with
$| \,z \,| < \rho$. (We mention here that $\rho \neq 0$ as noted above.)
We shall prove that $F(z)$ converges absolutely. Let
$\sigma \in \mathds{R}$ with $| \,z \,| < \sigma < \rho$. One then has
\[ \lim _{n \to \infty} {\bigl| \,a^{\,n} \,\bigr|}^{\,1/n} = \rho^{\,-1} < \sigma^{\,-1}, \]
so
\[ | \,a^{\,n} \,| \leq \sigma^{\,-n} \quad \text{for all integers} \quad n \geq k \]
for some integer $k \geq 1$ large enough. Using $| \,z \,| < \sigma$, we get
\begin{align*}
\sum _{n=k} ^{\infty} \,| \,z^{\,n} \,a^{\,n} \,|
 & = \,\sum _{n=k} ^{\infty} \,{| \,z \,| }^{\,n} \,| \,a^{\,n} \,| \\
 & \leq \,\sum _{n=k} ^{\infty} \,{| \,z \,|}^{\,n} \,\sigma^{\,-n}
    \leq \,\sum _{n=0} ^{\infty} \,{\bigl( \,{| \,z \,|} \,/ \,{\sigma} \,\bigr)}^{\,n}
    = \,\frac{1}{1 - \bigl ( \,{| \,z \,|} \,/ \,{\sigma} \,\bigr)},
\end{align*}
so $F(z)$ converges absolutely indeed.

Now let $z \in \mathds{C}$ with $| \,z \,| > \rho$. (Assuming that $\rho < \infty$.)
We shall show that $F(z)$ does not converge. With $\sigma := | \,z \,|$, we have
\[ \lim _{n \to \infty} {\bigl| \,a^{\,n} \,\bigr|}^{\,1/n} = \rho^{\,-1} > \sigma^{\,-1}, \]
so
\[ | \,a^{\,n} \,| \geq \sigma^{\,-n} \quad \text{for all integers} \quad n \geq k \]
for some integer $k \geq 1$ large enough. The partial sums for $F(z)$ then
do not form a Cauchy sequence. Indeed, for any integer $m \geq k$, we get
\begin{align*}
\bigl| \,\sum _{n=0} ^{m} \,z^{\,n} \,a^{\,n}
- \,\sum _{n=0} ^{m-1} \,z^{\,n} \,a^{\,n} \,\bigr|
& = | \,z^{\,m} \,a^{\,m} \,| \\
& = {| \,z \,|}^{\,m} \,| \,a^{\,m} \,| \geq { \bigl( \,| \,z \,| \,/ \,\sigma \,\bigr)}^{\,m} = 1.
\end{align*}
If the series $F(z)$ converges, we find by continuity of the multiplication
\begin{align*}
 F(z) \,(e-za) = (e-za) \,F(z) & =
 \,\sum _{n=0}^{\infty} \,z^{\,n} \,a ^{\,n}
 - \sum _{n=0}^{\infty}\,z^{\,n+1} \,a ^{\,n+1} \\
 & = \,\sum _{n=0}^{\infty} \,z^{\,n} \,a ^{\,n}
 - \sum _{n=1}^{\infty}\,z^{\,n} \,a ^{\,n} = e. \qedhere
\end{align*}
\end{proof}

\begin{lemma}\label{leftinv}%
Let $A$ be a unital Banach algebra.
Let $c$ be a left invertible element of $A$, with left inverse $c^{\,-1}$.
Let $a \in A$ be arbitrary.

For $z \in \mathds{C}$ with
$| \,z \,| < 1 \,/ \,\rlambda \bigl( \,a{c \,}^{-1} \,\bigr)$, we have that
$c-za$ is left invertible with left inverse ${c \,}^{-1} \,G(z)$, with
the geometric series
\[ G(z) := \,\sum _{n=0} ^{\infty} \,z^{\,n\,} {\bigl( \,a{c} ^{\,-1} \,\bigr)}^{\,n}. \]

Thus, if $| \,a \,| < 1 \,/ \,\bigl| \,{c \,}^{-1} \,\bigr|$, then the element
$c-a$ has a left inverse. This shows that the set of left invertible
elements of $A$ is open. \pagebreak
\end{lemma}

\begin{proof}
By the preceding lemma \ref{geoseries},
we have $G(z) \,(e-zac^{\,-1}) = e$
for $z \in \mathds{C}$ with
$| \,z \,| < 1 \,/ \,\rlambda \bigl( \,a{c \,}^{-1} \,\bigr)$.
For such $z$, using that $c^{\,-1}$ is a left inverse of $c$,
we then find $G(z) \,(c-za) = c$, and ${c \,}^{-1} \,G(z) \,(c-za) = e$.
Thus ${c \,}^{-1} \,G(z)$ is a left inverse of $c-za$ indeed.
\end{proof}

\begin{theorem}\label{twosidinv}%
Let $c$ be an invertible element of a unital Banach algebra $A$.
Let $a \in A$ be arbitrary.

For $z \in \mathds{C}$ with
$| \,z \,| < 1 \,/ \,\rlambda \bigl( \,a{c \,}^{-1} \,\bigr)$,
we have that $c-za$ is invertible with inverse
\[ c^{\,-1} \,\sum _{n=0} ^{\infty} \,z^{\,n\,} {\bigl( \,a{c} ^{\,-1} \,\bigr)}^{\,n}. \]
\end{theorem}

\begin{proof}
For $z \in \mathds{C}$ with
$| \,z \,| < 1 \,/ \,\rlambda \bigl( \,a{c \,}^{-1} \,\bigr)$,
we find by lemma \ref{geoseries}
\[ (c-za) \ c^{\,-1} \,\sum _{n=0} ^{\infty} \,z^{\,n\,} {\bigl( \,a{c} ^{\,-1} \,\bigr)}^{\,n}
= (e - zac^{\,-1}) \,\sum _{n=0} ^{\infty} \,z^{\,n\,} {\bigl( \,a{c} ^{\,-1} \,\bigr)}^{\,n} = e, \]
so the statement follows with the preceding lemma \ref{leftinv}.
\end{proof}

\begin{theorem}\label{GL(A)}%
The group of invertible elements in a unital  Banach algebra
$A$ is open and inversion is homeomorphic on this group.
\end{theorem}

\begin{proof}
Let $c$ be an invertible element of a unital Banach algebra $A$. For
$a \in A$ with $| \,a \,| < 1 \,/ \,\bigl| \,{c \,}^{-1} \,\bigr|$, we have from
the preceding theorem \ref{twosidinv} that $c-a$ is invertible with inverse
\[ c^{\,-1} \,\sum _{n=0}^{\infty } \,{\bigl( \,a{c}^{\,-1} \,\bigr)}^{\,n}, \]
so the set of invertible elements of $A$ is open. We then also have
\begin{align*}
\Bigl| \,{( c-a ) \,}^{-1} - {c \,}^{-1} \,\Bigr|
& = \biggl| \ {c \,}^{-1} \,\biggl( \,\sum _{n=0}^{\infty }
\,{\bigl( a{c \,}^{-1} \bigr) \,}^{n} -e \,\biggr) \,\biggr| \\
 & \leq \bigl| \ {c \,}^{-1} \,\bigr| \cdot
\biggl| \ \sum _{n=1}^{\infty } \,{\bigl( a{c \,}^{-1} \bigr) \,}^{n} \,\biggr| \\
 & \leq \bigl| \ {c \,}^{-1} \,\bigr| \cdot \bigl| \,a{c \,}^{-1} \,\bigr|
\cdot \sum _{n=0}^{\infty } \,{\bigl| \,a{c \,}^{-1} \,\bigr| \,} ^{n} \\
 & \leq | \,a \,| \cdot {\bigl| \ {c \,}^{-1}\,\bigr| \,}^2
\cdot \Bigl(\,1 - |\,a\,| \cdot \bigl| \,{c \,}^{-1} \,\bigr| \,\Bigr)^{-1},
\end{align*}
which shows that inversion is continuous.
\end{proof}

The next result is basic for spectral theory. It allows for a subtle
interplay of algebraic and analytic properties. We shall use it tacitly.

\begin{theorem}[the Spectral Radius Formula]\label{specradform}%
\index{concepts}{Theorem!Spectral Radius Formula}%
\index{concepts}{spectral!radius formula}%
For an element $a$ of a Banach algebra, we have:
\begin{itemize}
 \item[$(i)$] $\s(a)$ is a \underline{non-empty compact} set in the complex plane,
\item[$(ii)$] $\max\,\{\,|\,\lambda \,| : \lambda \in \s(a) \,\} = \rlambda(a)$%
                      $\vphantom{\wht{b}}$\ (the Spectral Radius Formula).
\end{itemize}
Please note that the right side of the equation depends on the norm,
whereas the left side does not.
\end{theorem}

\begin{proof}
We consider the function $f$, (taking values in the unitisation of the
Banach algebra), given by $f( \mu ):= ( e  - \mu a )^{-1}$, with its largest
domain in $\mathds{C}$. It is an analytic function: write
$f( \mu + z ) = ( c - za )^{\,-1}$ with $c := e - \mu a$, and apply theorem
\ref{twosidinv} to the effect that $f$ has a power series expansion at
each point $\mu$ of its domain, given by
\[ f( \mu + z ) =
\,\sum _{n=0} ^{\infty} \,z^{\,n\,} c^{\,-1} \,{\bigl( \,a{c} ^{\,-1} \,\bigr)}^{\,n}
\qquad \Bigl( \:| \,z \,| < 1 \,/ \,\rlambda \bigl( \,a{c \,}^{-1} \,\bigr) \,\Bigr). \]
The function $f$ is defined at the origin and its power series expansion
at the origin is simply
\[ f(z) = \,\sum _{n=0} ^{\infty} \,z^{\,n} \,a^{\,n} \]
because $c = e$ in this case. The radius of convergence of this power
series is $1\,/ \,\rlambda(a)$, according to lemma \ref{geoseries}. Hence
$e - z a$ is invertible for $| \,z \,| < 1 \,/ \,\rlambda(a)$. That is,
$\lambda \notin \s(a)$ for $|\,\lambda\,|>\rlambda(a)$. In other words, $\s(a)$
is contained in the disc $\{ \,z \in \mathds{C} : | \,z \,| \leq \rlambda(a) \,\}$.
Theorem \ref{GL(A)} shows that $\s(a)$ is closed, and thus compact.
It remains to be shown that $\s(a)$ contains a number of modulus
$\rlambda(a)$.

The function $f$ has a singular point $z\0$ with modulus equal to the
radius of convergence of the power series expansion at the origin
(if this radius is finite). (See e.g.\ \cite[Item 5, p.\ 234]{Rem} or
\cite[Theorem 16.2, p.\ 320]{RCAC} or the proof of
\cite[Theorem 3.3a p.\ 155]{Hen}.) That is,
$|\,z\0\,| = 1\,/ \,\rlambda(a)$ if $\rlambda(a) \neq 0$. In this case
$e-z\0a$ cannot be invertible, and consequently $1 / z\0 \in \s(a)$.
The statement follows for the case $\rlambda(a) \neq 0$.

If $\rlambda(a)=0$, we have to show that $0 \in \s(a)$.
However, if $0 \notin \s(a)$, then $a$ is invertible, and
$e={a \,}^n \,{({a \,}^{-1}) \,}^n$ for all integers $n \geq 1$. Then
$0 < |\,e\,| \leq \bigl|\,{a \,}^n\,\bigr| \cdot {\bigl|\,{a \,}^{-1}\,\bigr| \,}^n$,
and so $\rlambda(a) \geq {\bigl|\,{a \,}^{-1}\,\bigr| \,}^{-1} > 0$
by \ref{exponential}.\pagebreak
\end{proof}

\clearpage


\section[Properties of $\protect{\mathrm{r}_\lambda}$]%
{Properties of \texorpdfstring{$\protect{\rlambda}$}{r\80\137\83\273}}

\fancyhead[RO]{\SMALL{\S\ 8. \ PROPERTIES OF $\mathrm{r}_\lambda$}}

\begin{theorem}\label{specmodul}%
An element of a \underline{normed} algebra has non-empty
\linebreak spectrum. Indeed, if $a$ is an element of a normed
algebra, then $\s(a)$ contains a number of modulus $\rlambda(a)$.
\end{theorem}

\begin{proof} The statement in the case of a Banach algebra
follows from the Spectral Radius Formula \ref{specradform}.
The general case follows by considering the completion $B$
of a normed algebra $A$, say. One uses the fact that
then $\s_B (a) \subset \s_A (a)$ by \ref{specsubalg}
as $A$ is \twiddle-unital in $B$ by \ref{twidunitcompl}.
\end{proof}

\begin{theorem}[Gel'fand, Mazur]\label{Gel'fandMazur}%
\index{concepts}{Theorem!Gel'fand-Mazur}%
\index{concepts}{Gel'fand-Mazur Theorem}%
Let $A$ be a unital normed algebra which is a division ring.
Then $A$ is isomorphic to $\mathds{C}$ as an algebra.
Indeed, the map
\[ \mathds{C} \to A,\ \lambda \mapsto \lambda e \]
then is an algebra isomorphism onto $A$ which also is a homeomorphism.
\end{theorem}

\begin{proof} To prove that this map is surjective, let $a \in A$.
There exists $\lambda \in \s(a)$ by the preceding theorem \ref{specmodul}.
Then $\lambda e-a$ is not invertible in $A$, so that $\lambda e-a = 0$
because $A$ is a division ring. That is, $a =\lambda e$, which
proves that this map is surjective. The map is injective and
homeomorphic because $|\,\lambda e\,| = |\,\lambda \,|\cdot|\,e\,|$
and $|\,e\,| \neq 0$.
\end{proof}

\begin{proposition}\label{rlpowers}%
For a normed algebra $A$ and an element $a \in A$, we have
\[ \rlambda \bigl( {a}^{\,k\,} \bigr) = {\rlambda (a )}^{\,k}
\quad \text{for all integers} \quad k \geq 1. \]
\end{proposition}

\begin{proof} One calculates
\[ \rlambda \bigl( {a}^{\,k\,} \bigr)
= \lim\limits_{n \to \infty} \,{\Bigl| \,{{ \bigl( {a}^{\,k\,} \bigr)}^{\,n}} \,\Bigr|}^{\,1/n}
= \lim\limits_{n \to \infty} \,{\left( \,{\left|\,{a}^{\,kn}\,\right|}^{\,1/kn} \,\right)}^{\,k}
= {\rlambda (a)}^{\,k}. \qedhere \]
\end{proof}

\begin{proposition}\label{rlcomm}%
For a normed algebra $A$ and elements $a, b \in A$, we have
\[ \rlambda ( a b ) = \rlambda ( b a ). \]
\end{proposition}

\begin{proof}
We express ${( a b ) \,}^{n+1}$ through ${( b a ) \,}^n$:
\begin{align*}
& \ \rlambda ( a b )
= \lim\limits_{n \to \infty} {\bigl| \,{{( a b ) \,}^{n+1}} \,\bigr| \,}^{1/(n+1)}
= \lim\limits_{n \to \infty} {\bigl| \,a \,{( b a ) \,}^n \,b \,\bigr| \,}^{1/(n+1)} \\
\leq & \ \lim\limits_{n \to \infty} {\bigl( \,{| \,a \,| \cdot | \,b \,|} \,\bigr) \,}^{1/(n+1)} \cdot
\lim\limits_{n \to \infty} {\Bigr( \,{\bigl| \,{{( b a ) \,}^n} \,\bigr| \,}^{1/n} \,\Bigr) \,}^{n/(n+1)}
\leq \rlambda ( b a ).
\end{align*}
This also follows from \ref{speccomm} \& the
Spectral Radius Formula \ref{specradform}.\pagebreak
\end{proof}

\begin{proposition}\label{commrlsub}%
If $a,b$ are \underline{commuting} elements of a normed
\linebreak algebra $A$, then
\begin{gather*}
\rlambda(ab) \leq \rlambda(a) \,\rlambda(b), \\
\rlambda(a+b) \leq \rlambda (a) + \rlambda(b),
\end{gather*}
whence also
\[ |\,\rlambda(a) - \rlambda(b)\,|
\leq \rlambda(a-b) \leq |\,a-b\,|. \]
Thus, if $A$ is commutative, then $\rlambda$ is uniformly continuous.
\end{proposition}

\begin{proof}
By commutativity, we have
\[ { \bigl| \,{( \,a \,b \,) \,}^n \,\bigr| \,}^{1/n} = { \bigl|\,{a \,}^n \,{b \,}^n\,\bigr| \,}^{1/n}
\leq { \bigl| \,{a \,}^n \,\bigr| \,}^{1/n} \cdot { \bigl|\,{b \,}^n \,\bigr| \,}^{1/n}, \]
and it remains to take the limit to obtain the first inequality. In order
to prove the second inequality, let $s, t$ with $\rlambda(a) < s$,
$\rlambda(b) < t$ be arbitrary and define $c := {s \,}^{-1} \,a$, $d := {t \,}^{-1} \,b$.
As then $\rlambda(c)$, $\rlambda(d) <1$, there exists
$0 < \alpha < \infty$ with $|\,{c \,}^n\,|$, $|\,{d \,}^n\,| \leq \alpha$ for all
$n \geq 1$. We then have
\[ \bigl| \,{( \,a+b \,) \,}^n \,\bigr| \leq \sum _{k=0} ^n
\,\bigl({\,^{n\,}_{k\,}}\bigr)
\,{s \,}^k \,{t \,}^{n-k} \,\bigl| \,{c \,}^k \,\bigr| \cdot \bigl| \,{d \,}^{n-k} \,\bigr|
\leq {\alpha \,}^2 { \,( \,s + t \,) \,}^n. \]
It follows that
\[ { \bigl| \,{( \,a+b \,) \,}^n \,\bigr| \,}^{1/n} \leq {\alpha \,}^{2/n} \,( \,s + t \,). \]
By passing to the limit, we get $\rlambda(a+b) \leq s+t$,
cf.\ \ref{exponential}. Since $s$, $t$ with $\rlambda (a) < s$
and $\rlambda (b) < t$ are arbitrary, it follows
$\rlambda (a + b) \leq \rlambda (a) + \rlambda (b)$.
\end{proof}

\medskip
See also \ref{commspeccont} and \ref{specsub} below.

\begin{proposition}\label{rlunit}%
Let $A$ be a Banach algebra without unit. For all
$\mu \in \mathds{C}$, $a \in A$, we then have
\[ \rlambda(\mu e + a) \leq | \,\mu \,| + \rlambda(a)
\leq 3 \,\rlambda(\mu e + a). \]
\end{proposition}

\begin{proof}
The first inequality follows from \ref{commrlsub}. For the same reason
$\rlambda (a) \leq \rlambda (\mu e + a) + | \,\mu \,|$, and so
$| \,\mu \,| + \rlambda (a) \leq \rlambda (\mu e + a) + 2 \,| \,\mu \,|$.
It now suffices to prove that $| \,\mu \,| \leq \rlambda (\mu e + a)$.
We have $0 \in \s (a)$ by \ref{speczero}, and so
$\mu \in \s (\mu e + a)$, which implies
$| \,\mu \,| \leq \rlambda (\mu e + a)$, as required. \pagebreak
\end{proof}

In this paragraph, merely the next result relates to an involution.

\begin{proposition}\label{Brlsymm}%
If $A$ is a \underline{Banach} \st-algebra, then
\[ \rlambda(a^*) = \rlambda(a) \quad \text{for all} \quad a \in A. \]
\end{proposition}

\begin{proof}
Use the Spectral Radius Formula and \ref{specstar} which says that
\[ \s(a^*) = \overline{\s(a)} \quad \text{for all} \quad a \in A. \qedhere \]
\end{proof}

\medskip
The above result is not true for general normed \st-algebras: \ref{counterexbis}.

\begin{theorem}\label{Bdistance}%
Let $A$ be a Banach algebra, and let $a \in A$.
For $\mu \in \mathds{C} \setminus \s(a)$, we have
\[ \mathrm{dist} \,\bigl( \,\mu, \,\s(a) \,\bigr) =
{\rlambda \bigl( \,{(\mu e - a) \,}^{-1} \,\bigr) \,}^{-1}, \]
cf.\ the appendix \ref{distance}.
\end{theorem}

\begin{proof}
By the Rational Spectral Mapping Theorem, we have
\[ \s \,\bigl( \,{(\mu e - a) \,}^{-1} \,\bigr) =
\bigl\{ \,{(\mu - \lambda) \,}^{-1} : \lambda \in \s(a) \,\bigr\}. \]
From the Spectral Radius Formula, it follows that
\begin{align*}
\rlambda \bigl( \,{(\mu e - a) \,}^{-1} \,\bigr)
& = \max \,\bigl\{ \,{| \,\mu - \lambda \,| \,}^{-1} : \lambda \in \s(a) \,\bigr\} \\
& = {\min \,\bigl\{ \,| \,\mu - \lambda \,| : \lambda \in \s(a) \,\bigr\} \,}^{-1} \\
& = {\mathrm{dist} \,\bigl( \,\mu, \,\s(a) \,\bigr) \,}^{-1}. \qedhere
\end{align*}
\end{proof}

\begin{theorem}\label{commspeccont}%
Let $A$ be a Banach algebra. Let $\mathds{D}$ denote the
closed unit disc. For two \underline{commuting} elements
$a, b \in A$, we have
\[ \s(a) \subset \s(b) + \rlambda (b-a) \,\mathds{D}
\subset \s(b) + | \,b-a \,| \,\mathds{D}. \]
If $A$ is commutative, one says that the spectrum function is
uniformly continuous on $A$. (By symmetry in $a$ and $b$.)
See also \ref{Hermspeccont} below.
\end{theorem}

\begin{proof}
Suppose the first inclusion is not true. There then exists \linebreak
$\mu \in \s(a)$ with $\mathrm{dist}\bigl(\mu, \s(b)\bigr) > \rlambda (b-a)$.
Note that then $\mu e - b$ is \linebreak invertible. By \ref{Bdistance}, we
have $\rlambda \bigl( \,{(\mu e - b)\,}^{-1} \,\bigr) \,\rlambda (b-a) < 1$. From
\ref{commrlsub}, we get $\rlambda \bigl( \,{(\mu e - b) \,}^{-1} \,(b-a) \,\bigr) < 1$.
We have $\mu e-a = (\mu e - b) + (b-a)$. This implies that also
$\mu e-a = (\mu e - b) \,\bigl[ \,e + {(\mu e - b) \,}^{-1} \,(b-a) \,\bigr]$ is invertible,
a contradiction.\vphantom{$)^{-1}$}
\end{proof}

\medskip
See also \ref{commrlsub} above and \ref{specsub} below. \pagebreak

\clearpage


\section[The Pt\'ak Function $\protect{\mathrm{r}_\sigma}$]%
{The Pt\texorpdfstring{\'a}{\80\341}k Function %
\texorpdfstring{$\protect{\rsigma}$}{r\80\137\83\303}}

\fancyhead[RO]{\SMALL{\S\ 9. \ THE PT\'AK FUNCTION $\mathrm{r}_\sigma$}}

\begin{definition}[$\rsigma(a)$]%
\index{symbols}{r2@$\protect\rsigma(a)$}\index{concepts}{Pt\'{a}k!function}%
For an element $a$ of a normed \st-algebra $A$, we define
\[ \rsigma (a) := {\rlambda (a^*a) \,}^{1/2}. \]
The subscript $\sigma$ is intended to remind of the notation for
singular values. The function $\rsigma : a \mapsto \rsigma(a)$ $(a \in A)$
is called the \underline{Pt\'{a}k function}.
\end{definition}

\begin{proposition}\label{rsC*}
For  a normed \st-algebra $A$, and $a \in A$, we have
\[ \rsigma(a^*a) = {\rsigma(a) \,}^2 \quad
\text{as well as} \quad \rsigma(a^*) = \rsigma(a). \]
\end{proposition}

\begin{proof}
With \ref{rlpowers} we find
\[ \rsigma(a^*a)
= {\rlambda \bigl( {\,(a^*a)}^*{(a^*a)} \,\bigr)}^{\,1/2}
= {\rlambda \bigl( {\,(a^*a)}^{\,2\,} \bigr)}^{\,1/2}
= \rlambda(a^*a)={\rsigma(a)}^{\,2}. \]
With \ref{rlcomm} we get
\[ \rsigma(a^*) = {\rlambda(aa^*)}^{\,1/2}
= {\rlambda(a^*a)}^{\,1/2} = \rsigma(a). \qedhere \]
\end{proof}

\medskip
The following result is groundlaying of course.

\begin{theorem}\label{C*rs}%
Let $\bigl( \,A , \| \cdot \| \,\bigr)$ be a normed \st-algebra.
If $\bigl( \,A , \| \cdot \| \,\bigr)$ is a pre-C*-algebra, then
\[ \| \,a \,\| = \rsigma(a) \quad \text{for all} \quad a \in A, \]
and conversely.
\end{theorem}

\begin{proof}
If $\| \,a \,\| = \rsigma(a)$ for all $a\in A$, then $\bigl( \,A , \| \cdot \| \,\bigr)$
is a pre-C*-algebra by the preceding proposition \ref{rsC*}.
Conversely, let $\bigl( \,A , \| \cdot \| \,\bigr)$ be a pre-C*-algebra.
We then have
\[ \bigl\| \,{b\,}^{( \,{2 \,}^n \,)} \,\bigr\| = { \,\| \,b \,\|\,} ^{( \,{2 \,}^n \,)}
\quad \text{for all } b \in A\sa, \text{ and all integers } n \geq 0, \]
which is seen by induction as follows. The statement holds for $n = 0$.
For an integer $n \geq 1$, we compute
\begin{align*}
\bigl\| \,{b\,}^{\,( \,{2 \,}^n \,)} \,\bigr\|
 & = \bigl\| \,{\bigl( \,{b}^* \,b \,\bigl) \,}^{( \,{2 \,}^{n-1} \,)} \,\bigr\| 
 \tag*{because $b$ is Hermitian} \\
 & = {\| \,{b}^* \,b \,\| \,} ^{( \,{2 \,}^{n-1} \,)}
 \tag*{by induction hypothesis} \\
 & = {\| \,b \,\|\,}^{( \,{2 \,}^n \,)} \pagebreak \tag*{by the C*-property \ref{preC*alg}.}
\end{align*}
For arbitrary $a \in A$, the element $a^*a$ is Hermitian, so we find
\begin{align*}
{\rsigma(a) \,}^2
 & = \lim _{n \to \infty} { \bigl\| \,{( \,a^*a \,) \,}^{( \,{2 \,}^n \,)} \,\bigr\|\,}^{1 \,/ \,( \,{2 \,}^n \,)} \\
 & = \lim _{n \to \infty} {\Bigl( \,{\| \,a^*a \,\|\,}^{( \,{2 \,}^n) \,} \,\Bigr) \,}^{1\,/ \,( \,{2 \,}^n \,)}
     = \| \,a^*a \,\| = {\| \,a \,\|\,}^2. \qedhere
\end{align*}
\end{proof}

\begin{proposition}\label{Hermrseqrl}%
If $A$ is a normed \st-algebra, and if $a \in A$ is \linebreak
\underline{Hermitian}, then
\[ \rsigma(a) = \rlambda(a). \]
\end{proposition}

\begin{proof}
By \ref{rlpowers} we have
\[ {\rsigma(a)}^{\,2} = \rlambda(a^*a)
= \rlambda \bigl(a^{\,2\,} \bigr) = {\rlambda(a)}^{\,2}. \qedhere \]
\end{proof}

\begin{proposition}\label{Bnormalrsleqrl}%
If $A$ is a \underline{Banach} \st-algebra $A$, and if $b \in A$
is \underline{normal}, then
\[ \rsigma(b) \leq \rlambda(b). \]
\end{proposition}

\begin{proof} By \ref{commrlsub} and \ref{Brlsymm} we have
\[ {\rsigma(b) \,}^2 = \rlambda(b^*b) \leq
\rlambda(b^*) \,\rlambda(b) = {\rlambda(b) \,}^2. \qedhere \]
\end{proof}

\medskip
The following enhancement of \ref{C*rs} above is essential for
the theory of commutative C*-algebras.

\begin{theorem}\label{C*rl}%
Let $\bigl( \,A , \| \cdot \| \,\bigr)$ be a pre-C*-algebra. For every
\underline{normal} element $b \in A$, we have
\[ \| \,b \,\| = \rlambda(b). \]
\end{theorem}

\begin{proof}
We can assume that $A$ is a C*-algebra. By \ref{C*rs} and
\ref{Bnormalrsleqrl}, we get
\[ \| \,b \,\| = \rsigma(b) \leq \rlambda(b) \leq \| \,b \,\|. \qedhere \]
\end{proof}

\begin{corollary}
Let $\bigl( \,A , \| \cdot \| \,\bigr)$be a pre-C*-algebra. For a
\underline{normal} element $b$ of $A$ and every integer
$n \geq 1$ we have
\[ \| \,{b\,}^n \,\| = {\| \,b \,\|\,}^n. \]
\end{corollary}

\begin{proof} This follows now from \ref{rlpowers}. \pagebreak \end{proof}

\begin{theorem}\label{C*unitis}%
\index{concepts}{norm}\index{concepts}{norm!unitisation}%
\index{concepts}{unitisation}\index{symbols}{A8@$\protect\tld{A}$}%
Let $\bigl( \,A , \| \cdot \| \,\bigr)$ be a C*-algebra without unit.
The unitisation $\tld{A}$ then carries besides its C*-algebra
norm $\| \cdot \| $, the complete norm
$| \,\lambda e+a \,| := | \,\lambda \,| + \|\,a\,\|$
$\bigl( \,\lambda \in \mathds{C}$, $a \in A \,\bigr)$, cf.\ the proof of
\ref{Banachunitis}. Both norms are equivalent on $\tld{A}$ \ref{Cstarequiv}.
More precisely, we have
\[ \|\,b\,\| \leq |\,b\,| \leq 3\:\|\,b\,\| \quad \text{for normal} \quad b \in \tld{A}, \]
as well as
\[ \|\,a\,\| \leq |\,a\,| \leq 6\:\|\,a\,\| \quad \text{for arbitrary} \quad a \in \tld{A}. \]
\end{theorem}

\begin{proof}
The statement for normal elements follows from \ref{C*rl} and the fact that
\[ \rlambda(\mu e + a) \leq | \,\mu \,| + \rlambda(a) \leq 3 \,\rlambda(\mu e + a)
\quad \text{for all} \quad \mu \in \mathds{C},\ a \in A. \]
cf.\ \ref{rlunit}. The statement for arbitrary elements is proved as follows.
Please note first that the C*-algebra norm $\| \cdot \|$ on $\tld{A}$ is
dominated by the other norm $| \cdot |$, by \ref{unitCstar}.

An arbitrary element $a$ of $\tld{A}$ has a unique decomposition as
$a = b + \iu c$ with $b, c \in \tld{A}\sa$. From the isometry of the involution
in the C*-algebra norm $\| \cdot \|$ on $\tld{A}$, we then get 
\[ \| \,b \,\| , \| \,c \,\| \leq \| \,a \,\|. \]
Indeed, we have for example $b = \frac12 \,(a + a^*)$, so
\[ \| \,b \,\| \leq \frac12 \,\bigl( \,\| \,a \,\| + \| \,a^* \,\| \,\bigr) = \| \,a \,\|. \]
From what has been proved so far, it follows that
\begin{align*}
\| \,a \,\| \leq | \,a \,| = | \,b + \iu c \,| & \leq | \,b \,| + | \,c \,| \\
& \leq 3 \:\bigl( \,\| \,b \,\| + \| \,c \,\| \,\bigr) \\
& \leq 3 \:\bigl( \,\| \,a \,\| + \| \,a \,\| \,\bigr) = 6 \:\| \,a \,\|. \qedhere
\end{align*}
\end{proof}

\begin{proposition}\label{commrs}%
For elements $a,b$ of a normed \st-algebra such that
$a^*a$ commutes with $b\,b^*$, we have
\[ \rsigma(ab) \leq \rsigma(a)\,\rsigma(b). \]
In particular, the function $\rsigma$ is submultiplicative on a
commutative normed \st-algebra.
\end{proposition}

\begin{proof}
From \ref{rlcomm} and \ref{commrlsub} it follows that
\begin{align*}
\ & {\rsigma(ab)}^{\,2} = \rlambda(b^*a^*ab) = \rlambda(a^*abb^*) \\
\leq \ & \rlambda(a^*a)\,\rlambda(b\,b^*) = \rlambda(a^*a)\,\rlambda(b^*b)
= {\rsigma(a)}^{\,2\,} {\rsigma(b)}^{\,2}. \pagebreak \qedhere
\end{align*}
\end{proof}

\clearpage


\section{Automatic Continuity}

\fancyhead[RO]{\SMALL{\rightmark}}

\begin{definition}[contractive linear map]\index{concepts}{contractive}%
A linear map between (possibly real) normed spaces
is said to be \underline{contractive} if it is bounded with
norm not exceeding $1$.
\end{definition}

\begin{theorem}\label{contraux}%
Let $\pi : A \to B$ be a \st-algebra homomorphism from a Banach
\st-algebra $\bigl( \,A , | \cdot | \,\bigr)$ to a pre-C*-algebra
$\bigl( \,B , \| \cdot \| \,\bigr)$. We have
\[ \| \,\pi (a) \,\| \leq \rsigma (a) \leq \| \,a \,\| \quad \text{for all} \quad a \in A, \]
where the second $\| \cdot \|$ denotes the accessory \st-norm on $A$,
cf.\ \ref{auxnorm}.

In particular, $\pi$ is contractive in the accessory \st-norm $\| \cdot \|$ on $A$.
\end{theorem}

\begin{proof} We may assume that $B$ is complete. For $a \in A$,
we have by \ref{C*rs} that
\[ \| \,\pi (a) \,\| = \rsigma\bigl(\pi(a)\bigr) \leq \rsigma(a)
\leq {| \,a^*a \,| \,}^{1/2} = {\| \,a^*a \,\| \,}^{1/2} \leq \| \,a \,\|, \]
where the first inequality stems from the Spectral Radius Formula
and the fact that
$\s\bigl(\pi (a^*a)\bigr) \setminus \{ 0 \} \subset \s(a^*a) \setminus \{ 0 \}$,
cf.\ \ref{spechom}.
\end{proof}

\begin{corollary}%
Any \st-algebra homomorphism from a Banach \linebreak \st-algebra
with isometric involution to a pre-C*-algebra is contractive.
\end{corollary}

In particular:

\begin{corollary}\label{Cstarcontr}%
A \st-algebra homomorphism from a C*-algebra to a
pre-C*-algebra is contractive.
\end{corollary}

Whence the following fundamental result.

\begin{corollary}\label{Cstarisom}%
\index{concepts}{algebra!C*-algebra!isomorphism}%
\index{concepts}{isomorphism of C*-algebras}%
\index{concepts}{C6@C*-algebra!isomorphism}%
A \st-algebra isomorphism between C*-algebras is isometric,
and shall therefore be called a \underline{C*-algebra isomorphism}
or an \underline{isomorphism of C*-algebras}.
\end{corollary}

See also \ref{C*injisometric} \& \ref{reminjiso} below.

Our next aim is the automatic continuity result \ref{automcont}.
The following two lemmata will find a natural reformulation at a later place.
(See \ref{reform2} and \ref{reform1} below.) \pagebreak

\begin{lemma}\label{kerclosed}%
Let $\pi : A \to B$ be a \st-algebra homomorphism from a Banach \st-algebra
$A$ to a pre-C*-algebra $B$. Then $\ker \pi$ is closed in $A$.
\end{lemma}

\begin{proof} We may assume that $B$ is complete. We shall
use the fact that then
\[ \rlambda\bigl(\pi(a)\bigr) \leq \rlambda(a) \]
holds for all $a \in A$, cf.\ \ref{spechom}. Let $a$ belong to the closure of
$\ker \pi$. Let $x \in A$ be arbitrary. Then $xa$ also is in the closure of
$\ker \pi$. For a sequence $(k_n)$ in $\ker \pi$ converging to $xa$, we have
\[ \rlambda\bigl(\pi (xa)\bigr) = \rlambda\bigl(\pi (xa-k_n)\bigr)
\leq \rlambda(xa-k_n) \leq | \,xa-k_n \,| \to 0, \]
which implies
\[ \rlambda\bigl(\pi(xa)\bigr) = 0. \]
Choosing $x := a^*$, we obtain
\[ 0 = \rsigma\bigl(\pi(a)\bigr) = \| \,\pi (a) \,\|, \]
i.e.\ $a$ is in $\ker \pi$.
\end{proof}

\begin{lemma}\label{injclosgraph}%
Let $\pi : A \to B$ be a \st-algebra homomorphism from a normed
\st-algebra $A$ to a pre-C*-algebra $B$, which is continuous on
$A\sa$. If $\pi$ is injective, then the involution in $A$ has closed
graph.
\end{lemma}

\begin{proof}
Let $a = \lim a_n$, $b = \lim {a_n}^*$ in $A$. Let $\theta \geq 0$ be a
bound of $\pi$ on the real subspace $A\sa$. For all $c \in A$, and all
$n$, we then have
\begin{align*}
 {\| \,\pi \,( \,c-{a_n}^* \,) \,\| \,}^2
 & = \| \,\pi \,\bigl( \,{( \,c-{a_n}^* \,)}^* ( \,c-{a_n}^* \,) \,\bigr) \,\| \\
 & \leq \theta \cdot | \,{( \,c-{a_n}^* \,)}^* ( \,c-{a_n}^* \,) \,| \\
 & \leq \theta \cdot | \,c^*-a_n \,| \cdot | \,c-{a_n}^* \,|.
\end{align*}
Choosing $c := b$, and using that $(a_n)$ is bounded, we get
\[ \pi (b) = \lim \pi ({a_n}^*). \]
Similarly, choosing $c := a^*$, and using that $({a_n}^*)$ is
bounded, we get
\[ \pi (a^*) = \lim \pi ({a_n}^*). \]
Hence $\pi (b) = \pi (a^*)$, and this implies
$b = a^*$ if $\pi$ is injective.
\end{proof}

\medskip
We now have the following automatic continuity result.
\pagebreak

\begin{theorem}[automatic continuity]\label{automcont}%
A \st-algebra homomorphism from a Banach \st-algebra
to a pre-C*-algebra is continuous.
\end{theorem}

\begin{proof}
Let $\pi : A \to B$ be a \st-algebra homomorphism from a Banach \linebreak
\st-algebra $A$ to a pre-C*-algebra $B$. The kernel of $\pi$ is closed,
by \ref{kerclosed}. This implies that $C := A \,/ \ker \pi$ is a Banach \st-algebra
in the quotient norm, cf.\ the appendix \ref{quotspaces}. Furthermore, the map
$\pi$ factors to an injective \st-algebra homomorphism $\pi_1$ from $C$ to $B$.
The mapping $\pi_1$ is contractive in the accessory \st-norm on $C$,
cf.\ \ref{contraux}. This implies that the involution in $C$ has closed graph,
cf.\ \ref{injclosgraph}. But then the involution is continuous, as $C$ is complete.
Hence the accessory \st-norm on $C$ and the quotient norm on $C$ are equivalent,
cf.\ \ref{involcont}. It follows that $\pi_1$ is continuous in the quotient norm on $C$,
which is enough to prove the theorem. 
\end{proof}

\begin{proposition}\label{bdedlambcontract}%
Let $A,B$ be normed algebras and let $\pi : A \to B$ be a continuous algebra
homomorphism. We then have
\[ \rlambda\bigl(\pi(a)\bigr) \leq \rlambda(a) \quad \text{for all} \quad a \in A. \]
\end{proposition}

\begin{proof}
Let $c > 0$ with
\[ | \,\pi(a) \,| \leq c \,| \,a \,| \quad \text{for all} \quad a \in A. \]
For $a \in A$ and all integers $n \geq 1$, we then have
\[ \rlambda\bigl(\pi(a)\bigr) \leq { \bigl| \,{\pi(a) \,}^n \,\bigr| \,}^{1/n}
= {\bigl| \,\pi \bigl( {a \,}^n \bigr) \,\bigr| \,}^{1/n}
\leq {c \,}^{1/n} { \,\bigl| \,{a \,}^n \,\bigr| \,}^{1/n} \to \rlambda(a). \qedhere \]
\end{proof}

\begin{corollary}\label{commbdedcontr}%
Let $\pi$ be a continuous \st-algebra homomorphism from a normed
\st-algebra $A$ to a pre-C*-algebra $B$. For a \underline{normal}
\linebreak element $b$ of $A$, we then have
\[ \| \,\pi(b) \,\| \leq \rlambda(b) \leq | \,b \,|. \]
It follows that if $A$ is commutative then $\pi$ is contractive.
\end{corollary}

\begin{proof}
This follows now from \ref{C*rl}.
\end{proof}

\medskip
Hence the following enhancement of \ref{automcont}.

\begin{corollary}\label{commBcontr}%
A \st-algebra homomorphism from a commutative
Banach \st-algebra to a pre-C*-algebra is contractive.
\pagebreak
\end{corollary}

\clearpage


\section{The Boundary of the Spectrum}%
\label{boundary}

We proved in \ref{specsubalg} that if $B$ is a subalgebra of an algebra
$A$, then
\[ \s _A (b) \setminus \{0\} \subset \s _B (b) \setminus \{0\}
\quad \text{for all} \quad b \in B. \]
Also, if $B$ furthermore is \twiddle-unital in $A$, then
\[ \s _A (b) \subset \s _B (b) \quad \text{for all} \quad b \in B. \]
We set out to show a partial converse inclusion, see \ref{spbdry} below.

\begin{theorem}\label{invbded}%
Let $A$ be a unital Banach algebra. Assume that $a \in A$ is the limit
of a sequence $\{ \,a_n \,\}_{\,n \geq 1}$ of invertible elements of $A$.
The sequence $\{ \,{a_n \,}^{-1} \,\}_{\,n \geq 1}$ converges in $A$
to an inverse of $a$ if and only the sequence in question is bounded.
\end{theorem}

\begin{proof}
The ``only if'' part is trivial. So assume that the sequence is bounded.
Let $c > 0$ with $| \,{a_n \,}^{-1} \,| \leq c$ for all integers $n \geq 1$.
For all integers $m, n \geq 1$, we then have
\[ | \,{a_m \,}^{-1} - {a_n \,}^{-1} \,| =
| \,{a_m \,}^{-1} \,( \,a_n - a_m \,) \,{a_n \,}^{-1} \,|
\leq c^{\,2} \,| \,a_n - a_m \,|. \]
This shows that $\{ \,{a_n \,}^{-1} \,\}_{\,n \geq 1}$ is a Cauchy
sequence, and thus convergent to some limit $b \in A$. The
element $b$ is an inverse of $a$ by continuity of multiplication.
\end{proof}

\medskip
We shall now consider the algebra $\blop(V)$ of bounded
linear operators on a complex normed space $V$, cf.\ \ref{algbdedop}.

\begin{definition}[the approximate point spectrum, $\s_{ap} (a)$]%
\index{concepts}{approximate point spectrum}%
\index{symbols}{s25@$\s_{ap}(a)$}%
\index{concepts}{spectrum!of a bounded operator!approximate point}%
Let $V$ be a complex normed vector space. Let $a \in \blop(V)$
be a bounded linear operator on $V$. The
\underline{approximate point spectrum} of $a$ is the set $\s_{ap}(a)$
of those complex numbers $\lambda$ for which there is a sequence
$\{ \,x_n \,\}$ of \underline{unit} \underline{vectors} in $V$ with
\[ ( \,\lambda \mathds{1} - a \,) \,x_n \to 0. \]
\end{definition}

\begin{observation}\label{apsp}
It is easily seen that (under the above assumptions), if a complex number
$\lambda$ is in the approximate point spectrum $\s_{ap}(a)$ of $a$, then
the bounded linear operator $\lambda \mathds{1} - a$ cannot have any
bounded left inverse. In particular, we then have
\[ \s_{ap}(a) \subset \s_{\blop(V)} (a). \]
In other words, the approximate point spectrum of a bounded linear
operator on a complex normed space is part of the spectrum of the
operator with respect to the algebra of bounded linear operators on $V$.%
\pagebreak
\end{observation}

\begin{definition}[the boundary $\partial \,S$ of a set $S$]%
\index{symbols}{d@$\partial \,S$}\index{concepts}{boundary}%
If $S$ is a subset of a topological space $X$, let
$\partial \,S := \overline{S} \cap \overline{X \setminus S}$
denote the \underline{boundary} of $S$.
\end{definition}

As a partial converse to observation \ref{apsp} above, one has:

\begin{theorem}\label{bdryspap}%
The boundary of the spectrum of a bounded linear operator on a complex
\underline{Banach} space is part of the approximate point spectrum
of the operator. More precisely, let $V$ be a complex \underline{Banach}
space. Let $a \in \blop(V)$ be a bounded linear operator on $V$.
One then has
\[ \partial \,\s_{\blop(V)} (a) \subset \s_{ap} (a), \]
the boundary being taken with respect to the complex plane.
\end{theorem}

\begin{proof}
Let $\lambda \in \partial \,\s_{\blop(V)} (a)$, and let $\{ \,\lambda_n \,\}$
be a sequence of complex numbers converging to $\lambda$, such that
$( \,\lambda_n \mathds{1} - a \,)$ is invertible in $\blop(V)$ for all $n$.
Since $\blop(V)$ is a Banach algebra, Theorem \ref{invbded} above
implies that the sequence $\bigl\{ \,{( \,\lambda_n \mathds{1} - a \,)}^{\,-1} \,\bigr\}$
must be unbounded. By going over to a subsequence, we can assume
that $\bigl| \,{( \,\lambda_n \mathds{1} - a \,)}^{\,-1} \,\bigr| \geq n + 1$
for all integers $n \geq 1$. There then exists a sequence
$\{ \,y_n \,\}$ of unit vectors in $V$ with
$\alpha_n := \bigl| \,{( \,\lambda_n \mathds{1} - a \,)}^{\,-1}  \,y_n \,\bigr| \geq n$
for all integers $n \geq 1$. Consider the unit vectors 
\[ x_n := {\alpha_n}^{\,-1} \,{( \,\lambda_n \mathds{1} - a \,)}^{\,-1} \,y_n \in V. \]
We then have
\begin{align*}
( \,\lambda \mathds{1} - a \,) \,x_n & =
 ( \,\lambda - \lambda_n \,) \,x_n + ( \,\lambda_n \mathds{1} - a \,) \,x_n \\
 & = ( \,\lambda - \lambda_n \,) \,x_n + {\alpha_n}^{\,-1} \,y_n \quad \to \quad 0.
 \qedhere
\end{align*}
\end{proof}

\begin{proposition}\label{bdrysub}%
Let $C, D$ be two subsets of a topological space $X$. Assume that
$\partial \,D \subset C \subset D$. Then $\partial \,D \subset \partial \,C$.
\end{proposition}

\begin{proof}
We have $\partial \,D \subset C \subset \overline{C}$ as well as
$\partial \,D \subset \overline{X \setminus D} \subset \overline{X \setminus C}$.
Together, these two facts make that $\partial \,D \subset \partial \,C$.
\end{proof}

\medskip
For example $D$ the closed unit disc $\{ \,z \in \mathds{C} : | \,z \,| \leq 1 \,\}$,
and $C$ the annulus $\{ \,z \in \mathds{C} : \frac12 \leq | \,z \,| \leq 1 \,\}$
satisfy the requirements. Please note that in this case $\partial \,D$ is a
proper subset of $\partial \,C$.

\begin{theorem}[\v{S}ilov]\label{spbdry}%
\index{concepts}{Silov@\v{S}ilov}%
\index{concepts}{Theorem!Silov@\v{S}ilov}%
Let $B$ be a complete subalgebra of a normed algebra $A$,
and let $b \in B$. Then
\[ \partial \,\s _B (b) \subset \partial \,\s _A (b) \cap \,\s _A (b). \pagebreak \]
\end{theorem}

\begin{proof}
Consider the left regular representation $L : a \mapsto L\tla$ $(a \in \tld{A})$
of $\tld{A}$ on itself, cf.\ \ref{leftreg}. Then the left regular representation of
$\tld{B}$ on itself is given by
$b \mapsto L_{\textstyle{b}} | _{\text{\small{$\tld{B}$}}}$ $(b \in \tld{B})$,
because we may consider $\tld{B}$ as a subalgebra of $\tld{A}$ via the canonical
imbedding, cf.\ \ref{canimb}. We may equip $\tld{B}$ with the norm inherited from
$\tld{A}$. Then $\tld{B}$ will be complete, cf.\ the proof of \ref{Banachunitis}.
Let $b \in B\vphantom{\tld{B}}$ be fixed. We have
\begin{align*}
\partial \,\s_B (b)
 & = \partial \,\s_{\blop(\tld{B})} (L_{\textstyle{b}} | _{\text{\small{$\tld{B}$}}})
 \tag*{(by \ref{specunit} \& \ref{spleftregbded})} \\
 & \subset \s_{ap} (L_{\textstyle{b}} | _{\text{\small{$\tld{B}$}}})
 \tag*{(by \ref{bdryspap})} \\
 & \subset \s_{ap} (L_{\textstyle{b}}) \tag*{(obvious)} \\
 & \subset \s_{\blop(\tld{A})} (L_{\textstyle{b}}) = \s_A (b),
 \tag*{(by \ref{apsp}, \ref{spleftregbded} \& \ref{specunit})}
\end{align*}
which makes that
\[ \partial \,\s_B (b) \subset \s_A (b). \]
We also have
\[ \s _A (b) \setminus \{0\} \subset \s _B (b) \setminus \{0\}, \]
cf.\ \ref{specsubalg}. Using these two facts, it can be shown that any
$\lambda$ in $\partial \,\s_B (b)$ belongs as well to $\partial \,\s_A (b)$.
Indeed, for $\lambda \neq 0$, this follows from proposition \ref{bdrysub} above,
by looking at the metric space $\mathds{C} \setminus \{ 0 \}$.
For $\lambda = 0$, this also follows from proposition \ref{bdrysub} above,
using the fact that in this case $0 \in \s_B (b)$ and $0 \in \s_A (b)$.
\end{proof}

\begin{theorem}\label{nointerior}%
Let $B$ be a complete subalgebra of a normed algebra $A$, and let $b \in B$.
If $\s _B (b)$ has no interior (e.g.\ is real or is contained in the unit circle), then
\[ \s _B (b) \subset \s _A (b) \cap \,\partial \,\s _A (b). \]
\end{theorem}

\begin{proof}
If $\s _B (b)$ has no interior, then by the preceding Theorem \ref{spbdry} of \v{S}ilov, we get
\[ \s _B (b) = \partial \,\s _B (b) \subset \partial \,\s _A (b) \cap \s _A (b). \qedhere \]
\end{proof}

\medskip
Next wo results guaranteeing equality of spectra.

\begin{corollary}%
Let $B$ be a complete subalgebra of a normed algebra $A$,
and let $b \in B$. If $\s _B (b)$ has no interior (e.g.\ is real or is contained
in the unit circle), then
\[ \s _B (b) \setminus \{0\} = \s _A (b) \setminus \{0\}. \]
If $B$ furthermore is \twiddle-unital in $A$, then
\[ \s _B (b) = \s _A (b). \pagebreak \]
\end{corollary}

\begin{proof}
This follows from the preceding theorem \ref{nointerior} and \ref{specsubalg}.
\end{proof}

\begin{corollary}%
Let $B$ be a \twiddle-unital closed subalgebra of a \linebreak Banach
algebra $A$, and let $b \in B$ such that $\s _A (b)$ does not
separate the complex plane (e.g.\ is real). Then
\[ \s _B (b) = \s _A (b). \]
\end{corollary}

\begin{proof}
It suffices to prove that $\s _B (b) \subset \s _A (b)$, cf.\ \ref{specsubalg}.
So suppose that there exists $\lambda \in \s _B (b) \setminus \s _A (b)$.
Since $\s _A (b)$ does not separate the complex plane, there exists a
continuous path disjoint from $\s _A (b)$, that connects $\lambda$ with
the point at infinity. It is easily seen that this path meets the boundary of
$\s _B (b)$ in at least one point, which then also must belong to $\s _A (b)$,
by theorem \ref{spbdry}, a contradiction.
\end{proof}

\medskip
In the sense detailed below, the spectrum of an element, which has real
spectrum in a Banach algebra, is unchanged by going over to a closed
\twiddle-unital subalgebra still containing the element, or by going over
to a ``normed \twiddle-unital superalgebra'' of the original Banach algebra.

\begin{corollary}\label{spconvend}%
Let $C$ be a Banach algebra, and let $c$ be an \linebreak element of
$C$ with real spectrum.
\begin{itemize}
  \item[$(i)$] if $B$ is any closed \twiddle-unital subalgebra of $C$
                      containing $c$, then
                      \[ \s _C (c) = \s _B (c), \]
 \item[$(ii)$] if $A$ is any normed algebra containing $C$ as a
                      \twiddle-unital sub\-algebra, such that the norm of $A$ extends
                      that of $C$, then
                      \[ \s _C (c) = \s _A (c). \]
\end{itemize}
\end{corollary}

\clearpage


\section{Square Roots of Invertible Elements}

\begin{theorem}[Mertens]\label{Mertens}%
\index{concepts}{Theorem!Mertens}\index{concepts}{Mertens' Theorem}%
Let $A$ be a normed algebra. Let
\[ \sum _{n \geq 0} a_n\quad \text{and}\quad \sum _{n \geq 0} b_n \]
be two convergent series in $A$ and assume that the first of them
converges \underline{absolutely}. Put
\[ c_{n} :=\sum _{k=0}^n a_k b_{n-k} \]
for every integer $n \geq 0$. The series
\[ \sum _{n \geq 0} c_n \]
is called the \underline{Cauchy product} of the above two series. We have
\[ \sum _{n \geq 0} c_n = \biggl( \,\sum _{n \geq 0} a_n \biggr)
\cdot \biggl( \,\sum _{n \geq 0} b_n \biggr). \]
\end{theorem}

\begin{proof} Let
\begin{alignat*}{3}
A_n & := \sum _{k=0}^n a_k,
 & \qquad B_n & := \sum _{k=0}^n b_k,
 & \qquad C_n & := \sum _{k=0}^n c_k, \\
s & := \sum _{n \geq 0} a_n,
 & \qquad t & := \sum _{n \geq 0} b_n,
 & \qquad d_n & := B_n - t.
\end{alignat*}
Rearrangement of terms yields
\begin{align*}
C_n & = a\0 b\0 + (a\0 b_1 + a_1 b\0) + \cdots + ( a\0 b_n + a_1 b_{n-1} + \cdots + a_n b\0) \\
 & = a\0 B_n + a_1 B_{n-1} + \cdots + a_n B\0 \\
 & = a\0 ( t + d_n ) + a_1 ( t + d_{n-1} ) + \cdots + a_n ( t + d\0 ) \\
 & = A_n t + a\0 d_n +a_1 d_{n-1} + \cdots + a_n d\0.
\end{align*}
Put $h_n := a\0 d_n + a_1 d_{n-1} + \cdots + a_n d\0$. We have to show that
$C_n \to st.$ Since $A_n t \to st$, it suffices to show that $h_n \to 0$. So
let $\varepsilon > 0$ be given. Let $\alpha := \sum _{n \geq 0} | \,a_n \,|$,
which is finite by assumption. We have $d_n \to 0$ because $B_n \to t$.
Hence we can choose $N \geq 0$ such that $| \,d_n \,| \leq \varepsilon / (1 + \alpha)$
for all integers $n > N$, and for such $n$ we have
\begin{align*}
| \,h_n \,| & \leq | \,a_n d\0 + \cdots + a_{n-N} d_N \,|
+ | \,a_{n-(N+1)} d_{N+1} + \cdots + a\0 d_n \,| \\
 & \leq | \,a_n d\0 + \cdots + a_{n-N} d_N \,| + \varepsilon.
\end{align*}
Letting $n \to \infty$, we get $\limsup _{n \to \infty} | \,h_n \,| \leq \varepsilon$
by $a_n \to 0$. The statement follows because $\varepsilon > 0$ was arbitrary.
\pagebreak
\end{proof}

\medskip
The following lemma is the basis of
the uniqueness of the positive square root
of a positive real number for example.
Recall \ref{zerodivdef}.

\begin{lemma}\label{samesquare}%
Let $x, y$ be two commuting elements of a ring with the same square.
If $x+y$ is not a divisor of zero, nor zero, then $x=y$.
\end{lemma}

\begin{proof} The products of $x-y$ and $x+y$ are
${x}^{\,2} - {y}^{\,2} \pm ( xy - yx ) = 0$, so the condition implies that $x-y = 0$.
\end{proof}

\begin{theorem}\label{sqroot}\index{concepts}{square root|(}%
Let $A$ be a Banach algebra and let $a \in A$ with $\rlambda(a) < 1$. The series
\[ \sum\limits_{n=1}^{\infty} \,\bigl( _{\ \text{\small{$n$}}}^{1/2} \bigr) \,{a \,}^{n} \]
then converges absolutely to an element $b$ of $A$ with ${(\,e+b\,)}^{\,2} = e+a$
and $\rlambda(b) < 1 $. There is no other element $c$ of $\tld{A}$
with ${(\,e+c\,)}^{\,2} = e+a$ and $\rlambda(c) \leq 1$.
\end{theorem}

\begin{proof} If the above series converges absolutely, it follows by the \linebreak
Theorem of Mertens \ref{Mertens} concerning the Cauchy product
of series that ${(\,e+b\,)}^{\,2} = e+a$. Assume that $\rlambda(a) < 1$. The series for
$b$ converges absolutely by comparison with the geometric series \ref{geoseries}.
It shall next be shown that $\rlambda(b) < 1$. For $k \geq 1$, we have by
\ref{rlpowers} and \ref{commrlsub} that
\[ \rlambda \biggl( \,\sum\limits_{n=1}^k
\,\bigl( _{\ \text{\small{$n$}}}^{1/2} \bigr) \,{a \,}^n \biggr) \leq \sum\limits_{n=1}^k
\ \Bigl| \bigl( _{\ \text{\small{$n$}}}^{1/2} \bigr) \Bigr| \ {\rlambda ( a ) \,}^n. \]
Hence, by applying the last part of \ref{commrlsub} to the closed
subalgebra of $A$ generated by $a$, we get
\begin{align*}
\rlambda (b) \leq \sum\limits_{n=1}^{\infty}
\ \Bigl| \bigl( _{\ \text{\small{$n$}}}^{1/2} \bigr) \Bigr| \ {\rlambda(a) \,}^n
 & = -\sum\limits_{n=1}^{\infty} \ \bigl( _{\ \text{\small{$n$}}}^{1/2} \bigr)
        \,{\bigl(-\rlambda (a)\bigr) \,}^n \\
 & = 1-\sqrt{1-\rlambda (a)} < 1.
\end{align*}
Let $c$ now be an element of $\tld{A}$ with $\rlambda(c) \leq 1$ and
${(\,e+c\,)}^{\,2} = e+a$. It shall be shown that $c = b$. The elements
$x := e+b$, $y := e+c$ have the same square. Furthermore, we have
$y\,(\,e+a\,) = {(\,e+c\,)}^{\,3} = (\,e+a\,)\,y$, i.e.\ the element $y$ commutes
with $a$. By continuity, it follows that $y$ also commutes with $b$, hence
with $x$. Finally, $x+y = 2e+(\,b+c\,)$ is invertible as
$\rlambda(\,b+c\,) \leq  \rlambda(b) + \rlambda(c) < 2$, cf.\ \ref{commrlsub}.
It follows from the preceding lemma \ref{samesquare} that $x = y$, or,
in other words, that $b = c$. \end{proof}

\begin{lemma}[Ford's Square Root Lemma]\label{Ford}%
\index{concepts}{Ford's Square Root Lemma}%
\index{concepts}{Lemma!Ford's Square Root}%
\index{concepts}{Theorem!Ford's Square Root Lemma}%
If $a$ is a Hermitian \linebreak element of a Banach \st-algebra,
satisfying $\rlambda(a) < 1$, then the square root $e+b$ of $e+a$
(with $b$ as in the preceding theorem \ref{sqroot}) is \linebreak Hermitian. 
\end{lemma}

\begin{proof} This is obvious if the involution is continuous.
Now for the general case.
Consider the element $b$ of the preceding theorem \ref{sqroot}.
We have to show that $b^{\,*} = b$. The element $b^{\,*}$ satisfies
${(\,e+b^{\,*}\,)}^{\,2} = {(\,e+a\,)}^{\,*} = e+a$ as well as
$\rlambda(b^{\,*}) = \rlambda(b) < 1$ by \ref{Brlsymm}. So $b = b^{\,*}$
by the uniqueness statement of the preceding theorem \ref{sqroot}.
\end{proof}

\medskip
Ford's Square Root Lemma, although basic, was overlooked for a long time.
See Palmer \cite[p. 1159]{Palm}.

\begin{theorem}\label{possqroot}%
Let $A$ be a Banach algebra, and let $x$ be an element of $\tld{A}$ with
$\s(x) \subset \ ] \,0, \infty \,[$. (Please note: \ref{speczero}.) Then $x$ has a
unique square root $y$ in $\tld{A}$ with $\s(y) \subset \ ] \,0, \infty \,[$.
Furthermore $x$ has no other square root in $\tld{A}$ with spectrum in
$[ \,0, \infty \,[$. The element $y$ belongs to the closed subalgebra of $\tld{A}$
generated by $x$. If $A$ is a Banach \st-algebra, and if $x$ is Hermitian, so is $y$.
\end{theorem}

\begin{proof} We may assume that $\rlambda(x) \leq 1$. The
element $a := x-e$ then has $\s(a) \subset \ ]-1,0\,]$, whence
$\rlambda(a) < 1$. Consider now $b$ as in \ref{sqroot}. Then
$y := e + b$ is a square root of $e + a = x$, by \ref{sqroot}.
Furthermore, $\s(y) \subset \mathds{R} \setminus \{ 0 \}$ by the Rational
Spectral Mapping Theorem. Also $\rlambda(b) < 1$, which implies that
$\s(y) \subset \ ] \,0,\infty \,[$. Let now $z \in \tld{A}$ be another square
root of $x$ with $\s(z) \subset \,[ \,0,\infty \,[$. Then $\s(z) \subset \,[ \,0,1 \,]$
by the Rational Spectral Mapping Theorem. Thus, the element $c:= z-e$
has $\s(c) \subset \,[ \,-1,0 \,]$, so $\rlambda(c) \leq 1$. Since
${(\,e+c\,)}^{\,2} = x = e+a$, theorem \ref{sqroot} implies that
$c = b$, whence $z = y$. Let now $B$ denote the closed
subalgebra of $\tld{A}$ generated by $x$. Clearly $y$ belongs to the
closed subalgebra of $\tld{A}$ generated by $x$ and $e$, which is
$\mathds{C}e + B$, cf.\ the proof of \ref{Banachunitis}. So let $y = \mu e + f$
with $f \in B$ and $\mu \in \mathds{C}$. Then $x = {\mu}^{\,2\,} e + g$ with
$g \in B$. So if $\mu \neq 0$, then $e \in B$, whence $y \in B$. If $\mu = 0$,
then $y = f \in B$ as well. The last statement follows from the preceding
lemma \ref{Ford}.
\end{proof}
\index{concepts}{square root|)}

\medskip
For C*-algebras, a stronger result holds, cf.\ \ref{C*sqroot}. \pagebreak

\clearpage


\section{Hermitian Banach \texorpdfstring{$*$-}{\80\052\80\055}Algebras}%
\label{secHerm}

\begin{definition}[Hermitian \protect\st-algebras]%
\index{concepts}{Hermitian!s-algebra@\protect\st-algebra}%
\index{concepts}{algebra!Banach s-algebra@Banach \protect\st-algebra!Hermitian}
\index{concepts}{algebra!s-algebra@\protect\st-algebra!Hermitian}
A \st-algebra is said to be \linebreak \underline{Hermitian} if each of its
Hermitian elements has real spectrum.
\end{definition}

\begin{proposition}%
A \st-algebra $A$ is Hermitian if and only if its \linebreak unitisation $\tld{A}$ is.
\end{proposition}

\begin{proof}
For $a \in A\sa$ and $\lambda \in \mathds{R}$, we have, by \ref{specunit},
\[ \s_{\tld{A}} \,(\lambda e+a) = \lambda + \s_{\tld{A}} \,(a) = \lambda + \s_A \,(a).
\qedhere \]
\end{proof}

\medskip
The following result is responsible for all of the advanced properties of
Hermitian Banach \st-algebras.

\begin{theorem}\label{fundHerm}%
For a Banach \st-algebra $A$ the following statements are equivalent.
\begin{itemize}
   \item[$(i)$] \ $A$ is Hermitian,
  \item[$(ii)$] \ $\iu \notin \s(a)$ whenever $a \in A$ is Hermitian,
 \item[$(iii)$] \ $\rlambda(a) \leq \rsigma(a)$ for all $a \in A$,
 \item[$(iv)$] \ $\rlambda(b) = \rsigma(b)$ for all normal $b \in A$,
  \item[$(v)$] \ $\rlambda(b) \leq {|\,b^*b\,| \,}^{1/2}$ for all normal $b \in A$.
\end{itemize}
\end{theorem}

\begin{proof} (i) $\Rightarrow$ (ii) is trivial.

(ii) $\Rightarrow$ (iii): Assume that (ii) holds and that
$\rlambda(a) > \rsigma(a)$ for some $a \in A.$ Upon replacing $a$
with a suitable multiple, we can assume that $1 \in \s(a)$ and
$\rsigma(a) < 1$. There then exists a Hermitian element $b$ of
$\tld{A}$ such that ${b\,}^2 = e-a^*a$ (Ford's Square Root Lemma \ref{Ford}).
Since $b$ is invertible (Rational Spectral Mapping Theorem), it follows that
\[(e+a^*) \,(e-a) = {b\,}^2 + a^* -a = b \,\bigl( e + b^{-1} \,(a^* - a) \,b^{-1} \bigr) \,b. \]
Now the element $-\iu \,b^{-1} \,(a^*-a) \,b^{-1}$ is a Hermitian element of $A$,
so that $\iu \,\bigl( e+b^{-1} \,(a^*-a) \,b^{-1} \bigr)$ is invertible by assumption.
This implies that $e-a$ has a left inverse. One can apply a similar argument to
\[ (e-a) \,(e+a^*) = c \,\bigl( e+c^{-1} \,(a^*-a) \,c^{-1} \bigr) \,c, \]
where $c$ is a Hermitian element of $\tld{A}$ with ${c\,}^2 = e-aa^*$
(by $\rsigma(a^*) = \rsigma(a) < 1$, cf.\ \ref{rsC*}). Thus $e-a$ also has a
right inverse, so that $1$ would not be in the spectrum of $a$ by
\ref{leftrightinv}, a contradiction.

(iii) $\Rightarrow$ (iv) follows from the fact that
$\rsigma(b) \leq \rlambda(b)$ for all normal
$b \in A$, cf.\ \ref{Bnormalrsleqrl}.

(iv) $\Rightarrow$ (v) is trivial. 

(v) $\Rightarrow$ (i): Assume that (v) holds. Let $a$ be a Hermitian element
of $A$ and let $\lambda \in \s(a)$. We have to show that $\lambda$ is real.
For this purpose, we can assume that $\lambda$ is non-zero. Let $\mu$ be an
arbitrary real number, and let $n$ be an integer $\geq 1$. Define
$u := \lambda^{-1}a$ and $b := {(a + \iu \mu e) \,}^n \,u$. Then
${(\lambda + \iu \mu) \,}^n$ is in the spectrum of $b$
(Rational Spectral Mapping Theorem). This implies
\[ {|\,\lambda + \iu \mu \,|\,}^{2n} \leq {\rlambda( \,b \,) \,}^2 \leq |\,b^*b\,|
\leq \bigl| \,{ \bigl( \,{a \,}^2 + {\mu \,}^2 \,e \,\bigr) \,}^n \,\bigr| \cdot |\,u^*u\,|. \]
It follows that
\[ {| \,\lambda + \iu \mu \,|\,}^2
\leq {\bigl| \,{ \bigl( \,{a \,}^2 + {\mu \,}^2 \,e \,\bigr) \,}^n \,\bigr| \,}^{1/n}
\cdot {| \,u^*u \,| \,}^{1/n} \]
for all integers $n \geq 1$, so that by \ref{commrlsub}, we have
\[ {| \,\lambda + \iu \mu \,|\,}^2 \leq \rlambda \bigl( \,{a \,}^{2} + {\mu \,}^{2} \,e \,\bigr)
\leq \rlambda \bigl( \,{a \,}^2 \,\bigr)+ {\mu \,}^2. \]
Decomposing $\lambda = \alpha + \iu \beta $ with $\alpha $ and $\beta $ real,
one finds
\[ {\alpha \,}^2 + {\beta \,}^2 + 2 \beta \mu \leq \rlambda \bigl( \,{a \,}^2 \,\bigr) \]
for all real numbers $\mu$, so that $\beta$ must be zero, i.e.\ $\lambda$
must be real.
\end{proof}

\medskip
We now have the following fundamental fact:

\begin{theorem}\label{C*Herm}%
A C*-algebra is Hermitian.
\end{theorem}

\begin{proof}
For an element $a$ of a C*-algebra we have by \ref{rldef} and \ref{C*rs}:
\[ \rlambda(a) \leq \| \,a \,\| = \rsigma(a). \qedhere \]
\end{proof}

\begin{theorem}\label{Herminher}%
Every closed \st-subalgebra of a Hermitian Banach \st-algebra is Hermitian as well.
\end{theorem}

\begin{proof}
One uses the fact that a Banach \st-algebra $C$ is
Hermitian if and only if for every $a \in C$ one has
$\rlambda(a) \leq \rsigma(a)$.
\end{proof}

\begin{theorem}\label{specHermincl}%
Let $B$ be a complete Hermitian \st-subalgebra
of some normed \st-algebra $A$. We then have
\[ \s_B (b) \subset \s_A (b) \quad \text{for all} \quad b \in B. \]
\end{theorem}

\begin{proof}
Let $b \in B$, $\lambda \in \mathds{C} \setminus \s _A (b)$. Let $e_A$ denote
the unit in $\tld{A}$. Then $c := \lambda e_A - b$ in invertible in $\tld{A}$
(with inverse $c^{\,-1}$ in $\tld{A}$, say). The element
${( \lambda e_A -b )}^*( \lambda e_A -b )$ then also is invertible in $\tld{A}$
(with inverse $c^{\,-1} {(c^{\,-1})}^*$ in $\tld{A}$). That is,
$| \,\lambda \,| ^{\,2} \notin \s _A ( \,b^*b - \overline{\lambda} b - \lambda b^* )$.
Now, since $B$ is Hermitian, the Hermitian element
$b^*b - \overline{\lambda} b - \lambda b^*$ of $B$ has real spectrum in $B$.
Whence $\s _B ( \,b^*b - \overline{\lambda} b - \lambda b^* )
\subset \s _A ( \,b^*b - \overline{\lambda} b - \lambda b^* )$,
by theorem \ref{nointerior}. We conclude that
$| \,\lambda \,| ^{\,2} \notin \s _B ( \,b^*b - \overline{\lambda} b - \lambda b^* )$.
In other words, we have that ${( \lambda e_B -b )}^*( \lambda e_B -b )$
is invertible in $\tld{B}$, with $e_B$ denoting the unit in $\tld{B}$. Putting
$d := \lambda e_B - b \in \tld{B}$, we get that ${d}^{\,*}d$ is invertible in
$\tld{B}$, with inverse ${({d}^{\,*}d)}^{\,-1}$ in $\tld{B}$, say. Then
${({d}^{\,*}d)}^{\,-1} {d}^{\,*}$ is a left inverse of $d$ in $\tld{B}$. Considering
similarly the element $( \lambda e_A - b) {( \lambda e_A - b)}^*$ in $\tld{A}$,
we find that $d {d}^{\,*}$ is invertible in $\tld{B}$, with inverse
${(d {d}^{\,*})}^{\,-1}$ in $\tld{B}$, say. Then ${d}^{\,*} {(d {d}^{\,*})}^{\,-1}$
is a right inverse of $d$ in $\tld{B}$. Therefore $d = \lambda e_B - b$
is invertible in $\tld{B}$, by \ref{leftrightinv}. We have shown that
$\lambda \notin \s _B (b)$.
\end{proof}

\medskip
From \ref{specsubalg} it follows now

\begin{theorem}\label{specHermsubalg}%
If $B$ is a complete Hermitian \st-subalgebra
of some normed \st-algebra $A$, then
\[ \s_A (b) \setminus \{ 0 \} = \s_B (b) \setminus \{ 0 \}
\quad \text{for all} \quad b \in B. \]
If furthermore $B$ is \twiddle-unital in $A$, then
\[ \s_A (b) = \s_B (b) \quad \text{for all} \quad b \in B. \]
\end{theorem}

It is clear that \ref{specHermincl} and \ref{specHermsubalg} are
most of the time used in conjunction with \ref{Herminher}.

\bigskip
In the sense detailed below, within the realm of Hermitian
Banach \st-algebras, the ``punctured spectrum''
$\s (a) \setminus \{ 0 \}$ of an element $a$ is independent on
the Hermitian Banach \st-algebra with respect to which
the spectrum is taken.

\begin{corollary}\label{specHermend}%
If $A$ and $B$ are Hermitian Banach \st-algebras, such
that the norms of $A$ and $B$ coincide on some common
complete \st-subalgebra $C$, then
\[ \s _A (c) \setminus \{ 0 \} = \s _B (c) \setminus \{ 0 \}
\quad \text{for all} \quad c \in C. \]
\end{corollary}

\begin{proof} The common \st-subalgebra $C$ is a complete, hence
closed, automatically Hermitian \ref{Herminher} \st-subalgebra
of both $A$ and $B$, so it remains to apply the preceding theorem
\ref{specHermsubalg}. \pagebreak
\end{proof}

\medskip
Now towards unitary elements. First an introductory observation.

\begin{observation}\label{line}%
If $z$ and $\mu$ are real numbers with $\mu \neq 0$, then the
number $( \,z - \iu \mu \,) \,{( \,z + \iu \mu \,)}^{\,-1}$ is in the unit
circle, as any non-zero complex number and its complex conjugate
are of the same length. This says that if
$\mu \in \mathds{R} \setminus \{ 0 \}$, then the Moebius transformation
\[ z \mapsto ( \,z - \iu \mu \,) \,{( \,z + \iu \mu \,)}^{\,-1}
\qquad \bigl( \,z \in \mathds{C} \cup \{ \infty \} \,\bigr)  \]
takes the real line to the unit circle. Indeed, it maps the real line onto
\[ \{ z \in \mathds{C} : | \,z \,| = 1 \} \setminus \{ 1 \}, \]
the point at infinity being mapped to $1$. (As Moebius transformations
take generalised circles onto generalised circles.)
\end{observation}

\begin{definition}[Cayley transform]%
\index{concepts}{Cayley transform}\label{Cayley}%
Let $A$ be a Banach algebra, and let $a \in \tld{A}$.
For $\mu > \rlambda(a)$, one defines a \underline{Cayley transform}
$u$ of $a$ via
\[ u := ( \,a - \iu \mu e \,) \,{( \,a + \iu \mu e \,)}^{\,-1} \in \tld{A}. \]
\end{definition}

Please note that the assumption $\mu > \rlambda(a)$ implies that
$a + \iu \mu e$ is invertible in $\tld{A}$, so $u$ is well-defined.

\begin{proposition}\label{Moebius}%
Under the assumptions of definition \ref{Cayley}, we have that
$\s(u)$ is contained in the unit circle if and only if $\s(a)$ is real.
\end{proposition}

\begin{proof}
The Moebius transformation $r : z \mapsto r(z)$ given by
\[ r(z) := ( \,z - \iu \mu \,) \,{( \,z + \iu \mu \,)}^{\,-1}
\qquad \bigl( \,z \in \mathds{C} \cup \{ \infty \} \,\bigr) \]
has no pole on $\s(a)$ because $\mu > \rlambda(a)$.
Therefore the Rational Spectral Mapping Theorem implies that
\[ r \bigl( \s(a) \bigr) = \s\bigl( r(a) \bigr) = \s(u). \]
As noted above \ref{line}, the Moebius transformation $r$ maps
the real line onto
\[ \{ z \in \mathds{C} : | \,z \,| = 1 \} \setminus \{ 1 \}, \]
the point at infinity being mapped to $1$.
From the injective nature of a Moebius transformation it follows that
$\s(u)$ is contained in the unit circle if and only if $\s(a)$ is real.
\end{proof}

\smallskip
Next for \st-algebras. \pagebreak

\begin{definition}[unitary elements]\index{concepts}{unitary element}%
An element $u$ of a unital \linebreak \st-algebra is called
\underline{unitary} if $u^*u = uu^* = e$. The unitary elements
of a unital \st-algebra form a group under multiplication.
\end{definition}

The Cayley transformation $a \mapsto u$ behaves as expected
from \ref{line}:

\begin{proposition}\label{unitary}%
If $A$ is a Banach \st-algebra and if $a$ is any Hermitian
element of $\tld{A}$, then the Cayley transform $u$ of $a$
is a unitary element of $\tld{A}$. (Under the assumptions of
definition \ref{Cayley}.)
\end{proposition}

\begin{proof}
By commutativity \ref{comm2} we get
\begin{align*}
u^* & = \bigl[ \,{( \,a + \iu \mu e \,)}^{\,-1} \,\bigr]^* \,( \,a - \iu \mu e \,)^* \\
       & = {( \,a - \iu \mu e \,)}^{\,-1} \,( \,a + \iu \mu e \,) \\
       & = ( \,a + \iu \mu e \,) \,{( \,a - \iu \mu e \,)}^{\,-1},
\end{align*}
so that $u^*u = uu^* = e$.
\end{proof}

\begin{theorem}\label{specunitary}%
A Banach \st-algebra $A$ is Hermitian if and only if the spectrum
of every unitary element of $\tld{A}$ is contained in the unit circle.
\end{theorem}

\begin{proof}
Assume first that $A$ is Hermitian and let $u$ be any unitary element of
$\tld{A}$. We then have $\rlambda(u) = \rsigma(u) = 1$ by \ref{fundHerm}
(i) $\Rightarrow$ (iv). It follows that the spectrum of $u$ is contained in
the unit disk. Since also $u^{\,-1} = u^*$ is unitary, it follows from
${\s(u)}^{\,-1} = \s(u^{\,-1})$ (Rational Spectral Mapping Theorem) that
also ${\s(u)}^{\,-1}$ is contained in the unit disk. This makes that $\s(u)$
is contained in the unit circle.

Conversely, let $a$ be a Hermitian element of $A$ and let
$\mu > \rlambda(a)$. Then the Cayley transform
$u := ( \,a - \iu\mu e \,) \,{( \,a + \iu\mu e \,)}^{\,-1}$ is a unitary element
of $\tld{A}$, cf.\ \ref{unitary}. If $\s(u)$ is contained in the unit circle,
it follows that $\s(a)$ is real, cf. \ref{Moebius}.
\end{proof}

\begin{corollary}\label{C*unitary}
The spectrum of a unitary element of a unital \linebreak
C*-algebra is contained in the unit circle.
\end{corollary}

\begin{remark}\label{entire}
The definition of Cayley transforms as in \ref{Cayley} stems from
the theory of indefinite inner product spaces,
cf. e.g.\ \cite[Lemma 4.1 p.\ 38]{Bog}. This device allows us to avoid
the Spectral Mapping Theorem for entire functions in the proof of
\ref{specunitary}. \pagebreak
\end{remark}

\clearpage


\section{An Operational Calculus}

We next give a version of the operational calculus with an elementary
proof avoiding Zorn's Lemma. (See also \ref{opcalc} below.)

\begin{theorem}[the operational calculus, weak form]\label{weakopcalc}%
\index{concepts}{operational calculus}%
\index{concepts}{calculus!operational}%
Let $A$ be a C*-algebra, and let $a$ be a \underline{Hermitian}
element of $A$. Denote by $C\bigl(\s(a)\bigr)\0$ the C*-subalgebra
of $C\bigl(\s(a)\bigr)$ consisting of the continuous complex-valued functions
on $\s(a)$ vanishing at $0$. There is a unique \st-algebra homomorphism
\[ C\bigl(\s(a)\bigr)\0 \to A \]
which maps the identity function on $\s(a)$ to $a$. This mapping is an
isomorphism of C*-algebras from $C\bigl(\s(a)\bigr)\0$ onto the closed
subalgebra of $A$ generated by $a$. (Which also is the C*-subalgebra
of $A$ generated by $a$.) This mapping is called the
\underline{operational calculus} for $a$, and it is denoted by $f \mapsto f(a)$.
\end{theorem}

\begin{proof}
Let $\mathds{C}[x]\0$ denote the set of complex polynomials without
\linebreak constant term. For $p \in \mathds{C}[x]\0$, we shall denote by
$p|_{\s(a)} \in C\bigl(\s(a)\bigr)\0$ the restriction to $\s(a)$ of the corresponding
polynomial function. We can define an isometric mapping
$p|_{\text{\small{$\s(a)$}}} \mapsto p(a)$
$(p \in \mathds{C}[x]\0)$. Indeed, for $p \in \mathds{C}[x]\0$,
the element $p(a) \in A$ is normal, and one calculates
\begin{align*}
\| \,p(a) \,\| & = \rlambda\,\bigl( \,p(a) \,\bigr) \tag*{by \ref{C*rl}} \\
 & = \max \ \bigl\{ \,| \,\lambda \,| : \lambda \in \s \,\bigl( \,p(a) \,\bigr) \,\bigr\}
\tag*{by \ref{specradform}} \\
 & = \max \ \bigl\{ \,| \,p(\mu) \,| : \mu \in \s(a) \,\bigr\} \tag*{by \ref{ratspecmapthm}} \\
 & = \bigl| \,p|_{\text{\small{$\s(a)$}}} \,\bigr| _{\infty}.
\end{align*}
Thus, if $p|_{\text{\small{$\s(a)$}}} = q|_{\text{\small{$\s(a)$}}}$ for
$p,q \in \mathds{C}[x]\0$, then
$\bigl| \,(p-q)|_{\text{\small{$\s(a)$}}} \,\bigr| _{\infty} = 0$,
whence $\| \,(p-q) (a) \,\| = 0$, so $p(a) = q(a)$, and the mapping in
question is well-defined. The mapping in question is isometric as well.

The extended Stone-Weierstrass Theorem \ref{StWcomp} implies that
the set of complex polynomial functions vanishing at $0$ is dense in
$C\bigl(\s(a)\bigr)\0$. This together with the continuity result \ref{Cstarcontr}
already implies the uniqueness statement. The above mapping has a unique
extension to a continuous mapping from $C\bigl(\s(a)\bigr)\0$ to $A$. This
continuation then is a C*-algebra isomorphism from $C\bigl(\s(a)\bigr)\0$
onto the closed subalgebra of $A$ generated by $a$. It is used that the
continuation is isometric, which implies that its image is complete, hence
closed in $A$. \pagebreak
\end{proof}

\begin{corollary}\label{calcplus}%
Let $a$ be a Hermitian element of a C*-algebra.
If $f \in C\bigl(\s(a)\bigr)\0$ satisfies $f \geq 0$ pointwise on $\s(a)$,
then $f(a)$ is Hermitian and $\s\bigl(f(a)\bigr) \subset [ \,0, \infty \,[$.
\end{corollary}

\begin{proof}
Let $g \in C\bigl(\s(a)\bigr)\0$ with $f = {g \,}^2$ and $g \geq 0$ on $\s(a)$.
Then $g$ is Hermitian in $C\bigl(\s(a)\bigr)\0$ as $g$ is real-valued.
Hence $g(a)$ is Hermitian, \linebreak and thus
$\s \bigl(g(a)\bigr) \subset \mathds{R}$. So $f(a) = {g(a) \,}^2$ is Hermitian
and $\s \bigl(f(a)\bigr) = \s {\bigl(g(a)\bigr) \,}^2 \subset [ \,0, \infty \,[$ by the
Rational Spectral Mapping Theorem.
\end{proof}

\begin{theorem}\index{concepts}{square root}\label{C*sqroot}%
Consider a C*-algebra $A$, and let $a$ be a Hermitian element of $A$
with $\s(a) \subset [ \,0, \infty \,[$. Then $a$ has a unique Hermitian square
root with non-negative spectrum in $A$. This square root of $a$ belongs
to the closed subalgebra of $A$ generated by $a$. (Which also is the
C*-subalgebra of $A$ generated by $a$.) See also \ref{C*sqrootrestate} below.
\end{theorem}

\begin{proof}
For existence, consider $f \in C\bigl(\s(a)\bigr)\0$ given by $f(t) := \sqrt{t}$ for
$t \in \s(a) \subset [ \,0, \infty \,[$. Put ${a \,}^{1/2} := f(a)$, and apply \ref{calcplus}.
Now for uniqueness. Let $b$ be a Hermitian square root of $a$ with
$\s(b) \subset [ \,0, \infty \,[$. We have to show that $b = {a \,}^{1/2}$. Let $c$ be a
Hermitian square root of ${a \,}^{1/2}$ in the closed subalgebra of $A$
generated by ${a \,}^{1/2}$. Let $d$ be a Hermitian square root of $b$. Then
$d$ commutes with $b = {d \,}^2$, hence also with $a = {b \,}^2$. But then $d$
also commutes with ${a \,}^{1/2}$ and $c$. This is so because ${a \,}^{1/2}$ lies
in the closed subalgebra of $A$ generated by $a$ (by construction of
${a \,}^{1/2}$), and because $c$ is in the closed subalgebra of $A$ generated
by ${a \,}^{1/2}$ (by assumption on $c$). We conclude that the Hermitian
elements $a, {a \,}^{1/2}, b, c, d$ all commute. One now calculates
\begin{align*}
0 = & \,\bigl( a - {b \,}^{2\mspace{2mu}} \bigr) \,\bigl( {a \,}^{1/2} - b \bigr) \\
  = & \,\bigl( {a \,}^{1/2} - b \bigr) \,{a \,}^{1/2}
         \,\bigl( {a \,}^{1/2} - b \bigr) + \bigl( {a \,}^{1/2}  -b \bigr) \,b \,\bigl( {a \,}^{1/2} - b \bigr) \\
  = & \,{\bigl[ \,\bigl( {a \,}^{1/2} - b \bigr) \,c \,\bigr] \,}^2+{\bigl[ \,\bigl( {a \,}^{1/2} - b \bigr) \,d \,\bigr] \,}^2.
\end{align*}
Both terms in the last line have non-negative spectrum. Indeed, both terms
are squares of Hermitian elements. Please note \ref{Hermprod} here.
Since the two terms are opposites one of another, both must have spectrum
$\{0\}$, and thence must both be zero, cf.\ \ref{C*rl}. The difference of these
terms is
\[ 0 = \bigl( {a \,}^{1/2} - b \bigr) \,{a \,}^{1/2} \,\bigl({a \,}^{1/2} - b \bigr)
- \bigl( {a \,}^{1/2} - b \bigr) \,b \,\bigl( {a \,}^{1/2} - b \bigr)
= { \bigl( {a \,}^{1/2} - b \bigr) \,}^3. \]
With the C*-property \ref{preC*alg} it follows that
\[ 0 = \| \,{\bigl( {a \,}^{1/2} - b \bigr) \,}^4 \,\| = {\| \,{ \bigl( {a \,}^{1/2} - b \bigr) \,}^2 \,\|\,}^2
= {\| \,{a \,}^{1/2} - b \,\|\,}^4. \pagebreak \qedhere \]
\end{proof}

\clearpage


\section{Odds and Ends: Questions of Imbedding}%
\label{questimbed}

This paragraph is concerned with imbedding for example
a normal element of a normed \st-algebra in a commutative
closed \st-subalgebra. The problem is posed by
possible discontinuity of the involution.

\begin{definition}[\st-stable (or self-adjoint) and normal subsets]%
\label{selfadjointsubset}\index{concepts}{normal!subset}%
\index{concepts}{self-adjoint!subset}%
\index{concepts}{s054@\protect\st-stable}%
\index{concepts}{subset!self-adjoint}%
\index{concepts}{subset!s-stable@\protect\st-stable}%
\index{concepts}{subset!normal}%
Let $S$ be a subset of a \st-algebra. One says that $S$ is
\underline{\st-stable} or \underline{self-adjoint} if with each
element $a$, it also contains the adjoint $a^*$. One says
that $S$ is \underline{normal}, if it is \st-stable and if its
elements commute pairwise.
\end{definition}

\begin{definition}[the commutant, $S'$]%
\index{concepts}{commutant}%
Let $S$ be a subset of an \linebreak algebra $A$.
The \underline{commutant of $S$ in $A$}
is defined as the set of those elements of $A$,
which commute with every element of $S$.
It is denoted by $S'$.
\end{definition}

The reader will easily prove the following two observations.

\begin{observation}\label{commutantp1}%
Let $S$ be a subset of an algebra $A$.
The commutant $S'$ of $S$ in $A$ enjoys the following properties:
\begin{itemize}
   \item[$(i)$] the commutant $S'$ is a subalgebra of $A$,
  \item[$(ii)$] if $A$ is unital, then $S'$ is a unital subalgebra of $A$,
 \item[$(iii)$] if $A$ is a normed algebra, then $S'$ is closed in $A$.
\end{itemize}
If $A$ furthermore is a \st-algebra, one also has:
\begin{itemize}
 \item[$(iv)$] if $S$ is \st-stable, then $S'$ is a \st-subalgebra of $A$.
\end{itemize}
\end{observation}

\begin{observation}\label{commutantp2}%
Let $S,T$ be subsets of an algebra. We then have
\begin{itemize}
   \item[$(i)$] $S \subset T \Rightarrow T' \subset S'$,
  \item[$(ii)$] $S \subset T \Rightarrow S'' \subset T''$,
 \item[$(iii)$] $S \subset S''$,
 \item[$(iv)$] $S'= S'''$,
  \item[$(v)$] the elements of $S$ commute if and only if $S \subset S'$.
\end{itemize}
\end{observation}

\begin{proof}
The item (iv): from (iii), we have $S \subset S''$, and it follows by (i)
that $(S'')' \subset S'$. From (iii) again, one also has $S' \subset (S')''$.
\end{proof}

\begin{definition}[the second commutant, $S''$]%
\index{concepts}{commutant!second}%
Let $S$ be a subset of an algebra. The subset $S''$
is called the \underline{second commutant} of $S$.
\pagebreak
\end{definition}

From observations \ref{commutantp1} and \ref{commutantp2}
above, we immediately get:

\begin{observation}\label{scndcobs}%
Let $S$ be a subset of an algebra $A$.
The second commutant $S''$ of $S$ in $A$ has the following properties:
\begin{itemize}
   \item[$(i)$] the second commutant $S''$ is a subalgebra of $A$ containing $S$,
  \item[$(ii)$] if $A$ is unital, then $S''$ is a unital subalgebra of $A$,
 \item[$(iii)$] if $A$ is a normed algebra, then $S''$ is closed in $A$,
 \item[$(iv)$] the elements of $S$ commute if and only if $S''$ is commutative.
\end{itemize}
If $A$ furthermore is a \st-algebra, we also have:
\begin{itemize}
  \item[$(v)$] if $S$ is \st-stable, then $S''$ is a \st-subalgebra of $A$,
 \item[$(vi)$] if $S$ is normal, then $S''$ is a commutative \st-subalgebra of $A$.
\end{itemize}
\end{observation}

\begin{proof}
Item (iv): the ``only if'' part follows from \ref{commutantp2} (v) \& (ii),
while the ``if'' part follows from \ref{commutantp2} (iii).
\end{proof}

\begin{observation}\label{inv2comm}%
If $a$ is an invertible element of a unital algebra $A$, then its inverse
${a\,}^{-1}$ belongs to the second commutant $\{ \,a \,\}''$ in $A$.
\end{observation}

\begin{proof}
See \ref{comm2}.
\end{proof}

\begin{lemma}\label{scndcommspec}%
Let $S$ be a subset of an algebra $A$. Let $B$ denote the second
commutant of $S$ in $\tld{A}$. We then have
\[  \s_{B}(b) = \s_{A}(b) \quad \text{for all} \quad b \in B \cap A. \]
\end{lemma}

\begin{proof}
Let $b \in B \cap A$. We know that $\s_A(b) = \s_{\tld{A}}(b)$,
cf.\ \ref{specunit}. We also have $\s_{\tld{A}}(b) \subset \s_B(b)$,
cf.\ \ref{specsubalg}. So assume that (for some $\lambda \in \mathds{C}$)
the element $\lambda e - b$ is invertible in $\tld{A}$ with inverse
$d \in \tld{A}$. It is enough to prove that $d \in B$. Since $b \in B$,
we have $d \in \{ \,b \,\}'' \subset B'' = B$, cf.\ \ref{inv2comm} and
\ref{commutantp2} (ii) \& (iv), as $B$ is a (second) commutant.
\end{proof}

\begin{theorem}[Civin \& Yood]\index{concepts}{Civin}\label{CivinYood}%
\index{concepts}{Theorem!Civin \& Yood}\index{concepts}{Yood}%
Let $A$ be a normed algebra, and let $S$ be a subset consisting of
commuting elements of $A$. There then exists a commutative closed
unital subalgebra $B$ of $\tld{A}$ containing $S$ such that
\[ \s_{B}(b) = \s_{A}(b) \quad \text{for all} \quad b \in B \cap A. \]
If furthermore $A$ is a \st-algebra, and if $S$ is normal, then $B$
can be chosen to be a \st-subalgebra of $\tld{A}$.
\end{theorem}

\begin{proof}
Lemma \ref{scndcommspec} \& observation \ref{scndcobs}. \pagebreak
\end{proof}

\medskip
This often allows one to reduce proofs to the case of commutative unital
Banach algebras. (See for example \ref{specsub} and \ref{instability} below.)

\medskip
We note the following for reference in the ensuing definition.

\begin{observation}\label{normalscndcomm}%
A normal subset $S$ of a normed \st-algebra $A$ is contained
in a commutative closed \st-subalgebra of $A$. The second
commutant $S''$ of $S$ in $A$ is an instance of such a commutative
closed \st-subalgebra of $A$ containing the normal subset $S$.
\end{observation}

\begin{proof}
See observation \ref{scndcobs}.
\end{proof}

\begin{definition}\label{clogen}%
Let $S$ be a subset of a normed \st-algebra $A$.
The \underline{closed \st-subalgebra of $A$ generated by $S$}
is defined as the intersection of all closed \st-subalgebras of $A$
containing $S$. It is commutative if $S$ is normal, by the
preceding observation \ref{normalscndcomm}.
\end{definition}

Please note that when the involution in $A$ is continuous, then the
closed \st-subalgebra of $A$ generated by a subset $S$ is the closure
of the \st-subalgebra of $A$ generated by $S$. This holds in particular
for \linebreak C*-algebras, and we shall make tacit use of this fact in the
sequel.

\begin{remark}\label{Zorn}
So far we did not use the full version of the Axiom of Choice. (We
did use the countable version of the Axiom of Choice, however.
Namely in using the fact that if $S$ is a subset of a metric space,
then a point in the closure of $S$ is the limit of a sequence in $S$.)
The next chapter \ref{TheGel'fandTransformation} on the Gel'fand
transformation will make crucial use of Zorn's Lemma, which is
equivalent to the full version of the Axiom of Choice. The ensuing
chapter \ref{PositiveElements} on positive elements, however, is
completely independent of chapter \ref{TheGel'fandTransformation},
and does not require Zorn's Lemma. From a logical point of view,
one might well skip the following chapter \ref{TheGel'fandTransformation},
if the priority is to use Zorn's Lemma as late as possible.
\end{remark}

\clearpage


\chapter{The Gel'fand Transformation}%
\label{TheGel'fandTransformation}


\setcounter{section}{15}

\section{Multiplicative Linear Functionals}

\begin{definition}[multiplicative linear functionals, $\Delta(A)$]%
\index{concepts}{multiplicative!linear functional}%
\index{symbols}{D(A)@$\Delta(A)$}%
\index{concepts}{functional!multiplicative}%
If $A$ is an algebra, then a \underline{multiplicative linear functional}
on $A$ is a \underline{non-zero} linear functional $\tau$ on $A$
such that $\tau (ab) = \tau (a) \,\tau (b)$ for all $a,b \in A$. The set of
multiplicative linear functionals on $A$ is denoted by $\Delta (A)$.

Please note that if $A$ is unital with unit $e$, then $\tau (e) = 1$ for any
$\tau \in \Delta (A)$ by $\tau (e) \,\tau (a) = \tau (a)$ for some $a \in A$ with
$\tau (a) \neq 0$.
\end{definition}

\begin{proposition}\label{mlfbounded}%
Every multiplicative linear functional $\tau$ on a Banach
algebra $A$ is continuous with $| \,\tau \,| \leq 1$.
\end{proposition}

\begin{proof} Otherwise there would exist $a \in A$ with
$| \,a \,| < 1$ and $\tau (a) = 1$. Putting $b := \sum _{n=1}^{\infty } a^{n}$,
we would then get $a+ab = b$, and therefore $\tau (b) = \tau (a+ab)
= \tau (a) + \tau (a) \tau (b) = 1+ \tau (b)$, a contradiction.
\end{proof}

\begin{definition}[the Gel'fand transform of an element]%
\index{concepts}{Gel'fand!transform}%
\index{symbols}{a07@$\protect\wht{a}$}%
Let $A$ be an algebra. For $a \in A$, one defines
\[ \wht{a}(\tau) := \tau(a) \qquad \bigl( \,\tau \in \Delta (A) \,\bigr). \]
The function $\wht{a} : \tau \mapsto \wht{a}(\tau)$
$( \mspace{1mu}\tau \in \Delta(A) \mspace{1mu})$
is called the \underline{Gel'fand transform} of $a$.
\end{definition}

\begin{proposition}[the Gel'fand transformation]%
\index{concepts}{Gel'fand!transformation}%
If $A$ is an algebra with $\Delta(A) \neq \varnothing$, then the map
\begin{align*}
A & \to \mathds{C}^{\,\text{\Small{$\Delta(A)$}}} \\
a & \mapsto \wht{a}
\end{align*}
is an algebra homomorphism. It is called the \underline{Gel'fand transformation}.
\end{proposition}

\begin{proof}
For $a, b \in A$, and $\tau \in \Delta (A)$, we have for example that
\[ \wht{(ab)} ( \tau ) = \tau (ab) = \tau (a) \cdot \tau (b)
= \wht{\vphantom{b}a}\,(\tau) \cdot \wht{b}\,(\tau)
= \bigl( \,\wht{\vphantom{b}a} \cdot \wht{b} \,\bigr) (\tau). \qedhere \]
\end{proof}

\begin{proposition}\label{GelfTransf}%
If $A$ is a Banach algebra with $\Delta(A) \neq \varnothing$,
then the Gel'fand transformation is a contractive algebra homomorphism
\[ A \to {\ell\,}^{\infty} \bigl(\Delta (A)\bigr). \]
\end{proposition}

\begin{proof}
We have $\bigl| \,\wht{a} ( \tau ) \,\bigr| = | \,\tau (a) \,| \leq | \,a \,|$
by \ref{mlfbounded}, so $\bigl| \,\wht{a} \,\bigr| _{\infty} \leq | \,a \,|$.
\end{proof}

\medskip
We shall relate the range of a Gel'fand transform to the spectrum.
For example, we have:

\begin{proposition}
For an algebra $A$, and $a \in A$, one has
\[ \wht{a} \,\bigl( \Delta (A) \bigr) \,\subset \,\s (a). \]
\end{proposition}

\begin{proof}
We have to prove that for $\tau \in \Delta (A)$, one has $\tau (a) \in \s (a)$.
This, however, follows from \ref{spechom} because we can interpret $\tau$
as a homomorphism from $A$ to the algebra $\mathds{C}$. It is used that
$\tau$ is unital if $A$ is unital, and that $0 \in \s(a)$ for all $a \in A$ otherwise,
cf.\ \ref{speczero}.
\end{proof}

\medskip
The converse inclusion holds in a commutative unital Banach
\linebreak algebra, see \ref{rangeGTu} below.

We need to relate the spectrum to ideals, to be discussed next.

\begin{definition}[ideals]\index{concepts}{ideal!left}%
\index{concepts}{ideal}\index{concepts}{ideal!proper}%
\index{concepts}{ideal!maximal}%
Let $A$ be an algebra. A \underline{left ideal} in $A$ is a vector subspace
$I$ of $A$ such that $aI \subset I$ for all $a \in A$. A left ideal $I$ of $A$
is called \underline{proper} if $I \neq A$. A left ideal of $A$ is called a
\underline{maximal left ideal} or simply \underline{left maximal} if it is
proper and not properly contained in any other proper left ideal of $A$.
Likewise with ``right'' or ``two-sided'' instead of ``left''. If $A$ is commutative,
we drop these adjectives.
\end{definition}

Ideals are related to multiplicative linear functionals in the following way.
The kernel $\ker \tau$ of a multiplicative linear functional $\tau$ on an
algebra is a proper two-sided ideal, which is both left maximal and right
maximal. (Because $\ker \tau$ has co-dimension one.) The converse
holds in a unital Banach algebra, see \ref{taumaxid} below.

\begin{observation}
Let $A$ be an algebra, and let $I$ be a two-sided ideal in $A$.
The vector space $A\mspace{1mu}/I := \{ \,a+I \subset A : a \in A \,\}$ is an
algebra when equipped with the multiplication $(a+I) \cdot (b+I) := ab+I$
$(a, b \in A)$.
\end{observation}

Item (ii) of the next lemma serves to relate the spectrum to ideals.\pagebreak

\begin{lemma}\label{notinvideal}%
If $A$ is a unital algebra with unit $e$, then:
\begin{itemize}
   \item[$(i)$] a left ideal $I$ in $A$ is proper if and only if $e \notin I$,
  \item[$(ii)$] an element of $A$ is not left invertible in $A$ if and \\
only if it lies in some proper left ideal of $A$.
\end{itemize}
If furthermore $A$ is a unital Banach algebra, then:
\begin{itemize}
 \item[$(iii)$] if a left ideal $I$ of $A$ is proper, so is its closure $\overline{I}$,
 \item[$(iv)$] a maximal left ideal of $A$ is closed in $A$.
\end{itemize}
\end{lemma}

\begin{proof}
(i) is obvious. (ii) Assume that $a$ lies in a proper left ideal $I$ of $A$.
If $a$ had a left inverse $b$, one would have $e = ba \in I$, so that $I$
would not be proper. Conversely, assume that $a$ is not left invertible.
Consider the left ideal $I$ given by $I := \{ ba : b \in A \}$. Then
$e \notin I$, so that $a$ lies in the proper left ideal $I$. (iii) follows from
the fact that the set of left invertible elements of a unital Banach algebra
is an open neighbourhood of the unit, cf.\ \ref{leftinv}. (iv) follows from (iii).
\end{proof}

\begin{lemma}\label{inmaxideal}%
A proper left ideal in a unital algebra is contained in a maximal left ideal.
\end{lemma}

\begin{proof}
Let $I$ be a proper left ideal in a unital algebra $A$. Let $Z$ be the set of
proper left ideals in $A$ containing $I$, and order $Z$ by inclusion. It shall
be shown that $Z$ is inductively ordered. So let $C \neq \varnothing$ be a
chain in $Z$. Let $J := \bigcup \,C$. It shall be shown that $J \in Z$. The left
ideal $J$ is proper. Indeed one notes that $J$ does not contain the unit of
$A$, because otherwise some element of $C$ would contain the unit. Now
Zorn's Lemma yields a maximal element of $Z$.
\end{proof}

\begin{lemma}\label{field}%
Let $I$ be a proper two-sided ideal in a unital algebra $A$.
The unital algebra $A\mspace{1mu}/I$ then is a division ring
if and only if $I$ is both a maximal left ideal and a maximal right ideal.
\end{lemma}

\begin{proof}
Please note first that the algebra $A\mspace{1mu}/I$ is unital
because the two-sided ideal $I$ is proper, cf.\ \ref{notinvideal} (i).
For $b \in A \setminus I$, consider the left ideal
\[ Ab+I := \{ \,ab+i : a \in A,\ i \in I \,\}. \]
The left ideal $Ab+I$ contains $b$ as $A$ is unital. It clearly is the
smallest left ideal in $A$ containing both $b$ and $I$. It contains $I$
properly as $b \notin I$ by the assumption on the element $b$.

Hence the following statements are equivalent.
(We shall abbreviate $\underline{a} := a + I \in A\mspace{1mu}/I$
for all $a \in A$.)
\pagebreak

\par each non-zero element $\underline{b}$ of $A\mspace{1mu}/I$
is left invertible (please note \ref{leftrightinv}),
\par for all $\underline{b} \neq \underline{0}$ in $A\mspace{1mu}/I$
there exists $\underline{c}$ in $A\mspace{1mu}/I$ with
$\underline{c} \,\underline{b} = \underline{e}$,
\par for all $b \in A \setminus I$ there exist $c \in A$,
$i \in I$ with $cb + i = e$,
\par for all $b \in A \setminus I$ one has $e \in Ab+I$,
\par for all $b \in A \setminus I$ one has $Ab+I = A$,
\par for each left ideal $J$ containing $I$ properly, one has $J = A$,
\par the proper two-sided ideal $I$ is a maximal left ideal.
\end{proof}

\begin{theorem}\label{taumaxid}%
Let $A$ be a unital Banach algebra. The mapping $\tau \mapsto \ker \tau$
establishes a bijective correspondence between the set $\Delta (A)$
of multiplicative linear functionals $\tau$ on $A$ and the set of proper
two-sided ideals in $A$ which are both left maximal and right maximal.
\end{theorem}

\begin{proof}
The kernel of any $\tau \in \Delta (A)$ is a proper two-sided ideal in $A$,
and has co-dimension $1$ in $A$, so it is both left and right maximal.
Conversely, let $I$ be a proper two-sided ideal in $A$, which is both
left and right maximal. Then $A\mspace{1mu}/I$ is a normed algebra
by \ref{notinvideal} (iv), (cf.\ the appendix \ref{quotspace}), and a
division ring by \ref{field}. So it is isomorphic to $\mathds{C}$
as an algebra by the Gel'fand-Mazur Theorem \ref{Gel'fandMazur}.
Let $\pi : A\mspace{1mu}/I \to \mathds{C}$ be an isomorphism,
and put $\tau (a) := \pi (a+I)$ $(a \in A)$. Then $\tau \in \Delta (A)$
and $\ker \tau =I$. To see injectivity of $\tau \mapsto \ker \tau$,
let $\tau _1,\tau _2 \in \Delta (A)$, $\tau _1 \neq \tau _2$.
Let $a \in A$ with $\tau _1 (a) = 1$, $\tau _2 (a) \neq 1$. With
$b := a^2 - a$ we get
$\tau _k (b) = \tau _k (a) \cdot \bigl( \tau _k (a) -1 \bigr)$.
Thus $\tau _1 (b) = 0$, and either $\tau _2 (a) = 0$ or
else $\tau _2 (b) \neq 0$. So either
$b \in \ker \tau _1 \setminus \ker \tau _2$ or
$a \in \ker \tau _2 \setminus \ker \tau _1$.
\end{proof}

\medskip
In the remainder of this paragraph, we shall mainly consider
\linebreak \underline{commutative} Banach algebras. For
emphasis, we restate:

\begin{corollary}\label{taumaxidcomm}%
Let $A$ be a \underline{unital commutative} Banach algebra.
The mapping $\tau \mapsto \ker \tau$ establishes a bijective
correspondence between the set $\Delta (A)$ of multiplicative
linear functionals $\tau$ on $A$ and the set of maximal ideals
in $A$.
\end{corollary}

\begin{theorem}[abstract form of Wiener's Theorem]%
\label{abstrWien}%
\index{concepts}{Theorem!Wiener's!abstract form}%
\index{concepts}{Wiener's Theorem!abstract form}%
\index{concepts}{abstract!form of Wiener's Thm.}%
Let $A$ be a \underline{unital commutative} Banach algebra.
An element $a$ of $A$ is not invertible in $A$ if and only if the
function $\wht{a}$ vanishes at some point $\tau$ in $\Delta (A)$.
In other words, for an element $a$ of $A$ to have an
inverse, it is necessary and sufficient that the function
$\wht{a}$ should not vanish anywhere on $\Delta (A)$.
\end{theorem}

\begin{proof}
\ref{notinvideal}(ii), \ref{inmaxideal}, \ref{taumaxidcomm}. \pagebreak
\end{proof}

\begin{theorem}\label{rangetld}%
If $A$ is a commutative Banach algebra, then for $a \in \tld{A}$ we have
\[ \s(a) = \wht{a}\,\bigl(\Delta ( \tld{A} )\bigr). \]
(Please note the tilde.)
\end{theorem}

\begin{proof}
For $\lambda \in \mathds{C}$, the following statements are equivalent.
\begin{align*}
& \lambda \in \s(a), \\
& \lambda e-a \text{ is not invertible in } \tld{A}, \\
& \wht{ ( \lambda e - a ) } \,\bigl( \tld{ \tau } \bigr) = 0
\text{ for some } \tld{ \tau } \in \Delta (\tld{A}),
\quad \text{by \ref{abstrWien}} \\
& \lambda = \wht{a} \,\bigl( \tld{ \tau } \bigr)
\text{ for some } \tld{ \tau } \in \Delta (\tld{A}). \pagebreak \qedhere
\end{align*}
\end{proof}

\medskip
In particular:

\begin{theorem}\label{rangeGTu}%
If $A$ is a unital commutative Banach algebra, then for $a \in A$ we have
\[ \s(a) = \wht{a}\,\bigl( \Delta ( A ) \bigr). \]
\end{theorem}

\begin{theorem}%
If $A$ is a commutative Banach algebra without unit,
then for $a \in A$ we have
\[ \s(a) = \wht{a}\,\bigl(\Delta(A)\bigr) \cup \{ 0 \}. \]
(Please note \ref{speczero}.)
\end{theorem}

\begin{proof}
For $\lambda \in \mathds{C}$, the following statements are equivalent.
\begin{align*}
& \lambda \in \s(a), \\
& \lambda = \wht{a} \,\bigl( \tld{ \tau } \bigr)
\text{ for some } \tld{ \tau } \in \Delta (\tld{A}),
\quad \text{by \ref{rangetld}} \\
& \text{either } \lambda = 0 \text{ or } \lambda = \wht{a} \,( \tau )
\text{ for some } \tau \in \Delta (A), \\
& \text{(according as } \tld{ \tau } \text{ vanishes on all of } A
\text{ or not).} \qedhere
\end{align*}
\end{proof}

\medskip
The two preceding items together yield:

\begin{corollary}\label{rangeGT}%
If $A$ is a commutative Banach algebra, then for $a \in A$ we have
\[ \underline{\s(a) \setminus \{ 0 \}
= \wht{a} \,\bigl( \Delta (A) \bigr) \setminus \{ 0 \}}. \pagebreak \]
\end{corollary}

\begin{theorem}\label{normGTrl}%
For a commutative Banach algebra $A$, and $a \in A$, we have
\[ \bigl| \,\wht{a} \,\bigr|_{\infty} = \rlambda(a). \]
\end{theorem}

\begin{proof}
This follows from the preceding corollary \ref{rangeGT} by applying
the Spectral Radius Formula.
\end{proof}

\begin{corollary}\label{C*normGT}%
For a commutative C*-algebra $A$, and $a \in A$, we have
\[ \bigl| \,\wht{a} \,\bigr| _{\infty} = \|\,a\,\|. \]
\end{corollary}

\begin{proof}
By \ref{C*rl} we have $\rlambda (a) = \|\,a\,\|$.
\end{proof}

\begin{corollary}\label{C*GTisom}%
If $A$ is a commutative C*-algebra $\neq \{ 0 \}$, then
$\Delta (A) \neq \varnothing$ and the Gel'fand transformation
$A \to {\ell\,}^{\infty}\bigl(\Delta(A)\bigr)$ is an isometry. 
\end{corollary}

\begin{definition}[Hermitian functional]%
\label{Hermfunct}\index{concepts}{functional!Hermitian}%
\index{concepts}{Hermitian!functional}%
Any linear functional $\varphi$ on any \st-algebra $A$ is said to be
\underline{Hermitian} if it satisfies
\[ \varphi (a^*) = \overline{\varphi (a)} \quad \text{for all} \quad a \in A. \]
\end{definition}

\begin{theorem}\label{mlfHerm}%
A commutative Banach \st-algebra $A$ is Hermitian
precisely when each $\tau \in \Delta (A)$ is Hermitian.
\end{theorem}

\begin{proof}
Let $a \in A$. We then have
$\{ \tau (a) : \tau \in \Delta (A) \} \setminus \{ 0 \} = \s(a) \setminus \{ 0 \}$.
Thus, if every $\tau \in \Delta (A)$ is Hermitian, then $\s(a)$ is real for
Hermitian $a \in A$, so that $A$ is Hermitian. Conversely, if $A$ is
Hermitian, then each $\tau \in \Delta (A)$ assumes real values on
Hermitian elements. Given any $c \in A$, however, we can write
$c = a + \iu b$ with Hermitian $a, b \in A$, so that for $\tau \in \Delta (A)$
we have $\tau (c^*) = \tau (a) - \iu \tau (b)
= \overline{\tau (a) + \iu \tau (b)} = \overline{\tau (c)}$.
\end{proof}

\begin{corollary}\label{Hermstarhom}%
Let $A$ be a \underline{Hermitian} commutative Banach \linebreak
\st-algebra with $\Delta(A) \neq \varnothing$. The Gel'fand transformation
then is a \linebreak \st-algebra homomorphism. In particular, the range of
the Gel'fand \linebreak transformation then is a \st-subalgebra of
${\ell\,}^{\infty}\bigl(\Delta(A)\bigr)$ (not merely a \linebreak subalgebra).
\end{corollary}

Now a miscellaneous application of multiplicative linear functionals:%
\pagebreak

\begin{proposition}\label{commspecsub}%
Let $A$ be a commutative Banach algebra. For two elements $a,b$
of $A$ we have
\begin{align*}
\s(a+b) \subset & \ \s(a)+\s(b) \\
\s(ab) \subset & \ \s(a) \,\s(b).
\end{align*}
\end{proposition}

\begin{proof}
This is an application of \ref{rangetld}. Indeed, if $\lambda \in \s(ab)$,
there exists $\tld{\tau} \in \Delta (\tld{A})$ such that
$\tld{\tau} (ab) = \lambda$. But then $\lambda = \tld{\tau} (a) \,\tld{\tau} (b)$
where $\tld{\tau} (a) \in \s(a)$ and $\tld{\tau} (b) \in \s(b)$.
\end{proof}

\begin{theorem}\label{specsub}%
Let $A$ be a Banach algebra. If $a$, $b$ are \underline{commuting}
elements of $A$, then
\begin{align*}
\s(a+b) \subset & \ \s(a)+\s(b) \\
\s(ab) \subset & \ \s(a) \,\s(b).
\end{align*}
\end{theorem}

\begin{proof}
This follows from the preceding proposition \ref{commspecsub} by an
\linebreak application of the Theorem of Civin \& Yood \ref{CivinYood}.
Indeed, this result says that there exists a commutative
closed subalgebra $B$ of $\tld{A}$ con\-taining $a$ and $b$,
such that $\s _A (c) = \s _B (c)$ for $c \in \{ \,a, b, a+b, ab \,\}$.
\end{proof}

\medskip
For the preceding theorem \ref{specsub}, see also \ref{commrlsub}
and \ref{commspeccont}.

\begin{proposition}\label{Hermcomm}%
Let $A$ be a \underline{Hermitian} Banach $*$-algebra,
and let $b \in A$ be \underline{normal}. One then has
\[ \s ({b}^{*} b) = {| \,\s(b) \,|}^{\,2}. \]
In particular $\s({b}^{*} b) \subset [ \,0, \infty \,[$.
\end{proposition}

\begin{proof}
The closed \st-subalgebra $B$ of $\tld{A}$ generated by $b$, ${b}^{*}$,
and $e$ \ref{clogen} is an automatically Hermitian \ref{Herminher}
complete unital \st-sub\-algebra of $\tld{A}$ with $\s _A (c) = \s _B (c)$ for
$c \in \{ \,b, {b}^{*} b \,\}$, cf.\ \ref{specHermsubalg}. The algebra $B$ is
commutative because $b$ is normal, cf \ref{clogen}. By \ref{rangeGTu} and
\ref{mlfHerm}, the following statements are equivalent for $\mu \in \mathds{C}$.
\begin{align*}
 & \mu \in \s_B({b}^{*}b), \\
 & \text{there exists} \ \tau \in \Delta(B) \ \text{with} \ \mu = \tau ({b}^{*} b), \\
 & \text{there exists} \ \tau \in \Delta(B) \ \text{with} \ \mu = {| \,\tau (b) \,|}^{\,2}, \\
 & \mu \in {| \,\s_B(b) \,|}^{\,2}. \qedhere
\end{align*}
\end{proof}

Finally an application to harmonic analysis:
\pagebreak

\begin{example}%
Consider the Banach \st-algebra ${\ell\,}^1 (\mathds{Z})$,
cf.\ example \ref{l1G}. We are interested in knowing the set of
multiplicative linear functionals of ${\ell\,}^1(\mathds{Z})$. To this
end we note that a multiplicative linear functional on
${\ell\,}^1(\mathds{Z})$ is determined by its value at $\delta _1$.
So let $\tau$ be a multiplicative linear functional on
${\ell\,}^1(\mathds{Z})$ and put $\alpha := \tau (\delta _1)$.
We then have $|\,\alpha \,| \leq 1$ by \ref{mlfbounded}. But as
$\delta _{-1}$ is the inverse of $\delta _1$, we also have
$\tau (\delta _{-1}) = \alpha ^{-1}$ and $|\,\alpha ^{-1}\,| \leq 1$.
This makes that $\alpha$ belongs to the unit circle:
$\alpha = {\mathrm{e}}^{\,-\iu t}$ for some
$t \in \mathds{R} / 2 \pi \mathds{Z}$. For an integer $k$, we obtain
$\tau ( \delta _k ) = \bigl( \tau ( \delta _1 ) \bigr)^k = {\mathrm{e}}^{\,-\iu kt}$.
Proposition \ref{mlfbounded} implies that for $a \in {\ell\,}^1(\mathds{Z})$, we get
\[ \wht{a}(\tau) = \tau(a) = \sum _{k \in \mathds{Z}} a(k) \,{\mathrm{e}}^{\,-\iu kt}. \]
Conversely, by putting for some $t \in \mathds{R} / 2 \pi \mathds{Z}$
\[ \tau (a) := \sum _{k \in \mathds{Z}} a(k) \,{\mathrm{e}}^{\,-\iu kt} \qquad ( \,a \in {\ell\,}^1 (\mathds{Z}) \,), \]
a multiplicative linear functional $\tau$ on ${\ell\,}^1(\mathds{Z})$ is defined.
Indeed, for all $a, b \in {\ell\,}^1 (\mathds{Z})$ we find using \ref{CG}
\begin{align*}
\tau (ab) & = \sum _{k \in \mathds{Z}} (ab)(k) \,{\mathrm{e}}^{\,-\iu kt}
            = \sum _{k \in \mathds{Z}} \,\sum _{l+m=k} a(l) \,b(m) \,{\mathrm{e}}^{\,-\iu kt} \\
          & = \biggl( \,\sum _{l \in \mathds{Z}} a(l) \,{\mathrm{e}}^{\,-\iu lt} \biggr) \cdot
              \biggl( \,\sum _{m \in \mathds{Z}} b(m) \,{\mathrm{e}}^{\,-\iu mt} \biggr)
            = \tau (a) \cdot \tau (b).
\end{align*}
As a multiplicative linear functional is determined by its value at
$\delta _1$, it follows that $\Delta \bigl({ \ell\,}^1 (\mathds{Z})\bigr)$
can be identified with $\mathds{R} / 2 \pi \mathds{Z}$. Also,
a function on $\mathds{R} / 2 \pi \mathds{Z}$ is the Gel'fand transform
of an element of ${\ell\,}^1 (\mathds{Z})$ if and only if it has an
absolutely convergent Fourier expansion.%
\end{example}

From the abstract form of Wiener's Theorem \ref{abstrWien}, it follows now:

\begin{theorem}[Wiener]%
\index{concepts}{Theorem!Wiener's}\index{concepts}{Wiener's Theorem}%
Let $f$ be a function on $\mathds{R} / 2 \pi \mathds{Z}$
which has an absolutely convergent Fourier expansion. If
$f(t) \neq 0$ for all \linebreak $t \in \mathds{R} / 2 \pi \mathds{Z}$,
then also $1/f$ has an absolutely convergent Fourier
\linebreak expansion.
\end{theorem}

This easy proof of Wiener's Theorem was an early success
of the theory of Gel'fand.
Using \ref{mlfHerm}, the reader can now check that
${\ell\,}^1 (\mathds{Z})$ is a Hermitian Banach \st-algebra.
We shall return to the non-commutative theory in \ref{leftspectrum}.%
\pagebreak

\clearpage


\section{The Gel'fand Topology}%
\label{Gtopo}

\medskip
In this paragraph, let $A$ be a \underline{Banach} algebra.

\begin{definition}[the Gel'fand topology, the spectrum]%
\index{concepts}{Gel'fand!topology}%
\index{concepts}{topology!Gel'fand}\label{Gelftopo}%
\index{concepts}{spectrum!of an algebra}\index{symbols}{D(A)@$\Delta(A)$}%
We imbed the set $\Delta (A)$ in the closed unit ball, cf.\ \ref{mlfbounded},
of the dual normed space of $A$ equipped with the weak* topology,
cf.\ the appendix \ref{weak*top}. The relative topology on $\Delta (A)$
is called the \underline{Gel'fand topology}. When equipped with the
Gel'fand topology, $\Delta (A)$ is called the \underline{spectrum} of $A$.
\end{definition}

\begin{proposition}%
We have:
\begin{itemize}
  \item[$(i)$] the weak* closure $\overline{\Delta(A)}$ is a compact Hausdorff space.
 \item[$(ii)$] the compact Hausdorff space $\overline{\Delta (A)}$ is contained in $\Delta (A) \cup \{ 0 \}$.
\end{itemize}
\end{proposition}

\begin{proof}
(i) follows from Alaoglu's Theorem, cf.\ the appendix \ref{Alaoglu}.
(ii): every adherent point $\tau$ of $\Delta (A)$ is a linear functional
satisfying $\tau(ab) = \tau(a) \,\tau(b)$ for all $a, b \in A$.
Indeed, the weak* topology is just the topology of pointwise
convergence on $A$, cf.\ the appendix \ref{weak*point}.
\end{proof}

\medskip
Hence:

\begin{theorem}\label{mlf0compact}%
Either the Hausdorff space $\Delta(A)$ is compact or
else $\overline{\Delta(A)} = \Delta(A) \cup \{ 0 \}$,
in which case $\Delta (A)$ is locally compact and
$\overline{\Delta (A)}$ is a one-point compactification of
$\Delta (A)$, with $0$ the corresponding point at infinity,
cf.\ \ref{onepcomp}.
\end{theorem}

In particular:

\begin{corollary}%
The spectrum $\Delta(A)$ is a locally compact Hausdorff space.
\end{corollary}

\begin{lemma}[abstract Riemann-Lebesgue Lemma]%
\index{concepts}{Theorem!abstract!Riemann-Lebesgue Lemma}%
\index{concepts}{Theorem!Riem.-Leb.\ Lemma, abstract}%
\index{concepts}{abstract!Riemann-Lebesgue Lemma}%
\index{concepts}{Lemma!abstract Riemann-Lebesgue}%
\index{concepts}{Riem.-Leb.\ Lemma, abstract}%
If $\Delta (A) \neq \varnothing$, then for all $a \in A$ we have
\[ \wht{a} \in \cont\0\bigl(\Delta (A)\bigr). \]
\end{lemma}

\begin{proof}
The evaluations $\wht{a}$ $(a \in A)$ are continuous on $\Delta(A) \cup \{ 0 \}$,
by definition of the weak* topology, cf.\ the appendix \ref{weak*top}.
They vanish at $0$, which is the point at infinity of $\Delta (A)$ if the
latter is not compact. Hence the Gel'fand transforms $\wht{a}$ $(a \in A)$
belong to $\cont\0\bigl(\Delta(A)\bigr)$.
\pagebreak
\end{proof}

\medskip
From \ref{GelfTransf} we get:

\begin{corollary}%
If $\Delta (A) \neq \varnothing$, then the
Gel'fand transformation $a \mapsto \wht{a}$
is a contractive algebra homomorphism
\[ A \to \cont\0\bigl(\Delta (A)\bigr). \]
\end{corollary}

\begin{proposition}\label{unitalcomp}%
If $A$ is unital, then $\Delta (A)$ is a compact Hausdorff space.
\end{proposition}

\begin{proof}
If $A$ has a unit $e$, then $\tau (e) = 1$ for all $\tau \in \Delta (A)$,
whence $\sigma (e) = 1$ for all $\sigma \in \overline{\Delta (A)}$, as
the weak* topology is the topology of pointwise convergence on $A$,
cf.\ the appendix \ref{weak*point}. In this case $0$ cannot belong to
$\overline{\Delta (A)}$, so $\Delta (A)$ is compact by theorem
\ref{mlf0compact}.
\end{proof}

\medskip
We shall need the following result.

\begin{proposition}\label{notcomp}%
If $\Delta(A)$ is not compact, then the one-point compactification
of $\Delta(A)$ may be identified with $\Delta(\tld{A})$. More precisely,
in this case, the map which takes any $\tau \in \Delta(\tld{A})$ to its
restriction to $A$, is a homeomorphism from $\Delta(\tld{A})$ onto the
one-point compactification $\overline{\Delta(A)} = \Delta(A) \cup \{ 0 \}$
of $\Delta(A)$.
\end{proposition}

\begin{proof}
Assume that $\Delta(A)$ is not compact. Proposition \ref{mlf0compact}
implies that $\overline{\Delta(A)} = \Delta(A) \cup \{ 0 \}$ is the one-point
compactification of $\Delta(A)$. Since the spectrum of a unital Banach
algebra is compact by the preceding proposition \ref{unitalcomp}, the
algebra $A$ contains no unit, and $\Delta(\tld{A})$ is compact. The
map referred to in the statement takes $\Delta(\tld{A})$ to
$\Delta(A) \cup \{ 0 \}$, because the restriction of any element of
$\Delta(\tld{A})$ either is in $\Delta(A)$ or else vanishes identically
on $A$. Please note next that with $e$ denoting the unit in $\tld{A}$,
we have $\tld{A} = \mathds{C}e \oplus A$ and $\tld{\tau}(e) = 1$
for all $\tld{\tau} \in \Delta(\tld{A})$. In particular, an element of
$\Delta(\tld{A})$ is uniquely determined by its restriction to $A$,
so the map in question is injective. Also, any $\tau \in \Delta(A) \cup \{ 0 \}$
has an extension to some $\tld{\tau} \in \Delta(\tld{A})$ by putting
$\tld{\tau}(e) := 1$ and extending linearly, as is easily verified.
This means that the map in question is surjective. The map in
question is continuous by the universal property of the weak*
topology, cf.\ the appendix \ref{weak*top}. This map thus is a
continuous bijection from the compact space $\Delta(\tld{A})$
onto the Hausdorff space $\Delta(A) \cup \{ 0 \}$, and thereby
a homeomorphism, cf.\ the appendix \ref{homeomorph}.
\pagebreak
\end{proof}

\begin{proposition}\label{Hermdense}%
If $A$ is a \underline{Hermitian} commutative Banach \linebreak
\st-algebra with $\Delta (A) \neq \varnothing$, then the Gel'fand
transformation has dense range in $\cont\0\bigl(\Delta (A)\bigr)$.
\end{proposition}

\begin{proof}
One applies the Stone-Weierstrass Theorem \ref{StW}.
This requires \ref{Hermstarhom}.
\end{proof}

\medskip
The remainder of this paragraph is concerned with commutative
C*-algebras and their dual objects: locally compact Hausdorff spaces.

\begin{theorem}[the Commutative Gel'fand-Na\u{\i}mark Theorem]%
\label{commGN}\index{concepts}{Theorem!Gel'fand-Na\u{\i}mark!Commutative}%
\index{concepts}{Gel'fand-Naimark@Gel'fand-Na\u{\i}mark!Theorem!Commutative}%
If $A$ is a commutative C*-algebra $\neq \{ 0 \}$, then the
Gel'fand transformation establishes a C*-algebra
isomorphism from $A$ onto $\cont\0\bigl(\Delta (A)\bigr)$.
\end{theorem}

\begin{proof} Corollary \ref{C*GTisom} says that
$\Delta (A) \neq \varnothing$ and that the Gel'fand transformation is
isometric, and therefore injective. The Gel'fand \linebreak transformation
is a \st-algebra homomorphism by \ref{Hermstarhom}. It suffices now
to show that it is surjective. Its range is dense \ref{Hermdense} and
on the other hand complete (because $a \mapsto \wht{a}$ is
isometric) and thus closed. \end{proof}

\medskip
Also of great importance is the following result which identifies the
spectrum of a normal element $b$ of a C*-algebra $A$
with the spectrum \ref{Gelftopo} of the C*-subalgebra of $\tld{A}$
generated by $b$, $b^*$, and $e$. Please note that the
latter C*-subalgebra is commutative as $b$ is normal.

\begin{theorem}\label{spechomeom}%
Let $A$ be a C*-algebra and let $b$ be any normal element of
$A$. Denote by $B$ the C*-subalgebra of $\tld{A}$ generated by
$b$, $b^*$, and $e$. Then $\wht{b}$ is a homeomorphism from
$\Delta (B)$ onto $\s_B (b) = \s_A (b)$.
\end{theorem}

\begin{proof}
We have $\s_B (b) = \s_{\tld{A}} (b) = \s_A (b)$ by \ref{specHermsubalg}
\& \ref{specunit}. The map $\wht{b}$ is continuous by definition of the
weak* topology, cf.\ \ref{weak*top}. It maps $\Delta (B)$ surjectively onto
$\s_B (b)$, cf.\ \ref{rangeGTu}. It also is injective: if
$\wht{b} \,( \tau _1 ) = \wht{b} \,( \tau _2 )$ for some
$\tau _1, \tau _2 \in \Delta (B)$, then
$\tau _1 \bigl(p(b,b^*)\bigr) = \tau _2 \bigl(p(b,b^*)\bigr)$ for all
$p \in \mathds{C} [ z , \overline{z} ]$, by \ref{mlfHerm}. Since the
polynomial functions in $b$ and $b^*$ are dense in $B$, it follows
that $\tau _1 = \tau _2$, by \ref{mlfbounded}. The continuous
bijective function $\wht{b}$ is homeomorphic as its
domain $\Delta (B)$ is compact and its range $\s_B (b)$ is a
Hausdorff space, cf.\ the appendix \ref{homeomorph}.
\end{proof}

\medskip
The preceding two theorems \ref{commGN} \& \ref{spechomeom}
together yield the \linebreak following ``operational calculus''.
(See also \ref{weakopcalc} above.) \pagebreak

\begin{corollary}[the operational calculus]\label{opcalc}%
\index{concepts}{operational calculus}%
\index{concepts}{calculus!operational}%
Let $A$ be a  C*-algebra and let $b$ be a normal element of $A$.
Let $B$ be the C*-subalgebra of $\tld{A}$ generated by
$b$, $b^*$, and $e$. For $f \in C\bigl(\s(b)\bigr)$, one denotes by $f(b)$
the element of $B$ satisfying
\[ \wht{f(b)} = f \circ \wht{b} \in C\bigl(\Delta(B)\bigr). \]
The mapping
\begin{alignat*}{2}
C\bigl(\s(b)\bigr) \ & & \to & \ B \\
f \ & & \mapsto & \ f(b)
\end{alignat*}
is a C*-algebra isomorphism, called the
\underline{operational calculus} for $b$. It is
the only \st-algebra homomorphism from
$C\bigl(\s(b)\bigr)$ to $B$ mapping the identity
function to $b$ and the constant function $1$ to $e$.
The following \underline{Spectral Mapping Theorem} holds:
\[ \s\bigl(f(b)\bigr) = f\bigl(\s(b)\bigr) \qquad \Bigl( f \in C\bigl(\s(b)\bigr) \Bigr). \]
\end{corollary}

\begin{proof}
The uniqueness statement follows from \ref{Cstarcontr}
and the Stone-\linebreak Weierstrass Theorem \ref{StW}.
\end{proof}

\begin{theorem}\label{C*comp}
Let $A$ be a commutative C*-algebra $\neq \{ 0 \}$.
Then $\Delta(A)$ is compact if and only if $A$ is unital.
\end{theorem}

\begin{proof}
If $A$ is unital, then $\Delta(A)$ is compact by proposition \ref{unitalcomp}.
Conversely, if $\Delta(A)$ is compact, then
$\cont\0\bigl(\Delta(A)\bigr) = C\bigl(\Delta(A)\bigr)$
has a unit (as $\Delta(A) \neq \varnothing$ by \ref{C*GTisom}). So $A$ is
unital as well because $A$ is isomorphic to
$\cont\0\bigl(\Delta(A)\bigr) = C\bigl(\Delta(A)\bigr)$
by the Commutative Gel'fand-Na\u{\i}mark Theorem \ref{commGN}.
\end{proof}

\begin{corollary}\label{speconepcomp}
If $A \neq \{ 0 \}$ is a commutative C*-algebra without unit,
then $\Delta(\tld{A})$ can be identified with the one-point
compactification of $\Delta(A)$ via restriction of functionals.
\vphantom{$\tld{A}$}
\end{corollary}

\begin{proof}
This follows now from \ref{notcomp} together with \ref{C*comp}.%
\end{proof}

\begin{proposition}\label{C*zerodiv}%
\index{concepts}{divisors of zero}\index{concepts}{zero divisors}%
A commutative C*-algebra of dimension at least $2$ contains
divisors of zero, i.e.\ non-zero elements $a, b$ with $a b = 0$.
\end{proposition}

\begin{proof}
This follows from the Commutative Gel'fand-Na\u{\i}mark Theorem
\ref{commGN} together with \ref{exdivzero}. \pagebreak
\end{proof}

\medskip
Now for locally compact Hausdorff spaces.

What is the spectrum of the commutative C*-algebra $\cont\0(\Omega)$,
if $\Omega$ is a locally compact Hausdorff space $\neq \varnothing$?
Precisely $\Omega$, as shown next.

\begin{theorem}\label{nonews}\index{concepts}{evaluation}%
Let $\Omega$ be a locally compact Hausdorff space $\neq \varnothing$.
The spectrum $\Delta \bigl( \cont\0(\Omega) \bigr)$ then is homeomorphic
to $\Omega$.

Indeed, for each $x \in \Omega$, consider the \underline{evaluation}
$\varepsilon_{x}$ at $x$ given by
\begin{alignat*}{2}
\varepsilon_{x} : \ & & \cont\0(\Omega) \to \ & \mathds{C} \\
                              \ & & f \mapsto  \ & \varepsilon_{x} (f) := f(x).
\end{alignat*}
The map $x \mapsto \varepsilon_{x}$ then is a homeomorphism
$\Omega \to \Delta \bigl(\cont\0(\Omega)\bigr)$.

Clearly $\wht{f}(\varepsilon_{x}) = f(x)$ for all $f \in \cont (\Omega)$
and all $x \in \Omega$. Thus, upon identifying
$\varepsilon_{x} \in \Delta \bigl(\cont\0(\Omega)\bigr)$ with $x \in \Omega$,
the Gel'fand$\vphantom{\wht{f}}$ transformation on $\cont\0(\Omega)$
becomes the identity map.$\vphantom{\wht{f}}$
\end{theorem}

Before going to the proof, we note the following two corollaries.

\begin{corollary}%
Two locally compact Hausdorff spaces $\Omega$, $\Omega'$ $\neq \varnothing$
are homeomorphic whenever the C*-algebras $\cont\0(\Omega)$, $\cont\0(\Omega')$
are isomorphic as \st-algebras. (Cf.\ \ref{Cstarisom}.) In this sense a locally compact
Hausdorff space $\Omega \neq \varnothing$ is determined up to a homeomorphism
by the \st-algebraic structure of the C*-algebra $\cont\0(\Omega)$ alone. 
\end{corollary}

\begin{corollary}%
There is a bijective correspondence between the equivalence classes of
commutative C*-algebras $\neq \{ 0 \}$ modulo \st-algebra isomorphism
and the equivalence classes of locally compact Hausdorff spaces
$\neq \varnothing$ modulo homeomorphism.
\end{corollary}

\begin{proof}[\hspace{-3.8ex}\mdseries{\scshape{Proof of \protect\ref{nonews}}}]%
Each $\varepsilon_{x}$ is a multiplicative linear functional on $\cont \0(\Omega)$,
being non-zero as $\cont\0(\Omega)$ vanishes nowhere on $\Omega$, by \ref{necStW}.
The map $x \mapsto \varepsilon_{x}$ is injective as $\cont \0(\Omega)$ separates
the points of $\Omega$, by \ref{necStW}. This map is continuous in the weak*
topology, that is, the topology of pointwise convergence on $\cont \0(\Omega)$,
cf.\ the appendix \ref{weak*point}. Indeed, if $x_{\alpha} \to x$ in $\Omega$, then
$\varepsilon_{x_{\alpha}} (f) = f(x_{\alpha}) \to f(x) = \varepsilon_{x} (f)$ for all
$f \in \cont \0(\Omega)$.

Assume first that $\Omega$ is compact, so $\cont\0(\Omega) = \cont (\Omega)$.
It suffices to show that the above map $x \mapsto \varepsilon_{x}$ is surjective
onto $\Delta \bigl(\cont (\Omega)\bigr)$ as a continuous bijection from a
compact space to a Hausdorff space is a homeomorphism, cf.\ the appendix
\ref{homeomorph}. By corollary \ref{taumaxidcomm}, the sets
$\ker \varepsilon_{x} = \{ \,f \in \cont (\Omega) : f(x) = 0 \,\}$ $(x \in \Omega)$
are maximal ideals in $\cont (\Omega)$, and it is enough to show that each
maximal ideal in $\cont (\Omega)$ is of this form. It suffices prove that each
maximal ideal in $\cont (\Omega)$ is contained in some set of this form.
So let $I$ be a maximal ideal in $\cont (\Omega)$,
and assume that for each $x \in \Omega$ there exists $f_{x} \in I$ with
$f_{x} (x) \neq 0$. The open sets $\{ y \in \Omega : \,f_{x} (y) \neq 0 \,\}$
$(x \in \Omega)$ cover $\Omega$, so by passing to a finite subcover,
we get $f_1, \ldots , f_n \in I$ such that $g:= \sum_{i=1}^{n} \overline{f_i} f_i$
is $> 0$ on $\Omega$. \linebreak So $g \in I$ would be invertible in
$\cont (\Omega)$, contradicting lemma \ref{notinvideal} (ii).

Assume next that $\Omega$ is not compact, and consider the one-point
compactification $K$ of $\Omega$. Let $\infty \in K \setminus \Omega$ be
the corresponding point at infinity. We then may identify $\cont\0(\Omega)$
with the set of continuous complex-valued functions on $K$ vanishing at
$\infty$, cf.\ \ref{Cnaught}. Hence we may identify $\cont (K)$ with
$\mathds{C} 1_{\Omega} + \cont\0(\Omega)$ as a \st-algebra. (Indeed, if
$f \in \cont (K)$, then $f - f(\infty) 1_K$ vanishes at $\infty$.) Now
$\mathds{C}1_\Omega + \cont\0(\Omega)$ is a \linebreak
C*-subalgebra of $\cont_\mathrm{b}(\Omega)$, cf.\ the proof of
\ref{Banachunitis}. It can be identified with the unitisation of
$\cont\0(\Omega)$, by \ref{Cstarisom}. (Please note that $\cont\0(\Omega)$
contains no unit, by \ref{nounit}.) Hence $\cont (K)$ can be identified
with the unitisation of $\cont\0(\Omega)$, by \ref{Cstarisom} again.
Next note that $\cont\0(\Omega)$ satisfies the assumptions of corollary
\ref{speconepcomp}, because $\cont\0(\Omega) \neq \{ 0 \}$ by \ref{necStW},
and because $\cont\0(\Omega)$ contains no unit, as noted before. It follows
that $\Delta \bigl( \cont (K) \bigr)$ may be identified with the one-point
compactification of $\Delta \bigl( \cont\0(\Omega) \bigr)$. Therefore the map
\[ \varepsilon : \Omega \to \Delta\bigl(\cont\0(\Omega)\bigr),\ x \mapsto \varepsilon_x \]
is (essentially) a restriction of the map
\[ \delta : K \to \Delta\bigl(\cont (K)\bigr),\ y \mapsto \delta_y, \]
with $\delta_y$ the evaluation of functions in $\cont (K)$ at $y \in K$.
We have to prove that $\varepsilon$ is a homeomorphism.
We know from the preceding paragraph of this proof that $\delta$ is a
homeomorphism. The domain $K$ of $\delta$ differs from the domain
$\Omega$ of $\varepsilon$ only in the point $\infty$. Also,
$\Delta\bigl(\cont (K)\bigr)$ differs from
$\Delta\bigl(\cont\0(\Omega)\bigr)$ precisely in the point at infinity of
$\Delta\bigl(\cont\0(\Omega)\bigr)$. So $\delta$ must map $\infty$ to the
point at infinity of $\Delta\bigl(\cont\0(\Omega)\bigr)$, whence
$\varepsilon$ must be surjective, and thus homeomorphic, as $\delta$ is so.
\end{proof}

\begin{proposition}[one-point compactification as a spectrum]\label{compficspect}%
Let $\Omega \neq \varnothing$ be a locally compact Hausdorff space,
which is not compact. The one-point compactification of $\Omega$
can be identified with the spectrum of the unital C*-subalgebra
$\mathds{C} 1_\Omega + \cont\0(\Omega)$ of $\cont_\mathrm{b}(\Omega)$.
Also, the C*-algebra $\mathds{C} 1_\Omega + \cont\0(\Omega)$
may be identified with the unitisation of $\cont\0(\Omega)$.
\end{proposition}

\begin{proof}
This can be extracted from the preceding proof. \pagebreak
\end{proof}

\clearpage


\section{The Stone-\texorpdfstring{\v{C}}{\81\014}ech Compactification}%
\label{StoneCech}

\begin{definition}[completely regular spaces]%
\index{concepts}{completely regular!space}%
A topological space $\Omega$ is called \underline{completely regular}
if it is a Hausdorff space, and if for every non-empty proper closed subset
$C$ of $\Omega$ and every point $x \in \Omega \setminus C$, there exists
a continuous function $f$ on $\Omega$ taking values in $[ \,0, 1 \,]$, such that
$f$ vanishes on $C$ and $f(x) = 1$.
\end{definition}

Please note that every subspace of a completely regular Hausdorff space
is itself a completely regular Hausdorff space.

For example every normal Hausdorff space is completely regular
by Urysohn's Lemma \ref{Urysohn}. In particular, every compact
Hausdorff space is completely regular.

It follows that every locally compact Hausdorff space is completely regular
(by considering the one-point compactification). Also, metrisable spaces
are completely regular. (Scale and truncate the distance function from a
point to a fixed closed set, cf.\ the appendix \ref{distance}.)

\begin{definition}[compactification]\index{concepts}{compactification}%
A \underline{compactification} of a Hausdorff space $\Omega$ is a pair
$(K, \phi)$, where $K$ is a compact Hausdorff space, and $\phi$ is a
homeomorphism of $\Omega$ onto a dense subset of $K$. We shall
identify $\Omega$ with its homeomorphic image $\phi(\Omega) \subset K$,
and we will simply speak of ``the compactification $K$ of $\Omega$''.
\end{definition}

Please note that at most completely regular Hausdorff spaces can have
a compactification. (Because of the preceding comments.)

For example $\bigl( \,[ \,-1, 1 \,], \tanh \,\bigr)$ is a compactification of $\mathds{R}$.

\begin{definition}[equivalence of two compactifications]%
\index{concepts}{equivalent!compactifications}%
Any  two compactifications $(K_1,\phi_1)$, $(K_2,\phi_2)$ of a
Hausdorff space $\Omega$ are said to be \underline{equivalent}
in case there exists a homeomorphism $\theta : K_1 \to K_2$ with
$\theta \circ \phi_1 = \phi_2$.
\end{definition}

\begin{definition}[completely regular algebras]%
\index{concepts}{completely regular!algebra}%
Let $\Omega \neq \varnothing$ be a Hausdorff space, and let $A$ be a subalgebra
of $\cont_\mathrm{b}(\Omega)$. One says that $A$ is \underline{completely regular}
if for every non-empty proper closed subset $C$ of $\Omega$ and for every point
$x \in \Omega \setminus C$, there exists a function $f \in A$ that vanishes on $C$ and 
that satisfies $f(x) \neq 0$. It is not difficult to see that $\Omega$ is completely regular
if and only if the C*-algebra $\cont_\mathrm{b}(\Omega)$ is completely regular.
\pagebreak
\end{definition}

We can now characterise all compactifications of a completely
\linebreak regular Hausdorff space: 

\begin{theorem}[all compactifications]\label{allcomp}%
Let $\Omega \neq \varnothing$ be a completely regular Hausdorff space.
The following statements hold.
\begin{itemize}
   \item[$(i)$]   If $A$ is a completely regular unital C*-subalgebra of
                         $\cont_\mathrm{b}(\Omega)$, then $(\Delta(A), \varepsilon)$ is a
                         compactification of $\Omega$, where $\varepsilon$ is the map
                         associating with each $x \in \Omega$ the evaluation
                         $\varepsilon_{x}$ at $x$ of the functions in $A$.
                         (That is, $\varepsilon_x : A \to \mathds{C}, f \mapsto f(x)$.)
                         For $f \in A$, there exists a (necessarily unique) function
                         $g \in C\bigl( \Delta (A) \bigr)$ with $f = g \circ \varepsilon$.
                         It is interpreted as the ``continuation of $f$ to the compactification''.
                         It is given as the Gel'fand transform of $f$.
  \item[$(ii)$]   If $(K,\phi)$ is a compactification of $\Omega$, then
                         \[ A_K := \{ \,g \circ \phi : g \in \cont (K) \,\} \subset \cont_\mathrm{b}(\Omega) \]
                         is a completely regular unital C*-subalgebra of $\cont_\mathrm{b}(\Omega)$,
                         and the ``restriction map'' $\cont (K) \to A_K$,
                         $g \mapsto g \circ \phi$ is a C*-algebra isomorphism.
                         The compactification $(K,\phi)$ then is equivalent to the
                         compactification associated with $A := A_K$ as in $(i)$.
 \item[$(iii)$]   Two compactifications $K$, $K'$ of $\Omega$ are equivalent
                         if and only if $A_K = A_{K'}$.
\end{itemize}
In other words, the equivalence classes of compactifications $K$
of $\Omega$ are in bijective correspondence with the spectra of
the completely regular unital C*-subalgebras $A_K$ of
$\cont_\mathrm{b}(\Omega)$. The $A_K$ consisting of those functions in
$\cont_\mathrm{b}(\Omega)$ which have a continuation to $K$.
\end{theorem}

\begin{proof}
(i): Let $A$ be a completely regular unital C*-subalgebra of
$\cont_\mathrm{b}(\Omega)$. Then $\Delta(A)$ is a compact
Hausdorff space by \ref{unitalcomp}. For $x \in \Omega$, the
evaluation $\varepsilon_{x}$ at $x$ of the functions in $A$
is a multiplicative linear functional on $A$ because it is
non-zero as $A$ is unital in $\cont_\mathrm{b}(\Omega)$.
The map $\varepsilon : x \mapsto \varepsilon_{x}$ thus takes
$\Omega$ to $\Delta (A)$. Also, the Gel'fand transform $g$
of a function $f$ in $A$ satisfies $f = g \circ \varepsilon$ as
$f (x) = \varepsilon_{x} (f) = g (\varepsilon_{x})$ for all $x \in \Omega$.
The set $\varepsilon (\Omega)$ is dense in $\Delta (A)$. For otherwise
there would exist a non-zero continuous function $g$ on $\Delta (A)$
vanishing on $\varepsilon (\Omega)$, by Urysohn's Lemma \ref{Urysohn}.
The Commutative Gel'fand-Na\u{\i}mark Theorem \ref{commGN} would imply
that there exists a unique function $f \in A$ whose Gel'fand transform is $g$,
and since $g \neq 0$, one would have $f \neq 0$, contradicting
$f = g \circ \varepsilon = 0$. The map $\varepsilon : \Omega \to \Delta (A)$
is injective because $A$ separates the points of $\Omega$.
(Since $A$ is completely regular and $\Omega$ is Hausdorff.) 
This map is continuous in the weak* topology, that is, the
topology of pointwise convergence on $A$, cf.\ the appendix \ref{weak*point}.
For if $x_{\alpha} \to x$ in $\Omega$, then $f(x_{\alpha}) \to f(x)$
for all $f \in A$. To see that $\varepsilon$ is a homeomorphism from
$\Omega$ onto $\varepsilon (\Omega)$, it remains to be shown that
$\varepsilon : \Omega \to \varepsilon (\Omega)$ is open. So let $U$
be a non-empty proper open subset of $\Omega$. For $x \in U$,
let $f \in A$ such that $f$ vanishes on $\Omega \setminus U$
and $f(x) \neq 0$. Denote by $g$ the Gel'fand transform of $f$.
The set $V := g^{-1}(\mathds{C} \setminus \{0\})$ is an open
neighbourhood of $\varepsilon_{x}$ in $\Delta (A)$. Now the set
$\varepsilon (U)$ contains $V \cap \varepsilon (\Omega)$ and so
is a neighbourhood of $\varepsilon_{x}$ in $\varepsilon (\Omega)$.
Therefore $\varepsilon (U)$ is open in $\varepsilon (\Omega)$. 

(ii): Let $(K,\phi)$ be a compactification of $\Omega$. The map from $\cont (K)$
to $A_K$ given by $g \mapsto g \circ \phi$, is interpreted as ``restriction to
$\Omega$''. It clearly is a unital \st-algebra homomorphism. It is isometric
(and thus injective) by density of $\phi(\Omega)$ in $K$. Hence the range
$A_K$ of this map is complete, and thus a unital and completely regular
C*-subalgebra of $\cont_\mathrm{b}(\Omega)$. Since the above map is surjective
by definition, it is a C*-algebra isomorphism. It shall next be shown that
$(K,\phi)$ is equivalent to the compactification associated with $A := A_K$
as in (i). For $x \in K$, consider $\theta_x : f \to g(x)$, $(f \in A_K)$, where
$g$ is the unique function in $\cont (K)$ such that $f = g \,\circ \,\phi$, see above.
Then $\theta_x$ is a multiplicative linear functional on $A_K$ because
it is non-zero as $A_K$ is unital in $\cont_\mathrm{b}(\Omega)$. Now the map
$\theta : K \to \Delta(A_K)$, $x \mapsto \theta_x$ is the desired
homeomorphism $K \to \Delta(A_K)$. Indeed, this map is a
homeomorphism by theorem \ref{nonews} because $A_K$ is isomorphic
as a C*-algebra to $\cont (K)$, by the above. It is also clear that
$\theta \circ \phi = \varepsilon$.

(iii) follows essentially from the last statement of (ii).
\end{proof}

\begin{example}%
If $\Omega \neq \varnothing$ is a locally compact Hausdorff space which
is not compact, then the one-point compactification of $\Omega$ is
associated with $\mathds{C} 1_{\Omega} + \cont\0(\Omega)$, as can be
seen from proposition \ref{compficspect}.
\end{example}

\begin{definition}[Stone-\v{C}ech compactification]%
\index{concepts}{Stone-Cech@Stone-\v{C}ech compactification}%
\index{concepts}{compactification!Stone-Cech@Stone-\v{C}ech}%
Let $(K,\phi)$ be a \linebreak compactification of a Hausdorff space $\Omega$.
Then $(K,\phi)$ is called a \linebreak \underline{Stone-\v{C}ech compactification}
of $\Omega$ if every $f \in \cont_\mathrm{b}(\Omega)$ is of the form $f = g \circ \phi$
for some $g \in \cont (K)$. The function $g$ then is uniquely determined by
density of $\phi (\Omega)$ in $K$.

In other words, a Stone-\v{C}ech compactification of a Hausdorff space
$\Omega$ is a compact Hausdorff space $K$ containing $\Omega$ as
a dense subset, such that every function $f \in \cont_\mathrm{b}(\Omega)$ has a
(necessarily unique) continuation to $K$.
\end{definition}

\begin{corollary}%
Let $\Omega \neq \varnothing$ be a completely regular Hausdorff space.
There then exists a Stone-\v{C}ech compactification of $\Omega$,
and it is unique within equivalence. It is denoted by
$(\beta\/\Omega, \varepsilon)$, or simply by $\beta\/\Omega$.%
\pagebreak
\end{corollary}

\begin{proof}
Apply theorem \ref{allcomp} above with $A := \cont_\mathrm{b}(\Omega)$.
\end{proof}

\medskip
Hence the following memorable result.

\begin{corollary}
A Hausdorff space is completely regular if and only if it is homeomorphic
to a subspace of a compact Hausdorff space.
\end{corollary}

Thus for example every metrisable space is homeomorphic to a subspace
of a compact Hausdorff space.

(Also note that a Hausdorff space is locally compact if and only if it is
homeomorphic to an open subspace of a compact Hausdorff space.)

\begin{theorem}[universal property]%
\index{concepts}{universal!property!Stone-Cech@of Stone-\v{C}ech compactif.}
Let $\Omega \neq \varnothing$ be a completely regular Hausdorff space.
The Stone-\v{C}ech compactification $(\beta\/\Omega,\varepsilon)$ of
$\Omega$ has the following \underline{universal property}. Let $K$
be a compact Hausdorff space, and $\phi$ be a continuous function
$\Omega \to K$. Then $\phi$ has a (necessarily unique) continuation
to $\beta\/\Omega$. That is, there exists a (necessarily unique) continuous
function $\tld{\phi} : \beta\/\Omega \to K$ with $\tld{\phi} \circ \varepsilon = \phi$.
If $(K,\phi)$ is a compactification of $\Omega$, then $\tld{\phi}$ is surjective.
This shows among other things that the Stone-\v{C}ech compactification
$\beta\/\Omega$ of $\Omega$ is the ``largest'' compactification of $\Omega$
in the sense that any other compactification of $\Omega$ is a continuous
image of $\beta\/\Omega$.
\end{theorem}

\begin{proof}
Put $A := \cont_\mathrm{b}(\Omega)$ and $B := \cont (K)$. Let $(\beta K, i)$
be the Stone-\v{C}ech compactification of $K$. Then $i : K \to \beta K$ is a
homeomorphism as $K$ is compact. The map $\phi$ induces an ``adjoint''
map $\pi : B \to A$ by \linebreak putting $\pi (g) := g \circ \phi$, $(g \in B)$.
The map $\pi$ is a unital \st-algebra homomorphism, and hence in turn
induces a ``second adjoint'' map \linebreak $\pi ^* : \Delta (A) \to \Delta (B)$,
by letting $\pi ^* (\tau) := \tau \circ \pi$, $\bigl( \,\tau \in \Delta (A) \,\bigr)$. The
map $\pi ^*$ is continuous by the universal property of the weak* topology,
cf.\ the appendix \ref{weak*top}. We have $\pi ^* \circ \varepsilon = i \circ \phi$
as for $x \in \Omega$, $g \in B$, one computes
\begin{align*}
& \ \bigl( (\pi ^* \circ \varepsilon) (x) \bigr) (g)
= \bigl( \pi ^* (\varepsilon_{x}) \bigr) (g)
= ( \varepsilon_{x} \circ \pi ) (g)
= \varepsilon_{x} \bigl( \pi (g) \bigr)
= \varepsilon_{x} (g \circ \phi) \\
= & \ (g \circ \phi) (x)
= g \bigl( \phi (x) \bigr)
= \Bigl( i \bigl( \phi (x) \bigr) \Bigr) (g)
= \bigl( ( i \circ \phi ) (x) \bigr) (g).
\end{align*}
Now let $\tld{\phi} := i^{-1} \circ \pi ^* : \beta\/\Omega \to K$.
This map is continuous because $i$ is open. Also
$\tld{\phi} \circ \varepsilon = i^{-1} \circ \pi ^* \circ \varepsilon = \phi$ as
required. Finally, if $(K,\phi)$ is a compactification, then $\phi(\Omega)$
is dense in $K$. In this case, $\tld{\phi}(\beta\/\Omega)$ will be dense in
$K$, the more, and also compact, so that $\tld{\phi}(\beta\/\Omega) = K$.
\pagebreak
\end{proof}

\clearpage


\chapter{Positive Elements}%
\label{PositiveElements}


\setcounter{section}{18}

\section{The Positive Cone in a Hermitian Banach %
\texorpdfstring{$*$-}{\80\052\80\055}Algebra}

\medskip
\begin{reminder}[cone]\index{concepts}{cone}%
A \underline{cone} in a real vector space $V$ is a non-empty subset $C$
of $V$ such that for $c \in C$ and $\lambda > 0$ also $\lambda c \in C$.
\end{reminder}

\begin{definition}[positive elements]\index{symbols}{a1@$a \geq 0$}%
\index{symbols}{A9@$A_+$}\index{concepts}{positive!element}%
Let $A$ be a \st-algebra. We shall denote by
\[ A_+ := \bigl\{ \,a \in A\sa : \s _A (a) \subset [ \,0, \infty \,[ \,\bigr\} \]
the set of \underline{positive} elements in $A$. To indicate
that an element $a \in A$ is positive, we shall also write
\[ a \geq 0. \]
We stress that positive elements are required to be \underline{Hermitian}.
The set of positive elements in $A$ is a cone in $A\sa$.
\end{definition}

\begin{proposition}\label{hompos}%
Let $\pi : A \to B$ be a \st-algebra homomorphism from a \st-algebra
$A$ to a \st-algebra $B$. If $a \in A_+$, then $\pi (a) \in B_+$.
\end{proposition}

\begin{proof}
$\s\bigl(\pi (a)\bigr) \setminus \{ 0 \} \subset \s(a) \setminus \{ 0 \} $,
cf.\ \ref{spechom}.
\end{proof}

\begin{proposition}\label{plussubal}%
Consider a normed \st-algebra $A$ as well as a Hermitian complete
\st-subalgebra $B$ of $A$. (Cf.\ \ref{Herminher}.) For $b \in B$, we get
\[ b \in B_+ \Leftrightarrow b \in A_+.\]
\end{proposition}

\begin{proof}
$\s_B (b) \setminus \{ 0 \} = \s_A (b) \setminus \{ 0 \}$,
cf.\ \ref{specHermsubalg}.
\end{proof}

\begin{theorem}[(in-)stability of $A_+$ under multiplication]\label{instability}%
Let $A$ be a Hermitian Banach \st-algebra. If $a,b \in A_+$, then
the product $ab$ is in $A_+$ if and only if $a$ and $b$ commute.
\end{theorem}

\begin{proof}
For the ``only if'' part, see \ref{Hermprod}. Assume now that
$a, b \in A_+$ commute. We can then assume that $A$ is
commutative and unital by an application of the Theorem of
Civin \& Yood, cf.\ \ref{CivinYood}. Indeed, this result says that
there exists an automatically Hermitian \ref{Herminher}
commutative unital closed \st-subalgebra $B$ of $\tld{A}$
containing $a$ and $b$, such that $\s_A (c) = \s_B (c)$ for
$c \in \{ \,a, b, ab \,\}$. Now the elements $\frac{1}{n}e+a$
$(n \geq 1)$ have Hermitian square roots $a_n = {a_n}^*$,
cf.\ \ref{possqroot}. The elements $\frac{1}{n}e+b$ $(n \geq 1)$
have Hermitian square roots $b_n = {b_n}^*$. The elements
$a_n b_n$ are Hermitian by commutativity, and so have real
spectrum. Hence $\s \,\bigl(\,{( \,a_n b_n \,)}^{\,2}\,\bigr) \subset [\,0, \infty \,[$
for all integers $n \geq 1$. (By the Rational Spectral Mapping Theorem).
The continuity of the spectrum function in a commutative
Banach algebra \ref{commspeccont} shows that
\[ ab = \lim _{n \to \infty} {a_n}^2 \,{b_n}^2 =
\lim _{n \to \infty} {( \,a_n b_n \,) \,}^2 \geq 0. \qedhere \]
\end{proof}

\smallskip
Our next aim is the Shirali-Ford Theorem \ref{ShiraliFord}.
On the way, we shall see that the set of positive elements
of a Hermitian Banach \st-algebra $A$ is a closed convex
cone in $A\sa$.

\begin{proposition}\label{Hermsubmult}%
Let $A$ be a Hermitian Banach \st-algebra. For two
\underline{Hermitian} elements $a, b \in A$, we have
\[ \rlambda(ab) \leq \rlambda(a) \,\rlambda(b).\]
\end{proposition}

\begin{proof} 
By \ref{fundHerm} (i) $\Rightarrow$ (iii), we have
\[ \rlambda (ab) \leq \rsigma (ab) =
{\rlambda (baab) \,}^{1/2} = {\rlambda \bigl( \,{a \,}^2 \,{b\,}^2 \,\bigr) \,}^{1/2} \]
where we have made use of the fact that
$\rlambda (cd) = \rlambda (dc)$ for $c,d \in A$, cf.\ \ref{rlcomm}.
By induction it follows that
\[ \rlambda (ab) \leq
{\rlambda \Bigl( \,{a \,}^{( \,{2 \,}^{n})} \,{b \,}^{(\,{2\,}^{n})} \,\Bigr) \,}^{1 \,/ \,( \,{{2\,}^{n}} \,)}
\leq {\bigl|\,{a \,}^{( \,{2 \,}^{n})}\,\bigr| \,}^{1 \,/ \,( \,{{2 \,}^{n}} \,)}
\cdot {\bigl|\,{b \,}^{( \,{2 \,}^{n} \,)}\,\bigr| \,}^{1 \,/ \,( \,{{2 \,}^{n}} \,)} \]
for all $n \geq 1$. It remains to take the limit for $n \to \infty$. \end{proof}

\begin{theorem}\label{plusconvexcone}%
If $A$ is a Hermitian Banach \st-algebra
then $A_+$ is a convex cone in $A\sa$.
\end{theorem}

\begin{proof}
Let $a,b$ be positive elements of $A$. It suffices to prove that
$a+b$ is positive as $A_+$ is a cone. It is enough to show that
$e+a+b$ is invertible in $\tld{A}$, for the same reason.
Since $e+a$ and $e+b$ are invertible in $\tld{A}$, we may define
$c := a \,{( \,e+a \,)}^{-1}, d := b \,{( \,e+b \,)}^{-1}$. These two
elements of $A$ are Hermitian by \ref{Hermprod} and \ref{comm2}.
\pagebreak
The Rational Spectral Mapping Theorem yields $\rlambda(c) < 1$
and $\rlambda(d) < 1$. The preceding proposition \ref{Hermsubmult}
then gives $\rlambda(cd)<1$, so $e-cd$ is invertible in $\tld{A}$.
We now have $(\,e+a\,)\,(\,e-cd\,)\,(\,e+b\,) = e+a+b$, so $e+a+b$
is invertible in $\tld{A}$.
\end{proof}

\begin{proposition}\label{Hermrlcontsa}%
Let $A$ be a Hermitian Banach \st-algebra. Let $a, b \in A$
be \underline{Hermitian} elements. We then have
\[ \rlambda(a+b) \leq \rlambda(a)+\rlambda(b), \]
whence also
\[ | \,\rlambda(a)-\rlambda(b) \,| \leq \rlambda(a-b) \leq | \,a-b \,|.\]
It follows that $\rlambda$ is uniformly continuous on $A\sa$.
(See also \ref{Hermspeccont}.)
\end{proposition}

\begin{proof}
We have
\[ \rlambda(a) \,e \pm a \geq 0, \quad \rlambda(b) \,e \pm b \geq 0, \]
and so, by \ref{plusconvexcone}
\[ \bigl( \,\rlambda(a)+\rlambda(b) \,\bigr)\,e \pm (a+b) \geq 0, \]
whence
\[ \rlambda(a+b) \leq \rlambda(a)+\rlambda(b). \qedhere \]
\end{proof}

\begin{lemma}\label{critpos}%
Let $A$ be a Banach algebra, and let $a$ be an element of
$A$ with real spectrum. For $\tau \geq \rlambda(a)$, we have
\[ \min \,\s(a) = \tau - \rlambda ( \tau e - a ). \]
(Please note that $\s(\tau e-a) \subset [ \,0, \infty \,[$.) Hence the
following statements are equivalent.
\begin{itemize}
   \item[$(i)$] $\s(a) \subset [ \,0, \infty \,[$,
  \item[$(ii)$] $\rlambda(\tau e-a) \leq \tau$ for all $\tau \geq \rlambda(a)$,
 \item[$(iii)$] $\rlambda(\tau e-a) \leq \tau$ for any $\tau \geq \rlambda(a)$.
\end{itemize}
\end{lemma}

\begin{proof}
For $\tau \geq \rlambda(a)$, we obtain
\begin{align*}
\rlambda(\tau e-a) & = \max\,\{\,| \,\lambda \,| : \lambda \in \s(\tau e-a)\,\} \\
 & = \max \,\{ \,| \,\lambda \,| : \lambda \in \tau-\s(a) \,\} \\
 & = \max \,\{ \,| \,\tau-\lambda \,| : \lambda \in \s(a) \,\} \\
 & = \max \,\{ \,\tau-\lambda:\lambda\in \s(a) \,\} = \tau - \min \,\s(a). \qedhere
\end{align*}
\end{proof}

\begin{theorem}\label{plusclosed}%
If $A$ is a Hermitian Banach \st-algebra then $A_+$ is closed in $A\sa$.
\pagebreak
\end{theorem}

\begin{proof}
Let $(a_n)$ be a sequence in $A_+$ converging to an
element $a$ of $A\sa$. There then exists $\tau \geq 0$ such
that $\tau \geq | \,a_n \,| \geq \rlambda(a_n)$ for all $n$. It
follows that also $\tau \geq | \,a \,| \geq \rlambda(a)$.
By \ref{Hermrlcontsa} and \ref{critpos}, we have
\[ \rlambda(\tau e-a) = \lim_{n \to \infty} \rlambda(\tau e-a_n)
\leq \tau. \qedhere \]
\end{proof}

\begin{theorem}%
If $A$ is a Hermitian Banach \st-algebra then
$A_+$ is a closed convex cone in $A\sa$.
\end{theorem}

\begin{proof}
This is the theorems \ref{plusconvexcone} and \ref{plusclosed}.
\end{proof}

\begin{theorem}\label{preShiraliFord}%
Let $A$ be a Banach \st-algebra. Let $\alpha \geq 2$ and put
$\beta := \alpha+1$. The following statements are equivalent.
\begin{itemize}
   \item[$(i)$] $A$ is Hermitian,
  \item[$(ii)$] for each $a \in A$ one has $\s(a^*a) \subset [ \,-\lambda , \lambda \,]$, \\
where $\lambda := \max \,\bigl( \,( \,\s(a^*a) \cap [ \,0, \infty \,[ \,) \cup \{\,0\,\} \,\bigr)$,
 \item[$(iii)$] for all $a \in A$ with $\rsigma(a) < 1$, one has $\rsigma(c) \leq 1$, \\
where $c := {\alpha \,}^{-1} ( \beta a-aa^*a )$,
 \item[$(iv)$] $\s(a^*a)$ does not contain any $\lambda < 0$ whenever $a \in A$.
  \item[$(v)$] $e+a^*a$ is invertible in $\tld{A}$ for all $a \in A$.
\end{itemize}
\end{theorem}

\begin{remark}
In their definition of a C*-algebra, Gel'fand and Na\u{\i}mark
had to assume that $e+a^*a$ is invertible in $\tld{A}$ for all
$a \in A$. The above theorem reduces this to showing that a
C*-algebra is Hermitian \ref{C*Herm}. Before going to the
proof, we note the following consequence.
\end{remark}

\begin{theorem}[the Shirali-Ford Theorem]\label{ShiraliFord}%
\index{concepts}{Theorem!Shirali-Ford}\index{concepts}{Shirali-Ford Theorem}%
A  Banach \st-algebra $A$ is Hermitian if and only if $a^*a \geq 0$ for all $a \in A$.
\end{theorem}

\begin{proof}[\hspace{-3.85ex}\mdseries{\scshape{Proof of \protect\ref{preShiraliFord}}}]%
The proof is a simplified variant of the one in the book by Bonsall \& Duncan
\cite[p.\ 226]{BD}, and uses a polynomial rather than a rational function.

(i) $\Rightarrow$ (ii). Assume that $A$ is Hermitian, let $a \in A$, and
let $\lambda$ be as in (ii). We obtain
$\s(a^*a) \subset \ ] - \infty , \lambda \,]$ as $a^*a$ is Hermitian. Let
now $a=b + \iu c$ where $b,c \in A\sa$. We then have
$aa^*+a^*a = 2 \,({b\,}^2+{c\,}^2)$, whence
$\lambda e+aa^*= 2 \,({b\,}^2+{c\,}^2)+(\lambda e-a^*a) \geq 0$ by convexity of
$\tld{A}_+$ \ref{plusconvexcone}, so that $\s(aa^*) \subset [ \,- \lambda , \infty \,[$.
One concludes that $\s(a^*a) \subset [ \,- \lambda , \lambda \,]$ as
$\s(a^*a) \setminus \{ 0 \} = \s(aa^*) \setminus \{ 0 \}$, cf.\ \ref{speccomm}.
\pagebreak

Intermediate step. Consider the polynomial $p(x) = {\alpha \,}^{-2} \,x \,{( \beta -x) \,}^2$.
If $c$ is as in (iii), then $c^*c = p(a^*a)$, whence $\s(c^*c) = p\bigl(\s(a^*a)\bigr)$.
Please note that $p(1) = 1$ and hence $p(x) \leq 1$ for all $x \leq 1$ by monotony.

(ii) $\Rightarrow$ (iii). Assume that (ii) holds. Let $a \in A$ with
$\rsigma(a) \leq 1$. Then $\s(a^*a) \subset [-1,1]$ by (ii). Let $c$
be as in (iii). Then $\s(c^*c) \subset \,]-\infty,1]$ by the intermediate
step. Again by (ii), it follows that $\s(c^*c) \subset [-1,1]$.

(iii) $\Rightarrow$ (iv). Assume that (iii) holds and that (iv) does not hold. Let
$\delta := - \inf \,\bigcup _{\text{\smaller{$\,a \in A, \,\rsigma (a) < 1$}}}
\,\s(a^*a) \cap \mathds{R}$. We then have $\delta > 0$. Hence there
exists $a \in A$, $\gamma > 0$, such that $\rsigma (a) < 1$,
$- \gamma \in \s(a^*a)$, and $\gamma \geq {( \alpha / \beta ) \,}^2 \,\delta$.
With $c$ as in (iii), we then have $\rsigma (c) \leq 1$ by assumption. \linebreak
The intermediate step implies that $p( - \gamma ) \in \s( c^*c ) \cap \mathds{R}$, whence
for $0 < \theta < 1$ also $\theta \,p ( - \gamma ) \in \s \,( \,\theta \,c^*c \,) \cap \mathds{R}$.
So the definition of $\delta$ gives $\theta \,p( - \gamma ) \geq - \delta$ for $0 < \theta < 1$,
from which we get $p( - \gamma ) \geq - \delta$ by letting $\theta$ approach $1$.
That is ${\alpha \,}^{-2} \,\gamma \,{( \beta + \gamma ) \,}^2 \leq \delta$. Hence, using
$\gamma > 0$, we obtain the contradiction $\gamma < {( \alpha / \beta )}^{\,2\,}\delta$.

(iv) $\Rightarrow$ (v) is trivial.

(v) $\Rightarrow$ (i). Assume that (v) holds. Let $a \in A$ be Hermitian.
Then $-1 \notin \s({a \,}^2)$, so that $i \notin \s(a)$ by the Rational Spectral
Mapping Theorem. It follows via \ref{fundHerm} (ii) $\Rightarrow$ (i) that
$A$ is Hermitian.
\end{proof}

\begin{lemma}\label{rsHerm}%
Let $\pi : A \to B$ be a \st-algebra homomorphism from a Banach
\st-algebra $A$ to a Banach \st-algebra $B$. If $B$ is Hermitian, and if
\[ \rsigma \bigl(\pi (a)\bigr) = \rsigma(a)\quad \text{for all} \quad a \in A, \]
then $A$ is Hermitian.
\end{lemma}

\begin{proof}
This follows from \ref{preShiraliFord} (i) $\Leftrightarrow$ (iii).
\end{proof}

\begin{theorem}[stability of $A_+$ under \st-congruence]%
\label{stcongr}%
Let $A$ be a Hermitian Banach \st-algebra. If $a \in A_+$
then $c^*ac \in A_+$ for all $c \in \tld{A}$.
\end{theorem}

\begin{proof}
The elements $\frac{1}{n}e+a$ $(n \geq 1)$ have Hermitian square
roots $b_n = {b_n}^*$ in $\tld{A}\sa$, cf.\ \ref{possqroot}. It follows
from \ref{ShiraliFord} and \ref{plusclosed} that
\[ c^*ac = \lim _{n \to \infty} c^* {b_n}^2 c =
\lim _{n \to \infty} {(b_n c)}^*{(b_n c)} \geq 0. \qedhere \]
\end{proof}

\medskip
We also give the next result, in view of its strength, exhibited by
the ensuing corollary. \pagebreak

\begin{proposition}\label{sapreShiraliFord}%
Let $A$ be a Banach \st-algebra. Let $\alpha \geq 2$ and put
$\beta := \alpha+1$. The following statements are equivalent.
\begin{itemize}
   \item[$(i)$] $A$ is Hermitian,
  \item[$(ii)$] for each $a \in A\sa$ one has
  $\s ^{\,}\bigl( _{\,}{a}^{\,2} \,\bigr) \subset [ \,- \lambda , \lambda \,]$, \\
  where $\lambda := \max \,\bigl( \,( \,\s \bigl({a}^{\,2\,} \bigr) \cap [ \,0, \infty \,[ \,) \cup \{\,0\,\} \,\bigr)$,
 \item[$(iii)$] for all $a \in A\sa$ with $\rlambda (a) < 1$, one has $\rlambda(c) \leq 1$, \\
  where $c := {\alpha \,}^{-1} \bigl( \beta a - {a}^{\,3\,} \bigr)$,
 \item[$(iv)$] for all $a \in A\sa$ with $\rlambda(a) < 1$, and all $\lambda \in \s(a)$, \\
  one has $\bigl| \,{\alpha}^{\,-1} \lambda \,\bigl( \beta -{\lambda}^{\,2\,} \bigr) \,\bigr| \leq 1$,
  \item[$(v)$] $\s ^{\,}\bigl( _{\,}{a}^{\,2} \,\bigr)$ does not contain any $\lambda < 0$ for all $a \in A\sa$.
 \item[$(vi)$] $e + {a }^{\,2}$ is invertible in $\tld{A}$ for all $a \in A\sa$.
\end{itemize}
\end{proposition}

\begin{proof}
The proof follows the lines of the proof of theorem \ref{preShiraliFord},
cf.\ \ref{Hermrseqrl}.
\end{proof}

\begin{corollary}%
A Banach \st-algebra $A$ is Hermitian if (and only if) there exists a
compact subset $S$ of the open unit disc, such that $a \in A\sa$ and
$\rlambda (a) < 1$ together imply $\s(a) \subset \ ] -1,1 \,[ \ \cup \ S$.
\end{corollary}

\begin{proof}
Let $A$ be a Banach \st-algebra. Assume that $S$ is a compact subset
of the open unit disc such that $a \in A\sa$ and $\rlambda (a) < 1$
together imply $\s(a) \subset \ ] -1,1 \,[ \ \cup \ S$. To show that $A$ is
Hermitian, one applies (iv) of the preceding proposition \ref{sapreShiraliFord}
with $\alpha \geq 2$ so large that
\[ \alpha/(\alpha+2)\geq\max\,\{\,|\,\lambda\,| : \lambda \in S \,\}. \]
Indeed, with $\beta := \alpha +1$ we have the following.
Let $a \in A\sa$ with $\rlambda (a) < 1$, and let $\lambda \in \s(a)$.
If $\lambda$ does not belong to $] -1,1 \,[$, then it belongs to $S$, whence
\[|\,\lambda\,| \leq \alpha^{\,} / (\alpha+2) = \alpha^{\,} / (\beta+1)
\leq \alpha^{\,} / \bigl| \,\beta - {\lambda}^{\,2\,} \bigr|,\]
Also, if $\lambda$ belongs to $] -1,1 \,[$, then
\[ \bigl| \,{\alpha \,}^{-1} \lambda \,\bigl( \,\beta - {\lambda}^{\,2\,} \bigr) \,\bigr| < 1 \]
by differential calculus.
\end{proof}

\begin{remark}\label{remark2}%
Some of the results in this paragraph have easy proofs in presence of
commutativity, using multiplicative linear functionals, as in \ref{Hermcomm}.
Also, in \ref{instability}, the assumption that $A$ be \linebreak Hermitian
can be dropped, cf.\ \ref{specsub}. Our proofs leading to these
facts however use the Lemma of Zorn, cf.\ the remark \ref{Zorn}.
\pagebreak
\end{remark}

\clearpage

%


\section{C\texorpdfstring{*-}{\80\052\80\055}Seminorms on %
Banach \texorpdfstring{$*$-}{\80\052\80\055}Algebras}%
\label{PtakHerm}

\begin{definition}[C*-seminorm]\label{Cstarsemin}%
\index{concepts}{C8@C*-seminorm}%
\index{concepts}{C7@C*-property}%
An \underline{algebra seminorm} is a seminorm $p$
on an algebra $A$ such that
\[ p(ab) \leq p(a) \,p(b) \quad \text{for all} \quad a, b \in A. \]

A \underline{C*-seminorm} is an algebra seminorm $p$ on a
\st-algebra $A$, such that for all $a \in A$, one has
\[ p(a)^{\,2} = p(a^*a) \quad \text{as well as} \quad p(a^*) = p(a). \]
The first equality is called the \underline{C*-property}.

An equivalent definition is that a C*-seminorm is an algebra
seminorm $p$ on a \st-algebra $A$, satisfying
\[ p(a)^{\,2} \leq p(a^*a) \quad \text{for all} \quad a \in A, \]
cf.\ the proof of \ref{condC*}.
\end{definition}

\begin{theorem}\label{existhom}%
For a C*-seminorm $p$ on a \st-algebra $A$
there exists a \st-algebra homomorphism
$\pi$ from $A$ to a C*-algebra such that
\[ \| \,\pi(a) \,\| = p(a) \quad \text{for all} \quad a \in A. \]
\end{theorem}

\begin{proof}
It is easily seen that the set $I$ on which $p$ vanishes is a
\st-stable two-sided ideal in $A$. Furthermore, from
\[ | \,p(a)-p(b) \,| \leq p(a-b) \quad (a, b \in A) \]
it follows that $ \| \,a+I \,\| := p(a)$ $(a \in A)$ is well-defined.
This definition makes $A\mspace{1mu}/I$ into a pre-C*-algebra.
(As $p$ is a C*-seminorm.) We get a \st-algebra homomorphism
$\pi : a \mapsto a+I$ from $A$ to the pre-C*-algebra
$A\mspace{1mu}/I$ with $\| \,\pi(a) \,\| = p(a)$ for all $a \in A$.
It remains to complete $A\mspace{1mu}/I$.
\end{proof}

\begin{corollary}\label{contC*semin}%
A C*-seminorm $p$ on a \underline{Banach} \st-algebra $A$
is continuous and satisfies
\[ p(a) \leq \rsigma(a) \quad \text{for all} \quad a \in A. \]
\end{corollary}

\begin{proof}
The preceding theorem \ref{existhom} guarantees the existence
of a \linebreak \st-algebra homomorphism $\pi$ from $A$ to a
C*-algebra with $\| \,\pi(a) \,\| = p(a)$ for all $a \in A$. This \st-algebra
homomorphism $\pi$ satisfies $\| \,\pi(a) \,\| \leq \rsigma(a)$ by
theorem \ref{contraux}, and is continuous by theorem \ref{automcont}.
\end{proof}

\begin{theorem}[Pt\'{a}k]\label{PtakThm}%
\index{concepts}{Theorem!Pt\'{a}k}%
\index{concepts}{Pt\'{a}k!Theorem}%
The Pt\'{a}k function $\rsigma$ on a \underline{Hermitian} \linebreak
Banach \st-algebra is a C*-seminorm. \pagebreak
\end{theorem}

\begin{proof} Let $A$ be a Hermitian Banach \st-algebra.
The Pt\'{a}k function $\rsigma$ on $A$ satisfies the C*-property
by proposition \ref{rsC*}. The same proposition shows that
$\rsigma(a^*) = \rsigma(a)$ for all $a \in A$. To see that
$\rsigma$ is submultiplicative, let $a, b \in A$. By proposition
\ref{rlcomm} and the submultiplicativity property of $\rlambda$
shown in proposition \ref{Hermsubmult}, we have that
\begin{align*}
\rsigma(ab)^{\,2} & = \rlambda(b^*a^*ab) \\
 & = \rlambda(a^*abb^*) \\
 & \leq \rlambda(a^*a) \,\rlambda(bb^*) = \rsigma(a)^{\,2} \,\rsigma(b^*)^{\,2}
 = \rsigma(a)^{\,2} \,\rsigma(b)^{\,2}.
\end{align*}
We shall next show the intermediate result that for $c \in A$ one has
\[ \rlambda \,\bigl( \,{\frac12 \,(c+c^*)} \,\bigr) \leq \rsigma(c). \]
Let $c = a + \iu b$ with $a, b \in A\sa$. We then have
\[ c^*c + cc^* = 2 \,{ ( \,a^{\,2} + b^{\,2} \,)}. \]
Also, by convexity of $A_+$, cf.\ theorem \ref{plusconvexcone}, one has
\[ \rlambda \,( \,a^{\,2} + b^{\,2} \,) \,e - a^{\,2} =
\bigl\{ \,\rlambda \,( \,a^{\,2} + b^{\,2} \,) \,e
- ( \,a^{\,2} + b^{\,2} \,) \,\bigr\} + b^{\,2} \geq 0, \]
and so
\[ \rlambda \,( \,a^{\,2} \,) \leq \rlambda \,( \,a^{\,2} + b^{\,2} \,). \]
Whence, by proposition \ref{rlpowers} and from subadditivity of $\rlambda$
on $A\sa$, cf.\ proposition \ref{Hermrlcontsa}, we get
\begin{alignat*}{2}
 {\rlambda \,\Bigl( \,{\frac12 \,(c+c^*)} \,\Bigr)}^{\,2} & = {\rlambda(a)}^{\,2} & \\
 & = \rlambda ( \,a^{\,2} \,) & \\
 & \leq \rlambda \,( \,a^{\,2} + b^{\,2} \,)
 & \ = & \ \frac12 \,\Bigl( \,{\rlambda (c^*c + cc^*)} \,\Bigr) \\
 & & \ \leq & \ \frac12 \,\Bigl( {\,\rsigma(c)^{\,2} + \rsigma(c^*)^{\,2}} \,\Bigr)
 = \rsigma(c)^{\,2}.
\end{alignat*}
We can now show that $\rsigma$ is subadditive. For $a, b \in A$, by
subadditivity of $\rlambda$ on $A\sa$, cf.\ proposition \ref{Hermrlcontsa},
we have that
\begin{align*}
 \rsigma (a+b)^{\,2} & = \rlambda \bigl( \,(a^*+b^*)(a+b) \,\bigr) \\
 & \leq \rsigma(a)^{\,2} + \rsigma(b)^{\,2} + \rlambda(a^*b+b^*a).
\end{align*}
Using the intermediate result shown above, we find that
\[ \rlambda(a^*b+b^*a) \leq 2 \,\rsigma(a^*b)
\leq 2 \,\rsigma(a^*) \,\rsigma(b) = 2 \,\rsigma(a) \,\rsigma(b), \]
by submultiplicativity proved before. Collecting inequalities, we get
\[ \rsigma(a+b)^{\,2} \leq
 \rsigma(a)^{\,2} + \rsigma(b)^{\,2} + 2 \,\rsigma(a) \,\rsigma(b)
 = {\bigl( \,\rsigma(a) + \rsigma(b) \,\bigr)}^{\,2}, \]
as required. \pagebreak
\end{proof}

\medskip
This Theorem of Pt\'{a}k has a strong converse as we shall show next.

\begin{theorem}\label{Ptakconv}%
A Banach \st-algebra $A$ is Hermitian if and only if
the Pt\'{a}k function $\rsigma$ on $A$ is a seminorm.
\end{theorem}

\begin{proof}
The ``only if'' part is clear from the preceding Theorem of Pt\'{a}k
\ref{PtakThm}. Conversely assume that the Pt\'{a}k function $\rsigma$
on the Banach \linebreak \st-algebra $A$ is a seminorm. To show
that $A$ is Hermitian, we can go over to the closed \st-subalgebra
generated by a Hermitian element, by \ref{specsubalg}. We may
thus assume that $A$ is commutative, by \ref{clogen}. \linebreak We
then have $\rsigma(ab) \leq \rsigma(a) \,\rsigma(b)$ for all $a,b \in A$
by \ref{commrs}. We \linebreak also have $\rsigma(a^*a) = {\rsigma(a) \,}^2$
and $\rsigma(a^*) = \rsigma(a)$ by \ref{rsC*}. Thus $\rsigma$ is a
C*-seminorm on $A$, and from theorem \ref{existhom} we get a
\st-algebra homomorphism $\pi$ from $A$ to a C*-algebra with
$\rsigma(a) = \| \,\pi(a) \,\|$ for all $a \in A$. Now lemma \ref{rsHerm}
and \ref{C*rs} imply that $A$ is Hermitian.
\end{proof}

\begin{theorem}
A commutative Banach \st-algebra $B$ is Hermitian if and only if
the spectral radius $\rlambda$ on $B$ is a C*-seminorm.
\end{theorem}

\begin{proof}
A commutative Banach \st-algebra $B$ is Hermitian, if and
only if $\rlambda(b) = \rsigma(b)$ for all $b \in B$, cf.\ \ref{fundHerm}
(i) $\Leftrightarrow$ (iv). This is equivalent to
$\rlambda(b)^{\,2} = \rlambda(b^*b)$ for all $b \in B$,
which is precisely the C*-property for the spectral radius
$\rlambda$ on $B$. We conclude that actually $B$ is Hermitian
if and only if $\rlambda$ satisfies the C*-property. The statement
now follows from the fact that the spectral radius $\rlambda$ on
the commutative Banach \st-algebra $B$ is an algebra seminorm
by proposition \ref{commrlsub}.
\end{proof}

\medskip
We also note the following consequence of the Theorem of Pt\'{a}k.

\begin{corollary}\label{rscont}%
The Pt\'{a}k function $\rsigma$ on a Hermitian Banach \st-algebra is
uniformly continuous.
\end{corollary}

\begin{proof}
The Pt\'{a}k function $\rsigma$ on a Hermitian Banach \st-algebra is
a C*-seminorm by the Theorem of Pt\'{a}k \ref{PtakThm}, and thus
a continuous seminorm by corollary \ref{contC*semin}.
\end{proof}

\begin{definition}[the greatest C*-seminorm, $\mathrm{m}$]%
\index{symbols}{m(a)@$\mathrm{m}(a)$}\label{greatestC*sn}%
\index{concepts}{C8@C*-seminorm!greatest}%
Let $A$ be a Banach \st-algebra.
The \underline{greatest C*-seminorm} $\mathrm{m}$ on $A$ is defined as
\[ \mathrm{m}(a) := \sup \,\{ \,p(a) \geq 0 : p \text{ is a C*-seminorm on } A \,\}
 \leq \rsigma(a) \ \ \ (a \in A). \]
Please note that $\mathrm{m}(a) \leq \rsigma(a)$ for all $a \in A$.
\pagebreak
\end{definition}

\begin{proof}
This follows from the fact that a C*-seminorm $p$ on the Banach
\st-algebra $A$ satisfies $p(a) \leq \rsigma(a)$ for all $a \in A$,
cf.\ corollary \ref{contC*semin}.
\end{proof}

\medskip
From the theorems \ref{PtakThm} \& \ref{Ptakconv}, we then get

\begin{corollary}\label{PtakRaikov}%
A Banach \st-algebra $A$ is Hermitian if and only if the Pt\'{a}k function
$\rsigma$ on $A$ equals the greatest C*-seminorm $\mathrm{m}$ on $A$.
\end{corollary}

We are next going to establish two results reminiscent of
\ref{commspeccont} above.

\medskip
Let $\mathds{D}$ denote the closed unit disc. 

\begin{theorem}\label{Normalspeccont}%
Let $A$ be a Hermitian Banach \st-algebra. For an arbitrary element
$a \in A$, and a \underline{normal} element $b \in A$, we have
\[ \s(a) \subset \s(b) + \rsigma (b-a) \,\mathds{D}. \]
By \ref{rscont}, one says that the spectrum function is uniformly
continuous on the set of normal elements of $A$. (By symmetry
in $a$ and $b$.)
\end{theorem}

\begin{proof}
Suppose the inclusion is not true. There then exists $\mu \in \s(a)$
with $\mathrm{dist}\bigl(\mu, \s(b)\bigr) > \rsigma (b-a)$.
Note that then $\mu e - b$ is invertible. By \ref{Bdistance}, we find
$\rlambda \bigl( \,{(\mu e - b)}^{\,-1} \,\bigr) \,\rsigma (b-a) < 1$. This is the same
as $\rsigma \bigl( \,{(\mu e - b)}^{\,-1} \,\bigr) \,\rsigma (b-a) < 1$, by \ref{fundHerm}
(i) $\Rightarrow$ (iv), as ${(\mu e - b)}^{\,-1}$ is normal. From \ref{PtakThm},
we get $\rsigma \bigl( \,{(\mu e - b) \,}^{-1} \,(b-a) \,\bigr) < 1$,
and so $\rlambda \bigl( \,{(\mu e - b) \,}^{-1} \,(b-a) \,\bigr) < 1$, by \ref{fundHerm}
(i) $\Rightarrow$ (iii). We have $\mu e-a = (\mu e - b) + (b-a)$. Whence
$\mu e-a = (\mu e - b) \,\bigl[ \,e + {(\mu e - b) \,}^{-1} \,(b-a) \,\bigr]$ is invertible,
a contradiction.\vphantom{$)^{-1}$} This proof is similar to that of
\ref{commspeccont}.
\end{proof}

\medskip
Of particular interest is the following consequence.

\begin{corollary}\label{Hermspeccont}%
Let $A$ be a Hermitian Banach \st-algebra. For two
\underline{Hermitian} elements $a, b \in A$, we have
\[ \s(a) \subset \s(b) + \rlambda (b-a) \,\mathds{D}
\subset \s(b) + | \,b-a \,| \,\mathds{D}. \]
One says that the spectrum function is uniformly continuous
on $A\sa$. (By symmetry in $a$ and $b$.)
\end{corollary}

\begin{proof}
This follows from the preceding theorem \ref{Normalspeccont}
because for $a, b \in A\sa$, we have $a-b \in A\sa$, so
$\rsigma (b-a) = \rlambda (b-a)$, by \ref{Hermrseqrl}.
\end{proof}

\medskip
See also \ref{Hermrlcontsa} above. \pagebreak

\clearpage

%


\section{The Polar Factorisation of Invertible Elements}%
\label{polarfactoris}

\begin{definition}[$\mspace{1mu}a^{\,1/2}\mspace{2mu}$]%
\label{invpossqrt}\index{concepts}{square root}%
\index{symbols}{a17@$a^{\,1/2}$}%
Let $A$ be a Banach \st-algebra. An \underline{invertible}
positive element $a$ of $\tld{A}$ has a unique positive
square root in $\tld{A}$. This unique positive square
root of $a$ in $\tld{A}$ is denoted by $a^{\,1/2}$. It belongs
to the closed subalgebra of $\tld{A}$ generated by $a$.
With $a$, also $a^{\,1/2}$ is invertible.
\end{definition}

\begin{proof}
Simply apply theorem \ref{possqroot}.
\end{proof}

\medskip
The next three items require Hermitian Banach \st-algebras.

\begin{definition}[$ \,| \,a\,| \,$]\label{invabsval}%
\index{concepts}{absolute value}\index{symbols}{a2@${"|}\,a\,{"|}$}%
Let $A$ be a Hermitian Banach \st-algebra.
Let $a$ be an \underline{invertible} element of $\tld{A}$.
Then $a^*a$ is an invertible positive element of $\tld{A}$.
Hence the \underline{absolute value}
\[ | \,a \,| := {( \,a^*a \,)}^{\,1/2} \]
is well-defined by the preceding item \ref{invpossqrt}.
(Even though the notation clashes with the notation for the norm.)
Furthermore, the following statements hold.
The element $| \,a \,|$ is the unique positive square root
of $a^*a$ in $\tld{A}$. It belongs to the closed \st-subalgebra
of $\tld{A}$ generated by $a$. With $a$, also the absolute
value $| \,a \,|$ is invertible.
\end{definition}

\begin{proof}
We have $a^*a \geq 0$ by the Shirali-Ford Theorem \ref{ShiraliFord}
as $\tld{A}$ is Hermitian. Also, with $a$, the elements $a^*$ and
$a^*a$ are invertible, with respective inverses ${( \,a^{\,-1} \,)}^*$
and $a^{\,-1\,} {( \,a^{\,-1} \,)}^*$. Therefore $a^*a$ is an invertible
positive element of $\tld{A}$, and the statement follows from the
preceding item \ref{invpossqrt}.
\end{proof}

\begin{definition}%
[polar factorisation]\label{polfactdef}\index{concepts}{polar factorisation}%
If $A$ is a Hermitian Banach \linebreak \st-algebra, and if
$a \in \tld{A}$ is \underline{invertible}, then a pair
$(u, p) \in \tld{A} \times \tld{A}$ shall be called a
\underline{polar factorisation} of $a$, whenever
\[ a = u p, \quad \text{with } u \text{ unitary in } \tld{A},
\text{ and } p \text{ positive in } \tld{A}. \]
We stress that this is a good definition only for invertible elements.
\end{definition}

\begin{theorem}[existence and uniqueness of a polar factorisation]%
\label{polfactexuni}\index{concepts}{factorisation!polar}%
Let $A$ be a Hermitian Banach \st-algebra. An \underline{invertible}
element $a$ of $\tld{A}$ has a unique polar factorisation $(u, p)$.
The element $p$ is the absolute value $| \,a \,|$ of $a$. \pagebreak
\end{theorem}

\begin{proof}
Assume first that $(u,p)$ is a polar factorisation of $a$. Then
$a^*a = {p\,}^2$ because $a^*a = (up)^*(up) = pu^*up = {p\,}^2$
since $p = p^*$ (as $p$ is positive), and because $u^*u = e$
(as $u$ is unitary). In other words $p$ must be a positive square
root of $a^*a$ in $\tld{A}$. We conclude that $p$ is the absolute
values $| \,a \,|$ of $a$, cf.\ \ref{invabsval}. Since $| \,a \,|$ is
invertible by \ref{invabsval}, the equation $a = u \,| \,a \,|$ implies
$u = a \,{| \,a \,| \,}^{-1}$. In particular, $u$ is uniquely
determined, and so there is at most one polar factorisation
of $a$. We can also use the necessary relationship
$u = a \,{| \,a \,| \,}^{-1}$ to define $u$. We note that
$u := a \,{| \,a \,| \,}^{-1}$ satisfies $u^*u = e$ because
\[ u^*u = {\bigl({ \,| \,a \,| \,}^{-1} \,\bigr)}^*\,a^*\,a\,{| \,a \,| \,}^{-1}
=  {\bigl({ \,| \,a \,| \,}^{-1} \,\bigr)}^*\,{| \,a \,| \,}^2\,{| \, a \,| \,}^{-1} = e. \]
We also note that $u$ is invertible, as $a$ is invertible by assumption.
So $u$ has a left inverse $u^*$, as well as some right inverse.
It follows from \ref{leftrightinv} that $u^*$ is an inverse of $u$,
i.e.\ the element $u$ is unitary. In other words, a polar factorisation
of $a$ exists, and the proof is complete.
\end{proof}

\begin{lemma}[polar factorisation of adjoint]\label{poladj}%
Consider a Hermitian Banach \st-algebra $A$, and let $a$ be an
invertible element of $\tld{A}$. Let $( \,u, | \,a \,| \,)$ be the polar
factorisation of $a$. The following statements hold.
\begin{itemize}
  \item[$(i)$] $( \,u^*, | \,a^* \,| \,)$ is the polar factorisation of $a^*$,
 \item[$(ii)$] $| \,a^* \,| = u \cdot | \,a \,| \cdot u^*$.
\end{itemize}
\end{lemma}

\begin{proof}
Please note first that with $a$, the adjoint $a^*$ is invertible.
From $a = u \,| \,a \,|$, we get
\[ a^* = | \,a \,| \,u^* = u^* \,\bigl( \,u \cdot | \,a \,| \cdot u^* \,\bigr). \]
Now the element $u \cdot | \,a \,| \cdot u^*$ is positive by \ref{stcongr}.
Also, the element $u^*$ is unitary. Hence $( \,u^*, \,u \cdot | \,a \,| \cdot u^* \,)$
must be the the polar factorisation of $a^*$, by uniqueness.
It follows that $| \,a^* \,| = u \cdot | \,a \,| \cdot u^*$.
\end{proof}

\begin{theorem}[commutativity in the polar factorisation]\label{commpolfact}%
Let $A$ be a Hermitian Banach \st-algebra, and let $a$ be an
invertible element of $\tld{A}$. Let $( \,u, | \,a \,| \,)$ be the polar
factorisation of $a$. The following statements are equivalent.
\begin{itemize}
   \item[$(i)$] the factors $u$ and $| \,a \,|$ commute,
  \item[$(ii)$] $a$ is normal,
 \item[$(iii)$] the absolute values of $a$ and $a^*$ coincide. \pagebreak
\end{itemize}
\end{theorem}

\begin{proof}
The statements (ii) and (iii) are equivalent because
\[ a^*a = {| \,a \,|}^{\,2} \qquad \text{and} \qquad aa^* = {| \,a^* \,|}^{\,2} \]
(uniqueness of the positive square root of invertible positive elements,
cf.\ \ref{invpossqrt}.) In order to see that (i) and (iii) are equivalent, we
use the fact that
\[ | \,a^* \,| = u \cdot | \,a \,| \cdot u^*, \]
cf.\ the preceding lemma \ref{poladj}. If $u$ and $| \,a \,|$ commute, then
\[ | \,a^* \,| = u \cdot | \,a \,| \cdot u^* = | \,a \,| \,u \,u^* = | \,a \,|, \]
so $| \,a^* \,| = | \,a \,|$. Conversely, if $| \,a^* \,| = | \,a \,|$, then
\[ | \,a \,| \,u = | \,a^* \,| \,u = u \cdot | \,a \,| \cdot u^* \cdot u = u \,| \,a \,|, \]
so $u$ and $| \,a \,|$ commute.
\end{proof}

\bigskip
The reader will know the terms in the following two
definitions from Hilbert space theory.

\medskip
\begin{definition}[idempotents and involutory elements]%
\index{concepts}{involutory}\index{concepts}{idempotent}%
Let $A$ be an algebra. One says that $p \in A$ is an
\underline{idempotent}, if ${p\,}^2 = p$. One says that
$a \in \tld{A}$ is \underline{involutory}, if ${a\,}^2 = e$,
that is, if $a = a^{\,-1}$.
\end{definition}

\begin{definition}[projections and reflections]\label{projrefl}%
\index{concepts}{reflection}\index{concepts}{projection}%
\index{concepts}{orthogonal!projections}%
\index{concepts}{projections!orthogonal}%
\index{concepts}{complementary projections}%
\index{concepts}{projections!complementary}%
Let $A$ be a \st-algebra.

A \underline{projection} in the \st-algebra $A$ is a
Hermitian idempotent of $A$, that is,
an element $p \in A$ with $p^* = p = {p\,}^2$.

Two projections $p$ and $q$ in $A$ are called
\underline{orthogonal}, if $p \,q = 0$. Then also
$q \,p = 0$ by taking adjoints.

Two projections $p$ and $q$ in $\tld{A}$ are
called \underline{complementary}, if they are
orthogonal, and if $p + q = e$.

One says that $u \in \tld{A}$ is a \underline{reflection},
if $u$ is both Hermitian and unitary: $u = u^* = u^{\,-1}$.
\end{definition}

\begin{example}%
If $p$ and $q$ are two complementary projections
in a unital \st-algebra, then $p - q$ is a reflection.
Indeed, one computes:
\begin{align*}
( \,p - q \,)^* \,( \,p - q \,) = ( \,p - q \,) \,( \,p - q \,)^* & = {( \,p - q \,)}^{\,2} \\
& = {p}^{\,2} + {q}^{\,2} - p \,q - q \,p \\
& = p + q = e. \pagebreak
\end{align*}
\end{example}

\begin{theorem}[the structure of reflections]\label{structrefl}%
Let $u$ be a reflection in a unital \st-algebra. Then the
following statements hold.
\begin{itemize}
   \item[$(i)$] $u$ is involutory: ${u}^{\,2} = e$,
  \item[$(ii)$] $\s(u) \subset \{ -1, 1\}$,
 \item[$(iii)$] we have $u = p - q$, where the elements
                       $p := \frac12 \,( \,e + u \,)$ and $q:= \frac12 \,( \,e - u \,)$
                       are complementary projections.
\end{itemize}
\end{theorem}

\begin{proof}
The element $u$ is involutory because $u = u^* = u^{\,-1}$ implies that
${u}^{\,2} = e$. It follows from the Rational Spectral Mapping Theorem that
$\s(u) \subset \{ -1, 1\}$. Next, the elements $p$ and $q$ as in (iii)
obviously are Hermitian, and we have $p - q= u$ as well as $p + q = e$.
To show that $p$ and $q$ are idempotents, we note that
\[ {\biggl(\frac{e \pm u}{2}\biggr)}^2 = \frac{{e}^{\,2} + {u}^{\,2} \pm 2u}{4}
= \frac{2e \pm 2u}{4} = \frac{e \pm u}{2}. \]
To show that the projections $p$ and $q$ are orthogonal, we calculate:
\[ p \,q = \frac12 \,( \,e + u \,) \cdot \frac12 \,( \,e - u \,)
= \frac14 \,( \,{e}^{\,2} - {u}^{\,2} \,) = 0. \qedhere \]
\end{proof}

\begin{corollary}%
In a unital \st-algebra, the reflections are precisely the differences
of complementary projections.
\end{corollary}

\begin{theorem}[orthogonal decomposition]\label{anorthdeco}%
\index{concepts}{orthogonal!decomposition}%
\index{concepts}{decomposition!orthogonal}%
\index{symbols}{a3@$a_+$}\index{symbols}{a4@$a_-$}
If $A$ is a Hermitian Banach \st-algebra, and if $a$ is an
\underline{invertible} Hermitian element of $\tld{A}$,
there exists a unique decomposition $a = a _{+} - a _{-}$,
with $a _{+}$ and $a _{-}$ positive elements of
$\tld{A}$ satisfying $a _{+} \,a _{-} = a _{-} \,a _{+} = 0$,
namely $a _{+} = \frac12 \,\bigl( \,| \,a \,| + a \,\bigr)$ and
$a _{-} = \frac12 \,\bigl( \,| \,a \,| -a \,\bigr)$.
\end{theorem}

\begin{proof}
To show uniqueness, let $0 \leq a _{+}$, $a _{-} \in \tld{A}$
with $a = a _{+} - a _{-}$, and $a _{+} \,a _{-} = a _{-} \,a _{+} = 0$.
Then ${a\,}^2 = {(a _{+} - a _{-}) \,}^ 2 = {(a _{+} + a _{-}) \,}^2$, whence
$a _{+} + a _{-} \geq 0$ \ref{plusconvexcone} is a positive square
root of ${a \,}^2 = a^*a$. This makes that $a _{+} + a _{-} = | \,a \,|$,
the absolute value of $a$, cf.\ theorem \ref{invabsval}. The
elements $a _{+}$ and $a _{-}$ now are uniquely determined
by the linear system $a _{+} - a _{-} = a$, $a _{+} + a _{-} = | \,a \,|$,
namely $a _{+} = \frac12 \,\bigl( \,| \,a \,| + a \,\bigr)$,
$a _{-} = \frac12 \,\bigl( \,| \,a \,| -a \,\bigr)$.

Now for existence. By theorem \ref{polfactexuni}, there exists a
unique unitary $u \in \tld{A}$ such that $a = u \,| \,a \,|$.
Please note that with $a$ and $u$, also the element
$| \,a \,| = {u \,}^{-1} \,a$ is invertible.
We claim that the elements $a$, $| \,a \,|$, ${| \,a \,| \,}^{-1}$, $u$
commute pairwise. Since $a$ is Hermitian, this follows easily
from theorem \ref{commpolfact} and proposition \ref{comm2}.
\pagebreak

Our intuition says that $u$ should be Hermitian because $a$ is so.
Indeed $u = a { \,| \,a \,| \,}^{-1}$ is Hermitian by the commutativity
of $a$ and ${| \,a \,| \,}^{-1}$, cf.\ \ref{Hermprod}.
This makes that $u$ is a Hermitian unitary element, that is:
a reflection, and thus of the form $u = p - q$ with
$p$ and $q$ two complementary projections in
$\tld{A}$, cf.\ \ref{structrefl}. We can now put
$a _{+} := p \,| \,a \,|$, $a _{-} := q \,| \,a \,|$. Please note that a
projection in $\tld{A}$ is positive by the Rational Spectral
Mapping Theorem. So, to see that the elements $a _{+}$ and
$a _{-}$ are positive, we need to show that our two
projections commute with $| \,a \,|$, cf.\ theorem \ref{instability}.
This however follows from the commutativity of $u$ and $| \,a \,|$
by the special form of our projections, see theorem \ref{structrefl} (iii).
The commutativity of our two projections with $| \,a \,|$ is used
again in the following two calculations:
$a _{+} \,a _{-} = p \,| \,a \,| \,q \,| \,a \,| =  p \,q \,{| \,a \,| \,}^2 = 0$, and
$a _{-} \,a _{+} = q \,| \,a \,| \,p \,| \,a \,| = q \,p \,{| \,a \,| \,}^2 = 0$ by
the orthogonality of our two projections.
Also $a = u \,| \,a \,| = (p - q) \,| \,a \,| = a _{+} - a _{-}$.
\end{proof}

\medskip
A stronger result holds in C*-algebras, cf.\ theorem \ref{orthdec} below.

\bigskip
We have again avoided the Gel'fand transformation,
see our remarks \ref{Zorn} and \ref{remark2}.

\clearpage


\section{The Order Structure in a C\texorpdfstring{*-}{\80\052\80\055}Algebra}

\medskip
In this paragraph, let $(A,\| \cdot \|)$ be a C*-algebra.

\begin{theorem}%
We have $A_+ \cap (-A_+) = \{ 0 \}$.
\end{theorem}

\begin{proof}
For $a \in A_+ \cap (-A_+)$, one has $\s(a) = \{ 0 \}$, whence
$\| \,a \,\| = \rlambda(a) = 0$, cf.\ \ref{C*rl}, so that $a = 0$.
\end{proof}

\begin{definition}[proper convex cone]%
A convex cone $C$ in a real vector space is called
\underline{proper} if $C \cap (-C) = \{ 0 \}$.
\end{definition}

The positive cone $A_+$ thus is a proper closed convex cone in $A\sa$.

\begin{definition}[ordered vector space]%
Let $B$ be a real vector space and let $C$
be a proper convex cone in $B$. By putting
\[ a \leq b :\Leftrightarrow b-a \in C \qquad ( \,a,b \in B \,), \]
one defines an order relation in $B$. In this way $B$ becomes
an \underline{ordered} \underline{vector space} in the following
sense:
\begin{alignat*}{3}
    a \leq b & \Rightarrow a+c \leq b+c
    \quad &\text{for all} \quad &a, b, c \in B, \\
   a \leq b & \Rightarrow \lambda a \leq \lambda b
   \quad &\text{for all} \quad &a, b \in B, \lambda \in [ \,0, \infty \,[.
\end{alignat*}
\end{definition}

The Hermitian part $A\sa$ thus is an ordered vector space.

\begin{proposition}\label{ordernorm}%
For $a \in A\sa$, we have in $\tld{A}\sa$:
\[ - \| \,a \,\| \,e \leq a \leq \| \,a \,\| \,e. \]
\end{proposition}

\begin{proposition}\label{onbded}%
If $A$ is unital, then a subset of $A\sa$ is order-bounded
if and only if it is norm-bounded.
\end{proposition}

\begin{proof}
Assume that $A$ is unital, and that $B$ is a subset of $A\sa$.
\linebreak
If $B$ is norm-bounded, then it is order-bounded by the preceding
\linebreak
proposition \ref{ordernorm}. Conversely, let $B$ be order-bounded,
and let \linebreak $x$, $y \in A\sa$ with $x \leq b \leq y$ for all $b \in B$.
For $b \in B$, we then also have $- \| \,x \,\| \,e \leq b \leq \| \,y \,\| \,e$,
again by the preceding proposition \ref{ordernorm}.
Hence $\s (b) \subset \bigl[ \, - \| \,x \,\| , \| \,y \,\| \,\bigr]$,
so $\| \,b \,\| \leq \max \,\bigl\{ \,\| \,x \,\|, \| \,y \,\| \,\bigr\}$ as
$\| \,b \,\| = \rlambda (b)$, cf.\ \ref{C*rl}. \pagebreak
\end{proof}

\begin{theorem}[$\,{a \,}^{1/2}\:$]\label{C*sqrootrestate}%
\index{concepts}{square root}\index{symbols}{a17@$a^{\,1/2}$}%
A positive element $a \in A$ has a unique positive square root
in $A$. This unique positive square root of $a$ in $A$ is denoted by
${a \,}^{1/2}$. It belongs to the closed subalgebra of $A$ generated
by $a$. (Which also is the C*-subalgebra of $A$ generated by $a$.)
The above notation is compatible with the one introduced in \ref{invpossqrt}.
\end{theorem}

\begin{proof}
This is a restatement of theorem \ref{C*sqroot}.
\end{proof}

\begin{theorem}\label{sqrtShiraliFord}
\index{concepts}{positive!element!in C*-algebra}%
For $a \in A$, the following statements are equivalent.
\begin{itemize}
   \item[$(i)$] $a \geq 0$,
  \item[$(ii)$] $a = {b\,}^2$ for some $b$ in $A_+$,
 \item[$(iii)$] $a = c^*c$ for some $c$ in $A$.
\end{itemize}
\end{theorem}

\begin{proof}
(iii) $\Rightarrow$ (i): the Shirali-Ford Theorem \ref{ShiraliFord}.
\end{proof}

\begin{theorem}%
Let $\pi$ be a \st-algebra homomorphism from $A$ to a
\st-algebra $B$. Then $\pi$ is injective if and only if $a \in A_+$
and $\pi (a) = 0$ together imply $a = 0$. 
\end{theorem}

\begin{proof}
If $0 \neq b \in \ker \pi$, then $0 \neq b^*b \in \ker \pi \cap A_+$
by the C*-property \ref{preC*alg} and the preceding theorem
\ref{sqrtShiraliFord}.
\end{proof}

\begin{definition}[$\,| \,a \,|\,$]\label{absval}%
\index{concepts}{absolute value}\index{symbols}{a2@${"|}\,a\,{"|}$}%
The \underline{absolute value} of $a \in A$ is defined as
\[ |\,a\,| := {(a^*a) \,}^{1/2}. \]
Please note that $| \,a \,|$ lies in the closed \st-subalgebra
of $A$ generated by $a$. (Which also is the C*-subalgebra
of $A$ generated by $a$.) See \ref{C*sqrootrestate}. The
above notation is compatible with the one introduced in \ref{invabsval}.
\end{definition}

\begin{proposition}\label{abschar}%
An element $a$ of $A$ is positive if and only if
\[ a = | \,a \,|. \]
\end{proposition}

\begin{proof}
This follows from the uniqueness of the positive square root,
cf.\ \ref{C*sqrootrestate}.
\end{proof}

\begin{theorem}\label{homabs}%
If $\pi$ is a \st-algebra homomorphism from $A$ to another C*-algebra, then
\[ \pi \,\bigl(\,|\,a\,|\,\bigr) = \bigl| \,\pi (a) \,\bigr|
\quad \text{for all} \quad a \in A. \pagebreak \]
\end{theorem}

\begin{proof}
It suffices to note that $\pi \,\bigl(\,|\,a\,|\,\bigr)$ is the positive square root of
\linebreak
${\pi (a)}^*\pi (a)$. Indeed $\pi \,\bigl(\,|\,a\,|\,\bigr) \geq 0$ as $|\,a\,| \geq 0$,
cf.\ \ref{hompos}, and ${\pi \,\bigl(\,|\,a\,|\,\bigr) \,}^2 = \pi \,\bigl(\,{|\,a\,| \,}^2\,\bigr)
= \pi (a^*a) = {\pi (a)}^*{\pi (a)}$.
\end{proof}

\begin{proposition}\label{isonorm}%
For $a \in A$, we have
\[ \|\,|\,a\,|\,\| = \| \,a \,\| \quad \text{ as well as } \quad
\rsigma(a) = \rlambda \bigl(\,|\,a\,|\,\bigr). \]
\end{proposition}

\begin{proof}
One calculates with \ref{C*rl}, \ref{rlpowers} \& \ref{C*rs}:
\[ {\|\,|\,a\,|\,\|\,}^2 = \,{\rlambda \bigl(\,|\,a\,|\,\bigr) \,}^2
            = \,\rlambda \bigl(\,{|\,a\,|\,}^2\,\bigr)
            = \,\rlambda (a^*a)
            = \,{\rsigma (a) \,}^2 = {\| \,a \,\|\,}^2. \qedhere \]
\end{proof}

\begin{proposition}\label{monotnorm}%
For $a, b$ $\in A\sa$, if $0 \leq a \leq b$ then $\|\,a\,\| \leq \|\,b\,\|$.
\end{proposition}

\begin{proof}
If $0 \leq a \leq b$ then $b \leq \|\,b\,\|\,e$ and
$b \pm a \geq 0$. It follows that $\|\,b\,\| \,e \pm a \geq 0$,
whence $\|\,b\,\| \geq \rlambda(a) = \|\,a\,\|$.
\end{proof}

\begin{corollary}%
For $a, b$ $\in A$, if $|\,a\,| \leq |\,b\,|$ then $\|\,a\,\| \leq \|\,b\,\|$.
\end{corollary}

\begin{proof}
If $|\,a\,| \leq |\,b\,|$, then $\|\,|\,a\,|\,\| \leq \|\,|\,b\,|\,\|$, cf.\ \ref{monotnorm},
which is the same as $\|\,a\,\| \leq \|\,b\,\|$, cf.\ \ref{isonorm}.
\end{proof}

\begin{theorem}[the orthogonal decomposition]\label{orthdec}%
\index{concepts}{orthogonal!decomposition}%
\index{concepts}{decomposition!orthogonal}%
\index{symbols}{a3@$a_+$}\index{symbols}{a4@$a_-$}%
For a \underline{Hermitian} element $a \in A$, one defines
\[ a_+ := \frac{1}{2} \,\bigl(\,|\,a\,|+a\,\bigr)\quad and 
\quad a_- := \frac{1}{2} \,\bigl(\,|\,a\,|-a\,\bigr). \]
We then have
\begin{itemize}
   \item[$(i)$] $a_+, a_- \in A_+$,
  \item[$(ii)$] $a = a_+ - a_-$,
 \item[$(iii)$] $a_+ \,a_- = a_ - \,a_+ = 0$,
\end{itemize}
and these properties characterise $a_+$, $a_-$ uniquely.

\noindent Please note that the elements $a_+$, $a_-$ lie in the
closed subalgebra of $A$ generated by $a$. (Which also is the
C*-subalgebra of $A$ generated by $a$.) See \ref{absval} and
use that $a$ is Hermitian. \\
The above notation is compatible with the one introduced in
\ref{anorthdeco}. \pagebreak
\end{theorem}

\begin{proof} We shall first show uniqueness. So let $a_1, a_2 \in A_+$
with $a = a_1-a_2$ and $a_1\,a_2 = a_2 \,a_1= 0$. We obtain
${|\,a\,| \,}^2 = {a \,}^2 = {(a_1-a_2) \,}^2 = {(a_1+a_2) \,}^2$, whence
$|\,a\,| = a_1+a_2$ by the uniqueness of the positive square root in $A$,
and so $a_1 = a_+$, $a_2 = a_-$.
We shall next prove that $a_+$ and $a_-$ do enjoy these properties.
(ii) is clear. (iii) holds because for example
$a_+ \,a_- = \frac{1}{4} \,\bigl(\,|\,a\,|+a\,\bigr) \,\bigl(\,|\,a\,|-a\,\bigr)
= \frac{1}{4} \,\bigl(\,{|\,a\,| \,}^2-{a \,}^2\,\bigr)= 0$,
using that $| \,a \,|$ commutes with $a$, as $| \,a \,|$ lies in the closed
subalgebra of $A$ generated by the Hermitian element $a$.
For the proof of (i) we make use of the weak form of the operational
calculus \ref{weakopcalc}.
Let ``$\mathrm{abs}$'' denote the function $t \to | \,t \,|$ on $\s(a)$,
and let ``$\id$'' denote the identity function on $\s(a)$.
We have $a = \id (a)$ by a defining property of the operational
calculus. We also have
$| \,a \,| = {(a^*a) \,}^{1/2} = {\bigl({ \,a \,}^2 \,\bigr) \,}^{1/2} = \mathrm{abs}(a)$
by \ref{calcplus} together with the uniqueness of the positive
square root. Now consider the functions on $\s(a)$ given by
$f_+ := \frac{1}{2} \,( \mathrm{abs} + \id )$
and $f_- := \frac{1}{2} \,( \mathrm{abs} - \id )$.
We then have $f_+ (a), f_- (a) \in A_+$ by \ref{calcplus}.
From the homomorphic nature of the operational calculus, we have
$a = f_+ (a) - f_- (a)$ as well as $f_+ (a) f_- (a) = f_- (a) f_+(a) = 0$.
It follows from the uniqueness property shown above that
$f_+ (a) = a_+$ and $f_- (a) = a_-$, and this proves the statement.
\end{proof}

\begin{corollary}\label{generic}%
The positive part $A_+$ of $A$ is a generic
convex cone in $A\sa$, i.e.\ $A\sa = A_+ +(-A_+)$.
\end{corollary}

\begin{corollary}\label{sandwich}%
For a Hermitian element $a \in A$, we have
\[ -| \,a \,| \leq a \leq | \,a \,|. \]
\end{corollary}

\begin{proof}
We have $a \leq | \,a \,|$ because $| \,a \,| - a = 2\,a_- \geq 0$.
Similarly $- | \,a \,| \leq a$ because $| \,a \,| + a = 2\,a_+ \geq 0$.
\end{proof}

\begin{corollary}\label{product}%
Every Hermitian element of the C*-algebra $A$ is the
product of two Hermitian elements of $A$. Please note
that the factors necessarily commute, cf.\ \ref{Hermprod}.
\end{corollary}

\begin{proof} For a Hermitian element $a \in A$, we compute:
\begin{align*}
a = a_+ - a_-
   & = \Bigl[ \,{{( \,a_+ \,)}^{\,1/2}} \,\Bigr]^{\,2} - \Bigl[ \,{{( \,a_- \,)}^{\,1/2}} \,\Bigr]^{\,2} \\
   & = \Bigl[ \,{{( \,a_+ \,)}^{\,1/2} + {( \,a_- \,)}^{\,1/2}} \,\Bigr]
\cdot \Bigl[ \,{{( \,a_+ \,)}^{\,1/2} - {( \,a_- \,)}^{\,1/2}} \,\Bigr],
\end{align*}
using the fact that ${(\, a_+ \,)}^{\,1/2}$ and ${( \,a_- \,)}^{\,1/2}$ commute as
they lie in the closed subalgebra generated by $a$, cf.\ \ref{orthdec} and
\ref{C*sqrootrestate}. \pagebreak
\end{proof}

\medskip
We now turn to special C*-algebras.

\begin{theorem}[absolute values in C*-subalgebras of $\ell^{\,\infty}(\Omega)$]%
\label{Coabs}%
Let $\Omega$ be any non-empty set.
Let $A$ be any C*-subalgebra of $\ell^{\,\infty}(\Omega)$.
Then the absolute value of an element $f$ of the C*-algebra $A$ is given
by the pointwise absolute value $x \mapsto | \,f(x) \,|$, $(x \in \Omega)$.
\end{theorem}

\begin{proof}
Let $f \in A$. Let $| \,f \,| : x \mapsto | \,f(x) \,|$, $(x \in \Omega)$
denote the pointwise absolute value. The function $| \,f \,|$
as well as its pointwise square root $\sqrt{| \,f \,|}$ are in $A\sa$
by proposition \ref{existsqrt}. Then $| \,f \,| \in A_+$ by the Rational
Spectral Mapping Theorem because
\[ | \,f \,| = {\sqrt{| \,f \,|}}^{\ 2} \]
and because $\sqrt{| \,f \,|}$ is Hermitian. From ${| \,f \,|}^{\,2} = \overline{f}f$
one sees that $| \,f \,|$ is a positive square root of $\overline{f} f$.
So $| \,f \,|$ is the absolute value of $f$ by uniqueness of the
positive square root.
\end{proof}

\begin{corollary}\label{pointwiseorder}
Let $\Omega$ be any non-empty set.
Let $A$ be any \linebreak C*-subalgebra of $\ell^{\,\infty}(\Omega)$.
The order in $A\sa$ then is the pointwise order.
\end{corollary}

\begin{proof} This follows now from \ref{abschar}. \end{proof}

\begin{theorem}[positivity in $\blop(H)$]\label{posop}%
\index{concepts}{positive!bounded operator}%
\index{concepts}{operator!positive bounded}%
Let $H$ be a Hilbert space and let $a \in \blop(H)$.
The following properties are equivalent.
\begin{itemize}
  \item[$(i)$] $a \geq 0$,
 \item[$(ii)$] $\langle ax, x \rangle \geq 0$ for all $x \in H$.
\end{itemize}
\end{theorem}

\begin{proof} Each of the above two properties implies that $a$ is Hermitian. \\
(i) $\Rightarrow$ (ii). If $a \geq 0$, there exists
$b \in \blop(H)\sa$ such that $a = {b\,}^2$, and so
\[ \langle ax, x \rangle = \langle {b\,}^2x, x \rangle =
\langle bx, bx \rangle \geq 0\quad \text{for all } x \in H. \]
(ii) $\Rightarrow$ (i). Assume that $\langle ax, x \rangle \geq 0$ for
all $x \in H$ and let $\lambda <  0$. Then
\[ {\| \,(a-\lambda \mathds{1})x \,\| \,}^2 =
\ {\| \,ax \,\| \,}^2 - 2 \,\lambda \,\langle ax, x \rangle + {\lambda \,}^2 \,{\| \,x \,\| \,}^2
\geq \,{\lambda \,}^2 \,{\| \,x \,\| \,}^2. \]
This implies that $a-\lambda \mathds{1}$ has a bounded left inverse
(defined on the range of $a - \lambda \mathds{1}$). 
In particular, the null space of $a-\lambda \mathds{1}$ is $\{ 0 \}$.
It follows that the range of $a-\lambda \mathds{1}$ is dense in $H$,
as $a-\lambda \mathds{1}$ is Hermitian. So the left inverse of
$a-\lambda \mathds{1}$ has a continuation to a bounded linear
operator defined on all of $H$. So $a-\lambda \mathds{1}$ is left
invertible in $\blop(H)$. This Hermitian element then is invertible
in $\blop(H)$ according to \ref{Hermlrinv}. \pagebreak
\end{proof}

\begin{proposition}\label{prepordcompl}%
Let $H$ be a Hilbert space. Suppose that an element $a \in \blop(H)\sa$ satisfies
\[ \langle a x, x \rangle = \sup_{b \in B} \,\langle b x, x \rangle
\quad \text{for all} \quad x \in H \]
for some set $B \subset \blop(H)\sa$. Then
\[ a = \sup B \]
is the least upper bound (or supremum) of $B$ in $\blop(H)\sa$.
\end{proposition}

\begin{proof}
This is an application of the preceding theorem \ref{posop}.
For $x \in H$, we have that
\[ \langle a x, x \rangle = \sup_{b \in B} \,\langle b x, x \rangle \]
is an (the least) upper bound of $\{ \,\langle b x, x \rangle \geq 0 : b \in B \,\}$,
so $a$ is an upper bound of $B$, by \ref{posop}.
To prove that $a$ is the least upper bound of $B$, let $c \in \blop(H)\sa$
be an arbitrary upper bound of $B$. We have to prove that $c \geq a$.
For $x \in H$, we have by \ref{posop} that $\langle c x, x \rangle$ is an
upper bound of $S := \{ \,\langle b x, x \rangle \geq 0 : b \in B \,\}$, and
thus greater than or equal to the least upper bound of $S$, which is
\[ \sup_{b \in B} \,\langle b x, x \rangle = \langle a x, x \rangle. \]
Therefore $c \geq a$, by \ref{posop} again, as was to be shown.
\end{proof}

\begin{theorem}[the monotone completeness of $\blop(H)\sa$]\label{ordercompl}%
Let $H$ be a Hilbert space. Let $(a_i)_{i \in I}$ be an increasing
net in $\blop(H)\sa$, which is bounded above in $\blop(H)\sa$.
There then exists in $\blop(H)\sa$ the least upper bound
\[  \sup_{i \in I} a_i \]
of the set $\{ \,a_i \in \blop(H)\sa : i \in I \,\}$.

Moreover, the net $(a_i)_{i \in I}$ converges pointwise on $H$
to its supremum. That is, with
\[ a := \sup_{i \in I} a_i \]
we have
\[ a \mspace{2mu} x = \lim _{i \in I} a_i \mspace{2mu} x
\quad \text{for all} \quad x \in H. \pagebreak \]
\end{theorem}

\begin{proof} For $x \in H$, the net
$\bigl( \langle a_i x, x \rangle \bigr)_{i \in I}$
is increasing and bounded above, so that
$\lim _{i \in I} \,\langle a_i x, x \rangle =
\sup _{i \in I} \,\langle a_i x, x \rangle$
exists and is finite. By polarisation, we may define
a Hermitian sesquilinear form $\varphi$ on $H$ via
\[ \varphi (x, y) = \lim _{i \in I} \,\langle a_i x, y \rangle \quad (x,y \in H). \]
The Hermitian sesquilinear form $\varphi$ is bounded in the sense that
\[ \sup _{\|\,x\,\|,\,\|\,y\,\| \,\leq \,1} \,|\,\varphi \,(x,y)\,| < \infty. \]
This is so because any section $( a_i )_{\,\text{\small{$i \geq i\0$}}}$
of the net is norm-bounded by proposition \ref{onbded}.
Hence there is a unique $a \in \blop(H)\sa$ with
\[ \langle ax, y \rangle = \varphi (x,y) = \lim _{i \in I}
\,\langle a_ix,y \rangle\quad \text{for all} \quad x,y \in H. \]
For $x \in H$, we get that
$\langle ax,x \rangle = \sup _{i \in I} \,\langle a_ix,x \rangle$,
and it follows by the preceding proposition \ref{prepordcompl} that
$a$ indeed is the supremum of the set $\{ \,a_i \in \blop(H)\sa : i \in I \,\}$.
It remains to be shown that the net $(a_i)_{i \in I}$ converges
pointwise to $a$. For this purpose, we use that $a-a_i \geq 0$
for all $i \in I$. For $x \in H$, we have by the C*-property \ref{preC*alg} that
\begin{align*}
& \ {\|\,(a-a_i) \,x\,\|}^{\,2} \\
\leq & \ {\| \,{(a-a_i)}^{\,1/2} \,\|}^{\,2} \cdot {\| \,{(a-a_i)}^{\,1/2\,} x \,\|}^{\,2}
= \ \|\,a-a_i\,\| \,\langle (a-a_i) \,x, x \rangle.
\end{align*}
The statement follows because any section
$( \,\|\,a-a_i\,\| \,)_{\,\text{\small{$i \geq i\0$}}}$ is
de\-creasing by \ref{monotnorm}, and thus bounded.
\end{proof}

\begin{corollary}
Let $H$ be a Hilbert space. Let $(a_i)_{i \in I}$ be an increasing
net in $\blop(H)\sa$, which is bounded above in $\blop(H)\sa$.
For $a \in \blop(H)\sa$, the following statements are equivalent.
\begin{itemize}
   \item[$(i)$] $a = \sup_{i \in I} a_i$ in $\blop(H)\sa$,
  \item[$(ii)$] $a x = \lim_{i \in I} a_i x$ for all $x \in H$,
 \item[$(iii)$] $\langle a x, x \rangle = \lim_{i \in I} \,\langle a_i x, x \rangle$
                       for all $x \in H$,
 \item[$(iv)$] $\langle a x, x \rangle = \sup_{i \in I} \,\langle a_i x, x \rangle$
                       for all $x \in H$,
\end{itemize}
\end{corollary}

\begin{proof}
The least upper bound $\sup_{i \in I} a_i$ exists in $\blop(H)\sa$
by the preceding theorem \ref{ordercompl}. The implication
(i) $\Rightarrow$ (ii) follows from the preceding theorem
\ref{ordercompl}. (ii) $\Rightarrow$ (iii) is trivial. (iii) $\Rightarrow$ (iv)
follows from theorem \ref{posop}. (iv) $\Rightarrow$ (i) follows
from proposition \ref{prepordcompl}.
\end{proof}

\medskip
We return to abstract C*-algebras with the following theorem.
It will play a decisive role in the next paragraph. \pagebreak

\begin{theorem}[monotony of the inverse operation]%
\label{monotinv}%
Let $a,b \in \tld{A}\sa$ with $0 \leq a \leq b$ and with $a$ invertible.
Then $b$ is invertible and we have $0 \leq {b \,}^{-1} \leq {a \,}^{-1}$.
\end{theorem}

\begin{proof}
Since $a \geq 0$ is invertible, we have $a \geq \alpha e$ for some
$\alpha > 0$, so that also $b \geq \alpha e$, from which it follows that
$b \geq 0$ is invertible. Also ${b \,}^{-1} \geq 0$ by the Rational Spectral
Mapping Theorem. On one hand, we have
\[ 0 \leq a \leq b = {b \,}^{1/2} \,e \,{b \,}^{1/2} \]
whence, by \ref{stcongr},
\[ 0 \leq {b \,}^{-1/2} \,a \,{b \,}^{-1/2} \leq e \]
so that
\[ \bigl\| \,{b \,}^{-1/2} \,a \,{b \,}^{-1/2} \,\bigr\|
= \rlambda \bigl( \,{b \,}^{-1/2} \,a \,{b \,}^{-1/2} \,\bigr) \leq 1. \]
On the other hand, we have
\[ \bigl\| \,{a \,}^{1/2} \,{b \,}^{-1} \,{a \,}^{1/2} \,\bigr\|
\leq \bigl\| \,{a \,}^{1/2} \,{b \,}^{-1/2} \,\bigr\|
\cdot \bigl\| \,{b \,}^{-1/2} \,{a \,}^{1/2} \,\bigr\|. \]
By the isometry of the involution and the C*-property \ref{preC*alg},
it follows that
\[ \bigl\| \,{a \,}^{1/2} \,{b \,}^{-1} \,{a \,}^{1/2} \,\bigr\|
\leq {\bigl\| \,{a \,}^{1/2} \,{b \,}^{-1/2} \,\bigr\| \,}^{2} =
\bigl\| \,{b \,}^{-1/2} \,a \,{b \,}^{-1/2} \,\bigr\| \leq 1. \]
This implies
\[ 0 \leq {a \,}^{1/2} \,{b \,}^{-1} \,{a \,}^{1/2} \leq e \]
which is the same as $b^{-1} \leq a^{-1}$, by \ref{stcongr} again.
\end{proof}

\medskip
And now for something completely different:

\begin{definition}[positive linear functionals]%
A linear functional $\varphi$ on a \st-algebra is called
\underline{positive} if $\varphi (a^*a) \geq 0$ for all elements $a$.
\end{definition}

\begin{observation}\label{plfCstar}%
A linear functional $\varphi$ on the C*-algebra $A$ is positive
if and only if $\varphi (a) \geq 0$ for all $a \in A_+$,
cf.\ \ref{sqrtShiraliFord}.
\end{observation}

\begin{observation}\label{plfpointwise}%
Let $\Omega$ be any non-empty set.
Let $A$ be any C*-subalgebra of $\ell^{\,\infty}(\Omega)$.
A linear functional $\varphi$ on $A$ is positive
if and only if $\varphi (f) \geq 0$ for every $f \in A$ with
$f \geq 0$ pointwise on $\Omega$.
\end{observation}

\begin{proof}
This follows from the preceding observation
\ref{plfCstar} together with corollary \ref{pointwiseorder}.
\pagebreak
\end{proof}

\begin{theorem}\label{plfC*bded}%
\index{concepts}{positive!functional!on C*-algebra}%
Every positive linear functional on the C*-algebra $A$ is
Hermitian and continuous.
\end{theorem}

\begin{proof} Let $\varphi$ be a positive linear functional on $A$.
The fact that $\varphi$ is Hermitian \ref{Hermfunct}
follows from the orthogonal decomposition \ref{orthdec}. It shall
next be shown that $\varphi$ is bounded on the positive part
of the unit ball of $A$. If it were not, one could choose a sequence
$(a_n)$ in the positive part of the unit ball of $A$ such that
$\varphi (a_n) \geq n^2$ for all $n \geq 1$. The elements
\[ a := \sum _{n = 1} ^{\infty} \frac{1}{n^2} \,a_n,\quad
   b_m := \sum _{1 \leq n \leq m} \frac{1}{n^2} \,a_n,\quad
   a - b_m = \sum _{n \geq m+1} \frac{1}{n^2} \,a_n \]
would all lie in the closed convex cone $A_+$. It would follow that
\[ \varphi (a) \geq \varphi (b_m) =
    \sum _{1 \leq n \leq m} \frac{\varphi (a_n)}{n^2} \geq m \]
for all integers $m \geq 1$, a contradiction. We obtain
\[ \mu := \sup \,\bigl\{ \,\varphi (a) : a \in A_+, \| \,a \,\| \leq 1 \,\bigr\}
< \infty. \]
For $b$ in $A\sa$, we have $-|\,b\,| \leq b \leq |\,b\,|$,
cf.\ \ref{sandwich}, whence
\[ | \,\varphi (b) \,| \leq \varphi \bigl(\,|\,b\,|\,\bigr)
\leq \mu \,\| \,|\,b\,| \,\| = \mu \,\|\,b\,\|. \]
For $c$ in $A$ and $c = a + \iu b$ with $a,b \in A\sa$, we have
$\|\,a\,\|, \|\,b\,\| \leq \|\,c\,\|$ (by isometry of the involution),
and so $|\,\varphi (c) \,| \leq 2 \,\mu \,\| \,c \,\|$.
\end{proof}

\medskip
\begin{remark}
Accept for a moment the Commutative Gel'fand-Na\u{\i}mark
Theorem \ref{commGN}. (Recall the deal we made in \ref{Zorn}.)
We then see from \ref{pointwiseorder} \& \ref{Cstarlattice} that the
Hermitian part of a commutative C*-algebra is a vector lattice
(with respect to the order structure of this paragraph). It turns out,
however, that the Hermitian part of a non-commutative C*-algebra
is not a lattice, cf.\ \cite[1.4.9 p.\ 14]{PedC}.
\end{remark}

\clearpage


\section{The Canonical Approximate Unit in a C\texorpdfstring{*-}{\80\052\80\055}Algebra}%
\label{canapun}%

\begin{definition}%
A \underline{two-sided approximate unit} is a net
$(e_j)_{j \in J}$ in a normed algebra $A$, such that
\[ \lim _{\text{\footnotesize{$j \in J$}}} \,e_j \,a
= \lim _{\text{\footnotesize{$j \in J$}}} \,a \,e_j
= a \quad \text{for all} \quad a \in A. \]
\end{definition}

\begin{theorem}\label{C*canapprunit}%
\index{concepts}{unit!approximate!canonical}%
\index{symbols}{J(A)@$J(A)$}%
\index{concepts}{canonical!approximate unit}%
\index{concepts}{approximate unit!canonical}%
Let $A$ be a C*-algebra. We put
\[ J(A) := \{\,j \in A_+ : \|\,j\,\| < 1\,\}. \]
Then $(j)_{j \in J(A)}$ is a two-sided approximate unit in $A$.
It is called the \linebreak \underline{canonical approximate unit}.
\end{theorem}

\begin{proof}
It shall first be shown that $J(A)$ is directed up. Let $a,b \in J(A)$.
The elements
\[ a_1 := a \,{( \,e - a \,) \,}^{-1}, \quad b_1 := b \,{( \,e - b \,) \,}^{-1} \]
then belong to $A_+$ by the Rational Spectral Mapping Theorem.
Now consider
\[ c := ( \,a_1 + b_1 \,) \,{\bigl[ \,e + ( \,a_1 + b_1 \,) \,\bigr] \,}^{-1}. \]
Then $c$ is in $J(A)$. Moreover we have
\begin{align*}
 & \ e-{\bigl[ \,e + ( \,a_1 + b_1 \,) \,\bigr] \,}^{-1} \\
 = & \ \bigl[ \,e + ( \,a_1 + b_1 \,) \,\bigr]
 \,{\bigl[ e + ( \,a_1 + b_1 \,) \,\bigr] \,}^{-1} - {\bigl[ \,e + ( \,a_1 + b_1 \,) \,\bigr] \,}^{-1} \\
 = & \ \bigl[ \,e + ( \,a_1 + b_1 \,) -e \,\bigr] \,{\bigl[ \,e + ( \,a_1 + b_1 \,) \,\bigr] \,}^{-1} = c.
\end{align*}
From the monotony of the inverse operation \ref{monotinv} it follows that
\begin{align*}
c = e-{\bigl[ \,e + ( \,a_1 + b_1 \,) \,\bigr] \,}^{-1} & \geq e - {( \,e + a_1 \,) \,}^{-1} \\
& \geq e - {\bigl[ \,e + a \,{( \,e - a \,) \,}^{-1} \,\bigr] \,}^{-1} \\
& \geq e - {\bigl[ \,( \,e - a + a \,) \,{( \,e - a \,) \,}^{-1} \,\bigr] \,}^{-1} \\
& \geq e - ( \,e - a \,) = a,
\end{align*}
and similarly $c \geq b$, which proves that $J(A)$ is directed up.

It shall be shown next that the net $(j)_{j \in J(A)}$ is a two-sided approximate
unit in $A$. We claim that it suffices to show that for each $a \in A_+$ one has
\[ \lim _{\text{\footnotesize{$j \in J(A)$}}} \| \,a \,( \,e - j \,) \,a \,\| = 0. \pagebreak \]
To see this, we note the following. For $a \in A\sa$, and $j \in J(A)$, we have
\begin{align*}
    & \ {\| \,a - j \,a \,\| \,}^2 = {\| \,( \,e - j \,) \,a \,\| \,}^2 \\
 = & \ \| \,a \,{( \,e - j \,) \,}^2 \,a\,\| = \| \,a \,{( \,e - j \,) \,}^{1/2}
 \,( \,e - j \,) \,{( \,e - j \,) \,}^{1/2} \,a \,\| \\
\leq & \ \| \,a \,{( \,e - j \,) \,}^{1/2} \,{( \,e - j \,) \,}^{1/2} \,a \,\| = \| \,a \,( \,e - j \,) \,a \,\|.
\end{align*}
One also uses the fact that $A$ is linearly spanned by $A_+$,
cf.\ \ref{generic}. The isometry of the involution is used to pass
from multiplication on the left to multiplication on the right. 

So let $a \in A_+$, $\varepsilon > 0$. Put
\[ \alpha := \bigl( \,1+ \|\,a\,\| \,\bigr) / \varepsilon > 0, \]
and consider
\[ j\0 := \alpha a \,{( \,e + \alpha a \,) \,}^{-1}, \]
which belongs to $J(A)$ by the Rational Spectral Mapping Theorem.
For $j \in J(A)$ with $j \geq j\0$, we get
\begin{align*}
0 \leq a \,( \,e - j \,) \,a \leq a \,( \,e - j\0 \,) \,a
= & \ {a \,}^2 \,\bigl[ \,e - \alpha a \,{( \,e + \alpha a \,) \,}^{-1} \,\bigr] \\
 = & \ {a \,}^2 \,\bigl[ \,( \,e + \alpha a - \alpha a \,) \,{( \,e + \alpha a \,) \,}^{-1} \,\bigr] \\
 = & \ {a \,}^2 \,{( \,e + \alpha a \,) \,}^{-1} \\
 = & \ {\alpha \,}^{-1}\,a \cdot {\alpha a} \,{( \,e + \alpha a \,) \,}^{-1} \\
 \leq & \ {\alpha \,}^{-1}\,a,
\end{align*}
where the last inequality follows from the
Rational Spectral Mapping Theorem. Hence
\[ \|\,a \,( \,e - j \,) \,a \,\| \leq {\alpha \,}^{-1} \,\| \,a \,\|
= \varepsilon \cdot \frac{\| \,a \,\|}{1 + \|\,a\,\|} < \varepsilon. \qedhere \]
\end{proof}

\clearpage


\section{Quotients and Images of C\texorpdfstring{*-}{\80\052\80\055}Algebras}

\begin{reminder}[unital homomorphism]%
Recall from \ref{unitalhom} that if $A, B$ are unital algebras, then an algebra
homomorphism $A \to B$ is called \underline{unital} if it maps unit to unit.
\end{reminder}

\begin{proposition}[unitisation of a homomorphism]%
\label{homunit}\index{concepts}{homomorphism!unitisation}%
\index{concepts}{unitisation!of homomorphism}%
Let $A, B$ be normed algebras, and let $\pi : A \to B$ be a
non-zero algebra homomorphism with dense range. The mapping
$\pi$ then has a (necessarily unique) extension to a unital algebra
homomorphism $\tld{\pi} : \tld{A} \to \tld{B}$. We shall say that $\tld{\pi}$
is the \underline{unitisation} of $\pi$. If $\pi$ is injective, so is $\tld{\pi}$.
\end{proposition}

\begin{proof}
If $A$ has a unit $e_A$, then $\pi (e_A) \neq 0$ is a unit in $B$,
because $\pi$ is non-zero, and because $\pi (A)$ is dense in $B$.
If $A$ has no unit, we extend $\pi$ to a unital algebra homomorphism
$\tld{A} \to \tld{B}$ by requiring that unit be mapped to unit. Injectivity
then is preserved because the unit in $\tld{B}$ is linearly independent
of $\pi (A)$, as $\pi (A)$ has no unit by injectivity.
\end{proof}

\begin{lemma}\label{injpos}%
Let $A$, $B$ be C*-algebras.
Let $\pi : A \to B$ be an injective \st-algebra homomorphism.
For $a \in A$, we then have
\[ \pi (a) \geq 0 \Leftrightarrow a \geq 0. \]
\end{lemma}

\begin{proof}
We know that $a \geq 0 \Rightarrow \pi (a) \geq 0$, cf.\ \ref{hompos}.
Let now $a \in A$ with $\pi (a) \geq 0$. Then
$\pi(a) = \bigl| \,\pi(a) \,\bigr| = \pi \,\bigl( \,| \,a \,| \,\bigr)$,
cf.\ \ref{abschar} \& \ref{homabs}. Hence $a = | \,a \,| \geq 0$
by injectivity of $\pi$, as was to be shown.
\end{proof}

\begin{lemma}\label{posposrs}%
Let $A$, $B$ be two unital Banach \st-algebras, of which $B$ is Hermitian.
Let $\pi : A \to B$ be a unital \st-algebra homomorphism. Assume that
\[ \pi (b) \geq 0 \Rightarrow b \geq 0
\qquad \text{for every} \quad b \in A\sa. \]
Then
\[ \rsigma \bigl( \pi (a) \bigr) = \rsigma (a)
\qquad \text{for every} \quad a \in A.  \]
\end{lemma}

\begin{proof}
We have $\rlambda \bigl(\pi (a)\bigr) \leq \rlambda (a)$
for all $a \in A$ because
\[ \s_B \bigl(\pi(a)\bigr) \setminus \{ 0 \} \subset \s_A (a) \setminus \{ 0 \}, \]
cf.\ \ref{spechom}.
\pagebreak

Let now $b$ be a Hermitian element of $A$. Then
\[ \rlambda \bigl(\pi(b)\bigr) \,e_B \pm \pi(b) \geq 0 \]
as $B$ is Hermitian. Our assumption then implies
\[ \rlambda \bigl(\pi(b)\bigr) \,e_A \pm b \geq 0, \]
whence $\rlambda (b) \leq \rlambda \bigl(\pi(b)\bigr)$.
Collecting inequalities, we find
\[ \rlambda \bigl(\pi(b)\bigr) = \rlambda(b)
\quad \text{for every Hermitian element} \quad b \in A, \]
which is enough to prove the statement.
\end{proof}

\begin{theorem}\label{C*injisometric}%
An injective \st-algebra homomorphism $\pi$  from a \linebreak
C*-algebra $A$ to a pre-C*-algebra $B$ is isometric.
As a consequence, $\pi (A)$ is complete.
\end{theorem}

\begin{proof}
We can assume that $B$ is a C*-algebra, and that $\pi$ has dense
range and is non-zero. The unitisation $\tld{\pi}$ then is well-defined
and injective by proposition \ref{homunit}. The preceding two lemmata
imply that $\rsigma \bigl(\pi(a)\bigr) = \rsigma(a)$ for all $a \in A$,
which is the same as $\| \,\pi(a) \,\| = \| \,a \,\|$ for all $a \in A$, cf.\ \ref{C*rs}.
\end{proof}

\begin{remark}\label{reminjiso}
This is a strong improvement of \ref{Cstarisom} in so far
as it allows us to conclude that $\pi(A)$ is a C*-algebra.
\end{remark}

Next an expression for the quotient norm.
(Cf.\ the appendix \ref{quotspace}.)
Here we use the canonical approximate unit,
see the preceding paragraph, \ref{canapun}.

\begin{lemma}[quotient norm]\label{quotexpr}%
\index{concepts}{quotient!norm}\index{concepts}{norm!quotient}%
Let $I$ be a \st-stable \ref{selfadjointsubset} closed \linebreak
two-sided ideal in a C*-algebra $A$. For $a \in A$ we then have
\[ \inf _{\text{\footnotesize{$c \in I$}}} \,\| \,a + c \,\|
= \lim _{\text{\footnotesize{$j \in J(I)$}}} \,\| \,a - a \,j \,\|. \]
\end{lemma}

\begin{proof}
Let $a \in A$. Denote $\alpha := \inf _{\,\text{\footnotesize{$c \in I$}}} \,\| \,a + c \,\|$.
For $\varepsilon > 0$, take $c \in I$ with
$\| \,a + c \,\| < \alpha + \varepsilon / 2$. Then choose
$j\0 \in J(I)$ such that $\| \,c - c \,j \,\| < \varepsilon / 2$
whenever $j \in J(I)$, $j \geq j\0$. Under
these circumstances we have
\begin{align*}
\| \,a - a \,j \,\| \leq & \ \| \,( \,a + c \,) \,( \,e - j \,)\,\| + \|\,c - c \,j \,\| \\
      \leq & \ \| \,a + c \,\| + \|\,c - c \,j \,\| < \alpha + \varepsilon. \pagebreak \qedhere
\end{align*}
\end{proof}

\begin{theorem}[factorisation]%
\label{factor}\index{concepts}{quotient!C*-algebra}\label{C*quotient}%
Let $I$ be a \st-stable \ref{selfadjointsubset} closed
two-sided ideal in a C*-algebra $A$. Then
$A\mspace{1mu}/I$ is a C*-algebra.
\end{theorem}

\begin{proof}
It is clear that $A\mspace{1mu}/I$ is a Banach \st-algebra, cf.\ the appendix
\ref{quotspaces}. The preceding lemma \ref{quotexpr} gives 
the following expression for the square of the quotient norm:
\begin{align*}
{| \,a + I \,| \,}^2 &
= \inf _{\text{\footnotesize{$c \in I$}}} {\| \,a + c \,\| \,}^2 \\
 & = \lim _{\text{\footnotesize{$j \in J(I)$}}} \,{\| \,a \,( \,e - j \,)\,\| \,}^2 \\
 & = \lim _{\text{\footnotesize{$j \in J(I)$}}} \,\| \,( \,e - j \,) \,a^*a \,( \,e - j \,)\,\| \\
 & \leq \lim _{\text{\footnotesize{$j \in J(I)$}}} \,\| \,a^*a \,( \,e - j \,)\,\|
 = \inf _{\text{\footnotesize{$c \in I$}}} \| \, a^*a+c \,\|
 = |\,a^*a+I\,|.
\end{align*}
The statement follows now from \ref{condC*}.
\end{proof}

\begin{theorem}%
Let $\pi$ be any \st-algebra homomorphism from a \linebreak C*-algebra
$A$ to a pre-C*-algebra. Then $A \,/ \ker \pi$ is a C*-algebra.
The \st-algebra isomorphism
\begin{alignat*}{2}
 & A \,/ \ker \pi & \ \to & \ \pi (A)  \\
 & a+\ker \pi & \ \mapsto & \ \pi (a)
\end{alignat*}
is isometric, and hence the range $\pi (A)$ is a C*-algebra.
\end{theorem}

\begin{proof} It follows from \ref{Cstarcontr} that $\ker \pi$
is closed, so that $A \,/ \ker \pi$ is a C*-algebra by the
preceding theorem \ref{factor}. The statement then follows
from \ref{C*injisometric}.
\end{proof}

\medskip
In other words:

\begin{theorem}[factorisation of homomorphisms]\label{homfact}%
Let $\pi$ be a \linebreak \st-algebra homomorphism from a C*-algebra
$A$ to a pre-C*-algebra. \linebreak Then $A \,/ \ker \pi$ and $\pi(A)$ are
C*-algebras. The \st-algebra homomorphism $\pi$ factors to an
isomorphism of C*-algebras from $A \,/ \ker \pi$ onto $\pi(A)$.
\end{theorem}

In particular, the C*-algebras form a subcategory of the category of
\st-algebras, whose morphisms are the \st-algebra homomorphisms.
\pagebreak

\clearpage

\addtocontents{toc}{\protect\pagebreak}


\part{Representations and Positive Linear Functionals}\label{part2}

\addtocontents{toc}{\protect\vspace{0.4em}}

\chapter{General Properties of Representations}


\setcounter{section}{24}

\section{Continuity Properties of Representations}

\medskip
For a complex vector space $V$, we will denote by
$\mathrm{End}(V)$ the algebra of linear operators on $V$.
Here ``$\mathrm{End}$'' stands for ``endomorphism''.

\begin{definition}[representation on a pre-Hilbert space]%
\index{concepts}{representation}\index{symbols}{R3@$\range(\pi)$}%
\index{concepts}{representation!space}\label{repdef}%
If $A$ is a \linebreak \st-algebra, and if $(H, \langle \cdot , \cdot \rangle)$
is a \underline{pre}-Hilbert space, then a \underline{representation} of $A$
on $H$ is an algebra homomorphism
\[ \pi : A \to \mathrm{End}(H) \]
such that
\[ \langle \pi (a)x,y \rangle = \langle x, \pi (a^*)y \rangle \]
for all $x,y \in H$ and for all $a \in A$. One then defines the \underline{range}
of $\pi$ as
\[ \range(\pi) := \pi (A) = \{ \pi (a) \in \mathrm{End}(H) : a \in A \}. \]
One says that $(H, \langle \cdot , \cdot \rangle)$ is the
\underline{representation space} of $\pi$. (Recall that all the pre-Hilbert
spaces in this book shall be \underline{complex} pre-Hilbert spaces.)
\end{definition}

\begin{theorem}\label{HellToepl}
Let $\pi$ be a representation of a \st-algebra $A$ on a \underline{Hilbert}
space $H$. The operators $\pi(a)$ $(a \in A)$ then all are bounded. \linebreak
We may thus consider $\pi$ as a \st-algebra homomorphism $A \to \blop(H)$.
\end{theorem}

\begin{proof}
Decompose $\pi (a) = \pi (b) + \iu \pi (c)$, where $b,c \in A\sa$ and use the
Hellinger-Toeplitz Theorem, which says that a symmetric operator, defined
on all of a Hilbert space, is bounded, see e.g.\ Kreyszig \cite[10.1-1 p.\ 525]{Krey}.
(See also Pedersen \cite[Proposition 2.3.11 p.\ 59]{PedB}.)
\end{proof}

\medskip
We shall now be interested in a criterion that guarantees boundedness
of the operators $\pi (a)$ for a representation $\pi$ acting on a pre-Hilbert
space merely: \ref{weaksa} - \ref{sigmaAsa}. \pagebreak

\begin{definition}[positive linear functionals]%
\index{concepts}{functional!positive}\index{concepts}{positive!functional}%
A linear functional $\varphi$ on a \st-algebra $A$ is called
\underline{positive} if $\varphi (a^*a) \geq 0$ for all $a \in A$.
\end{definition}

If $\pi$ is a representation of a \st-algebra $A$ on a pre-Hilbert space $H$,
then for all $x \in H$, the linear functional given by
\[ a \mapsto \langle \pi(a) x, x \rangle \qquad ( \,a \in A \,) \]
is a positive linear functional. Indeed, for $x \in H$ and $a \in A$, we have
\[ \langle \pi (a^*a) x, x \rangle = \langle \pi (a) x, \pi (a) x \rangle \geq 0. \]

\begin{definition}[weak continuity]%
\index{concepts}{representation!weakly continuous!on Asa@on $A\protect\sa$}%
\index{concepts}{weakly continuous!As@on $A\protect\sa$}%
\index{concepts}{representation!weakly continuous}%
\index{concepts}{weakly continuous}%
Let $\pi$ be a representation of a \linebreak normed \st-algebra $A$ on a
pre-Hilbert space $H$. We shall say that $\pi$ is
\underline{weakly continuous on a subset} $B$ of $A$, if for all $x \in H$,
the positive linear functional
\[ a \mapsto \langle \pi (a) x, x \rangle \]
is continuous on $B$. We shall say that the representation $\pi$ is
\underline{weakly} \underline{continuous}, if $\pi$ is weakly continuous
on all of $A$.
\end{definition}

\begin{theorem}\label{weaksa}%
Let $\pi$ be a representation of a normed \st-algebra $A$ on a
pre-Hilbert space $H$. Assume that $\pi$ is weakly continuous on $A\sa$.
The operators $\pi (a)$ $(a \in A)$ then all are bounded with
\[ \|\,\pi (a) \,\| \leq \rsigma(a)\quad \text{for all} \quad a \in A. \]
\end{theorem}

\begin{proof}
Let $a \in A$ and $x \in H$ with $\|\,x\,\| \leq 1$ be fixed. For $n \geq 0$, put
\[ r_n := {\Bigl({ \,\| \,\pi (a) \,x \,\| \,}^2 \,\Bigr) \,}^{({ \,{2 \,}^n \,})}
\qquad \text{and} \qquad
s_n := \langle \,\pi \,\Bigl( { \,( \,a^*a \,) \,}^{( \,{2 \,}^n \,)} \,\Bigr) \,x , \,x \,\rangle. \]
We then have $r_n \leq s_n$ for all $n \geq 0$. Indeed, this obviously holds
for $n = 0$. Let now $n$ be any integer $\geq 0$. The Cauchy-Schwarz
inequality yields
\begin{align*}
{| \,s_n \,| \,}^2 \,= &
\ { \Bigl|\,\langle \,\pi \,\Bigl( { \,( \,a^*a \,) \,}^{( \,{2 \,}^n \,)} \,\Bigr) \,x, \,x \,\rangle \,\Bigr| \,}^2 \\
    \leq & \ \langle \,\pi \,\Bigl( { \,( \,a^*a \,) \,}^{( \,{2 \,}^n \,)} \,\Bigr) \,x,
    \,\pi \,\Bigl( { \,( \,a^*a \,) \,}^{( \,{2 \,}^n \,)} \,\Bigr) \,x \,\rangle {\,\| \,x \,\| \,}^2 \\
    \leq & \ \langle \,\pi \,\Bigl( { \,( \,a^*a \,) \,}^{( \,{2 \,}^{n+1} \,)} \,\Bigr) \,x, \,x \,\rangle = s_{n+1}.
\end{align*}
Thus, if $r_n \leq s_n$, then $s_n = |\,s_n\,|$, and hence
\[ r_{n+1} = {r_n \,}^2 \leq {s_n \,}^2 = {|\,s_n\,| \,}^2 \leq s_{n+1}. \pagebreak \]
By induction it follows that $r_n \leq s_n$ for all $n \geq 0$.
Let now $c > 0$ be such that
\[ |\, \langle \pi (b)x, x \rangle \,| \leq c \,|\,b\,| \]
for all $b \in A\sa$. For $n \geq 0$, we get
\begin{align*}
\|\,\pi (a) \,x \,\| \ & \leq
\ {\Biggl[ { \,\langle \,\pi \,\Bigl({ \,( \,a^*a \,) \,}^{( \,{2 \,}^n \,) \,}\Bigr) \,x,
\,x \,\rangle \,}^{1\,/ \,( \,{{2 \,}^{n}} \,)} \,\Biggr] \,}^{1 \,/ \,2} \\
 & \leq \ {c \,} ^{1 \,/  \,(\,{{2 \,}^{n+1}}\,)}
 \cdot{\Biggl[\,{ \,\Bigl| \,{( \,a^*a \,) \,}^{( \,{2 \,}^n \,)} \,\Bigr| \,}^{1 \,/ \,( \,{{2\,}^{n}} \,)} \,\Biggr] \,}^{1 \,/ \,2}
\to \,\rsigma(a).
\end{align*}
It follows that
\[ \|\,\pi (a) \,x \,\| \leq \rsigma(a). \qedhere \]
\end{proof}

\begin{definition}[$\sigma$-contractive linear maps]%
\label{sigmacontr}\index{concepts}{s1@$\sigma$-contractive}%
\index{concepts}{representation!s-contractive@$\sigma$-contractive}%
\index{concepts}{contractive!s-contractive@$\sigma$-contractive}%
A linear map $\pi$ from a normed \st-algebra $A$ to a normed
space shall be called \underline{$\sigma$-contractive} if
\[ |\,\pi (a)\,| \leq \rsigma(a) \]
holds for all $a \in A$.
\end{definition}

\begin{proposition}\label{sigmaAsa}
Let $\pi$ be a $\sigma$-contractive linear map from a normed
\st-algebra $A$ to a normed space. We then have:
\begin{itemize}
   \item[$(i)$] $\pi$ is contractive in the accessory \st-norm on $A$, see \ref{auxnorm},
  \item[$(ii)$] $|\,\pi (a)\,| \leq \rlambda(a) \leq |\,a\,|$ for all Hermitian $a \in A$,
 \item[$(iii)$] $\pi$ is contractive on $A\sa$.
\end{itemize}
\end{proposition}

\begin{proof}
(i) follows from the fact that $\rsigma(a) \leq \|\,a\,\|$ holds for all $a \in A$,
where $\|\,\cdot\,\|$ denotes the accessory \st-norm. (ii) follows from the fact
that for a Hermitian element $a$ of $A$, one has $\rsigma(a) = \rlambda(a)$,
cf.\ \ref{Hermrseqrl}.
\end{proof}

\begin{corollary}\label{reprisomcontr}
If $A$ is a normed \st-algebra with isometric involution, then every
$\sigma$-contractive linear map from $A$ to a normed space is contractive.
\end{corollary}

\begin{proposition}\label{commsubBsigmacontr}
If $A$ is a commutative \st-subalgebra of a Banach \st-algebra, then every
$\sigma$-contractive linear map from $A$ to a normed space is contractive.
\end{proposition}

\begin{proof}
This follows from the fact that in a Banach \st-algebra, one has
$\rsigma(b) \leq \rlambda(b)$ for all normal elements $b$,
cf.\ \ref{Bnormalrsleqrl}. \pagebreak
\end{proof}

\begin{proposition}\label{continvol}
Let $A$ be a normed \st-algebra. For the following statements, we have
\[ (i) \Rightarrow (ii) \Rightarrow (iii) \Rightarrow (iv) \Rightarrow (v) \Rightarrow (vi). \]
\begin{itemize}
   \item[$(i)$] the involution in $A$ is continuous,
  \item[$(ii)$] the involution maps sequences converging to $0$
                       to bounded \linebreak sequences,
 \item[$(iii)$] the map $a \mapsto a^*a$ is continuous at $0$,
 \item[$(iv)$] the Pt\'{a}k function $\rsigma$ is continuous at $0$,
  \item[$(v)$] the Pt\'{a}k function $\rsigma$ is bounded in a
neighbourhood of $0$,
 \item[$(vi)$] every $\sigma$-contractive linear map from $A$
to a normed space is continuous.
\end{itemize}
\end{proposition}

\begin{corollary}\label{Hermcont}%
If $A$ is a \st-subalgebra of a Hermitian Banach \linebreak \st-algebra,
then every $\sigma$-contractive linear map from $A$ to a normed space
is continuous.
\end{corollary}

\begin{proof}
The Pt\'{a}k function $\rsigma$ on a Hermitian Banach \st-algebra
is uniformly continuous, cf.\ \ref{rscont}.
\end{proof}

\medskip
In the preceding corollary \ref{Hermcont}, the assumption of a Hermitian
nature can be dropped, cf.\ \cite[vol.\ II Theorem 11.1.4 p.\ 1157]{Palm}.

\bigskip
Next for slightly stronger assumptions: weakly
continuous representations on Hilbert spaces.

\begin{theorem}\label{uniformbd}%
Let $H, H\0$ be normed spaces, and let $C$ be a subset of the
normed space $\blop(H,H\0)$ of bounded linear operators $H \to H\0$.
If $H$ is complete, then $C$ is bounded in $\blop(H,H\0)$ if and only if
\[ \sup _{T \in C} |\,\ell(Tx)\,| < \infty \]
for every $x$ in $H$ and every continuous linear functional $\ell$ on $H\0$.
\end{theorem}

\begin{proof} The uniform boundedness principle is applied twice.
At the first time, one imbeds $H\0$ isometrically into its second dual.
\end{proof}

\begin{theorem}\label{weakcontHilbert}
Let $\pi$ be a \underline{weakly continuous} representation of a normed
\st-algebra $A$ on a \underline{Hilbert} space $H$. Then $\pi$ is continuous.
\linebreak Furthermore, for a normal element $b$ of $A$, we have
\[ \| \,\pi(b) \,\| \leq \rlambda(b) \leq | \,b \,|. \]
It follows that if $A$ is commutative, then $\pi$ is contractive. \pagebreak
\end{theorem}

\begin{proof}
The operators $\pi (a)$ $(a \in A)$ are bounded by \ref{HellToepl} or \ref{weaksa}.
The linear functionals $A \ni a \mapsto \langle \pi (a) x, y \rangle$ $(x, y \in H)$
are continuous by polarisation. One then applies the preceding theorem
\ref{uniformbd} with \linebreak $C := \{ \,\pi (a) \in \blop(H) : a \in A, \,|\,a\,| \leq 1 \,\}$
to the effect that $\pi$ is con\-tinuous. The remaining statements follow from
\ref{commbdedcontr}.
\end{proof}

\medskip
Now for Banach \st-algebras.

\begin{lemma}\label{preBanachbded}
Let $\varphi$ be a positive linear functional on a \underline{unital}
\linebreak Banach \st-algebra $A$. For $a \in A\sa$, we then have
\[ |\,\varphi (a)\,| \leq \varphi (e) \,\rlambda (a) \leq \varphi (e) \,| \,a \,|. \]
In particular, $\varphi$ is continuous on $A\sa$. See also \ref{eBbded}
below.
\end{lemma}

\begin{proof}
Let $b$ be a Hermitian element of $A$ with $\rlambda(b) < 1$.
There exists by Ford's Square Root Lemma \ref{Ford} a Hermitian square
root $c$ of $e-b$. One thus has $c^*c = c^2 = e-b$, and it follows that
\[ 0 \leq \varphi (c^*c) = \varphi (e-b), \]
whence $\varphi (b)$ is real (because $\varphi (e) = \varphi (e^*e)$ is), and
\[ \varphi (b) \leq \varphi (e). \]
Upon replacing $b$ with $-b$, we obtain
\[ |\,\varphi (b) \,| \leq \varphi (e). \]
Let now $a$ be an arbitrary Hermitian element of $A$. For
$\gamma > \rlambda(a)$, the above implies that
\[ |\,\varphi (\gamma^{-1}a)\,| \leq \varphi (e), \]
whence
\[ |\,\varphi (a)\,| \leq \varphi (e) \,\gamma. \]
Since this holds for arbitrary $\gamma > \rlambda(a)$, we get
\[ |\,\varphi (a)\,| \leq \varphi (e) \,\rlambda (a). \qedhere \]
\end{proof}

\begin{theorem}\label{Banachbounded}%
Let $\pi$ be a representation of a Banach \st-algebra $A$ on a pre-Hilbert
space $H$. The operators $\pi(a)$ $(a \in A)$ then all are bounded, and
the representation $\pi$ is continuous. \pagebreak
\end{theorem}

\begin{proof}
Extend $\pi$ to a representation $\tld{\pi}$ of $\tld{A}$ on $H$ by
requiring that $\tld{\pi} (e) = \mathds{1}$ if $A$ has no unit (and by
extending linearly). By the preceding lemma \ref{preBanachbded},
the positive linear functional $a \to \langle \tld{\pi} (a) x, x \rangle$
$(a \in \tld{A})$ is continuous on $\tld{A}\sa$ for each $x \in H$.
This implies that $\pi$ is weakly continuous on $A\sa$. It follows
by \ref{weaksa} that the operators $\pi (a)$ are bounded on $H$.
We obtain that $\range(\pi)$ is a pre-C*-algebra, so that $\pi$ is
continuous by theorem \ref{automcont}.
\end{proof}

\begin{corollary}
A representation of a commutative Banach \linebreak
\st-algebra on a pre-Hilbert space is contractive.
\end{corollary}

\begin{proof}
This follows now from \ref{weakcontHilbert}.
\end{proof}

\begin{corollary}
A representation of a Banach \st-algebra with \linebreak isometric
involution on a pre-Hilbert space is contractive.
\end{corollary}

\begin{proof}
This follows now from \ref{weaksa} and \ref{reprisomcontr}.
\end{proof}

\begin{corollary}\label{C*repcontract}%
Every representation of a C*-algebra on a pre-Hilbert space is contractive.
\end{corollary}

\begin{definition}[faithful representation]%
\index{concepts}{representation!faithful}%
\index{concepts}{faithful representation}%
A representation of a \st-algebra on a pre-Hilbert
space is called \underline{faithful} if it is injective.
\end{definition}

\begin{theorem}\label{faithisometric}
A faithful representation of a C*-algebra on a \linebreak Hilbert space is isometric.
\end{theorem}

\begin{proof} \ref{C*injisometric}. \end{proof}

\begin{theorem}\label{rangeC*}
If $\pi$ is a representation of a C*-algebra $A$ on a Hilbert space, then
$A \,/ \ker \pi$ and $\range(\pi)$ are C*-algebras. The representation $\pi$
factors to a C*-algebra isomorphism from $A \,/ \ker \pi$ onto $\range(\pi)$.
\end{theorem}

\begin{proof} \ref{homfact}. \end{proof}

\medskip
For further continuity results, see \ref{bdedstatesdef} - \ref{kerpsi}.

\begin{remark}%
The condition of isometry of the involution can sometimes be weakened
to requiring that $| \,a^*a \,| \leq {| \,a \,|}^{\,2}$ holds for all $a$.\linebreak%
\pagebreak%
\end{remark}

\clearpage


\section{Direct Sums}%
\label{DirectSums}

\begin{definition}[direct sum of Hilbert spaces]%
\index{concepts}{direct sum!of Hilbert spaces}%
\index{symbols}{p01@$\oplus_{i \in I} x_i$}%
\index{symbols}{p02@$\oplus_{i \in I} H_i$}\label{dirHil}%
Let $\{\,H_i\,\}_{\,i \in I}$ be a family of Hilbert spaces. The direct sum
$\oplus _{i \in I} \,H_i$ consists of all families $\{\,x_i\,\}_{\,i \in I}$ such
that $x_i \in H_i$ $(i \in I)$ and $\{ \,\| \,x_i \,\| \,\}_{\,i \in I} \in {\ell\,}^2(I)$.
The element $\{\,x_i\,\} _{\,i \in I}$ then is denoted by $\oplus _{i \in I} \,x_i$.
The direct sum $\oplus _{i \in I} \,H_i$ is a Hilbert space with respect to
the inner product
\[ \langle \oplus _{i \in I} \,x_i,\oplus _{i \in I} \,y_i \rangle
:= \sum _{i \in I} \langle x_i, y_i \rangle
\qquad ( \,\oplus _{i \in I} \,x_i, \oplus _{i \in I} \,y_i \in \oplus _{i \in I} \,H_i \,). \]
\end{definition}

\begin{proof}
Let $\oplus _{i \in I} \,x_i$, $\oplus _{i \in I} \,y_i
\in \oplus _{i \in I} \,H_i$. We then have
\begin{align*}
{\biggl( \,\sum _{i \in I} {\| \,x_i+y_i \,\| \,}^2 \biggr)}^{1/2}
\leq & \ {\biggl( \,\sum _{i \in I} {\Bigl( \,\| \,x_i \,\|
+ \|\, y_i \,\| \,\Bigr) \,}^2 \,\biggr)}^{1/2} \\
\leq & \ {\biggl( \,\sum _{i \in I}{\| \,x_i \,\| \,}^2 \biggr)}^{1/2}
+ {\biggl( \,\sum _{i \in I} {\| \,y_i \,\| \,}^2 \biggr)}^{1/2}
\end{align*}
so that $\oplus _{i \in I} \,H_i$ is a vector space. We also have
\[ \sum _{i \in I} | \,\langle x_i ,y_i \rangle \,|
\leq \ \sum _{i \in I}\|\, x_i \,\| \cdot \| \,y_i \,\|
\leq \ {\biggl( \,\sum _{i \in I} {\| \,x_i \,\| \,}^2 \biggr)}^{1/2}
     {\biggl( \,\sum _{i \in I} {\| \,y_i \,\| \,}^2 \biggr)}^{1/2}, \]
so that $\oplus _{i \in I} \,H_i$ is a pre-Hilbert space.
Let $( \,\oplus _{i \in I} \,x_{i,n} \,)_{\,n}$ be a Cauchy sequence in
$\oplus _{i \in I} \,H_i$. For given $\varepsilon > 0$, there then exists
$n\0(\varepsilon)$ such that
\[ {\biggl( \,\sum _{i \in I} {\| \,x_{i,m} - x_{i,n} \,\| \,}^2 \biggr)}^{1/2}
\leq \varepsilon\quad \text{for all} \quad m,n \geq n\0(\varepsilon). \]
From this we conclude that for each $i \in I$, the sequence
$( \,x_{i,n} \,)_{\,n}$ is a Cauchy sequence in
$H_i$, and therefore converges to an element $x_i$ of $H_i$.
With $F$ a finite subset of $I$, we have
\[ {\biggl( \,\sum _{i \in F} {\| \,x_{i,m} - x_{i,n} \,\| \,}^2 \biggr)}^{1/2}
\leq \varepsilon\quad \text{for all} \quad m,n \geq n\0(\varepsilon), \]
so that by letting $m \to \infty$, we get
\[ {\biggl( \,\sum _{i \in F} {\| \,x_i - x_{i,n} \,\| \,}^2 \biggr)}^{1/2}
\leq \varepsilon\quad \text{for all} \quad n \geq n\0(\varepsilon), \]
and since $F$ is an arbitrary finite subset of $I$, it follows that
\[ {\biggl( \,\sum _{i \in I} {\| \,x_i - x_{i,n} \,\| \,}^2 \biggr)}^{1/2}
\leq \varepsilon\quad \text{for all} \quad n \geq n\0(\varepsilon).
\pagebreak \]
This not only shows that $\oplus _{i \in I} \,(x_i -x_{i,n})$, and hence
$\oplus _{i \in I} \,x_i$, belong to $\oplus _{i \in I} \,H_i$, but also that
$( \,\oplus _{i \in I} \,x_{i,n} \,)_{\,n}$ converges in
$\oplus _{i \in I} \,H_i$ to $\oplus _{i \in I} \,x_i$, so that
$\oplus _{i \in I} \,H_i$ is complete.
\end{proof}

\medskip
If the family $\{ \,H_i \,\}_{\,i \in I}$ is a family of pairwise
orthogonal subspaces of a Hilbert space $H$, we imbed
the direct sum $\oplus _{i \in I} H_i$ in $H$.

\begin{definition}[direct sum of operators]\label{dirop}%
\index{concepts}{direct sum!of operators}%
\index{symbols}{p03@$\oplus_{i \in I} b_i$}%
Consider a  family of bounded linear operators $\{\,b_i\,\}_{\,i \in I}$
on Hilbert spaces $H_i$ $(i \in I)$ respectively. Assume that
$\sup _{i \in I} \,\|\,b_i\,\| < \infty$. One then defines
\[ ( \oplus _{i \in I} \,b_i ) ( \oplus _{i \in I} \,x_i ) := \oplus _{i \in I} \,(b_i x_i)
\qquad ( \,\oplus _{i \in I} \,x_i \in \oplus _{i \in I} \,H_i \,). \]
Then $\oplus _{i \in I} \,b_i$ is a bounded linear operator on
$\oplus _{i \in I} H_i$ with norm
\[ \| \oplus _{i \in I} b_i \,\| = \sup _{i \in I} \,\| \,b_i \,\|.  \]
\end{definition}

\begin{proof}
With $c := \sup _{i \in I} \,\| \,b_i \,\|$, we have
\begin{align*}
{\| \,( \oplus _{i \in I} \,b_i ) ( \oplus _{i \in I} \,x_i ) \,\| \,}^2
\,& = \  \sum _{i \in I} \,{\| \,b_i x_i \,\| \,}^2 \\
  & \leq \ \sum _{i \in I} \,{\| \,b_i \,\| \,}^2 \cdot {\| \,x_i \,\| \,}^2 \\
  & \leq \ {c \,}^2 \cdot \,\sum _{i \in I} \,{\| \,x_i \,\| \,}^2
= {c \,}^2 \cdot {\,\| \oplus _{i \in I} x_i \,\| \,}^2,
\end{align*}
so $\oplus _{i \in I} b_i$ is a bounded linear operator on
$\oplus _{i \in I} H_i$ with $\| \oplus _{i \in I} b_i \,\| \leq c$.
We obviously have $\| \oplus _{i \in I} b_i \,\| \geq c$.
Indeed, clearly $\| \oplus _{i \in I} b_i  \,\| \geq \| \,b_i \,\|$ for all $i \in I$.
This says that $\| \oplus _{i \in I} b_i \,\|$ is an upper bound of the set
$S := \{ \,\| \,b_i \,\| : i \in I \,\}$, so greater than or equal to the least upper
bound of the set $S$, which is $\sup _{i \in I} \| \,b_i \,\| = c$.
\end{proof}

\begin{definition}[direct sum of representations]%
\label{directsumsigma}\index{concepts}{direct sum!of representations}%
\index{concepts}{representation!direct sum}%
\index{symbols}{p04@$\oplus_{i \in I} \pi_i$}%
Consider a family $\{ \,\pi_i \,\} _{\,i \in I}$ of representations
of a \st-algebra $A$ on Hilbert spaces $H_i$ $(i \in I)$ respectively. If
$\{ \,\| \,\pi _i (a) \,\| \,\} _{\,i \in I}$ is bounded for each $a \in A$, we may
define the direct sum representation $\oplus _{i \in I} \,\pi _i$ given by
\begin{alignat*}{2}
\oplus _{i \in I} \,\pi _i : \ & A \ &\to \ & \blop (\oplus _{i \in I} \,H_i) \\
 \ & a \ &\mapsto \ & \oplus _{i \in I} \pi _i (a).
\end{alignat*}
Hence the direct sum of a family of $\sigma$-contractive
representations of a normed \st-algebra is well-defined and
$\sigma$-contractive. \pagebreak
\end{definition}

\clearpage


\section{Cyclic and Non-Degenerate Representations}%
\label{CyclicNonDeg}

\medskip
In this paragraph, let $\pi$ be a representation of a \st-algebra $A$
on a Hilbert space $(H, \langle \cdot , \cdot \rangle)$.

\begin{definition}[invariant subspaces]%
\index{concepts}{invariant subspace}\index{concepts}{subspace!invariant}%
Consider a bounded linear operator $b$ on $H$. A subspace $M$ of $H$
is called \underline{invariant} under $b$, if $bx \in M$ for all $x \in M$.

A subspace $M$ of $H$ is called \underline{invariant} under
$\pi$, if it is invariant under all the operators $\pi(a)$ $(a \in A)$.
\end{definition}

\begin{proposition}%
If $M$ is a subspace of $H$ invariant under $\pi$, so is $M^{\perp}$.
\end{proposition}

\begin{proof}
Let $M$ be a subspace of $H$ invariant under $\pi$. Let $y \in M^{\perp}$,
i.e.\ $\langle y, x \rangle = 0$ for all $x \in M$. For $a \in A$ we then have
\[ \langle \pi (a)y, x \rangle = \langle y,\pi (a^*)x \rangle = 0 \]
for all $x \in M$, so that also $\pi (a)y \in M^{\perp}$.
\end{proof}

\medskip
This is why invariant subspaces are also called \underline{reducing}
subspaces.\index{concepts}{reducing}\index{concepts}{subspace!reducing}

\begin{definition}[subrepresentations]\label{subrep}%
\index{concepts}{subrepresentation}%
\index{concepts}{representation!subrepresentation}%
\index{symbols}{p08@$\pi _M$}%
If $M$ is a closed subspace of $H$ invariant under $\pi$, let $\pi _M$
denote the representation of $A$ on $M$ defined by
\begin{alignat*}{2}
\pi _M : \ & A & \ \to & \  \blop(M) \\
           & a & \ \mapsto & \ \pi (a) |_M.
\end{alignat*}
The representation $\pi _M$ is called a
\underline{subrepresentation} of $\pi$.
\end{definition}

\begin{proposition}\label{comminvar}%
Let $M$ be a closed subspace of $H$, and let $p$ be the projection
on $M$. If $p$ commutes with a bounded linear operator $b$ on $H$,
then $M$ is invariant under $b$.
\end{proposition}

\begin{proof}
Assume that $p$ commutes with a bounded linear operator $b$
on $H$. For $x \in M$, we then have $b x = b p x = p b x \in M$,
so that $M$ is invariant under $b$. \pagebreak
\end{proof}

\begin{definition}[commutants]%
\index{concepts}{commutant}%
\index{symbols}{p1@$\pi '$}\label{repcommutant}%
One denotes by
\[ \pi' := \{ \,c \in \blop(H) : c \,\pi (a) = \pi (a)c \text{ for all } a \in A \,\} \]
the \underline{commutant} of $\pi$. Similarly, let $S \subset \blop(H)$.
The set of all operators in $\blop(H)$ which commute with every operator in
$S$ is called the commutant of $S$ and is denoted by $S'$.
(Cf.\ \ref{questimbed}.)
\end{definition}

\begin{theorem}\label{invarcommutant}%
\index{concepts}{commutant}%
\index{concepts}{subspace!invariant!commutant}%
Let $M$ be a closed subspace of $H$ and let $p$ be the projection on $M$.
Then $M$ is invariant under $\pi$ if and only if $p \in \pi'$.
\end{theorem}

\begin{proof} Assume that $p \in \pi'$. For $x \in M$ and
$a \in A$, we then have $\pi (a)x = \pi (a)px = p \pi (a)x \in M$, so
that $M$ is invariant under $\pi$. Conversely, assume that $M$
is invariant under $\pi$. For $x \in H$ and $a \in A$, we get
$\pi (a)px \in M$, whence $p \pi (a)px = \pi (a)px$, so that
$p \pi (a)p = \pi (a)p$ for all $a \in A$. By applying the involution in
$\blop(H)$, and replacing $a$ by $a^*$, we obtain that also
$p \pi (a)p = p \pi (a)$, whence $p \pi (a) = \pi (a)p$, i.e.\ $p \in \pi'$.
\end{proof}

\begin{definition}[non-degenerate representations]%
\label{nondegenerate}\index{concepts}{non-degenerate}%
\index{concepts}{representation!non-degenerate}%
The representation $\pi$ is called \underline{non-degenerate}
if the closed invariant subspace
\[ N:=\{ \,x \in H : \pi (a)x = 0 \text{ for all } a \in A \,\} \]
is $\{ 0 \}$. Consider the closed invariant subspace
\[ R := \overline{\mathrm{span}} \,\{ \,\pi (a) x \in H: a \in A,\ x \in H \,\}. \]
We have $R^{\perp} = N$, so that $\pi$ is non-degenerate
if and only if $R$ is all of $H$. Also $H = R \oplus N$, or
$\pi = \pi _R \oplus \pi _N$. This says that $\pi$ is the direct sum
of a non-degenerate representation and a null representation. In
particular, by passing to the subrepresentation $\pi _R$, we can
always assume that $\pi$ is non-degenerate.
\end{definition}

\begin{proof}
It suffices to prove that $R^{\perp} = N$. Let $a \in A$ and $x \in H$.
For $y \in N$, we find
\[ \langle \pi (a) x, y \rangle = \langle x, \pi (a^*) y \rangle = 0, \]
so $N \subset R^{\perp}$. On the other hand, for $y \in R^{\perp}$,
we find
\[ \langle \pi (a) y, x \rangle = \langle y, \pi (a^*) x \rangle = 0, \]
which shows that $\pi (a) y \perp H$, or $y \in N$, so
$R^{\perp} \subset N$.
\end{proof}

\medskip
If $\pi$ is non-degenerate, we will also say that $A$
\underline{acts non-degenerately} on $H$ via $\pi$.
\pagebreak

\begin{observation}\label{unitnondeg}%
A representation of a unital \st-algebra on a Hilbert space $\neq \{ \,0 \,\}$ is
non-degenerate if and only if it is unital \ref{unitalhom}.%
\end{observation}

\begin{observation}\label{subnondeg}%
If $\pi$ is non-degenerate, so is every subrepresentation of $\pi$.
\end{observation}

\begin{definition}[cyclic representations]%
\index{concepts}{representation!cyclic}\index{concepts}{cyclic!vector}%
\index{concepts}{cyclic!representation}\index{concepts}{vector!cyclic}%
One says that $\pi$ is \underline{cyclic} if $H$
contains a non-zero vector $x$ such that the set
\[ \{ \,\pi (a)x \in H : a \in A \,\} \]
is dense in $H$. The vector $x$ then is called a
\underline{cyclic vector}. In particular, a cyclic representation
is non-degenerate.
\end{definition}

Of great importance is the following observation.

\begin{lemma}[cyclic subspaces]\label{cyclicsubspace}%
\index{concepts}{cyclic!subspace}\index{concepts}{subspace!cyclic}%
Assume that $\pi$ is non-degenerate, and that $c$ is a
non-zero vector in $H$. Consider the closed invariant subspace
\[ M := \overline{\pi (A)c}. \]
Then $c$ belongs to $M$. In particular, $\pi _M$ is cyclic with cyclic
vector $c$. One says that $M$ is a \underline{cyclic subspace} of $H$.
\end{lemma}

\begin{proof}
Let $p$ denote the projection upon $M$. We then have $p \in \pi'$
by theorem \ref{invarcommutant}. Thus, for all $a \in A$, we get
\[ \pi(a)(\mathds{1}-p)c = (\mathds{1}-p) \pi (a)c = 0, \]
which implies that $(\mathds{1}-p)c = 0$ as $\pi$ is non-degenerate.
Therefore $c = pc \in M$.
\end{proof}

\begin{corollary}[essential representation]%
If $\pi$ is non-degenerate then $\pi$ is \underline{essential},
in the sense that the subset
\[ \{ \,\pi (a) x \in H: a \in A,\ x \in H \,\} \]
is dense in $H$.
\end{corollary}

\begin{observation}\label{dirsumnondeg}%
It is easily seen that a direct sum of non-degenerate representations
is again non-degenerate.
\end{observation}

In the converse direction we have: \pagebreak

\begin{theorem}\index{concepts}{direct sum!decomposition}%
\index{concepts}{representation!direct sum!decomposition}%
\index{concepts}{decomposition!direct sum}\label{directsumdeco}%
If $\pi$ is non-degenerate and if $H \neq \{ 0 \}$, then
$\pi$ is the direct sum of cyclic subrepresentations.
\end{theorem}

\begin{proof}
Let $Z$ denote the set of all non-empty sets consisting of mutually
orthogonal, cyclic \ref{cyclicsubspace} closed invariant subspaces
of $H$. The set $Z$ is not empty since for $0 \neq c \in H$, we may
consider $M_c := \overline{\pi (A)c}$, so that $\{\, M_c \,\} \in Z$ by
\ref{cyclicsubspace}. Furthermore, $Z$ is inductively ordered by
inclusion. Thus, Zorn's Lemma guarantees the existence of a maximal
element $F$ of $Z$. Consider now $H_1 := \oplus _{M \in F} M$.
We claim that $H_1 = H$. Indeed, if $0 \neq c \in {H_1}^{\perp}$, then
$M_c := \overline{\pi (A)c}$ is a cyclic closed invariant subspace of
$H$ by \ref{cyclicsubspace}, and $F \cup \{ \,M_c \,\}$ would be a set
in $Z$ strictly larger than the maximal set $F$. It follows that
$\pi = \oplus _{M \in F} \,\pi _M$. 
\end{proof}

\medskip
Our next aim is lemma \ref{coeffequal}, which needs some preparation.

\begin{proposition}\label{closureunitary}%
Let $I$ be an isometric linear operator defined on a dense subspace
of a Hilbert space $H_1$, with dense range in a Hilbert space $H_2$.
Then the closure $U$ of $I$ is a unitary operator from $H_1$ onto $H_2$.
\end{proposition}

\begin{proof}
The range of $U$ is a complete, hence closed
subspace of $H_2$, containing the dense range of $I$.
\end{proof}

\begin{definition}[spatial equivalence]\label{spatequiv}%
\index{concepts}{spatially equivalent!representations}%
\index{concepts}{equivalent!spatially!representations}%
Let $\pi_1$, $\pi_2$ be two representations of the \st-algebra $A$
on Hilbert spaces $H_1$, $H_2$  respectively. The representations
$\pi_1$ and $\pi_2$ are called \underline{spatially equivalent}
if there exists a unitary operator $U$ from $H_1$ to $H_2$ such that
\[ U \pi_1 (a) = \pi_2 (a) U \quad \text{for all} \quad a \in A. \]
\end{definition}

\begin{definition}[intertwining operators, $C(\pi_1,\pi_2)$]%
\label{intertwinop}\index{concepts}{intertwining operator}%
\index{concepts}{operator!intertwining}%
\index{symbols}{C45@$C(\pi_1,\pi_2)$}%
Let $\pi_1$, $\pi_2$ be two representations of the \st-algebra $A$
on Hilbert spaces $H_1$, $H_2$ \linebreak respectively.
A bounded linear operator $b$ from $H_1$ to $H_2$
is said to \underline{intertwine} $\pi_1$ with $\pi_2$, if
\[ b \pi_1 (a) = \pi_2 (a) b \quad \text{for all} \quad  a \in A. \]
The set of all bounded linear operators intertwining $\pi_1$ with
$\pi_2$ is \linebreak denoted by $C(\pi_1,\pi_2)$. Please note that
$C(\pi, \pi) = \pi'$, cf.\ \ref{repcommutant}.
\end{definition}

The following result is almost miraculous. \pagebreak

\begin{lemma}\label{coeffequal}%
Let $\pi_1$, $\pi_2$ be cyclic representations of the \st-algebra $A$
on Hilbert spaces $H_1$, $H_2$, with cyclic vectors $c_1$, $c_2$.
For $a \in A$, let
\begin{align*}
\varphi_1 (a) & := \langle \pi_1 (a) c_1, c_1 \rangle, \\
\varphi_2 (a) & := \langle \pi_2 (a) c_2, c_2 \rangle .
\end{align*}
If $\varphi_1 = \varphi_2$, then $\pi_1$ and $\pi_2$ are spatially
equivalent. Furthermore, there then exists a unique unitary operator
in $C(\pi_1, \pi_2)$ taking $c_1$ to $c_2$.
\end{lemma}

\begin{proof} 
Consider the subspaces
\begin{align*}
K_1 & := \{ \,\pi_1 (a) c_1 \in H_1 : a \in A \,\} \\
K_2 & := \{ \,\pi_2 (a) c_2 \in H_2 : a \in A \,\},
\end{align*}
which are dense in $H_1$ and $H_2$ respectively (because $c_1$ and $c_2$
are cyclic for $\pi_1$ and $\pi_2$ respectively).

Assume that $U$ is a unitary operator in $C(\pi_1 , \pi_2)$ taking $c_1$ to $c_2$.
On one hand, we then have
\[ U \pi_1 (a) = \pi_2 (a) U \quad \text{for all} \quad a \in A, \]
On the other hand $U c_1 = c_2$. These two facts together imply that
\[ U \text{ takes any vector } \pi_1 (a) c_1 \text{ to } \pi_2 (a) c_2
\quad ( \,a \in A \,). \tag*{$(*)$} \]
Since $K_1$ is dense in $H_1$, there can exist at most one unitary operator
$U$ with the properties as above, by continuity of $U$.

It shall next be shown that the above requirement $(*)$ indeed leads to a unitary
operator $U$ in $C(\pi_1 , \pi_2)$ taking $c_1$ to $c_2$. We set out to prove
that an isometric operator $I : K_1 \to K_2$ is well-defined by putting
\[ Iz := \pi_2 (a) c_2 \in K_2 \text{ when } z = \pi_1 (a) c_1 \in K_1 \quad ( \,a \in A \,), \]
in agreement with $(*)$. For $a \in A$, and $z := \pi_1 (a) c_1 \in K_1$,
we then have by assumption
\begin{align*}
\langle Iz, Iz \rangle & = \langle \pi_2 (a) c_2, \pi_2 (a) c_2 \rangle \\
 & = \langle \pi_2 (a^*a) c_2, c_2 \rangle \\
 & = \varphi _2 (a^*a) \\
 & = \varphi _1 (a^*a) \\
 & = \langle \pi_1 (a^*a) c_1, c_1 \rangle \\
 & = \langle \pi_1 (a) c_1, \pi_1 (a) c_1 \rangle = \langle z, z \rangle.
\end{align*}
It follows that $I$ is well-defined. Indeed, if for $a, b \in A$, one has
\[ \pi_1 (a) c_1 = \pi_1 (b) c_1, \]
then with $z := \pi_1 (a-b) c_1$, we have $z = 0$. We obtain that
$0 = \langle Iz, Iz \rangle$, so $0 = Iz = I \pi_1 (a-b) c_1 = \pi_2 (a-b) c_2$, that is
\[ \pi_2 (a) c_2 = \pi_2 (b) c_2, \]
as was to be shown.

It can clearly be seen from the above that $I$ is isometric.
Furthermore $I : K_1 \to K_2$ is densely defined in $H_1$ with dense range
in $H_2$. It follows that the closure $U$ of $I$ is a unitary operator from
$H_1$ onto $H_2$, cf.\ \ref{closureunitary}.

It shall next be shown that this operator $U$ intertwines $\pi_1$ with $\pi_2$.
Indeed, if $z = \pi_1 (a) c_1 \in K_1$ with $a \in A$, then we have for all
$b \in A$ that
\begin{align*}
\pi_2 (b) Iz & = \pi_2 (b)\pi_2 (a) c_2 \\
 & = \pi_2 (ba) c_2 \\
 & = I \pi_1 (ba) c_1 \\
 & = I \pi_1 (b) \pi_1 (a) c_1 = I \pi_1 (b) z.
\end{align*}
By continuity of $U$ and $\pi (b)$ \ref{HellToepl}, it follows that
\[ \pi_2 (b) U = U \pi_1 (b) \quad \text{for all} \quad b \in A, \]
so that $U$ is a unitary operator in $C(\pi_1 , \pi_2)$.

It remains to be shown that $U$ takes $c_1$ to $c_2$. For $a \in A$,
we have
\[ U \pi_1 (a) c_1 = \pi_2 (a) c_2, \]
whence
\[ \pi_2 (a) U c_1 = \pi_2 (a) c_2. \]
That is, we have
\[ \pi_2 (a) \bigl( U c_1 - c_2 \bigr) = 0 \quad \text{for all} \quad a \in A. \]
Since $\pi_2$ is cyclic, and thereby non-degenerate, it follows that
\[ U c_1 - c_2 = 0, \]
or $U c_1 = c_2$.
\end{proof}

\medskip
As an important example, we shall establish next that
``the commutant of $\blop(H)$ is trivial'', cf.\ \ref{BHirred} below.

\begin{definition}%
For vectors $x$, $y \in H$, we shall denote by $x \odot y$
the bounded linear operator on $H$ (of rank zero or one) given by
\[ (x \odot y) \,z := \langle z, y \rangle \,x \qquad ( \,z \in H \,). \pagebreak \]
\end{definition}

\begin{theorem}\label{tensor}%
Assume that $H \neq \{ 0 \}$, and let $\{ \,x_i \,\}_{i \in I}$
be an \linebreak orthonormal basis of $H$. If some bounded
linear operator $c$ on $H$ \linebreak commutes with all the
bounded linear operators $x_i \odot x_j$ $(i, j \in I)$, then
$c$ is a scalar multiple of the unit operator on $H$.
\end{theorem}

\begin{proof}
If $c$ commutes with each $x_i \odot x_j$ $(i, j \in I)$, we get
\[ \langle c z, x_j \rangle \,x_i = \langle z, x_j \rangle \,c x_i \quad
\text{for all} \quad z \in H, \ i, j \in I. \]
Especially for $z = x_k$ $(k \in I)$, we obtain
\[ \langle c x_k, x_j \rangle \,x_i
= \langle x_k ,x_j \rangle \,c x_i = \delta_{kj} c x_i. \]
In particular we have
\[ \langle c x_k, x_j \rangle = 0 \quad
\text{for all} \quad k, j \in I, \ k \neq j \]
and
\[ \langle c x_j, x_j \rangle \,x_i = c x_i \quad
\text{for all} \quad i, j \in I. \]
It follows that there exists some $\lambda \in \mathds{C}$ such that
\[ \langle c x_j, x_j \rangle = \lambda \quad \text{for all} \quad j \in I. \]
Thus
\[ \langle c x_i, x_j \rangle = \lambda \,\delta_{ij} \quad
\text{for all} \quad i, j \in I. \qedhere \]
\end{proof}

\bigskip
\begin{corollary}\label{BHirred}%
We have
\[ \blop(H)' = \mathds{C} \mathds{1}. \]
\end{corollary}

\clearpage


\chapter{States}

\setcounter{section}{27}


\section{The GNS Construction}

\begin{introduction}\index{concepts}{GNS construction|(}%
Given a representation $\pi$ of a \st-algebra $A$ on a Hilbert
space $H$, one may consider the positive linear functionals
\[ a \mapsto \langle \pi (a) x, x \rangle \qquad ( \,a \in A \,) \]
for each $x \in H$.

In this paragraph as well as in the next one, we
shall be occupied with inverting this relationship.

That is, given a normed \st-algebra $A$ and a positive linear functional
$\varphi$ on $A$ satisfying certain conditions, we shall construct a
representation $\pi _{\varphi}$ of $A$ on a Hilbert space $H _{\varphi}$
in such a way that
\[ \varphi (a) = \langle \pi _{\varphi} (a) x, x \rangle \qquad ( \,a \in A \,) \]
for some $x \in H _{\varphi}$. The present paragraph serves as a rough
foundation for the ensuing one, which puts the results into sweeter words.
\end{introduction}

\begin{proposition}[induced Hilbert form]\label{inducedHf}%
If $\varphi$ is a positive linear functional on a \st-algebra $A$, then
\[ \langle a, b \rangle_\varphi := \varphi (b^*a) \qquad ( \,a,b \in A \,) \]
defines a positive Hilbert form $\langle \cdot , \cdot \rangle _\varphi$
on $A$ in the sense of the following definition.
\end{proposition}

\begin{definition}%
[positive Hilbert forms \hbox{\protect\cite[p.\ 349]{Dieu}}]%
\label{Hilbertform}\index{concepts}{Hilbert form}%
\index{concepts}{positive!Hilbert form}%
A \underline{positive} \underline{Hilbert form} on a \st-algebra $A$ is a positive
semidefinite sesquilinear form $\langle \cdot , \cdot \rangle$ on $A$ such that
\[ \langle ab, c \rangle = \langle b, a^*c \rangle \quad \text{for all} \quad a,b,c \in A. \]
Please note that for all $a, b, c \in A$ we then also have
\begin{align*}
\langle a, b \rangle & = \overline{\langle b, a \rangle}, 
\tag*{\textit{(the sesquilinear form is Hermitian)}} \\
{| \,\langle a, b \rangle \,| \,}^2 & \leq \langle a, a \rangle \cdot \langle b, b \rangle.
\pagebreak \tag*{\textit{(Cauchy-Schwarz inequality)}}%
\end{align*}%
\end{definition}

\begin{definition}\label{quotientA}%
\index{concepts}{isotropic}\index{concepts}{subspace!isotropic}%
\index{symbols}{A9@$\underline{A}$}%
\index{symbols}{a5@$\underline{a}$}%
Let $\langle \cdot , \cdot \rangle$ be a positive Hilbert form on a
\st-algebra $A$. Denote by $I$ the isotropic subspace of the inner
product space \linebreak $(A, \langle \cdot, \cdot \rangle)$, that is
$I = \{ \,b \in A : \langle b, c \rangle = 0 \text{ for all } c \in A \,\}$.
The isotropic subspace $I$ is a left ideal in $A$ as
$\langle ab, c \rangle = \langle b, a^* c \rangle$ holds for all $a, b, c \in A$.
Denote by $\underline{A}$ the quotient space $A\mspace{1mu}/I$.
For $a \in A$, let
\[ \underline{a} := a+I \in \underline{A}. \]
In $\underline{A}$ one defines an inner product by
\[ \langle \underline{a} , \underline{b} \rangle :=
\langle a, b \rangle \qquad ( \,a, b \in A \,). \]
In this way $\underline{A}$ becomes a pre-Hilbert space (as
$I = \{ \,b \in A : \langle b, b \rangle = 0 \,\}$ by the Cauchy-Schwarz
inequality).
\end{definition}

\begin{definition}[the GNS representation]\label{GNS}%
\index{concepts}{representation!GNS}%
Let $\langle \cdot, \cdot \rangle$ be a positive Hilbert form on a
\st-algebra $A$. For $a \in A$, the translation operator
\begin{align*}
\pi (a) : \underline{A} & \to \underline{A} \\
\underline{b} & \mapsto \pi (a) \underline{b} := \underline{ab}
\end{align*}
is well-defined because $I$ is a left ideal. One thus obtains a
representation $\pi$ of $A$ on the pre-Hilbert space $\underline{A}$.
It is called the \underline{representation} \underline{associated with
$\langle \cdot , \cdot \rangle$}. If all of the operators $\pi (a)$ $(a \in A)$
are bounded, then by continuation of these operators, we get a
representation of $A$ on the completion of $\underline{A}$.
This latter representation is called the
\underline{GNS representation associated with $\langle \cdot , \cdot \rangle$}.
\end{definition}

\begin{definition}[$\pi _{\varphi}$, $H_{\varphi}$]%
\label{GNSvarphi}\index{symbols}{H1@$H_{\varphi}$}%
\index{concepts}{GNS construction!Hphi@$H _{\varphi}$}%
\index{concepts}{GNS construction!piphi@$\pi_ {\varphi}$}%
\index{symbols}{p2@$\pi_ {\varphi}$}\index{concepts}{representation!GNS}%
Let $\varphi$ be a positive linear functional on a \st-algebra $A$.
Consider the Hilbert form $\langle \cdot , \cdot \rangle _{\varphi}$ induced by
$\varphi$, as in \ref{inducedHf}. One says that the representation associated
with $\langle \cdot , \cdot \rangle _{\varphi}$ is the representation associated
with $\varphi$. One denotes by $H_{\varphi}$ the completion of $\underline{A}$.
If the operators $\pi (a)$ $(a \in A)$ are bounded, one
denotes by $\pi _{\varphi}$ the GNS representation associated
with $\langle \cdot , \cdot \rangle _{\varphi}$. It is then called the
\underline{GNS representation associated with $\varphi$}.
\end{definition}

The acronym GNS stands for Gel'fand, Na\u{\i}mark and Segal, the
originators of the theory.

\begin{definition}[weak continuity on $A\sa$]%
\label{weakbdedsadef}%
\index{concepts}{functional!w@weakly continuous on $A\protect\sa$}%
\index{concepts}{weakly continuous!As@on $A\protect\sa$}%
A positive linear functional $\varphi$ on a normed \st-algebra
$A$ shall be called \underline{weakly continuous}
\underline{on $A\sa$} if for each $b \in A$, the positive linear
functional $a \mapsto \varphi (b^*ab)$ is continuous on $A\sa$.
\pagebreak
\end{definition}

\begin{theorem}\label{weakbdedsarep}%
Let $\varphi$ be a positive linear functional on a normed \linebreak
\st-algebra $A$. Then $\varphi$ is weakly continuous on $A\sa$ if
and only if the \linebreak representation associated with $\varphi$
is weakly continuous on $A\sa$. In this event, the GNS representation
$\pi _{\varphi}$ is well-defined and is a $\sigma$-contractive
representation of $A$ on $H_{\varphi}$. Cf.\ \ref{weaksa}, \ref{sigmacontr}.
\end{theorem}

\begin{proof}
For $a,b \in A$ we have $\varphi (b^*ab) =
\langle \pi (a) \underline{b} , \underline{b} \rangle _{\varphi}$. 
\end{proof}

\begin{corollary}\label{Banachweakbdedsa}%
Let $A$ be a Banach \st-algebra. Then every positive linear
functional on $A$ is weakly continuous on $A\sa$.
\end{corollary}

\begin{proof}
This follows now from theorem \ref{Banachbounded}.
\end{proof}

\begin{definition}[variation \hbox{\cite[p.\ 307]{Gaal}}]%
\index{concepts}{variation}%
\index{symbols}{v(phi)@$v(\varphi)$}\label{variation}%
\index{concepts}{GNS construction!v(phi)@$v(\varphi)$}%
A positive linear functional $\varphi$ on a \st-algebra $A$
is said to have \underline{finite variation} if
\[ {| \,\varphi (a) \,| \,}^2 \leq \gamma \,\varphi (a^*a) \quad \text{for all} \quad a \in A \]
holds for some $\gamma \geq 0$. If this is the case, we shall say that
\[ v ( \varphi ) := \,\inf \,\{ \,\gamma \geq 0 :
{| \,\varphi (a) \,| \,}^2 \leq \gamma \,\varphi (a^*a) \text{ for all } a \in A \,\} \]
is the \underline{variation} of $\varphi$. We then also have
\[ {| \,\varphi (a) \,| \,}^2 \leq v ( \varphi ) \,\varphi (a^*a) \quad \text{for all} \quad a \in A. \]
\end{definition}

\begin{proof}
Assume that $\varphi$ has finite variation. With 
\[ S := \{ \,\gamma \geq 0 : {| \,\varphi (a) \,|}^{\,2} \leq \gamma \,\varphi (a^*a)
\ \text{for all} \ a \in A \,\}\]
we have that
\[ v (\varphi) = \inf S = \text{the greatest lower bound of} \ S. \]
Let now $b \in A$ be fixed. It must be shown that
\[ {| \,\varphi (b) \,|}^{\,2} \leq v (\varphi) \,\varphi (b^*b). \]
We have
\[ {| \,\varphi (b) \,|}^{\,2} \leq \gamma \,\varphi (b^*b)
\quad \text{for all} \quad \gamma \in S. \]
That is,
\[ {| \,\varphi (b) \,|}^{\,2}\ \text{is a lower bound of} \ S \,\varphi (b^*b). \]
Since
\[ v (\varphi) \,\varphi (b^*b) \ \text{is the greatest lower bound of} \ S \,\varphi (b^*b), \]
it follows that
\[ {| \,\varphi (b) \,|}^{\,2} \leq v (\varphi) \,\varphi (b^*b), \]
as was to be shown.
\end{proof}

\begin{proposition}\label{variationinequal}%
Let $\pi$ be a representation of a \st-algebra $A$ on a \linebreak
pre-Hilbert space $H$. For $x \in H$, consider the positive linear
functional $\varphi$ on $A$ defined by
\[ \varphi(a):=\langle \pi(a)x, x \rangle \qquad ( \,a \in A \,). \]
Then $\varphi$ is Hermitian and has finite variation
$v(\varphi) \leq \langle x, x \rangle$.
\end{proposition}

\begin{proof}
For $a \in A$, we have
\[ {| \,\varphi (a) \,| \,}^2 = {|\,\langle \pi (a) x, x \rangle \,| \,}^2
\leq {\|\,\pi(a)x\,\| \,}^2 \,{ \|\,x\,\| \,}^2 = \varphi(a^*a) \,\langle x, x \rangle, \]
so that $v ( \varphi ) \leq \langle x, x \rangle$. Finally $\varphi$
is Hermitian \ref{Hermfunct}, because for $a \in A$, we have
\[ \varphi(a^*)=\langle\pi(a^*)x,x \rangle=\langle x,\pi(a)x \rangle
=\overline{\langle \pi (a) x,x \rangle} = \overline{\varphi(a)}. \qedhere \]
\end{proof}

\begin{proposition}\label{variationequal}%
Let $\pi$ be a \underline{non-degenerate} representation of a
\linebreak \st-algebra $A$ on a \underline{Hilbert} space $H$.
For $x \in H$, consider the positive \linebreak linear functional
$\varphi$ on $A$ defined by
\[ \varphi (a):= \langle \pi (a) x, x \rangle \qquad ( \,a \in A \,). \]
The variation of $\varphi$ then is
\[ v(\varphi) = \langle x, x \rangle. \]
\end{proposition}

\begin{proof}
The preceding proposition \ref{variationinequal} shows that
$v ( \varphi ) \leq \langle x, x \rangle $. In order to prove the
opposite inequality, consider the closed subspace
$M := \overline{\pi (A) x}$ of $H$. Then $x \in M$ because $\pi$ is
non-degenerate, cf.\ \ref{cyclicsubspace}. There thus exists a
sequence $(a_n)$ in $A$ such that
\[ x = \lim _{n \to \infty} \pi (a_n) x. \]
We then have
\begin{align*}
 & \,\left| \ {| \,\varphi (a_n) \,| \,}^2
 - {\| \,x \,\| \,}^2 \cdot \varphi ({a_n}^*{a_n}) \ \right| \\
 = & \,\left| \ {| \,\langle \pi (a_n) x, x \rangle \,| \,}^2
- {\| \,x \,\| \,}^2 \cdot {\| \,\pi (a_n) x \,\| \,}^2 \ \right| \\
 \leq & \,\left| \ {| \,\langle \pi (a_n) x, x \rangle \,| \,}^2
 - {\| x \| \,}^4 \ \right|
 + {\| \,x \,\| \,}^2 \cdot \left| \ {\| \,\pi (a_n) x \,\| \,}^2
 - {\| \,x \,\| \,}^2 \ \right|
\end{align*}
so that
\[ \lim _{n \to \infty} \left| \ {| \,\varphi (a_n) \,| \,}^2 -
\,{\| \,x \,\| \,}^2 \cdot \varphi ({a_n}^*{a_n}) \ \right| = 0, \]
by
\[ \lim _{n \to \infty} \langle \pi (a_n) x, x \rangle = {\| \,x \,\| \,}^2 \]
and
\[ \lim _{n \to \infty} \| \,\pi (a_n) x \,\| = \| \,x \,\|. \pagebreak \qedhere \]
\end{proof}

\begin{definition}[reproducing vector, $c_{\varphi}$]%
\label{reprovectordef}%
\index{symbols}{c@$c _{\varphi}$}\index{concepts}{reproducing vector}%
\index{concepts}{GNS construction!cphi@$c _{\varphi}$}%
Let $\varphi$ be a positive \linebreak linear functional on a \st-algebra $A$.
A \underline{reproducing vector for $\varphi$} is a vector $c_{\varphi}$
in $H_{\varphi}$ such that
\[ \varphi (a) = \langle \underline{a}, c_{\varphi} \rangle _{\varphi}
\quad \text{for all} \quad a \in A. \]
It is then uniquely determined by $\varphi$. We shall reserve the
notation $c_{\varphi}$ for the reproducing vector of $\varphi$.
\end{definition}

\begin{proposition}%
Let $\varphi$ be a positive linear functional on a unital \st-algebra $A$.
Then $\underline{e}$ is the reproducing vector for $\varphi$.
\end{proposition}

\begin{proof}
For $a \in A$ we have
\[ \varphi (a) = \varphi (e^*a) =
\langle \underline{a}, \underline{e} \rangle _{\varphi}. \qedhere \]
\end{proof}

\begin{proposition}\label{reprovectorfinvar}%
Let $\varphi$ be a positive linear functional on a \linebreak
\st-algebra $A$. Then $\varphi$ has a reproducing vector if
and only if $\varphi$ has finite variation. In the affirmative
case, we have
\[ v ( \varphi ) = \langle c_{\varphi}, c_{\varphi} \rangle _{\varphi}. \]
\end{proposition}

\begin{proof}
If $\varphi$ has a reproducing vector $c_{\varphi}$,
then the Cauchy-Schwarz inequality implies that
\[ {| \,\varphi(a) \,| \,}^2
= {| \,\langle \underline{a}, c_{\varphi} \rangle _{\varphi} \,| \,}^2
\leq {\| \,\underline{a} \,\| \,}^2 \cdot {\| \,c_{\varphi} \,\| \,}^2
= \varphi (a^*a) \cdot {\| \,c_{\varphi}\,\| \,}^2, \]
which shows that
\[ v ( \varphi ) \leq {\| \,c_{\varphi} \,\| \,}^2. \tag*{$(*)$} \]
Conversely assume that $\varphi$ has finite variation. For
$a \in A$ we obtain
\[ | \,\varphi (a) \,| \leq {v ( \varphi ) \,}^{1/2} \cdot {\varphi (a^*a) \,}^{1/2} =
{v ( \varphi ) \,}^{1/2} \cdot \| \,\underline{a} \,\|. \]
Thus $\varphi$ vanishes on the isotropic subspace of
$( A , \langle \cdot , \cdot \rangle _{\varphi} )$, and so $\varphi$
may be considered as a bounded linear functional of norm
$\leq {v(\varphi) \,}^{1/2}$ on the pre-Hilbert space $\underline{A}$.
Hence there exists a vector $c_{\varphi}$ in $H_{\varphi}$ with
\[ \| \,c_{\varphi} \,\| \leq {v ( \varphi ) \,}^{1/2} \tag*{$(**)$} \]
such that
\[ \varphi (a) = \langle \underline{a}, c_{\varphi} \rangle _{\varphi} \]
for all $a \in A$. We have
\[ v ( \varphi ) = \langle c_{\varphi}, c_{\varphi} \rangle _{\varphi} \]
by the inequalities $(*)$ and $(**)$. \pagebreak
\end{proof}

\begin{theorem}\label{finvarcyclic}%
Let $\varphi$ be a positive linear functional of finite vari\-ation on a
normed \st-algebra $A$, which is weakly continuous on $A\sa$.
Let $c_{\varphi}$ be the reproducing vector of $\varphi$. We then have
\[ \pi_{\varphi}(a)c_{\varphi}=\underline{a} \quad \text{for all} \quad a \in A. \]
Thus, if $\varphi \neq 0$, then $\pi _{\varphi}$ is cyclic with cyclic vector
$c_{\varphi}$. Moreover
\[ \varphi(a)=\langle\pi_{\varphi}(a)c_{\varphi},c_{\varphi}\rangle_{\varphi}
\quad \text{for all} \quad a \in A.  \]
\end{theorem}

\begin{proof}
The GNS representation $\pi _{\varphi}$ exists by \ref{weakbdedsarep}.
For $a, b \in A$, we compute
\[ \langle\underline{b},\pi_{\varphi}(a)c_{\varphi}\rangle_{\varphi}
 = \langle\pi_{\varphi}(a^*)\underline{b},c_{\varphi}\rangle_{\varphi}
 = \langle\underline{a^*b},c_{\varphi}\rangle_{\varphi}
 = \varphi(a^*b) = \langle\underline{b},\underline{a}\rangle_{\varphi}, \]
which implies that $\pi_{\varphi}(a)c_{\varphi}=\underline{a}$ for all
$a \in A$. One then inserts this into the defining relation of a
reproducing vector.
\end{proof}

\begin{corollary}\label{Banachfinvarbded}%
Every positive linear functional of finite vari\-ation
on a Banach \st-algebra is continuous.
\end{corollary}

\begin{proof}
Let $\varphi$ be a positive linear functional of finite variation on a Banach
\st-algebra $A$. Then $\varphi$ is weakly continuous on $A\sa$ by
\ref{Banachweakbdedsa}. Hence the GNS representation $\pi _{\varphi}$
exists by \ref{weakbdedsarep}, and from \ref{finvarcyclic} we get
$\varphi(a) = \langle \pi _{\varphi} (a) c_{\varphi} , c_{\varphi} \rangle _{\varphi}$
for all $a \in A$. The continuity of $\pi _{\varphi}$, which is guaranteed by
\ref{Banachbounded}, now implies that $\varphi$ is continuous. 
\end{proof}

\begin{theorem}\label{finvarconvex}%
Let $A$ be a normed \st-algebra. The set of positive \linebreak
linear functionals of finite variation on $A$ which are weakly
continuous \linebreak on $A\sa$, is a convex cone, and the variation
$v$ is additive and \linebreak $\mathds{R}_+$-homogeneous
on this convex cone.
\end{theorem}

\begin{proof}
It is elementary to prove that $v$ is $\mathds{R}_+$-homogeneous.
Let $\varphi _1$ and $\varphi _2$ be two positive linear
functionals of finite variation on $A$, which are weakly
continuous on $A\sa$. The sum $\varphi := \varphi_1 + \varphi_2$
then also is weakly continuous on $A\sa$. Let $\pi _1$ and $\pi _2$
be the GNS representations associated with $\varphi _1$ and
$\varphi _2$. Consider the direct sum $\pi := \pi _1 \oplus \pi _2$.
It is a non-degenerate representation. Let
$\langle \cdot , \cdot \rangle$ denote the inner product in the
direct sum $H_{\text{\footnotesize{$\varphi_1$}}}
\oplus H_{\text{\footnotesize{$\varphi_2$}}}$.
Let $c _1$ and $c _2$ be the reproducing vectors of
$\varphi _1$ and $\varphi _2$ respectively. For
$c := c _1 \oplus c _2$, one finds with \ref{finvarcyclic} that
\[ \langle \pi (a) c, c \rangle
= \langle \pi _1 (a) c _1, c _1 \rangle _{\text{\footnotesize{$\varphi_1$}}}
+ \langle \pi _2 (a) c _2, c_2 \rangle _{\text{\footnotesize{$\varphi_2$}}}
= \varphi _1 (a) + \varphi _2 (a) = \varphi (a) \]
for all $a \in A$. Therefore, according to \ref{variationequal}, the
variation of $\varphi$ is
\[ v ( \varphi ) = \langle c, c \rangle
= \langle c _1, c _1 \rangle _{\text{\footnotesize{$\varphi_1$}}}
+ \langle c _2, c _2 \rangle _{\text{\footnotesize{$\varphi_2$}}}
= v ( \varphi _1 ) + v ( \varphi _2 ). \pagebreak \qedhere \]
\end{proof}

\clearpage


\section{States on Normed \texorpdfstring{$*$-}{\80\052\80\055}Algebras}

\medskip
In this paragraph, let $A$ be a normed \st-algebra.

\begin{definition}[states, $\statespace (A)$]\index{concepts}{state}%
\index{concepts}{functional!state}\index{symbols}{S5@$\statespace (A)$}%
A positive linear functional $\psi$ on $A$ shall be called a \underline{state}
if it is weakly continuous on $A\sa$, and has finite variation $v(\psi) = 1$. The
set of states on $A$ is denoted by $\underline{\statespace (A)}$.
\end{definition}

\begin{definition}[quasi-states, $\quasistates (A)$]%
\index{symbols}{QS(A)@$\quasistates (A)$}%
\index{concepts}{quasi-state}\index{concepts}{functional!quasi-state}%
A \underline{quasi-state} on $A$ shall be a positive linear functional $\varphi$
on $A$ of finite variation $v(\varphi) \leq 1$, which is weakly continuous on
$A\sa$. The set of quasi-states on $A$ is denoted by $\underline{\quasistates (A)}$.
\end{definition}

\begin{theorem}\label{SAconvex}%
The sets $\statespace (A)$ and $\quasistates (A)$ are convex.
\end{theorem}

\begin{proof} \ref{finvarconvex}. \end{proof}

\begin{theorem}\label{staterep}\index{concepts}{state}%
Let $\pi$ be a non-degenerate $\sigma$-contractive representation
of $A$ on a Hilbert space $H$. If $x$ is a unit vector in $H$, then
\[ \psi (a) := \langle \pi (a) x, x \rangle \qquad ( \,a \in A \,) \]
defines a state on $A$.
\end{theorem}

\begin{proof} \ref{sigmaAsa}, \ref{variationequal}. \end{proof}

\begin{theorem}\label{repstate}\index{concepts}{representation!GNS}%
Let $\psi$ be a state on $A$. The GNS representation $\pi _{\psi}$ then is a
$\sigma$-contractive cyclic representation on the Hilbert space $H_{\psi}$.
The reproducing vector $c _{\psi}$ is a cyclic unit vector for $\pi _{\psi}$. We have
\[ \psi (a) = \langle \pi _{\psi} (a) c _{\psi}, c _{\psi} \rangle _{\psi}
\quad \text{for all} \quad a \in A. \]
\end{theorem}

\begin{proof}
\ref{weakbdedsarep}, \ref{reprovectorfinvar} and \ref{finvarcyclic}.
\end{proof}

\begin{corollary}\label{propquasi}%
Every quasi-state $\varphi$ on $A$ is $\sigma$-contractive and Hermitian.
In the accessory \st-norm we have $\| \,\varphi \,\| \leq 1$. In particular,
$\varphi$ is contractive on $A\sa$.
\end{corollary}

\begin{proof}
This follows now from \ref{variationinequal} and \ref{sigmaAsa}. \pagebreak
\end{proof}

\begin{definition}\label{topQS}%
We imbed the set $\quasistates (A)$ of quasi-states on $A$ in the unit ball of
the dual normed space of the real normed space $A\sa$, cf.\ \ref{propquasi}.
The set $\quasistates (A)$ then is a compact Hausdorff space in the \linebreak
weak* topology. (By Alaoglu's Theorem, cf.\ the appendix \ref{Alaoglu}.)
\end{definition}

\begin{theorem}%
Let $\pi$ be a cyclic $\sigma$-contractive representation of $A$ on a Hilbert
space $H$. There then exists a state $\psi$ on $A$ such that $\pi$ is spatially
equivalent to $\pi _{\psi}$. Indeed, each cyclic unit vector $c$ of $\pi$ gives
rise to such a state $\psi$ by putting
\[ \psi (a) := \langle \pi (a) c, c \rangle \qquad ( \,a \in A \,). \]
There then exists a unique unitary operator intertwining $\pi$ with $\pi _{\psi}$
taking $c$ to the reproducing vector $c_{\psi}$.
\end{theorem}

\begin{proof}
This follows from \ref{coeffequal} because
\[ \langle \pi (a) c, c \rangle = \psi (a) =
\langle \pi _{\psi} (a) c_{\psi}, c_{\psi} \rangle _{\psi}
\quad \text{for all} \quad a \in A. \qedhere \]
\end{proof}

\begin{theorem}\label{stateratio}%
The set of states on $A$ parametrises the class of cyclic $\sigma$-contractive
representations of $A$ on Hilbert spaces up to spatial equivalence. More
precisely, if $\psi$ is a state on $A$, then $\pi _{\psi}$ is a \linebreak
$\sigma$-contractive cyclic representation of $A$. Conversely, if $\pi$ is any
\linebreak $\sigma$-contractive cyclic representation of $A$ on a Hilbert
space, there \linebreak exists a state $\psi$ on $A$ such that $\pi$ is spatially
equivalent to $\pi _{\psi}$.
\end{theorem}

Also of importance is the following uniqueness theorem. (Indeed there is
another way to define the GNS representation associated with a state.)

\begin{theorem}%
The GNS representation $\pi_{\psi}$, associated with a state $\psi$
on $A$, is determined up to spatial equivalence by the condition
\[ \psi (a) = \langle \pi _{\psi} (a) c _{\psi} , c _{\psi} \rangle _{\psi}
\quad \text{for all} \quad a \in A, \]
in the following sense. If $\pi$ is another cyclic representation of $A$ on a
Hilbert space, with cyclic vector $c$, such that
\[ \psi (a) = \langle \pi (a) c , c \rangle \quad \text{for all} \quad a \in A, \]
then $\pi$ is spatially equivalent to $\pi _{\psi}$, and there exists a unique
unitary operator intertwining $\pi$ with $\pi _{\psi}$, taking $c$ to $c _{\psi}$.
\index{concepts}{GNS construction|)}
\end{theorem}

\begin{proof} \ref{coeffequal}. \pagebreak \end{proof}

\begin{proposition}\label{plusstate}%
If $\psi$ is a state on $A$, then
\[ \psi (a) \geq 0 \quad \text{for all} \quad a \in A_+. \]
\end{proposition}

\begin{proof}
If $a \in A_+$ then $\pi _{\psi} (a) \in \blop(H) _+$, cf.\ \ref{hompos}, whence, by \ref{posop},
\[ \psi (a) = \langle \pi _{\psi} (a) c _{\psi}, c _{\psi} \rangle _{\psi} \geq 0. \qedhere \]
\end{proof}

\medskip
The remaining items \ref{bdedstatesdef} - \ref{kerpsi} give a sort of
justification for our terminology ``state'' for important classes of normed \st-algebras.

\begin{definition}[normed \protect\st-algebras with continuous states]%
\index{concepts}{algebra!normed s-algebra@normed \protect\st-algebra!with continuous states}%
\index{concepts}{states!with continuous states}\label{bdedstatesdef}%
\index{concepts}{continuous states!normed \protect\st-algebra with}%
We shall say that $A$ \underline{has continuous states}, or that $A$ is a normed
\st-algebra \underline{with continuous states}, if every state on $A$ is continuous.
\end{definition}

In practice, normed \st-algebras have continuous states:

\begin{proposition}\label{bdedstatessuffcond}%
$A$ has continuous states whenever any of the three following conditions is satisfied:
\begin{itemize}
   \item[$(i)$] the involution in $A$ is continuous,
  \item[$(ii)$] $A$ is complete, that is, $A$ is a Banach \st-algebra,
 \item[$(iii)$] $A$ is a commutative \st-subalgebra of a Banach \st-algebra.
 \item[$(iv)$] $A$ is a \st-subalgebra of a Hermitian Banach \st-algebra.
\end{itemize}
\end{proposition}

\begin{proof}
\ref{continvol}, \ref{Banachfinvarbded}, \ref{commsubBsigmacontr}, \ref{Hermcont}.
\end{proof}

\begin{theorem}\label{bdedstatesthm}%
If $A$ has continuous states, then a $\sigma$-contractive
representation $\pi$ of $A$ on a Hilbert space $H$ is continuous.
\end{theorem}

\begin{proof} Let $x$ be a vector in the unit ball of $H$. The positive linear
functional $a \mapsto \langle \pi (a) x, x \rangle$ $(a \in A)$
then is a quasi-state (by \ref{sigmaAsa} as well as \ref{variationinequal}),
and hence is continuous. This says that $\pi$ is weakly continuous on the
Hilbert space $H$, and thus continuous, cf.\ \ref{weakcontHilbert}.
\end{proof}

\begin{corollary}\label{bdedstatesGNSbded}%
If $A$ has continuous states, then the GNS \linebreak
representations associated with states on $A$ are continuous.
\end{corollary}

\begin{remark}\label{kerpsi}
It is well-known that a linear functional $\psi$ on a normed space is
continuous if and only if $\ker \psi$ is closed, and thence if and only
if $\psi$ has closed graph. (Hint: factor the functional over its kernel,
and use the quotient norm, cf.\ the appendix \ref{quotspace}.)
\pagebreak
\end{remark}

\clearpage


\section{Unit or Approximate Unit and Continuous Involution}

\begin{proposition}\label{vare}%
A positive linear functional $\varphi$ on a unital \linebreak \st-algebra
$A$ is Hermitian and has finite variation $v(\varphi) = \varphi(e)$.
\end{proposition}

\begin{proof}
For $a \in A$, we have
\begin{align*}
\varphi (a^*) = \varphi (a^*e) = \langle e, a \rangle _{\varphi}
= \overline{\langle a, e \rangle} _{\varphi} = \overline{\varphi (e^*a)}
= \overline{\varphi (a)},
\end{align*}
as well as
\begin{align*}
{| \,\varphi (a) \,| \,}^2 & = {| \,\varphi (e^*a) \,| \,}^2
= {| \,\langle a, e \rangle _{\varphi} \,| \,}^2 \\
 & \leq \langle a, a \rangle _{\varphi} \,\langle e, e \rangle _{\varphi}
 = \varphi (e^*e) \,\varphi (a^*a) = \varphi (e) \,\varphi (a^*a).\ \qedhere
\end{align*}
\end{proof}

\begin{corollary}%
Let $A$ be a unital normed \st-algebra with continuous states. Then a
state on $A$ can be defined as a continuous positive linear functional
$\psi$ on $A$ with $\psi (e) = 1$.
\end{corollary}

\begin{theorem}\label{eBbded}%
A positive linear functional on a unital Banach \linebreak \st-algebra is
continuous. A state on a unital Banach \st-algebra can be defined as an
automatically continuous positive linear functional $\psi$ with $\psi (e) = 1$.
\end{theorem}

\begin{proof}
This follows from \ref{vare} and \ref{Banachfinvarbded}.
\end{proof}

\begin{proposition}[extensibility]\label{extendable}
\index{concepts}{extensible}\index{concepts}{functional!extensible}%
A positive linear functional $\varphi$ on a \st-algebra
$A$ without unit can be extended to a positive linear functional
$\tld{\varphi}$ on $\tld{A}$ if and only if $\varphi$ is Hermitian
and has finite variation. In this case, $\varphi$ is called
\underline{extensible}. One then can choose for
$\tld{\varphi}(e)$ any real value $\geq v(\varphi)$, and
the choice $\tld{\varphi}(e) := v(\varphi)$ is minimal.
\end{proposition}

\begin{proof}
The ``only if'' part follows from \ref{vare}. Conversely,
let $\varphi$ be Hermitian and have finite variation. Extend
$\varphi$ to a linear functional $\tld{\varphi}$ on $\tld{A}$ with
$\tld{\varphi}(e) := \gamma$ for some $\gamma \geq v (\varphi)$.
For $a \in A$ and $\lambda \in \mathds{C}$, we get
(using the Hermitian nature of $\varphi$):
\begin{align*}
\tld{\varphi}\,\bigl( \,(\lambda e+a)^* (\lambda e+a) \,\bigr)
= & \ \tld{\varphi} \,\bigl( \,{| \,\lambda \,| \,}^2 \,e
+ \overline{\lambda} a + \lambda a^* + a^*a \,\bigr) \\
= & \ \gamma \,{| \,\lambda \,| \,}^2 +
 2 \,\mathrm{Re} \,\bigl( \,\overline{\lambda} \varphi (a) \bigr) + \varphi (a^*a).
\end{align*}
Let now $\lambda =: r \exp (\iu \alpha)$, $a =: \exp (\iu \alpha) \,b$
with $r \in [ \,0, \infty \,[$, $\alpha \in \mathds{R}$, and $b \in A$.
The above expression then becomes
\[ \gamma \,{r \,}^2 + 2 \,r \,\mathrm{Re} \,\bigl( \varphi (b) \bigr) + \varphi (b^*b).
\pagebreak \]
This expression is non-negative because by construction we have
\[ {\bigl| \,\mathrm{Re} \,\bigl( \varphi (b) \bigr) \,\bigr| \,}^2 \leq {| \,\varphi (b) \,| \,}^2
\leq v (\varphi) \,\varphi (b^*b) \leq \gamma \,\varphi (b^*b). \]
The choice $\tld{\varphi}(e) := v (\varphi)$ is minimal as
a positive linear functional $\psi$ on $\tld{A}$ extending
$\varphi$ satisfies $v(\varphi) \leq v(\psi) = \psi(e)$, cf.\ \ref{vare}.
\end{proof}

\begin{proposition}\label{extquasi}%
Let $\varphi$ be a quasi-state on a normed  \st-algebra $A$
without unit. If $A$ is a pre-C*-algebra and if the completion of $A$
contains a unit, assume that furthermore $\varphi$ is a state.
By extending $\varphi$ to a linear functional $\tld{\varphi}$ on $\tld{A}$
with $\tld{\varphi}(e) := 1$, we obtain a state $\tld{\varphi}$ on $\tld{A}$
which extends $\varphi$. Every state on $\tld{A}$ is of this form.
\end{proposition}

\begin{proof} Recall that a quasi-state is Hermitian by \ref{propquasi},
and thus extensible. We shall show that $\tld{\varphi}$ is continuous on
$\tld{A}\sa$. It is clear that $\tld{\varphi}$ is continuous on $\tld{A}\sa$ with
respect to the norm $| \,\lambda \,| + \| \,a \,\|$, $(\lambda \in \mathds{R}$,
$a \in A\sa)$. Thus, if $A$ is not a pre-C*-algebra, then $\tld{\varphi}$ is
continuous on $\tld{A}\sa$ by definition \ref{normunitis}.
Assume now that $A$ is a pre-C*-algebra. Let $B$ be the completion of $A$.
If $B$ contains no unit, we apply \ref{Cstarequiv} to the effect that $\tld{\varphi}$
is continuous on $\tld{A}\sa$ with respect to the norm as in \ref{preCstarunitis}.
Also the last statement is fairly obvious in the above two cases.
If $B$ does have a unit, then $A$ is dense in $\tld{A}$ by \ref{complunit}.
We see that the states on $A$ and $\tld{A}$ are related by continuation.
(This uses the continuity of the involution.) This way of extending states
is compatible with the way described in the proposition, by \ref{vare}.
The statement follows immediately.
\end{proof}

\begin{corollary}\label{statecanext}%
A state on a normed \st-algebra can be extended to a state on the unitisation.
\end{corollary}

\begin{definition}[approximate units]%
\index{concepts}{approximate unit}\index{concepts}{unit!approximate}%
Let $A$ be a normed algebra. A \underline{left approximate unit} in $A$
is a net $(e_i) _{i \in I}$ in $A$ such that
\[ a = \lim _{i \in I} e_i a \quad \text{for all} \quad a \in A. \]
Analogously one defines a right approximate unit, a one-sided
approximate unit, as well as a two-sided approximate unit.
A one-sided or two-sided approximate unit $(e_i) _{i \in I}$ in $A$ is
called \underline{bounded} if the net $(e_i) _{i \in I}$ is bounded in $A$.
\end{definition}

Recall that a C*-algebra has a bounded two-sided approximate unit
contained in the open unit ball, cf.\ \ref{C*canapprunit}. \pagebreak

\begin{proposition}\label{bdedfinvar}%
Let $\varphi$ be a continuous positive linear functional on a normed
\st-algebra $A$ with continuous involution and bounded one-sided
approximate unit $(e_i) _{i \in I}$. Then $\varphi$ has finite variation.
\end{proposition}

\begin{proof}
Let $c$ denote a bound of the approximate unit and let $d$ denote
the norm of the involution. The Cauchy-Schwarz inequality implies
\begin{align*}
{| \,\varphi (e_i a) \,| \,}^2
 & = {| \,\langle a, {e_i}^* \rangle _{\varphi} \,| \,}^2
\leq \langle a, a \rangle _{\varphi}
\,\langle {e_i}^*, {e_i}^*  \rangle _{\varphi} \\
 & \leq \varphi (e_i {e_i}^*) \,\varphi (a^*a)
\leq | \,\varphi \,| \,d\,{c\,}^2 \,\varphi (a^*a)
\end{align*}
for all $a \in A$. Hence in the case of a bounded left approximate unit
\[ {| \,\varphi (a) \,| \,}^2 \leq d\,{c\,}^2 \,| \,\varphi \,| \,\varphi (a^*a) \]
for all $a \in A$. If $(e_i)_{i \in I}$ is a bounded right approximate unit,
then $({e_i}^*)_{i \in I}$ is a bounded left approximate unit by continuity
of the involution.
\end{proof}

\begin{theorem}\label{varnorm}\index{symbols}{v(phi)@$v(\varphi)$}%
Let $A$ be a normed \st-algebra with isometric involution and with
a bounded one-sided approximate unit in the closed unit ball.
If $\varphi$ is a continuous positive linear functional on $A$, then
$\varphi$ has finite variation given by
\[ v (\varphi) = | \,\varphi \,|. \]
\end{theorem}

\begin{proof} The preceding proof shows that $v (\varphi) \leq | \,\varphi \,|$.
To prove the converse inequality, please note that $\pi_{\varphi}$ is contractive
by \ref{weakbdedsarep} \& \ref{reprisomcontr}. By \ref{finvarcyclic} we have
\[ | \,\varphi (a) \,| =
   | \,\langle \pi_{\varphi} (a) c_{\varphi}, c_{\varphi} \rangle_{\varphi} \,|
   \quad \text{for all} \quad a \in A, \]
whence, by \ref{reprovectorfinvar},
\[ | \,\varphi (a) \,| \leq | \,\pi_{\varphi} (a) \,| \cdot {\| \,c_{\varphi} \,\| \,}^2 \leq
   | \,a \,| \cdot v (\varphi) \quad \text{for all} \quad a \in A. \qedhere \]
\end{proof}

\begin{corollary}\label{stateredef}%
Let $A$ be a normed \st-algebra with isometric involution and with
a bounded one-sided approximate unit in the closed unit ball.
Then a state on $A$ can be defined as a continuous positive linear
functional of norm $1$.
\end{corollary}

\begin{proof}
This follows from the fact that $A$ has continuous states,
cf.\ \ref{bdedstatessuffcond} (i).
\end{proof}

\medskip
For the preceding two results, see also \ref{postVarop} below. \pagebreak

\begin{corollary}\label{C*varnorm}
A positive linear functional $\varphi$ on a C*-algebra
is continuous and has finite variation given by
\[ v (\varphi) = \| \,\varphi \,\|. \]
It follows that a state on a C*-algebra can be defined as an
automatically continuous positive linear functional of norm $1$.
\end{corollary}

\begin{proof}
This follows now from \ref{plfC*bded} and \ref{C*canapprunit}.
\end{proof}

\medskip
Now for the unital case.

\begin{corollary}\label{unitalcase}%
Let $A$ be a normed \st-algebra with isometric \linebreak involution and with
unit $e$ of norm $1$. If $\varphi$ is a continuous positive linear
functional on $A$, then
\[ | \,\varphi \,| = \varphi (e). \]
\end{corollary}

\begin{proof}
This is a trivial consequence of \ref{vare} and \ref{varnorm}.
\end{proof}

\medskip
The next two results are frequently used. 

\begin{theorem}\label{posunitalC*}
A continuous linear functional $\varphi$ on a unital \linebreak
C*-algebra $A$ is positive if and only if
\[ \| \,\varphi \,\| = \varphi (e). \]
\end{theorem}

\begin{proof}
The ``only if'' part follows from the preceding corollary \ref{unitalcase}.
Suppose now that $\| \,\varphi \,\| = \varphi (e)$. We can then assume
that also $\| \,\varphi \,\| = 1$. If $a \in A$ and $\varphi (a^*a)$
is not $\geq 0$, there exists a (large) closed disc
$\{ z \in \mathds{C} : | \,z\0 - z \,| \leq \rho \}$ which contains $\s(a^*a)$,
but does not contain $\varphi (a^*a)$. (Please note that
$\s(a^*a) \subset [0,\infty[$ by the Shirali-Ford Theorem \ref{ShiraliFord}.)
Consequently $\s(z\0 e - a^*a)$ is contained in the disc
$\{ z \in \mathds{C} : | \,z \,| \leq \rho \}$, and so
$\| \,z\0 e -a^*a \,\| = \rlambda (z\0 e - a^*a) \leq \rho$, cf.\ \ref{C*rl}.
Hence the contradiction
\[ | \,z\0 - \varphi (a^*a) \,| = | \,\varphi (z\0 e - a^*a) \,|
\leq \| \,\varphi \,\| \cdot \| \,z\0 e - a^*a \,\| \leq \rho. \qedhere \]
\end{proof}

\begin{corollary}\label{stateunitalC*}
A state on a unital C*-algebra can be defined
as a continuous linear functional $\psi$ such that
\[ \| \,\psi \,\| = \psi (e) = 1. \]
\end{corollary}

\begin{proof}
This follows now from \ref{vare}. \pagebreak
\end{proof}

\begin{theorem}\label{repapprunit}%
Let $A$ be a normed \st-algebra with a bounded left approximate unit
$(e_i) _{i \in I}$. Let $\pi$ be a continuous non-degenerate representation
of $A$ on a Hilbert space $H$. We then have
\[ \lim _{i \in I} \pi (e_i) x = x \quad \text{for all} \quad x \in H. \]
\end{theorem}

\begin{proof}
Suppose first that $x \in H$ is of the form $x = \pi (a)y$ for
some $a \in A$, $y \in H$. By continuity of $\pi$ it follows that
\[ \lim _{i \in I} \| \,\pi (e_i) x - x \,\|
= \lim _{i \in I} \| \,\pi (e_i a-a)y \,\| = 0. \]
Thus $\lim _{i \in I} \pi (e_i)z = z$ for all $z$ in the dense subspace
\[ Z := \mathrm{span} \,\{ \,\pi (a)y \in H : a \in A, y \in H \,\}. \]
Denote by $c$ a bound of the representation, and by $d$ a bound
of the approximate unit. Given $x \in H$, $\varepsilon > 0$, there exists
$z \in Z$ with
\[ \| \,x-z \,\| < \frac{\varepsilon}{3\,(\,1+cd\,)}. \]
Then let $i \in I$ such that for all $j \in I$ with $j \geq i$ one has
\[ \| \,\pi (e_j)z-z \,\| < \frac{\varepsilon}{3}. \]
For all such $j$ we get
\begin{align*}
&\ \| \,\pi (e_j)x-x \,\| \\
\leq &\ \| \,\pi (e_j)x-\pi (e_j)z \,\| + \| \,\pi (e_j)z-z \,\| + \| \,z-x \,\| < \varepsilon.
 \qedhere
\end{align*}
\end{proof}

\medskip
We next give a description of an important class of states obtained
from the Riesz Representation Theorem, cf.\ the appendix \ref{Riesz}.

\begin{lemma}\label{C0Riesz}%
Let $\Omega \neq \varnothing$ be a locally compact Hausdorff space.
For a state $\psi$ on $\cont\0(\Omega)$ there exists a unique inner regular
Borel probability measure $\mu$ on $\Omega$ such that 
\[ \psi (f) = \int f \,\diff \mu\quad \text{for all} \quad f \in \cont\0(\Omega). \]
(For the definition of an ``inner regular Borel probability measure'',
see the appendix \ref{Boreldef} \& \ref{inregBormeas}.)
\end{lemma}

\begin{proof}
This is an application of the Riesz Representation Theorem, cf.\ the
appendix \ref{Riesz}, using the items \ref{plfpointwise}, \ref{C*varnorm},
and \ref{Ccdense}.\pagebreak
\end{proof}

\begin{definition}[$\probmeas (\Omega)$]%
\index{symbols}{M1@$\probmeas (\Omega)$}%
If $\Omega \neq \varnothing$ is a locally compact Hausdorff space,
one denotes by $\probmeas (\Omega)$ the set of inner regular
Borel probability measures on $\Omega$, cf.\ the appendix
\ref{Boreldef} \& \ref{inregBormeas}.
\end{definition}

\begin{theorem}\label{predisint}%
For a state $\psi$ on a commutative C*-algebra $A$, there
exists a unique measure $\mu \in \probmeas \bigl(\Delta(A)\bigr)$ such that
\[ \psi (a) = \int \wht{a} \,\diff \mu\quad \text{for all} \quad a \in A. \]
\end{theorem}

\begin{proof}
We note that $A \neq \{ 0 \}$ because $\psi$ is a state on $A$. It follows
that the Gel'fand transformation $a \mapsto \wht{a}$ establishes an
isomorphism of the C*-algebra $A$ onto the C*-algebra $\cont\0\bigl(\Delta(A)\bigr)$,
by the Commutative Gel'fand-Na\u{\i}mark Theorem \ref{commGN}. By defining
\[ \wht{\psi}(\wht{\vphantom{\psi}a}) := \psi (a)
\qquad ( \,a \in A \,), \]
we get a state $\wht{\psi}$ on $\cont\0\bigl(\Delta(A)\bigr)$. By lemma \ref{C0Riesz}
there exists a measure $\mu \in \probmeas \bigl(\Delta(A)\bigr)$ such that
\[ \psi (a) = \wht{\psi}(\wht{\vphantom{\psi}a})
= \int \wht{\vphantom{\psi}a} \,\diff \mu
\quad \text{for all} \quad a \in A. \]
Uniqueness is now clear from lemma \ref{C0Riesz} again.
\end{proof}

\medskip
We shall return to this theme in \ref{Commutativity}.

The impatient reader can go from here to part \ref{part3} of the book, on the
spectral theory of representations.

\clearpage


\section{The Continuity Theorem of Varopoulos}%
\label{Varopsect}

Our next objective is the Continuity Theorem of Varopoulos \ref{Varopoulos}.
As it is little used, we do not give a complete proof. We state without proof
two closely related Factorisation Theorems of Paul J.\ Cohen and Nicholas
Th.\ Varopoulos. We give them here in a form needed later.

\begin{theorem}[Paul J.\ Cohen]\label{Cohen}%
\index{concepts}{factorisation!Theorem!Cohen}%
\index{concepts}{Theorem!Cohen's Factorisation}%
\index{concepts}{Cohen's Factorisation Thm.}%
Let $A$ be a Banach algebra with a bounded right approximate unit.
Then every element of $A$ is a product of two elements of $A$.
\end{theorem}

\begin{theorem}[Nicholas Th.\ Varopoulos]\label{Varopoulosfact}%
\index{concepts}{Theorem!Varopoulos!Factorisation}%
\index{concepts}{Varopoulos' Theorem!Factorisation}%
\index{concepts}{factorisation!Theorem!Varopoulos}%
Let $A$ be a  Banach algebra with a bounded right approximate unit. For
a sequence $(a_n)$ in $A$ converging to $0$, there exists a sequence
$(b_n)$ in $A$ converging to $0$ as well as $c \in A$ with $a_n = {b_n}c$
for all $n$.
\end{theorem}

Proofs can be found for example at the following places: \cite[\S\ 6]{Zel},
\cite[\S\ 11]{BD}, \cite[section V.9]{FeDo} or \cite[section 5.2]{Palm}.

\begin{theorem}\label{preVarop1}%
Let $A$ be a Banach \st-algebra with a bounded right approximate unit.
A positive Hilbert form $\langle \cdot , \cdot \rangle$ on $A$, cf.\ \ref{Hilbertform},
is bounded in the sense that there exists $\gamma \geq 0$ such that
\[ | \,\langle a, b \rangle \,| \leq \gamma \cdot | \,a \,| \cdot | \,b \,|
\quad \text{for all} \quad a, b \in A. \]
\end{theorem}

\begin{proof}
Let $\pi$ be the representation associated with $\langle \cdot , \cdot \rangle$,
cf.\ \ref{GNS}. The representation $\pi$ is continuous by \ref{Banachbounded}.
It shall be shown in a first step that $\langle \cdot , \cdot \rangle$
is continuous in the first variable. So let $(a_n)$ be a sequence in $A$
converging to $0$. By the Factorisation Theorem of N.\ Th.\ Varopoulos
(theorem \ref{Varopoulosfact} above), there exists a sequence $(b_n)$
converging to $0$ in $A$ as well as $c \in A$, such that $a_n = {b_n}c$
for all $n$. For arbitrary $d \in A$, we obtain
\[ \langle a_n, d \rangle = \langle {b_n}c, d \rangle
= \langle \pi(b_n) \underline{c}, \underline{d} \rangle \to 0 \]
by continuity of $\pi$ (theorem \ref{Banachbounded}). It follows that
$\langle \cdot , \cdot \rangle$ is separately continuous in both variables
(as $\langle \cdot , \cdot \rangle$ is Hermitian). An application of the
uniform boundedness principle then shows that there exists
$\gamma \geq 0$ such that
\[ | \,\langle a, b \rangle \,| \leq \gamma \cdot | \,a \,| \cdot | \,b \,|
\quad \text{for all} \quad a, b \in A. \pagebreak  \qedhere \]
\end{proof}

\begin{theorem}\label{preVarop2}%
Let $A$ be a normed \st-algebra with a bounded right approximate
unit $(e_i)_{i \in I}$. Let $\langle \cdot , \cdot \rangle$ be a bounded
non-zero positive Hilbert form on $A$. The corresponding GNS
representation $\pi$ then is cyclic. Furthermore, there exists a cyclic
vector $c$ for $\pi$, such that the positive linear functional of finite variation
$a \mapsto \langle \pi (a) c, c \rangle$ $(a \in A)$ induces the positive
Hilbert form $\langle \cdot , \cdot \rangle$. (Cf.\ \ref{inducedHf}.) The
vector $c$ can be chosen to be any adherent point of the net
$(\underline{e_i})_{i \in I}$ in the weak topology. 
\end{theorem}

\begin{proof}
By boundedness of $\langle \cdot , \cdot \rangle$, the representation
associated with $\langle \cdot , \cdot \rangle$ is weakly continuous on
$A\sa$, and thus the representing operators are bounded, cf.\ \ref{weaksa}.
Continuation of the representing operators to the completion of $\underline{A}$,
yields a GNS representation $\pi$ associated with $\langle \cdot , \cdot \rangle$
on the completion of $\underline{A}$, cf.\ \ref{GNS}. Again by boundedness of
$\langle \cdot , \cdot \rangle$, the net $(\underline{e_i})_{i \in I}$ is bounded
and therefore has an adherent point in the weak topology on the completion
of $\underline{A}$. Let $c$ be any such adherent point. By going over to a
subnet, we can assume that $(\underline{e_i})_{i \in I}$ converges weakly to
$c$. It shall be shown that $\pi (a)c = \underline{a}$ for all $a \in A$, which
implies that $c$ is cyclic for $\pi$. It is enough to show that for $b \in A$, one has
$\langle \pi (a)c, \underline{b} \rangle = \langle \underline{a}, \underline{b} \rangle$.
So one calculates
\begin{align*}
 & \langle \pi (a)c, \underline{b} \rangle
= \langle c, \pi (a)^* \underline{b} \rangle
= \lim _{i \in I} \,\langle \underline{e_i}, \pi (a)^* \underline{b} \rangle \\
= \,&\lim _{i \in I} \,\langle \pi (a)\underline{e_i}, \underline{b} \rangle
= \lim _{i \in I} \,\langle a e_i, b \rangle
= \langle a, b \rangle = \langle \underline{a}, \underline {b} \rangle,
\end{align*}
where we have used the continuity of $\langle \cdot, \cdot \rangle$.
Consider now the positive linear functional $\varphi$ on $A$ defined by
\[ \varphi (a) := \langle \pi (a) c, c \rangle \qquad (a \in A). \]
We have
\[ \varphi (a^*a) = \langle \pi (a^*a) c, c \rangle = {\| \,\pi (a)c \,\| \,}^2
= {\| \,\underline{a} \,\| \,}^2
= \langle \underline{a}, \underline{a} \rangle
= \langle a, a \rangle \]
for all $a \in A$, whence, after polarisation,
\[ \langle a, b \rangle = \varphi(b^*a) \quad \text{for all} \quad a, b \in A. \]
That is, $\varphi$ induces the positive Hilbert form
$\langle \cdot , \cdot \rangle$, cf.\ \ref{inducedHf}.
\end{proof}

\medskip
We see from the proofs that so far it is the bounded right
approximate units which are effective. Now for the case
of a bounded two-sided approximate unit. We shall prove
that, in the preceding theorem \ref{preVarop2}, if we are
in presence of a bounded two-sided approximate unit
$(e_i)_{i \in I}$, and if $A$ is complete, then the cyclic
vector $c$ can be chosen to be the limit in norm of the
net $(\underline{e_i})_{i \in I}$.
\pagebreak

\begin{theorem}
Let $A$ be a Banach \st-algebra with a bounded two-sided
approximate unit $(e_i)_{i \in I}$. Let $\langle \cdot , \cdot \rangle$
be a non-zero positive Hilbert form on $A$. The corresponding GNS
representation $\pi$ then is cyclic. Furthermore, there exists a cyclic
vector $c$ for $\pi$, such that the positive linear functional of finite
variation $a \mapsto \langle \pi (a) c, c \rangle$ $(a \in A)$ induces
the positive Hilbert form $\langle \cdot , \cdot \rangle$. The vector
$c$ can be chosen to be the limit in norm of the net
$(\underline{e_i})_{i \in I}$.
\end{theorem}

\begin{proof}
The Hilbert form $\langle \cdot , \cdot \rangle$ is bounded by
\ref{preVarop1}, so \ref{preVarop2} is applic\-able. Only the last
statement needs to be proved. The representation $\pi$ is
continuous by \ref{Banachbounded}, and from the preceding
proof, we have $\underline{e_i} = \pi (e_i)c$ for all $i \in I$. Since
$(e_i)_{i \in I}$ among other things is a bounded left approximate
unit, \ref{repapprunit} implies that $\underline{e_i} = \pi (e_i)c \to c$
because $\pi$ is cyclic by the preceding theorem \ref{preVarop2},
and thus non-degenerate. 
\end{proof}

\medskip
Next for bounded one-sided approximate units:

\begin{theorem}[Varopoulos, Shirali]\label{Varopoulos}%
\index{concepts}{Theorem!Varopoulos!Continuity}%
\index{concepts}{Varopoulos' Theorem!Continuity}%
A positive linear functional on a Banach \st-algebra with a bounded
one-sided approximate unit has finite variation, and thence also is
continuous, cf.\ \ref{Banachfinvarbded}.
\end{theorem}

\begin{proof} In the case of a bounded right approximate unit, the
conclusion follows from the theorems \ref{preVarop1} and
\ref{preVarop2} by an application of Cohen's Factorisation Theorem
\ref{Cohen}. In presence of a bounded left approximate unit, we
change the multiplication to $(a,b) \mapsto ba$, and use the fact
that the functional turns out to be Hermitian, cf.\ \ref{variationinequal}.
\end{proof}

\medskip
This result is due to Varopoulos for the special case of a continuous
involution. It was extended to the general case by Shirali.

\begin{corollary}\index{symbols}{v(phi)@$v(\varphi)$}\label{postVarop}%
Let $A$ be a Banach \st-algebra with isometric \linebreak involution
and with a bounded one-sided approximate unit in the closed unit ball.
A positive linear functional $\varphi$ on $A$ is continuous and
has finite variation given by
\[ v (\varphi) = | \,\varphi \,|. \]
A state on $A$ can be defined as an automatically
continuous positive linear functional of norm $1$.
\end{corollary}

\begin{proof}
This follows now from \ref{varnorm} and \ref{stateredef}.
\end{proof}

\medskip
For C*-algebras, we don't need this result, see \ref{C*varnorm}.
\pagebreak

\clearpage


\addtocontents{toc}{\protect\vspace{0.2em}}

\chapter{The Enveloping C\texorpdfstring{*-}{\80\052\80\055}Algebra}

\setcounter{section}{31}


\section{C\texorpdfstring{*}{\80\052}(A)}

\medskip
In this paragraph, let $A$ be a normed \st-algebra.

\begin{observation}%
A state $\psi$ on $A$ satisfies
\[ {\psi (a^*a) \,}^{1/2} \leq \rsigma(a) \qquad \text{for all} \quad a \in A. \]
\end{observation}

\begin{proof} Theorem \ref{repstate} implies that
\[ {\psi(a^*a) \,}^{1/2}
= {\langle \pi _{\psi}(a^*a) c_{\psi}, c_{\psi} \rangle _{\psi}}^{1/2}
= \| \,\pi _{\psi}(a)c_{\psi} \,\| \leq \rsigma(a). \qedhere \]
\end{proof}

\begin{definition}[$\,\| \cdot \|\0\,$]%
\index{symbols}{a6@${"|}{"|}\,a\,{"|}{"|}\0$}%
For $a \in A$, one defines
\[ \| \,a \,\|\0 := \,\sup _{\text{\small{$\psi \in \statespace (A)$}}}
\,{\psi (a^*a) \,}^{1/2} \leq \rsigma(a). \]
If $\statespace (A) = \varnothing$, one puts $\| \,a \,\|\0 := 0$ for all $a \in A$.
\end{definition}

We shall see that $\| \cdot \|\0$ is a C*-seminorm \ref{Cstarsemin}
with a certain extremality property.

\begin{theorem}\label{boundedo}%
If $\pi$ is a representation of $A$ on a pre-Hilbert space $H$,
which is weakly continuous on $A\sa$, then
\[ \| \,\pi (a) \,\| \leq \| \,a \,\|\0 \qquad \text{for all} \quad a \in A. \]
\end{theorem}

\begin{proof}
We may assume that $H$ is a Hilbert space $\neq \{ 0 \}$
and that $\pi$ is non-degenerate. (Indeed, this follows from
\ref{weaksa} \& \ref{nondegenerate}.) Let $x \in H$ with
$\| \,x \,\| = 1$. Consider the state $\psi$ on $A$ defined by
\[ \psi (a) := \langle \pi (a) x, x \rangle \qquad ( \,a \in A \,), \]
cf.\ \ref{staterep}. We obtain
\[ \| \,\pi (a) x \,\| = {\langle \pi (a) x, \pi (a) x \rangle \,}^{1/2}
= {\psi (a^*a) \,}^{1/2} \leq \| \,a \,\|\0. \pagebreak \qedhere \]
\end{proof}

\begin{definition}[the universal representation, $\pi\univ$]%
\index{concepts}{representation!universal}%
\index{concepts}{universal!representation}%
\index{symbols}{p3@$\pi_{u}$}\index{symbols}{H2@$H_{u}$}%
We define
\[ \pi\univ := \oplus \,_{\text{\small{$\psi \in \statespace (A)$}}}
\,\pi _{\text{\small{$\psi$}}}, \qquad
H\univ := \oplus \,_{\text{\small{$\psi \in \statespace (A)$}}}
\,H _{\text{\small{$\psi$}}}. \]
One says that $\pi \univ$ is the \underline{universal representation}
of $A$ and that $H \univ$ is the universal Hilbert space of $A$.
The universal representation is \linebreak $\sigma$-contractive
by \ref{directsumsigma}, and non-degenerate by \ref{dirsumnondeg}. 
\end{definition}

\begin{theorem}\label{GNuniv}%
We have
\[ \| \,\pi \univ (a) \,\| = \| \,a \,\|\0 \quad \text{for all} \quad a \in A. \]
\end{theorem}

\begin{proof}
The representation $\pi \univ$ is $\sigma$-contractive as noted
before, so that, by \ref{boundedo},
\[ \| \,\pi \univ (a) \,\| \leq \| \,a \,\|\0 \]
for all $a \in A$. In order to show the converse inequality, let
$a \in A$ be fixed. Let $\psi \in \statespace (A)$. It suffices then to prove that
$\psi (a^*a) \leq {\| \,\pi \univ (a) \,\| \,}^2$. We have
\begin{align*}
\psi (a^*a)
& = \langle \pi _{\psi} (a^*a) c_{\psi}, c_{\psi} \rangle _{\psi}
= \langle \pi _{\psi} (a) c_{\psi}, \pi _{\psi} (a) c_{\psi} \rangle _{\psi} \\
& = {\| \,\pi _{\psi} (a) c_{\psi} \,\| \,}^2
\leq {\| \,\pi _{\psi} (a) \,\| \,}^2 \leq {\| \,\pi \univ (a) \,\| \,}^2. \qedhere
\end{align*}
\end{proof}

\begin{definition}[the Gel'fand-Na\u{\i}mark seminorm, $\| \cdot \|\0$]%
\index{concepts}{Gel'fand-Naimark@Gel'fand-Na\u{\i}mark!seminorm}%
\label{GNSseminorm}%
We see from the preceding theorem \ref{GNuniv} that $\| \cdot \| \0$
is a C*-seminorm on $A$. This C*-seminorm $\| \cdot \| \0$
is called the \underline{Gel'fand-Na\u{\i}mark seminorm}. It satisfies
\[ \| \,a \,\|\0 \leq \rsigma(a) \leq \| \,a \,\| \quad \text{for all} \quad a \in A, \]
where $\| \cdot \|$ denotes the accessory \st-norm on $A$.
\end{definition}

\begin{proposition}\label{bstGNc}%
If $A$ has continuous states, then the Gel'fand-Na\u{\i}mark seminorm
is continuous on $A$. If $A$ furthermore is commutative, then the
Gel'fand-Na\u{\i}mark seminorm is dominated by the norm in $A$.
\end{proposition}

\begin{proof}
This follows from \ref{bdedstatesthm} and \ref{weakcontHilbert}.
\end{proof}

\medskip
A basic object in representation theory is the following enveloping
\linebreak C*-algebra. It is useful for pulling down results. \pagebreak

\begin{definition}[the enveloping C*-algebra, $\Cstar(A)$]%
\index{symbols}{s035@\protect\st-rad(A)}%
\index{concepts}{s04@\protect\st-radical}%
\index{symbols}{C5@$\Cstar(A)$}\index{symbols}{j@$j$}%
\index{concepts}{C6@C*-algebra!enveloping}%
\index{concepts}{enveloping C*-algebra}%
\index{concepts}{algebra!C*-algebra!enveloping}%
The \underline{\st-radical} of $A$ is defined to be the \st-stable
\ref{selfadjointsubset} two-sided ideal on which $\| \cdot \|\0$ vanishes.
It shall be denoted by \st-$\mathrm{rad}(A)$.
We introduce a norm $\| \cdot \|$ on $A\mspace{1mu}/$\st-$\mathrm{rad}(A)$
by defining
\[ \| \,a+\text{\st-}\mathrm{rad}(A) \,\| := \| \,a \,\|\0 \quad \text{for all} \quad a \in A. \]
Then $(A\mspace{1mu}/$\st-$\mathrm{rad}(A), \| \cdot \,\|)$ is isomorphic as a normed
\st-algebra to the range of the universal representation, and thus is a
pre-C*-algebra. Let $(\Cstar(A),\| \cdot \|)$ denote the completion. It is
called the \underline{enveloping} \underline{C*-algebra} of $A$. We put
\[ j : A \to \Cstar(A),\quad a \mapsto a+\text{\st-}\mathrm{rad}(A). \]
\end{definition}

\begin{proposition}\label{jdense}%
The \st-algebra homomorphism $j : A \to \Cstar(A)$ is
$\sigma$-contractive and its range is dense in $\Cstar(A)$.
\end{proposition}

\begin{proposition}%
If $A$ has continuous states, then the mapping \linebreak
$j : A \to \Cstar(A)$ is continuous. If furthermore $A$ is commutative,
then $j$ is contractive.
\end{proposition}

\begin{proof}
This follows from \ref{bstGNc}.
\end{proof}

\medskip
We now come to a group of statements which collectively might be
called the ``\underline{universal property}'' of the enveloping C*-algebra.
\index{concepts}{universal!property!of enveloping C*-algebra}

\begin{proposition}\index{symbols}{p4@$\pi \0$}%
Let $H$ be a Hilbert space. The assignment
\[ \pi := \pi \0 \circ j \]
establishes a bijection from the set of all representations $\pi \0$ of
$\Cstar(A)$ on $H$ onto the set of all $\sigma$-contractive representations
$\pi$ of $A$ on $H$.
\end{proposition}

\begin{proof} Let $\pi \0$ be a representation of $\Cstar(A)$
on $H$. Then $\pi \0$ is contractive by \ref{C*repcontract},
and it follows that $\pi := \pi \0 \circ j$ is a $\sigma$-contractive
representation of $A$ on $H$, by \ref{jdense}. Conversely, let $\pi$
be a $\sigma$-contractive representation of $A$ on $H$. We then have
$\| \,\pi (a) \,\| \leq \| \,a \,\|\0$ by \ref{boundedo}. It follows that $\pi$
vanishes on \st-$\mathrm{rad}(A)$ and thus induces a representation
$\pi \0$ of $\Cstar(A)$ such that $\pi = \pi \0 \circ j$. To show injectivity, let
$\rho \0$, $\sigma \0$ be two representations of $\Cstar(A)$ on $H$ with
$\rho \0 \circ j = \sigma \0 \circ j$. Then $\rho \0$ and $\sigma \0$
agree on the dense set $j(A)$, hence everywhere, by continuity.
\pagebreak
\end{proof}

\begin{proposition}\label{piopi}%
Let $\pi$ be a representation of $A$ on a Hilbert space $H$, which is
weakly continuous on $A\sa$. Let $\pi \0$ be the corresponding
representation of $\Cstar(A)$. Then $\range(\pi \0)$ is the closure of
$\range(\pi)$. It follows that
\begin{itemize}
   \item[$(i)$] $\pi$ is non-degenerate if and only if $\pi \0$ is so,
  \item[$(ii)$] $\pi$ is cyclic if and only if $\pi \0$ is so,
 \item[$(iii)$] ${\pi \0}' = {\pi}'.$
\end{itemize}
\end{proposition}

\begin{proof}
The range of $\pi \0$ is a C*-algebra by \ref{rangeC*}.
Furthermore $\range(\pi)$ is dense in $\range(\pi \0)$ by
\ref{jdense}, from which it then follows that $\range(\pi \0)$ is
the closure of $\range(\pi)$.
\end{proof}

\begin{proposition}%
If $\psi$ is a state on $A$, then
\[ | \,\psi(a) \,| \leq \| \,a \,\|\0 \quad \text{for all} \quad a \in A. \]
\end{proposition}

\begin{proof} We have
\[ | \,\psi (a) \,| \leq {\psi (a^*a) \,}^{1/2} \leq
\sup _{\varphi \in \statespace (A)} {\varphi (a^*a) \,}^{1/2} = \| \,a \,\|\0. \qedhere \]
\end{proof}

\begin{corollary}\label{linboundedo}%
If $\varphi$ is a positive linear functional of  finite variation
on $A$, which is weakly continuous on $A\sa$, then
\[ | \,\varphi (a) \,| \leq v( \varphi ) \,\| \,a \,\|\0 \quad \text{for all} \quad a \in A. \]
\end{corollary}

\begin{proposition}\index{symbols}{f1@$\varphi \0$}%
By putting
\[ \varphi := \varphi \0 \circ j \]
we establish a bijection from the set of positive linear functionals
$\varphi \0$ on $\Cstar(A)$ onto the set of positive linear functionals
$\varphi$ of finite variation on $A$, which are weakly continuous on
$A\sa$. Then $v ( \varphi ) = v ( \varphi \0 )$.
\end{proposition}

\begin{proof} 
Please note first that every positive linear functional on a \linebreak
C*-algebra is continuous and has finite variation, by \ref{C*varnorm}.

Let $\varphi \0$ be a positive linear functional on $\Cstar(A)$.
Then $\varphi := \varphi \0 \circ j$ is a positive linear functional on $A$,
which is weakly continuous on $A\sa$ because $j$ is contractive on $A\sa$
cf.\ \ref{jdense} \& \ref{sigmaAsa}. For $a \in A$, we have
\begin{gather*}
{| \,\varphi (a) \,| \,}^2 = {| \,( \varphi \0 \circ j ) (a) \,| \,}^2, \\
\varphi (a^*a) = ( \varphi \0 \circ j ) (a^*a),
\end{gather*}
so that $v ( \varphi ) = v ( \varphi \0 |_{j(A)})$. By continuity of $\varphi \0$
on $\Cstar(A)$ and by continuity of the involution in $\Cstar(A)$, it follows that
$v ( \varphi ) = v ( \varphi \0 )$. \pagebreak

Conversely, let $\varphi$ be a positive linear functional of finite variation
on $A$, which is weakly continuous on $A\sa$. By \ref{linboundedo},
$\varphi$ also is bounded with respect to $\| \cdot \|\0$. It therefore
induces a positive linear functional $\varphi \0$ on $\Cstar(A)$
such that $\varphi = \varphi \0 \circ j$. To show injectivity, let
$\varphi \0$, ${\varphi \0}'$ be positive linear functionals on $\Cstar(A)$ with
$\varphi \0 \circ j = {\varphi \0}' \circ j$. Then $\varphi \0$ and ${\varphi \0}'$
coincide on the dense set $j(A)$, hence everywhere, by continuity of
$\varphi \0$ and ${\varphi \0}'$ on $\Cstar(A)$. 
\end{proof}

\begin{reminder}%
We imbed the set $\quasistates (A)$ of quasi-states on $A$ in the unit ball of
the dual normed space of the real normed space $A\sa$, cf.\ \ref{propquasi}.
The set $\quasistates (A)$ then is a compact Hausdorff space in the \linebreak
weak* topology. (By Alaoglu's Theorem, cf.\ the appendix \ref{Alaoglu}.)
\end{reminder}

\begin{theorem}\label{affinehomeom}%
The assignment
\[ \varphi := \varphi \0 \circ j \]
is an \underline{affine homeomorphism} from the set $\quasistates\bigl(\Cstar(A)\bigr)$ of
quasi-states $\varphi \0$ on $\Cstar(A)$ onto the set $\quasistates (A)$ of quasi-states
$\varphi$ on $A$, and hence from the set $\statespace\bigl(\Cstar(A)\bigr)$ onto the set
$\statespace (A)$.
\end{theorem}

\begin{proof}
This assignment is continuous by the universal property of the weak* topology,
cf.\ the appendix \ref{weak*top}. The assignment, being a continuous bijection
from the compact space $\quasistates\bigl(\Cstar(A)\bigr)$ to the Hausdorff space
$\quasistates (A)$, is a homeomorphism, cf.\ the appendix \ref{homeomorph}. 
\end{proof}

\medskip
The next two results put the ``extension processes'' for states
and their GNS representations in a useful relationship.

\begin{proposition}\label{piopsio}%
Let $\psi$ be a state on $A$, and let $\psi\0$ be the corresponding state
on $\Cstar(A)$. Let $(\pi _{\psi}) \0$ denote the representation of $\Cstar(A)$
corresponding to the GNS representation $\pi _{\psi}$. We then have
\[ \psi \0 (b) = \langle (\pi _{\psi}) \0 (b) c _{\psi},c _{\psi} \rangle _{\psi}
\quad \text{ for all} \quad b \in \Cstar(A). \]
\end{proposition}

\begin{proof}
Consider the state $\psi _1$ on $\Cstar(A)$ defined by
\[ \psi _1 (b) := \langle (\pi _{\psi}) \0 (b) c _{\psi}, c _{\psi} \rangle _{\psi}
\qquad \bigl( \,b \in \Cstar(A) \,\bigr). \]
We get
\[ \psi _1 \bigl(j(a)\bigr)
= \langle (\pi _{\psi}) \0 \bigl(j(a)\bigr) c _{\psi}, c_{\psi} \rangle _{\psi}
= \langle \pi _{\psi} (a)c _{\psi}, c_{\psi} \rangle _{\psi} = \psi (a), \]
for all $a \in A$, so that $\psi _1 = \psi \0$. \pagebreak
\end{proof}

\begin{corollary}\label{oequiv}%
For a state $\psi$ on $A$, the representation $\pi _{(\psi \0)}$ is spatially
equivalent to $(\pi _{\psi}) \0$, and there exists a unique unitary operator
intertwining $\pi _{(\psi \0)}$ with $(\pi _{\psi}) \0$ taking $c _{(\psi \0)}$
to $c _{\psi}$.
\end{corollary}

\begin{proof} \ref{coeffequal}. \end{proof}

\begin{proposition}\label{reform2}%
If $A$ has continuous states, then \st-$\mathrm{rad}(A)$ is closed in $A$.
\end{proposition}

\begin{proof}
This is a consequence of \ref{bstGNc}. For Banach \st-algebras,
this also follows from \ref{kerclosed}.
\end{proof}

\medskip
We close this paragraph with a discussion
of \st-semisimple normed \linebreak \st-algebras.

\begin{definition}[\st-semisimple normed \st-algebras]%
\index{concepts}{s05@\protect\st-semisimple}%
We  will say that $A$ is \underline{\st-semisimple} if
\st-$\mathrm{rad}(A) = \{0\}$.
\end{definition}

Please note that $A$ is \st-semisimple if and only
if its universal representation is faithful.

\begin{proposition}\label{faithsemi}%
If $A$ has a faithful representation on a pre-Hilbert space,
which is weakly continuous on $A\sa$, then $A$ is \st-semi\-simple.
\end{proposition}

\begin{proof}
A representation which is weakly continuous on
$A\sa$ vanishes on the \st-radical by \ref{boundedo}.
\end{proof}

\begin{proposition}\label{stsemipreC*}%
If $A$ is \st-semisimple, then $\bigl( \,A, \| \cdot \|\0 \,\bigr)$
is a pre-C*-algebra and $\Cstar(A)$ is its completion. 
\end{proposition}

\begin{proposition}\label{reform1}%
If $A$ is \st-semisimple, then the involution in $A$ has closed
graph. In particular, it then follows that $A\sa$ is closed in $A$.
\end{proposition}

\begin{proof}
This follows from \ref{injclosgraph}.
\end{proof}

\begin{corollary}\label{stsemiscont}%
The involution in any \st-semisimple Banach \linebreak
\st-algebra is continuous.
\pagebreak
\end{corollary}

\clearpage


\section{The Theorems of Ra\texorpdfstring{\u{\i}}{\81\055}kov and of %
Gel'fand \texorpdfstring{\&}{\80\046} Na\texorpdfstring{\u{\i}}{\81\055}mark}

Our next aim is a result of Ra\u{\i}kov \ref{Raikov}, of which the
Gel'fand-Na\u{\i}mark Theorem \ref{Gel'fandNaimark} is a special case.

\begin{theorem}\label{Krein}%
Let $C$ be a convex cone in a real vector space $B$. Let $f\0$ be a linear
functional on a subspace $M\0$ of $B$ such that $f\0(x) \geq 0$ for all
$x \in M\0 \cap C$. Assume that $M\0+C = B$. The functional $f\0$
then has a linear extension $f$ to $B$ such that $f(x) \geq 0$ for all $x \in C$.
\end{theorem}

\begin{proof}
Assume that $f\0$ has been extended to a linear functional $f_1$ on a
subspace $M_1$ of $B$ such that $f_1(x) \geq 0$ for all $x \in M_1  \cap C$.
Assume that $M_1  \neq B$. Let $y \in B$, $y \notin M_1$ and put
$M_2 := \mathrm{span} \,( M_1 \cup \{ \,y \,\} )$. Define
\[ E' := M_1 \cap (y-C), \quad E'' := M_1 \cap (y+C). \]
The sets $E'$ and $E''$ are non-void because $y, -y \in M\0+C$
(by the assumption that $M\0+C = B$).
If $x' \in E'$, and $x'' \in E''$, then $y-x' \in C$ and $x''-y \in C$, so
$x''-x' \in M_1 \cap C$, whence $f_1(x') \leq f_1(x'')$. It follows that with
\[ a := \sup _{x' \in E'} f_1(x'),\quad b := \inf _{x'' \in E''} f_1(x'') \]
one has $-\infty < a \leq b < \infty$. With $a \leq c \leq b$ we then have
\[ f_1(x') \leq c \leq f_1(x'') \quad \text{for all} \quad x' \in E', x'' \in E''. \]
Define a linear functional $f_2$ on
$M_2 = \mathrm{span} \,( M_1  \cup \{ \,y \,\} )$ by
\[ f_2(x+\alpha y) := f_1(x)+\alpha c
\quad \text{for all} \quad x \in M_1, \alpha \in \mathds{R}. \]
It shall be shown that $f_2(x) \geq 0$ for all $x \in M_2 \cap C$. So let
$x \in M_1$, $\alpha \in \mathds{R} \setminus \{0\}$. There then exist $\beta > 0$
and $z \in M_1$ such that either $x+\alpha y = \beta (z+y)$ or $x+\alpha y =
\beta (z-y)$. Thus, if $x+\alpha y \in C$, then either $z+y \in C$ or $z-y
\in C$. If $z+y \in C$, then $-z \in E'$, so that $f_1(-z) \leq c$, or
$f_1(z) \geq -c$, whence $f_2(z+y) \geq 0$. If $z-y \in C$, then $z \in E''$,
so that $f_1(z) \geq c$, whence $f_2(z-y) \geq 0$. Thus $f_1$ has been
extended to a linear functional $f_2$ on the subspace $M_2$ properly
containing $M_1$, such that $f_2(x) \geq 0$ for all $x \in M_2 \cap C$.

One defines
\begin{align*}
Z := \{ \,(M,f) : \ & M \text{ is a subspace of } B \text{ containing } M\0, \\
 & f \text{ is a linear functional on } M \text{ extending } f\0 \\
 & \text{such that } f(x) \geq 0 \text{ for all } x \in M \cap C \,\}. 
\end{align*}
Applying Zorn's Lemma yields a maximal element $(M,f)$ of $Z$ (with the
obvious ordering). It follows from the above that $M = B$. \pagebreak
\end{proof}

\begin{theorem}\label{extstate}%
If $A$ is a Hermitian Banach \st-algebra, then a state on a closed
\st-subalgebra of $A$ can be extended to a state on all of $A$.
\end{theorem}

\begin{proof} Let $D$ be a closed \st-subalgebra of the Hermitian Banach
\linebreak
\st-algebra $A$. Put $E := \mathds{C}e+D$, where $e$ is the unit in $\tld{A}$.
Then $E$ is a Hermitian Banach \st-algebra\vphantom{$\tld{A}$}, cf.\ \ref{Herminher}
as well as the proof of \ref{Banachunitis}. Let $\psi$ be a state on $D$. Putting
$\tld{\psi}(e) := 1$ and extending linearly, we get a state $\tld{\psi}$ on $E$ which
extends $\psi$, cf.\ the proof of \ref{extendable}, as well as \ref{eBbded}. We want
to apply the preceding theorem \ref{Krein} with $B := \tld{A}\sa$, $C := \tld{A}_+$,
$M\0 := E\sa$, $f\0 := \tld{\psi}|_{\text{\small{$E\sa$}}}$. The assumptions are
satisfied as $\tld{A}_+$ is convex by \ref{plusconvexcone}, $\tld{\psi}$ is
Hermitian by \ref{propquasi}, $\tld{\psi}(b) \geq 0$ for all $b \in E_+$ by \ref{plusstate},
$E_+ = \tld{A}_+ \cap E\sa$ by \ref{plussubal}, and $E\sa+\tld{A}_+ = \tld{A}\sa$
because for $a \in \tld{A}\sa$ one has $\rlambda(a)e+a \geq 0$. By the preceding
extension theorem \ref{Krein}, it follows that $\tld{\psi}|_{\text{\small{$E\sa$}}}$ has
an extension to a linear functional $f$ on $\tld{A}\sa$ such that $f(a) \geq 0$ for all
$a \in \tld{A}_+$. Let $\tld{\varphi}$ be the linear extension of $f$ to $\tld{A}$. We
then have $\tld{\varphi} \,(a^*a) \geq 0$ for all $a \in \tld{A}$ because
$a^*a \in \tld{A}_+$ by the Shirali-Ford Theorem \ref{ShiraliFord}. Moreover,
$\tld{\varphi}$ is a state on $\tld{A}$ because $\tld{\varphi}(e) = \tld{\psi}(e) = 1$,
cf.\ \ref{eBbded}. Let $\varphi$ be the restriction of $\tld{\varphi}$ to $A$. Then
$\vphantom{\tld{A}}\varphi$ is a quasi-state extending $\psi$, and thus a
state\vphantom{$\tld{A}$}.
\end{proof}

\begin{theorem}\label{posequiv}%
Let $A$ be a Hermitian Banach \st-algebra, and let $a$ be a
Hermitian element of $A$. The following properties are equivalent.
\begin{itemize}
   \item[$(i)$] $a \geq 0$,
  \item[$(ii)$] $\pi \univ (a) \geq 0$,
 \item[$(iii)$] $\psi(a) \geq 0$ \quad for all \quad $\psi \in \statespace (A)$.
\end{itemize}
\end{theorem}

\begin{proof}
The implications (i) $\Rightarrow$ (ii) $\Leftrightarrow$ (iii) hold for
normed \st-algebras in general. Indeed (i) $\Rightarrow$ (ii) follows
from \ref{hompos} while (ii) $\Leftrightarrow$ (iii) follows from
$\pi \univ (a) \geq 0 \Leftrightarrow \langle \pi\univ (a)x,x\rangle \geq 0$
for all $x \in H\univ$, cf.\ \ref{posop}. The closed \st-subalgebra $B$ of $A$
generated by $a$ is a commutative \ref{clogen} Hermitian \ref{Herminher}
Banach \st-algebra. One notes that any multiplicative linear functional
$\tau$ on $B$ is a state. Indeed, as $\tau$ is Hermitian \ref{mlfHerm},
we have $| \,\tau (b) \,|\,^2 = \tau (b^*b) $ for all $b \in B$,
so that $\tau$ is positive and has variation $1$. Furthermore, $\tau$ is
continuous by \ref{mlfbounded}. The proof of (iii) $\Rightarrow$ (i)
goes as follows. If $\psi(a) \geq 0$ for all $\psi$ in $\statespace (A)$ then
by the preceding theorem \ref{extstate} also $\psi(a) \geq 0$ for all $\psi$ in
$\statespace (B)$. In particular $\tau(a) \geq 0$ for all $\tau \in \Delta (B)$.
It follows from \ref{rangeGT} that $a \in B_+$, whence $a \in A_+$ by
\ref{specsubalg}.
\pagebreak
\end{proof}

\begin{theorem}[Ra\u{\i}kov's Criterion]\label{Raikov}%
\index{concepts}{Theorem!Ra\u{\i}kov's Criterion}%
\index{concepts}{Ra\u{\i}kov's Criterion}%
For a Banach \st-algebra $A$, the following statements are equivalent.
\begin{itemize}
   \item[$(i)$] $A$ is Hermitian,
  \item[$(ii)$] $\| \,\pi \univ (a) \,\| = \rsigma(a)$ \quad for all \quad $a \in A$,
 \item[$(iii)$] $\| \,a \,\|\0 = \rsigma(a)$ \quad for all \quad $a \in A$.
\end{itemize}
\end{theorem}

\begin{proof}
(ii) $\Leftrightarrow$ (iii): $\| \,a \,\|\0 = \| \,\pi \univ (a) \,\|$, cf.\ \ref{GNuniv}.
(ii) $\Rightarrow$ (i): \ref{rsHerm} and \ref{C*rs}.
(i) $\Rightarrow$ (ii): The Gel'fand-Na\u{\i}mark seminorm on $\tld{A}$
coincides on $A$ with $\| \cdot \|\0$, because a state on $A$ can be
extened to a state on $\tld{A}$, by \ref{statecanext}, and because the
restriction to $A$ of a state on $\tld{A}$ is a quasi-state on $A$. One
can thus assume that $A$ is unital. Then $\pi \univ$ is unital, by
\ref{unitnondeg}. If $A$ is Hermitian, then the preceding theorem
\ref{posequiv} and lemma \ref{posposrs} imply that for every
$a \in A$, one has $\rsigma \bigl(\pi \univ(a)\bigr) = \rsigma(a)$,
or $\| \,\pi \univ (a) \,\| = \rsigma(a)$, cf.\ \ref{C*rs}.
\end{proof}

\begin{theorem}\label{maincommHerm}%
If $A$ is a commutative Hermitian Banach \st-algebra with
$\Delta (A) \neq \varnothing$, then for all $a \in A$, we have
\[ \| \,a \,\|\0 = | \,\wht{a} \,|_{\infty} \]
\end{theorem}

\begin{proof}
$\rsigma(a) = \rlambda(a) = | \,\wht{a} \,|_{\infty}$
by \ref{fundHerm} (iv) \& \ref{normGTrl}.
\end{proof}

\begin{corollary}%
Let $A$ be a commutative Hermitian Banach \linebreak \st-algebra
with $\Delta (A) \neq \varnothing$. Then $A$ is \st-semisimple if and
only if the Gel'fand transformation $A \to \cont\0\bigl(\Delta(A)\bigr)$
is injective.
\end{corollary}

\begin{corollary}%
Let $A$ be a commutative Hermitian Banach \linebreak \st-algebra
with $\Delta (A) \neq \varnothing$. Then $\Cstar(A)$ is isomorphic
as a C*-algebra to $\cont\0\bigl(\Delta(A)\bigr)$.
\end{corollary}

\begin{proof}
By \ref{maincommHerm}, the pre-C*-algebra
$\bigl(A\mspace{1mu}/$\st-$\mathrm{rad}(A),\| \cdot \|\bigr)$
is isomorphic to
$\{ \,\wht{a} \in \cont\0\bigl(\Delta(A)\bigr) : a \in A \,\}$,
which is dense in $\cont\0\bigl(\Delta(A)\bigr)$ by \ref{Hermdense}.%
\end{proof}

\begin{theorem}[the Gel'fand-Na\u{\i}mark Theorem]%
\index{concepts}{Theorem!Gel'fand-Na\u{\i}mark}\label{Gel'fandNaimark}%
\index{concepts}{Gel'fand-Naimark@Gel'fand-Na\u{\i}mark!Theorem}%
If $(A, \| \cdot \|)$ is a C*-algebra, then for all $a \in A$, we have
\[  \| \,\pi \univ (a) \,\| = \| \,a \,\|\0 = \| \,a \,\|, \]
so that the universal representation of $A$ establishes an isomorphism
of \linebreak C*-algebras from $A$ onto a C*-algebra of bounded linear
operators on a Hilbert space. Furthermore $\Cstar(A)$ can be identified
with $A$. \pagebreak
\end{theorem}

\begin{proof}
This follows from Ra\u{\i}kov's Criterion \ref{Raikov} together with the fact
that $\rsigma(a) = \| \,a \,\|$ for all elements $a$ of a C*-algebra, cf.\ \ref{C*rs}.
\end{proof}

\medskip
The Gel'fand-Na\u{\i}mark Theorem has important consequences in
\linebreak connection with C*-seminorms. Please note that the
following result essentially contains the Gel'fand-Na\u{\i}mark Theorem.

\begin{theorem}\label{existrep}%
For a C*-seminorm $p$ on a \st-algebra $A$, there exists a
representation $\pi$ of $A$ on a Hilbert space, such that
\[ p(a) = \| \,\pi(a) \,\| \quad \text{for all} \quad a \in A. \]
\end{theorem}

\begin{proof}
Let $p$ be a C*-seminorm on a \st-algebra $A$. There then exists a
\st-algebra homomorphism $\pi_1$ from $A$ to a C*-algebra $B$,
such that $p(a) = \| \,\pi_1 (a) \,\|$ for all $a \in A$, cf.\ \ref{existhom}.
By the Gel'fand-Na\u{\i}mark Theorem \ref{Gel'fandNaimark},
there exists an isometric \st-algebra isomorphism $\pi_2$ from
the C*-algebra $B$ onto a C*-algebra of bounded linear operators on
a Hilbert space. Then $\pi := \pi_2 \circ \pi_1$ is a representation of
$A$ on a Hilbert space with $p(a) = \| \,\pi(a) \,\|$ for all $a \in A$.
\end{proof}

\medskip
We obtain the following fundamental extremality property of
the Gel'fand-Na\u{\i}mark seminorm on Banach \st-algebras.

\begin{theorem}\label{GNsnchar}%
Let $A$ be a Banach \st-algebra. Then the Gel'fand-Na\u{\i}mark
seminorm on $A$ is the greatest C*-seminorm on $A$:
\[ \| \,a \,\|\0 = \mathrm{m}(a) \quad \text{for all} \quad a \in A, \]
cf.\ \ref{greatestC*sn}.
\end{theorem}

\begin{proof}
The Gel'fand-Na\u{\i}mark seminorm is a C*-seminorm by \ref{GNuniv},
whence $\| \,a \,\|\0 \leq \mathrm{m}(a)$ for all $a \in A$. Conversely,
for the greatest \linebreak C*-seminorm on $A$, there exists a
representation $\pi$ of $A$ on a Hilbert space with
$\mathrm{m}(a) = \| \,\pi(a) \,\|$ for all $a \in A$,
cf.\ the preceding theorem \ref{existrep}. Now $\| \,\pi(a) \,\| \leq \| \,a \,\|\0$
for all $a \in A$ by \ref{boundedo} as representations of Banach \st-algebras
on Hilbert spaces are continuous, cf.\ \ref{Banachbounded}. Hence
$\mathrm{m}(a) \leq \| \,a \,\|\0$ for all $a \in A$, and the statement follows.
\end{proof}

\medskip
From Ra\u{\i}kov's Criterion \ref{Raikov} we now regain \ref{PtakRaikov}.

\bigskip
It is not difficult to formulate and prove a version of the preceding theorem
\ref{GNsnchar} for normed \st-algebras instead of Banach \st-algebras.
\pagebreak

\clearpage


\section{Commutativity}%
\label{Commutativity}

\medskip
In this paragraph, let $A$ be a \underline{commutative} normed
\st-algebra.

\begin{definition}[$\Delta^*\protect\bsa(A)$]\label{mlfHbsa}%
\index{symbols}{D(A)bsa@$\Delta^*\protect\bsa(A)$}%
We denote by $\Delta^*\bsa(A)$ the set of \linebreak Hermitian
multiplicative linear functionals on $A$, which are bounded on $A\sa$.
Please note that
\[ \Delta^*\bsa(A) = \statespace (A) \cap \Delta (A). \]
\end{definition}

As in \ref{Gtopo} we obtain:

\begin{definition}\label{topDbsa}%
We equip the set $\Delta^*\bsa(A)$ with the topology \linebreak
inherited from $\quasistates (A)$, cf.\ \ref{topQS}. It becomes in this way a locally
\linebreak compact Hausdorff space.
\end{definition}

\begin{proposition}%
If $A$ furthermore is a Hermitian commutative Banach \st-algebra, then
$\Delta^*\bsa(A) = \Delta(A)$, also in the sense of topo\-logical spaces.
(As $A$ is spanned by $A\sa$, cf.\ the appendix \ref{weak*point}.)
\end{proposition}

\begin{proposition}%
\index{symbols}{a07@$\protect\wht{a}$}\label{commdense}%
If $\Delta^*\bsa(A) \neq \varnothing$, the Gel'fand transformation
\begin{align*}
A & \to \cont\0\bigl(\Delta^*\bsa(A)\bigr) \\
a & \mapsto \wht{a}
\end{align*}
then is a $\sigma$-contractive \st-algebra homomorphism,
whose range is dense in $\cont\0\bigl(\Delta^*\bsa(A)\bigr)$.
(We shall abbreviate
$\wht{a}|_{\Delta^*_{\text{\small{bsa}}}(A)}$ to $\wht{a}$).
\end{proposition}

\begin{proof}
One applies the Stone-Weierstrass Theorem \ref{StW}.
\end{proof}

\begin{proposition}\label{tauhom}%
The homeomorphism $\varphi\0 \mapsto \varphi\0 \circ j$ \ref{affinehomeom}
from \linebreak $S\bigl(\Cstar(A)\bigr)$ onto $\statespace (A)$ restricts to a
homeomorphism from $\Delta\bigl(\Cstar(A)\bigr)$ onto $\Delta^*\bsa(A)$.
\end{proposition}

\begin{corollary}%
If $\Delta^*\bsa(A) \neq \varnothing$, then $\Cstar(A)$ is isomorphic as a \linebreak
C*-algebra to $\cont\0\bigl(\Delta^*\bsa(A)\bigr)$, cf.\ \ref{commGN}.
One then has
\[ \| \,a \,\|\0 = \bigl| \,\wht{a}|_{\Delta^*_{\text{\small{bsa}}}(A)} \,\bigr|_{\infty}
\quad \text{for all} \quad a \in A. \]
It follows that $A$ is \st-semisimple if and only if the Gel'fand trans\-formation
$A \to \cont\0\bigl(\Delta^*\bsa(A)\bigr)$ is injective. \pagebreak
\end{corollary}

\begin{theorem}[the abstract Bochner Theorem]%
\label{Bochner}\index{concepts}{Theorem!abstract!Bochner}%
\index{concepts}{abstract!Bochner Theorem}%
\index{concepts}{Bochner Thm., abstr.}%
The formula
\[ \psi (a) = \int \wht{a} \,\diff \mu \qquad (a \in A) \]
establishes an affine bijection from
$\probmeas \bigl(\Delta^*\bsa(A)\bigr)$\,onto $\statespace (A)$.
\end{theorem}

\begin{proof} Let $\psi$ be a state on $A$. Let $\psi \0$ be the
corresponding state on $\Cstar(A)$. By \ref{predisint}, there exists a
unique measure $\mu \0 \in \probmeas \Bigl(\Delta\bigl(\Cstar(A)\bigr)\Bigr)$
such that
\[ \psi \0(b) = \int \wht{b} \,\diff \mu \0
\quad \text{for all } b \in \Cstar(A). \]
Let $\mu$ be the image measure of $\mu \0$ under the homeomorphism
$\tau \0 \mapsto \tau \0 \circ j$ from $\Delta\bigl(\Cstar(A)\bigr)$ onto
$\Delta^*\bsa(A)$, cf.\ \ref{tauhom}. (For the concept of image measure,
see the appendix \ref{imagebegin} - \ref{imageend}.) For $a \in A$, we get
\begin{align*}
\psi (a) & = \psi \0\bigl(j(a)\bigr) \\
& = \int \wht{j(a)} \,\diff \mu \0 \\
& = \int \wht{j(a)}(\tau \0) \,\diff \mu \0 (\tau \0) \\
& = \int \tau \0 \bigl(j(a)\bigr) \,\diff \mu \0 (\tau \0) \\
& = \int \wht{a} (\tau \0 \circ j) \,\diff \mu \0 (\tau \0) \\
& = \int \wht{a} (\tau ) \,\diff \mu (\tau ) = \int \wht{a} \,\diff \mu.
\end{align*}
Uniqueness of $\mu$ follows from the fact that the functions
$\wht{a}$ $(a \in A)$ are dense in $\cont\0\bigl(\Delta^*\bsa(A)\bigr)$,
cf.\ \ref{commdense}.

Conversely, let $\mu$ be a measure in $\probmeas \bigl(\Delta^*\bsa(A)\bigr)$
and consider the linear functional $\varphi$ on $A$ defined by
\[ \varphi (a) := \int \wht{a} \,\diff \mu \qquad (a \in A). \]
Then $\varphi$ is a positive linear functional on $A$ because for
$a \in A$, we have
\[ \varphi (a^*a) = \int {| \,\wht{a} \,| \,}^2 \,\diff \mu \geq 0. \pagebreak \]
The positive linear functional $\varphi$ is contractive on $A\sa$
because all the $\tau \in \Delta^*\bsa(A)$ are contractive on $A\sa$
as they are states. Indeed
\[ | \,\varphi (b) \,| \leq \int | \,\tau(b) \,| \,\diff \mu (\tau) \leq | \,b \,| 
\quad \text{for all} \quad b \in A\sa. \]
Furthermore, we have $v (\varphi ) \leq 1$ by the Cauchy-Schwarz inequality:
\[ {| \,\varphi (a) \,| \,}^2 = {\biggl| \,\int \wht{a} \,\diff \mu \,\biggr| \,}^2
\leq \int {| \,\wht{a} \,| \,}^2 \,\diff \mu \ \cdot \int {1 \,}^2 \,\diff \mu = \varphi (a^*a) \]
for all $a \in A$. Thus $\varphi = \lambda \psi$ with $\psi$ a state
on $A$ and $0 \leq \lambda \leq 1$. By the above, there exists a unique
measure $\nu \in \probmeas \bigl(\Delta^*\bsa(A)\bigr)$ such that
\[ \psi (a) = \int \wht{a} \,\diff \nu \quad \text{for all} \quad a \in A. \]
We obtain
\[ \varphi (a) = \int \wht{a} \,\diff (\lambda \nu ). \]
as well as
\[ \varphi (a) = \int \wht{a} \,\diff \mu. \]
By density of the functions $\wht{a}$ in
$\cont\0\bigl(\Delta^*\bsa(A)\bigr)$, it follows
$\mu = \lambda \nu$, whence $\lambda = 1$,
so that $\varphi = \psi$ is a state on $A$.
\end{proof}

\medskip
See also \ref{Bochnerbis} \& \ref{Bremark} below.

\begin{counterexample}\label{counterex}%
There exists a commutative unital nor\-med \st-algebra $A$ such that
every Hermitian multiplicative linear functional on $A$ is contractive
on $A\sa$, yet discontinuous, and thereby a discontinuous state.
\end{counterexample}

\begin{proof} We take $A := \mathds{C}[\mathds{Z}]$ as
\st-algebra, cf.\ \ref{CG}. (We shall define a norm later on.)
Let $\tau$ be a Hermitian multiplicative linear functional on $A$.
If we put $u := \tau(\delta _1)$, then
\[ \tau(a) = \sum _{n \in \mathds{Z}} \,a(n) \,u^{\,n}
\quad \text{for all} \quad a \in A. \]
With $e := \delta_0$ the unit in $A$, we have
\begin{align*}
{| \,u \,| \,}^2 & = u \,\overline{u} = \tau(\delta_1) \,\overline{\tau(\delta _1)}
= \tau(\delta_1) \,\tau({\delta _1}^*) \\
 & = \tau(\delta_1) \,\tau(\delta _{-1}) = \tau(\delta_1 \delta_{-1}) = \tau (e) = 1,
\pagebreak
\end{align*}
which makes that $u = \tau(\delta _1)$ belongs to the unit circle. The
mapping taking a Hermitian multiplicative linear functional $\tau$ on
$A$ to $u = \tau(\delta _1)$ is surjective onto the unit circle, because
if $u$ is in the unit circle then
\[ \tau(a) := \sum _{n \in \mathds{Z}} \,a(n) \,u^{\,n} \quad (a \in A)  \]
defines a Hermitian multiplicative linear functional on $A$ such that
$\tau(\delta _1) = u$. This mapping also is injective because a Hermitian
multiplicative linear functional on $A$ is uniquely determined by its value
on $\delta _1$. In other words, the Hermitian multiplicative linear functionals
$\tau$ on $A$ correspond bijectively to the points $u = \tau (\delta _1)$ in
the unit circle.

We shall now introduce a norm on $A$. Let $\gamma > 1$ be fixed. Then
\[ | \,a \,| := \sum _{n \in \mathds{Z}} \,| \,a(n) \,| \,{\gamma \,}^n \]
defines an algebra norm on $A$. (We leave the verification to the reader.)
If $a \in A\sa$, then $a(-n) = \overline{a(n)}$
for all $n \in \mathds{Z}$, so for $a \in A\sa$ we have
\[ | \,a \,| = | \,a(0) \,| + \sum _{n \geq 1} \,| \,a(n) \,|
\,\Bigl( {\gamma \,}^n + \frac{1}{{\gamma \,}^n} \Bigr). \]
Let $\tau$ be a Hermitian multiplicative linear functional on $A$.
Then $\tau(\delta _n) = \bigl(\tau(\delta _1)\bigr)^n$ is in the unit circle
for all $n \in \mathds{Z}$, so for $a \in A\sa$ we find
\[ | \,\tau(a) \,| \leq | \,a(0) \,| + 2 \,\sum _{n \geq 1} \,| \,a(n) \,| . \]
We see that if $a \in A\sa$, then $| \,\tau(a) \,| \leq | \,a \,|$, so $\tau$
is contractive on $A\sa$. However
\begin{alignat*}{2}
 & | \,\delta _n \,|  = {\gamma \,}^n \to 0 & &
\quad \text{for } n \to -\infty, \\
 & | \,\tau(\delta _n) \,|  = 1 & &
\quad \text{for all } n \in \mathds{Z},
\end{alignat*}
which shows that $\tau$ is discontinuous.
\end{proof}

\begin{remark}\label{counterexbis}
This gives an example of a commutative normed \linebreak \st-algebra
which cannot be imbedded in a Banach \st-algebra, see
\ref{bdedstatessuffcond} (iii). The above normed \st-algebra also serves
as a counterexample for the statements \ref{Brlsymm} \& \ref{Bnormalrsleqrl}
with normed \st-algebras instead of Banach \st-algebras. Please
note that the above commutative unital normed \linebreak \st-algebra
does not carry any continuous Hermitian multiplicative linear functional.
\pagebreak
\end{remark}

\clearpage


\addtocontents{toc}{\protect\vspace{0.2em}}

\chapter{Irreducible Representations}

\setcounter{section}{34}


\section{General \texorpdfstring{$*$-}{\80\052\80\055}Algebras: Indecomposable Functionals}

\medskip
Let throughout $A$ be a \st-algebra and $H \neq \{0\}$ a Hilbert space.

\begin{definition}[irreducible representations]%
\index{concepts}{representation!irreducible}\index{concepts}{irreducible}%
A representation of $A$ on $H$ is called \underline{irreducible}
if it is non-zero and if it has no closed reducing subspaces other
than $\{0\}$ and $H$. We shall then also say that the \st-algebra
$A$ \underline{acts irreducibly} on $H$ via this representation.
\end{definition}

\begin{proposition}\label{irredcyclic}%
A non-zero representation $\pi$ of $A$ on $H$ is irreducible
if and only if every non-zero vector in $H$ is cyclic for $\pi$.
\end{proposition}

\begin{proof} Assume that $\pi$ is irreducible. Then $\pi$ must be
non-degenerate, cf.\ \ref{nondegenerate}. Let $x$ be a non-zero vector
in $H$ and consider $M = \overline{\range(\pi)x}$. Then $x \in M$,
because $\pi$ is non-degenerate, cf.\ \ref{cyclicsubspace}. In particular,
$M$ is a closed reducing subspace $\neq \{0\}$, so that $M = H$.
Conversely, any vector in a proper closed invariant subspace cannot be cyclic.
\end{proof}

\begin{lemma}[Schur's Lemma]%
\index{concepts}{Theorem!Schur's Lemma}%
\index{concepts}{Schur's Lemma}%
\index{concepts}{Lemma!Schur's}\label{Schur}%
A non-zero representation $\pi$ of $A$ on $H$
is irreducible if and only if $\pi' = \mathds{C}\mathds{1}$.
\end{lemma}

\begin{proof} Assume that $\pi' = \mathds{C}\mathds{1}$. Let $M$
be a closed reducing subspace of $\pi$ and let $p$ be the projection on
$M$. Then $p$ belongs to $\pi'$ by \ref{invarcommutant}, whence either
$p = 0$ or $p = \mathds{1}$. Assume conversely that $\pi$ is irreducible.
Then $\pi'$ is a unital C*-subalgebra of $\blop(H)$. Consider a Hermitian
element $a$ of $\pi'$ and let $C$ be the C*-subalgebra of $\pi'$
generated by $a$ and $\mathds{1}$. Please note that $C$ is commutative.
In order to prove the lemma, it suffices to show that $C$ is one-dimensional.
So assume that $C$ has dimension at least $2$. Then $C$ then contains
divisors of zero, i.e.\ non-zero elements $c, d$ with $c d = 0$, cf.\ \ref{C*zerodiv}.
Choosing $x \in H$ with $d x \neq 0$, we reach the contradiction
$c H = c \overline{\pi (A) d x} \subset \overline{\pi (A) c d x} = \{0\}$
as the vector $d x \neq 0$ is cyclic by the preceding proposition
\ref{irredcyclic}. \pagebreak
\end{proof}

\begin{theorem}\label{commirroned}%
If $A$ is commutative, then a non-zero representation $\pi$ of
$A$ on $H$ is irreducible if and only if $H$ is one-dimensional.
\end{theorem}

\begin{proof}
Since $A$ is commutative, we have $\range(\pi) \subset \pi'$.
Thus, if $\pi$ is irreducible, then $\range(\pi) \subset \pi' = \mathds{C}\mathds{1}$
by the preceding Lemma of Schur \ref{Schur}. Then every subspace of
$H$ is invariant under $\pi$, so $H$ must be one-dimensional.
\end{proof}

\medskip
An irreducible representation of a non-commutative \st-algebra
need not be one-dimensional. Indeed:

\begin{proposition}%
The C*-algebra $\blop(H)$ acts irreducibly on $H$.
\end{proposition}

\begin{proof}
This follows from \ref{BHirred} via Schur's Lemma \ref{Schur}.
\end{proof}

\begin{proposition}\label{spherical}%
Let $\pi$ be an irreducible representation of $A$ on $H$.
Assume that $x, y$ are vectors in $H$ with
\[ \langle \pi (a) x, x \rangle = \langle \pi (a) y, y \rangle
\quad \text{for all} \quad a \in A. \]
There then exists a complex number $u$ of modulus $1$ with $y = ux$.
\end{proposition}

\begin{proof} Assume first that both $x$ and $y$ are non-zero. Then $x, y$
are cyclic vectors for $\pi$, by \ref{irredcyclic}. So the assumption implies that
there exists a unitary operator $U \in \pi' = C(\pi,\pi)$ taking $x$ to $y$,
cf.\ \ref{coeffequal}. Since $\pi$ is irreducible, it follows from Schur's Lemma
\ref{Schur} that $U = u\mathds{1}$ with $u \in \mathds{C}, | \,u \,| = 1$. Assume
next that $x = 0$, $y \neq 0$. The equation
\[ 0=\langle \pi (a)x,x \rangle =\langle\pi (a)y,y \rangle \qquad ( \,a \in A \,) \]
then would imply that the vector $y \neq 0$ was orthogonal to all the vectors
$\pi (a)y$ $(a \in A)$, so that $y$ could not be a cyclic vector, in contradiction
with \ref{irredcyclic}.
\end{proof}

\begin{lemma}\label{posoprep}%
Let $\varphi$ be a positive semidefinite sesquilinear form on
$H$ with $\sup _{\,\| \,x \,\| \leq 1} \varphi (x, x) < \infty$. There
then exists a unique linear operator $a \in \blop(H)_+$ such that
\[ \varphi (x, y) = \langle ax, y \rangle\quad ( \,x,y \in H \,). \]
\end{lemma}

\begin{proof} We have $\sup _{\,\| \,x \,\|, \| \,y \,\| \leq 1} | \,\varphi (x,y) \,| < \infty$
by the Cauchy-Schwarz \linebreak inequality. It follows from a well-known
theorem that there exists a unique bounded linear operator $a$ on $H$
with $\langle ax,y \rangle = \varphi (x,y)$ for all $x,y \in H$. The operator
$a$ is positive because $\varphi$ is positive semidefinite, cf.\ \ref{posop}.
\pagebreak
\end{proof}

\medskip
Our next aim is theorem \ref{indecequiv} below, which leads to
theorem \ref{irredindecpure}.

\begin{definition}[subordination]%
\index{concepts}{subordinate}\index{concepts}{functional!subordinate}%
Let $\varphi_1, \varphi_2$ be positive linear \linebreak functionals on $A$.
One says that $\varphi_1$ is \underline{subordinate} to $\varphi_2$ if
there is $\lambda \geq 0$ such that $\lambda \varphi_2 - \varphi_1$ is
a positive linear functional.
\end{definition}

\begin{theorem}\label{subordrep}%
Let $\pi$ be a cyclic representation of $A$ on $H$ and let $c$ be a cyclic
vector for $\pi$. Consider the positive linear functional $\varphi$ on $A$
defined by
\[ \varphi (a) := \langle \pi (a) c, c \rangle\quad ( \,a \in A \,). \]
The equation
\[ \varphi_1 (a) = \langle b \pi(a) c, c \rangle\quad ( \,a \in A \,) \]
establishes a bijection between operators $b \in {\pi'}_+$
and Hermitian positive linear functionals $\varphi_1$ on $A$
of finite variation that are subordinate to $\varphi$.
\end{theorem}

\begin{proof}
Let $b \in {\pi'}_+$. It is immediate that
\[ \varphi_1 (a) := \langle b \pi (a)c,c \rangle
= \langle \pi(a) b^{1/2} c, b^{1/2} c \rangle \quad ( \,a \in A \,) \]
defines a Hermitian positive linear functional of finite variation on $A$,
cf.\ \ref{C*sqrootrestate} and \ref{variationinequal}. To show that $\varphi_1$
is subordinate to $\varphi$, we note that $b \leq \| \,b \,\| \,\mathds{1}$, whence
\[ \| \,b \,\| \,\langle x,x \rangle - \langle bx,x \rangle \geq 0 \]
for all $x \in H$. With $x := \pi (a) c$ it follows that
\[ \| \,b \,\| \,\langle \pi (a^*a)c,c \rangle
- \langle b \pi (a^*a)c,c \rangle \geq 0, \]
which says that $\| \,b \,\| \,\varphi - \varphi_1$ is a positive linear
functional on $A$.

Conversely, let $\varphi_1$ be a Hermitian positive linear functional
of finite variation subordinate to $\varphi$. There then exists
$\lambda \geq 0$ such that
\[ 0 \leq \varphi_1 (a^*a) \leq \lambda \,\varphi (a^*a)  \tag*{(\st)} \]
for all $a \in A$.

For $x = \pi (f)c, y = \pi (g)c$, $(f,g \in A)$, we put
\[ \alpha (x,y) := \varphi_1 (g^*f). \]
This expression depends only on $x,y$. Indeed, let for example also
$x = \pi (h)c$, i.e.\ $\pi (f-h)c = 0$. We obtain
\[ \varphi \,\bigl((f-h)^*(f-h)\bigr) = \langle \pi (f-h)c, \pi (f-h)c \rangle = 0, \]
so that by (\st),
\[ \varphi_1 \bigl((f-h)^*(f-h)\bigr) = 0. \pagebreak \]
The Cauchy-Schwarz inequality now implies that
\[ \varphi_1 \bigl(g^*(f-h)\bigr) = 0,  \]
whence
\[ \varphi_1 (g^*f) = \varphi_1 (g^*h), \]
as claimed. Denote by $H\0$ the set of all vectors $x \in H$ of the form
$x = \pi (f)c$ $(f \in A)$. The vector space $H\0$ is dense in $H$ because
$c$ is cyclic for $\pi$. The positive semidefinite sesquilinear form $\alpha$
is bounded on $H \0$ since (\st) may be rewritten as
\[ \alpha (x, x) \leq \lambda \,\langle x, x \rangle \tag*{(\st\st)} \]
for all $x \in H\0$. Therefore $\alpha$ has a unique continuation to a bounded
positive semidefinite sesquilinear form on $H$. The inequality (\st\st) then
is valid for all $x \in H$. By \ref{posoprep} there is a unique operator
$b \in \blop(H)_+$ such that
\[ \alpha (x,y) = \langle bx,y \rangle
\quad \text{for all} \quad x, y \in H. \]
It shall be shown that $b \in \pi'$. For $x = \pi (f)c, y = \pi (g)c$, and $a \in A$,
we have
\[ \langle b \pi (a)x,y \rangle = \langle b \pi (af)c, \pi (g)c \rangle
= \alpha \bigl(\pi(af)c, \pi(g)c\bigr) = \varphi_1 (g^*af), \]
as well as
\begin{align*}
\langle \pi (a)bx,y \rangle & = \langle bx,\pi (a^*)y \rangle
= \langle b \pi (f)c,\pi (a^*g)c \rangle = \alpha \bigl(\pi(f)c, \pi(a^*g)c\bigr) \\
 & = \varphi_1 \bigl((a^*g)^*f\bigr) = \varphi_1 (g^*af),
\end{align*}
so that
\[ \langle b \pi (a)x,y \rangle = \langle \pi (a)bx,y \rangle \]
for all $x,y \in H\0$, and thus for all $x,y \in H$, whence $b \in \pi'$.
In order to prove that
\[ \varphi_1 (a) = \langle b \pi (a)c,c \rangle
\quad \text{for all} \quad a \in A, \]
one first notices that
\[ {| \,\varphi_1 (a) \,| \,}^2 \leq v(\varphi_1) \,\varphi_1(a^*a)
\leq v(\varphi_1) \,\lambda \,\varphi (a^*a)
= \lambda \,v(\varphi_1) \,{\| \,\pi (a)c \,\| \,}^2, \]
so that $\pi (a)c \mapsto \varphi_1 (a)$ is a bounded linear functional on
$H\0$, hence extends to a bounded linear functional on $H$. (The expression
$\varphi_1 (a)$ depends only on $\pi (a)c$ as is seen as follows: if
$\pi (f)c = \pi (g)c$, we have
\[ {| \,\varphi_1 (f) - \varphi_1 (g) \,| \,}^2
= {| \,\varphi_1 (f-g) \,| \,}^2 \leq \lambda \,v(\varphi_1 ) \,{\| \,\pi (f-g)c \,\| \,}^2
= 0.) \]
There thus exists a vector $d \in H$ with
\[ \varphi_1 (a) = \langle \pi (a)c,d \rangle
\quad \text{for all} \quad a \in A. \pagebreak \]
For $f, g \in A$, we get
\begin{align*}
\langle \pi (f)c,b \pi (g)c \rangle & = \langle b \pi (f)c, \pi (g)c \rangle
= \alpha (\pi (f)c,\pi (g)c) \\
 & = \varphi_1 (g^*f) = \langle \pi (g^*f)c,d \rangle
= \langle \pi (f)c, \pi (g) d \rangle,
\end{align*}
so that
\[ b \pi (g)c = \pi (g)d \quad \text{for all} \quad g \in A. \]
For $a \in A$, it follows that
\[ \langle b \pi (a)c,c \rangle = \langle \pi (a)d,c \rangle
= \overline{\langle \pi (a^*)c,d \rangle}
= \overline{\varphi_1(a^*)} = \varphi_1(a). \]
Uniqueness of $b \in \pi'$ satisfying
\[ \varphi_1(a) = \langle b \pi (a) c,c \rangle \quad \text{for all} \quad a \in A \]
follows from the identity
\[ \varphi_1 (g^*f) = \langle b \pi (f) c,\pi (g) c \rangle
\quad \text{for all} \quad f, g \in A, \]
and the fact that $H\0$ is dense in $H$. 
\end{proof}

\begin{definition}[indecomposable positive linear functionals]%
\index{concepts}{indecomposable}\index{concepts}{functional!indecomposable}%
Let $\varphi$ be a Hermitian positive linear functional
of finite variation on $A$. The functional $\varphi$ is called
\underline{indecomposable} if every other Hermitian positive
linear functional of finite variation on $A$, which is
subordinate to $\varphi$, is a multiple of $\varphi$.
\end{definition}

\begin{theorem}\label{indecequiv}%
Let $\pi$ be a cyclic representation of $A$ on $H$ and let $c$ be a cyclic
vector for $\pi$. Consider the positive linear functional $\varphi$ on $A$
given by
\[ \varphi (a) := \langle \pi (a)c,c \rangle\quad ( \,a \in A \,). \]
Then $\pi$ is irreducible if and only if $\varphi$ is indecomposable.
\end{theorem}

\begin{proof} Assume that $\pi$ is irreducible and let $\varphi_1$ be
a Hermitian positive linear functional of finite variation subordinate to
$\varphi$. There exists $b \in {\pi'}_+$ with
$\varphi_1 (a) = \langle b \pi (a) c,c \rangle$ for all $a \in A$. Since
$\pi$ is irreducible, it follows that $b = \lambda \mathds{1}$ for some
$\lambda \geq 0$, so that $\varphi_1 = \lambda \varphi$. Conversely,
let $\varphi$ be indecomposable and let $p$ denote the projection
on a closed invariant subspace of $\pi$. We then have $p \in {\pi'}_+$,
cf.\ \ref{invarcommutant}, so that the Hermitian positive linear functional
$\varphi_1$ of finite variation defined by
$\varphi_1(a) := \langle p \pi (a)c,c \rangle$ $(a \in A)$, is subordinate
to $\varphi$, whence $\varphi_1 = \lambda \varphi$ for some
$\lambda \in \mathds{R}$. Since $p$, as an element of ${\pi'}_+$,
is uniquely determined by $\varphi_1$, it follows that
$p = \lambda \mathds{1}$, but then $\lambda = 0$ or $1$ as $p$
is a projection, and so $\pi$ is irreducible. \pagebreak
\end{proof}

\clearpage


\section{Normed \texorpdfstring{$*$-}{\80\052\80\055}Algebras: Pure States}

Throughout this paragraph, let $A$ be a normed \st-algebra,
and \linebreak assume that $\psi$ is a state on $A$.

\begin{definition}[pure states, $\purestates (A)$]%
\index{concepts}{pure state}\index{concepts}{state!pure}%
\index{symbols}{P9@$\purestates (A)$}%
\index{concepts}{functional!state!pure}%
One says that $\psi$ is a \linebreak \underline{pure state} on $A$
if it is an extreme point of the convex set $\statespace (A)$ of states
on $A$, cf.\ \ref{SAconvex}. The set of pure states on $A$ is denoted
by $\purestates (A)$.
\end{definition}

\begin{observation}\label{homPS}%
The affine homeomorphism $\varphi \0 \mapsto \varphi \0 \circ j$
from $\statespace \bigl(\Cstar(A)\bigr)$ onto $\statespace (A)$,
cf.\ \ref{affinehomeom}, restricts to a homeomorphism from
$\purestates \bigl(\Cstar(A)\bigr)$ onto $\purestates (A)$.
\end{observation}

\begin{proof}
The homeomorphism $\varphi \0 \mapsto \varphi \0 \circ j$
from $\statespace \bigl(\Cstar(A)\bigr)$
onto $\statespace (A)$ is affine, cf.\ \ref{affinehomeom}.
\end{proof}

\begin{lemma}\label{subordo}%
Let $\psi \0$ be the state on $\Cstar(A)$ corresponding to $\psi$.
Then $\psi$ is indecomposable precisely when $\psi \0$ is so.
\end{lemma}

\begin{proof}
This is an application of theorem \ref{indecequiv}.
The representation $\pi _{(\psi \0)}$ is spatially equivalent to
$(\pi _{\psi}) \0$, cf.\ \ref{oequiv}. Therefore $\pi _{(\psi \0)}$ is
irreducible precisely when $(\pi _{\psi}) \0$ is. Now $(\pi _{\psi}) \0$
is irreducible if and only if $\pi _{\psi}$ is irreducible, as is shown
by \ref{piopi} (iii).
\end{proof}

\begin{theorem}\label{indecopure}%
The state $\psi$ is indecomposable if and only if it is a pure state.
\end{theorem}

\begin{proof} We shall first prove this if $A$ has continuous
involution and a bounded left approximate unit. Let the state
$\psi$ be indecomposable. Assume that
$\psi = \lambda _1 \psi _1 + \lambda _2 \psi _2$ where
$\lambda _1, \lambda _2 > 0, \lambda _1 + \lambda _2 = 1$ and
$\psi _1, \psi _2 \in \statespace (A)$. Since $\lambda _k \psi _k$ is subordinate
to $\psi$, we have $\lambda _k \psi _k = \mu _k \psi$ for some $\mu _k \geq 0$.
We obtain $\lambda _k = v(\lambda _k \psi _k) = v(\mu _k \psi ) = \mu _k$
so that $\psi _k = \psi$ for $k = 1,2$. Therefore $\psi$ is an extreme point
of $\statespace (A)$. Let next the state $\psi$ be pure. Consider a Hermitian
positive linear functional $\psi _1$ of finite variation on $A$ subordinate
to $\psi$. Then $\psi _1$ is bounded on $A\sa$ by applying
\ref{subordrep} to $\pi_{\psi}$. But then $\psi _1$ and $\psi$ are bounded
as the involution in $A$ is continuous by assumption.
It must be shown that $\psi _1$ is a multiple of $\psi$. Assume without loss
of generality that $\psi _1 \in \statespace (A)$. There then exists $\lambda > 0$
such that $\psi - \lambda \psi _1 =: \psi'$ is a positive linear functional. This
functional is continuous and thus has finite variation because by assumption
$A$ has continuous involution and a bounded left approximate unit,
cf.\ \ref{bdedfinvar}. It must be shown that $\psi _1$ is a multiple of $\psi$.
Assume that $\psi' \neq 0$. Put $\psi _2 := \psi'/\mu$ where $\mu := v(\psi') > 0$
as $\psi' \neq 0$. Then $\psi _2$ is a state on $A$ with
$\psi = \lambda \psi _1 + \mu \psi _2$. It follows that $1 = \lambda + \mu$
by additivity of the variation, cf.\ \ref{finvarconvex}. Since $\lambda, \mu > 0$,
it follows that $\psi _1 = \psi _2 = \psi$ because $\psi$ is an extreme point
of $\statespace (A)$, and in particular $\psi _1$ is a multiple of  $\psi$.
This proves the theorem in presence of a continuous involution and a
bounded left approximate unit. We shall pull down the general case from
$\Cstar(A)$. On one hand, $\psi$ is a pure state precisely when $\psi \0$
is a pure state, cf.\ \ref{homPS}. On the other hand, $\psi$ is indecomposable
if and only if $\psi \0$ is indecomposable, by \ref{subordo}.
\end{proof}

\begin{theorem}\label{irredindecpure}%
The following conditions are equivalent.
\begin{itemize}
   \item[$(i)$] $\pi _{\psi}$ is irreducible,
  \item[$(ii)$] $\psi$ is indecomposable,
 \item[$(iii)$] $\psi$ is a pure state.
\end{itemize}
\end{theorem}

\begin{proof}
\ref{indecequiv} and \ref{indecopure}.
\end{proof}

\begin{theorem}\label{cPSD}%
If $A$ is \underline{commutative}, then $\purestates (A) = \Delta^*\bsa(A)$.
\end{theorem}

\begin{proof} \ref{commirroned} and \ref{mlfHbsa}. \end{proof}

\begin{theorem}%
The set of pure states on $A$ parametrises the class of irreducible
$\sigma$-contractive representations of $A$ on Hilbert spaces up to
spatial equivalence. More precisely, if $\psi$ is a pure state on $A$,
then $\pi _{\psi}$ is an irreducible $\sigma$-contractive representation
of $A$. Conversely, if $\pi$ is an irreducible $\sigma$-contractive
representation of $A$ on a Hilbert space, there exists a pure state $\psi$
on $A$ such that $\pi$ is spatially equivalent to $\pi _{\psi}$.
(Cf.\ \ref{stateratio}.)
\end{theorem}

\begin{observation}[non-uniqueness]%
If $\pi$ is any multi-dimensional irreducible $\sigma$-contractive
representation of some normed \st-algebra (necessarily non-commutative),
there exists a continuum of pure states $\varphi$
such that $\pi$ is spatially equivalent to $\pi _{\varphi}$.
Indeed, this can be seen from proposition \ref{spherical}.
\end{observation}

\begin{proposition}\label{extrquasi}%
The set of extreme points of the non-empty compact convex set $\quasistates (A)$
is $\purestates (A) \cup \{0\}$. Cf.\ \ref{SAconvex} \& \ref{topQS}. \pagebreak
\end{proposition}

\begin{proof}
It shall first be shown that $0$ is an extreme point of $\quasistates (A)$. So
suppose that $0 = \lambda \,\varphi _1 +(1-\lambda) \,\varphi _2$
with $0 < \lambda < 1$ and $\varphi _1, \varphi _2 \in \quasistates (A)$. For
$a \in A$, we then have $\varphi _k (a^*a) = 0$ by $\varphi _k (a^*a) \geq 0$,
and so $\varphi _k (a) = 0$ by
${| \,\varphi _k (a) \,|}^{\,2} \leq v(\varphi _k) \,\varphi _k (a^*a)$.
It shall next be shown that each $\psi \in \purestates (A)$ is an
extreme point of $\quasistates (A)$. So let $\psi \in \purestates (A)$,
$\psi = \lambda \,\varphi _1 +(1-\lambda ) \,\varphi _2$, where
$0 < \lambda < 1$ and $\varphi _1$, $\varphi _2 \in \quasistates (A)$. By
additivity of the variation \ref{finvarconvex}, we have
$1 = \lambda \,v(\varphi _1)+( 1 - \lambda ) \,v(\varphi _2)$,
whence $v(\varphi _1) = v(\varphi _2) = 1$, so that $\varphi _1$
and $\varphi _2$ are states, from which it follows that
$\psi = \varphi _1 = \varphi _2$. It remains to be shown that
$\varphi \in \quasistates (A)$ with $0 < v(\varphi ) < 1$ is not an extreme point of
$\quasistates (A)$. This is so because then
$\varphi = v(\varphi) \cdot \bigl(v(\varphi)^{\,-1\,} \varphi \bigr)
+ \bigl(1-v(\varphi)\bigr) \cdot 0$. 
\end{proof}

\begin{corollary}%
The non-empty compact convex set $\quasistates (A)$ is the  closed convex hull of
$\purestates (A) \cup \{0\}$.
\end{corollary}

\begin{proof}
The Kre\u{\i}n-Milman Theorem says that a non-empty compact convex set
in a separated locally convex topological vector space is the closed convex
hull of its extreme points.
\end{proof}

\begin{corollary}\label{existPS}%
If a normed \st-algebra has a state, then it has a pure state.
\end{corollary}

\begin{theorem}%
The state $\psi$ is in the closed convex hull of $\purestates (A)$.
\end{theorem}

\begin{proof} The state $\psi$ is in the closed convex hull of
$\purestates (A) \cup \{0\}$, so that $\psi$ is the limit of a net
$(\varphi _i)_{i \in I}$ where each $\varphi _i$ $(i \in I)$ is of the form
\[ \lambda \cdot 0 + \sum _{j=1}^{n} \,\lambda _j \cdot \psi _j \]
with
\[ \lambda + \sum _{j=1}^{n} \,\lambda _j =1,
\ \lambda \geq 0, \ \lambda _j \geq 0,
\ \psi _j \in \purestates (A)\ (j = 1, \ldots ,n). \]
From \ref{finvarconvex} we get
\[ v(\varphi _i) = \sum _{j=1}^{n} \,\lambda _j \leq 1, \]
so $\limsup v(\varphi _i) \leq 1$. We shall show that
$\liminf v(\varphi _i) \geq 1$, from which $\lim v(\varphi _i) = 1$.
Assume that $\alpha := \liminf v(\varphi _i) < 1$. 
By going over to a subnet, we can then assume that
$\alpha = \lim v(\varphi _i)$. Now we use the fact that the topology on
$\quasistates (A)$ is the topology of pointwise convergence on $A\sa$,
cf.\ the appendix \ref{weak*point}, and the fact that $A$ is spanned by
$A\sa$. For all $a \in A$, one would have
\[ {| \,\psi (a) \,|}^{\,2} = \lim {| \,\varphi _i (a) \,|}^{\,2}
\leq \lim v(\varphi _i) \,\varphi _i(a^*a) = \alpha \,\psi (a^*a), \]
which would imply $v(\psi ) \leq \alpha < 1$, a contradiction. We have
thus shown that $\lim v(\varphi _i) = 1$. In particular, by going over to
a subnet, we may assume that $v(\varphi _i) \neq 0$ for all $i \in I$.
The net $\bigl( v(\varphi _i)^{\,-1\,} \varphi _i \bigr) _{i \in I}$ then
converges to $\psi$, and consists of convex combinations of pure states.
\end{proof}

\begin{definition}[the spectrum, $\wht{A}$]%
\index{concepts}{spectrum!of an algebra}\index{symbols}{AA@$\wht{A}$}%
We shall say that two states on $A$ are spatially equivalent if their
corresponding GNS representations are spatially equivalent. This
defines an equivalence relation on $\statespace (A)$. We define
the \underline{spectrum} $\wht{A}$ of $A$ as the set of spatial
equivalence classes of $\purestates (A)$.
\end{definition}

Please note that if $A$ is commutative, then $\wht{A}$ may be
identified with $\Delta^*\bsa(A)$, cf.\ \ref{cPSD}. 

\begin{definition}%
We shall say that
$(\psi _{\text{\small{$\lambda$}}}) _{\text{\small{$\lambda \in \wht{A}$}}}$
is a choice function if $\psi _{\text{\small{$\lambda$}}} \in \lambda$ for each
$\lambda \in \wht{A}$.
There exists$\vphantom{ _{\text{\small{$\lambda\in \wht{A}$}}}}$
a choice function
$(\psi _{\text{\small{$\lambda$}}}) _{\text{\small{$\lambda \in \wht{A}$}}}$
by the axiom of choice$\vphantom{\tld{A}}$ (cf.\ \ref{existPS}).
\end{definition}

\begin{definition}[the reduced atomic representation, $\pi \ra$]%
\index{symbols}{p5@$\pi_{ra}$}\index{symbols}{H3@$H_{ra}$}%
\index{concepts}{representation!reduced atomic}%
If \linebreak
$(\psi _{\text{\small{$\lambda$}}}) _{\text{\small{$\lambda \in \wht{A}$}}}$
is a choice function, we define
\[ \pi \ra := \oplus \,_{\text{\small{$\lambda \in \wht{A}$}}}
\ \pi _{\text{\small{$\psi _{\text{\small{$\lambda$}}}$}}}. \]
One says that $\pi \ra$ is a reduced atomic representation of $A$. Any two
reduced atomic representations of $A$ are spatially equivalent, so that
one speaks of ``the'' \underline{reduced atomic representation} of $A$.
\end{definition}

Our next aim is theorem \ref{ranorm}.

\begin{theorem}%
If $A$ is a Hermitian Banach \st-algebra, and if $B$ is a closed \st-subalgebra
of $A$, then every pure state on $B$ can be extended to a pure state on $A$.
\pagebreak
\end{theorem}

\begin{proof}
Let $\varphi$ be a pure state on $B$. Let $K$ be the set of states on $A$ which
extend $\varphi$, cf.\ \ref{extstate}. Then $K$ is a non-empty compact convex
subset of $\statespace (A)$. (Indeed every adherent point of $K$ in the compact
set $\quasistates(A)$ extends $\varphi$, and thus is a state, too.) It follows that
$K$ has an extreme point, $\psi$, say. It shall be shown that  $\psi$ is a pure state.
So let $\psi = \lambda \psi _1+(1-\lambda )\psi _2$ with
$0 < \lambda < 1$ and $\psi _1, \psi _2 \in \statespace (A)$. We then also have
$\varphi = \psi |_B = \lambda \psi _1|_B + (1-\lambda ) \psi _2|_B$, so that
$\psi _1|_B = \psi _2|_B = \psi |_B = \varphi$, by \ref{extrquasi}. It follows that
$\psi _1, \psi _2$ belong to $K$. As $\psi$ is an extreme point of $K$, one has
$\psi _1 = \psi _2$, and so $\psi$ is a pure state.
\end{proof}

\begin{theorem}%
If $A$ is a Hermitian Banach \st-algebra and if $a$ is a Hermitian element
of $A$, then the following properties are equivalent.
\begin{itemize}
   \item[$(i)$] $a \geq 0$,
  \item[$(ii)$] $\pi \ra (a) \geq 0$,
 \item[$(iii)$] $\psi (a) \geq 0$ \quad for all \quad $\psi \in \purestates (A)$.
\end{itemize}
\end{theorem}

\begin{proof}
This follows in the same way as \ref{posequiv}.
\end{proof}

\begin{theorem}\label{ranorm}%
For all $a \in A$ we have
\[ \| \,\pi \ra (a) \,\| = \| \,a \,\|\0. \]
\end{theorem}

\begin{proof} Assume first that $A$ is a Hermitian Banach \st-algebra.
As in \ref{Raikov}, it follows that $\| \,\pi \ra (a) \,\| = \rsigma(a)$.
If $A$ is a C*-algebra, we therefore get $\| \,\pi \ra (a) \,\| = \| \,a \,\|$, cf.\ \ref{C*rs}.
In case $A$ is a general normed \st-algebra, we pull down the
result from the enveloping C*-algebra. For $a \in A$, we have
\[ \| \,a \,\|\0 = \| \,j(a) \,\| = \| \,\pi \ra \bigl(j(a)\bigr) \,\| = \| \,\pi \ra (a) \,\|, \]
where we have used \ref{homPS}.
\end{proof}

\begin{corollary}%
The following statements are equivalent.
\begin{itemize}
   \item[$(i)$] $A$ is \st-semisimple,
  \item[$(ii)$] $\pi\ra$ is faithful,
 \item[$(iii)$] the irreducible $\sigma$-contractive representations of $A$ on Hilbert
                       spaces separate the points of $A$.
\end{itemize}
\end{corollary}

\begin{proof} The proof is left to the reader. \pagebreak \end{proof}

\clearpage


\section{The Left Spectrum in a Hermitian Banach \texorpdfstring{$*$-}{\80\052\80\055}Algebra}%
\label{leftspectrum}

In this paragraph, let $A$ be a \underline{Hermitian} Banach \st-algebra.

\begin{proposition}\label{leftidstate}%
If $A$ is unital, then for a proper left ideal $I$ in $A$, there exists a state
$\varphi$ on $A$ such that
\[ \varphi(a^*a) = 0 \quad \text{for all} \quad a \in I. \]
\end{proposition}

\begin{proof}
Let $M\0$ denote the real subspace of $A\sa$ spanned by $e$ and the
Hermitian elements of $I$. Define a linear functional $f\0$ on $M\0$ by
\[ f\0 (\lambda e + y) := \lambda \qquad \bigl( \,\lambda \in \mathds{R},\ y \in I\sa \,\bigr). \]
Please note here that $e \notin I$, cf.\ \ref{notinvideal} (i). Let $C$ denote the
convex cone $A_+$ in $A\sa$, cf.\ \ref{plusconvexcone}. We have $M\0 + C = A\sa$
as for $a \in A\sa$ one has $\rlambda(a)e+a \geq 0$. Furthermore $f\0(x) \geq 0$
for all $x \in M\0 \cap C$. Indeed, if $x = \lambda e + y \in M\0 \cap C$ with $y \in I\sa$,
then $\lambda e - x = -y$ is a Hermitian element of $I$, which is a proper left ideal,
so that $\lambda e -x$ is not left invertible by \ref{notinvideal} (ii), which implies
that $f\0 (x) = \lambda \in \s(x) \subset [ \,0, \infty \,[$. It follows from \ref{Krein} that $f\0$
has a linear extension $f$ to $A\sa$ such that $f(x) \geq 0$ for all $x \in C = A_+$.
Then $f(a^*a) \geq 0$ for all $a \in A$ by the Shirali-Ford Theorem \ref{ShiraliFord}.
The linear extension $\varphi$ of $f$ from $A\sa$ to $A$ satisfies the requirement,
cf.\ \ref{eBbded}.
\end{proof}

\begin{proposition}\label{ispropleft}%
If $\varphi$ is a state on a unital normed  \st-algebra $B$, then
\[ M := \{ \,b \in B : \varphi(b^*b) = 0 \,\} \]
is a proper left ideal of $B$.
\end{proposition}

\begin{proof}
Let $\langle \cdot , \cdot \rangle _{\varphi}$ denote the positive Hilbert
form induced by $\varphi$, cf.\ \ref{inducedHf}. An element $b$ of $B$
is in $M$ if and only if $\langle b , c \rangle _{\varphi} = 0$ for all $c \in B$
by the Cauchy-Schwarz inequality. So $M$ is the isotropic subspace
of $\langle \cdot , \cdot \rangle _{\varphi}$, and thus a left ideal in $B$,
cf.\ \ref{quotientA}. The left ideal $M$ is proper as
$\varphi(e^*e) = \varphi(e) = 1$, cf.\ \ref{vare}.
\end{proof}

\begin{proposition}\label{maxleftstate}%
If $A$ is unital, then for a maximal left ideal $M$ in $A$ there exists a state
$\varphi$ on $A$ such that
\[ M = \{ \,a \in A : \varphi(a^*a) = 0 \,\}. \]
\end{proposition}

\begin{proof}
There is a state $\varphi$ on $A$ such that $\varphi(a^*a) = 0$ for all
$a \in M$, cf.\ \ref{leftidstate}. The set $\{ \,a \in A : \varphi(a^*a) = 0 \,\}$
is a proper left ideal in $A$ containing $M$, cf.\ \ref{ispropleft} and thus
equal to $M$ by maximality of $M$.\pagebreak%
\end{proof}

\medskip
See also \ref{maxleftpurestate} below.

\begin{proposition}\label{invstate}%
If $A$ is unital, then an element $a$ of $A$ is left invertible if and only if
\[ \varphi(a^*a) > 0 \quad \text{for all} \quad \varphi \in \statespace(A). \]
\end{proposition}

\begin{proof}
Let $a \in A$. If $a$ is not left invertible, then $a$ lies in a proper
left ideal by \ref{notinvideal} (ii), which is contained in a maximal left
ideal by \ref{inmaxideal}, so that $\varphi(a^*a) = 0$ for some state
$\varphi$ on $A$ by proposition \ref{maxleftstate}. Conversely, if
$\varphi(a^*a) = 0$ for some state $\varphi$ on $A$, then $a$ lies
in the proper left ideal $\{ \,b \in A : \varphi(b^*b) = 0 \,\}$, cf.\ proposition
\ref{ispropleft}, so that $a$ cannot be left invertible by \ref{notinvideal} (ii).
\end{proof}

\begin{theorem}\label{maxleftpurestate}%
If $A$ is unital, then for a maximal left ideal $M$ in $A$, there exists a pure
state $\psi$ on $A$ such that
\[ M = \{ \,a \in A : \psi(a^*a) = 0 \,\}. \]
\end{theorem}

\begin{proof}
The set $K$ of those states $\varphi$ on $A$ with
\[ M = \{ \,a \in A : \varphi(a^*a) = 0 \,\} \]
is non-empty by \ref{maxleftstate}. Furthermore, $K$ is a compact
convex subset of the state space of $A$. (Indeed, if $\ell$ is in the
closed convex hull of $K$ inside the compact convex set
$\quasistates(A)$, then $\ell$ is a state, by \ref{eBbded}, with
\[ M \subset \{ \,a \in A : \ell \,(a^*a) = 0 \,\}. \]
But the set on the right hand side is a proper left ideal, and therefore
equal to $M$, by maximality of $M$. So $\ell \in K$.) Thus $K$ has
an extreme point, $\psi$ say. It remains to be shown that the state
$\psi$ is pure. So let
\[ \psi = \lambda \varphi_1 + \mu \varphi_2 \]
with $\varphi_1$, $\varphi_2 \in \statespace (A)$, and
$\lambda, \mu \in \ ] \,0, 1 \,[$ such that $\lambda + \mu = 1$.
We get
\[ M \subset \{ \,a \in A : \varphi_i (a^*a) = 0 \,\} \]
by positivity of $\varphi_i$, for $i \in \{ \,1, 2 \,\}$. As above, it follows
that both sets are equal, i.e.\ $\varphi_1, \varphi_2 \in K$, which implies
that $\psi = \varphi_1 = \varphi_2$.
\end{proof}

\begin{corollary}\label{invpstate}%
If $A$ is unital, then an element $a$ of $A$ is left invertible if and only if
\[ \psi(a^*a) > 0 \quad \text{for all} \quad \psi \in \purestates(A). \]
\end{corollary}

\begin{proof}
In the proof of \ref{invstate}, replace ``state'' by ``pure state''. \pagebreak
\end{proof}

A similar characterisation of the right invertible elements holds
as well (\,by changing the multiplication to $(a,b) \mapsto ba$\,),
and thus we get: 

\begin{corollary}\label{normalleftright}%
A normal element of a unital Hermitian Banach \linebreak \st-algebra is
left invertible if and only if it it is right invertible. In this event, the element
is invertible, by \ref{leftrightinv}. See also \ref{Hermlrinv}.
\end{corollary}

\begin{definition}[$\mathrm{E}(b)$]%
\label{Ea}\index{symbols}{E(a)@$\mathrm{E}(a)$}%
Let $B$ be a normed \st-algebra, and let $b \in B$. We shall consider
the set $\mathrm{E}(b)$ defined as the set of those pure states $\psi$ on $B$,
for which $c_{\psi}$ is an eigenvector of $\pi_{\psi}(b)$ to the
eigenvalue $\psi(b)$. That is,
\[ \mathrm{E}(b) := \{ \,\psi \in \purestates (B) : \pi_{\psi}(b)c_{\psi} = \psi(b)c_{\psi} \,\}. \]
\end{definition}

We have the following characterisation of $\mathrm{E}(b)$:

\begin{proposition}\label{Eachar}%
For an element $b$ of a normed \st-algebra $B$, we have
\[ \mathrm{E}(b) = \{ \,\psi \in \purestates (B) : {| \,\psi(b) \,| \,}^2 = \psi(b^*b) \,\}. \]
\end{proposition}

\begin{proof}
For a state $\psi$ on $B$, we calculate
\begin{align*}
&\ {\| \,\bigl( \pi_{\psi}(b) - \psi(b) \bigr) c_{\psi} \,\| \,}^2 \\
= &\ \langle \bigl( \pi_{\psi}(b) - \psi(b) \bigr) c_{\psi} ,
\bigl( \pi_{\psi}(b) - \psi(b) \bigr) c_{\psi} \rangle_{\psi} \\
= &\ {\| \,\pi_{\psi}(b)c_{\psi} \,\| \,}^2 + {\| \,\psi(b)c_{\psi} \,\| \,}^2
- 2 \,\mathrm{Re} \,\bigl( \overline{\psi(b)} \langle \pi_{\psi(b)}c_{\psi},
c_{\psi} \rangle_{\psi} \bigr) \\
= &\ \psi(b^*b) - {| \,\psi(b) \,| \,}^2. \qedhere
\end{align*}
\end{proof}

\begin{definition}[the left spectrum, $\mathrm{lsp}_B(b)$]%
\index{symbols}{L15@$\mathrm{lsp}(a)$}%
\index{concepts}{spectrum!of an element!left}%
If $b$ is an element of an algebra $B$, we define the
\underline{left spectrum} of $b$ as the set $\mathrm{lsp}_B(b)$
of those complex numbers $\lambda$ for which $\lambda e - b$
is not left invertible in $\tld{B}$. We shall also abbreviate
$\mathrm{lsp}(b) := \mathrm{lsp}_B(b)$.
\end{definition}

\begin{observation}%
For a normal element of a Hermitian Banach \st-algebra,
the left spectrum coincides with the spectrum, cf.\ \ref{normalleftright}.
\end{observation}

\begin{theorem}\label{leftirred}%
If $A$ is unital, then for an element $a$ of $A$, we have
\[ \mathrm{lsp}_A(a) = \{ \,\psi(a) : \psi \in \mathrm{E}(a) \,\}. \]
In particular, each point in the left spectrum is an eigenvalue
in some irreducible representation, cf.\ \ref{Ea}. \pagebreak
\end{theorem}

\begin{proof}
Let $a \in A$, $\lambda \in \mathds{C}$ and put $b := \lambda e - a$.
For a pure state $\psi$ on $A$, we compute
\begin{align*}
\psi(b^*b) & = {| \,\lambda \,| \,}^2 + \psi(a^*a)
- 2 \,\mathrm{Re} \,\bigl( \,\overline{\lambda}\psi(a) \bigr) \\
 & = {| \,\lambda - \psi(a) \,| \,}^2 + \psi(a^*a) - {| \,\psi(a) \,| \,}^2.
\end{align*}
Now
\[ \psi(a^*a) - {| \,\psi(a) \,| \,}^2 \geq 0 \]
as $v(\psi) = 1$. Therefore $\psi(b^*b)$ vanishes if and only if both
$\lambda = \psi(a)$ and ${| \,\psi(a) \,| \,}^2 = \psi(a^*a)$. The statement
follows now from the results \ref{invpstate} and \ref{Eachar}.
\end{proof}

\begin{corollary}\label{lspC}%
If $A$ is unital, then for an element $a$ of $A$, we have
\[ \mathrm{lsp}_A(a) = \mathrm{lsp}_{\Cstar(A)} \bigl( j(a) \bigr). \]
Here, the left spectrum can be replaced with the right spectrum
and with the spectrum.
\end{corollary}

\begin{proof}
The bijection $\psi\0 \mapsto \psi\0 \circ j$ from $\purestates \bigl(\Cstar(A)\bigr)$
to $\purestates (A)$, cf.\ \ref{homPS}, restricts to a bijection from
$\mathrm{E}\bigl( j(a) \bigr)$ to $\mathrm{E}(a)$, cf.\ \ref{Eachar}.
The statement concerning the right spectrum follows by changing
the multiplication in $A$ to $(a,b) \mapsto ba$. The statement
concerning the spectrum follows from \ref{leftrightinv}.
\end{proof}

\medskip
Now for the case in which $A$ is not necessarily unital.

\begin{theorem}\label{leftirrednu}%
For an element $a$ of $A$, we have
\[ \mathrm{lsp}_A (a) \setminus \{0\}
= \{ \,\psi (a) : \psi \in \mathrm{E}(a) \,\} \setminus \{0\}. \]
\end{theorem}

\begin{proof}
The inclusion ``$\supset$'' follows from \ref{statecanext}, \ref{leftirred},
and \ref{specunit}. Conversely, let $0 \neq \lambda \in \mathrm{lsp}_A (a)$.
We then also have $0 \neq \lambda \in \mathrm{lsp}_{\tld{A}} (a)$ by
\ref{specunit}. Hence there exists a pure state $\tld{\psi}$ on $\tld{A}$
with $0 \neq \lambda = \tld{\psi}(a)$ and ${| \,\tld{\psi} (a) \,| \,}^2 = \tld{\psi} (a^*a)$,
cf.\ \ref{leftirred} \& \ref{Eachar}. The restriction $\psi$ of $\tld{\psi}$ to $A$
then is a quasi-state with $0 \neq \lambda = \psi (a)$, and
${| \,\psi (a) \,| \,}^2 = \psi (a^*a)$. The fact that
$0 \neq {| \,\psi (a) \,| \,}^2 = \psi (a^*a)$ implies that $\psi$ actually is a state.
The state $\psi$ is pure because every state $\varphi$ on $\tld{A}$ satisfies
$\varphi (e) = 1$, cf. \ref{vare}. This makes that $\psi \in \mathrm{E}(a)$,
cf.\ \ref{Eachar}.
\pagebreak
\end{proof}

\begin{corollary}\label{lspCw}%
For an element $a$ of $A$, we have
\[ \mathrm{lsp}_A(a) \setminus \{0\}
= \mathrm{lsp}_{\Cstar(A)} \bigl( j(a) \bigr) \setminus \{0\}. \]
Here, the left spectrum can be replaced with the right spectrum
and with the spectrum.
\end{corollary}

\begin{proof}
The proof is the same as that of \ref{lspC}.
\end{proof}

\medskip
It is not very far from here to Ra\u{\i}kov's Criterion \ref{Raikov}.

\begin{theorem}
For an element $a$ of $A$, we have
\[ \mathrm{lsp}_A(a) \setminus \{0\}
= \mathrm{lsp} \bigl( \pi\ra (a) \bigr) \setminus \{0\}
= \mathrm{lsp} \bigl( \pi\univ (a) \bigr) \setminus \{0\}, \]
and the latter two sets consist entirely of eigenvalues.
\end{theorem}

\begin{proof}
This follows now from \ref{leftirrednu}, \ref{Ea}, and from the fact that
\[ \mathrm{lsp} \bigl( \pi (a) \bigr) \setminus \{0\}
\subset \mathrm{lsp} ( a ) \setminus \{0\} \]
for an algebra homomorphism $\pi$, cf.\ the proof of \ref{spechom}.
\end{proof}

\clearpage


\addtocontents{toc}{\protect\vspace{0.4em}}

\part{Spectral Theory of Representations}\label{part3}

\addtocontents{toc}{\protect\vspace{0.4em}}

\chapter{Spectral Theorems and von Neumann Algebras}

\setcounter{section}{37}


\section{Spectral Measures}

\medskip
For the remainder of this paragraph, let $H \neq \{0\}$ denote a
Hilbert space, let $\mathcal{E}$ denote a $\sigma$-algebra on
a set $\Omega \neq \varnothing$, and let $P$ denote a spectral
measure defined on $\mathcal{E}$ and acting on $H$, in the
sense of the following definition.

\begin{definition}[spectral measures]\label{specmeasdef}%
\index{concepts}{spectral!measure}\index{concepts}{measure!spectral}%
Let $H \neq \{0\}$ be a Hilbert space, and let $\mathcal{E}$ be
a $\sigma$-algebra on a set $\Omega \neq \varnothing$. Then a
\underline{spectral measure} defined on $\mathcal{E}$ and acting
on $H$ is a map
\[ P : \mathcal{E} \to \blop(H) \]
such that
\begin{itemize}
   \item[$(i)$] $P(\Delta)$ is a projection in $H$ for each $\Delta \in \mathcal{E}$,
  \item[$(ii)$] $P(\Omega) = \mathds{1}$,
 \item[$(iii)$] for all $x \in H$, the function
\end{itemize}
\begin{align*}
     \quad \langle P x, x \rangle : \mathcal{E} & \to [ \,0, \infty \,[ \\
      \Delta & \mapsto \langle P(\Delta) x, x \rangle  = {\| \,P(\Delta) x \,\| \,}^2
\end{align*}
\qquad\quad\ \ is a measure.

\medskip
So, if $x$ is a unit vector in $H$, then $\langle P x, x \rangle$ is a probability measure.

For $x$, $y \in H$ arbitrary, one defines a bounded complex measure
$\langle P x, y \rangle$ on $\Omega$ by polarisation:
\[ \langle Px,y \rangle \,:=
\,\frac14 \,\sum _{k=1}^{4} \,\iu ^k \,\langle P(x + \iu ^ky),(x + \iu ^ky) \rangle,
\quad \text{cf.\ the appendix \ref{bdedmeas}.} \]
We then have $\langle P x, y \rangle (\Delta) = \langle P(\Delta) x, y \rangle$
for all $\Delta \in \mathcal{E}$, and all $x,y \in H$.

By abuse of notation, one also puts
\[ P := \{ \,P (\Delta) \in \blop(H) : \Delta \in \mathcal{E} \,\}. \pagebreak \]
\end{definition}

\begin{lemma}\label{projsum}%
Let $p, q$ be projections in $H$. Then their sum $p+q$ is a
projection if and only if $p$ and $q$ are orthogonal \ref{projrefl}.
\end{lemma} 

\begin{proof}
It is easily seen that the sum of two orthogonal projections is a projection again.
Conversely, if $p+q$ is a projection, then
\[ p+q = {( \,p+q \,)}^{\,2} = p + p\,q + q\,p +q, \]
which implies that $p\,q + q\,p = 0$. Multiplication from the left and from
the right with $q$ yields $q\,p\,q + q\,p\,q = 0$, and so $q\,p\,q = 0$.
The C*-property \ref{preC*alg} implies ${\| \,p\,q \,\|}^{\,2} = \| \,q\,p\,q \,\| = 0$,
that is, $p\,q = 0$.
\end{proof}

\begin{proposition}\label{rep}%
The following statements hold.
\begin{itemize}
   \item[$(i)$] $P$ is finitely additive, i.e.\ if $D$ is a \underline{finite}
   collection of mutually disjoint sets in $\mathcal{E}$, then
  \[ P \,\Bigl(\bigcup _{\text{\footnotesize{$\Delta \in D$}}} \Delta \,\Bigr)
  = \sum _{\text{\footnotesize{$\Delta \in D$}}} P(\Delta), \]
  \item[$(ii)$] $P$ is orthogonal, i.e.\ if $\Delta_1, \Delta_2 \in \mathcal{E}$
  and $\Delta_1 \cap \Delta_2 = \varnothing$, then
  \[ P(\Delta_1) \,P(\Delta_2) = 0, \]
   \item[$(iii)$] the events in $\mathcal{E}$ are independent, i.e.\ for
  $\Delta_1, \Delta_2 \in \mathcal{E}$, we have
  \[ P(\Delta_1 \cap \Delta_2) = P(\Delta_1) \,P(\Delta_2). \]
  In particular $P(\Delta_1)$ and $P(\Delta_2)$ \underline{commute}.
\end{itemize}
\end{proposition}

\begin{proof} (i): For all $x, y \in H$, we have
\[ \langle \,P \,\Bigl(\bigcup _{\text{\footnotesize{$\Delta \in D$}}} \Delta \,\Bigr) \,x, y \,\rangle
= \sum _{\text{\footnotesize{$\Delta \in D$}}} \langle \,P(\Delta) \,x, y \,\rangle
= \langle \sum _{\text{\footnotesize{$\Delta \in D$}}} P(\Delta) \,x, y \,\rangle \]
by finite additivity of the measure
$\langle P x, y \rangle : \Delta \mapsto \langle P(\Delta) x, y \rangle$, cf.\ \ref{specmeasdef}.
\par \noindent (ii): Let $\Delta_1,\Delta_2 \in \mathcal{E}$
with $\Delta_1 \cap \Delta_2 = \varnothing$.
By (i), we have
\[ P(\Delta_1) + P(\Delta_2) = P(\Delta_1 \cup \Delta_2), \]
which says that the sum $P(\Delta_1) + P(\Delta_2)$ is a projection.
The preceding lemma \ref{projsum} implies now that $P(\Delta_1) \,P(\Delta_2) = 0$.
\par \noindent (iii): This follows from (i) and (ii) in the following way.
\begin{align*}
& \ P(\Delta_1)P(\Delta_2) \\
= & \ P\bigl((\Delta_1 \cap \Delta_2) \cup (\Delta_1 \setminus \Delta_2)\bigr)
\cdot P\bigl((\Delta_2 \cap \Delta_1) \cup (\Delta_2 \setminus \Delta_1)\bigr) \\
= & \ \bigl(P(\Delta_1 \cap \Delta_2)+P(\Delta_1 \setminus \Delta_2)\bigr)
\cdot \bigl(P(\Delta_2 \cap \Delta_1)+P(\Delta_2 \setminus \Delta_1)\bigr) \\
= & \ P(\Delta_1 \cap \Delta_2). \pagebreak \qedhere
\end{align*}
\end{proof}

\begin{definition}[$\stepfun (\mathcal{E})$]%
\index{symbols}{S6@$\stepfun(\mathcal{E})$}\index{concepts}{step functions}%
We denote by $\stepfun (\mathcal{E})$ the \st-subalgebra of
$\bigl( \mspace{2mu} {\ell}^{\,\infty}(\Omega), | \cdot |_\infty \mspace{2mu} \bigr)$
consisting of all $\mathcal{E}$-step functions on $\Omega$.
\end{definition}

\begin{definition}[$I_P$]\index{symbols}{I1@$I_P$}%
Let $I_P$ be the linear map on $\stepfun (\mathcal{E})$ with \linebreak
$I_P(1_{\text{\small{$\Delta$}}}) = P(\Delta)$
for all $\Delta \in \mathcal{E}$. That is, if
\[ h = \sum _{\text{\footnotesize{$\Delta \in D$}}}
\alpha_{\text{\footnotesize{$\Delta$}}} \,1_{\text{\footnotesize{$\Delta$}}} \]
is an $\mathcal{E}$-step function on $\Omega$, with $D \subset \mathcal{E}$
a finite partition of $\Omega$, then
\[ I_P (h) := \sum _{\text{\footnotesize{$\Delta \in D$}}}
\alpha_{\text{\footnotesize{$\Delta$}}} \,P(\Delta). \]
Here ``$_{\,}I_P$'' stands for ``$_{\,}$integral with respect to $P$''.
\end{definition}

\begin{observation}
The map $I_P$ is a representation of $\stepfun (\mathcal{E})$ on $H$
by proposition \ref{rep} above.
\end{observation}

\begin{proposition}\label{stepint}%
For an $\mathcal{E}$-step function $h$ on $\Omega$, we have
\[ \langle I_P (h) x, x \rangle = \int h \,\diff \langle P x, x \rangle
\quad \text{whenever} \quad x \in H. \]
\end{proposition}

\begin{proof}
Indeed, if
\[ h = \sum _{\text{\footnotesize{$\Delta \in D$}}}
\alpha_{\text{\footnotesize{$\Delta$}}} \,1_{\text{\footnotesize{$\Delta$}}} \]
is a $\mathcal{E}$-step function on $\Omega$, with $D \subset \mathcal{E}$
a finite partition of $\Omega$, then
\begin{align*}
   \langle I_P (h) x, x \rangle
   & = \langle \sum _{\text{\footnotesize{$\Delta \in D$}}}
          \alpha_{\text{\footnotesize{$\Delta$}}} \,P(\Delta) \,x, x \,\rangle \\
   & = \sum _{\text{\footnotesize{$\Delta \in D$}}}
          \alpha_{\text{\footnotesize{$\Delta$}}} \,\langle P(\Delta) \,x, x \rangle \\
   & = \sum _{\text{\footnotesize{$\Delta \in D$}}}
          \alpha_{\text{\footnotesize{$\Delta$}}}
          \int \,1_{\text{\footnotesize{$\Delta$}}} \,\diff \,\langle P x, x \rangle
       = \int h \,\diff \langle P x, x \rangle. \pagebreak \qedhere
\end{align*}
\end{proof}

\begin{definition}[$\bmeas(\mathcal{E})$]%
\index{symbols}{M2@$\bmeas(\mathcal{E})$}%
We denote by $\bmeas(\mathcal{E})$ the subspace of
$\bigl( \mspace{2mu} {\ell}^{\,\infty}(\Omega), | \cdot |_\infty \mspace{2mu} \bigr)$
consisting of all bounded complex-valued
$\mathcal{E}$-measurable functions on $\Omega$,
cf.\ the appendix \ref{measurability}. \pagebreak
\end{definition}

\begin{proposition}\label{closure}%
The vector space $\bmeas(\mathcal{E})$ is the closure of the \st-subalgebra
$\stepfun (\mathcal{E})$ of $\mathcal{E}$-step functions in the C*-algebra
$\bigl( \mspace{2mu} {\ell}^{\,\infty}(\Omega), | \cdot |_\infty \mspace{2mu} \bigr)$.
\end{proposition}

\begin{proof}
Let $\overline{\stepfun (\mathcal{E})}$ denote the closure of $\stepfun (\mathcal{E})$ in
the C*-algebra ${\ell}^{\,\infty}(\Omega)$. We have $\overline{\stepfun (\mathcal{E})}
\subset \bmeas(\mathcal{E})$ as even a pointwise limit of a sequence of
\linebreak $\mathcal{E}$-measurable functions is $\mathcal{E}$-measurable.
Conversely, we also have \linebreak
$\bmeas(\mathcal{E}) \subset \overline{\stepfun (\mathcal{E})}$
because every bounded non-negative $\mathcal{E}$-measurable function is the
uniform limit of an increasing sequence of non-negative $\mathcal{E}$-step functions.
\end{proof}

\begin{corollary}%
The set $\bmeas(\mathcal{E})$ is a C*-subalgebra of ${\ell}^{\,\infty}(\Omega)$.
\end{corollary}

\begin{definition}[spectral integrals]%
Let $f \in \bmeas(\mathcal{E})$, and let \linebreak $a \in \blop(H)$.

One writes
\[ a = \int f \,\diff P \qquad \text{(weakly)} \]
if
\[ \langle ax,x \rangle = \int f \,\diff \langle Px,x \rangle
\quad \text{for all } x \in H. \]
By polarisation, we then also have
\[ \langle ax,y \rangle = \int f \,\diff \langle Px,y \rangle
\quad \text{for all } x,y \in H, \]
so that there is at most one such operator $a$ for given $f$ and $P$.

We shall write
\[ a = \int f \,\diff P \qquad \text{(in norm)} \]
if for every $\mathcal{E}$-step function $h$ on $\Omega$, one has
\[ \| \,a - I_P(h) \,\| \leq | \,f-h \,|_{\infty}. \]
\end{definition}

\begin{proposition}%
Let $f \in \bmeas(\mathcal{E})$ and $a \in \blop(H)$. Assume that
\[ a = \int f \,\diff P \qquad \text{(in norm).} \]
We then also have
\[ a = \int f \,\diff P \qquad \text{(weakly)}. \]
In particular, there is at most one such operator $a$ for given $f$ and $P$.
\pagebreak
\end{proposition}

\begin{proof}
Since the $\mathcal{E}$-measurable function $f$ is bounded, there exists
a sequence $(h_n)$ of $\mathcal{E}$-step functions converging uniformly
to $f$, cf.\ \ref{closure}. From
\[ a = \int f \,\diff P \qquad \text{(in norm)} \]
we have
\[ \| \,a - I_P ( h_n ) \,\| \leq | \,f - h_n \,|_{\infty} \to 0, \]
that is,
\[ a = \lim _{n \to \infty} I_P(h_n).  \]
With \ref{stepint} we get for all $x \in H$:
\[ \langle ax,x \rangle
   = \lim _{n \to \infty} \langle I_P(h_n)x,x \rangle
   = \lim _{n \to \infty} \int h_n \,\diff \langle Px,x \rangle
   = \int f \,\diff \langle Px,x \rangle, \]
which is the same as
\[ a = \int f \,\diff P \qquad \text{(weakly)}. \qedhere \]
\end{proof}

\begin{theorem}[$\pi_{P}(f)$]\index{symbols}{p6@$\pi_P$}%
For $f \in \bmeas(\mathcal{E})$ there exists a (necessarily
\linebreak unique) operator $\pi_{P}(f) \in \blop(H)$ such that
\[ \pi _P(f) = \int f \,\diff P \qquad \text{(in norm)}. \]
The map
\begin{align*}
\pi _P : \bmeas(\mathcal{E}) & \to \blop(H) \\
             f & \mapsto \pi_{P}(f)
\end{align*}
then defines a \underline{representation} $\pi_P$
of $\bmeas(\mathcal{E})$ on $H$.
\end{theorem}

\begin{proof}
For an $\mathcal{E}$-step function $h$ on $\Omega$, we may write
\[ h = \sum _{\text{\footnotesize{$\Delta \in D$}}}
\alpha_{\text{\footnotesize{$\Delta$}}} \,1_{\text{\footnotesize{$\Delta$}}} \]
where $D \subset \mathcal{E}$ is a finite partition of $\Omega$. We obtain
\[ I_P(h) = \sum _{\text{\footnotesize{$\Delta \in D$}}}
\alpha_{\text{\footnotesize{$\Delta$}}} \,P(\Delta). \]
In order to show that the representation $I_P$ is contractive:
\[ \| \,I_P(h) \,\| \leq | \,h \,|_{\infty}, \pagebreak \]
we note that for $x \in H$, we have, by \ref{rep} (ii),
\begin{align*}
{\| \,I_P(h)x \,\| \,}^2 & = \sum _{\text{\footnotesize{$\Delta \in D$}}}
{| \,\alpha_{\text{\footnotesize{$\Delta$}}} \,| \,}^2 \ {\| \,P(\Delta)x \,\| \,}^2 \\
 & \leq {| \,h \,|_{\infty} \,}^2 \sum _{\text{\footnotesize{$\Delta \in D$}}}
{\| \,P(\Delta)x \,\| \,}^2 = {| \,h \,|_{\infty} \,}^2 \ {\| \,x \,\| \,}^2.
\end{align*}
This implies that the mapping
\[ I_P : h \mapsto I_P(h) \qquad \bigl( \,h \in \stepfun (\mathcal{E}) \,\bigr) \]
has a unique extension to a continuous mapping
\[ \pi _P : \bmeas(\mathcal{E}) \to \blop(H) \]
The latter mapping is a contractive representation
(since $I_P$ is so), from which we get that
\[ \pi _P(f) = \int f \,\diff P \qquad \text{(in norm)} \]
for each $f \in \bmeas(\mathcal{E})$.
\end{proof}

\begin{proposition}\label{weaklyinnorm}%
Let $a$ be a bounded linear operator on $H$ such that
\[ a = \int f \,\diff P \qquad \text{(weakly)} \]
for some $f \in \bmeas(\mathcal{E})$. We then also have
\[ a = \int f \,\diff P  \qquad \text{(in norm)}. \]
\end{proposition}

\begin{proof}
We have $a = \pi _P(f)$ by $\pi _P(f) = \int f \,\diff P$ (weakly) as well.%
\end{proof}

\begin{proposition}\label{normpointwise}%
For $f \in \bmeas(\mathcal{E})$ and $x \in H$, we have
\[ \| \, \pi_P (f) \,x \,\| 
= \| \,f \,\|_{\text{\small{$\langle Px,x \rangle,2$}}}
= {\biggl(\int {| \,f \,| \,}^2 \ \diff \langle Px,x \rangle \biggr)}^{1/2}. \]
\end{proposition}

\begin{proof} One calculates
\[ {\| \,\pi _P(f) \,x \,\| \,}^2 = \langle \pi_P\bigl({ \,| \,f \,| \,}^2 \,\bigr) \,x, x \rangle
= \int {| \,f \,| \,}^2 \,\diff \langle Px,x \rangle. \pagebreak \qedhere \]
\end{proof}

\begin{theorem}\label{spanP}%
The range of the representation $\pi _P$ is given by
\[ \range(\pi _P) = \overline{\mathrm{span}}(P). \]
It follows that $\range(\pi _P)$ is the C*-subalgebra of
$\blop(H)$ generated by $P$.
\end{theorem}

\begin{proof}
By \ref{rangeC*} it follows that $\range(\pi _P)$ is a
C*-algebra and so a closed subspace of $\blop(H)$,
containing $\mathrm{span}(P)$ as a dense subset.
\end{proof}

\begin{corollary}\label{spprojcomm}%
We have that
\[ {\pi _P}' = P', \]
cf.\ \ref{repcommutant}.
\end{corollary}

\begin{definition}[$P$-a.e.]\index{symbols}{P95@$P$-a.e.}%
A property applicable to the  points in $\Omega$ is said
to hold \underline{$P$-almost everywhere} ($P$-a.e.)
if it holds $\langle Px, x \rangle$-almost everywhere for
all $x \in H$.
\end{definition}

\begin{definition}[$\mathrm{N}(P)$]%
\index{symbols}{N(P)@$\mathrm{N}(P)$}%
We shall denote by $\mathrm{N}(P)$ the set of those functions
in $\bmeas(\mathcal{E})$ which vanish $P$-almost everywhere.
\end{definition}

\begin{theorem}\label{kerspectralint}%
For $f \in \bmeas(\mathcal{E})$, we have
\[ \pi _P (f) = 0 \ \Leftrightarrow \ f = 0 \ P\text{-a.e.} \]
In other words,
\[ \ker \pi _P = \mathrm{N}(P). \]
In particular, $\mathrm{N}(P)$ is a closed \st-stable
\ref{selfadjointsubset} ideal in $\bmeas(\mathcal{E})$.
\end{theorem}

\begin{proof}
This follows from
\[ \| \, \pi_P (f) \,x \,\|  = \| \,f \,\|_{\text{\small{$\langle Px,x \rangle,2$}}}
\quad \text{for all} \quad x \in H, \]
cf.\ \ref{normpointwise}.
\end{proof}

\begin{definition}[$\rmLeb ^{\infty}(P)$]%
\index{symbols}{L2@$\rmLeb ^{\infty}(P)$}%
We denote
\[ \rmLeb ^{\infty}(P) := \bmeas(\mathcal{E})/\mathrm{N}(P), \]
which is a commutative C*-algebra, cf.\ \ref{C*quotient}.
\end{definition}

\begin{theorem}\label{LPmain}%
The representation $\pi _P$ factors to an isomorphism
of C*-algebras from $\rmLeb ^{\infty}(P)$ onto $\range(\pi _P)$.
\end{theorem}

\begin{proof}
This follows from \ref{rangeC*} and \ref{kerspectralint}. \pagebreak
\end{proof}

\begin{theorem}\label{spprojform}%
Every projection in $\range(\pi _P)$ is of the
form $P(\Delta)$ for some $\Delta \in \mathcal{E}$.
\end{theorem}

\begin{proof} Let $f \in \bmeas(\mathcal{E})$ be such
that $\pi _P(f) = {\pi _P(f)}^{\,2}$. We obtain $f = {f }^{\,2}$ $P$-a.e.
It follows that $f$ is $P$-a.e.\ equal to either $0$ or $1$.
With $\Delta = {f}^{\,-1}(1) \in \mathcal{E}$, we thus have
$f = 1_{\text{\small{$\Delta$}}}$ $P$-a.e., and consequently
$\pi _P(f) = \pi _P(1_{\text{\small{$\Delta$}}}) = P(\Delta)$.
\end{proof}

\begin{theorem}\label{piPpos}%
For $f \in \bmeas(\mathcal{E})$, we have
\[ \pi _P(f) \geq 0 \ \Leftrightarrow \ f \geq 0 \ P\text{-a.e.} \]
\end{theorem}

\begin{proof}
Let $| \,f \,|$ denote the absolute value of $f$ in the C*-algebra
$\bmeas(\mathcal{E})$, cf.\ \ref{absval}. (The element
$| \,f \,|$ is given by the pointwise absolute function
$x \mapsto | \,f(x) \,|$, cf.\ \ref{Coabs}.) We then have
\[ \pi _P\bigl( \,| \,f \,| \,\bigr) = \bigl| \,\pi _P(f) \,\bigr|, \]
cf.\ \ref{homabs}. Using the fact that in a C*-algebra, an element
$a$ is positive if and only if $a = | \,a \,|$, cf.\ \ref{abschar}, we get 
\begin{align*}
\pi _P(f) \geq 0 & \ \Leftrightarrow \ \pi _P(f) = \bigl| \,\pi _P(f) \,\bigr| \\
& \ \Leftrightarrow \ \pi _P(f) = \pi _P\bigl( \,| \,f \,| \,\bigr)
\ \Leftrightarrow \ f = | \,f \,| \ P \text{-a.e.}
\end{align*}
by \ref{kerspectralint}. Since the element $| \,f \,|$ is given
by the pointwise absolute function $x \mapsto | \,f(x) \,|$
(as noted before), we have
\[ f = | \,f \,| \ P \text{-a.e.} \ \Leftrightarrow \ f \geq 0 \ P\text{-a.e.}, \]
and the statement follows.
\end{proof}

\begin{theorem}[the Monotone Convergence Theorem]%
\label{monconvthm}%
If $(f_n)$ is an upper bounded sequence in
$\bmeas(\mathcal{E})\sa$ such that
$f_n \leq f_{n+1}$ $P$-a.e.\ for all $n$, then
the pointwise supremum $f := \sup_n f_n$
is in $\bmeas(\mathcal{E})\sa$, and
\[ \pi_P(f) = \sup_n \,\pi_P(f_n) \quad \text{in } \blop(H)\sa. \]
Furthermore
\[ \pi_P(f) \,x = \lim_n \,\pi_P(f_n) \,x
\quad \text{for all} \quad x \in H. \pagebreak \]
\end{theorem}

\begin{proof}
Please note first that the order in $\bmeas(\mathcal{E})\sa$
is the pointwise order by \ref{pointwiseorder}.
The pointwise supremum $f := \sup_n f_n$ exists as a
real-valued function because the sequence $(f_n)$ is
upper bounded in $\bmeas(\mathcal{E})\sa$.
Then also $f \in \bmeas(\mathcal{E})\sa$.
On one hand, the sequence $\{ \,\pi_P (f_n) \,\}$
converges pointwise on $H$ to its supremum in $\blop(H)\sa$,
cf.\ \ref{ordercompl}, as it is increasing and upper bounded
by \ref{piPpos} above. On the other hand, the sequence
$\{ \,\pi_P (f_n) \,\}$ converges pointwise on $H$ to $\pi_P(f)$, as
\[ \| \,\pi _P(f) \,x - \pi _P(f_n) \,x \,\|
= {\biggl( \int {| \,f-f_n \,| \,}^2 \,\diff \langle Px,x \rangle \biggr)}^{1/2} \to 0 \]
for all $x \in H$, by Lebesgue's Monotone Convergence Theorem
and by \ref{normpointwise}. (This shows already the last statement.)
As a consequence of these two facts, we get that
\[ \pi_P(f) = \sup_n \,\pi_P(f_n) \quad \text{in } \blop(H)\sa. \qedhere \]
\end{proof}

\begin{corollary}[strong $\sigma$-continuity]\label{strsgcont}%
If $(\Delta _n)$ is an increasing sequence of sets in $\mathcal{E}$, then
\[ P \,\Bigl( \,\bigcup _n \Delta _n \Bigr) \,x = \lim _{n \to \infty} P ( \Delta _n ) \,x
\quad \text{for all} \quad x \in H. \]
\end{corollary}

From \ref{rep} (i), it follows now:

\begin{corollary}[strong $\sigma$-additivity]%
If $(\Delta _n)$ is a sequence of pairwise disjoint sets in $\mathcal{E}$, then
\[ P \,\Bigl( \,\bigcup _n \Delta _n \Bigr) \,x = \sum _n P ( \Delta _n ) \,x
\quad \text{for all} \quad x \in H. \]
\end{corollary}

That is, the series $\sum_n P(\Delta_n)$ converges pointwise on $H$.

Please note however that an infinite series of non-zero projections in $H$
never converges in the norm of $\blop(H)$, as the terms would have to
converge to $0$ in norm.

\begin{theorem}[the Dominated Convergence Theorem]%
\label{domconvergence}%
If $(f_n)$ is a norm-bounded sequence in $\bmeas(\mathcal{E})$
which converges pointwise $P$-a.e.\ to a function
$f \in \bmeas(\mathcal{E})$, then
\[ \lim _{n \to \infty} \pi _P(f_n) \,x = \pi _P(f) \,x
\quad \text{for all} \quad x \in H. \pagebreak \]
\end{theorem}

\begin{proof}
This follows from Lebesgue's Dominated Convergence Theorem
because by \ref{normpointwise} we have
\[ \| \,\pi _P(f) \,x - \pi _P(f_n) \,x \,\|
= {\biggl( \int {| \,f-f_n \,| \,}^2 \,\diff \langle Px,x \rangle \biggr)}^{1/2} \to 0 \]
for all $x \in H$.
\end{proof}

\medskip
Next up are images of spectral measures, which
are needed in the following paragraph already.
For the case of images of ordinary probability measures,
we refer to the appendix \ref{imagebegin} - \ref{imageend}.

\begin{definition}[image of a spectral measure]\label{imagespdef}%
\index{concepts}{spectral!measure!image}%
\index{concepts}{image!spectral measure}%
\index{concepts}{measure!spectral!image}%
Consider another $\sigma$-algebra $\mathcal{E}'$ on another set
$\Omega' \neq \varnothing$, and let $f : \Omega \to \Omega'$ be
some \linebreak $\mathcal{E}$-$\mathcal{E}'$ measurable function. Then
\begin{alignat*}{2}
 f(P) : \mathcal{E}' & \to           &\ & \blop(H)_+ \\
                 \Delta' & \mapsto &\ & P\bigl(f^{-1}(\Delta')\bigr)
\end{alignat*}
defines a spectral measure $f(P) = P \circ f^{-1}$ defined on $\mathcal{E}'$,
called the \underline{image} of $P$ under $f$.
\end{definition}

\begin{proof} This is easily verified. \end{proof}

\medskip
See also the appendix \ref{imagebegin}
for the case of ordinary probability measures.

\begin{theorem}[integration with respect to an image measure]%
\label{imagespthm}%
Let $\mathcal{E}'$ be another $\sigma$-algebra on another set
$\Omega' \neq \varnothing$, and let $f : \Omega \to \Omega'$
be an $\mathcal{E}$-$\mathcal{E}'$ measurable function.
For $g \in \bmeas(\mathcal{E'})$, one has
\[ \int g\ \diff\mspace{2mu}f (P) = \int (g \circ f) \,\diff P,
\quad \text{that is,} \quad \pi_{f(P)}\,(g) = \pi_P \,(g \circ f). \]
\end{theorem}

\begin{proof}
One shows this first for $\mathcal{E'}$-step functions on $\Omega'$.
One then uniformly approximates functions in $\bmeas(\mathcal{E'})$
by $\mathcal{E}'$-step functions on $\Omega'$, and uses the contractivity
of the representations associated with the spectral measures.
\end{proof}

\medskip
See also the appendix \ref{imageend} for the case of ordinary
probability measures.

The preceding theorem \ref{imagespthm} will be extended in
theorem \ref{imagesputhm} below.
\pagebreak

\clearpage


\section{Spectral Theorems}

\begin{definition}[resolution of identity]%
\index{concepts}{resolution of identity}%
\index{concepts}{spectral!measure!resolution of id.}%
\index{concepts}{measure!spectral!resolution of id.}%
Let $\Omega \neq \varnothing$ be a Hausdorff space. Then a
\underline{resolution of identity} on $\Omega$ is a spectral measure
$P$, defined on the Borel $\sigma$-algebra of $\Omega$, acting on a
Hilbert space $H \neq \{0\}$, such that for all unit vectors $x$ in $H$,
the Borel probability measure $\langle Px, x \rangle$ is inner regular,
cf.\ the appendix \ref{Boreldef} \& \ref{inregBormeas}.
\end{definition}

\begin{definition}[the support, $\mathrm{supp}(P)$]%
\index{symbols}{s4@$\mathrm{supp}(P)$}%
\index{concepts}{support!of a resolution of identity}%
\index{concepts}{resolution of identity!support}%
Let $P$ be a resolution of identity on a Hausdorff space $\Omega \neq \varnothing$.
If there exists a largest open subset $\Delta$ of $\Omega$ with $P(\Delta) = 0$,
then $\Omega \setminus \Delta$ is called the \underline{support} of $P$.
The support of $P$ exists, as the union $\Delta$ of all open subsets $O$
of $\Omega$ with $P(O) = 0$ satisfies $P(\Delta) = 0$. (By inner regularity
and by a compactness argument.) The support of $P$ is denoted by
$\mathrm{supp}(P)$.
\end{definition}

\begin{observation}\label{suppimbed}%
Let $P$ be a resolution of identity on a Hausdorff space $\Omega \neq \varnothing$.
If the support of $P$ is all of $\Omega$, then $\cont_\mathrm{b}(\Omega)$
is imbedded in $\rmLeb ^{\infty}(P)$ as two continuous functions
differ on an open subset. This imbedding is isometric by \ref{C*injisometric}.
\end{observation}

Similarly for the support of inner regular Borel probability measures
on Hausdorff spaces $\neq \varnothing$.

\begin{theorem}[the Spectral Theorem, archetypal form]\label{spthmC*}%
\index{concepts}{Theorem!Spectral!commutative C*-algebra}%
\index{concepts}{spectral!theorem!commutative C*-algebra}%
\index{concepts}{spectral!resolution!commutative C*-algebra}%
Given \linebreak a commutative C*-algebra $B$ of bounded linear operators
acting non-degenerately on a Hilbert space $H \neq \{0\}$, there
is a unique resolution of identity $P$ on $\Delta(B)$, acting on $H$, such that
\[ b = \int \wht{b} \,\diff P \qquad \text{(in norm)}
\qquad \text{for all} \quad b \in B. \]
One says that $P$ is the \underline{spectral resolution} of $B$.
We have
\[ B' = P', \]
and the support of $P$ is all of $\Delta(B)$.
\end{theorem}

\begin{proof}
If $x$ is a unit vector in $H$, then
\[ b \mapsto \langle bx,x \rangle \qquad (b \in B) \]
is a state on $B$, cf.\ \ref{staterep}. It follows that there is a unique inner regular
Borel probability measure $\langle Px,x \rangle$ on $\Delta(B)$ such that
\[ \langle bx,x \rangle = \int \wht{b} \,\diff \langle Px,x \rangle
\quad \text{for all} \quad b \in B, \pagebreak \]
cf.\ \ref{predisint}. Polarisation yields bounded complex measures
$\langle Px,y \rangle$ on $\Delta(B)$ such that
\[ \langle bx,y \rangle = \int \wht{b} \,\diff \langle Px,y \rangle \]
for all $x,y \in H$ and all $b \in B$, cf.\ the appendix \ref{bdedmeas}. For a Borel
set $\Delta$ of $\Delta(B)$, consider the positive semidefinite sesquilinear form
$\varphi$ defined by
\[ \varphi (x,y) := \int 1_{\textstyle\Delta} \,\diff \langle Px,y \rangle
\qquad ( \,x, y \in H \,). \]
Since $\varphi (x,x) \leq 1$ for all $x$ in the unit ball of $H$, there
exists a unique operator $P(\Delta) \in \blop(H)_+$ such that
\[ \varphi (x,y) = \langle P(\Delta)x,y \rangle \quad \text{for all} \quad x, y \in H, \]
cf.\ \ref{posoprep}. In other words
\[ \langle P(\Delta)x,y \rangle = \int 1_{\textstyle\Delta} \,\diff \langle Px,y \rangle
\quad \text{for all} \quad x, y \in H. \]
In particular, $\Delta \mapsto \langle P(\Delta)x,y \rangle$ is the measure
$\langle P x, y \rangle$ for all $x, y \in H$.

Let $\Delta, \Delta'$ be two Borel subsets of $\Delta(B)$. We set out to
show that $P(\Delta) P(\Delta') = P(\Delta \cap \Delta')$. Let $x$, $y \in H$.
For $a$, $b \in B$, we have
\[ \ \int \wht{\vphantom{b}a} \,\wht{b} \,\diff \langle Px,y \rangle
= \ \langle abx,y \rangle
= \ \int \wht{\vphantom{b}a} \,\diff \langle Pbx,y \rangle. \tag*{$(\st)$} \]
Now $\{ \,\wht{a} : a \in B \,\} = \cont\0\bigl(\Delta(B)\bigr)$, see the Commutative
Gel'fand-Na\u{\i}mark Theorem \ref{commGN}. It follows that
\[ \wht{b} \cdot \langle Px,y \rangle = \langle Pbx,y \rangle \]
in the sense of equality of complex measures.
The integrals in $(\st)$ therefore remain equal if $\wht{a}$ is replaced by
$1_{\textstyle\Delta}$. Hence
\begin{align*}
\int 1_{\textstyle\Delta} \,\wht{b} \,\diff \langle Px,y \rangle
= & \ \int 1_{\textstyle\Delta} \,\diff \langle Pbx,y \rangle \\
= & \ \langle P(\Delta) bx,y \rangle \\
= & \ \langle bx,P(\Delta)y \rangle
=     \int \wht{b} \,\diff \langle Px,P(\Delta)y \rangle. \tag*{$(\st\st)$}
\end{align*}
The same reasoning as above shows that the integrals in $(\st\st)$ \pagebreak
remain equal if $\wht{b}$ is replaced by $1_{\textstyle\Delta'}$. Consequently
we have
\begin{align*}
\langle P(\Delta \cap \Delta') x,y \rangle
= & \ \int 1_{\textstyle{\Delta \cap \Delta'}} \,\diff \langle Px,y \rangle \\
= & \ \int 1_{\textstyle{\Delta\vphantom{'}}} 1_{\textstyle{\Delta'}} \,\diff \langle Px,y \rangle \\
= & \ \int 1_{\textstyle{\Delta'}} \,\diff \langle Px,P(\Delta)y \rangle \\
= & \ \langle P(\Delta')x,P(\Delta)y \rangle
=     \langle P(\Delta)P(\Delta')x,y \rangle,
\end{align*}
so that $P(\Delta \cap \Delta') = P(\Delta)P(\Delta')$ indeed.

This implies that $P(\Delta) \in \blop(H)_+$ is an idempotent, and thus a projection,
for each Borel set $\Delta$ of $\Delta(B)$. It follows that
$P\bigl(\Delta(B)\bigr) = \mathds{1}$ as
${\| \,P\bigl(\Delta(B)\bigr) x \,\|}^{\,2} = \langle P \bigl(\Delta(B)\bigr) x, x\rangle = 1 \neq 0$
for each unit vector $x$ in $H$.
Hence $P : \Delta \mapsto P(\Delta)$ is a spectral measure
and thus a resolution of identity.

From
\[ \langle bx,y \rangle = \int \wht{b} \,\diff \langle Px,y \rangle \]
it follows that
\[ b = \int \wht{b} \,\diff P \qquad \text{(weakly)}, \]
whence also
\[ b = \int \wht{b} \,\diff P \qquad \text{(in norm)}, \]
cf.\ \ref{weaklyinnorm}.

Uniqueness follows from the fact that
$\{ \,\wht{b}: b \in B \,\} = \cont\0\bigl(\Delta(B)\bigr)$,
see the Commutative Gel'fand-Na\u{\i}mark Theorem \ref{commGN}.

It shall now be shown that the support of $P$ is all of $\Delta(B)$,
i.e.\ $\varnothing$ is the largest open subset $\Delta$ of $\Delta(B)$
with $P(\Delta) = 0$. Let $\Delta$ be a non-empty open subset of
$\Delta(B)$ and let $x \in \Delta$. There then exists
$f \in \cont_\mathrm{c}(\Delta(B)$ such that
\[ 0 \leq f \leq 1,\ f(x) = 1,\ \mathrm{supp} (f) \subset \Delta. \]
Let $b \in B$ with $\wht{b} = f$. We then have $b \neq 0$, as well as
$b \geq 0$ by \ref{pointwiseorder}. Since $f \leq 1_{\textstyle\Delta}$,
it follows by \ref{piPpos}:
\[ 0 \lneq b = \pi _P \bigl( \,\wht{b} \,\bigr)
= \pi _P(f) \leq \pi _P(1_{\textstyle\Delta}) = P(\Delta), \]
whence $P(\Delta) \neq 0$. \pagebreak

It shall next be shown that $B' = P'$. For $a \in \blop(H)$, the following
statements are equivalent.
\begin{align*}
& a \in P', \\
& P(\Delta)a = aP(\Delta) \qquad \text{ for all Borel sets } \Delta, \\
& \langle P(\Delta) ax,y \rangle = \langle aP(\Delta)x,y \rangle
\qquad \text{ for all Borel sets } \Delta, \text{ for all } x, y \in H, \\
& \langle P(\Delta)ax,y \rangle = \langle P(\Delta)x,a^*y \rangle
\qquad \text{ for all Borel sets } \Delta, \text{ for all } x, y \in H, \\
& \langle Pax,y \rangle = \langle Px,a^*y \rangle
\qquad \text{ for all } x, y \in H, \\
& \int \wht{b} \,\diff \langle Pax,y \rangle
= \int \wht{b} \,\diff \langle Px,a^*y \rangle
\qquad \text{ for all } b \in B, \text{ for all } x, y \in H, \\
& \langle bax,y \rangle = \langle bx,a^*y \rangle
\qquad \text{ for all } b \in B, \text{ for all } x, y \in H, \\
& \langle bax,y \rangle = \langle abx,y \rangle
\qquad \text{ for all } b \in B, \text{ for all } x, y \in H, \\
& ba = ab \qquad \text{ for all } b \in B, \\
& a \in B'.
\end{align*}
The proof is complete. 
\end{proof}

\medskip
Our next aim is the Spectral Theorem for a normal
bounded linear operator \ref{spthmnormalbded}.
We shall need the following result.

\begin{theorem}[Fuglede - Putnam - Rosenblum]%
\index{concepts}{Theorem!Fuglede-Putnam-Rosenblum}%
\index{concepts}{Fuglede-Putn.-Ros.\ Thm.}\label{FPR}%
Let $A$ be a \linebreak pre-C*-algebra. Let $n_1, n_2$ be normal
elements of $A$. Let $a \in A$ and assume that $an_1 = n_2a$.
We then also have $a{n_1}^* = {n_2}^*a$.

In particular, if $n$ is a normal bounded linear operator on a
Hilbert space, then $\{n\}' = \{n^*\}'$.

(Please note that this latter fact is trivial if the operator $n$ is
\linebreak Hermitian or unitary, cf.\ \ref{comm2}.)
\end{theorem}

\begin{proof}
We may assume that $A$ is a C*-algebra.
Please note that it follows from the hypothesis that $a{n_1}^k = {n_2}^ka$
for all integers $k \geq 0$, so if $p \in \mathds{C}[z]$, then
$a p(n_1) = p(n_2)a$. It follows that
\[ a \,\exp(\iu \overline{z}n_1) = \exp(\iu \overline{z}n_2) \,a \]
for all $z \in \mathds{C}$, or equivalently
\[ a = \exp(\iu \overline{z}n_2) \,a \,\exp(-\iu \overline{z}n_1). \]
Since $\exp(x+y) = \exp(x)\exp(y)$ when $x$ and $y$ commute, \pagebreak
the normality of $n_1$ and $n_2$ implies
\begin{align*}
 & \exp\bigl(\iu z{n_2}^*\bigr) \,a \,\exp\bigl(-\iu z{n_1}^*\bigr) \\
= & \ \exp\bigl(\iu z{n_2}^*\bigr)\exp\bigl(\iu \overline{z}n_2\bigr)
\,a \,\exp\bigl(-\iu \overline{z}n_1\bigr)\exp\bigl(-\iu z{n_1}^*\bigr) \\
= & \ \exp\bigl(\iu (z{n_2}^*+\overline{z}n_2)\bigr)
\,a \,\exp\bigl(-\iu (\overline{z}n_1+z{n_1}^*)\bigr)
\end{align*}
The expressions $z{n_2}^*+\overline{z}n_2$ and
$\overline{z}n_1+z{n_1}^*$ are Hermitian, and hence \linebreak
$\exp\bigl(\iu (z{n_2}^*+\overline{z}n_2)\bigr)$ and
$\exp\bigl(-\iu (\overline{z}n_1+z{n_1}^*)\bigr)$ are unitary. This implies
\[ \| \,\exp\bigl(\iu z{n_2}^*\bigr) \,a \,\exp\bigl(-\iu z{n_1}^*\bigr) \,\| \leq \| \,a \,\| \]
for all $z \in \mathds{C}$. However
\[ z \mapsto \exp\bigl(\iu z{n_2}^*\bigr) \,a \,\exp\bigl(-\iu z{n_1}^*\bigr) \]
is an entire function, hence constant by Liouville's Theorem, so that
\[ \exp\bigl(\iu z{n_2}^*\bigr) \,a = a \,\exp\bigl(\iu z{n_1}^*\bigr)
\quad \text{for all} \quad z \in \mathds{C}. \]
By equating the coefficients for $z$, we obtain ${n_2}^*a = a{n_1}^*$.
\end{proof}

\medskip
(We remark in parentheses that the above result immediately carries
over to \st-semisimple normed \st-algebras, cf.\ \ref{stsemipreC*}.)

\begin{theorem}%
[the Spectral Theorem for normal bounded linear operators]%
\label{spthmnormalbded}%
\index{concepts}{Theorem!Spectral!normal bounded operator}%
\index{concepts}{spectral!theorem!normal bounded operator}%
\index{concepts}{spectral!resolution!normal bounded operator}%
Let $b$ be a normal bounded linear operator on a Hilbert space
$H \neq \{0\}$. There then is a unique resolution of identity $P$
on $\s(b)$, acting on $H$, such that
\[ b = \int _{\text{\small{$\s(b)$}}} \id _{\text{\small{$\s(b)$}}} \,\diff P
\quad \text{(in norm)}. \]
This is usually written as
\[ b = \int z \,\diff P(z). \]
One says that $P$ is the \underline{spectral resolution} of $b$. We have
\[ \{ \,b \,\}' = P', \]
and the support of $P$ is all of $\s(b)$. \pagebreak
\end{theorem}

\begin{proof} Let $B$ be the C*-subalgebra of $\blop(H)$ generated by
$b$, $b^*$, and $\mathds{1}$. It is commutative as $b$ is normal. By
\ref{spechomeom}, we may identify $\Delta(B)$ with $\s(b)$ in such a way
that $\wht{b}$ becomes $\id _{\text{\small{$\s(b)$}}}$, and existence follows.
On the other hand, if $Q$ is another resolution of identity on $\s(b)$ with
\[ b = \int z \,\diff Q(z), \]
then
\[ p(b,b^*) = \int p(z,\overline{z}) \,\diff Q(z) \]
for all $p \in \mathds{C}[z,\overline{z}]$. Uniqueness follows then from
the Stone-Weierstrass Theorem \ref{StW}. The statement concerning
the commutants follows from the preceding theorem \ref{FPR}.
\end{proof}

\begin{addendum}[the Borel functional calculus]%
\index{concepts}{function!of an operator}\label{Borelfunct}%
\index{concepts}{Borel!functional calculus}%
\index{concepts}{Borel!function!bounded}%
\index{concepts}{calculus!Borel functional}%
Let $b$ be a normal bounded linear operator on a Hilbert space
$H \neq \{0\}$. Let $P$ be the spectral resolution of $b$. For a
bounded Borel function $f$ on $\s(b)$, one puts
\[ f(b) := \pi_P(f) = \int f \,\diff P \quad \text{(in norm)}, \]
and one says that $f(b)$ is a \underline{bounded Borel function} of $b$.
The C*-algebra isomorphism
\begin{align*}
\rmLeb ^{\infty}(P) & \to \range(\pi_P) \\
f+\mathrm{N}(P) & \mapsto f(b) = \pi _P(f)
\end{align*}
is called the \underline{Borel functional calculus} for $b$,
cf.\ \ref{LPmain}. This mapping extends the operational calculus
$C\bigl(\s(b)\bigr) \to \blop(H)$, cf.\ \ref{opcalc}, \ref{suppimbed}.
\end{addendum}

\begin{remark}\label{rangespres}%
The range of the operational calculus for $b$ is the \linebreak
C*-subalgebra of $\blop(H)$ generated by $b$, $b^*$, and $\mathds{1}$.
We shall give a similar characterisation of the range of the Borel
functional calculus for $b$, see \ref{PBorel} \& \ref{functions} below.
\end{remark}

A point $x$ of a topological space is called isolated, if $\{ x \}$ is open.
\index{concepts}{isolated point}

\begin{theorem}[eigenvalues]\label{eival}%
\index{concepts}{eigenvalue}\index{concepts}{eigenspace}%
Let $b$ be a normal bounded linear operator on a Hilbert space $H \neq \{0\}$.
Let $P$ be its spectral resolution. A complex number $\lambda \in \s(b)$
is an eigenvalue of $b$ if and only if $P\bigl(\{ \,\lambda \,\}\bigr) \neq 0$, i.e.\ if
$P$ has an atom at $\lambda$, in which case the range of
$P\bigl(\{ \,\lambda \,\}\bigr)$ is the eigenspace corresponding to the
eigenvalue $\lambda$.

Hence an isolated point of the spectrum of $b$ is an eigenvalue of $b$.%
\pagebreak%
\end{theorem}

\begin{proof}
Assume first that $P\bigl(\{ \lambda \}\bigr) \neq 0$, and let
$0 \neq x \in H$ with $P\bigl(\{ \lambda \}\bigr) \,x = x$. For such $x$,
we have by \ref{rep} (ii) that
\begin{align*}
b\,x & = \pi_P(\mathrm{id}) \,x
= \pi_P(\mathrm{id}) \,P(\{ \lambda \}) \,x \\
 & = \Bigl[ \,\pi_P \,\bigl( \,\mathrm{id} \cdot
 1_{\text{\small{$\s(b) \setminus \{ \lambda \}$}}} \,\bigr)
+ \pi_P \,\bigl( \,\mathrm{id} \cdot 1_{\text{\small{$\{ \lambda \}$}}} \,\bigr) \,\Bigr]
\,P(\{ \lambda \}) \,x \\
 & = \Bigl[ \,b \,P \,\bigl( \,\s(b) \setminus \{ \lambda \} \,\bigr)
+ \lambda \,P ( \{ \lambda \} ) \,\Bigr] \,P(\{ \lambda \}) \,x
= \lambda \,P\bigl(\{ \lambda \}\bigr)^2\,x = \lambda \,x,
\end{align*}
so that $x$ is an eigenvector of $b$ to the eigenvalue $\lambda$.
Conversely, let $x \neq 0$ be an eigenvector of $b$ to the eigenvalue $\lambda$.
For $n \geq 1$, consider the function $f_n$ given by
$f_n(z) = (\lambda - z)^{-1}$ if $| \,\lambda -z \,| > \frac1n$,
and $f_n(z) = 0$ if $| \,\lambda - z \,| \leq \frac1n$. Then
\[ \pi _P(f_n)(\lambda \mathds{1} - b) = \pi _P \bigl(f_n(\lambda - \id)\bigr)
= P\bigl(\bigl\{ \,z : | \,z - \lambda \,| > \textstyle{\frac1n} \,\bigr\}\bigr). \]
We get $P\bigl(\bigl\{ \,z : | \,z - \lambda \,| > \textstyle{\frac1n} \,\bigr\}\bigr) \,x = 0$
as $(\lambda \mathds{1} - b)x = 0$.
Letting $n \to \infty$, we obtain $P\bigl(\{ \,z : z \neq \lambda \,\}\bigr) \,x = 0$,
cf.\ \ref{strsgcont}, whence $P(\{ \lambda \}) \,x = x$.
The last statement follows as the support of $P$ is all of $\s(b)$.
\end{proof}

\begin{definition}[$\Delta^*(A)$]%
\index{symbols}{D(A)@$\Delta^*(A)$}\label{Deltastar}%
If $A$ is a \st-algebra, we denote the set of Hermitian multiplicative
linear functionals on $A$ by $\Delta^*(A)$ and equip it
with the weak* topology, cf.\ the appendix \ref{weak*top}.
\end{definition}

\begin{definition}[$\s(\pi), \pi^*$]%
\label{specpi}\index{symbols}{s3@$\protect\s(\pi)$}%
Let $A$ be a commutative \st-algebra. Let $\pi$ be a non-zero
representation of $A$ on a Hilbert space $H \neq \{0\}$.
Let $B$ denote the closure of $\range(\pi)$ in $\blop(H)$. We define
\[ \s(\pi) := \{ \,\tau \circ \pi \in \mathds{C}^{\,\text{\Small{$A$}}} : \tau \in \Delta(B) \,\}. \]
Then $\s(\pi) \neq \varnothing$ is a subset of $\Delta^*(A)$ with $\s(\pi) \cup \{0\}$
weak* compact. The ``adjoint'' map
\begin{alignat*}{2}\index{symbols}{p7@$\pi^*$}
\pi^* : \Delta(B) & \to &\ & \s(\pi) \\
             \tau & \mapsto &\ & \tau \circ \pi.
\end{alignat*}
is a homeomorphism.
\end{definition}

\begin{proof}
For $\tau \in \Delta(B)$, the functional $\tau \circ \pi$ is non-zero
by density of $\range(\pi)$ in $B$. We have
$\varnothing \neq \s(\pi) \subset \Delta^*(A)$ by \ref{C*GTisom}
and \ref{mlfHerm}. Extend the ``adjoint'' map $\pi^*$ to a map
$\Delta(B) \cup \{0\} \to \s(\pi) \cup \{0\}$, taking $0$ to $0$. The
extended map is bijective and continuous by the universal property
of the weak* topology, cf.\ the appendix \ref{weak*top}. Since the
domain $\Delta(B) \cup \{0\}$ is compact \ref{mlf0compact}, and the
range $\s(\pi) \cup \{0\}$ is Hausdorff, it follows that the extended
map is a homeomorphism, cf.\ the appendix \ref{homeomorph}.
In particular, $\s(\pi) \cup \{0\}$ is compact. \pagebreak
\end{proof}

\medskip
The following item is applicable to $\Omega := \s(\pi)$
with $\s(\pi)$ as in the preceding definition \ref{specpi}.
We stress that $\s(\pi)$ will be a locally compact Hausdorff
space $\neq \varnothing$.

\begin{proposition}\label{cloloc}%
Let $A$ be a \st-algebra. Let $\Omega \neq \varnothing$ be a subset of
$\Delta^*(A)$ with $\Omega \cup \{0\}$ weak* compact. Then $\Omega$
is a closed and locally compact subset of $\Delta^*(A)$. The functions
$\wht{a}| _{\text{\small{$\Omega$}}}$ $(a \in A)$ form a dense subset of
$\cont\0(\Omega)$.
\end{proposition}

\begin{proof}
The set $\Omega$ is closed in $\Delta^*(A)$ as it is the intersection
of $\Delta^*(A)$ with the compact Hausdorff space $\Omega \cup \{0\}$.
We also see that $\Omega$ is locally compact because it differs from
the compact Hausdorff space $\Omega \cup \{0\}$ only in the point $0$.
Also, $\Omega \cup \{0\}$ is a one-point compactification of $\Omega$,
if $\Omega$ is not compact. The evaluations $\wht{a}$ $(a \in A)$
are continuous on $\Omega \cup \{0\}$ by definition of the weak* topology,
cf.\ the appendix \ref{weak*top}. Since the evaluations $\wht{a}$ $(a \in A)$
vanish at the point $0$, we conclude that the functions
$\wht{a}| _{\text{\small{$\Omega$}}}$ $(a \in A)$ actually are in
$\cont\0(\Omega)$. These functions are dense in $\cont\0(\Omega)$ by
the Stone-Weierstrass Theorem \ref{StW}. (Please note that this uses the
fact that the functionals in $\Omega$ are Hermitian.)
\end{proof}

\begin{remark}\label{sppisa}%
Let $A$ be a commutative \st-algebra. Let $\pi$ be a non-zero
representation of $A$ on a Hilbert space. Since a multiplicative
linear functional on a Banach algebra is contractive \ref{mlfbounded},
we have the following facts. If $\pi$ is weakly continuous, then
each $\tau \in \s(\pi)$ is contractive, cf.\ \ref{weakcontHilbert}.
Also, $\pi$ is weakly continuous on $A\sa$, then each $\tau \in \s(\pi)$
is contractive on $A\sa$, cf.\ \ref{weaksa}, \ref{sigmaAsa}. In this case
$\s(\pi) \subset \Delta^*\bsa(A)$ as topological spaces, cf.\ \ref{topDbsa}
\& \ref{Deltastar}. (As the \st-algebra $A$ is spanned by $A\sa$,
cf.\ the appendix \ref{weak*point}.)
\end{remark}

\begin{observation}%
Let $A$ be a commutative \st-algebra. Let $\pi$ be a non-zero
representation of $A$ on a Hilbert space. For $a \in A$, we get
\[ \s\bigl(\pi(a)\bigr) \setminus \{0\} = \wht{a}\,\bigl(\s(\pi)\bigr) \setminus \{0\}, \]
whence, by \ref{C*rl},
\[ \| \,\pi(a) \,\| = \bigl| \,\wht{a}| _{\text{\small{$\s(\pi)$}}} \,\bigr|_{\infty}. \]
\end{observation}

\begin{proof}
If $B$ denotes the closure of $\range(\pi)$, and if $a \in A$, then
\begin{align*}
\s\bigl(\pi(a)\bigr) \setminus \{0\}
& = \wht{\pi (a)}\,\bigl(\Delta(B)\bigr) \setminus \{0\} \qquad \text{by \ref{rangeGT}} \\
& = \wht{a}\,\bigl(\s(\pi)\bigr) \setminus \{0\}. \pagebreak \qedhere
\end{align*}
\end{proof}

\begin{theorem}[the Spectral Theorem for representations]%
\label{spthmrep}%
\index{concepts}{Theorem!Spectral!representation}%
\index{concepts}{spectral!theorem!representation}%
\index{concepts}{spectral!resolution!representation}%
Let $\pi$ be a non-degenerate representation of a commutative \st-algebra
$A$ on a \linebreak Hilbert space $H \neq \{0\}$. There then exists a unique
resolution of identity $P$ on $\s(\pi)$, acting on $H$, such that
\[ \pi(a) = \int _{\text{\small{$\s(\pi)$}}} \wht{a}|_{\text{\small{$\s(\pi)$}}} \,\diff P
\quad \text{(in norm)} \qquad \text{for all} \quad a \in A. \]
One says that $P$ is the \underline{spectral resolution} of $\pi$. We have
\[ \pi' = P', \]
and the support of $P$ is all of $\s(\pi)$.
\end{theorem}

\begin{proof} Uniqueness follows from the fact that the functions
$\wht{a}| _{\text{\small{$\s(\pi)$}}}$ $(a \in A)$ are dense in
$\cont\0\bigl(\s(\pi)\bigr)$, cf.\ \ref{cloloc}. For existence, consider the closure
$B$ of $\range(\pi)$. It is a commutative C*-algebra acting non-degenerately
on $H$. The spectral resolution $Q$ of $B$ is a resolution of identity on
$\Delta(B)$ such that
\begin{gather*}
b = \int \wht{b} \,\diff Q \quad \text{for all} \quad b \in B, \\
B' = Q',
\end{gather*}
and such that the support of $Q$ is all of $\Delta(B)$. Let $P := \pi^*(Q)$
be the image of $Q$ under the ``adjoint'' map
\begin{align*}
\pi^* : \Delta(B) & \to \s(\pi) \\
                \tau & \mapsto \tau \circ \pi,
\end{align*}
which is a homeomorphism, cf.\ \ref{specpi}. (For the notion of an
image of a spectral measure, recall \ref{imagespdef} \& \ref{imagespthm}.)
For $a \in A$, we have
\[ \int_{\text{\small{$\s(\pi)$}}} \wht{a}|_{\text{\small{$\s(\pi)$}}} \,\diff P
= \int \wht{a} \circ \pi^* \,\diff Q = \int \wht{\pi(a)} \,\diff Q = \pi(a), \]
by \ref{imagespthm}. Since $\range(\pi)$ is dense in $B$, we have
\[ \pi' = B' = Q' = P'. \]
Since the ``adjoint'' map $\pi^*$ is a homeomorphism, it is clear that
the support of $P$ is all of $\s(\pi)$.
\end{proof}

\medskip
The uniqueness statement in the preceding theorem \ref{spthmrep}
assumes that the resolution of identity lives on $\s(\pi)$.
A stronger uniqueness statement relaxing this assumption will be
given in addendum \ref{spuniq} below. We need the following lemma.
\pagebreak

\begin{lemma}\label{unilemma}%
Let $A$ be a commutative \st-algebra. Let $\pi$ be a non-zero
representation of $A$ on a Hilbert space $H \neq \{0\}$. We then have
\[ \s(\pi) = \{ \,\sigma \in \Delta^*(A) : | \,\sigma(a) \,| \leq \| \,\pi (a) \,\|
\ \text{for all}\ a \in A \,\}. \]
\end{lemma}

\begin{proof}
Let $B$ denote the closure of $\range(\pi)$ in $\blop(H)$. For
$\sigma \in \s(\pi)$ there exists $\tau \in \Delta(B)$ with $\sigma = \tau \circ \pi$.
By contractivity of $\tau$, cf.\ \ref{mlfbounded}, we then get
\[ | \,\sigma(a) \,| \leq \| \,\pi(a) \,\| \quad \text{ for all} \quad a \in A. \tag*{$(\st)$} \]
Conversely, let $\sigma \in \Delta^*(A)$ such that $(\st)$ holds. Then $\sigma$
vanishes on $\ker \pi$, as is seen from $(\st)$. Thus $\sigma$ factors to a
functional $\varphi \in \Delta^* \bigl( \range(\pi) \bigr)$ such that
$\varphi\bigl(\pi(a)\bigr) = \sigma (a)$ for all $a \in A$. (The functional $\varphi$
is not identically zero because $\sigma$ isn't.) The inequality $(\st)$ implies that
$\varphi$ is contractive, and so $\varphi$ has a unique continuation to a
functional $\tau \in \Delta(B)$. It is immediate that $\tau \circ \pi = \sigma$, and
consequently $\sigma \in \s(\pi)$.
\end{proof}

\begin{addendum}\label{spuniq}%
Let $\pi$ be a non-degenerate representation of a commutative \st-algebra
$A$ on a Hilbert space $H \neq \{0\}$. Assume that $\Omega$ is a subset of
$\Delta^*(A)$ with $\Omega \cup \{0\}$ weak* compact, and that $Q$ is a resolution
of identity on $\Omega$, acting on $H$, such that
\[ \pi(a) = \int _{\text{\small{$\Omega$}}} \wht{a}|_{\text{\small{$\Omega$}}} \,\diff Q
\qquad \text{(in norm)} \qquad \text{for all} \quad a \in A. \]
Assume also that the support of $Q$ is all of $\Omega$.
\par
Then $\Omega = \s(\pi)$ and $Q$ is the spectral resolution of $\pi$ as in \ref{spthmrep}.
\end{addendum}

\begin{proof}
Let $a \in A$. Then $\| \,\pi(a) \,\| = \| \,\pi_Q ( \,\wht{a}|_{\text{\small{$\Omega$}}} ) \,\|$
equals the quotient norm of $\wht{a}|_{\text{\small{$\Omega$}}} + \mathrm{N}(Q)$
in the quotient C*-algebra $\rmLeb ^{\infty}(Q)$, cf.\ \ref{LPmain}. This value
is precisely $| \,\wht{a}|_{\text{\small{$\Omega$}}} \,| _{\,\infty}$ by continuity of
$\wht{a}|_{\text{\small{$\Omega$}}}$ and because the support of $Q$ is all of $\Omega$,
cf.\ \ref{suppimbed}. Therefore
$\bigl| \,\sigma (a) \,\bigr| = \bigl| \ \wht{a}|_{\text{\small{$\Omega$}}} (\sigma) \,\bigr|
\leq \| \,\pi(a) \,\|$ holds for all $\sigma \in \Omega$ and all $a \in A$. The preceding
lemma \ref{unilemma} implies that $\Omega \subset \s(\pi)$. We also have that
$\Omega$ is a closed subset of $\s(\pi)$, by \ref{cloloc}. We thus may consider
$Q$ as a resolution of identity on $\s(\pi)$, whose support is $\Omega$. Letting
$P$ denote the spectral resolution of $\pi$, we get $Q = P$ by the
fact that the functions $\wht{a}| _{\text{\small{$\s(\pi)$}}}$ $(a \in A)$
are dense in $\cont\0\bigl(\s(\pi)\bigr)$, cf.\ \ref{cloloc}. Hence
$\Omega = \mathrm{supp}(Q) = \mathrm{supp}(P) = \s(\pi)$.
\end{proof}

\medskip
We remark that we cannot take all of $\Delta ^* (A)$ as the space on
which a resolution of identity lives, as it can fail to be locally compact.

The hurried reader can go from here to chapter \ref{unbdself}
on unbounded self-adjoint operators. \pagebreak

\clearpage

%


\section{Von Neumann Algebras}

Let throughout $H$ denote a Hilbert space.

\medskip
We now come to second commutants in $\blop(H)$.
We encountered second commutants already in \ref{questimbed}.

\begin{reminder}\index{concepts}{commutant}%
Let $S \subset \blop(H)$. The set of all operators in $\blop(H)$ which
commute with every operator in $S$ is called the \underline{commutant}
of $S$ and is denoted by $S'$. Please note that $S \subset S''$.
\end{reminder}

\begin{observation}%
For two subsets $S$, $T$ of $\blop(H)$, we have
\[ T \subset S \Rightarrow S' \subset T'. \]
\end{observation}

\begin{observation}\label{thirdcomm}%
For a subset $S$ of $\blop(H)$, we have
\[ S''' = S'. \]
\end{observation}

\begin{proof}
From $S \subset S''$ it follows that $(S'')' \subset S'$.
Also $S' \subset (S')''$.
\end{proof}

\medskip
Recall that a subset of a \st-algebra is said to be \underline{\st-stable}
(or self-adjoint), if with each element $a$, it also contains the adjoint
$a^*$, cf.\ \ref{selfadjointsubset}.

\begin{observation}%
Let $S$ be a subset of $\blop(H)$. If $S$
is \st-stable, then $S'$ is a C*-algebra.
\end{observation}

\begin{definition}%
\index{concepts}{algebra!von Neumann}%
\index{concepts}{von Neumann!algebra}%
A \underline{von Neumann algebra}
on $H$ is a subset of $\blop(H)$ which is the commutant of a
\st-stable subset of $\blop(H)$.
\end{definition}

\begin{observation}%
If $H \neq \{0\}$, then each von Neumann algebra
on $H$ is a unital C*-subalgebra of $\blop(H)$.
\end{observation}

\begin{example}\label{exvN}%
We have that $\blop(H)$ is a von Neumann algebra as
$\blop(H) = (\mathds{C}\mathds{1})'$.
We have that $\mathds{C}\mathds{1}$ is a von Neumann algebra
because $\mathds{C}\mathds{1} = \blop(H)'$, cf.\ \ref{BHirred}.
If $\pi$ is a representation of a \st-algebra on $H$,
then $\pi'$ is a von Neumann algebra.
\end{example}

\begin{remark}
\index{concepts}{algebra!W*-algebra}\index{concepts}{W*-algebra}%
A \underline{W*-algebra} is a C*-algebra which is isomorphic
as a C*-algebra to a von Neumann algebra, see
\cite[vol.\ I, Def.\ III.3.1, p.\ 130]{Tak}.
The pronuniation is ``weak star''.%
\pagebreak
\end{remark}

Von Neumann algebras come in pairs:

\begin{theorem}\label{prepchar}%
If $\mathscr{M}$ is a von Neumann algebra, then
\begin{itemize}
  \item[$(i)$] $\mathscr{M}'$ is a von Neumann algebra,
 \item[$(ii)$] $\mathscr{M}'' = \mathscr{M}$.
\end{itemize}
\end{theorem}

\begin{proof}
This follows from observation \ref{thirdcomm}.
\end{proof}

The property in the following corollary is often taken as the
definition of von Neumann algebras. 

\begin{corollary}\index{concepts}{commutant!second}%
A \st-subalgebra $A$ of $\blop(H)$ is a von
Neumann algebra if and only if $A'' = A$.
\end{corollary}

A von Neumann algebra is closed under pointwise limits
in the following sense.

\begin{proposition}\label{ptwcl}%
Let $\mathscr{M}$ be a von Neumann algebra on
$H$. Let $(a_i)_{i \in I}$ be a net in $\mathscr{M}$,
which converges pointwise on $H$ to some limit
$a \in \blop(H)$, that is,
\[ a \mspace{2mu} x = \lim_{i \in I} a_i \mspace{2mu} x
\quad \text{for all} \quad x \in H. \]
The pointwise limit $a$ then belongs to $\mathscr{M}$ as well.
\end{proposition}

\begin{proof}
The von Neumann algebra $\mathscr{M}$ is the commutant
$S'$ of some \st-stable subset $S$ of $\blop(H)$. We have that
the $a_i$ $(i \in I)$ all are in $\mathscr{M} = S'$. We get that also
$a \in S' = \mathscr{M}$, because for $c \in S$ and $x \in H$,
we compute:
\[ a \mspace{2mu} c \mspace{2mu} x
= \lim_{i \in I} \,a_i \mspace{2mu} c \mspace{2mu} x
= \lim_{i \in I} \,c \mspace{2mu} a_i \mspace{2mu} x
= c \mspace{2mu} a \mspace{2mu} x. \qedhere \]
\end{proof}

\begin{corollary}\label{ordclosed}%
Let $\mathscr{M}$ be a von Neumann algebra on $H$.
Let $(a_i)_{i \in I}$ be an increasing net in $\mathscr{M}\sa$,
which is upper bounded in $\blop(H)\sa$. The least upper bound
$\sup_{i \in I} \,a_i$ in $\blop(H)\sa$ then also belongs to $\mathscr{M}\sa$.
\end{corollary}

\begin{proof}
Theorem \ref{ordercompl} shows that the net $(a_i)_{i \in I}$
converges pointwise on $H$ to its least upper bound
$\sup_{i \in I} a_i$ in $\blop(H)\sa$. The least upper bound
$\sup_{i \in I} a_i$ then belongs to $\mathscr{M}\sa$ by
the preceding proposition \ref{ptwcl}. \pagebreak
\end{proof}

\medskip
This says that if $\mathscr{M}$ is a von Neumann algebra on $H$,
then $\mathscr{M}\sa$ is ``monotone closed'' in $\blop(H)\sa$.
In particular, $\mathscr{M}\sa$ then is monotone complete.

\begin{proposition}\label{prepgen}%
If $S$ is a \st-stable subset of $\blop(H)$, then the set $S''$
is the smallest von Neumann algebra on $H$ containing $S$.
\end{proposition}

\begin{proof}
The set $S''$ is a von Neumann algebra on $H$ containing $S$.
If $\mathscr{M}$ is a von Neumann algebra on $H$ containing $S$,
then $S \subset \mathscr{M}$ whence $\mathscr{M}' \subset S'$
and thus $S'' \subset \mathscr{M}'' = \mathscr{M}$.
\end{proof}

\begin{definition}\index{symbols}{W*(1)@$\vonNeumAlg (S)$}%
\index{symbols}{W*(2)@$\vonNeumAlg (a)$}%
\index{symbols}{W*(3)@$\vonNeumAlg (\pi)$}%
If $S$ is an arbitrary subset of $\blop(H)$, define
\[ S^* := \{ \,a^* \in \blop(H) : a \in S \,\}. \]
It follows from the preceding proposition \ref{prepgen} that
\[ \vonNeumAlg (S) := (S \cup S^*)'' \]
is the smallest von Neumann algebra on $H$ containing $S$.
One says that $\vonNeumAlg (S)$ is
\underline{the von Neumann algebra generated by $S$}.
If $a$ is a bounded linear operator on $H$, we put
\[ \vonNeumAlg (a) := \vonNeumAlg (\{a\}). \]
If $\pi$ is a representation of a \st-algebra
on a Hilbert space, we denote
\[ \vonNeumAlg (\pi) := \vonNeumAlg \bigl(\range(\pi)\bigr), \]
so that then $\vonNeumAlg (\pi) = \range(\pi)'' = (\pi')'$.
\end{definition}

\begin{proposition}\label{Wstarcomm}%
The von Neumann algebra generated by a \linebreak normal
subset \ref{selfadjointsubset} of $\blop(H)$ is commutative.
\end{proposition}

\begin{proof}
See \ref{scndcobs}.
\end{proof}

\begin{proposition}\label{W*normal}%
If $b \in \blop(H)$ is \underline{normal}, then $\vonNeumAlg (b) = \{ b \}''$.
\end{proposition}

\begin{proof}
The Fuglede-Putnam-Rosenblum Theorem \ref{FPR} says
that $\{ b \}' = \{ b^* \}'$, whence $\vonNeumAlg (b) = \{ b, b^* \}'' = \{ b \}''$.
\end{proof}

\begin{proposition}%
A non-zero representation $\pi$ of a \st-algebra on $H$
is irreducible if and only if $\vonNeumAlg (\pi) = \blop(H)$.
\end{proposition}

\begin{proof}
This follows from Schur's Lemma \ref{Schur} and \ref{BHirred}. \pagebreak
\end{proof}

\begin{proposition}\label{spprojsubset}%
Let $\pi$ be a non-degenerate representation of a commutative \st-algebra
on a Hilbert space $\neq \{0\}$. Let $P$ be the spectral resolution of $\pi$.
We then have $\range(\pi _P) \subset \vonNeumAlg (\pi)$. In particular, for
the spectral projections, we have $P \subset \vonNeumAlg (\pi)$.
\end{proposition}

\begin{proof}
We have
\[ {\pi _P}' = P' = {\pi}', \]
see \ref{spprojcomm} \& \ref{spthmrep}, whence
\[ \range(\pi _P) \subset \range(\pi _P)'' = ({\pi _P}')' = (\pi ')' = \vonNeumAlg (\pi).
\qedhere \]
\end{proof}

\begin{proposition}\label{gensame}%
Let $\pi$ be a non-degenerate representation of a commutative \st-algebra
on a Hilbert space $\neq \{0\}$. Let $P$ be the spectral \linebreak resolution
of $\pi$. Then $\pi _P$ and $\pi$ generate the same von Neumann
\linebreak algebra. Furthermore we have $\vonNeumAlg (\pi) = \vonNeumAlg (P)$.
\end{proposition}

\begin{proof} This follows again from ${\pi _P}' = P' = \pi'$. \end{proof}

\medskip
Please note that $\range(\pi _P)$ is merely the C*-algebra generated
by $P$, cf.\ \ref{spanP}. On a non-separable Hilbert space, it can happen
that $\vonNeumAlg (\pi)$ is strictly larger than $\range(\pi _P)$, cf.\ \ref{nonsep}
below.

\begin{theorem}\label{Abelian}%
Let $\mathscr{M}$ be a commutative von Neumann  algebra on a Hilbert
space $\neq \{0\}$. Let $P$ be the spectral  resolution of $\mathscr{M}$.
We then have $\mathscr{M} = \range(\pi _P)$.
\end{theorem}

\begin{proof} By the Spectral Theorem we have
$\mathscr{M} \subset \range(\pi _P)$, and by \ref{spprojsubset} we
also have $\range(\pi _P) \subset \vonNeumAlg (\mathscr{M}) = \mathscr{M}$.
\end{proof}

\medskip
A useful observation is the following one:

\begin{corollary}\label{disconn}%
Let $\mathscr{M}$ be a commutative von Neumann algebra on a Hilbert
space $\neq \{0\}$. Let $P$ be the spectral resolution of $\mathscr{M}$.
The imbedding \ref{suppimbed} of $C\bigl(\Delta(\mathscr{M})\bigr)$ into
$\rmLeb ^{\infty}(P)$ then is an isomorphism of C*-algebras. In particular,
for a bounded Borel function $f$ on $\Delta(\mathscr{M})$, there exists a
unique continuous function which is $P$-a.e.\ equal to $f$.
\end{corollary}

\begin{proof}
The isometric imbedding \ref{suppimbed} is surjective by
$\mathscr{M} = \range(\pi _P)$, cf.\ the preceding theorem
\ref{Abelian}. \pagebreak
\end{proof}

\begin{definition}[$\mathscr{P}(A)$]\label{setofproj}%
\index{symbols}{P97@$\mathscr{P}(A)$}%
For a \st-algebra $A$, we shall denote by $\mathscr{P}(A)$
the set of projections in $A$, cf.\ \ref{projrefl}.
\end{definition}

\begin{lemma}\label{W*lemma}%
Let $\mathscr{M}$ be a commutative von Neumann algebra on a Hilbert
space $\neq \{0\}$. Let $P$ be the spectral resolution of $\mathscr{M}$.
We then have $\mathscr{P}(\mathscr{M}) = P$ and
$\mathscr{M} = \overline{\mathrm{span}}\bigl(\mathscr{P}(\mathscr{M})\bigr)$.
\end{lemma}

\begin{proof}
This follows from $\mathscr{M} = \range(\pi _P) = \overline{\mathrm{span}}(P)$,
cf.\ \ref{Abelian} \& \ref{spanP}, and from the fact that every projection
in $\mathscr{M} = \range(\pi _P)$ is of the form $P(\Delta)$ for some Borel set
$\Delta$ of $\Delta(\mathscr{M})$, cf.\ \ref{spprojform}.
\end{proof}

\begin{theorem}\label{W*main}%
A von Neumann algebra $\mathscr{M}$ satisfies
\[ \mathscr{M} = \overline{\mathrm{span}}\bigl(\mathscr{P}(\mathscr{M})\bigr). \]
For emphasis: a von Neumann algebra is the closed linear span of its
\linebreak projections.
\end{theorem}

\begin{proof}
An arbitrary element $c \in \mathscr{M}$ can be decomposed as $c = a + \iu b$
with $a,b \in \mathscr{M}\sa$. It now suffices to apply the preceding lemma
\ref{W*lemma} to the commutative von Neumann algebras
$\vonNeumAlg (a), \vonNeumAlg (b) \subset \mathscr{M}$, cf.\ \ref{Wstarcomm}.
\end{proof}

\medskip
The above theorem is of fundamental importance for representation
\linebreak theory. With its help, questions concerning a von
Neumann algebra $\mathscr{M}$ can be transferred to questions
concerning the set $\mathscr{P}(\mathscr{M})$ of projections in
$\mathscr{M}$. When applied to the commutant of a representation,
this yields Schur's Lemma \ref{Schur}.

\begin{corollary}%
A von Neumann algebra $\mathscr{M}$ satisfies
\[ \mathscr{M}' = \mathscr{P}(\mathscr{M})', \]
whence also $\mathscr{M} = \vonNeumAlg \bigl(\mathscr{P}(\mathscr{M})\bigr)$.
\end{corollary}

\begin{lemma}\label{singlevec}%
Let $\pi$ be a non-degenerate representation of some \linebreak \st-algebra
$A$ on $H$. Consider $b \in \vonNeumAlg (\pi)$. Then for every $x \in H$,
there exists a sequence $(a_n)$ in $A$ with
\[ b \mspace{2mu} x = \lim_{n \to \infty} \pi(a_n) \mspace{2mu} x. \pagebreak \]
\end{lemma}

\begin{proof}
We can assume that $x \neq 0$. Consider the closed invariant subspace
$M := \overline{\range(\pi)x}$. Then $M$ is cyclic under $\pi$, with cyclic vector
$x$ belonging to $M$, cf.\ \ref{cyclicsubspace}. The projection $p$ on $M$
is in $\pi'$ by \ref{invarcommutant}, and so commutes with $b \in \vonNeumAlg (\pi)$.
This in turn implies that $M$ is invariant under $b$, by \ref{comminvar}.
Since $x \in M$, we then have $b \mspace{2mu} x \in M$, so by definition of
$M$, there exists a sequence $(a_n)$ in $A$ with
\[ b \mspace{2mu} x = \lim_{n \to \infty} \pi(a_n) \mspace{2mu} x. \qedhere \]
\end{proof}

\begin{definition}[amplifications, $\oplus _n H$, $\oplus_n b$, $\oplus_n \pi$]%
\index{concepts}{amplification}\index{symbols}{p05@$\oplus_n H$}%
\index{symbols}{p06@$\oplus_n b$}\index{symbols}{p07@$\oplus_n \pi$}%
\label{amplification}%
The direct sum $\oplus_n H$ of countably infinitely many identical copies
of $H$ is called the \underline{amplification} of $H$, cf.\ \ref{dirHil}.
Similarly , if $b$ is a bounded linear operator on $H$, one considers
the amplification $\oplus_n b$ of $b$, cf.\ \ref{dirop}. 
Also, if $\pi$ is a representation of a \st-algebra on $H$, one considers
the amplification $\oplus_n \pi$ of $\pi$, cf.\ \ref{directsumsigma}.
\end{definition}

\begin{lemma}\label{ampli}%
Let $\pi$ be a representation of a \st-algebra $A$ on $H$.
For any $b \in \vonNeumAlg (\pi)$, we have
$\oplus_n b \in \vonNeumAlg (\oplus_n \pi)$.
\end{lemma}

\begin{proof}
The operators $c$ in $\blop (\oplus_n H)$ have a ``matrix representation''
$(c_{mn})_{m,n}$ with each $c_{mn} \in \blop(H)$. If $p_n$ denotes the
projection of $\oplus_n H$ on the n-th term, then the
``matrix entries'' of $c \in \blop (\oplus_n H)$ are given by
$c_{mn} := p_m \mspace{2mu} c \mspace{3mu} p_n \in \blop(H)$.

We claim that an operator $c \in \blop (\oplus_n H)$ belongs to $(\oplus_n \pi)'$
if and only if $c_{mn} \in \pi'$ for all $m, n$.
Indeed, for all $m, n$ and all $a \in A$, we compute
\begin{align*}
 & \ p_m \bigl( c \oplus_n \pi(a) - \oplus_n \pi(a) c \bigr) p_n \\
 = & \ p_m c p_n \pi(a) - \pi(a) p_m c p_n \\
 = & \ c_{mn} \pi(a) - \pi(a) c_{mn}.
\end{align*}
(As $\oplus_n \pi(a)$ is sort of a block diagonal matrix.)
Our claim is proved by looking at when the above expression
vanishes for all $a \in A$.

Let now $b \in \vonNeumAlg (\pi)$. Then for arbitrary $c \in (\oplus_n \pi)'$ the
operator $b$ commutes with all $c_{mn}$ because the latter are in $\pi'$ by the
above. It follows that $\oplus_n b$ commutes with $c$, as is shown as follows.
For all $m, n$ we have
\[ p_m ( \oplus_n b \mspace{2mu} c ) p_n = b c_{mn}
= c_{mn} b = p_m ( c \mspace{2mu} {\oplus_n b} ) p_n. \]
(Again because $\oplus_n b$ is sort of a block diagonal matrix.)
We obtain that $\oplus_n b \in \vonNeumAlg (\oplus_n \pi)$ as $c \in (\oplus_n \pi)'$
was arbitrary. \pagebreak
\end{proof}

\begin{lemma}\label{prevN}%
Let $\pi$ be a non-degenerate representation of some \linebreak \st-algebra
$A$ on $H$. Consider $b \in \vonNeumAlg (\pi)$. Then for any sequence
$(x_n)$ in $H$ with $\bigl( \,\| \,x_n \,\| \,\bigr) \in \ell^{\,2}$, and any
$\varepsilon > 0$, there exists an element $ a \in A$ with
\[ {\biggl( \,\sum _n \,{\bigl\| \,\bigl( \,b - \pi(a) \,\bigr) \,x_n \,\bigr\|\,}^2 \,\biggl)}^{\,1/2}
< \varepsilon. \]
\end{lemma}

\begin{proof}
Put $x := \oplus _n x_n \in \oplus _n H$ and consider
$\oplus_n b \in \vonNeumAlg (\oplus_n \pi)$, cf.\ the preceding lemma \ref{ampli}.
By lemma \ref{singlevec}, there exists an element $a \in A$ with
\[ \bigl\| \,( \,\oplus _n \,b \,) \,x - \bigl( \,\oplus _n \,\pi \,(a) \,\bigr) \,x \,\bigr\|
< \varepsilon. \qedhere \]
\end{proof}

\medskip
The remainder of this paragraph won't be used in the sequel.

\begin{definition}[four topologies on $\blop(H)$]%
\index{concepts}{topology!strong operator}%
The \underline{strong operator} \underline{topology} on $\blop(H)$ is the weak
topology \ref{weaktopdef} induced by the seminorms
\[ \blop (H) \ni a \mapsto \| \,ax \,\| \qquad ( \,x \in H \,). \]

\index{concepts}{topology!weak operator}%
The \underline{weak operator topology} on $\blop(H)$
is the weak topology induced by the linear functionals
\[ \blop (H) \ni a \mapsto \langle ax,y \rangle \qquad ( \,x, y \in H \,). \]

The weak operator topology on $\blop(H)$ is weaker than the strong operator
topology on $\blop(H)$ which in turn is weaker than the norm topology.%

\index{concepts}{topology!1sigma-strong@$\sigma$-strong}%
The \underline{$\sigma$-strong topology} on $\blop(H)$ is the weak
topology induced by the seminorms
\[ \blop (H) \ni a \,\mapsto \,{\biggl( \,\sum _n {\,\| \,ax_n \,\|\,}^2 \biggr)}^{1/2} \]
with $(x_n)$ any sequence in $H$ such that
$\bigl( \,\| \,x_n \,\| \,\bigr) \in {\ell\,}^2$.

\index{concepts}{topology!1sigma-weak@$\sigma$-weak}%
The \underline{$\sigma$-weak topology} on $\blop(H)$ is the weak
topology induced by the linear functionals
\[ \blop (H) \ni a \,\mapsto \,\sum _n \,\langle ax_n,y_n \rangle \]
with $(x_n)$, $(y_n)$ any sequences in $H$
such that $\bigl( \,\| \,x_n \,\| \,\bigr), \bigl( \,\| \,y_n \,\| \,\bigr) \in {\ell\,}^2$.

The $\sigma$-weak topology is weaker than the $\sigma$-strong topology
which in turn is weaker than the norm topology.

Also, the weak operator topology is weaker than the $\sigma$-weak
topology, and similarly the strong operator topology is weaker that the
$\sigma$-strong topology. \pagebreak
\end{definition}

\begin{theorem}[von Neumann's Bicommutant Theorem]%
\index{concepts}{Theorem!von Neumann!Bicommutant}\label{bicomm}%
\index{concepts}{Bicommutant Theorem}%
\index{concepts}{von Neumann!Theorem!Bicommutant}%
Let $\pi$ be a non-degenerate representation of a \st-algebra on $H$.
Then $\vonNeumAlg (\pi)$ is the closure of $\range(\pi)$ in any of these four
topologies: strong operator topology, weak operator topology, $\sigma$-strong
topology, and $\sigma$-weak topology. 
\end{theorem}

\begin{proof}
We shall first prove that $\vonNeumAlg (\pi)$ is closed in the weak operator topology.
It will then be closed in all of the four topologies of the statement, because
the weak operator topology is the weakest among them. So let $(b_i)_{i \in I}$
be a net in $\vonNeumAlg (\pi)$ converging to some $b \in \blop(H)$ in the weak operator
topology. We have to show that $b \in \vonNeumAlg (\pi)$. So let $c \in \pi'$. In order to
prove that $b$ commutes with $c$, we note that for all $x, y \in H$, one has,
by the appendix \ref{netconv}:
\begin{align*}
 \langle bcx,y \rangle & = \lim_{i \in I} \,\langle b_i cx, y \rangle \\
 & = \lim_{i \in I} \,\langle c b_i x, y \rangle \\
 & = \lim_{i \in I} \,\langle b_i x, {c}^{*} y \rangle \\
 & = \langle b x, {c}^{*} y \rangle = \langle cbx,y \rangle.
\end{align*}
It now suffices to prove that $\range(\pi)$ is dense in $\vonNeumAlg (\pi)$ in the
$\sigma$-strong operator topology, as the $\sigma$-strong topology is the
strongest among the four topologies of the statement. So let $b \in \vonNeumAlg (\pi)$.
Let $\varepsilon > 0$ and finitely many sequences $(y_{m,n})_n$ in $H$
with $\bigl( \,\| \,y_{m,n} \,\| \,\bigr)_n \in {\ell\,}^2$ for each $m$ be given.
It is enough to prove that there exists $a \in A$ with
\[ {\biggl( \,\sum _n \,{\bigl\| \,\bigl( \,b - \pi(a) \,\bigr) \,y_{m,n} \,\bigr\|\,}^2 \,\biggl)}^{\,1/2}
< \varepsilon. \]
for all $m$. (By the form of a base for this topology, cf.\ the appendix \ref{weaktopdef}.)
This however follows by applying the preceding lemma \ref{prevN} to a sequence
$(x_n)_n$ in $H$ with $\bigl( \,\| \,x_n \,\| \,\bigr) \in {\ell\,}^2$ which incorporates
each of the finitely many sequences $(y_{m,n})_n$. (A finite product of countable sets
is countable.)
\end{proof}

\begin{corollary}%
Let $A$ be a \st-subalgebra of $\blop(H)$ acting non-degenerately on $H$.
Then $\vonNeumAlg (A)$ is the closure of $A$ in any of the following four topologies:
strong operator topology, weak operator topology, $\sigma$-strong
topology, and $\sigma$-weak topology.
\end{corollary}

We obtain the following characterisation of von Neumann algebras.\pagebreak

\begin{corollary}%
The von Neumann algebras on $H$ are precisely the \st-subalgebras of
$\blop(H)$ acting non-degenerately on $H$, which are closed in any, hence all,
of these four topologies: strong operator topology, weak operator topology,
$\sigma$-strong topology, and $\sigma$-weak topology.
\end{corollary}

\begin{remark}%
Some authors define von Neumann algebras as \linebreak \st-algebras
of bounded linear operators on a Hilbert space, which are closed in the
strong operator topology, say. This is a more inclusive definition than ours,
but one can always put oneself in our situation by considering a smaller
Hilbert space, cf.\ \ref{nondegenerate}.
\end{remark}

We mention here that among the four topologies considered above,
the $\sigma$-weak topology is of greatest interest, because it actually
is a weak* topology resulting from a predual of $\blop(H)$, namely the
Banach \linebreak \st-algebra of trace class operators on $H$,
equipped with the trace-norm.

\clearpage

%


\section{Multiplication Operators}

\begin{definition}[$\probmeas (\Omega)$]%
If $\Omega \neq \varnothing$ is a locally compact Hausdorff space,
one denotes by $\probmeas (\Omega)$ the set of inner regular
Borel probability measures on $\Omega$, cf.\ the appendix
\ref{Boreldef} \& \ref{inregBormeas}.
\end{definition}

\begin{theorem}[multiplication operators]\label{multop}%
\index{symbols}{M15@$M_{\mu}(g)$}%
Let $\Omega \neq \varnothing$ be a locally  compact Hausdorff
space, and let $\mu \in \probmeas (\Omega)$. If $g$ is a
$\mu$-measurable complex-valued function on $\Omega$, one defines
\[ \domain(g) := \{ \,f \in \rmLeb ^2(\mu) : gf \in \rmLeb ^2(\mu) \,\} \]
as well as
\begin{align*}
M\smu(g) : \domain(g) & \to \rmLeb ^2(\mu) \\
f & \mapsto gf.
\end{align*}
One says that $M\smu(g)$ is the \underline{multiplication operator}
with $g$.

Then $M\smu(g)$ is bounded if and only if
$g \in {\mathscr{L} \,}^{\infty}(\mu)$, in which case
\[ \| \,M\smu(g) \,\| = \| \,g \,\|_{\,\text{\small{$\mu$}},\infty}. \]
For $M\smu(g)$ to be bounded, it suffices to be bounded on
$\cont_\mathrm{c}(\Omega)$.
\end{theorem}

\begin{proof}
If $g \in {\mathscr{L} \,}^{\infty}(\mu)$, one easily sees that
$\| \,M\smu(g) \,\| \leq \| \,g \,\|_{\,\text{\small{$\mu$}},\infty}$. The
converse inequality is established as follows. Assume that
\[ \| \,gh \,\|_{\,\text{\small{$\mu$}},2}
\leq c \,\| \,h \,\|_{\,\text{\small{$\mu$}},2}
\quad \text{for all} \quad h \in \cont_\mathrm{c}(\Omega). \]
It shall be shown that $\| \,g \,\|_{\,\text{\small{$\mu$}},\infty} \leq c$.
Let $X := \{ \,t \in \Omega : | \,g(t) \,| > c \,\}$, which is a
$\mu$-measurable set, and thus $\mu$-integrable. (Because
$\mu$ is a probability measure.) We have to prove that $X$ is
a $\mu$-null set. For this it is enough to show that each compact
subset of $X$ is a $\mu$-null set. (By inner regularity.) So let $K$
be a compact subset of $X$. For $h \in \cont_\mathrm{c}(\Omega)$
with $1_K \leq h \leq 1_\Omega$, we get
\begin{align*}
\int {| \,g \,| \,}^2 \,1_K \,\diff \mu
& \,\leq \int {| \,g \,{h \,}^{1/2} \,| \,}^2 \,\diff \mu
= {\| \,g \,{h \,}^{1/2} \,\|_{\,\text{\small{$\mu$}},2} \,}^2 \\
& \,\leq {c \,}^2 \,{\| \,{h \,}^{1/2} \,\|_{\,\text{\small{$\mu$}},2} \,}^2
= {c \,}^2 \int h \,\diff \mu.
\end{align*}
Thus, by taking the infimum over all such functions $h$, we obtain
\[ \int {| \,g \,| \,}^2 \,1_K \,\diff \mu \,\leq {c \,}^2 \,\mu(K), \]
cf.\ the appendix \ref{Riesz}. Since $| \,g \,| > c$ on $K$, the set $K$
must be $\mu$-null, whence $X$ must be a $\mu$-null set (as noted before).
We have shown that $\| \,g \,\|_{\,\text{\small{$\mu$}},\infty} \leq c$. \pagebreak
\end{proof}

\begin{lemma}\label{Linftycomm}%
Let $\mu \in \probmeas (\Omega)$, where $\Omega \neq \varnothing$ is a
locally compact Hausdorff space. If a bounded linear operator $b$
on $\rmLeb ^2(\mu)$ commutes with every $M\smu(f)$ with
$f \in \cont_\mathrm{c}(\Omega)$, then $b = M\smu(g)$ for some
$g \in {\mathscr{L} \,}^{\infty}(\mu)$.
\end{lemma}

\begin{proof}
We note first that $1_{\Omega} \in {\mathscr{L}}^{\,2}(\mu)$
as $\mu$ is a probability measure.
For every $f \in \cont_\mathrm{c}(\Omega)$ we have
\[ bf = b \bigl(M\smu(f)1_{\Omega} \bigr)
= M\smu(f)(b1_{\Omega}) = M\smu(f)g \]
with $g := b1_{\Omega} \in {\mathscr{L} \,}^2(\mu)$. We continue to compute
\[ M\smu(f)g = fg = gf = M\smu(g)f. \]
The operators $b$ and $M\smu(g)$ thus coincide on $\cont_\mathrm{c}(\Omega)$.
Then $M\smu(g)$ is bounded by the last statement of the preceding
theorem \ref{multop}, and so $g \in {\mathscr{L} \,}^{\infty}(\mu)$. By
density of $\cont_\mathrm{c}(\Omega)$ in $\rmLeb ^2(\mu)$, cf.\ the appendix
\ref{CcdenseinLp}, it follows that $b$ and $M\smu(g)$ coincide everywhere.
\end{proof}

\begin{definition}%
[maximal commutative \protect\st-subalgebras]%
\index{concepts}{maximal!commutative}%
If $B$ is a \st-algebra, then by a \underline{maximal commutative}
\st-subalgebra of $B$ we shall mean a commutative
\st-subalgebra of $B$ which is not properly
contained in any other commutative \st-subalgebra of $B$.
\end{definition}

\begin{proposition}\label{maxcommprime}%
A \st-subalgebra $A$ is maximal commutative if and only if $A' = A$.
\end{proposition}

\begin{proof}
Both properties imply that $A$ is a commutative, so $A \subset A'$.
Now the \st-subalgebra generated by $A$ and some Hermitian
element $b$ is commutative if and only if $b \in A'$.
\end{proof}

\begin{observation}[maximal commutative von Neumann algebras]%
\index{concepts}{algebra!von Neumann!maximal commutative}%
\index{concepts}{von Neumann!algebra!maximal commutative}%
If $H$ is a Hilbert space, then a maximal commutative \st-subalgebra
of $\blop(H)$ is a von Neumann algebra, by \ref{maxcommprime}.
Therefore, such an object is simply called a
\underline{maximal commutative von Neumann algebra}.
\end{observation}

\begin{theorem}\label{maxcommL}%
\index{symbols}{L3@$\rmLeb ^{\infty}(\mu)$}%
Let $\mu \in \probmeas (\Omega)$, with $\Omega \neq \varnothing$
a locally compact Hausdorff space. The mapping
\begin{align*}
\rmLeb ^{\infty}(\mu) & \to \blop\bigl(\rmLeb ^2(\mu)\bigr) \\
g & \mapsto M\smu(g)
\end{align*}
establishes an isomorphism of C*-algebras onto a maximal commutative
\linebreak \st-subalgebra of $\blop\bigl(\rmLeb ^2(\mu)\bigr)$.
In short, one speaks of the maximal \linebreak commutative von
Neumann algebra $\rmLeb ^{\infty}(\mu)$. Note that $\rmLeb ^{\infty}(\mu)$
is the commutant of the \st-algebra formed by the operators $M\smu(f)$ with
$f \in \cont_\mathrm{c}(\Omega)$.\pagebreak
\end{theorem}

\begin{proof}
This follows from the preceding items of this paragraph.
\end{proof}

\begin{definition}[cyclic \st-algebras of operators]\label{cyclicsalg}%
\index{concepts}{cyclic!1s@\protect\st-subalgebra of $\protect\blop(H)$}
\index{concepts}{cyclic!vector}\index{concepts}{vector!cyclic}%
Let $H \neq \{ 0 \}$ be a Hilbert space, and let $B$
be a \st-subalgebra of $\blop (H)$. A non-zero vector
$c \in H$ is called \underline{cyclic} under $B$, if
the subspace $Bc$ is dense in $H$. In this case,
the \st-algebra $B$ is called \underline{cyclic}.
\end{definition}

\begin{observation}\label{cyclicL}%
Let $\mu \in \probmeas (\Omega)$, with $\Omega \neq \varnothing$
a locally compact Hausdorff space. The maximal commutative
von Neumann algebra $\rmLeb ^{\infty}(\mu)$ then is cyclic with
cyclic vector $1_\Omega \in \rmLeb ^2(\mu)$.
\end{observation}

\begin{proof}
Please note first that indeed $1_\Omega \in \rmLeb ^2(\mu)$
as $\mu$ is a probability measure. The statement then follows
for example from the density of $\cont_\mathrm{c}(\Omega)$
in $\rmLeb ^2(\mu)$, cf.\ the appendix \ref{CcdenseinLp}.
\end{proof}

\medskip
We finally touch on a concept closely related to maximal commutative
von Neumann algebras.

\begin{definition}[multiplicity-free representation]\label{multfreedef}%
\index{concepts}{representation!multiplicity-free}%
\index{concepts}{multiplicity-free}%
A representation $\pi$ of a \st-algebra on a Hilbert space is called
\underline{multiplicity-free}, if its commutant $\pi'$ is commutative.
\end{definition}

For a motivation for this terminology, see Mosak \cite[Definition 6.16 p.\ 82]{MosS}.

\begin{proposition}\label{multfreechar}%
Let $\pi$ be a representation of a \underline{commutative} \linebreak
\st-algebra on a Hilbert space. Then $\pi$ is multiplicity-free if and only
if the von Neumann algebra $\vonNeumAlg(\pi)$ is maximal commutative.
\end{proposition}

\begin{proof}
Please note first that $\vonNeumAlg(\pi)' = \pi'$, by \ref{thirdcomm}.
As $\vonNeumAlg(\pi)$ is commutative, by \ref{Wstarcomm},
it follows that $\vonNeumAlg(\pi) \subset \vonNeumAlg(\pi)' = \pi'$.
Now, if $\pi$ is multiplicity-free, then $\pi'$ is commutative, so
$\pi' \subset (\pi')'= \vonNeumAlg(\pi)$. As a consequence, we then get
$\vonNeumAlg(\pi) = \pi' = \vonNeumAlg(\pi)'$, that is, $\vonNeumAlg(\pi)$
is maximal commutative, cf.\ \ref{maxcommprime}. Conversely, if
$\vonNeumAlg(\pi)$ is maximal commutative, then
$\pi' = \vonNeumAlg(\pi)' = \vonNeumAlg(\pi)$ is commutative,
that is, $\pi$ is multiplicity-free.
\end{proof}

\medskip
Mosak \cite[Lemma 7.2 p.\ 87 f.]{MosS} treats multiplication operators
with possibly unbounded measures. This requires the concept of a
``local null set'', which is already incorporated in the treatments
\cite{FW}, \cite{CFWvI}, \cite{CFW}. See e.g.\ \cite[p.\ 106]{FW}.
\pagebreak

\clearpage


\section{The Spectral Representation}

In this paragraph, we shall give the prototypes of cyclic representations
of commutative \st-algebras on Hilbert spaces.

\begin{definition}[the spectral representation]\label{specrepdef}%
\index{concepts}{representation!spectral}%
\index{symbols}{p8@$\pi_{\mu}$}%
\index{concepts}{spectral!representation}%
Let $A$ be a \st-algebra. Let $\mu$ be an inner regular Borel
probability measure on a subset of $\Delta^*(A)$ such that
$\mathrm{supp}(\mu) \cup \{ 0 \}$ is weak* compact. The set
$\Omega := \mathrm{supp}(\mu)$ then is a non-empty locally
compact Hausdorff space. We may consider $\mu$ as an
inner regular Borel probability measure on $\Omega$, the
support of which is all of $\Omega$. The functions
$\wht{a} := \wht{a}| _{\text{\small{$\Omega$}}}$ $(a \in A)$
form a dense subset of $\cont\0(\Omega)$, cf.\ \ref{cloloc}.
In particular, these functions are bounded, so the
multiplication operators $M\smu(\wht{a})$ $(a \in A)$
are bounded operators on $\rmLeb ^2(\mu)$. This makes
that the \underline{spectral representation} $\pi \smu$ is
well-defined by
\begin{align*}
\pi \smu : A & \to \blop\bigl(\rmLeb ^2(\mu)\bigr) \\
a & \mapsto M\smu(\wht{a}).
\end{align*}
Please note that the spectral representation $\pi \smu$ is
cyclic, with cyclic unit vector $1_{\Omega} \in \rmLeb ^2(\mu)$
(as $\mu$ is assumed to be a probability measure),
cf.\ the appendix \ref{CcdenseinLp}.
The philosophy is that the spectral representation $\pi \smu$ is
particularly simple as it acts through multiplication operators.
\end{definition}

\begin{theorem}[diagonalisation]\label{Bochnerbis}%
Let $A$ be a commutative \st-algebra. Let $\pi$ be a cyclic
representation of $A$ on a Hilbert space $H \neq \{0\}$.
Let $c$ be a unit cyclic vector for $\pi$, and consider the
positive linear functional $\varphi$ given by
\[ \varphi(a) := \langle \pi(a) c, c \rangle \qquad (a \in A). \]
We then have the following.

$(i)$ There exists a unique measure
$\mu \in \probmeas \bigl(\s(\pi)\bigr)$ such that
\[ \varphi(a) = \int \wht{a}| _{\text{\small{$\s(\pi)$}}} \,\diff \mu
\quad \text{for all} \quad a \in A, \]
namely
\[ \mu = \langle P c , c \rangle \quad
\text{with $P$ the spectral resolution of $\pi$.} \]

$(ii)$ The representation $\pi$ is spatially equivalent to the
spectral representation $\pi \smu$, and there exists a unique
unitary operator $U$ intertwining $\pi$ with $\pi \smu$, taking
$c$ to $1 \in \rmLeb ^2(\mu)$. One says that $U$
\underline{diagonalises} $\pi$.

$(iii)$ It follows that the intertwining operator $U$ is related to the
Gel'fand transformation in the following way: the operator $U$
maps any vector $\pi(a) c$ with $a \in A$ to $\wht{a}$. It may
thus be viewed as a sort of ``Gel'fand-Plancherel transformation''.
\pagebreak
\end{theorem}

\begin{proof}
Please note that $\s(\pi)$ is a subset of $\Delta^*(A)$
with $\s(\pi) \cup \{ 0 \}$ weak* compact, cf.\ \ref{specpi}.
Uniqueness of $\mu$ follows from the fact that
$\{ \,\wht{a}| _{\text{\small{$\s(\pi)$}}} : a \in A \,\}$
is dense in $\cont\0\bigl(\s(\pi)\bigr)$, cf.\ \ref{cloloc}.
Consider $\mu := \langle P c , c \rangle\in \probmeas \bigl(\s(\pi)\bigr)$.
Then $\mu$ satisfies the assumptions of \ref{specrepdef},
because $\mathrm{supp}(\mu) \cup \{ 0 \}$ is closed in
$\s(\pi) \cup \{ 0 \}$, and thus weak* compact. We have
\[ \varphi (a) = \langle \pi(a) c, c \rangle
= \int \wht{a}| _{\text{\small{$\s(\pi)$}}} \,\diff \mu
\quad \text{for all} \quad a \in A \]
by
\[ \pi(a) = \int _{\text{\small{$\s(\pi)$}}} \wht{a}|_{\text{\small{$\s(\pi)$}}} \,\diff P
\quad \text{(weakly)} \quad \text{for all} \quad a \in A. \]
Now the spectral representation $\pi \smu$ is cyclic, with unit cyclic
vector $1_{\Omega} \in \rmLeb ^2(\mu)$, cf.\ \ref{specrepdef}, where
$\Omega := \mathrm{supp}(\mu)$. With the linear functional
\[ \psi(a) := \langle \pi \smu (a) 1_{\Omega} , 1_{\Omega} \rangle
= \int \wht{a}| _{\text{\small{$\s(\pi)$}}} \,\diff \mu \qquad (a \in A), \]
we then have $\varphi = \psi$. The statement follows now from
\ref{coeffequal}.
\end{proof}

\begin{observation}\index{concepts}{Theorem!abstract!Bochner}%
\label{Bremark}\index{concepts}{abstract!Bochner Theorem}%
\index{concepts}{Bochner Thm., abstr.}%
If $A$ furthermore is a commutative \underline{normed} \linebreak \st-algebra,
if $\psi$ is a state on $A$, and if in the above we choose $\pi := \pi _{\psi}$,
$c := c _{\psi}$, then the measure obtained is  essentially the same as the one
provided by the abstract Bochner Theorem \ref{Bochner}, cf.\ \ref{sppisa}.
\end{observation}

\begin{theorem}\index{symbols}{P95@$P_{\mu}$}\label{exspectres}%
Under the hypotheses of \ref{specrepdef}, and furthermore assuming
that $A$ is commutative, we have the following.

For a Borel set $\Delta$ of $\Omega$, one defines a projection
\[ P\smu(\Delta) := M\smu(1_{\textstyle\Delta}) \in \blop\bigl(\rmLeb ^2(\mu)\bigr). \]
The map
\[ P \smu : \Delta \mapsto P \smu(\Delta)
\qquad ( \,\Delta \text{ a Borel set of } \Omega \,) \]
then is a resolution of identity on $\Omega$. It satisfies
\[ \langle P \smu x,y \rangle = x\, \overline{\vphantom{b}y}
\cdot \mu \quad \text{for all} \quad x,y \in \rmLeb ^2(\mu), \]
so that the support of $P \smu$ is all of $\Omega$.

We get $\s(\pi \smu) = \Omega$, and that $P \smu$
is the \underline{spectral resolution} of $\pi \smu$.

For every bounded Borel function $f$ on $\Omega$, we have
\[ \pi _{\text{\scriptsize{$P$}} \smu}(f) = M \smu(f). \]
The range $\range(\pi _{\text{\scriptsize{$P$}} \smu})$
is equal to the maximal commutative von Neumann algebra $\rmLeb ^{\infty}(\mu)$,
cf.\ \ref{maxcommL}. With $Q \smu$ denoting the restriction of $P \smu$
to the Baire $\sigma$-algebra of $\Omega$ \ref{Bairedef}, we also have that
$\range(\pi _{\text{\scriptsize{$Q$}} \smu})
= \range(\pi _{\text{\scriptsize{$P$}} \smu})$.\pagebreak
\end{theorem}

\begin{proof}
For a Borel set $\Delta$ of $\Omega$ and $x,y$ in $\rmLeb ^2(\mu)$, we have
\[ \langle P \smu (\Delta) \,x, y \rangle
= \langle M\smu(1_{\textstyle\Delta}) \,x, y \rangle
= \int 1_{\textstyle\Delta} \,x\, \overline{\vphantom{b}y} \,\diff \mu. \]
Therefore, the map
$\langle P \smu \,x, y \rangle : \Delta \mapsto
\langle P \smu (\Delta) \,x, y \rangle$ is a measure, namely
\[ \langle P \smu \,x, y \rangle = x\, \overline{\vphantom{b}y} \cdot \mu. \]
It follows that for each unit vector $x$ in $\rmLeb ^2(\mu)$, the measure
$\langle P \smu x, x \rangle$ is an inner regular Borel probability measure
on $\Omega$. Hence $P \smu (\Omega) = \mathds{1}$, as
${\| \,P \smu (\Omega) x \,\|}^{\,2} = \langle P \smu (\Omega) x, x \rangle = 1 \neq 0$
for all unit vectors $x$. So $P \smu$ is a resolution of identity on $\Omega$.
The support of $P \smu$ is all of $\Omega$ as already the support of
$\langle P \smu 1_{\Omega} , 1_{\Omega} \rangle = \mu$ is all of
$\Omega$. Next, for $a \in A$, we have
\[ \langle \pi \smu (a) \,x, y \rangle
= \langle M\smu\bigl(\,\wht{\vphantom{b}a}\,\bigr) \,x, y \rangle
= \int \wht{\vphantom{b}a} \,x\, \overline{\vphantom{b}y} \,\diff \mu
= \int \wht{\vphantom{b}a} \,\diff \langle P \smu x, y \rangle \]
for all $x,y \in \rmLeb ^2(\mu)$, which says that
\[ \pi \smu (a) = \int _{\Omega} \wht{a}|_{\Omega} \,\diff P \smu
\quad \text{(weakly).} \]
Since the support of $P \smu$ is all of $\Omega$, we get that
$\Omega = \s(\pi \smu)$, and that $P \smu$ is the spectral
resolution of $\pi \smu$. (This follows from the uniqueness
property of the spectral resolution as in the addendum \ref{spuniq}.)
For a bounded Borel function $f$ on $\Omega$, we have
\[ \pi _{\text{\scriptsize{$P$}} \smu}(f) = \int f \,\diff P \smu\quad \text{(weakly),} \]
so that for $x,y \in \rmLeb ^2(\mu)$ we get
\[ \langle \pi _{\text{\scriptsize{$P$}} \smu} (f) \,x, y \rangle
= \int f \,\diff \langle P \smu x, y \rangle
= \int f \,x\, \overline{\vphantom{b}y} \,\diff \mu = \langle M \smu(f) \,x, y \rangle. \]
This implies that
\[ \pi _{\text{\scriptsize{$P$}} \smu} (f) = M\smu (f). \]
This already shows that $\range(\pi _{\text{\scriptsize{$Q$}} \smu})
\subset \range(\pi _{\text{\scriptsize{$P$}} \smu}) \subset \rmLeb ^{\infty}(\mu)$.
Conversely, to see that
$\rmLeb ^{\infty}(\mu) \subset \range(\pi _{\text{\scriptsize{$Q$}} \smu})$,
we use that a function $f \in {\mathscr{L} \,}^{\infty}(\mu)$ is $\mu$-a.e.\ equal
to a bounded Baire function on $\Omega$, cf.\ the appendix \ref{Baire}.
\end{proof}

\begin{corollary}\label{spectralsupp}%
Let $\pi$ be a cyclic representation of a commutative \st-algebra
on a Hilbert space $H \neq \{0\}$. Let $c$ be a unit cyclic vector for $\pi$.
Let $P$ denote the spectral resolution of $\pi$.
Let $\mu := \langle P c , c \rangle \in \probmeas \bigl(\s(\pi)\bigr)$,
as in \ref{Bochnerbis}. Then the support of $\mu$ is all of $\s(\pi)$,
and $\pi _{\text{\scriptsize{$P$}}}$ is spatially equivalent to
$\pi _{\text{\scriptsize{$P$}} \smu}$ as in \ref{exspectres}.
Indeed, every unitary operator intertwining $\pi$ with $\pi \smu$
also intertwines $\pi _{\text{\scriptsize{$P$}}}$ with
$\pi _{\text{\scriptsize{$P$}} \smu}$. \pagebreak
\end{corollary}

\begin{proof}
Apply the uniqueness of the spectral resolution
as in the addendum \ref{spuniq}.
Use the proof of the appendix \ref{Baire} \& \ref{specrepdef}.%
\end{proof}

\begin{theorem}\label{spectralrepmain}%
Let $\pi$ be a cyclic representation of a commutative \linebreak
\st-algebra on a Hilbert space $H \neq \{0\}$. Let $P$ be the
spectral resolution of $\pi$. Let $Q$ denote the restriction of $P$
to the Baire $\sigma$-algebra of $\s(\pi)$. We then have that
\[ \range(\pi _Q) = \range(\pi _P) = \vonNeumAlg (\pi) = \pi' \]
is a maximal commutative von Neumann algebra.
\end{theorem}

\begin{proof}
Let $c$ be a cyclic unit vector for $\pi$ and let $\mu$ be the measure
$\langle Pc,c \rangle$. Let $P \smu$ be the spectral resolution of the
spectral representation $\pi \smu$, as in \ref{exspectres}. Then $\pi$
is spatially equivalent to $\pi \smu$, cf.\ \ref{Bochnerbis}, and
$\pi _{\text{\scriptsize{$P$}}}$ is spatially equivalent to
$\pi _{\text{\scriptsize{$P$}} \smu}$, by \ref{spectralsupp}.
(Note that both spatial equivalences can be implemented by the same
unitary operator.) Let $Q \smu$ denote the restriction of $P \smu$
to the Baire $\sigma$-algebra. The equalities
\[ \range(\pi _{\text{\scriptsize{$Q$}} \smu})
= \range(\pi _{\text{\scriptsize{$P$}} \smu})
= \vonNeumAlg (\pi \smu) = {\pi \smu}' \]
follow from the fact that
$\range(\pi _{\text{\scriptsize{$Q$}} \smu})
= \range(\pi _{\text{\scriptsize{$P$}} \smu})$
is a maximal commutative von Neumann algebra,
cf.\ the last two statements of theorem \ref{exspectres}. The
proof is completed by applying the spatial equivalences above.
\end{proof}

(Without the assumption of cyclicity, we merely get the inclusions
\[ \range(\pi _Q) \subset \range(\pi _P) \subset \vonNeumAlg (\pi) \subset \pi', \]
the last inclusion holding by commutativity of
$\vonNeumAlg (\pi)$, cf.\ \ref{Wstarcomm}.)

\begin{corollary}[reducing subspaces]%
\index{concepts}{reducing}\index{concepts}{subspace!reducing}%
\index{concepts}{invariant subspace}\index{concepts}{subspace!invariant}%
Let $\pi$ be a cyclic representation of a commutative \st-algebra on a
Hilbert space $H \neq \{0\}$. Let $P$ denote the spectral resolution of
$\pi$. A closed subspace $M$ of $H$ reduces the cyclic representation
$\pi$, if and only if $M$ is the range of some $P(\Delta)$, where $\Delta$
is a Borel (or even Baire) subset of $\s(\pi)$.
\end{corollary}

\begin{proof}
Let $Q$ be the restriction of $P$ to the Baire $\sigma$-algebra. Let $M$
be a closed subspace of $H$ and let $p$ denote the projection on $M$.
The subspace $M$ reduces $\pi$ if and only if $p \in \pi' = \range(\pi _{Q})$,
cf.\ \ref{invarcommutant} and the preceding theorem \ref{spectralrepmain}.
However a projection in $\range(\pi _{Q})$ is of the form $Q(\Delta)$ for
some Baire subset $\Delta$ of $\s(\pi)$, cf.\ \ref{spprojform}.
\end{proof}

\medskip
Please note that the above two results do not explicitly involve any
specific cyclic vector. \pagebreak

\begin{theorem}\label{cyclicfree}%
A cyclic representation of a commutative \st-algebra on a Hilbert space
is multiplicity-free \ref{multfreedef}.
\end{theorem}

\begin{proof}
The commutant of such a representation is (maximal)
commutative by theorem \ref{spectralrepmain}.
\end{proof}

\begin{definition}[perfect measures]\label{prefect}%
\index{concepts}{measure!perfect}\index{concepts}{perfect}%
An inner regular Borel probability measure $\mu$ on a compact
Hausdorff space $K \neq \varnothing$ is called
\underline{perfect}, if the support of $\mu$ is all of $K$,
and if for every $f \in \mathscr{L}^{\,\infty} (\mu)$, there exists a
(necessarily unique) function in $\cont(K)$, which is
$\mu$-a.e.\ equal to $f$. We may then identify the C*-algebras
$\rmLeb ^{\infty}(\mu)$ and $\cont(K)$. (As in \ref{suppimbed}.)
\end{definition}

\begin{definition}[spatial equivalence]%
\index{concepts}{equivalent!spatially!von Neumann algebra}%
\index{concepts}{spatially equivalent!von Neumann algebra}%
Two von Neumann algebras $\mathscr{M}_1$ and $\mathscr{M}_2$
on Hilbert spaces $H_1$ and $H_2$ respectively, are called
\underline{spatially equivalent}, if there exists a unitary operator
$U : H_1 \to H_2$ such that $U \mathscr{M}_1 U^{-1} = \mathscr{M}_2$.
\end{definition}

\begin{theorem}\label{commcyclicchar}%
For a commutative von Neumann algebra $\mathscr{M}$ on a
Hilbert space $\neq \{0\}$, the following statements are equivalent.
\begin{itemize}
 \item[$(i)$] $\mathscr{M}$ is cyclic,
\item[$(ii)$] $\mathscr{M}$ is spatially equivalent to the maximal
                      commutative von Neumann algebra $\rmLeb ^{\infty}(\mu)$
                      for some perfect inner regular Borel probability measure
                      $\mu$ on the compact Hausdorff space $\Delta(\mathscr{M})$.
\item[$(iii)$] $\mathscr{M}$ is spatially equivalent to the maximal
                      commutative von Neumann algebra $\rmLeb ^{\infty}(\mu)$
                      for an inner regular Borel probability measure $\mu$
                      on a locally compact Hausdorff space $\neq \varnothing$.
\end{itemize}
\end{theorem}

\begin{proof}
(i) $\Rightarrow$ (ii): Use \ref{Abelian} in conjunction with \ref{Bochnerbis},
\ref{exspectres} \& \ref{spectralsupp}. Perfectness follows from \ref{spectralsupp},
\ref{disconn} \& the appendix \ref{Baire}.
(ii) $\Rightarrow$ (iii): trivial. (iii) $\Rightarrow$ (i): \ref{cyclicL}.
\end{proof}

\medskip
We see that the prototypes of the cyclic commutative von Neumann algebras
are the maximal commutative von Neumann algebras $\rmLeb ^{\infty}(\mu)$
with $\mu$ a perfect inner regular Borel probability measure on a compact
Hausdorff space $\neq \varnothing$.

\bigskip
See also \ref{maxcommchar} - \ref{sepspatequiv} below.

\bigskip
The reader can take a shortcut from here to chapter \ref{unbdself}
on the Spectral Theorem for unbounded self-adjoint operators.
\pagebreak

\clearpage


\addtocontents{toc}{\protect\vspace{0.2em}}

\chapter{Separability}%
\label{supplmat}

\setcounter{section}{42}

\begin{center}
The present chapter \ref{supplmat} can be skipped at first reading.
\end{center}

\begin{center}
The first two paragraphs of this chapter do not use separability yet.
\end{center}


\medskip
\section{Borel \texorpdfstring{$*$-}{\80\052\80\055}Algebras}%
\label{Borelalg}

\medskip
In this paragraph, let $H$ be a Hilbert space $\neq \{ 0 \}$.

\begin{definition}[Borel \protect\st-algebras, \hbox{\cite[4.5.5]{PedC}}]
\index{concepts}{Borel!0s0@\protect\st-algebra}%
\index{concepts}{algebra!Borel@Borel \protect\st-algebra}%
A \underline{Borel \st-algebra} on $H$ is a C*-subalgebra $\mathscr{B}$ of
$\blop(H)$ such that $\mathscr{B}\sa$ contains the supremum in $\blop(H)\sa$
of each increasing \underline{sequence} in $\mathscr{B}\sa$ which is upper
bounded in $\blop(H)\sa$. (Cf.\ \ref{ordercompl}.)
\end{definition}

\begin{observation}\label{NeuBor}%
Every von Neumann algebra on $H$ is a Borel \st-algebra, cf.\ \ref{ordclosed}.
\end{observation}

\begin{proposition}\label{ranspecint}%
If $P$ is a spectral measure acting on $H$, then $\range(\pi _P)$ is a
commutative Borel \st-algebra on $H$ containing $\mathds{1}$.
\end{proposition}

\begin{proof}
Let $P$ be a spectral measure, defined on a $\sigma$-algebra
$\mathcal{E}$ on a set $\Omega \neq \varnothing$, acting on $H$.
Then $\range(\pi _P)$ is a C*-subalgebra of $\blop(H)$, by \ref{spanP}.
Let $(a_n)$ be an increasing sequence in $\range(\pi _P)\sa$ which
is upper bounded in $\blop(H)\sa$. Then $(a_n)$ is norm-bounded, and
thus order-bounded in $\range(\pi _P)\sa$ as well. This follows from two
successive applications of \ref{onbded}, using that both $\blop(H)$ and
$\range(\pi _P)$ are unital. So there exists $g \in \bmeas(\mathcal{E})\sa$
with $a_n \leq \pi _P (g)$ for all $n$. Choose a sequence $(f_n)$  in
$\bmeas(\mathcal{E})\sa$ with $\pi _P (f_n) = a_n$ for all $n$.
Please note that then $f_n \leq g$ $P$-a.e.\ for all $n$, by \ref{piPpos}.
We can furthermore assume that the sequence $(f_n)$ is upper bounded
in $\bmeas(\mathcal{E})\sa$. Namely by replacing $(f_n)$ with
the sequence $(f_n \wedge g)$, cf.\ \ref{Cstarlattice}. Indeed,
$f_n \wedge g = f_n$ $P$-a.e.\ for all $n$. The new sequence $(f_n)$
satisfies $f_n \leq f_{n+1}$ $P$-a.e.\ for all $n$, by \ref{piPpos}. So
$(f_n)$ fulfils the hypotheses of the Monotone Convergence Theorem
\ref{monconvthm}, and this finishes the proof.
\pagebreak
\end{proof}

Since the intersection of Borel \st-algebras on $H$ is a Borel \st-algebra on $H$,
we may put:

\begin{definition}[$\mathscr{B}_1(\cdot)$]%
\index{symbols}{Bm1@$\mathscr{B}_1(S)$}%
\index{symbols}{Bm2@$\mathscr{B}_1(a)$}%
\index{symbols}{Bm3@$\mathscr{B}_1(\pi)$}%
For a subset $S$ of $\blop(H)$, let $\mathscr{B}(S)$ denote the Borel \st-algebra
on $H$ generated by $S$. That is, $\mathscr{B}(S)$ is defined as the intersection
of all Borel \st-algebras on $H$ containing $S$. Then $\mathscr{B}(S)$ is the
smallest Borel \st-algebra on $H$ containing $S$.

We shall then also write $\mathscr{B}_1(S) := \mathscr{B}(S\cup\{\mathds{1}\})$.

If $a$ is a bounded linear operator on $H$, we put
$\mathscr{B}_1(a) := \mathscr{B}_1 \bigl( \{ a \} \bigr)$.

If $\pi$ is a non-degenerate representation of a \st-algebra $A$ on $H$,
we put $\mathscr{B}_1(\pi) := \mathscr{B}_1\bigl(\range(\pi)\bigr)$.
\end{definition}

\begin{proposition}\label{Borelcomm}%
If $N$ is a normal \ref{selfadjointsubset} subset of $\blop(H)$, then
$\mathscr{B}(N)$ is commutative. Hence, if $b$ is a normal bounded
linear operator on $H$, then $\mathscr{B}_1(b)$ is commutative.
\end{proposition}

\begin{proof}
Let $A$ denote the C*-subalgebra of $\blop(H)$ generated by $N$ and
$\mathds{1}$, which is commutative. With $P$ denoting the spectral resolution
of $A$, we have that $\range(\pi _P)$ is a commutative Borel \st-algebra on $H$
containing $N$ and $\mathds{1}$, cf.\ \ref{ranspecint}. This is enough to prove the
statement. An alternative is to use \ref{Wstarcomm} in conjunction with \ref{NeuBor}.
\end{proof}

\begin{proposition}\label{Borelspecint}%
If $P$ is a spectral measure acting on $H$, then
\[ \range(\pi _P) = \mathscr{B}(P). \]
\end{proposition}

\begin{proof}
By \ref{ranspecint}, we have that $\range(\pi _P)$ is a Borel \st-algebra
containing $P$. It follows that $\mathscr{B}(P) \subset \range(\pi _P)$.
In the converse direction, we have $\range(\pi _P) = \overline{\mathrm{span}}(P)$,
cf.\ \ref{spanP}. Whence also $\range(\pi _P) \subset \mathscr{B}(P)$,
as $\mathscr{B}(P)$ is a C*-subalgebra of $\blop(H)$ containing $P$.
\end{proof}

\medskip
The main result of this paragraph is the following.
Its importance stems mostly from the ensuing consequences.

\begin{theorem}\label{QBorel}%
Let $A$ be a commutative \st-algebra.
Let $\pi$ be a non-degenerate representation of $A$ on $H$.
Let $P$ be the spectral resolution of $A$, and let $Q$ be the
restriction of $P$ to the Baire $\sigma$-algebra of $\s(\pi)$.
We then have
\[ \mathscr{B}_1(\pi) = \range(\pi _Q). \pagebreak \]
\end{theorem}

\begin{proof}
We remind the reader of the fact that the Baire $\sigma$-algebra of $\s(\pi)$
is the smallest $\sigma$-algebra $\mathcal{E}$ on $\s(\pi)$ such that
the functions in $\cont\0\bigl(\s(\pi)\bigr)$ are $\mathcal{E}$-measurable,
cf.\ the appendix \ref{Bairedef}. In particular, the functions in
\[ F := \{ \,\wht{a}|_{\s(\pi)} \in \cont\0\bigl(\s(\pi)\bigr) : a \in A \,\} \]
are bounded Baire functions, and thus in the domain of $\pi_Q$. Hence
\[ \range(\pi) = \{ \,\pi _Q (\wht{a}|_{\s(\pi)}) \in \blop(H) : a \in A \,\}, \tag*{$(*)$} \]
by the spectral theorem. Next we note that the Baire $\sigma$-algebra also
is the smallest $\sigma$-algebra $\mathcal{E}$ on $\s(\pi)$ such that the
functions in $F$ are $\mathcal{E}$-measurable, by density of $F$ in
$\cont\0\bigl(\s(\pi)\bigr)$, cf.\ \ref{cloloc}. Hence the Baire
$\sigma$-algebra of $\s(\pi)$ is generated by the sets of the form
\[ \{ \,\wht{a}|_{\s(\pi)} < \alpha \,\} \qquad ( \,a \in A\sa,\ \alpha \in \mathds{R} \,). \]

We shall show next that $Q \subset \mathscr{B}_1(\pi)$.
It is enough to prove that $Q(\Delta) \in \mathscr{B}_1(\pi)$
for $\Delta = \{ \,\wht{a}|_{\s(\pi)} < \alpha \,\}$ with $a \in A\sa$,
$\alpha \in \mathds{R}$. (This is enough indeed, because of the
preceding and because $\mathscr{B}_1(\pi)$ is a Borel \st-algebra
containing $\mathds{1}$, keeping in mind the Monotone Convergence
Theorem \ref{monconvthm}.) For this purpose, consider the bounded functions
\[ f_n := 0 \vee \biggl\{ \,1_{\s(\pi)} \wedge
\Bigl[ \,n \,( \,\alpha 1_{\s(\pi)} - \wht{a}|_{\s(\pi)} \,) \,\Bigr] \,\biggr\}. \]
for all integers $n \geq 1$. It is easily seen that the sequence $(f_n)$ is
increasing with supremum $1_{\Delta}$. Also, the functions $f_n$ all
belong to the closed \st-subalgebra - of the C*-subalgebra of bounded
Baire functions on $\s(\pi)$ - generated by $\wht{a}|_{\s(\pi)}$ and
$1_{\s(\pi)}$, cf.\ \ref{Cstarlattice}. Therefore, the $\pi _Q (f_n)$ will all be
in $\mathscr{B}_1(\pi)$, by $(*)$, the continuity of $\pi _Q$, and the fact
that $\mathscr{B}_1(\pi)$ is a C*-subalgebra of $\blop(H)$ containing
$\range(\pi)$ as well as $\mathds{1}$.
It follows from the Monotone Convergence Theorem \ref{monconvthm}
that $Q(\Delta) = \pi _Q (1_{\Delta}) = \sup_n \pi _Q (f_n)$ also is in
$\mathscr{B}_1(\pi)$. We have shown that $Q \subset \mathscr{B}_1(\pi)$.

From $Q \subset \mathscr{B}_1(\pi)$, we get
\[ \mathscr{B}_1(Q) \subset \mathscr{B}_1(\pi). \tag*{$(**)$} \]
We also have that $\range(\pi) \subset \range(\pi _Q) = \mathscr{B}_1(Q)$,
by $(*)$ and proposition \ref{Borelspecint} above. Whence also
\[ \mathscr{B}_1(\pi) \subset \mathscr{B}_1(Q) = \range(\pi _Q). \tag*{$(***)$} \]
From $(**)$ and $(***)$, we find
\[ \mathscr{B}_1(\pi) = \range(\pi _Q), \]
as was to be shown. \pagebreak
\end{proof}

Next up is a converse to proposition \ref{ranspecint}.

\begin{theorem}\label{Borelgen}%
Consider a commutative Borel \st-algebra $\mathscr{B}$ on $H$
containing $\mathds{1}$. Let $P$ be the spectral resolution
of $\mathscr{B}$. Let $Q$ be the restriction of $P$ to the Baire
$\sigma$-algebra of $\Delta(\mathscr{B})$. We then have
\[ \mathscr{B} = \mathscr{B}(Q) = \range(\pi _Q) = \overline{\mathrm{span}}(Q). \]
In particular, the set $\mathscr{P}(\mathscr{B})$ of projections in $\mathscr{B}$,
cf.\ \ref{setofproj}, equals $Q$:
\[ \mathscr{P}(\mathscr{B}) = Q. \]
\end{theorem}

\begin{proof}
From the preceding theorem \ref{QBorel}, we have
\[ \mathscr{B} = \range(\pi _Q), \]
whence also $Q \subset \mathscr{B}$.
Proposition \ref{Borelspecint} implies that
\[ \range(\pi _Q) = \mathscr{B}(Q). \]
Theorem \ref{spanP} shows that
\[ \range(\pi _Q) = \overline{\mathrm{span}}(Q). \]
The last statement follows from theorem \ref{spprojform}.
\end{proof}

\medskip
From the preceding theorem \ref{Borelgen} and proposition \ref{ranspecint},
we have:

\begin{corollary}[characterisation]\label{Borelchar}%
The C*-algebras $\range(\pi _P)$ with $P$ a spectral measure acting on $H$,
can be characterised as the commutative Borel \st-algebras on $H$ containing
the unit operator.
\end{corollary}

\begin{remark}
Sometimes a $\sigma$-algebra is called a ``Borel structure''.
\end{remark}

We now have the following generalisation of theorem \ref{W*main}.

\begin{theorem}%
A Borel \st-algebra on $H$ containing $\mathds{1}$ is
the closed linear span of its projections.
\end{theorem}

\begin{proof}
The statement for the commutative case follows from theorem \ref{Borelgen}
above. The statement for the general case follows by decomposing an
element $c$ of a Borel \st-algebra $\mathscr{B}$ containing $\mathds{1}$
as $c = a + \iu b$, with $a, b \in \mathscr{B}\sa$. Indeed
$\mathscr{B}_1(a), \mathscr{B}_1(b)$ are commutative by proposition \ref{Borelcomm},
and contained in $\mathscr{B}$. Please note that the above proof is similar
to the proof of theorem \ref{W*main}. \pagebreak
\end{proof}

\medskip
Next two miscellaneous results for single normal bounded linear operators.

\begin{theorem}\label{PBorel}%
Let $b$ be a normal bounded linear operator on $H$. Let $P$ be the
spectral resolution of $b$. One then has
\[ \range(\pi _P) = \mathscr{B}_1(b). \]
In particular, the range of the Borel functional calculus
\begin{align*}
\rmLeb ^{\infty}(P) & \to \range(\pi_P) \\
f+\mathrm{N}(P) & \mapsto f(b) = \pi _P(f)
\end{align*}
is the Borel \st-algebra on $H$ generated by $b$, $b^*$, and $\mathds{1}$.
See also \ref{rangespres}.
\end{theorem}

\begin{proof}
The spectral resolution $P$ of $b$ can be identified with the spectral
resolution of the C*-subalgebra $A$ of $\blop(H)$ generated by $b$,
$b^*$, and $\mathds{1}$, cf.\ the proof of \ref{spthmnormalbded}.
It is easily seen that $\mathscr{B}_1(b) = \mathscr{B}(A)$. (Indeed,
$\mathscr{B}(A)$ contains $b$ and $\mathds{1}$, and thus $\mathscr{B}_1(b)$.
Conversely, $\mathscr{B}_1(b)$ is a Borel \st-algebra containing $b$, $b^*$,
and $\mathds{1}$, and thus contains $A$, and so even $\mathscr{B}(A)$.)
Since $\s(b)$ is a compact metric space, the Borel and Baire $\sigma$-algebras
on $\s(b)$ coincide, cf.\ the appendix \ref{metric}. Therefore, theorem
\ref{QBorel} implies now that $\range(\pi _P) = \mathscr{B}(A) = \mathscr{B}_1(b)$.
\end{proof}

\medskip
Borel \st-algebras containing $\mathds{1}$ allow a Borel functional
calculus for normal elements, in the following sense.

\begin{corollary}%
Let $\mathscr{B}$ be a Borel \st-algebra on $H$ containing $\mathds{1}$.
If $b$ is a normal operator in $\mathscr{B}$, then every bounded Borel
function of $b$ lies in $\mathscr{B}$. In particular the spectral resolution
of $b$ takes values in $\mathscr{B}$.
\end{corollary}

\begin{proof}
By the preceding theorem \ref{PBorel}, a bounded Borel function
of $b$ lies in $\mathscr{B}_1(b)$, which is contained in $\mathscr{B}$.
The last statement follows from the fact that a spectral projection
$P(\Delta)$ of $b$, with $P$ the spectral resolution of $b$, and
$\Delta$ a Borel subset of $\s(b)$, is the bounded Borel function
of $b$ given by $1_{\Delta}(b) = \pi _P (1_{\Delta}) = P(\Delta)$.
\end{proof}

\medskip
In \ref{sepBorel} below, we shall prove the convenient criterion
that a commutative Borel \st-algebra containing $\mathds{1}$
and acting on a separable Hilbert space, is a von Neumann algebra.
\pagebreak

\clearpage

%


\section{Subrepresentations and Amplifications}%
\label{subsum}

We shall prove the following two somewhat technical
facts as a preparation for the ensuing \ref{rangepara}.

\begin{proposition}\label{propsubrep}%
Let $\pi$ be a non-degenerate representation of a
commutative \st-algebra $A$ on a Hilbert space
$H \neq \{ 0 \}$. Let $M \neq \{ 0 \}$ be an invariant
subspace of $\pi$, and consider the non-degenerate
subrepresentation $\pi _M$, cf.\ \ref{subrep} \& \ref{subnondeg}.
Let $P$ and $Q$ be the spectral resolutions of $\pi$
and $\pi _M$ respectively. The subspace $M$ is
invariant under $\pi _P$. We have that $\s(\pi _M)$
is a closed subset of $\s(\pi)$, and may therefore
consider $Q$ as a resolution of identity on $\s(\pi)$
with support $\s(\pi _M)$. We then get that $\pi_Q$
is the subrepresentation ${(\pi _P)} _M$ of $\pi_P$:
\[ \pi _Q = {(\pi _P)} _M. \]
\end{proposition}

\begin{proof}
The subspace $M$ is invariant under $\pi _P$ by \ref{invarcommutant},
as $\pi$ and $\pi _P$ have the same commutant,
cf.\ \ref{spprojcomm} \& \ref{spthmrep}. From \ref{unilemma} we get
\[ \s(\pi) = \{ \,\sigma \in \Delta^*(A) : | \,\sigma(a) \,| \leq \| \,\pi (a) \,\|
\ \text{for all}\ a \in A \,\}, \]
and similarly with $\pi _M$ in place of $\pi$. Since
$\| \,\pi _M (a) \,\| \leq \| \,\pi (a) \,\|$ for all $a \in A$, it follows
that $\s(\pi _M) \subset \s(\pi)$. Since $\s(\pi _M)$ is closed in
$\Delta^*(A)$, by \ref{cloloc}, it is closed in $\s(\pi)$ as well.
We may therefore consider $Q$ as a resolution of identity on $\s(\pi)$
with support $\s(\pi _M)$. The function $P_M : \Delta \mapsto P(\Delta)|_M$
$\bigl( \,\Delta$ a Borel subset of $\s(\pi) \,\bigr)$ is a
resolution of identity on $\s(\pi)$, acting on $M$. Since now
$\pi _{(\text{\SMALL{$P$}}_{\text{\tiny{$M$}}})} (\wht{a})
= \pi_M (a) = \pi_Q (\wht{a})$ for all $a \in A$, the uniqueness
of the spectral resolution as in the addendum \ref{spuniq} now
implies that $Q = P_M$, and the statement follows.
\end{proof}

\begin{proposition}\label{propampli}
Let $\pi$ be a non-degenerate representation of a commutative
\st-algebra on a Hilbert space $H \neq \{ 0 \}$. Consider the
amplification $\oplus _n \pi$ on $\oplus _n H$, cf.\ \ref{amplification}.
Let $P$ and $Q$ be the spectral resolutions of $\pi$ and $\oplus _n \pi$
respectively. We then have the equalities
\[ \s(\oplus _n \pi) = \s(\pi) \quad \text{as well as} \quad\pi _Q = \oplus _n \pi _P. \]
\end{proposition}

\begin{proof}
The C*-algebra
$\overline{\range(\oplus _n \pi)}
= \{ \,\oplus _n b \in \blop(\oplus _n H) : b \in \overline{\range(\pi)} \,\}$
is isomorphic to the C*-algebra $\overline{\range(\pi)}$.
Hence $\Delta(\overline{\range(\oplus _n \pi)})$ and
$\Delta(\overline{\range(\pi)})$ may be identified. It follows that
$\s(\oplus _n \pi) = \s(\pi)$. The uniqueness of the spectral
resolution implies that $Q(\Delta) = \oplus _n P (\Delta)$ for every
Borel set $\Delta$ of $\s(\pi)$, whence the statement. \pagebreak
\end{proof}

\clearpage

%


\section{Separability: the Range}%
\label{rangepara}

\medskip
In this paragraph, let $H$ be a Hilbert space $\neq \{0\}$.

\begin{introduction}\label{decospace}%
Let $\pi$ be a non-degenerate representation of a commutative
\st-algebra on $H$, and let $P$ be the spectral resolution of $\pi$.
The representation $\pi _P$ factors to an isomorphism of C*-algebras
from $\rmLeb ^{\infty} (P)$ onto $\range(\pi _P)$, cf.\ \ref{LPmain}. It shall
be shown in this paragraph that if $H$ is separable, then (on the side
of the \textit{range} of the factored map), the \linebreak C*-algebra
$\range(\pi _P)$ is all of the von Neumann algebra $\vonNeumAlg (\pi)$,
cf.\ \ref{spprojsubset}. It shall be shown in the next paragraph that if
$H$ is separable, then (on the side of the \textit{domain} of the factored
map), the C*-algebra $\rmLeb ^{\infty} (P)$ can be identified with the maximal
commutative von Neumann algebra $\rmLeb ^{\infty} ( \langle Px, x \rangle )$
for some unit vector $x \in H$. Thus, in the case of a separable Hilbert
space, the representation $\pi _P$ factors to a C*-algebra isomorphism
of \textit{von Neumann algebras}. This C*-algebra isomorphism
$\rmLeb ^{\infty}( \langle P x, x \rangle ) \to \vonNeumAlg (\pi)$ is called the
\underline{$\rmLeb ^{\infty}$ functional calculus}. The inverse of this map
faithfully represents $\vonNeumAlg (\pi)$ by multiplication operators living on
$\s(\pi)$. (So much for the titles of these two paragraphs.)
\end{introduction}

\begin{theorem}\label{gap}%
Consider a non-degenerate representation $\pi$ of a commutative
\st-algebra on $H$. Let $P$ be the spectral resolution of $\pi$, and let
$b \in \vonNeumAlg (\pi)$. Then for any sequence $(x_n)$ in $H$, there
exists a bounded Baire function $f$ on $\s(\pi)$ such that for all $n$, one has
\[ b \,x_n = \pi _P (f) \,x_n. \]
\end{theorem}

\begin{proof}
We shall first consider the case of a single vector $x \in H$.
We can assume that $x \neq 0$. The closed invariant subspace
$M := \overline{\range(\pi) \,x}$ is cyclic under $\pi$, with cyclic
vector $x \in M$, cf.\ \ref{cyclicsubspace}.
The projection $p$ on $M$ is in $\pi'$ by \ref{invarcommutant},
and so commutes with $b \in \vonNeumAlg (\pi)$. This in turn implies
that $M$ is invariant under $b$, by \ref{comminvar}.
Consequently $b | _M \in \vonNeumAlg ( \pi _M )$.
Let $Q$ denote the spectral resolution of the cyclic representation
$\pi _M$. Let $S$ denote the restriction of $Q$ to the Baire
\linebreak $\sigma$-algebra of $\s(\pi _M)$. We have
$\vonNeumAlg ( \pi _M ) = \range(\pi _S)$ by \ref{spectralrepmain}.
It follows that $b |_M \in \range(\pi _S)$, whence there exists a
bounded Baire function $g$ on $\s(\pi _M)$ with $b |_M = \pi _S (g)$.
In particular $b \,x = \pi _Q (g) \,x$. Now $\s(\pi _M)$ is a closed
subset of $\s(\pi)$, cf.\ \ref{propsubrep}, so the function $g$ has an
extension to a bounded Baire $f$ function on $\s(\pi)$, see the
appendix \ref{bBaireext}. If we consider $Q$ as a resolution of
identity on $\s(\pi)$ with support $\s(\pi _M)$, we have
$\pi _Q = {( \pi _P )} _M$, cf.\ \ref{propsubrep}. In particular, we get
$b \,x = \pi _P (f) \,x$. \pagebreak

For the general case, we can assume that
$\bigl( \,\| \,x_n \,\| \,\bigr) \in {\ell\,}^2$. Put
$x := \oplus _n x_n \in \oplus _n H$. We have
$\oplus _n b \in \vonNeumAlg ( \oplus _n \pi )$ by \ref{ampli}.
Let $T$ denote the spectral resolution of $\oplus _n \pi$.
By the above, there exists a bounded Baire function
$f$ on $\s( \oplus _n \pi )$ with $\oplus _n b \,x = \pi _T (f) \,x$.
Since $\s( \oplus _n \pi ) = \s(\pi)$ and $\pi _T = \oplus _n \pi _P$
by \ref{propampli}, the function $f$ satisfies the requirements.
\end{proof}

\medskip
An immediate consequence is:

\begin{theorem}\label{sepmain}%
Consider a non-degenerate representation $\pi$ of a commutative
\st-algebra on $H$. Let $P$ be the spectral resolution of $\pi$,
and let $Q$ be the restriction of $P$ to the Baire $\sigma$-algebra
of $\s(\pi)$. If $H$ is \underline{separable}, then
\[ \vonNeumAlg (\pi) = \range(\pi_P) = \range(\pi_Q). \]
\end{theorem}

\begin{theorem}[Riesz - von Neumann]\label{functions}%
\index{concepts}{Theorem!Riesz - von Neumann}%
\index{concepts}{Riesz - von Neumann Thm.}%
Let $b$ be a normal bounded linear operator on $H$. If $H$ is
separable, then a bounded linear operator on $H$ is a bounded
Borel function of $b$ \ref{Borelfunct} if and only if it commutes
with every bounded linear operator on $H$ which commutes with $b$.
\end{theorem}

\begin{proof}
The spectral resolution $P$ of $b$ can be identified with the spectral
resolution of the C*-subalgebra $A$ of $\blop(H)$ generated by $b$,
$b^*$, and $\mathds{1}$, cf.\ the proof of \ref{spthmnormalbded}.
The preceding theorem \ref{sepmain} therefore implies that the set of
bounded Borel functions of $b$, that is, $\range(\pi _P)$, is a
von Neumann algebra, and thus equal to $\vonNeumAlg (b)$.
Since $b$ is normal, one also has $\vonNeumAlg (b) = \{ b \}''$,
cf.\ \ref{W*normal}.
\end{proof}

\begin{theorem}\label{sepBorel}%
A commutative Borel \st-algebra $\mathscr{B}$ containing $\mathds{1}$ and
acting on a separable Hilbert space $\neq \{ 0 \}$, is a von Neumann algebra.
\end{theorem}

\begin{proof}
Let $P$ be the spectral resolution of $\mathscr{B}$ and let
$Q$ be the \linebreak restriction of $P$ to the Baire $\sigma$-algebra
of $\Delta(\mathscr{B})$. We then get
$\vonNeumAlg (\mathscr{B}) = \range(\pi _Q) = \mathscr{B}$,
by \ref{sepmain} \& \ref{Borelgen}.
\end{proof}

\medskip
Hence the following memorable result, cf.\ \ref{Abelian} \& \ref{Borelchar}.

\begin{theorem}\label{sepspm}%
If $P$ is a spectral measure acting on a separable Hilbert space
$\neq \{ 0 \}$, then $\range(\pi _P)$ is a von Neumann algebra.
\end{theorem}

\begin{remark}\label{nonsep}
The Riesz - von Neumann Theorem \ref{functions}, and with it
the theorems \ref{sepmain}, \ref{sepBorel}, and \ref{sepspm},
fail to hold in non-separable Hilbert spaces, see \cite [p.\ 65]{Nag}.
\pagebreak
\end{remark}

\clearpage

%


\section{Separability: the Domain}%
\label{separating}

\medskip
In this paragraph, let $H$ be a Hilbert space $\neq \{ 0 \}$, and let $B$
be a \st-subalgebra of $\blop(H)$ acting non-degenerately on $H$.

\begin{definition}[separating vectors]%
\index{concepts}{vector!separating}\index{concepts}{separating!vector}%
A non-zero vector $x \in H$ is called \underline{separating} for $B$,
whenever $bx = 0$ for some $b \in B$ implies $b = 0$.
\end{definition}

\begin{theorem}\label{cyclicsepcomm}%
A non-zero vector $x \in H$ is cyclic under $B$ \ref{cyclicsalg}
if and only if it is separating for $B'$.
\end{theorem}

\begin{proof}
Assume first that $x$ is cyclic under $B$, and let $c \in B'$ with $cx = 0$.
Then for $b \in B$, we have $cbx = bcx = 0$, whence $c = 0$ as
$Bx$ is dense in $H$. Next assume that $x$ is separating for $B'$.
With $p$ denoting the projection on the closed subspace $\overline{Bx}$,
we have $p \in B'$ by \ref{invarcommutant}, as well as $px = x$ by
\ref{cyclicsubspace}. Thus $\mathds{1} - p \in B'$ and
$( \mathds{1} - p ) x = 0$, which implies $\mathds{1} - p = 0$ as $x$ is
separating for $B'$. Hence $p = \mathds{1}$, which says that $x$ is
cyclic under $B$.
\end{proof}

\begin{corollary}
Whenever $C$ is a commutative cyclic \st-subalgebra of $\blop(H)$ with
cyclic vector $x$, then $x$ also is a separating vector for $C$.
\end{corollary}

\begin{proof}
Since $C$ is commutative, we have $C \subset C'$.
\end{proof}

\begin{theorem}\label{commsepsep}%
If $B$ is commutative and if $H$ is separable,
then $B$ has a separating vector.
\end{theorem}

\begin{proof}
The separable space $H$ is the countable direct sum of closed
subspaces $H_n$ (with each $n \geq 1$) invariant under $B$,
and such that for each $n$, the subspace $H _n$ contains a unit
vector $c_n$ cyclic under $B |_{{\textit{\Small{H}}}_\textit{\tiny{{n}}}}$,
cf.\ \ref{directsumdeco}. It shall be shown that the vector
$x := \sum _{n} {2}^{-n/2} c_n$ in the unit ball of $H$ is a separating
vector for $B$. So let $b \in B$ with $bx = 0$. With $p_n$ denoting
the projection on the closed invariant subspace $H _n$, we have
$p_n \in B'$, cf.\ \ref{invarcommutant}. Using commutativity of $B$,
we get
\[ b B c_n = b B p_n x = B b p_n x = B p_n b x= \{0\}, \]
which implies that $b$ vanishes on each $H _n$, and thus on
all of $H$.
\end{proof}

\begin{theorem}\label{sepmaxcomm}%
Let $\pi$ be a non-degenerate representation of a \linebreak \st-algebra
$A$ on $H$. Assume that $A$ is commutative, and that $H$ is \linebreak
separable. Then $\pi$ is cyclic if and only if $\pi$ is multiplicity-free
\ref{multfreedef}.\pagebreak
\end{theorem}

\begin{proof}
The ``only if'' part holds without the assumption of separability on $H$,
cf.\ \ref{cyclicfree}. The ``if'' part holds without the assumption
of commutativity on $A$, as we shall show next. If $\pi'$ is
commutative, then since $H$ is separable, $\pi'$ has a separating
vector, by \ref{commsepsep}, which will be cyclic under $\pi$,
by \ref{cyclicsepcomm}.
\end{proof}

\medskip
This gives a characterisation of cyclicity in terms of only the
algebraic structure of the commutant.

\begin{theorem}\label{maxcommchar}%
For a commutative von Neumann algebra $\mathscr{M}$ on a
\underline{separable} Hilbert space $\neq \{0\}$, the following
statements are equivalent.
\begin{itemize}
   \item[$(i)$] $\mathscr{M}$ is maximal commutative,
  \item[$(ii)$] $\mathscr{M}$ is cyclic,
 \item[$(iii)$] $\mathscr{M}$ is spatially equivalent to the maximal
                       commutative von Neumann algebra $\rmLeb ^{\infty}(\mu)$
                       for some perfect inner regular Borel probability measure
                       $\mu$ on the compact Hausdorff space $\Delta(\mathscr{M})$.
\end{itemize}
\end{theorem}

\begin{proof}
(i) $\Leftrightarrow$ (ii): \ref{multfreechar} and \ref{sepmaxcomm}.
(ii) $\Leftrightarrow$ (iii): \ref{commcyclicchar}.
\end{proof}

\begin{observation}
A normal subset of a \st-algebra $A$ is contained in a
maximal commutative \st-subalgebra of $A$, by Zorn's Lemma.
\end{observation}

\begin{corollary}[joint diagonalisation]\label{sepspatequiv}%
Let $N$ be a normal subset of $\blop(H)$. If $H$ is separable,
there exists a perfect inner regular Borel probability measure
$\mu$ on a compact Hausdorff space $K \neq \varnothing$,
as well as a unitary operator $U : H \to \rmLeb ^2(\mu)$, such
that for each operator $b \in N$, there is a function
$f \in \cont(K)$ with $UbU^{\,-1} = M\smu(f)$.
\end{corollary}

\begin{definition}[scalar spectral measure]%
\label{scalspmdef}%
\index{concepts}{measure!scalar spectral}%
\index{concepts}{spectral!measure!scalar}%
\index{concepts}{measure!spectral!scalar}%
\index{concepts}{scalar spectral measure}%
Let $P$ be a spectral \linebreak measure, defined on a $\sigma$-algebra
$\mathcal{E}$. Then a probability measure $\mu$, \linebreak defined on
$\mathcal{E}$, is called a \underline{scalar spectral measure}
for $P$, if for $\Delta \in \mathcal{E}$, one has $P(\Delta) = 0$
precisely when $\mu(\Delta) = 0$. That is, $P$ and $\mu$ should
be mutually absolutely continuous.
\end{definition}

We restrict the definition of scalar spectral measures to probability
measures for the reason that the applications - foremostly the crucial
item \ref{identify} below - require bounded measures.

\bigskip
The following two results deal with the existence of scalar spectral measures.
\pagebreak

\begin{proposition}[separating vectors for spectral measures]%
\index{concepts}{separating!vector}\index{concepts}{vector!separating}%
Let $P$ be a spectral measure, defined on a $\sigma$-algebra
$\mathcal{E}$, and acting on $H$. Then for a unit vector $x \in H$,
the probability measure $\langle P x, x \rangle$ is a scalar spectral
measure for $P$ if and only if $x$ is \underline{separating} for $P$.
(In the obvious extended sense that for $\Delta \in \mathcal{E}$,
one has $P(\Delta)x = 0$ only when $P(\Delta) = 0$.)
\end{proposition}

\begin{proof}
For $x \in H$ and $\Delta \in \mathcal{E}$, we have
\[ {\| \,P(\Delta)x \,\| \,}^2 = \langle P x , x \rangle (\Delta), \]
so $P(\Delta)x = 0$ if and only if $\langle P x , x \rangle (\Delta) = 0$.
Therefore $x$ is separating for $P$ if and only if $\Delta \in \mathcal{E}$
and $\langle P x , x \rangle (\Delta) = 0$ imply $P(\Delta) = 0$,
i.e.\ if and only if $\langle P x , x \rangle$ is a scalar spectral measure. 
\end{proof}

\begin{theorem}\label{exscalspm}%
Let $P$ be a spectral measure acting on $H$. If $H$ is separable,
then there exists a unit vector $x \in H$ such that the probability
measure $\langle P x, x \rangle$ is a scalar spectral measure for $P$.
The vector $x$ can be chosen to be any unit separating vector for $P$,
for example any unit separating vector for $\range(\pi_P)$.
\end{theorem}

\begin{proof}
The range $\range(\pi_P)$ has a separating vector by \ref{commsepsep}.
\end{proof}

\medskip
The next two results exhibit the purpose of scalar spectral measures.

\begin{theorem}\label{identify}%
Let $P$ be a spectral measure, defined on a $\sigma$-algebra
$\mathcal{E}$. If $P$ admits a probability measure $\mu$, defined
on $\mathcal{E}$, as a scalar spectral measure, then we may identify
the C*-algebras $\rmLeb ^{\infty}(P)$ and $\rmLeb ^{\infty}(\mu)$.
\end{theorem}

\begin{proof}
Assume that some probability measure $\mu$, defined on
$\mathcal{E}$, is a scalar spectral measure for $P$. A function $f$
in ${\mathscr{L} \,}^{\infty}(\mu)$ then is \linebreak $\mu$-a.e.\ equal
to a function $g \in \bmeas(\mathcal{E})$, cf.\ the appendix
\ref{general}, and the function $f$ is $\mu$-a.e.\ zero if and only if
the function $g$ is $P$-a.e.\ zero, by \ref{scalspmdef}.
It remains to apply \ref{Cstarisom}.
\end{proof}

\begin{theorem}\label{scalpurpose}%
Let $P$ be a spectral measure, defined on a $\sigma$-algebra
$\mathcal{E}$. If $P$ admits a probability measure $\mu$,
defined on $\mathcal{E}$, as a scalar spectral measure, then
$\pi_P$ factors to a C*-algebra isomorphism from $\rmLeb ^{\infty}(\mu)$
onto $\range(\pi_P)$.
\end{theorem}

\begin{proof}
This follows now from \ref{LPmain}. \pagebreak
\end{proof}

\medskip
We have the following remarkable result:

\begin{theorem}\label{deep}%
Let $P$ be a spectral measure acting on $H$. If $H$ is separable,
there exists a unit vector $x \in H$ such that $\pi_P$ factors to a
C*-algebra isomorphism from $\rmLeb ^{\infty} ( \langle P x, x \rangle )$
onto $\vonNeumAlg (P)$. The vector $x$ can be chosen to be any unit vector
in $H$, such that the probability measure $\langle P x, x \rangle$ is
a scalar spectral measure for $P$, for example any unit separating
vector for $P$, or even $\vonNeumAlg (P)$.
\end{theorem}

\begin{proof}
Since $H$ is separable, we know from \ref{sepspm} that $\range(\pi_P)$
is a von Neumann algebra, namely $\vonNeumAlg (\pi_P)$, which equals
$\vonNeumAlg (P)$ by \ref{spprojcomm}. The statement follows now from
\ref{exscalspm} and \ref{scalpurpose}.
\end{proof}

\begin{theorem}[the $\rmLeb ^{\infty}$ functional calculus]%
Let $\pi$ be any non-degenerate representation of a commutative
\st-algebra on $H$, and let $P$ be the spectral resolution of $\pi$.
If $H$ is separable, there exists a unit vector $x \in H$ such that
$\pi_P$ factors to a C*-algebra isomorphism from
$\rmLeb ^{\infty} ( \langle P x, x \rangle )$ onto $\vonNeumAlg (\pi)$.
The vector $x$ can be chosen to be any unit vector in $H$, such
that the probability measure $\langle P x, x \rangle$ is a scalar
spectral measure for $P$, for example any unit separating vector
for $P$, or even $\vonNeumAlg (\pi)$.
\end{theorem}

\begin{proof}
This follows from \ref{sepmain}, \ref{exscalspm} and \ref{scalpurpose}.
The result also follows from the preceding theorem \ref{deep} together
with \ref{gensame}.
\end{proof}

\medskip
Similarly, we obtain the following representation theorem
for commutative von Neumann algebras on separable
Hilbert spaces $\neq \{0\}$.

\begin{theorem}\label{vNCstar}%
A commutative von Neumann algebra $\mathscr{M}$ on a
\linebreak separable Hilbert space $\neq \{0\}$ is isomorphic
as a C*-algebra to $\rmLeb ^{\infty}(\mu)$ for some perfect
inner regular Borel probability measure $\mu$ on the
\linebreak compact Hausdorff space $\Delta(\mathscr{M})$.
The measure $\mu$ can be chosen to be any inner regular
Borel probability measure on $\Delta(\mathscr{M})$ which
is a scalar spectral measure for the spectral resolution of
$\mathscr{M}$.
\end{theorem}

\begin{proof}
For perfectness, see \ref{disconn} and the appendix \ref{Baire}.
\end{proof}

\medskip
Please note that here we only have a C*-algebra isomorphism,
while in \ref{maxcommchar}, we have a spatial equivalence
of von Neumann algebras, which is a stronger condition.
\pagebreak

\clearpage


\addtocontents{toc}{\protect\vspace{0.2em}}

\chapter{Single Unbounded Self-Adjoint Operators}%
\label{unbdself}

\setcounter{section}{46}


\section{Spectral Integrals of Unbounded Functions}

\medskip
In this paragraph, let $\mathcal{E}$ be a $\sigma$-algebra on a set
$\Omega \neq \varnothing$, let $H \neq \{0\}$ be a Hilbert space, and let
$P$ be a spectral measure defined on $\mathcal{E}$ and acting on $H$.

We first have to extend the notion of ``linear operator''.

\begin{definition}[linear operator in $H$]\index{concepts}{operator}%
By a \underline{linear operator in $H$} we mean a linear operator
defined on a subspace of $H$, taking values in $H$. We stress that
we then say: linear operator \underline{in} $H$; not \underline{on} $H$.
The domain (of definition) of such a linear operator $a$ is denoted
by $\domain(a)$.\index{symbols}{D2@$\domain(a)$}
\end{definition}

\begin{definition}[$\mathcal{M(E)}$]%
\index{symbols}{M3@$\mathcal{M(E)}$}%
We shall denote by $\mathcal{M(E)}$ the set of \linebreak
complex-valued $\mathcal{E}$-measurable functions on
$\Omega$, cf.\ the appendix, \ref{measurability}.
\end{definition}

We shall associate with each $f \in \mathcal{M(E)}$ a linear operator
in $H$, defined on a dense subspace $\mathcal{D}_f$ of $H$.
Next the definition of $\mathcal{D}_f$.

\begin{proposition}[$\mathcal{D}_f$]\index{symbols}{D3@$\mathcal{D}_f$}%
Let $f \in \mathcal{M(E)}$. One defines
\[ \mathcal{D}_f :=
\{ \,x \in H : f \in {\mathscr{L} \,}^{2}(\langle Px,x \rangle) \,\}. \]
The set $\mathcal{D}_f$ is a dense subspace of $H$.
\end{proposition}

\begin{proof}
For $x,y \in H$ and $\Delta \in \mathcal{E}$, we have
(the parallelogram identity):
\[ {\| \,P(\Delta)(x+y) \,\| \,}^2 + { \,\| \,P(\Delta)(x-y) \,\| \,}^2
= 2 \,\bigl( \,{\| \,P(\Delta)x \,\| \,}^2 + {\| \,P(\Delta)y \,\| \,}^2 \,\bigr), \]
whence
\[ \langle P(\Delta)(x+y),(x+y) \rangle
\leq 2 \,\bigl( \,\langle P(\Delta)x,x \rangle + \langle P(\Delta)y,y \rangle \,\bigr), \]
so that $\mathcal{D}_f$ is closed under addition. It shall next be
proved that $\mathcal{D}_f$ is dense in $H$. For $n \geq 1$ let
$\Delta_n := \{ \,| f \,| < n \,\}$. We have $\Omega = \cup _n \Delta _n$
as $f$ is complex-valued. Let $\Delta \in \mathcal{E}$.
For $z$ in the range of $P(\Delta _n)$, we have
\[ P(\Delta) z = P(\Delta) P(\Delta _n) z = P(\Delta \cap \Delta_n) z, \pagebreak \]
so that
\[ \langle Pz,z \rangle (\Delta) = \langle Pz,z \rangle (\Delta \cap \Delta _n), \]
and therefore
\[ \int {| \,f \,| \,}^2 \,\diff \langle Pz,z \rangle
= \int {| \,f \,| \,}^2 \,1_{\textstyle\Delta _n} \,\diff \langle Pz,z \rangle
\leq {n \,}^2 \,{\| \,z \,\| \,}^2 < \infty. \]
This shows that the range of $P(\Delta _n)$ is contained in
$\mathcal{D}_f$. We have
\[ y = \lim _{n \to \infty} P(\Delta _n) y \quad \text{for all} \quad y \in H \]
by $\Omega = \cup _n \Delta _n$ and by \ref{strsgcont}.
This says that $\mathcal{D}_f$ is dense in $H$.
\end{proof}

\begin{definition}[spectral integrals]\label{uspint}%
Consider some linear operator $a$ in $H$, defined on a
subspace $\domain(a)$ of $H$. Let $f \in \mathcal{M(E)}$.

One writes
\[ a = \int f \,\diff P \qquad \text{(weakly)} \]
if $\domain(a) = \mathcal{D}_f$ and if
\[ \langle ax,x \rangle = \int f \,\diff \langle Px,x \rangle \]
for all $x \in \mathcal{D}_f$. (Please note here that
${\mathscr{L} \,}^{2}(\langle Px,x \rangle) \subset {\mathscr{L} \,}^{1}(\langle Px,x \rangle)$
for all $x \in H$ by H\"older's inequality using boundedness of $\langle Px,x \rangle$).
By polarisation, there is at most one such operator $a$ for given $f$ and $P$.

We shall write
\[ a = \int f \,\diff P \qquad \text{(pointwise)} \]
if $\domain(a) = \mathcal{D}_f$ and if for all $x \in \mathcal{D}_f$ and all
$h \in \bmeas(\mathcal{E})$ one has
\[ \| \,\bigl(a - \pi _P(h)\bigr) x \,\| 
= \| \,f-h \,\|_{\text{\small{$\langle Px,x \rangle,2$}}}
= {\biggl(\int {| \,f-h \,| \,}^2 \ \diff \langle Px,x \rangle \biggr)}^{1/2}. \]
(Please note that $\bmeas(\mathcal{E})
\subset {\mathscr{L} \,}^{2}(\langle Px,x \rangle)$
for all $x \in H$ by boundedness of $\langle Px,x \rangle$.)
\end{definition}

\begin{proposition}%
Let $f \in \mathcal{M(E)}$ and let $a$ be a linear operator in $H$. Assume that
\[ a = \int f \,\diff P \qquad \text{(pointwise)}. \]
We then also have
\[ a = \int f \,\diff P \qquad \text{(weakly)}. \]
In particular, there is at most one such operator $a$ for given $f$ and $P$.
\pagebreak
\end{proposition}

\begin{proof}
For $n \geq 1$, put
$f_n := f \,1_{\text{\small{$\{ \,| \,f \,| < n \,\}$}}} \in \bmeas(\mathcal{E})$.
For $x \in \mathcal{D}_f$, the assumption implies that
\[ ax = \lim _{n \to \infty} \pi _P(f_n)x \]
by the Lebesgue Dominated Convergence Theorem,
whence, by the same theorem, (as
${\mathscr{L} \,}^{2}(\mu) \subset {\mathscr{L} \,}^{1}(\mu)$ for bounded $\mu$),
\[ \langle ax,x \rangle
= \lim _{n \to \infty} \langle \pi _P(f_n)x,x \rangle
= \lim _{n \to \infty} \int f_n \,\diff \langle Px,x \rangle
= \int f \,\diff \langle Px,x \rangle. \qedhere \]
\end{proof}

\begin{theorem}[$\Psi_P(f)$]\index{symbols}{P99@$\Psi _P(f)$}%
For $f \in \mathcal{M(E)}$ there exists a (necessarily unique)
linear operator $\Psi _P(f)$ in $H$ such that
\[ \Psi _P(f) = \int f \,\diff P \qquad \text{(pointwise)}. \]
\end{theorem}

\begin{proof}
Let $x \in \mathcal{D}_f$ be fixed. Look at
$\bmeas(\mathcal{E})$ as a subspace of \linebreak
${\mathscr{L} \,}^{2}(\langle Px,x \rangle)$.
Consider the map
\[ I_x : {\mathscr{L} \,}^{2}(\langle Px,x \rangle) \supset
\bmeas(\mathcal{E}) \to H,\quad h \mapsto \pi _P(h) x. \]
Since $\bmeas(\mathcal{E})$ is dense in
${\mathscr{L} \,}^{2}(\langle Px,x \rangle)$, and since $I_x$
is (sort of) an isometry by \ref{normpointwise}, it follows
that $I_x$ has a unique continuation to an isometry
$\rmLeb ^2(\langle Px,x \rangle) \to H$, which shall also be
denoted by $I_x$. Putting \linebreak $\Psi _P(f)x := I_x(f)$,
we obtain a function $\Psi _P(f) : \mathcal{D}_f \to H$.
For any $h \in \bmeas(\mathcal{E})$, we have
\[ \| \,\bigl( \Psi _P(f) - \pi _P(h)\bigr)x \,\|
= \| \,I_x(f-h) \,\|
= \| \,f-h \,\|_{\text{\small{$\langle Px,x \rangle,2$}}} \]
because $I_x$ is an isometry. In particular, the
function $\Psi _P(f)$ is linear as it is approximated
pointwise by linear operators $\pi _P(h)$. We conclude that
\[ \Psi _P(f) = \int f \,\diff P \qquad \text{(pointwise)}. \qedhere \]
\end{proof}

\begin{corollary}%
Let $f \in \mathcal{M(E)}$. Let $a$ be a linear operator in $H$ such that
\[ a = \int f \,\diff P \qquad \text{(weakly)}. \]
We then also have
\[ a = \int f \,\diff P \qquad \text{(pointwise)}. \]
\end{corollary}

\begin{proof}
We have $a = \Psi _P(f)$ by $\Psi _P(f) = \int f \,\diff P$ (weakly) as well.%
\pagebreak
\end{proof}

\begin{proposition}\label{heqnull}%
Let $f \in \mathcal{M(E)}$. For $x \in \mathcal{D}_f$, we then have
\[ \| \,\Psi _P (f) \,x \,\| 
= \| \,f \,\|_{\text{\small{$\langle Px,x \rangle,2$}}}
= {\biggl(\int {| \,f \,| \,}^2 \ \diff \langle Px,x \rangle \biggr)}^{1/2}. \]
\end{proposition}

\begin{proof}
To see this, put $h := 0$ in the defining relation \ref{uspint} for
\[ \Psi _P(f) = \int f \,\diff P \qquad \text{(pointwise)}. \qedhere \]
\end{proof}

\medskip
For the next item, please recall \ref{imagespdef} \& \ref{imagespthm}.

\begin{theorem}[integration with respect to an image measure]%
\label{imagesputhm}%
\index{concepts}{spectral!measure!image}%
\index{concepts}{image!spectral measure}%
\index{concepts}{measure!spectral!image}%
Let $\mathcal{E}'$ be another $\sigma$-algebra on another set
$\Omega' \neq \varnothing$, and let $f : \Omega \to \Omega'$
be an $\mathcal{E}$-$\mathcal{E}'$ measurable function.
For $g \in \mathcal{M}(\mathcal{E'})$, one has
\[ \int g\ \diff\mspace{2mu}f (P) = \int (g \circ f) \,\diff P,
\quad \text{that is,} \quad \Psi_{f(P)}\,(g) = \Psi_P \,(g \circ f). \]
\end{theorem}

\begin{proof}
This follows from the appendix \ref{imageend}.
\end{proof}

\begin{lemma}\label{Dlemma}%
Let $g \in \mathcal{M(E)}$ and $x \in \mathcal{D}_g$.
For $y := \Psi _P(g)x$, we then have
\[ \langle Py,y \rangle = {| \,g \,| \,}^2 \cdot \langle Px,x \rangle. \]
Moreover, for $h \in \bmeas(\mathcal{E})$, we have
\[ \pi _P(h) y = \Psi _P(hg)x. \]
\end{lemma}

\begin{proof}
With $g_n := g \,1_{\text{\small$\{ \,| \,g \,| < n \,\}$}}$
for all integers $n \geq 1$, we have
\[ y = \Psi _P(g)x = \lim _{n \to \infty} \pi _P(g_n)x. \]
For $h \in \bmeas(\mathcal{E})$, we obtain
\begin{align*}
\pi _P(h)y = \pi _P(h) \Psi _P(g)x
& = \lim _{n \to \infty} \pi _P(h) \pi _P(g_n)x \\
& = \lim _{n \to \infty} \pi _P(h g_n)x = \Psi _P(hg)x.
\end{align*}
The last equality follows from $\mathcal{D}_{hg} \supset \mathcal{D}_g$
by boundedness of $h$. In particular, from \ref{heqnull} we get
\[ \int {| \,h \,| \,}^2 \,\diff \langle Py,y \rangle = {\| \,\pi_P(h)y \,\| \,}^2
= {\| \,\Psi_P(hg)x \,\| \,}^2 = \int {| \,hg \,| \,}^2 \,\diff \langle Px,x \rangle \]
for all $h \in \bmeas(\mathcal{E})$. This means that
\[ \langle Py,y \rangle
= {| \,g \,| \,}^2 \cdot \langle Px,x \rangle. \pagebreak \qedhere \]
\end{proof}

\begin{definition}[extension]%
If $a$ and $b$ are linear operators in $H$ with respective
domains $\domain(a)$ and $\domain(b)$, one writes $a \subset b$,
if $b$ is an \underline{extension} of $a$, i.e.\ if
$\domain(a) \subset \domain(b)$ and $ax = bx$ for all $x \in \domain(a)$.
\end{definition}

For operations with linear operators $a, b$ in $H$, we equip the result with
the \underline{maximal} domain. For example
$\domain(a+b) := \domain(a) \cap \domain(b)$,
and $\domain(ab) := \{ \,x \in \domain(b) : b x \in \domain(a) \,\}$.

\begin{theorem}[Addition and Multiplication Theorems]\label{AddMult}%
\index{concepts}{Theorem!Addition \& Multiplication}%
\index{concepts}{Addit.\ \& Multiplicat.\ Thm.}%
For $f,g$ in $\mathcal{M(E)}$ we have
\begin{gather*}
\Psi _P(f) + \Psi _P(g) \subset \Psi _P(f+g), \\
\Psi _P(f) \,\Psi _P(g) \subset \Psi _P(fg).
\end{gather*}
The precise domains are given by
\begin{gather*}
D\bigl(\Psi _P(f) + \Psi _P(g)\bigr) = \mathcal{D}_{| \,f \,| + | \,g \,|}, \\
D\bigl(\Psi _P(f) \,\Psi _P(g)\bigr) = \mathcal{D}_{fg} \cap \mathcal{D}_g.
\end{gather*}
\end{theorem}

\begin{proof}
We shall first prove the statements relating to addition.
Let $f,g \in \mathcal{M(E)}$. If $x \in D\bigl(\Psi _P(f) + \Psi _P(g)\bigr)
= D\bigl(\Psi _P(f)\bigr) \cap D\bigl(\Psi _P(g)\bigr)$, then
$f, g \in {\mathscr{L} \,}^{2}(\langle Px,x \rangle)$, whence
$f + g \in {\mathscr{L} \,}^{2}(\langle Px,x \rangle)$, so that
$x \in \mathcal{D}_{f+g}$. It follows that
\[ \Psi _P(f+g)x = \Psi _P(f)x+\Psi _P(g)x \]
because $\Psi _P$ is approximated pointwise by the linear map $\pi _P$.
We have
\[ D\bigl(\Psi _P(f)+\Psi _P(g)\bigr) = \mathcal{D}_{| \,f \,|+| \,g \,|}. \]
Indeed, if $x \in H$, then for the complex-valued $\mathcal{E}$-measurable
functions $f$ and $g$, we have that both
$f, g \in {\mathscr{L} \,}^{2}(\langle Px,x \rangle)$
if and only if $| \,f \,|+| \,g \,| \in {\mathscr{L} \,}^{2}(\langle Px,x \rangle)$.

We shall now prove the statements relating to composition.
Let $g \in \mathcal{M(E)}$, $x \in \mathcal{D}_g$, and put $y := \Psi _P(g)x$.
Let $f \in \mathcal{M(E)}$. Lemma \ref{Dlemma} shows that
$y \in \mathcal{D}_f$ if and only if $x \in \mathcal{D}_{fg}$, and so
\[ D\bigl(\Psi _P(f) \,\Psi _P(g)\bigr) = \mathcal{D}_{fg} \cap \mathcal{D}_g
\qquad (f,g \in \mathcal{M(E)}). \]
For $n \geq 1$, define $f_n := f \,1_{\text{\small{$\{ \,| \,f \,| < n \,\}$}}}$.
If $x \in \mathcal{D}_{fg} \cap \mathcal{D}_g$, then
$f_n g \to fg$ in ${\mathscr{L} \,}^{2}(\langle Px,x \rangle)$ and so
$f_n \to f$ in ${\mathscr{L} \,}^{2}(\langle Py,y \rangle)$ by \ref{Dlemma}.
Applying \ref{Dlemma} with $h := f_n$, we obtain
\begin{align*}
\Psi _P(f) \Psi _P(g) x = \Psi _P(f) y & = \lim _{n \to \infty} \pi _P(f_n) y \\
& = \lim _{n \to \infty} \Psi _P(f_n g) x = \Psi _P(fg) x. \pagebreak \qedhere
\end{align*}
\end{proof}

\begin{definition}[self-adjoint linear operators]\label{selfadjop}%
\index{concepts}{self-adjoint!operator}%
\index{concepts}{operator!self-adjoint}%
\index{concepts}{adjoint!operator}%
\index{concepts}{operator!adjoint}%
\index{concepts}{symmetric operator}%
\index{concepts}{operator!symmetric}%
A linear operator $b$ in $H$ is called \underline{symmetric} if
$\langle b x, y \rangle = \langle x, b y \rangle$ for all $x, y \in \domain(b)$.

Let $a$ be a linear operator in $H$ defined on a \underline{dense}
subspace $\domain(a)$. One defines
\[ \domain(a^*) := \{ \,y \in H : \exists \ y^* \in H \text{ with }
                                   \langle ax,y \rangle = \langle x,y^* \rangle
                                   \ \forall \ x \in \domain(a) \,\}, \]
which is a subspace of $H$. The associated $y^*$ is determined
uniquely because $\domain(a)$ is dense in $H$. Thus
\[ a^*y := y^* \qquad \bigl( \,y \in \domain(a^*) \,\bigr) \]
determines a linear operator $a^*$ in $H$, with domain $\domain(a^*)$.
The operator $a^*$ is called the \underline{adjoint} of $a$. The
operator $a$ is called \underline{self-adjoint} if $a = a^*$.
The operator $a$ is symmetric if and only if $a \subset a^*$.
\end{definition}

\begin{definition}[closed operators]%
\index{concepts}{closed operator}\index{concepts}{operator!closed}%
A linear operator $a$ in $H$ with domain $\domain(a) \subset H$,
is called \underline{closed}, if its graph is closed in $H \times H$.
\end{definition}

\begin{proposition}\label{adjclosed}%
The adjoint of a densely defined linear operator in $H$ is closed.
\end{proposition}

\begin{proof}
Let $a$ be a densely defined linear operator in $H$.
Consider $(y, z)$ in the closure of the graph of $a^*$.
Choose a sequence $(y_n)$ in $\domain(a^*)$ such that
$y_n \to y$ and $a^*y_n \to z$ in $H$. For $x \in \domain(a)$, we get
\[ \langle ax,y \rangle = \lim _{n \to \infty} \langle ax,y_n \rangle
   = \lim _{n \to \infty} \langle x,a^*y_n \rangle = \langle x,z \rangle, \]
whence $y \in \domain(a^*)$ and $a^*y = z$. Thus $(y, z)$ is in the
graph of $a^*$.
\end{proof}

\begin{corollary}\label{selfadjclosed}%
A self-adjoint linear operator in $H$ is closed.
\end{corollary}

\begin{proposition}\label{submult}%
If $a,b$ are densely defined linear operators in $H$,
and if $ab$ is densely defined as well, then
\[ b^* a^* \subset (ab)^*. \]
\end{proposition}

\begin{proof}
Let $y \in \domain(b^*a^*)$. For arbitrary $x \in \domain(ab)$, we get
\[ \langle abx,y \rangle = \langle bx,a^*y \rangle \]
because $bx \in \domain(a)$ and $y \in \domain(a^*)$. We also have
\[ \langle bx,a^*y \rangle = \langle x,b^*a^*y \rangle \]
because $x \in \domain(b)$ and $a^*y \in \domain(b^*)$. It follows that
$\langle abx,y \rangle = \langle x,b^*a^*y \rangle$.
Hence $y \in D\bigl((ab)^*\bigr)$ and $(ab)^*y = b^*a^*y$.
\pagebreak
\end{proof}

\begin{theorem}%
For $f \in \mathcal{M(E)}$, the following statements hold:
\begin{itemize}
   \item[$(i)$] ${\Psi _P(f) \,}^* = \Psi _P(\overline{f})$,
  \item[$(ii)$] $\Psi _P(f)$ is closed,
 \item[$(iii)$] ${\Psi _P(f) \,}^* \,\Psi _P(f) = \Psi _P \bigl( \,{| \,f \,| \,}^2 \,\bigr)
= \Psi _P(f) \,{\Psi _P(f) \,}^*$.
\end{itemize}
\end{theorem}

\smallskip
\begin{proof}
(i): For $x \in \mathcal{D}_f$, and
$y \in \mathcal{D}_f = \mathcal{D}_{\overline{f}}$ we have by polarisation
\begin{align*}
\langle \Psi _P(f)x,y \rangle = \int f \,\diff \langle Px,y \rangle
& = \overline{\int \overline{f} \,\diff \langle Py,x \rangle} \\
& = \overline{\langle \Psi _P(\overline{f}) y,x \rangle}
= \langle x,\Psi _P(\overline{f})y \rangle,
\end{align*}
so that $y \in D\bigl({\Psi _P(f) \,}^*\bigr)$ and
$\Psi _P(\overline{f}) \subset {\Psi _P(f) \,}^*$.
In order to prove the reverse inclusion, consider the truncations
$h_n := 1_{\text{\small{$\{ \,| \,f \,| < n \,\}$}}}$ for $n \geq 1$.
The Multiplication Theorem \ref{AddMult} yields
\[ \Psi _P(f) \,\pi _P(h_n) = \pi _P (f h_n). \]
One concludes with the help of \ref{submult} that
\[ \pi _P(h_n) \,{\Psi _P(f) \,}^*
\subset {[ \,\Psi _P(f) \,\pi _P(h_n) \,] \,}^*
= {\pi _P(f h_n) \,}^* = \pi _P(\overline{f} h_n). \]
For $z \in D \,\bigl( \,{\Psi _P(f) \,}^* \,\bigr)$ and $v = {\Psi _P(f) \,}^* \,z$ it follows
\[ \pi _P(h_n)v = \pi _P(\overline{f} h_n)z. \]
Hence
\[ \int {| \,\overline{f} h_n \,| \,}^2 \,\diff \langle Pz,z \rangle
= \int h_n \,\diff \langle Pv,v \rangle \leq \langle v,v \rangle \]
for all $n \geq 1$, so that $z \in \mathcal{D}_{\overline{f}}$.

(ii) follows from (i) together with \ref{adjclosed}.

(iii) follows now from the Multiplication Theorem  \ref{AddMult}
because $\mathcal{D}_{\overline{f}f} \subset \mathcal{D}_f$
by the Cauchy-Schwarz inequality:
\[ {\left| \,\int {| \,f \,| \,}^2 \,\diff \langle Px,x \rangle \,\right| \,}^2
\leq \int {| \,f \,| \,}^4 \,\diff \langle Px,x \rangle
\cdot \int {1 \,}^2 \,\diff \langle Px,x \rangle. \qedhere \]
\end{proof}

\bigskip
\begin{corollary}\label{Psifselfadj}%
If $f$ is a real-valued function in $\mathcal{M(E)}$,
then $\Psi _P(f)$ is a self-adjoint linear operator in $H$.
\end{corollary}

\clearpage


\section{The Spectral Theorem for Self-Adjoint Operators}

In this paragraph, let $a : H \supset \domain(a) \to H$ be a
self-adjoint linear operator in a Hilbert space $H \neq \{0\}$,
defined on a dense subspace $\domain(a)$ of $H$.

\begin{definition}[the spectrum]%
\index{concepts}{spectrum!of a self-adjoint operator}%
One says that $\lambda \in \mathds{C}$ is a regular value of the
self-adjoint linear operator $a$, if $\lambda \mathds{1}- a$ is injective
as well as surjective onto $H$, and if its left inverse is bounded.
The \underline{spectrum} $\s(a)$ of $a$ is defined as the set of
those complex numbers which are not regular values of $a$.
\end{definition}

\begin{theorem}\label{specisreal}%
The spectrum of the self-adjoint operator $a$ is real.
\end{theorem}

\begin{proof}
Let $\lambda \in \mathds{C} \setminus \mathds{R}$
and let $\lambda =: \alpha + \iu \beta$ with $\alpha$ and $\beta$ real.
For $x \in \domain(a)$, we get
\[ {\| \,\bigl( ( \alpha \mathds{1} - a ) + \iu \beta \mathds{1} \bigr) x \,\| \,}^2
= {\| \,( \alpha \mathds{1} - a ) x \,\| \,}^2 + {\beta \,}^2 \,{\| \,x \,\| \,}^2
\geq {\beta \,}^2 \,{\| \,x \,\| \,}^2, \]
which shows that $\lambda \mathds{1} - a$ has a bounded left inverse.
It remains to be shown that this left inverse is everywhere defined.
It shall first be shown that the range $\range(\lambda \mathds{1} -a)$
is closed. Let $(y_n)$ be a sequence in $\range(\lambda \mathds{1} - a)$
which converges to an element $y \in H$. There then exists a sequence
$(x_n)$ in $\domain(a)$ such that $y_n = (\lambda \mathds{1} - a)x_n$
for all $n$. Next we have
\[ \| \,x_n-x_m \,\| \,\leq \,{| \,\beta \,| \,}^{-1}
\cdot \,\| \,( \lambda \mathds{1} - a ) (x_n-x_m) \,\|
\,= \,{| \,\beta \,| \,}^{-1} \cdot \,\| \,y_n-y_m \,\|, \]
so that $(x_n)$ is a Cauchy sequence and thus converges to
some $x \in H$. Since $a$ is closed \ref{selfadjclosed}, this implies
that $x \in \domain(a)$ and $y = ( \lambda \mathds{1} - a ) x$, so
that $y \in \range(\lambda \mathds{1}-a)$, which shows that
$\range(\lambda \mathds{1}-a)$ is closed. It shall be shown
that $\range(\lambda \mathds{1} - a) = H$. Let $y \in H$ be
orthogonal to $\range(\lambda \mathds{1} - a)$. We have to
show that $y = 0$. We get
\[ \langle (\lambda \mathds{1}-a)x,y \rangle = 0 = \langle x,0 \rangle
\quad \text{for all } x \in \domain(a), \]
which means that
$(\lambda \mathds{1} - a)^*y = (\overline{\lambda}\mathds{1}-a)y$
exists and equals zero. It follows that $y = 0$ because
$\overline{\lambda}\mathds{1}-a$ is injective by the above.
\end{proof}

\begin{definition}[the Cayley transform]%
\index{concepts}{Cayley transform}%
The linear operator
\[ u := (a - \iu \mathds{1}) (a + \iu \mathds{1})^{-1} \]
is called the \underline{Cayley transform} of $a$. Please note that
$u$ is well-defined as a linear operator taking $H$ to itself
by the preceding theorem \ref{specisreal}.\pagebreak
\end{definition}

\begin{proposition}\label{onenoteival}%
The Cayley transform $u$ of $a$ is unitary and $1$ is not an
eigenvalue of $u$. 
\end{proposition}

\begin{proof} For $x \in \domain(a)$ we have
\[ {\| \,(a - \iu \mathds{1}) x \,\| \,}^2 = {\| \,ax \,\| \,}^2 + {\| \,x \,\| \,}^2
= {\| \,(a + \iu \mathds{1}) x \,\| \,}^2, \]
so that
\[ \| \,u y \,\| = \| \,(a - \iu \mathds{1}) (a + \iu \mathds{1})^{-1} y \,\| = \| \,y \,\|
\quad \text{for all} \quad y \in \domain(u), \]
which shows that $u$ is isometric. It follows from theorem \ref{specisreal}
that $u$ is bijective, hence unitary. It shall next be shown that $1$
is not an eigenvalue of $u$. For $x \in \domain(a)$, let
$y := (a + \iu \mathds{1})x$. We then have $uy = (a - \iu \mathds{1})x$.
By subtraction we obtain $(\mathds{1}-u)y = 2 \iu x$. Thus, if
$(\mathds{1}-u)y = 0$, then $x = 0$, whence $y = 0$. This
shows the injectivity of $\mathds{1}-u$.
\end{proof}

\begin{proposition}\label{spau}%
Let
\[ u := (a - \iu \mathds{1}) (a + \iu \mathds{1})^{-1} \]
be the Cayley transform of $a$.
For $\lambda \in \mathds{R}$, consider
\[ \nu := (\lambda - \iu) (\lambda + \iu)^{-1}. \]
Then $\lambda \in \s(a)$ if and only if $\nu \in \s(u)$.
\end{proposition}

\begin{proof}
We compute
\begin{align*}
 &\ u-\nu \mathds{1}\\
= &\ \left[ (a - \iu \mathds{1})-\frac{\lambda -\iu}{\lambda + \iu} \,(a + \iu \mathds{1}) \right]
\cdot (a + \iu \mathds{1})^{-1} \\
= &\ \frac{1}{\lambda + \iu} \,\Bigl[ \,(\lambda + \iu) \,(a - \iu \mathds{1})
    -(\lambda - \iu) \,(a + \iu \mathds{1}) \,\Bigr] \cdot (a + \iu \mathds{1})^{-1} \\
= &\ \frac{2 \iu}{\lambda + \iu} \ (a - \lambda \mathds{1}) \cdot (a + \iu \mathds{1})^{-1}.
\end{align*}
This shows that $u-\nu \mathds{1}$ is injective precisely when
$a-\lambda \mathds{1}$ is so. Also, for the ranges, we have
$\range(u - \nu \mathds{1}) = \range(a - \lambda \mathds{1})$,
so that the statement follows from the Closed Graph Theorem.
(Indeed, an injective operator is closed if and only if its left
inverse is closed, as is easily seen.)
\end{proof}

\begin{corollary}%
The spectrum $\s(a)$ of $a$ is a non-empty closed subset of the real line.
In particular, it is locally compact. \pagebreak
\end{corollary}

\begin{proof}
The Moebius transformation $\lambda \mapsto \nu$ of the preceding
proposition \ref{spau} maps $\mathds{R} \cup \{\infty\}$ homeomorphically
onto the unit circle, the point $\infty$ being mapped to $1$, cf.\ \ref{line}.
Now the spectrum of a unitary operator on a Hilbert space is a non-empty
compact subset of the unit circle, cf.\ \ref{C*unitary} and \ref{specradform}.
If $a$ had empty spectrum, then the spectrum of its Cayley transform would
consist of $1$ alone, which would therefore be an eigenvalue, cf.\ \ref{eival},
in contradiction with \ref{onenoteival}.%
\end{proof}

\begin{proposition}\label{swapadj}%
Let $b, c$ be two densely defined linear operators in $H$ such that
$b \subset c$. Then $c^* \subset b^*$.
\end{proposition}

\begin{proof}
Let $y \in \domain(c^*)$. For all $x \in \domain(b)$, we have
\[ \langle bx,y \rangle = \langle x,c^*y \rangle \]
because $b \subset c$. It follows that $y \in \domain(b^*)$
and $b^*y = c^*y$, so $c^* \subset b^*$.
\end{proof}

\medskip
For the next item, please recall \ref{selfadjop}.

\begin{definition}[maximal symmetric operators]%
\index{concepts}{operator!symmetric!maximal}%
\index{concepts}{symmetric operator!maximal}%
\index{concepts}{maximal!symmetric}%
A symmetric \linebreak linear operator in $H$ is called
\underline{maximal symmetric} if it has no proper symmetric extension.
\end{definition}

\begin{proposition}\label{ismaxsymm}%
The self-adjoint linear operator $a$ is maximal symmetric.
\end{proposition}

\begin{proof}
Suppose that $b$ is a symmetric extension of $a$. By proposition
\ref{swapadj} above, we then have
\[ b \subset b^* \subset a^* = a \subset b, \]
so that $a = b$.
\end{proof}

\begin{definition}\label{commute}%
\index{concepts}{operator!commuting}\index{concepts}{commuting operator}%
A \underline{bounded} linear operator $c \in \blop(H)$ is said to
\underline{commute} with $a$ if $\domain(a)$ is invariant under $c$
and if $cax = acx$ for all $x \in \domain(a)$.
\end{definition}

\begin{proposition}%
A bounded linear operator $c \in \blop(H)$ commutes with $a$
if and only if $ca \subset ac$, as is seen from a moment's thought.
\end{proposition}

\begin{definition}\index{symbols}{a65@$\{a\}'$}%
The set of bounded linear operators in $\blop(H)$
commuting with $a$ is denoted by $\{a\}'$. \pagebreak
\end{definition}

\begin{theorem}%
[the Spectral Theorem for unbounded  self-adjoint operators]%
\index{concepts}{Theorem!Spectral!self-adjoint unb.\ operator}%
\index{concepts}{spectral!theorem!self-adjoint unb.\ operator}%
\index{concepts}{spectral!resolution!self-adjoint unb.\ operator}%
\label{specthmunbded}
For the self-adjoint linear operator $a$ in $H$, there exists a unique
resolution of identity $P$ on $\s(a)$, acting on $H$, such that
\[ a = \int _{\text{\small{$\s(a)$}}} \id_{\text{\small{$\s(a)$}}} \,\diff P
\qquad \text{(pointwise)}. \]
This is usually written as
\[ a = \int \lambda \,\diff P(\lambda). \]
One says that $P$ is the \underline{spectral resolution} of $a$. We have
\[ \{a\}' = P', \]
and the support of $P$ is all of $\s(a)$. 
\end{theorem}

We split the proof into two lemmata.

\begin{lemma}\label{firstlemmaspecthm}%
Let $u$ denote the Cayley transform of $a$. Put
\[ \Omega := \s(u) \setminus \{1\}. \]
The function $\kappa$ given by
\[ \kappa : \lambda \mapsto (\lambda - \iu)(\lambda + \iu)^{-1} \]
maps $\s(a)$ homeomorphically onto $\Omega$, cf.\ \ref{spau} \& \ref{line}. \\
Let $Q$ be the spectral resolution of $u$. Then $Q(\{1\}) = 0$, so that also
\[ u = \int _{\Omega} \id_{\Omega} \,\diff Q = \int \id \,\diff (Q|_{\Omega})
       \qquad \text{(pointwise)}. \]
Let $P := \kappa^{-1}(Q|_{\Omega})$ be the image of $Q|_{\Omega}$
under $\kappa^{-1}$. Then $P$ is a spectral resolution on $\s(a)$ such that
\[ a = \int _{\text{\small{$\s(a)$}}} \id_{\text{\small{$\s(a)$}}} \,\diff P
      \qquad \text{(pointwise)}. \]
Let conversely $P$ be a resolution of identity on $\s(a)$ such that
\[ a = \int _{\text{\small{$\s(a)$}}} \id_{\text{\small{$\s(a)$}}} \,\diff P
      \qquad \text{(pointwise)}. \]
The image $\kappa(P)$ of $P$ under $\kappa$ then is $Q|_{\Omega}$.
\pagebreak
\end{lemma}

\begin{proof}
We have $Q(\{1\}) = 0$ as $1$ is not an eigenvalue of $u$, see
\ref{onenoteival} and \ref{eival}. Hence also
\[ u = \int \id \,\diff Q = \int _{\Omega} \id_{\Omega} \,\diff Q = \int \id \,\diff (Q|_{\Omega}). \]
As in \ref{imagespdef} above, let
\[ P := \kappa^{-1}(Q|_{\Omega}) \]
be the image of $Q|_{\Omega}$ under $\kappa^{-1}$ and consider
\[ b := \Psi _P(\id) = \int \id \,\diff P = \int \id \,\diff \bigl(\kappa^{-1}(Q|_{\Omega})\bigr), \]
which is a self-adjoint operator, cf.\ \ref{Psifselfadj}. By \ref{imagesputhm} above, we have
\[ b = \int _{\Omega}\kappa^{-1} \,\diff Q = \int _{\Omega} \iu \,\frac{1+\id}{1-\id} \,\diff Q. \]
It shall be shown that $b = a$. Since
\[ \mathds{1}-u = \int _{\Omega} (1-\id) \,\diff Q, \]
the Multiplication Theorem \ref{AddMult} gives
\[ b \,(\mathds{1}-u) = \int _{\Omega} \iu \,(1+\id) \,\diff Q = \iu \,(\mathds{1}+u), \tag*{$(i)$} \]
and in particular
\[ \range(\mathds{1}-u) \subset \domain(b). \tag*{$(ii)$} \]
On the other hand, for $x \in \domain(a)$, put
\[ (a + \iu \mathds{1}) \,x =: z. \]
We then have
\[ (a - \iu \mathds{1}) \,x = u \mspace{2mu} z, \]
so
\begin{align*}
(\mathds{1}-u) \,z & = 2 \,\iu \,x \tag*{$(iii)$} \\
(\mathds{1}+u) \,z & = 2 \,a \,x,
\end{align*}
whence
\[ a \,(\mathds{1}-u) \,z = a \,2 \,\iu \,x = \iu \,2 \,a \,x = \iu \,(\mathds{1}+u) \,z \tag*{$(iv)$} \]
for all $z \in \range(a + \iu \mathds{1})$, which is equal to $H$ as $\s(a)$ is real,
cf.\ \ref{specisreal}. From $(ii)$ and $(iii)$ we have
\[ \domain(a) = \range(\mathds{1}-u) \subset \domain(b). \tag*{$(v)$}  \]
From $(i)$, $(iv)$, and $(v)$, it follows that $b$ is a self-adjoint extension of $a$.
As $a$ is maximal symmetric by \ref{ismaxsymm}, this yields that
$a = b = \Psi _P(\id)$.\pagebreak

Let conversely $P$ be a resolution of identity on $\s(a)$ such that
\[ a = \int _{\text{\small{$\s(a)$}}} \id_{\text{\small{$\s(a)$}}} \,\diff P. \]
Let $R : = \kappa(P)$ be the image of $P$ under $\kappa$ and put
\[ v := \Psi _{R} (\id) = \int \id \,\diff R = \int \id \,\diff \bigl(\kappa(P)\bigr). \]
By \ref{imagesputhm} above again, we find
\[ v= \int \kappa \,\diff P = \int \frac{\id - \iu}{\id + \iu} \,\diff P, \]
so that the Multiplication Theorem \ref{AddMult} gives
\[ v \,(a + \iu \mathds{1}) = a - \iu \mathds{1}, \]
or
\[ v = (a - \iu \mathds{1})(a + \iu \mathds{1})^{-1} = u. \]
By the uniqueness of the spectral resolution, we must have
\[ \kappa(P) = R = Q|_{\Omega}. \qedhere \]
\end{proof}

\begin{lemma}%
If $u$ is the Cayley transform of $a$, then
\[ \{a\}' = \{u\}'. \]
\end{lemma}

\begin{proof}
If $c \in \{a\}'$ then
\[ c \,(a \pm \iu \mathds{1}) \subset (a \pm \iu \mathds{1}) \,c, \]
whence
\[ (a \pm \iu \mathds{1})^{-1} \,c = c \,(a \pm \iu \mathds{1})^{-1}. \]
It follows that $c \,u = u \,c$.

Conversely, let $c \in \{u\}'$. From the formulae $(iv)$ and $(v)$
of the preceding lemma \ref{firstlemmaspecthm}, $\mathds{1}-u$
has range $\domain(a)$, and
\[ a \,(\mathds{1}-u) = \iu \,(\mathds{1}+u). \]
Therefore we have
\begin{align*}
c \,a \,(\mathds{1}-u) & = \iu \,c \,(\mathds{1}+u) \\
 & = \iu \,(\mathds{1}+u) \,c \\
 & = a \,(\mathds{1}-u) \,c = a \,c \,(\mathds{1}-u),
\end{align*}
whence $c \,a \subset a \,c$. \pagebreak
\end{proof}

\begin{theorem}[the inverse Cayley transform]%
The self-adjoint operator $a$ can be regained from its Cayley transform $u$
as the inverse Cayley transform
\[ a = \iu \,(\mathds{1}+u) \,{(\mathds{1}-u)}^{\,-1}. \]
Please note here that $\mathds{1}-u$ is injective, by \ref{onenoteival}.
\end{theorem}

\begin{proof}
See $(iv)$ and $(v)$ of the proof of lemma \ref{firstlemmaspecthm}.
\end{proof}

\begin{definition}[function of a self-adjoint operator]%
\index{concepts}{function!of an operator}\label{unbdedf}%
Let $P$ be the spectral resolution of the self-adjoint linear
operator $a$. If $f$ is a complex-valued Borel function on $\s(a)$,
one writes
\[ f(a) := \Psi_P (f) = \int f \,\diff P \quad \text{(pointwise)}, \]
and one says that $f(a)$ is a \underline{function} of $a$.
\end{definition}

\begin{example}
The Cayley transform
\[ u := (a - \iu \mathds{1}) (a + \iu \mathds{1})^{-1} \]
is the function $\kappa (a)$ with
\begin{alignat*}{2}
\kappa : \s(a)       & &\ \to            &\ \mathds{C} \\
               \lambda & &\ \mapsto  &\ (\lambda - \iu)(\lambda + \iu)^{-1}.
\end{alignat*}
\end{example}

\begin{proof}
See the proof of the converse part of lemma \ref{firstlemmaspecthm}.
\end{proof}

\clearpage


\section{Application: an Initial Value Problem}%
\label{scopug}

\medskip
In this paragraph, let $H$ be a Hilbert space $\neq \{0\}$.

\begin{definition}%
By a \underline{strongly continuous one-parameter unitary} \underline{group}
on $H$, we shall understand a function
\begin{alignat*}{2}
\mathds{R} & \to           &\ & \blop(H) \\
                   t & \mapsto &\ & U_t
\end{alignat*}
such that
\begin{itemize}
  \item[$(i)$] $U_t$ is unitary for all $t \in \mathds{R}$,
 \item[$(ii)$] $U_{t+s} = U_t \,U_s$ for all $s$, $t \in \mathds{R}$,
                      \quad (the group property)
\item[$(iii)$] for each $x \in H$, the function $t \mapsto U_t \,x$
                       is continuous on $\mathds{R}$.
\end{itemize}
\end{definition}

\begin{definition}[infinitesimal generator]%
\index{concepts}{infinitesimal generator}%
Let $t \mapsto U_t$ be a strongly continuous one-parameter unitary
group on $H$. Then the \underline{infinitesimal} \underline{generator}
of $t \mapsto U_t$ is the linear operator $a$ in $H$, with domain
$\domain(a)$, characterised by the condition that for $x, y \in H$, one has
\[ x \in \domain(a),\ y = a x \quad \Leftrightarrow \quad
\ - \frac{1}{\iu} \frac{\diff}{\diff t} \Bigl|_{\text{\small{$t=0$}}} \,U_t \,x
\text{\ exists and equals\ } y. \]
This is also written as
\[ a = - \frac{1}{\iu} \frac{\diff}{\diff t} \Bigl|_{\text{\small{$t=0$}}} \,U_t
\quad \text{(pointwise)}. \]
\end{definition}

\begin{theorem}[solution to an initial value problem]%
Let $a$ be a self-adjoint linear operator in $H$, and let $P$ be its
spectral resolution. For $t \in \mathds{R}$, we denote by $U_t$ the
bounded linear operator on $H$ given by
\[ U_t :=
\exp(- \iu t a) = \int _{\text{\small{$\s(a)$}}} \exp(- \iu t \lambda) \,\diff P (\lambda)
\quad \text{(pointwise)}, \]
cf.\ \ref{unbdedf}. (Please note that $t \mapsto U_t$ is the Fourier transform of $P$.)

Then $t \mapsto U_t$ is a strongly continuous one-parameter
unitary group on $H$ with infinitesimal generator $a$.

For $x \in \domain(a)$, one then actually has
\[ - \frac{1}{\iu} \frac{\diff}{\diff t} \Bigl|_{\text{\small{$t=s$}}} \,U_t \,x
= a \,U_s \,x \quad \text{for all} \quad s \in \mathds{R}. \]

That is, $u : t \mapsto U_t \,x$ is a
\underline{solution to the initial value problem}
\[ - \frac{1}{\iu} \frac{\diff}{\diff t} \,u(t) = a \,u(t),\quad u(0) = x. \pagebreak \]
\end{theorem}

\begin{proof}
The spectral resolution $P$ of $a$ lives on the spectrum of $a$,
which is a subset of the real line, cf.\ \ref{specisreal}. The
fact that $\pi_P$ is a representation on $H$ therefore implies that
$U_t$ is unitary for all $t \in \mathds{R}$, and that $U_{t+s} = U_t \,U_s$
for all $s$, $t \in \mathds{R}$. Moreover, the function $t \mapsto U_t \,x$
\linebreak is continuous on $\mathds{R}$ for every $x \in H$ by the Dominated
Convergence Theorem \ref{domconvergence}. Thus $t \mapsto U_t$
$(t \in \mathds{R})$ is a strongly continuous one-parameter unitary group.
Please note that $U_0 = \mathds{1}$, by $U_0 = U_0 U_0$.

We shall show that $a$ is the infinitesimal generator of $t \mapsto U_t$.
For $x \in \domain(a)$ and $t \in \mathds{R}$, we have
\[ {\Bigl\| \ {- \frac{1}{\iu t} \,(U_t \,x - x) - ax} \ \Bigr\| \,}^2
= \int {| \,f_t \,| \,}^2 \,\diff \langle Px, x \rangle, \]
with
\[ f_t (\lambda) = - \frac{1}{\iu t} \,\bigl(\exp(- \iu t \lambda) - 1 + \iu t \lambda \bigr)
\qquad \bigl( \,\lambda \in \s(a) \,\bigr). \]
By power series expansion of $\exp(- \iu t \lambda)$, one sees that $f_t \to 0$
pointwise as $t \to 0$. We have $| \,\exp( \iu s) -1 \,| \leq | \,s \,|$ for all real
$s$ (as chord length is smaller than or equal to arc length). It follows that
\[ | \,f_t \,| \leq 2 \cdot | \,\mathrm{id}_{\text{\small{$\s(a)$}}} \,|
\quad \text{for all} \quad t \in \mathds{R}. \]
Since the identical function on $\s(a)$ is in
${\mathscr{L} \,}^2( \langle Px, x \rangle )$,
the Lebesgue Dominated Convergence Theorem implies that
\[ \lim _{t \to 0} - \frac{1}{\iu t} \,(U_t \,x - x) = a x
\quad \text{for all} \quad x \in \domain(a). \]
Thus, if $b$ denotes the infinitesimal generator of $t \mapsto U_t$, we
have shown that $b$ is an extension of $a$. Now $b$ is symmetric
because for $x$, $y \in \domain(b)$ and $t \in \mathds{R}$, one has
\[ \langle \,- \frac{1}{\iu t} \,(U_t \,x - x), y \,\rangle
= \langle \,- \frac{1}{\iu t} \,(U_t - \mathds{1}) \,x, y \,\rangle
= \langle \,x, \frac{1}{\iu t} \,(U_{-t} \,y - y) \,\rangle. \]
It follows that $b = a$ because $a$ is maximal symmetric, cf.\ \ref{ismaxsymm}.

The second statement follows from the group property of $t \mapsto U_t$
and from the fact that $\domain(a)$ is invariant under $U_t$, by \ref{commute}.
(Indeed, $U_t$ commutes with $a$ by \ref{specthmunbded} and \ref{spprojcomm}.)
\end{proof}

\medskip
This is of prime importance for quantum mechanics, where the initial
value problem is the Schr\"odinger equation for the evolution in time
of the quantum mechanical state described by the vector $u(t)$. \pagebreak

\clearpage

\addtocontents{toc}{\protect\pagebreak}


\phantomsection

\addcontentsline{toc}{part}{Appendix}

\begin{center} \textbf{\huge{Appendix}} \end{center}


\setcounter{section}{50}

\phantomsection

\addcontentsline{toc}{chapter}{\texorpdfstring{\ \ \ \S\ 50.\quad Quotient Spaces}%
{\textsection\ 50. Quotient Spaces}}

\fancyhead[LE]{\SMALL{APPENDIX}}
\fancyhead[RO]{\SMALL{\S\ 50. \ QUOTIENT SPACES}}

\smallskip
\begin{center} \textbf{\S\ 50.\:\:Quotient Spaces} \end{center}

\setcounter{theorem}{0}

\begin{reminder}[distance from a point to a non-empty set]
\index{concepts}{distance function}\label{distance}%
Let $\bigl( \,X, d \,\bigr)$ be a metric space. Let $A$ be a non-empty subset
of $X$. The \underline{distance} from a point $x \in X$ to $A$ is defined as
\[ \mathrm{dist} (x, A) := \inf \,\{ \,d(x, a) \geq 0 : a \in A \,\}. \]
(See e.g.\ \cite[p.\ 253]{Eng}.)

\noindent The distance function $x \mapsto \mathrm{dist} (x, A)$ $(x \in X)$ is
continuous, and
\[ \{ \,x \in X : \mathrm{dist} (x, A) = 0 \,\} = \overline{A}. \]
(See \cite[4.1.10 \& 4.1.11 p.\ 254]{Eng}.)

\noindent In particular, if $A$ is closed, then for $x \in X$, we have
$\mathrm{dist} (x, A) = 0$ if and only if $x \in A$.
\end{reminder}

\begin{reminder}[quotient space]\label{quotspace}%
\index{concepts}{quotient!space}\index{concepts}{quotient!norm}%
\index{concepts}{norm!quotient}%
Let $\bigl( \,A, | \cdot | \,\bigr)$ be a normed space and
let $B$ be a \underline{closed} subspace of $A$. The
\underline{quotient space} $A/B$ is defined as
\[ A/B := \{ a+B \subset A : a \in A \}. \]
The canonical projection (here denoted by an underscore)
\begin{alignat*}{2}
\_ \,: A\ & & \to     &\ A/B \\
       a\ & & \mapsto &\ a+B =: \underline{a}
\end{alignat*}
is a linear map from $A$ onto $A/B$. Putting
\[ | \,\underline{a}\, | := \inf \,\{ \,|\,c\,| : c \in \underline{a} \,\}
\qquad ( \,\underline{a} \in A/B \,) \]
then defines a norm on A/B. It is called the
\underline{quotient norm} on $A/B$.
\end{reminder}

\begin{proof}
For $\underline{a}, \underline{b} \in A/B$ and $c, d \in B$ arbitrary, one has
\[ | \,\underline{a} + \underline{b} \,| \leq | \,a + b + c + d \,| \leq | \,a + c \,| + | \,b + d \,|. \]
Taking infima over $c \in B$ and $d \in B$, we find that
\[ | \,\underline{a} + \underline{b} \,| \leq | \,\underline{a} \,| + | \,\underline{b} \,|, \]
showing that $a \mapsto | \,\underline{a} \,|$ indeed is a seminorm.
To see that this seminorm actually is a norm, we first note that $| \,\underline{a} \,|$
is the distance from $0$ to $\underline{a}$ and that $\underline{a} = a + B$
is closed as $B$ is closed by assumption. Therefore $| \,\underline{a} \,| = 0$
if and only if $0 \in \underline{a}$, that is, if and only if $\underline{a} = 0$.
\pagebreak
\end{proof}

\begin{theorem}\label{quotspaces}%
Let $A$ be a normed space and let $B$
be a \underline{closed} subspace of $A$. If $A$ is a Banach space,
then $A/B$ is a Banach space. If both $A/B$ and $B$ are Banach
spaces, then $A$ is a Banach space.
\end{theorem}

\begin{proof}
Assume first that $A$ is complete. Let
$(\underline{c_k})_{k \geq 0}$ be a Cauchy sequence in $A/B$. We can then
find a subsequence $(\underline{d_n})_{n \geq 0}$ of
$(\underline{c_k})_{k \geq 0}$ such that
\[ | \,\underline{d_{n+1}} - \underline{d_n} \,| < 2^{-(n+1)} \]
for all $n \geq 0$. We construct a sequence $(a_n)_{n \geq 0}$ in $A$ such
that
\[ \underline{a_n} = \underline{d_n} \]
and
\[ | \,a_{n+1} -a_n \,| < 2^{-(n+1)} \]
for all $n \geq 0$. This is done as follows. Let $a\0 \in \underline{d\0}$.
Suppose that $a_n$ has been constructed and let
$d_{n+1} \in \underline{d_{n+1}}$. We obtain
\[ \inf \,\{ \,| \,d_{n+1} -a_n +b \,| : b \in B \,\}
= | \,\underline{d_{n+1}} - \underline{a_n} \,|
= | \,\underline{d_{n+1}} - \underline{d_n} \,| < 2^{-(n+1)} \]
so that there exists $b_{n+1} \in B$ such that
\[ | \,d_{n+1} -a_n +b_{n+1} \,| < 2^{-(n+1)}.  \]
We then put $a_{n+1} := d_{n+1} + b_{n+1} \in \underline{d_{n+1}}$.
This achieves the construction of $(a_n)_{n \geq 0}$ as required.
But then $(a_n)_{n \geq 0}$ is a Cauchy sequence in $A$ and converges
in $A$ by assumption. By continuity of the projection,
$(\underline{d_n})_{n \geq 0}$ converges in $A/B$, so that also
$(\underline{c_k})_{k \geq 0}$ converges in $A/B$.

Assume now that both $A/B$ and $B$ are complete. Let $(a_n)_{n \geq 0}$
be a Cauchy sequence in $A$. Then $(\underline{a_n})_{n \geq 0}$
is a Cauchy sequence in $A/B$ because the projection is contractive. By
assumption, there exists \linebreak
$a \in A$ such that $(\underline{a_n})_{n \geq 0}$
converges to $\underline{a}$. For $n \geq 0$, we can find $b_n \in B$ with
\[ |\,(a-a_n)+b_n\,|
< | \,\underline{a}-\underline{a_n} \,| + \frac{1}{n}. \]
But then
\begin{align*}
| \,b_m-b_n \,| & = \bigl|\,[ \,(a-a_m) + b_m \,]
                  - [ \,(a-a_n) + b_n \,] + (a_m-a_n) \,\bigr| \\
              & \leq |\,\underline{a}-\underline{a_m}\,| + \frac{1}{m} +
                     |\,\underline{a}-\underline{a_n}\,| + \frac{1}{n} +
                     |\,a_m-a_n\,|,
\end{align*}
from which we see that $(b_n)_{n \geq 0}$ is a Cauchy sequence in
$B$ and consequently convergent to some element $b$ in $B$. Finally we have
\[ |\,(a+b)-a_n\,| \leq |\,(a-a_n)+b_n\,| + |\,b-b_n\,| \]
which implies that $(a_n)_{n \geq 0}$ converges to $a + b$. \pagebreak
\end{proof}

\clearpage


\setcounter{section}{51}

\phantomsection

\addtocontents{toc}{\protect\vspace{-0.65em}}

\addcontentsline{toc}{chapter}{\texorpdfstring{\ \ \ \S\ 51.\quad Reminder on Topology}%
{\textsection\ 51. Reminder on Topology}}

\fancyhead[LE]{\SMALL{APPENDIX}}
\fancyhead[RO]{\SMALL{\S\ 51. \ REMINDER ON TOPOLOGY}}

\bigskip
\begin{center} \textbf{\S\ 51.\:\:Reminder on Topology} \end{center}
\medskip

\setcounter{theorem}{0}

We gather here the topological ingredients of the Commutative
Gel'fand-Na\u{\i}mark Theorem.
\hypertarget{remtop}{}%

\begin{definition}[the weak topology]\label{weaktopdef}%
\index{concepts}{topology!weak}\index{concepts}{topology!initial}%
Let $X$ be a set. Let $\{ Y_i \}_{i \in I}$ be a family of topological spaces,
and let $\{ f_i \}_{i \in I}$ be a family of functions $f_i : X \to Y_i$ $(i \in I)$.
Among the topologies on $X$ for which all the functions $f_i$ $(i \in I)$
are continuous, there exists a coarsest topology. It is the topology
generated by the base consisting of all the sets of the form
$\bigcap_{t \in T} {f_t}^{-1} (O_t)$, where $T$ is a finite subset of $I$ and
$O_t$ is an open subset of $Y_t$ for all $t \in T$. This topology is called
the \underline{weak topology induced by $\{ f_i \}_{i \in I}$}. It is also called
the initial topology induced by $\{ f_i \}_{i \in I}$.
\end{definition}

\begin{proof}
The proof is easy. Otherwise, see e.g.\ \cite[Prop.\ 1.4.8 p.\ 31]{Eng}.
\end{proof}

\medskip
The philosophy is, that the weaker a topology is, the more quasi-compact
subsets it has.

\begin{theorem}[convergence]\label{netconv}%
Let $X$ be a topological space carrying the weak topology induced by a
family $\{ f_i \}_{i \in I}$ of functions $f_i : X \to Y_i$ $(i \in I)$. A net
$(x_\lambda)_{\lambda \in \Lambda}$ in $X$ then converges to a point
$x \in X$ if and only if for all $i \in I$ the net
$\bigl( f_i(x_\lambda) \bigr)_{\lambda \in \Lambda}$
converges to $f_i(x)$ in $Y_i$.
\end{theorem}

\begin{proof}
If $x_\lambda \to x$, then $f_i(x_\lambda) \to f_i(x)$ by continuity of $f_i$ $(i \in I)$.
Assume now that $f_i(x_\lambda) \to f_i(x)$ $(i \in I)$. Every neighbourhood of $x$ contains
a base set of the form $\cap_{t \in T} {f_t}^{-1} (O_t)$, containing $x$, where $T$ is a
finite subset of $I$, and $O_t$ is an open subset of $Y_t$ $(t \in T)$. For $t \in T$, there
exists $\mu_t \in \Lambda$ such that $f(x_\lambda) \in O_t$ for all $\lambda \in \Lambda$
with $\lambda \geq \mu_t$. Since $\Lambda$ is directed up, there exists
$\mu \in \Lambda$ with $\mu \geq \mu_t$ for all $t$ in the finite set $T$. Then
$x_\lambda \in \cap_{t \in T} {f_t}^{-1} (O_t)$ for all $\lambda \in \Lambda$ with
$\lambda \geq \mu$.
\end{proof}

\begin{corollary}[the universal property]%
\index{concepts}{universal!property!of weak topology}%
Let $X$ be a topological space carrying the weak topology induced by a
family $\{ f_i \}_{i \in I}$ of functions. The following property is called the
\underline{universal property} of the weak topology induced by $\{ f_i \}_{i \in I}$.
A function $g$ from a topological space to $X$ is continuous if and only
if the compositions $f_i \circ g$ are continuous for all $i \in I$.
\end{corollary}

\begin{proof}
This follows from the preceding theorem \ref{netconv}. \pagebreak
\end{proof}

\begin{definition}[the weak* topology]%
\label{weak*top}\index{concepts}{topology!weak*}%
\index{concepts}{universal!property!of weak* topology}%
\index{concepts}{evaluation}%
Let $A$ be a real or complex vector space, and let $A^*$ be its dual space.
The \underline{evaluations} $\wht{a}$ at elements $a \in A$ are defined as
\begin{alignat*}{2}
\wht{a} : A^* &           \to & & \ \mathds{R} \text{ or } \mathds{C} \\
              \tau \,& \mapsto & & \ \wht{a}(\tau) := \tau (a).
\end{alignat*}
The \underline{weak* topology} on $A^*$ is the weak topology induced by
the family $\{ \,\wht{a} \,\}_{a \in A}$ of evaluations at elements of $A$.
In particular, the evaluations $\wht{a}$ $(a \in A)$ are continuous in the weak*
topology, and the weak* topology is the coarsest topology on $A^*$ such that
the evaluations $\wht{a}$ $(a \in A)$ are continuous.
The \underline{universal property} of the weak* topology says that a function
$g$ from a topological space to $A^*$ is continuous with respect to the weak*
topology on $A^*$, if and only if the compositions $\wht{a} \circ g$ $(a \in A)$
are continuous.
\end{definition}

\begin{proposition}[pointwise convergence]\label{weak*point}%
Let $A$ be a real or complex vector space, and let $A^*$ be its dual space.
The weak* topology on $A^*$ is just the topology of pointwise convergence
on $A$.
\end{proposition}

\begin{proof}
According to \ref{netconv}, a net $(\tau_\lambda)_{\lambda \in \Lambda}$
in $A^*$ converges to a functional $\tau \in A^*$ in the weak* topology if
and only if $\wht{a}(\tau_\lambda) \to \wht{a}(\tau)$ for all $a \in A$, which
is the same as $\tau_\lambda(a) \to \tau(a)$ for all $a \in A$.
\end{proof}

\begin{proposition}%
Let $X$ be a topological space carrying the weak topology induced
by a family $\{ f_i \}_{i \in I}$ of functions $f_i: X \to Y_i$ $(i \in I)$.
If the set $\{ \,f_i : i \in I \,\}$ separates the points of $X$ \ref{sepp}, and
if all the spaces $Y_i$ $(i \in I)$ are Hausdorff spaces, then $X$ is
a Hausdorff space. 
\end{proposition}

\begin{proof}
The proof is very easy and left to the reader. Otherwise, see
\cite[Prop.\ 1.5.3 pp.\ 23 f.]{PedT}.
\end{proof}

\begin{observation}%
Let $A$ be a real or complex vector space, and let $A^*$ be its dual space.
The functions $\wht{a}$ $(a \in A)$ separate the points of $A^*$.
\end{observation}

\begin{proof}
Let $\sigma, \tau \in A^*$ with $\sigma \neq \tau$. There then exists
$a \in A$ with $\sigma(a) \neq \tau(a)$, or, in other words,
$\wht{a}(\sigma) \neq \wht{a}(\tau)$.
\end{proof}

\begin{corollary}%
The weak* topology is a Hausdorff topology. \pagebreak
\end{corollary}

The importance of the weak* topology stems for a large part from the
following compactness result.

\begin{theorem}[Alaoglu's Theorem]\index{concepts}{Alaoglu's Theorem}%
\label{Alaoglu}\index{concepts}{Theorem!Alaoglu}%
Let $A$ be a (possibly real) \linebreak normed space, and let $A'$ be
its dual normed space. The closed unit ball in $A'$ then is a compact
Hausdorff space in (the relative topology induced by) the weak* topology.
\end{theorem}

\begin{proof}
Let $(\tau_\lambda)_{\lambda \in \Lambda}$ be a universal net in the
closed unit ball of $A'$. We have to prove that this net converges in the
closed unit ball of $A'$. (See e.g.\ \cite[Theorem 17.4 p.\ 118]{Willa}.)
For every $a \in A$, we have that the image net
$\bigl( \tau_\lambda(a) \bigr)_{\lambda \in \Lambda}$ is a universal
net \cite[Theorem 11.11 p.\ 76]{Willa}, contained in a compact subset
of the scalar field, as $| \,\tau_\lambda (a) \,| \leq | \,a \,|$
for all $\lambda \in \Lambda$. Hence this image net
$(\tau_\lambda(a))_{\lambda \in \Lambda}$ converges to
some scalar $\tau (a)$ with $| \,\tau (a) \,| \leq | \,a \,|$.
The resulting function $\tau : a \mapsto \tau (a)$ $(a \in A)$ a linear
functional on $A$ because the weak* topology is just the topology
of pointwise convergence, cf.\ \ref{weak*point}. For the same reason,
the net $(\tau_\lambda)_{\lambda \in \Lambda}$ converges to $\tau$
in the weak* topology. The functional $\tau$ is seen to be in the
closed unit ball of $A'$, which finishes the proof.
\end{proof}

\begin{proposition}%
Let $f$ be a continuous function from a quasi-compact
space $\Omega$ to a Hausdorff space. Then $f$ is a closed function
(in the sense that $f$ maps closed sets to closed sets). If $f$ is injective,
then $f$ is an imbedding. If $f$ is bijective, then $f$ is a homeomorphism.
\end{proposition}

\begin{proof}
A closed subset $C$ of the quasi-compact space $\Omega$ is
quasi-compact, and so is its continuous image $f(C)$. The set
$f(C)$ then is closed, as a quasi-compact subset of a Hausdorff
space is closed. The statement follows now from the fact that an
injective function, when considered as a function onto its range,
is closed if and only if it is open.
\end{proof}

\medskip
For emphasis, we restate:

\begin{theorem}\label{homeomorph}%
A continuous bijection from a quasi-compact space to a Hausdorff space
is a homeomorphism.
\end{theorem}

\begin{corollary}%
The topology of a compact Hausdorff space cannot be strengthened
without losing the quasi-compactness property, and it cannot be
weakened without losing the Hausdorff property. \pagebreak
\end{corollary}

\clearpage


\setcounter{section}{52}

\phantomsection

\addtocontents{toc}{\protect\vspace{-0.65em}}

\addcontentsline{toc}{chapter}%
{\texorpdfstring{\ \ \ \S\ 52.\quad Complements to Integration Theory}%
{\textsection\ 52. Complements to Integration Theory}}

\fancyhead[LE]{\SMALL{APPENDIX}}
\fancyhead[RO]{\SMALL{S\ 52. \ COMPLEMENTS TO INTEGRATION THEORY}}

\medskip
\begin{center} \textbf{\S\ 52.\:\:Complements to Integration Theory} \end{center}
\setcounter{theorem}{0}

We gather here some less standard facts of integration theory,
especially if they are not contained in the book \cite{FW} by
Filter \& Weber, which we use as a reference. A more classical
type of reference is Bauer \cite{Bau}, but the terminologies differ
slightly. However the material in this text goes through for both
sets of terminology, or roughly so at least. Some terminology
will thus be left undefined. There is general agreement on what
is a probability measure defined on a $\sigma$-algebra on a
non-empty set, which is what we will be using most of the time:
\hypertarget{appint}{}%

\begin{definition}[bounded complex measure]\label{bdedmeas}%
\index{concepts}{measure!bounded complex}%
\index{concepts}{bounded complex measure}%
A \underline{bounded complex} \underline{measure} on a set
$\Omega \neq \varnothing$ shall be defined as a complex linear
com\-bination of probability measures defined on a common
$\sigma$-algebra on $\Omega$. We stress that we shall deal
exclusively with bounded complex measures on non-empty sets.
(Except for implicit use of counting measure.)
\end{definition}

\begin{definition}[the Borel $\sigma$-algebra]\label{Boreldef}%
\index{concepts}{Borel!0sigma@$\sigma$-algebra}%
Let $\Omega \neq \varnothing$ be a Hausdorff space.
The \underline{Borel $\sigma$-algebra} of $\Omega$ is defined as the
$\sigma$-algebra generated by the open subsets of $\Omega$ (or
equivalently by the closed subsets of $\Omega$). The members of the
Borel $\sigma$-algebra of $\Omega$ are called \underline{Borel sets}
of $\Omega$.
\end{definition}

\begin{definition}[inner regular Borel probability measure]%
\index{concepts}{measure!inner regular}\index{concepts}{inner regular}%
\index{concepts}{measure!Borel}\label{inregBormeas}%
Let $\Omega$ be a Hausdorff space $\neq \varnothing$.
A \underline{Borel probability measure} on $\Omega$ is a
probability measure defined on the Borel $\sigma$-algebra of $\Omega$.
Such a Borel probability measure $\mu$ on $\Omega$ is called
\underline{inner regular}, if for every Borel set $\Delta \subset \Omega$, one has
\[ \mu(\Delta) = \sup \,\{ \,\mu(K) \geq 0 : K \text{ is a compact subset of } \Delta \,\}. \]
\end{definition}

\begin{proposition}[outer regularity]\label{outerreg}%
Let $\Omega \neq \varnothing$ be a Hausdorff space.
Every inner regular Borel probability measure $\mu$ on $\Omega$ also
is \underline{outer regular}, in the sense that for every Borel set
$\Delta \subset \Omega$, one has
\[ \mu(\Delta) = \inf \,\{ \,\mu(O) \geq 0 :
O \text{ is an open subset of } \Omega \text{ containing } \Delta \,\}. \]
\end{proposition}

\begin{proof}
Apply inner regularity to the complement of $\Delta$.
\end{proof}

\medskip
We won't use the preceding fact. The reader is warned that in \ref{outerreg},
the assumption that $\mu$ be a probability measure cannot be dropped
without replacement. Finally, we mention that that we call ``inner regular'',
is called ``regular'' in \cite{FW}. \pagebreak

\begin{definition}[measurability]\label{measurability}%
\index{concepts}{measurable}%
Let $\mathcal{E}$, $\mathcal{E}'$ be two $\sigma$-algebras on two
non-empty sets $\Omega,$ $\Omega'$ respectively. A function
$f : \Omega \to \Omega'$ is called \linebreak
\underline{$\mathcal{E}$-\,$\mathcal{E}'$ measurable} if
$f^{-1}(\Delta') \in \mathcal{E}$ for all $\Delta' \in \mathcal{E}'$.
A complex-valued function on $\Omega$ is called
\underline{$\mathcal{E}$-measurable}, if it is $\mathcal{E}$-$\mathcal{E}'$
measurable with $\mathcal{E}'$ the Borel $\sigma$-algebra of $\mathds{C}$.
A complex-valued function on $\Omega$ is $\mathcal{E}$-measurable
if and only if its real and imaginary parts are,
cf.\ \cite[Remark 2 p.\ 134 \& Equation (22.2) p.\ 133]{Bau}.
A real-valued function $f$ on $\Omega$ is \linebreak
$\mathcal{E}$-measurable if and only if $\{ \,f < \alpha \,\} \in \mathcal{E}$
for all $\alpha \in \mathds{R}$, cf.\ \cite[Theorem 9.2 p.\ 50 f.]{Bau}.
\end{definition}

\begin{definition}[Borel function]%
\label{continuous}\index{concepts}{Borel!function}%
Let $\Omega \neq \varnothing$ be a Hausdorff space.
A complex-valued function on $\Omega$, which is $\mathcal{E}$-measurable
with $\mathcal{E}$ the Borel $\sigma$-algebra on $\Omega$, is called a
\underline{Borel function} on $\Omega$. For example, every continuous
complex-valued function on $\Omega$ is a Borel function on $\Omega$.
\end{definition}

We shall use the Riesz Representation Theorem in the following form.
The proof follows easily from other versions of the theorem.

\begin{theorem}[the Riesz Representation Theorem]%
\label{Riesz}\index{concepts}{Theorem!Riesz Representation}%
\index{concepts}{Riesz Representation Thm.}%
Let $\Omega \neq \varnothing$ be a locally compact Hausdorff space.
Let $\ell$ be a linear functional on $\cont_\mathrm{c} (\Omega)$,
which is positive in the sense that $\ell(f) \geq 0$ for every
$f \in \cont_\mathrm{c} (\Omega)$ with $f \geq 0$ pointwise on $\Omega$.
Assume that $\ell$ is bounded with norm $1$ on the pre-C*-algebra
$\cont_\mathrm{c}( \Omega)$. There then exists a unique inner
regular Borel probability measure $\mu$ on $\Omega$ such that
\[ \ell (f) = \int f \,\diff \mu\quad \text{for every}
\quad f \in \cont_\mathrm{c} (\Omega). \]
The measure of a compact subset $K$ of $\Omega$ is given by
\[ \mu (K) = \inf \,\{ \,\ell (f) \geq 0 : f \in \cont_\mathrm{c} (\Omega),
1_K \leq f \leq 1_\Omega \,\}. \]
\end{theorem}

For a proof, see for example \cite{FW}: Theorem 5.20 p.\ 178 \& Theorem 5.21 p.\ 181
\& Corollary 5.12 (a) p.\ 170 \& Theorem 5.7 (b) p.\ 166.
Or else, see \cite{Bau}: Theorem 29.3 p.\ 180 \& Formula (28.13) p.\ 176.

\begin{definition}[image measure]\label{imagebegin}%
\index{concepts}{image!measures}%
\index{concepts}{measure!image}%
Let $(\Omega, \mathcal{E}, \mu)$ be a probability space. Let
$\mathcal{E}'$ be a $\sigma$-algebra on a set $\Omega' \neq \varnothing$.
Let $f : \Omega \to \Omega'$ be an $\mathcal{E}\text{-}\mathcal{E}'$
measurable function. The \underline{image measure} $f(\mu)$
is defined as
\begin{alignat*}{2}
 f(\mu) : \mathcal{E}' & \to           &\ & [ \,0, 1 \,] \\
                     \Delta' & \mapsto &\ & \mu\bigl( f^{-1}(\Delta') \bigr).
\end{alignat*}
That is, $f(\mu) = \mu \circ f^{-1}$. \pagebreak
\end{definition}

If $f$ has the meaning of a random variable on $(\Omega, \mathcal{E}, \mu)$,
then $f(\mu)$ is the distribution of the random variable $f$.

\begin{theorem}\label{imageend}%
Let $(\Omega, \mathcal{E}, \mu)$ be a probability space. Let $\mathcal{E}'$
be a \linebreak $\sigma$-algebra on a set $\Omega' \neq \varnothing$.
Let $f : \Omega \to \Omega'$ be an $\mathcal{E}$-\,$\mathcal{E}'$
measurable function. For an $\mathcal{E}'$-measurable
complex-valued function $g$ on $\Omega'$, we have
\[
g \in {\mathscr{L} \,}^{1} \bigl(f(\mu)\bigr) \ \Leftrightarrow
\ g \circ f \in {\mathscr{L} \,}^{1}(\mu) \ \Rightarrow
\ \int g\ \diff\mspace{2mu}f (\mu) = \int (g \circ f) \,\diff \mu.
\]
\end{theorem}

\begin{proof}
Prove this first for $\mathcal{E}'$-step functions.
Approximate a non-negative $\mathcal{E}'$-measurable function $g$
by an increasing sequence of $\mathcal{E}'$-step functions and apply
Lebesgue's Monotone Convergence Theorem.
\end{proof}

\begin{definition}[the Baire $\sigma$-algebra]%
\index{concepts}{Baire!0sigma@$\sigma$-algebra}%
\index{concepts}{Baire!function}\label{Bairedef}%
Let $\Omega \neq \varnothing$ be a locally compact Hausdorff
space. The \underline{Baire $\sigma$-algebra} of $\Omega$
is defined as the \linebreak $\sigma$-algebra generated
by the sets $\{ \,f < \alpha \,\}$ with $f$ a continuous
real-valued function on $\Omega$ vanishing at infinity,
and $\alpha \in \mathds{R}$. This is the smallest
$\sigma$-algebra $\mathcal{E}$ on $\Omega$ such that
the functions in $\cont\0(\Omega)$ are $\mathcal{E}$-measurable.
The members of the Baire $\sigma$-algebra of $\Omega$
are called \underline{Baire sets} of $\Omega$. The complex-valued
functions on $\Omega$ which are $\mathcal{E}$-measurable
with $\mathcal{E}$ the Baire $\sigma$-algebra of $\Omega$,
are called \underline{Baire functions} on $\Omega$.
\end{definition}

What we here call the Baire $\sigma$-algebra,
is in fact called the \underline{restricted}
Baire $\sigma$-algebra by some authors.

\begin{proposition}\label{metric}%
Let $\Omega \neq \varnothing$ be a locally compact Hausdorff
space. The Baire $\sigma$-algebra of $\Omega$ is contained
in the Borel $\sigma$-algebra of $\Omega$. If $\Omega$ is a
compact metric space, then the Baire $\sigma$-algebra of
$\Omega$ coincides with the Borel $\sigma$-algebra of $\Omega$.
\end{proposition}

\begin{proof}
The first statement holds because the Baire $\sigma$-algebra
is the smallest $\sigma$-algebra with respect to which the
functions in $\cont\0(\Omega)$ are measurable, and because
every function in $\cont(\Omega)$ is a Borel function,
cf. \ref{continuous}. Let now $\Omega$ be a compact metric
space. For the converse inclusion, it suffices to show that
every closed subset of $\Omega$ is a Baire set. Any
closed subset of $\Omega$ is of the form $\{ \,f = 0 \,\}$ for
some $f \in \cont(\Omega)$ by considering the distance
function from this set, cf.\ \ref{distance}. Since $\Omega$
is compact, we have $\cont(\Omega) = \cont\0(\Omega)$,
so the above function $f$ is a Baire function, which makes
that $\{ \,f = 0 \,\}$ is a Baire set. \pagebreak
\end{proof}

\medskip
The next result is well-known.

\begin{theorem}\label{CcdenseinLp}%
Let $\mu$ be an inner regular Borel probability measure on a locally
compact Hausdorff space $\Omega \neq \varnothing$. Then for every
function $f \in {\mathscr{L} \,}^{p} \,(\mu)$ with $1 \leq p < \infty$,
there exists a sequence $\{\,f_n\,\}$ in $\cont_\mathrm{c}(\Omega)$
which converges to $f$ in ${\mathscr{L} \,}^{p} \,(\mu)$
and $\mu$-almost everywhere.
\end{theorem}

For a proof, see e.g.\ \cite[Theorem 6.15 p.\ 204]{FW}.

\begin{corollary}\label{Baire}%
Let $\mu$ be an inner regular Borel probability \linebreak measure
on a locally compact Hausdorff space $\Omega \neq \varnothing$.
Then a function in ${\mathscr{L} \,}^{\infty}(\mu)$ is $\mu$-a.e.\ equal
to a bounded Baire function on $\Omega$.
\end{corollary}

\begin{proof}
Let $f \in {\mathscr{L} \,}^{\infty}(\mu)$. As $\mu$ is a probability measure,
we have ${\mathscr{L} \,}^{\infty}(\mu) \subset {\mathscr{L} \,}^{1}(\mu)$, so
we apply the preceding theorem \ref{CcdenseinLp} with $p = 1$.
Let $\{\,f_n\,\}$ be a sequence as described. The truncated function
\[ \limsup_{n \to \infty} \ \Bigl\{ \,\bigl[ \,\mathrm{Re}\,(f_n)
\wedge \| \,\mathrm{Re}\,(f) \,\|_{\,\text{\small{$\mu$}},\infty} \,1_{\,\Omega} \,\bigr]
\,\vee \,\bigl[ \,- \| \,\mathrm{Re}\,(f) \,\|_{\,\text{\small{$\mu$}},\infty} \,1_{\,\Omega} \,\bigr]
\,\Bigr\} \]
is a bounded Baire function which is $\mu$-a.e.\ equal to $\mathrm{Re}\,(f)$.
\end{proof}

\medskip
We also have the similar result \ref{general} below. For the proof,
we need the following theorem, which is well-known.

\begin{theorem}\label{pless}%
Let $(\Omega, \mathcal{E}, \mu)$ be a probability space.
Then a function in ${\mathscr{L} \,}^{p} \,(\mu)$ with $1 \leq p < \infty$
is $\mu$-a.e.\ equal to a complex-valued $\mathcal{E}$-measurable
function on $\Omega$.
\end{theorem}

For a proof, see e.g. \cite[Theorem 3.2.13 p.\ 480]{CFW}.

\begin{corollary}\label{general}%
Let $(\Omega, \mathcal{E}, \mu)$ be a probability space.
Then every function in ${\mathscr{L} \,}^{\infty} \,(\mu)$ is
$\mu$-a.e.\ equal to a bounded $\mathcal{E}$-measurable
complex-valued function on $\Omega$.
\end{corollary}

\begin{proof}
Let $f \in {\mathscr{L} \,}^{\infty}(\mu)$. As $\mu$ is a probability measure,
we have ${\mathscr{L} \,}^{\infty}(\mu) \subset {\mathscr{L} \,}^{1}(\mu)$, so
we apply the preceding theorem \ref{pless} with $p = 1$.  There exists a
real-valued $\mathcal{E}$-measurable function $g$ on $\Omega$, which is
$\mu$-a.e.\ equal to $\mathrm{Re} \,(f)$. The truncated function
\[ \bigl[ \,g \wedge \| \,\mathrm{Re}\,(f) \,\|_{\,\text{\small{$\mu$}},\infty} \,1_{\,\Omega} \,\bigr]
\,\vee \,\bigl[ \,- \| \,\mathrm{Re}\,(f) \,\|_{\,\text{\small{$\mu$}},\infty} \,1_{\,\Omega} \,\bigr] \]
is a bounded $\mathcal{E}$-measurable function,
which is $\mu$-a.e.\ equal to $\mathrm{Re} \,(f)$.\pagebreak
\end{proof}

\clearpage

%


\setcounter{section}{53}

\phantomsection

\addtocontents{toc}{\protect\vspace{-0.65em}}

\addcontentsline{toc}{chapter}%
{\texorpdfstring{\ \ \ \S\ 53.\quad Extension of Baire Functions}%
{\textsection\ 53. Extension of Baire Functions}}

\fancyhead[LE]{\SMALL{APPENDIX}}
\fancyhead[RO]{\SMALL{\S\ 53. \ EXTENSION OF BAIRE FUNCTIONS}}

\medskip
\begin{center} \textbf{\S\ 53.\:\:Extension of Baire Functions} \end{center}
\setcounter{theorem}{0}

Our sole aim in this paragraph is to provide the reader with theorem
\ref{bBaireext} below.
\hypertarget{extBaire}{}%

In this paragraph, let $\Omega \neq \varnothing$ be a locally compact
Hausdorff space, and let $C \neq \varnothing$ be a \underline{closed}
subset of $\Omega$. Please note that then $C$ is a locally compact
Hausdorff space $\neq \varnothing$ in its own right, and it thus makes
sense to consider $\cont\0(C)$ as well as Baire functions on $C$.

It is important to realise that we suppose $C$ to be closed in $\Omega$,
but not necessarily a Baire set of $\Omega$.

\begin{proposition}\label{coptinfty}%
Assume that $\Omega$ is not compact. Let $K$ be the one-point
compactification of $\Omega$, and let $\infty \in K \setminus \Omega$
be the corresponding point at infinity. Then either $C$ is compact, or
else the one-point compactification of $C$ is $C \cup \{ \infty \}$,
considered as a subspace of $K$.
\end{proposition}

\begin{proof}
Consider the neighbourhood filters
of the points at infinity, using that a subset of $C$
is compact if and only if it is of the intersection
of $C$ with a compact subset of $\Omega$, as $C$
is closed in $\Omega$. Cf.\ reminder \ref{onepcomp}.%
\end{proof}

\medskip
We remind of the following classical result, in view of the ensuing
proposition.

\begin{reminder}[the Tietze Extension Theorem]\label{Tietze}%
\index{concepts}{Theorem!Tietze Extension}%
\index{concepts}{Tietze Extension Theorem}%
If $K$ is a normal Hausdorff space, and if $C$ is a closed subset
of $K$, then every continuous real-valued function on $C$ can be
extended to a continuous real-valued function on $K$.
\end{reminder}

For a proof, see for example \cite[Theorem 2.1.8 p.\ 69]{Eng}.

\begin{proposition}\label{restrC0}%
The set $\cont\0(C)$ consists of the restrictions to $C$ of the functions
in $\cont\0(\Omega)$.
\end{proposition}

\begin{proof}
Assume first that $\Omega$ is compact. Then the closed subset $C$
of $\Omega$ is compact as well. Hence the restriction of a function of
$\cont\0(\Omega)$ to $C$ will be in $\cont(C) = \cont\0(C)$. Conversely,
every function in $\cont(C)$ can be extended to a function in
$\cont(\Omega) = \cont\0(\Omega)$, by the preceding Tietze Extension
Theorem \ref{Tietze}.\pagebreak

So we can assume that $\Omega$ is not compact. Let $K$ denote the
one-point compactification of $\Omega$, and let $\infty \in K \setminus \Omega$
denote the corresponding point at infinity.

Assume now that $C$ is compact. Then the restrictions of functions
in $\cont\0(\Omega)$ to $C$ will be in $\cont(C) = \cont\0(C)$ again.
Conversely, a function in $\cont(C)$ can be extended to a function in
$\cont(K)$, by the Tietze Extension Theorem \ref{Tietze}. The latter
function will automatically be bounded as $K$ is compact, and thus
need only be multiplied by a function in $\cont(K)$ which is equal to
$1$ on $C$ and vanishes at $\infty$, to yield an extension in
$\cont\0(\Omega)$ of the original function in $\cont(C)$. (Existence
of such a function to multiply with follows from Urysohn's Lemma
\ref{Urysohn}.)

Assume next that $C$ is not compact. Then the one-point compactification
of $C$ is $C \cup \{ \infty \}$ considered as a subspace of $K$, by
proposition \ref{coptinfty} above. Thus the restriction of a function in
$\cont\0(\Omega)$ to $C$ will be in $\cont\0(C)$. The converse part
follows by applying the Tietze Extension Theorem to the compact, 
hence closed, subset $C \cup \{ \infty \}$ of $K$.
\end{proof}

\begin{proposition}\label{subBaire}%
A subset $X$ of $C$ is a Baire set of $C$ if and only if it is of the form
$X = C \cap Y$ with $Y$ some Baire set of $\Omega$.
\end{proposition}

\begin{proof}
By definition, the Baire $\sigma$-algebra of $\Omega$ is generated by
\[ S := \Bigl\{ \,\{ \,f < \alpha \,\} \subset \Omega :
f \in \cont\0(\Omega)\sa, \ \alpha \in \mathds{R} \,\Bigr\}, \]
and the Baire $\sigma$-algebra of $C$ is generated by
\[ T := \Bigl\{ \,\{ \,g < \alpha \,\} \subset C :
g \in \cont\0(C)\sa, \ \alpha \in \mathds{R} \,\Bigr\}. \]
Now we have
\[ T = \{ \,X \subset C : X = C \cap Y \text{ for some } Y \in S \,\} \]
because $\cont\0(C)$ consists of the restrictions to $C$ of the functions
in $\cont\0(\Omega)$, cf.\ the preceding proposition \ref{restrC0}. The set
\[ \mathcal{E} :=
\{ \,X \subset C : X = C \cap Y \text{ for some Baire set } Y \text{ of } \Omega \,\}, \]
clearly is the $\sigma$-algebra on $C$ generated by $T$ (as $S$ generates
the Baire $\sigma$-algebra of $\Omega$.) It follows that $\mathcal{E}$ is the
Baire $\sigma$-algebra of $C$, as the latter is generated by $T$.
\end{proof}

\begin{theorem}\label{bBaireext}%
Every bounded Baire function on $C$ can be extended to a
bounded Baire function on $\Omega$.%
\pagebreak
\end{theorem}

\begin{proof}
It suffices to show this for non-negative functions.
So let $f$ be a non-negative bounded Baire function on $C$.
There then exists an increasing sequence $(f_n)$ of non-negative
Baire-step functions on $C$ converging uniformly to $f$.
By the preceding proposition \ref{subBaire}, the Baire-step functions
$f_n$ can be extended to non-negative Baire-step functions $g_n$
on $\Omega$. The truncated function
\[ \limsup_{n \to \infty} \ \bigl( \,| \,f \,|_{\,\infty\,} 1_{\,\Omega} \,\bigr) \wedge g_n \]
then is a non-negative bounded Baire function on $\Omega$ extending $f$.
\end{proof}

\clearpage

%


\backmatter


\phantomsection

\addcontentsline{toc}{part}{Back Matter}

\phantomsection

\fancyhead[LE]{\SMALL{BIBLIOGRAPHY}}
\fancyhead[RO]{\SMALL{BIBLIOGRAPHY}}

\clearpage


\fancyhead[LE]{\SMALL{LIST OF SYMBOLS}}
\fancyhead[RO]{\SMALL{LIST OF SYMBOLS}}

\Printindex{symbols}{List of Symbols}

\fancyhead[LE]{\SMALL{INDEX}}
\fancyhead[RO]{\SMALL{INDEX}}

\Printindex{concepts}{Index}

\clearpage





\begin{thebibliography}{99}

\centerline{\textbf{Algebra}}

\vspace{1ex}


\bibitem{Kost} A.\ I.\ Kostrikin. \emph{Introduction to Algebra}.
(Universitext), Springer-Verlag, 1982


\vspace{1ex}

\begin{center} \hspace{-2.3em} \textbf{Analysis} \end{center}

\vspace{1ex}


\bibitem{AmE} H.\ Amann, J.\ Escher. \emph{Analysis}.
3 volumes, Translated from the German, Birkh\"{a}user Verlag,
vol.\ I: 2005, vol.\ II: 2008, vol.\ III: 2009


\bibitem{RuPr} W.\ Rudin. \emph{Principles of Mathematical Analysis}.
Third Edition, (international student edition), McGraw-Hill International
Book Company, 1976


\vspace{1ex}

\begin{center} \hspace{-2.3em} \textbf{Complex Analysis} \end{center}

\vspace{1ex}


\bibitem{Hen} P.\ Henrici. \emph{Applied and Computational Complex
Analysis}. vol.\ I, Wiley-Interscience, 1974


\bibitem{Rem} R.\ Remmert. \emph{Theory of Complex Functions}.
Readings in Mathematics, GTM 122, Springer, 1991


\bibitem{RCAC} W.\ Rudin. \emph{Real and Complex Analysis}. Third Edition,
(international edition), Mathematics Series, McGraw-Hill Book Company, 1987


\vspace{1ex}

\begin{center} \hspace{-2.3em} \textbf{Topology} \end{center}

\vspace{1ex}


\bibitem{Eng} R.\ Engelking. \emph{General Topology}. Revised and
completed edition, Sigma Series in Pure Mathematics, vol.\ 6, Heldermann
Verlag, Berlin, 1989


\bibitem{TopFoll} G.\ B.\ Folland. \emph{Real Analysis: Modern Techniques
and Their Applications}. 2nd edition, Wiley Interscience, 1999


\bibitem{PedT} G.\ K.\ Pedersen. \emph{Analysis Now}. GTM 118,
Springer-Verlag, New York, Corrected second printing, 1995


\bibitem{Top} B.\ von Querenburg. \emph{Mengentheoretische Topologie}.
Dritte neu bearbeitete und erweiterte Auflage, Springer-Verlag, Berlin -
Heidelberg - New York, 2001

\bibitem{Schu} H.\ Schubert. \emph{Topologie}. 4. Auflage,
Mathematische Leitf\"{a}den, B.\ G.\ Teubner, Stuttgart, 1975

\bibitem{Willa} S.\ Willard. \emph{General Topology}. Dover, 2004


\vspace{1ex}

\begin{center} \hspace{-2.3em} \textbf{Integration Theory, applicable to this book}
\end{center}

\vspace{1ex}


\bibitem{Bau} H.\ Bauer. \emph{Measure and Integration Theory}.
Studies in Mathematics 26, Walter de Gruyter, Berlin, New York, 2001


\bibitem{CFWvI} C.\ Constantinescu, W.\ Filter, K.\ Weber in collaboration
with A.\ Sontag. \emph{Integration Theory}. vol.\ I, Wiley-Interscience, 1985


\bibitem{CFW} C.\ Constantinescu, W.\ Filter, K.\ Weber in collaboration
with A.\ Sontag. \emph{Advanced Integration Theory}. Mathematics and
its Applications, Kluwer Academic Publishers, 1998
\pagebreak


\bibitem{FW} W.\ Filter, K.\ Weber. \emph{Integration Theory}. Mathematics
Series, Chapman \& Hall, 1997


\vspace{1.2ex}

\begin{center} \hspace{-2.3em}
\textbf{other valuable books on Integration Theory, less well suited for this text}
\end{center}

\vspace{1.4ex}


\bibitem{Flo} K.\ Floret. \emph{Ma\ss- und Integrationstheorie}. Teubner
Studienb\"ucher Mathematik, B.\ G.\ Teubner, Stuttgart, 1981


\bibitem{FollIt} G.\ B.\ Folland. \emph{Real Analysis}. Modern Techniques
and Their Applications, 2nd edition, John Wiley \& Sons, 1999


\bibitem{Halm} P.\ R.\ Halmos. \emph{Measure Theory}. Springer-Verlag,
New York, 1974


\bibitem{Kant} S.\ Kantorovitz. \emph{Introduction to Modern Analysis}.
Oxford Graduate Texts in Mathematics 8, Oxford, 2006


\bibitem{Rao} M.\ M.\ Rao. \emph{Measure Theory and Integration}.
Second Edition, Revised and Expanded, Marcel Dekker, New York - Basel,
2004


\bibitem{RudIt} W.\ Rudin. \emph{Real and Complex Analysis}.
Third edition, McGraw-Hill International Editions, 1987


\vspace{1.2ex}

\begin{center} \hspace{-2.3em} \textbf{Functional Analysis} \end{center}

\vspace{1.2ex}


\bibitem{Conw} J.\ B.\ Conway. \emph{A Course in Functional Analysis, second
edition}. GTM 96, Springer-Verlag, New York - Berlin - Heidelberg - Tokyo, 1990


\bibitem{Krey} E.\ Kreyszig. \emph{Introductory Functional Analysis with
Applications}. Wiley Classics Library Edition, John Wiley \& Sons, 1989


\bibitem{PedB} G.\ K.\ Pedersen. \emph{Analysis Now}. GTM 118,
Springer-Verlag, New York, Corrected second printing, 1995


\bibitem{ReSi1} M.\ Reed, B.\ Simon. \emph{Functional Analysis}. Methods of
Modern Mathematical Physics, vol.\ I.\ Academic Press, New York, 1972


\bibitem{Sche} M.\ Schechter. \emph{Principles of Functional Analysis}.
2nd Edition. Graduate Studies in Mathematics vol.\ 36, American Mathematical
Society, 2002


\bibitem{Wer} D.\ Werner. \emph{Funktionalanalysis}. 6te, korrigierte Auflage,
Springer-Verlag, Berlin - Heidelberg, 2007


\vspace{1.2ex}

\begin{center} \hspace{-2.3em} \textbf{Banach Algebras} \end{center}

\vspace{1.2ex}


\bibitem{Aup} B.\ Aupetit. \emph{A Primer on Spectral Theory}.
Universitext, Springer-Verlag, New York, 1991


\bibitem{BD} F.\ F.\ Bonsall, J.\ Duncan. \emph{Complete Normed Algebras}.
Ergebnisse der Mathematik und Ihrer Grenzgebiete, Band 80, Springer-Verlag,
Berlin - Heidelberg - New York, 1973


\bibitem{BB} N.\ Bourbaki. \emph{Th\'{e}ories spectrales, Chapitres 1 et 2}.
\'{E}l\'{e}ments de Math\'{e}matique, R\'{e}impression inchang\'{e}e de
l'\'{e}dition originale de 1967, Springer-Verlag, 2007


\bibitem{GaalB} S.\ A.\ Gaal. \emph{Linear Analysis and Representation Theory}.
Die Grundlehren der math.\ Wissenschaften in Einzeldarstellungen,
Band 198, Springer-Verlag, New York, 1973


\bibitem{GRSh} I.\ Gel'fand, D.\ Ra\u{\i}kov, G.\ Shilov. \emph{Commutative
Normed Rings}. Chelsea Publishing Company, New York, 1964


\bibitem{HelBA} A.\ Ya.\ Helemskii. \emph{Banach and Locally Convex Algebras}.
Oxford University Press, 1993


\bibitem{Kan} E.\ Kaniuth. \emph{A Course in Commutative Banach Algebras}.
GTM 246, Springer, New York, 2009
\pagebreak


\bibitem{LarBA} R.\ Larsen. \emph{Banach Algebras, an introduction}.
Marcel Dekker, New York, 1973


\bibitem{LoBA} L.\ H.\ Loomis. \emph{An Introduction to Abstract Harmonic Analysis}.
van Nostrand, 1953


\bibitem{Mos} R.\ D.\ Mosak. \emph{Banach Algebras}. Chicago-London,
University of Chicago Press, 1975


\bibitem{MullB} V.\ M\"{u}ller. Normed algebras with involution,
\emph{Manuscripta Math.} \textbf{75} no. 2, pp.\ 197--210, 1992


\bibitem{Mur} G.\ J.\ Murphy. \emph{C*-Algebras and Operator Theory}.
Academic Press, London, 1990


\bibitem{NaiBA} M.\ A.\ Na\u{\i}mark. \emph{Normed Algebras}. Third American
edition, translated from the second Russian edition, Wolters-Noordhoff Publishing,
Groningen, 1972


\bibitem{Neu} M.\ A.\ Neumark. \emph{Normierte Algebren}. German translation
of the second Russian edition, Verlag Harri Deutsch, Thun \& Frankfurt am Main,
1990


\bibitem{Palm} T.\ W.\ Palmer. Encyclopedia of Mathematics and its
Applications, vol.\ 49: \emph{Algebras and Banach Algebras}, \& vol.\ 79:
\emph{Banach Algebras and the General Theory of \st-Algebras}. Cambridge
University Press, 1994, 2001


\bibitem{Ri} C.\ E.\ Rickart. \emph{General Theory of Banach Algebras}. Van
Nostrand, Princeton, 1960


\bibitem{RFAF} W.\ Rudin. \emph{Functional Analysis}. International Series in
Pure and Applied Mathematics, McGraw-Hill, 2nd Edition 1991


\bibitem{Zel} W.\ \.{Z}elazko. \emph{Banach Algebras}. Elsevier, New York, 1973


\bibitem{ZhuB} K.\ Zhu. \emph{An Introduction to Operator Algebras}.
CRC Press, 1993


\vspace{1.1ex}

\begin{center}
\hspace{-2.3em} \textbf{Representations and Positive Linear Functionals}
\end{center}

\vspace{1.1ex}


\bibitem{Dieu} J.\ Dieudonn\'{e}. \emph{Treatise on Analysis}. vol.\ II,
Enlarged and Corrected Printing, Academic Press, 1976 


\bibitem{DixP} J.\ Dixmier. \emph{Les C*-alg\`{e}bres et leurs
repr\'{e}sentations}. 2\`eme \'ed.\ Gauthier-Villars, 1969. R\'eimpression
1996 \'editions Jacques Gabay


\bibitem{FeDo} J.\ M.\ G.\ Fell, R.\ S.\ Doran, \emph{Representations of
\st-Algebras, Locally Compact Groups, and Banach \st-Algebraic Bundles}.
vol.\ I, Pure and Applied Mathematics, vol.\ 125, Academic Press,
1988


\bibitem{Gaal} S.\ A.\ Gaal. \emph{Linear Analysis and Representation Theory}.
Die Grundlehren der math.\ Wissenschaften in Einzeldarstellungen,
Band 198, Springer-Verlag, New York, 1973


\bibitem{MosR} R.\ D.\ Mosak. \emph{Banach Algebras}. Chicago-London,
University of Chicago Press, 1975


\bibitem{MullR} V.\ M\"{u}ller. Normed algebras with involution,
\emph{Manuscripta Math.} \textbf{75} no. 2, pp.\ 197--210, 1992


\bibitem{MurR} G.\ J.\ Murphy. \emph{C*-Algebras and Operator Theory}.
Academic Press, London, 1990


\bibitem{NaiRP} M.\ A.\ Na\u{\i}mark. \emph{Normed Algebras}. Third American
edition, translated from the second Russian edition, Wolters-Noordhoff Publishing,
Groningen, 1972


\bibitem{NeuR} M.\ A.\ Neumark. \emph{Normierte Algebren}. German translation
of the second Russian edition, Verlag Harri Deutsch, Thun \& Frankfurt am Main,
1990


\bibitem{RiR} C.\ E.\ Rickart. \emph{General Theory of Banach Algebras}. Van
Nostrand, Princeton, 1960


\bibitem{Willi} J.\ H.\ Williamson. \emph{Lectures on Representation Theory for
Banach Algebras and Locally Compact Groups}. Matscience Report 54,
The Institute of Mathematical Sciences, Madras-20, India, 1968
\pagebreak


\vspace{1ex}

\begin{center} \hspace{-2.3em} \textbf{Spectral Theory of Representations} \end{center}

\vspace{1ex}


\bibitem{AcGl} N.\ I.\ Achieser, I.\ M.\ Glasmann. \emph{Lineare Operatoren
im Hilbert-Raum}. Verlag Harri Deutsch, 8-te erweiterte Auflage, Thun, 1981


\bibitem{AkGl} N.\ I.\ Akhiezer, I.\ M.\ Glazman. \emph{Theory of Linear
Operators in Hilbert space}. Frederick Ungar, vol.\ I, 1961, vol.\ 2, 1963.
Reprinted 1993,\ two volumes bound as one, Dover


\bibitem{Conw2} J.\ B.\ Conway. \emph{A Course in Functional Analysis, second
edition}. GTM 96, Springer-Verlag, New York - Berlin - Heidelberg - Tokyo, 1990


\bibitem{FollSp} G.\ B.\ Folland. \emph{A Course in Abstract Harmonic Analysis}.
CRC Press, 1995


\bibitem{Hal} P.\ R.\ Halmos. \emph{Introduction to Hilbert Space and the Theory
of Spectral Multiplicity}. 2nd edition, Chelsea Publishing Company, New York, 1957


\bibitem{HelST} A.\ Ya.\ Helemskii. \emph{Banach and Locally Convex Algebras}.
Oxford University Press, 1993


\bibitem{MosS} R.\ D.\ Mosak. \emph{Banach Algebras}. Chicago-London,
University of Chicago Press, 1975


\bibitem{MurS} G.\ J.\ Murphy. \emph{C*-Algebras and Operator Theory}.
Academic Press, London, 1990


\bibitem{NaiST} M.\ A.\ Na\u{\i}mark. \emph{Normed Algebras}. Third American
edition, translated from the second Russian edition, Wolters-Noordhoff Publishing,
Groningen, 1972


\bibitem{NeuS} M.\ A.\ Neumark. \emph{Normierte Algebren}. German translation
of the second Russian edition, Verlag Harri Deutsch, Thun \& Frankfurt am Main,
1990


\bibitem{PedST} G.\ K.\ Pedersen. \emph{Analysis Now}. GTM 118,
Springer-Verlag, New York, Corrected second printing, 1995


\bibitem{RFA} W.\ Rudin. \emph{Functional Analysis}. International Series in
Pure and Applied Mathematics, McGraw-Hill, 2nd Edition 1991


\bibitem{SeKu} I.\ E.\ Segal, R.\ A.\ Kunze. \emph{Integrals and Operators}.
Second Revised and Enlarged Edition, Grundlehren der math.\ Wissenschaften,
vol.\ 228, Springer-Verlag, Berlin - Heidelberg - New York, 1978


\bibitem{Nag} B.\ Sz.-Nagy. \emph{Spektraldarstellung Linearer Transformationen
des Hilbertschen Raumes}. Ergebnisse der Mathematik und ihrer Grenzgebiete,
Band 39, Berich\-tigter Nachdruck, Springer-Verlag, Berlin - Heidelberg - New York,
1967


\bibitem{Weid} J.\ Weidmann. \emph{Lineare Operatoren in Hilbertr\"{a}umen}.
Math.\ Leitf\"aden, B.\ G.\ Teubner, Stuttgart, vol.\ I 2002, vol.\ II 2003


\bibitem{ZhuST} K.\ Zhu. \emph{An Introduction to Operator Algebras}.
CRC Press, 1993



\vspace{3ex}

\begin{center}
\hspace{-2.3em} \underline{Further Reading:}
\end{center}

\vspace{2ex}

\begin{center}
\hspace{-2.3em} \textbf{C*-Algebras and von Neumann Algebras (advanced texts)}
\end{center}

\vspace{2ex}



\bibitem{BrRo} O.\ Bratteli, D.\ Robinson. \emph{Operator Algebras and
Quantum Statistical Mechanics}. 2 volumes, 2nd printing of the 2nd edition,
TMP, Springer-Verlag, 2002


\bibitem{DixvN} J.\ Dixmier. \emph{Les alg\`{e}bres d'op\'{e}rateurs dans
l'espace hilbertien (alg\`{e}bres de von Neumann)}. 2\`eme \'ed.\
Gauthier-Villars, 1969. R\'eimpression 1996 \'editions Jacques Gabay


\bibitem{DixC} J.\ Dixmier. \emph{Les C*-alg\`{e}bres et leurs
repr\'{e}sentations}. 2\`eme \'ed.\ Gauthier-Villars, 1969. R\'eimpression
1996 \'editions Jacques Gabay


\bibitem{GaalC} S.\ A.\ Gaal. \emph{Linear Analysis and Representation Theory}.
Die Grundlehren der math.\ Wissenschaften in Einzeldarstellungen,
Band 198, Springer-Verlag, New York, 1973
\pagebreak


\bibitem{KaRi} R.\ V.\ Kadison, J.\ R.\ Ringrose. \emph{Fundamentals of the
Theory of Operator Algebras}, vol.\ I, II: Academic Press, 1983, 1986,
vol.\ III, IV: Birkh\"{a}user, 1991, 1992


\bibitem{PedC} G.\ K.\ Pedersen. \emph{C*-Algebras and their
Automorphism Groups}. London Mathematical Society Monographs
No.\ 14, Academic Press, London, 1979


\bibitem{Sak} S.\ Sakai. \emph{C*-Algebras and W*-Algebras}. Ergebnisse der
Mathematik und Ihrer Grenzgebiete, vol.\ 60, Springer-Verlag, 1971. Reprinted
1998, Classics in Mathematics


\bibitem{Seg} I.\ E.\ Segal. \emph{Decompositions of Operator Algebras I and II}.
Memoirs of the AMS, Number 9, Providence, 1951


\bibitem{Tak} M.\ Takesaki. \emph{Theory of Operator Algebras}.
Encyclopaedia of Mathematical Sciences, 3 volumes, 2nd printing
of the 1979 First Edition, Springer-Verlag, Berlin - Heidelberg - New York, 2001


\vspace{1ex}

\begin{center} \hspace{-2.3em}
\textbf{Banach Algebra texts containing basic Abstract Harmonic Analysis} \end{center}

\vspace{1ex}


\bibitem{BH} N.\ Bourbaki. \emph{Th\'{e}ories spectrales, Chapitres 1 et 2}.
\'{E}l\'{e}ments de Math\'{e}matique, R\'{e}impression inchang\'{e}e de
l'\'{e}dition originale de 1967, Springer-Verlag, 2007


\bibitem{KanH} E.\ Kaniuth. \emph{A Course in Commutative Banach Algebras}.
GTM 246, Springer, New York, 2009


\bibitem{LarAHA} R.\ Larsen. \emph{Banach Algebras, an introduction}.
Marcel Dekker, New York, 1973


\bibitem{LoAHA} L.\ H.\ Loomis. \emph{An Introduction to Abstract Harmonic Analysis}.
van Nostrand, 1953


\bibitem{NaiHA} M.\ A.\ Na\u{\i}mark. \emph{Normed Algebras}. Third American
edition, translated from the second Russian edition, Wolters-Noordhoff Publishing,
Groningen, 1972


\bibitem{NeuH} M.\ A.\ Neumark. \emph{Normierte Algebren}. German translation
of the second Russian edition, Verlag Harri Deutsch, Thun \& Frankfurt am Main,
1990


\vspace{1ex}

\begin{center} \hspace{-2.3em} \textbf{Unbounded Operator Algebras} \end{center}

\vspace{1ex}


\bibitem{KoS} K.\ Schm\"udgen. \emph{Unbounded Operator Algebras
and Representation Theory}. Operator Theory, Vol.\ 37, Birkh\"auser
Verlag, 1990


\vspace{1ex}

\begin{center} \hspace{-2.3em} \textbf{Indefinite Inner Product Spaces} \end{center}

\vspace{1ex}


\bibitem{Bog} J.\ Bogn\'{a}r. \emph{Indefinite Inner Product Spaces}.
Ergebnisse der Mathematik und ihrer Grenzgebiete, Band 78,
Springer-Verlag, New York - Heidelberg - Berlin, 1974


\vspace{1ex}

\begin{center} \hspace{-2.3em} \textbf{Positive Definite Functions} \end{center}

\vspace{1ex}


\bibitem{BCR} C.\ Berg, J.\ P.\ R.\ Christensen, P.\ Ressel.
\emph{Harmonic Analysis on Semigroups, Theory of Positive
Definite and Related Functions}. GTM 100, Springer, 1984


\bibitem{HG} H.\ Gl\"{o}ckner. \emph{Positive Definite Functions on
Infinite-Dimensional Convex Cones}. Memoirs of the American
Mathematical Society, Number 789 (second of 3 numbers),
Volume 166, November 2003 


\vspace{1ex}

\begin{center} \hspace{-2.3em} \textbf{For the Cauchy-Bochner Integral} \end{center}

\vspace{1ex}


\bibitem{AMR} R.\ Abraham, J.\ E.\ Marsden, T.\ Ratiu.
\emph{Manifolds, Tensor Analysis, and Applications}.
Addison-Wesley Publishing Company Inc, Massachusetts,
1983


\bibitem{AmECB} H.\ Amann, J.\ Escher. \emph{Analysis}.
3 volumes, Translated from the German, Birkh\"{a}user Verlag,
vol.\ I: 2005, vol.\ II: 2008, vol.\ III: 2009
\pagebreak


\end{thebibliography}
\end{document}